\documentclass[10pt,a4paper]{article}
\usepackage{xcolor}
\usepackage{mdframed}
\usepackage[utf8]{inputenc}
\usepackage{amsmath, amsthm, amssymb, bbm, bm}
\usepackage{tikz-cd}
\usepackage{enumerate}
\usepackage{appendix}
\usepackage[hidelinks]{hyperref}
\usepackage{color}

\definecolor{mygray}{gray}{0.75}

\newtheorem{tvrz}{Proposition}[section]
\newtheorem{lemma}[tvrz]{Lemma}
\newtheorem{theorem}[tvrz]{Theorem}
\newtheorem{cor}[tvrz]{Corollary}
\theoremstyle{definition}
\newtheorem{definice}[tvrz]{Definition}
\theoremstyle{remark}
\newtheorem{rem}[tvrz]{Remark}
\theoremstyle{definition}
\newtheorem{mdexample}[tvrz]{Example}
\newenvironment{example}%
{\begin{mdframed}[topline=false, rightline=false, bottomline=false, linewidth=0.2em, linecolor=mygray, innerleftmargin=0.5em, innerrightmargin=0,leftmargin=-0.7em]\begin{mdexample}}%
{\end{mdexample}\end{mdframed}}

\def\^{\wedge}
\def\<{\langle}
\def\>{\rangle}
\def\M{\mathcal{M}}
\def\O{\mathcal{O}}
\def\N{\mathbb{N}}
\def\cN{\mathcal{N}}
\def\cK{\mathcal{K}}
\def\cL{\mathcal{L}}
\def\cS{\mathcal{S}}
\def\X{\mathfrak{X}}

\def\R{\mathbb{R}}
\def\C{\mathcal{C}}

\def\Q{\mathcal{Q}}
\def\B{\mathcal{B}}
\def\E{\mathcal{E}}

\def\czC{\check{C}}
\def\czH{\check{H}}

\def\F{\mathcal{F}}
\def\G{\mathcal{G}}
\def\U{\mathcal{U}}
\def\V{\mathcal{V}}
\def\J{\mathcal{J}}

\def\P{\mathcal{P}}
\def\H{\mathcal{H}}
\def\Z{\mathbb{Z}}
\def\I{\mathcal{I}}

\def\ol{\overline}
\def\fPsi{\mathbf{\Psi}}

\def\frJ{\mathfrak{J}}
\def\frU{\mathfrak{U}}

\def\fp{\mathbf{p}}
\def\fq{\mathbf{q}}
\def\fr{\mathbf{r}}

\def\bbz{\mathbbm{z}}
\def\bby{\mathbbm{y}}

\def\bbu{\mathbbm{u}}

\def\fD{\mathbf{D}}

\def\dr{\mathrm{d}}

\def\1{\mathbbm{1}}

\def\f1{\mathbf{1}}

\def\A{\mathcal{A}}
\def\~{\widetilde}
\def\dlim{\varinjlim}
\def\tr{\triangleright}
\def\tl{\triangleleft}

\def\frS{\mathfrak{S}}
\newcommand{\Li}[1]{\mathcal{L}_{#1}}

\newcommand\ul[1]{\underline{#1}}

\DeclareMathOperator{\gdim}{gdim}
\DeclareMathOperator{\colim}{colim}

\DeclareMathOperator{\Op}{\textbf{Op}}

\DeclareMathOperator{\AgMod}{\mathit{A}\text{-}\textbf{gMod}}

\DeclareMathOperator{\gcRng}{\textbf{gcRng}}

\DeclareMathOperator{\gRng}{\textbf{gRng}}
\DeclareMathOperator{\Rng}{\textbf{Rng}}
\DeclareMathOperator{\Et}{\textbf{Et}}

\DeclareMathOperator{\Set}{\textbf{Set}}

\DeclareMathOperator{\ev}{ev}
\DeclareMathOperator{\pg}{pg}
\DeclareMathOperator{\lc}{lc}
\DeclareMathOperator{\fin}{\hspace{0.3mm}f}
\DeclareMathOperator{\lfin}{\hspace{0.3mm}lf}

\DeclareMathOperator{\Vect}{\textbf{Vec}}
\DeclareMathOperator{\gVect}{\textbf{gVec}}

\DeclareMathOperator{\PSh}{\textbf{PSh}}
\DeclareMathOperator{\gAb}{\textbf{gAb}}
\DeclareMathOperator{\Sh}{\textbf{Sh}}
\DeclareMathOperator{\Sff}{Sff}
\DeclareMathOperator{\gMan}{\textbf{gMan}}
\DeclareMathOperator{\gSet}{\textbf{gSet}}
\DeclareMathOperator{\Man}{\textbf{Man}}

\DeclareMathOperator{\As}{\textbf{As}}
\DeclareMathOperator{\gAs}{\textbf{gAs}}

\DeclareMathOperator{\gcAs}{\textbf{gcAs}}
\DeclareMathOperator{\cAs}{\textbf{cAs}}

\DeclareMathOperator{\gLRS}{\textbf{gLRS}}

\DeclareMathOperator{\gVBun}{\textbf{gVBun}}

\DeclareMathOperator{\im}{im}
\DeclareMathOperator{\Der}{Der}

\DeclareMathOperator{\GL}{GL}

\DeclareMathOperator{\op}{op}
\DeclareMathOperator{\odd}{odd}

\DeclareMathOperator{\Lin}{Lin}
\DeclareMathOperator{\lin}{lin}
\DeclareMathOperator{\gDer}{gDer}

\DeclareMathOperator{\an}{an}

\DeclareMathOperator{\gr}{gr}
\DeclareMathOperator{\rk}{rk}
\DeclareMathOperator{\grk}{grk}
\DeclareMathOperator{\supp}{supp}

\begin{document}
\begin{flushright}
\today
% preprint number (if any)
\end{flushright}
\vspace{0.7cm}
\begin{center}
 %\vskip1cm

\baselineskip=13pt {\Large \bf{Global Theory of Graded Manifolds}\\}
 \vskip0.5cm
 {\it Dedicated to Frank Vincent Zappa}  
 \vskip0.7cm
 {\large{Jan Vysoký$^{1}$}}\\
 \vskip0.6cm
$^{1}$\textit{Faculty of Nuclear Sciences and Physical Engineering, Czech Technical University in Prague\\ Břehová 7, 115 19 Prague 1, Czech Republic, jan.vysoky@fjfi.cvut.cz}\\
\vskip0.3cm
\end{center}

\begin{abstract}
A theory of graded manifolds can be viewed as a generalization of differential geometry of smooth manifolds. It allows one to work with functions which locally depend not only on ordinary real variables, but also on $\Z$-graded variables which can either commute or anticommute, according to their degree. To obtain a consistent global description of graded manifolds, one resorts to sheaves of graded commutative associative algebras on second countable Hausdorff topological spaces, locally isomorphic to a suitable ``model space''. 

This paper aims to build robust mathematical foundations of geometry of graded manifolds. Some known issues in their definition are resolved, especially the case where positively and negatively graded coordinates appear together. The focus is on a detailed exposition of standard geometrical constructions rather then on applications. Necessary excerpts from graded algebra and graded sheaf theory are included. 
\end{abstract}


\tableofcontents
\section*{Introduction} \label{sec_introduction}
\addcontentsline{toc}{section}{\nameref{sec_introduction}}
In recent years, graded manifolds rose to prominence as a useful tool both in differential geometry and mathematical physics. In a nutshell, they can be viewed as a modification of supergeometry, allowing for a more refined grading of local coordinates governed by the abelian group $\mathbb{Z}$ instead of just $\mathbb{Z}_{2}$. Most notably, the theory of graded manifolds is a crucial mathematical notion for BV \cite{batalin1983quantization,batalin1984gauge} and AKSZ \cite{alexandrov1997geometry} formalism in quantum field theory. There is a plethora of excellent materials covering this topic, see e.g. \cite{qiu2011introduction,ikeda2017lectures,2011RvMaP..23..669C}. On the other hand, graded manifolds provide an elegant alternative description of Courant algebroids \cite{roytenberg2002structure}, they can be used to construct Drinfel'd doubles for Lie bialgebroids \cite{Voronov:2001qf}, Lie algebroids and their higher analogues \cite{Voronov:2010hj} or they are utilized for integration of Courant algebroids \cite{li2011integration, vsevera2015integration}. Note that this is not intended to be a complete list of references. 

Essentially, graded manifolds can be defined in two distinct ways. Note that by ``graded'' we exclusively mean ``$\Z$-graded''. First, one can define them as supermanifolds whose structure sheaf is equipped with an additional $\Z$-grading. It can be either compatible with the supermanifold grading \cite{Kontsevich:1997vb, severa2001some}, or not \cite{Voronov:2019mav}. Alternatively, one can attempt to mimic the Berezin--Leites and Kostant viewpoint on supermanifolds \cite{leites1980introduction, kostant1977graded} and define graded manifolds directly as a sheaf of $\Z$-graded commutative associative algebras over a second countable Hausdorff topological space, locally isomorphic to an appropriate ``local model'' sheaf. This was carried out e.g. in  \cite{mehta2006supergroupoids, jubin2019differential}, and used also in \cite{2011RvMaP..23..669C}. The main idea is to consider a finite-dimensional graded vector space $V$ and its corresponding symmetric algebra $S(V)$. For each open set $U \subseteq \R^{n}$, one can then consider a graded commutative associative algebra
\begin{equation} \label{eq_modelcommutative}
\C^{\infty}_{n}(U) \otimes_{\R} S(V),
\end{equation}
where $\C^{\infty}_{n}$ denotes the sheaf of smooth functions on $\R^{n}$. After choosing some (homogeneous) basis $(\xi_{\mu})_{\mu=1}^{m}$ of $V$, elements of this algebra can be viewed as polynomials in the variables $\xi_{\mu}$ with coefficients in the algebra of smooth functions on $U$, where 
\begin{equation}
\xi_{\mu} \xi_{\nu} = (-1)^{|\xi_{\mu}||\xi_{\nu}|} \xi_{\nu} \xi_{\mu}.
\end{equation}
However, it was noted in \cite{fairon2017introduction} that the assignment (\ref{eq_modelcommutative}) in general does not define a (graded) locally ringed space, an important assumption of the sheaf-theoretic approach to graded manifolds. This can be fixed by considering just non-negatively graded manifolds (usually called $N$-manifolds in the literature). Unfortunately, this does not help with the other issue - it does not define a sheaf of graded commutative associative algebras. Note that this was first brought to my attention in the bachelor thesis \cite{bialas}. The solution to this issue, proposed in \cite{fairon2017introduction}, is to resort to formal power series in ``purely graded'' coordinates $\xi_{\mu}$. Additionally, this brings new subtle issues with the very notion of graded objects (vector spaces, rings, algebras) which have to be resorted in order to get a consistent theory of graded manifolds. In this paper, we stick to the second approach. More precisely, we aim towards the following statement:

\textit{A graded manifold is a graded locally ringed space $\M = (M, \C^{\infty}_{\M})$, where $M$ is a second countable Hausdorff topological space, locally isomorphic to a graded domain. 
}

Taking into account the issues described above, this involves settling for a precise definition of a graded locally ringed space and choosing an appropriate local model for graded manifolds, their ``prototypical'' example, called henceforth a \textit{graded domain}. Note that one of our intentions was to evade the dangerous waters of functional analysis and popular ``by a suitable completion'' arguments, in order to not obscure the algebraic nature of graded manifolds. 

This brings us to the main question. What is the purpose of this paper? There are already many research papers aimed towards examples and applications in geometry and physics. Our main philosophy is quite the opposite. We would like to explore the territory of graded manifolds \textit{themselves}, with an emphasis on their global description. In particular, we investigate the graded analogues of most of standard notions of differential geometry. Our big motivation was to find a $\Z$-graded counterpart of the excellent book \cite{carmeli2011mathematical}. 

To achieve this goal, we have decided to build everything from the ground up. The reader is only required to have a basic knowledge of (linear) algebra, differential geometry and category theory. No prior knowledge of sheaf theory or algebraic geometry is needed to understand the text. Except for very few occasions, we provide explicit proofs of all propositions, moving the more complicated ones to the appendix to not disrupt the flow of the text. 

The paper is organized as follows:

\begin{enumerate}[$\bullet$]
\item In Section \ref{sec_galgbasics}, we introduce basic notions of graded algebra. Let us emphasize that we took a bit unconventional approach. Graded objects are \textit{not} assumed to be ordinary ones allowing for a direct sum decomposition labeled by $\Z$, but rather sequences of ordinary objects labeled by $\Z$, with graded morphisms being sequences of ordinary ones. It turns out that not only this viewpoint is necessary to work with formal power series, but it is also more ``category friendly''. 

Graded abelian groups are defined as sequences of ordinary abelian groups. Since there is a notion of a tensor product of two graded abelian groups, one can define graded (commutative) rings. We introduce a notion of local graded rings and their basic properties. Similarly, graded vector spaces are sequences of ordinary vector spaces and a tensor product of two graded vector spaces is used to define graded (commutative) associative algebras. 

Tensor and symmetric algebras of a given graded vector space are constructed as examples of graded (commutative) associative algebras. Importantly, we generalize this notion to obtain a so called extended symmetric algebra with coefficients in a given graded commutative associative algebra, essential in order to define a local model for graded manifolds. The final subsection is devoted to graded modules and their derivations. 
\item Section \ref{sec_grsheaves} serves as a necessary introduction to sheaf theory. We start by defining presheaves and sheaves on a topological space valued in the category of graded commutative associative algebras. After initial definitions, we construct an example which is of utmost importance for this paper, see Example \ref{ex_thesheaf}. On the other hand, in Example \ref{ex_notasheaf} we argue why (\ref{eq_modelcommutative}) in general fails to define a sheaf.

Next, we define stalks of presheaves of graded commutative associative algebras and prove that they exist. Consequently, we may use them to describe the canonical sheafification procedure making each presheaf of graded commutative associative algebras into a sheaf. 

In the following subsection, we introduce the category of graded locally ringed spaces. Each its object consists of a topological space $X$ together with a sheaf of graded commutative associative algebras on $X$, such that all its stalks are graded local rings. This is a fundamental mathematical notion of this paper, since graded manifolds \textit{are} just special examples of graded locally ringed spaces. In Example \ref{ex_twogLRS}, we show that under certain assumptions, the sheaf from Example \ref{ex_thesheaf} defines a graded locally ringed space. 

Finally, we in detail examine the category of sheaves of graded $\A$-modules, where $\A$ is a sheaf of graded commutative associative algebras. In particular, we focus on its subcategory of locally finitely and freely generated ones. There is a good reason for this - graded vector bundles will be \textit{defined} as locally finitely and freely generated sheaves of $\A$-modules, where we choose $\A$ to be a structure sheaf of a graded manifold. 

\item In Section \ref{sec_grman}, we deal with the main subject of this paper, graded manifolds. We start with the definition of graded domains, examples of graded locally ringed spaces serving as a local model for graded manifolds. After examining their basic properties, we use a graded analogue of Hadamard's lemma to understand the morphisms of graded domains. It turns out that they have a quite rigid structure, see Theorem \ref{thm_gradedomaintheorem}.

Having the notion of a graded domain in hand, we arrive to the definition of a graded manifold. We show that its underlying topological space is always an ordinary smooth manifold, and there is a canonical \textit{body map} from its structure sheaf to the sheaf of ordinary smooth functions. We give some elementary examples and examine standard properties of functions on graded manifolds (sections of its structure sheaf).

In the following subsection, we formulate and prove a set of collation and gluing theorems, allowing one to construct graded manifolds from collections of local data satisfying certain consistency conditions. An example of this procedure is used to construct a degree shifted vector bundle. 

We then prove that partitions of unity exist on graded manifolds and explain standard manipulations with them. We list some of the vast consequences for the properties of  structure sheaves of a graded manifolds. In particular, there exist graded bump functions, locally defined functions can be in some sense extended to globally defined ones, and quotients of a structure sheaf by a sheaf of ideals always form a \textit{sheaf}.

The final subsection shows that the category of graded manifolds has binary products. We can thus construct a product graded manifold $\M \times \cN$ of two graded smooth manifolds together with canonical projections, having the universal property. This determines it uniquely up to a graded diffeomorphism. 

\item Section \ref{sec_infinitesimal} is devoted to infinitesimal properties of graded manifolds. A tangent space to a graded manifold at a given point of its underlying topological space is defined as a graded vector space of graded derivations from the respective stalk to the space of real numbers. We define a differential of a graded smooth map at a given point and argue that it has all the usual properties. Using graded local charts, we construct a usual (total) basis of a tangent space, consisting of ``partial derivatives''. This allows us to prove that the graded dimension is invariant with respect to graded diffeomorphisms. 

Vector fields on a graded manifold $\M$ are introduced as the sheaf of graded derivations of the structure sheaf $\C^{\infty}_{\M}$. We prove that it defines a locally freely and finitely generated sheaf of graded $\C^{\infty}_{\M}$-modules by constructing a local frame consisting of usual ``coordinate'' vector fields. We prove that vector fields can be evaluated at a point, defining thus a tangent vector at the respective tangent space. Graded commutator of two vector field is introduced and we show that this makes the sheaf of vector fields into a sheaf of graded Lie algebras. Every graded smooth map $\phi$ allows one to talk about $\phi$-related vector fields, and we show how this is compatible with differentials and graded commutators. 

In the final subsection, we prove the graded version of one of the most fundamental statements of differential geometry, the inverse function theorem. We then use it to prove the local immersion and submersion theorems, providing one with convenient graded local charts adapted to immersions and submersions. The graded version of the implicit function theorem is derived. 

\item In Section \ref{sec_vb}, we deal with graded vector bundles over graded manifolds. They are defined as locally finitely and freely generated sheaves of graded $\C^{\infty}_{\M}$-modules of a constant graded rank. We show that each graded vector bundle has its dual, and use this fact to define a category of graded vector bundles over arbitrary graded manifolds. 

Subbundles and quotients by subbundles are introduced. Most importantly, there is a natural notion of a pullback graded vector bundle. Other canonical operations with vector bundles are introduced, like Whitney sums, tensor products and degree shifts. Transition maps and functions are constructed from a given local trivialization and their form for dual, pullback and degree shifted vector bundles is examined. 

In the final subsection, we use the transition functions to construct a total space of a graded vector bundle. We obtain a graded manifold over an ordinary vector bundle, together with a projection to the base manifold. It is unique up to a graded diffeomorphism and we show how the original sheaf of sections can be obtained from its sheaf of smooth functions. 

\item A sheaf of differential forms on a graded manifold is defined in Section \ref{sec_exterior} as a sheaf of functions on a degree shifted tangent bundle. Note that since our graded manifolds can have coordinates of arbitrary negative degree, this requires one to ``shift degrees just enough'' to ensure that the coordinate $1$-forms have a positive degree. We show that for each $p \in \Z$, there is a canonical subsheaf of $p$-forms and introduce an alternate more convenient grading. $0$-forms can be then identified with functions and $1$-forms can be identified with sections of the cotangent bundle. 

The exterior differential and the interior product can be naturally obtained as vector fields on the degree shifted tangent bundle, hence by construction, they define graded derivations of the sheaf of differential forms. The Lie derivative is then given as their graded commutator, and a full list of Cartan relations is satisfied by all three operations. We show that differential forms can be pulled along graded smooth maps, and on $1$-forms, this can be viewed as an induced morphism of the respective tangent bundles.  

In the final subsection, we define the de Rham cohomology for graded manifolds. It is easy to see that it is trivial for non-zero degree forms. We prove the Poincaré lemma for the degree zero de Rham cohomology, showing that for graded domains, it is isomorphic to the ordinary de Rham cohomology of the underlying open subset. In particular, this shows that  closed forms on graded manifolds are locally exact. We prove that globally, the degree zero de Rham cohomology of any graded manifold is isomorphic to the Čech cohomology of the constant sheaf, hence to the ordinary de Rham cohomology. In particular, the de Rham cohomology in its presented form cannot distinguish among two graded manifolds over the same underlying manifold. 

\item In Section \ref{sec_submanifolds}, the last one in this paper, we deal with submanifolds. Immersed and (closed) embedded submanifolds are defined in a straightforward way. We prove the extension property for functions on embedded submanifolds. One can identify the tangent space of any immersed submanifold with a subspace of the tangent space of the ambient graded manifold and we examine the subset of vector fields tangent to a submanifold. The tangent bundle to any immersed submanifold can be identified with a subbundle of the restricted tangent bundle, hence allowing for a definition of a normal bundle.

We spend some time by examining a special class of sheaves of ideals of a structure sheaf, called sheaves of regular ideals. In a nutshell, for certain points of the underlying manifold, they are finitely generated on some its neighborhood, such that the generating graded subset consists of global functions independent at that point. This kind of sheaves of ideals plays a crucial part in the algebraic description of closed embedded submanifolds explored in the following subsection. More precisely, we show that every closed embedded submanifold $(\cS,\iota)$ of $\M$ defines and is uniquely determined by a regular sheaf of ideals $\ker( \iota^{\ast}) \subseteq \C^{\infty}_{\M}$. 

After defining the notion of transversal graded smooth maps, one can use this observation to construct an inverse image of a closed embedded submanifold. We observe that it can be actually  defined as a pullback in the category of graded smooth manifolds. Consequently, this can be used as a general definition of an inverse image of a submanifold. As an important example, one can define regular level set submanifolds for any regular value of a graded smooth map.

In the final subsection, we show how inverse image submanifolds can be utilized to obtain the well defined notion of a fiber product of two manifolds. We prove that for a pair of transversal maps, the fiber product exists. In particular, we can define an intersection of two immersed submanifolds and prove that if they are transversal, they intersect cleanly. 
\end{enumerate}

\section{Basics of graded algebra} \label{sec_galgbasics}
Let us establish the most important conventions which will be used throughout the entire paper. 

A \textbf{graded set} is a sequence $S = \{ S_{k} \}_{k \in \Z}$, where for each $k \in \Z$, $S_{k}$ is a set. We usually write $x \in S$ and say that \textbf{$x$ is an element of $S$}, if there exists unambiguous $k \in \Z$, such that $x \in S_{k}$. This $k$ is called the \textbf{degree of $x$} and it is denoted as $|x|$. We say that a graded set $U$ is a subset of a graded set $S$, if $U_{k} \subseteq S_{k}$ for all $k \in \Z$. We will write $U \subseteq S$. 

A \textbf{graded mapping} $\varphi: S \rightarrow T$ from a graded set $S$ to a graded set $T$ is a sequence $\varphi = \{ \varphi_{k} \}_{k \in \Z}$, where for each $k \in \Z$, $\varphi_{k}: S_{k} \rightarrow T_{k}$ is an ordinary mapping of sets. Whenever possible, for any $x \in S$, we will write just $\varphi(x)$ instead of $\varphi_{|x|}(x)$. Graded sets together with graded mappings form a \textbf{category of graded sets} denoted as $\gSet$. 

We will use these conventions for every subsequent ``graded'' category. Note that a graded set $S$ \textit{is not} a set. Moreover, in our formalism there is no place for inhomogeneous (without a defined degree) elements.
\subsection{Graded abelian groups and rings}
By a \textbf{graded abelian group}, we mean a sequence $A = \{ A_{k} \}_{k \in \Z}$, where for each $k \in \Z$, $A_{k}$ is an ordinary abelian group. By a \textbf{graded morphism} $\varphi: A \rightarrow B$ of two graded abelian groups $A$ and $B$, we mean a sequence $\varphi = \{ \varphi_{k} \}_{k \in \Z}$, where for each $k \in \Z$, $\varphi_{k}: A_{k} \rightarrow B_{k}$ is a group homomorphism. Graded abelian groups together with graded morphisms form a category $\gAb$. 

\begin{rem}
To get in habit of using an imprecise notation, let $a,b \in A$. We write $a + b$ whenever it makes sense, that is iff $|a| = |b|$ and the addition is carried out in the abelian group $A_{|a|} = A_{|b|}$. Similarly, we will often use $a = 0$ instead of $a = 0_{|a|}$, where $0_{|a|}$ is the neutral element in $A_{|a|}$. We also use $0$ to denote the trivial graded abelian group $\{ 0_{k} \}_{k \in \Z}$. 
\end{rem}

A \textbf{graded subgroup} $B \subseteq A$ is a collection $B = \{ B_{k} \}_{k \in \Z}$, where for each $k \in \Z$, $B_{k} \subseteq A_{k}$ is a subgroup. For every $A,B \in \gAb$, define the tensor product $A \otimes_{\Z} B$ for each $k \in \Z$ by 
\begin{equation}
(A \otimes_{\Z} B)_{k} := \bigoplus_{i \in \Z} A_{i} \otimes_{\Z} B_{k-i}.
\end{equation}
One can view $\Z$ as a graded abelian group. It is then easy to see that $\otimes_{\Z}$ makes $\gAb$ into a monoidal category with $\Z$ playing the role of its unit object. Moreover, one can introduce a braiding $\tau_{AB}: A \otimes_{\Z} B \rightarrow B \otimes_{\Z} A$ defined on generators by 
\begin{equation} \label{eq_gAbbraiding}
\tau_{AB}(a \otimes b) := (-1)^{|a||b|} b \otimes a,
\end{equation}
for all $a \in A$ and $b \in B$. This makes $\gAb$ into a symmetric monoidal category. 

\begin{definice} \label{def_gring}
Let $R$ be a graded abelian group together with a graded morphism $\mu: R \otimes_{\Z} R \rightarrow R$. Let  $1 \in R_{0}$ be an element called the \textbf{ring unit of $R$}. For all $r,r' \in R$, we will henceforth use the notation
\begin{equation}
r \cdot r' := \mu( r \otimes r').
\end{equation} 
We say that $(R,\mu,1)$ is a \textbf{graded ring}, if for all $r,r',r'' \in R$, one has the associativity rule
\begin{equation}
r \cdot (r' \cdot r'') = (r \cdot r') \cdot r'',
\end{equation}
and the multiplication with the unit does nothing: $r \cdot 1 = 1 \cdot r = r$, for all $r \in R$. We say that $(R,\mu,1)$ is a \textbf{graded commutative ring}, if for all $r,r' \in R$, one has 
\begin{equation}
r \cdot r' = (-1)^{|r||r'|} r' \cdot r. 
\end{equation}
A graded morphism $\varphi: R \rightarrow S$ (an arrow in $\gAb$) is called a \textbf{graded ring morphism}, if $\varphi(r \cdot r') = \varphi(r) \cdot \varphi(r')$ and $\varphi(1) = 1$, for all $r,r' \in R$. Graded rings together with graded ring morphisms form a category which we denote as $\gRng$, and graded commutative rings form its full subcategory $\gcRng$. 
\end{definice}
Note that for a graded (commutative) ring $R$, the component $R_{0}$ is an ordinary (commutative) ring. The definition of graded rings is justified by the following observation:
\begin{tvrz} \label{tvrz_ringsaremonoids}
Graded (commutative) rings are (commutative) monoids in the symmetric monoidal category $(\gAb,\otimes_{\Z},\tau)$. 
\end{tvrz}

A graded subgroup $J \subseteq R$ is called a \textbf{left ideal}, if for all $r \in R$ and $j \in J$, one has $r \cdot j \in J$. Right ideals are defined similarly. If $J$ is both left and right ideal, we call it an \textbf{ideal}. If $R$ is graded commutative, all these notions coincide. For any ideal $J$, there is a unique graded ring structure on the quotient $R / J$, where $(R / J)_{k} := R_{k} / J_{k}$ for all $k \in \Z$, making the canonical quotient map $q: R \rightarrow R/J$ into a graded ring morphism. 

For a graded ring $R$, there is a graded set $\frU(R) \subseteq R$ defined for each $k \in \Z$ as 
\begin{equation}
\frU(R)_{k} := \{ r \in R_{k} \; | \; \text{there is } s \in R_{-k} \text{ such that } r \cdot s = s \cdot r = 1 \},
\end{equation}
called the \textbf{group of units of $R$}. $R$ is called a \textbf{graded division ring}, if $\frU(R)_{k} = R_{k} - \{0_{k}\}$ for all $k \in \Z$. 

A left ideal $J \subsetneq R$ is called a \textbf{maximal left ideal}, if for any other left ideal $I \subsetneq R$, the inclusion $J \subseteq I$ implies $J = I$. Using the Zorn's lemma, one can show that any left ideal $J \subsetneq R$ is contained in some maximal left ideal. In particular, if $R \neq 0$, there always exists some maximal left ideal. For $R \neq 0$, one can thus define the \textbf{left Jacobson radical} $\frJ(R)$ as the intersection of all maximal left ideals. For $R = 0$, set $\frJ(R) := 0$. One can show that $\frJ(R)$ is always an ideal.

\begin{definice}
Let $R$ be a non-zero graded ring. We say that $R$ is a \textbf{local graded ring}, if it contains a \textit{unique} maximal left ideal. In such a case, $\frJ(R)$ is \textit{the} unique maximal left ideal. 

Let $R$ and $S$ be local graded rings. A graded ring morphism $\varphi: R \rightarrow S$ is called a \textbf{local graded ring morphism}, if $\varphi( \frJ(R)) \subseteq \frJ(S)$. Local graded rings together with local graded ring morphisms form a subcategory of $\gRng$. 
\end{definice}

One can restate this definition in several equivalent ways. The proof of the following proposition can be obtained by modifying the one of Theorem 19.1 in \cite{lam2013first}.
\begin{tvrz} \label{tvrz_grlocal}
Let $R$ be a non-zero graded ring. Then the following statements are equivalent:
\begin{enumerate}[(i)]
\item $R$ contains a unique maximal left ideal.
\item $R$ contains a unique maximal right ideal.
\item $R / \frJ(R)$ is a graded division ring.
\item $R - \frU(R)$ is an ideal in $R$.
\item $R - \frU(R)$ is a graded subgroup of $R$. 
\item For any $k \in \Z$, the condition $r + s \in \frU(R)_{k}$ implies that either $r \in \frU(R)_{k}$ or $s \in \frU(R)_{k}$. 
\end{enumerate}
If $R$ is a local graded ring, then one has $\frJ(R) = R - \frU(R)$. 
\end{tvrz}
This proposition has two consequences important for this paper.
\begin{cor} \label{cor_isomorphismarelocal}
Let $R$ and $S$ be two local graded rings and let $\varphi: R \rightarrow S$ be a graded ring isomorphism. Then it is a local graded ring morphism and $\varphi(\frJ(R)) = \frJ(S)$. 
\end{cor}
\begin{proof}
By the previous proposition, we have $\frJ(R) = R - \frU(R)$ and $\frJ(S) = S - \frU(S)$. To prove that $\varphi(\frJ(R)) \subseteq \frJ(S)$, it thus suffices to show that whenever $\varphi(r) \in \frU(S)$, then $r \in \frU(R)$. But if $s$ is the two-sided inverse to $\varphi(r)$, then $\varphi^{-1}(s)$ is obviously the two-sided inverse to $r$. Hence $\varphi$ is a local graded ring morphism. The inclusion $\frJ(S) \subseteq \varphi(\frJ(R))$ is proved similarly. 
\end{proof}
\begin{cor} \label{cor_isoringtolocalislocal}
Let $R$ be a local graded ring and $S$ a graded ring. Let $\varphi: R \rightarrow S$ be a graded ring isomorphism. Then $S$ is also local. 
\end{cor}
\begin{proof}
In the previous proof, we have argued that $\varphi( R - \frU(R)) = S - \frU(S)$. If $R$ is local, the left hand side is a graded subgroup of $S$ by Proposition \ref{tvrz_grlocal}-$(v)$ for $R$. Hence $S - \frU(S)$ is a graded subgroup and the same statement proves that $S$ is local. 
\end{proof}
\begin{rem} \label{rem_gringsum}
To any graded ring $R$, one can always assign a direct sum $R_{\oplus} := \bigoplus_{k \in \Z} R_{k}$. It can be equipped with a structure of an ordinary ring. In fact, the assignment $R \mapsto R_{\oplus}$ defines a functor from $\gRng$ to $\Rng$. The components $R_{k}$ then form subgroups of the abelian group $R_{\oplus}$ and the ring multiplication satisfies $R_{i} \cdot R_{j} \subseteq R_{i+j}$. Most authors define graded rings as ordinary rings that allow such a direct sum decomposition, see e.g. Chapter XVI, \S 6. in \cite{lang2005algebra}. Note that when $R$ is a local graded ring, $R_{\oplus}$ does not need to be a local ring. Interestingly, the converse is true \cite{bergman}.
\end{rem}

\subsection{Graded vector spaces and algebras}
For the purposes of this paper, all vector spaces are assumed to be real. By a \textbf{graded vector space}, we mean a sequence $V = \{ V_{k} \}_{k \in \Z}$, where for each $k \in \Z$, $V_{k}$ is an ordinary vector space. By a \textbf{graded linear map} $\varphi: V \rightarrow W$ of two graded vector spaces $V$ and $W$, we mean a sequence $\varphi = \{ \varphi_{k} \}_{k \in \Z}$, where for each $k \in \Z$, $\varphi_{k}: V_{k} \rightarrow W_{k}$ is a linear map. Graded vector spaces together with graded linear maps form a category $\gVect$. 

A \textbf{graded linear subspace} $W \subseteq V$ is a collection $W = \{ W_{k} \}_{k \in \Z}$, where for each $k \in \Z$, $W_{k} \subseteq V_{k}$ is a linear subspace. 

Let $V,W \in \gVect$ and let $\varphi = \{ \varphi_{k} \}_{k \in \Z}$, where for each $k \in \Z$ and fixed $\ell \in \Z$, $\varphi_{k}: V_{k} \rightarrow W_{k+\ell}$ is a linear map. We say that $\varphi$ is a \textbf{graded linear map of degree $\ell$}. We denote the degree $\ell$ of $\varphi$ as $|\varphi|$. For any $v \in V$, we thus have $|\varphi(v)| = |\varphi| + |v|$. Such mappings form a vector space which we denote as $\Lin_{\ell}(V,W)$. Note that $\gVect(V,W) = \Lin_{0}(V,W)$. We can thus form a graded vector space $\Lin(V,W) := \{ \Lin_{\ell}(V,W) \}_{\ell \in \Z}$. We write $\Lin(V)$ for $\Lin(V,V)$. 

Viewing the field $\R$ as a graded vector space, one can to each $V \in \gVect$ assign its \textbf{graded dual} $V^{\ast} := \Lin(V,\R)$. Note that for each $k \in \Z$, one has $(V^{\ast})_{k} = (V_{-k})^{\ast}$.
\begin{rem} \label{rem_gvectsum}
Similarly to Remark \ref{rem_gringsum}, to any $V \in \gVect$, one may assign a vector space $V_{\oplus} = \bigoplus_{k \in \Z} V_{k}$, thus defining a functor from $\gVect$ to $\Vect$. Most of the literature defines graded vector spaces as vector spaces allowing such a direct sum decomposition. However, note that in general $\Vect(V_{\oplus},W_{\oplus}) \neq \Lin(V,W)_{\oplus}$. In particular, it can happen that $(V_{\oplus})^{\ast} \neq (V^{\ast})_{\oplus}$. It thus seems more natural to work with sequences of objects instead of their direct sums. 
\end{rem}
\begin{definice} \label{def_degreeshifted}
Let $V \in \gVect$ and let $\ell \in \Z$. A \textbf{degree shifted vector space} $V[\ell]$ is a graded vector space defined by $(V[\ell])_{k} := V_{k+\ell}$. 
\end{definice} 
\begin{rem}
Degree shifted vector spaces allow for an equivalent description of graded linear maps of degree $\ell$. Indeed, obviously $\Lin_{\ell}(V,W) \cong \Lin_{0}(V,W[\ell]) = \gVect(V,W[\ell])$. 
\end{rem}

A \textbf{graded subspace} $W \subseteq V$ of $V \in \gVect$ is a collection $W = \{W_{k}\}_{k \in \Z}$ such that $W_{k} \subseteq V_{k}$ is a vector subspace for every $k \in \Z$. For any sequence $\{ V_{\alpha} \}_{\alpha \in I} \subseteq \gVect$, direct sums and products are defined component-wise, that is 
\begin{equation}
(\bigoplus_{\alpha \in I} V_{\alpha})_{k} := \bigoplus_{\alpha \in I} (V_{\alpha})_{k}, \; \; (\prod_{\alpha \in I} V_{\alpha})_{k} := \prod_{\alpha \in I} (V_{\alpha})_{k},
\end{equation}
for all $k \in \Z$. Similarly to the ordinary case, both notions coincide for a finite indexing set $I$. 

To any $V \in \gVect$, we may assign a sequence $\gdim(V) := (\dim(V_{k}))_{k \in \Z}$ called the \textbf{graded dimension} of $V$. We define the \textbf{total dimension} of $V$ as $\dim(V) := \sum_{k \in \Z} \dim(V_{k})$ and say that $V$ is finite-dimensional, if $\dim(V) < \infty$. Note that this requires $\dim(V_{k}) < \infty$ for all $k \in \Z$ and $\dim(V_{k}) \neq 0$ only for finitely many $k \in \Z$. 

Suppose $\dim(V) < \infty$. A tuple $(\xi_{\mu})_{\mu = 1}^{\dim(V)}$ is called a \textbf{total basis} of $V$, if there is a collection $\{I_{k} \}_{k \in \Z}$ of mutually disjoint subsets $I_{k} \subseteq \{1, \dots, \dim(V)\}$ with $\#{I_{k}} = \dim(V_{k})$, such that $(\xi_{\mu})_{\mu \in I_{k}}$ is the basis of $V_{k}$. In particular, for each $k \in \Z$ and $\mu \in I_{k}$, we have $|\xi_{\mu}| = k$. One can always choose a total basis. 

\begin{example} \label{ex_Rtosequence}
Let $(n_{j})_{j \in \Z}$ be any sequence of non-negative integers. By $\R^{(n_{j})}$ we denote a graded vector space $(\R^{(n_{j})})_{k} := \R^{n_{k}}$. Clearly $\gdim(\R^{(n_{j})}) = (n_{j})_{j \in \Z}$ and $\dim(\R^{(n_{j})}) = \sum_{j \in \Z} n_{j}$.  

Let $V \in \gVect$. If $\dim(V_{k}) < \infty$ for all $k \in \Z$, a choice of a basis in each of the components determines an isomorphism $V \cong \R^{\gdim(V)}$. 

Moreover, let $\R^{(n_{j})}_{\ast}$ denote the graded vector space obtained by throwing out the degree zero component of $\R^{(n_{j})}$, that is $(\R^{(n_{j})}_{\ast})_{k} = \R^{n_{k}}$ for $k \neq 0$ and $(\R^{(n_{j})}_{\ast})_{0} = \{0\}$. We can thus write 
\begin{equation} \R^{(n_{j})} = \R^{n_{0}} \oplus \R^{(n_{j})}_{\ast},
\end{equation} where we view $\R^{n_{0}}$ as a (trivially) graded vector space. 
\end{example}

For any $V,W \in \gVect$, one can define their tensor product $V \otimes_{\R} W$ for each $k \in \Z$ by
\begin{equation}
(V \otimes_{\R} W)_{k} := \bigoplus_{i \in \Z} V_{i} \otimes_{\R} W_{k-i}. 
\end{equation}
$\otimes_{\R}$ makes $\gVect$ into a monoidal category with the role of the unit object played by the field $\R$ viewed as a graded vector space. The braiding $\tau$ can be given by the same formula as (\ref{eq_gAbbraiding}), hence making $\gVect$ into a symmetric monoidal category. We can thus consider the following notion.
\begin{definice} \label{def_galgebra}
Let $A$ be a graded vector space together with a graded linear map $\mu: A \otimes_{\R} A \rightarrow A$. Let  $1 \in A_{0}$ be an element called the \textbf{algebra unit of $A$}. For all $a,a' \in A$, we will henceforth use the notation
\begin{equation}
a \cdot a' := \mu(a \otimes a').
\end{equation} 
We say that $(A,\mu,1)$ is a \textbf{graded associative algebra}, if for all $a,a',a'' \in A$, one has the associativity rule
\begin{equation}
a \cdot (a' \cdot a'') = (a \cdot a') \cdot a'',
\end{equation}
and the multiplication with the unit does nothing: $a \cdot 1 = 1 \cdot a = a$, for all $a \in A$. We say that $(A,\mu,1)$ is a \textbf{graded commutative associative algebra}, if for all $a,a' \in A$, one has 
\begin{equation}
a \cdot a' = (-1)^{|a||a'|} a' \cdot a. 
\end{equation}
A graded linear map $\varphi: A \rightarrow B$ (an arrow in $\gVect$) is called a \textbf{graded algebra morphism}, if $\varphi(a \cdot a') = \varphi(a) \cdot \varphi(a')$ and $\varphi(1) = 1$, for all $a,a' \in A$. Graded associative algebras together with graded algebra morphisms form a category which we denote as $\gAs$, and graded commutative associative algebras form its full subcategory $\gcAs$. 
\end{definice}
Note that for a graded (commutative) associative algebra $A$, the component $A_{0}$ is an ordinary (commutative) algebra. We also have an analogue of Proposition \ref{tvrz_ringsaremonoids}.
\begin{tvrz}
Graded (commutative) associative algebras are (commutative) monoids in the symmetric monoidal category $(\gVect,\otimes_{\R},\tau)$. 
\end{tvrz}
Note that we consider only associative algebras with a unit. Observe that every graded vector space can be viewed as a graded abelian group and every graded (commutative) associative algebra can be naturally interpreted as an example of a graded (commutative) ring. 

Left, right and (two-sided) ideals $J \subseteq A$ are defined similarly to the case of graded abelian groups. If $J$ is a two-sided ideal, there us a unique graded associative algebra structure on the quotient graded vector space $A / J$ making the quotient map $q: A \rightarrow A / J$ into a graded algebra morphism. 
There is an important way how to obtain ideals in graded associative algebras.

\begin{tvrz} \label{tvrz_idealgenerated}
Let $A \in \gAs$ and let $S \subseteq A$ be a graded subset. Then there is the unique smallest ideal $\<S \> \subseteq A$ satisfying $S \subseteq \<S\>$.
\end{tvrz}
\begin{proof}
For every $k \in \Z$, define the subspace $\<S\>_{k}$ as 
\begin{equation}
\<S\>_{k} := \R\{ a \cdot s \cdot b \; | \; a,b \in A, \; s \in S, \; \; |a| + |s| + |b| = k \},
\end{equation}
where $\R$ indicates the linear hull. It is easily checked to be an ideal in $A$ containing $S$. Note that this is true only for associative algebras with a unit. Obviously, every ideal $J$ satisfying $S \subseteq J$ must contain $\<S\>$, hence $\<S\>$ is the smallest of such ideals.
\end{proof}
\begin{rem} \label{rem_tensorproductofgcAs} Let us make some important observations.
\begin{enumerate}[(i)]
\item Let $A,B \in \gcAs$. One can make their tensor product $A \otimes_{\R} B$ into a graded commutative associative algebra. Indeed, set $1 := 1_{A} \otimes 1_{B}$ and 
\begin{equation}
\mu( (a \otimes b) \otimes (a' \otimes b')) := (-1)^{|b||a'|} \mu_{A}(a \otimes a') \otimes \mu_{B}(b \otimes b'),
\end{equation}
for all $a,a' \in A$ and $b,b' \in B$. It is an easy check that $(A \otimes_{\R} B, \mu, 1)$ forms a graded commutative associative algebra. In fact, observe that the braiding $\tau_{AB}: A \otimes_{\R} B \rightarrow B \otimes_{\R} A$ becomes a graded algebra morphism, so $\otimes_{\R}$ and $\tau$ make $\gcAs$ into a symmetric monoidal category. Note that the unit object $\R$ is a graded commutative associative algebra with the product given by the ordinary multiplication of real numbers. 

\item Suppose that $\{ A_{\alpha} \}_{\alpha \in I}$ is a family of graded commutative associative algebras. One can make the direct product $A := \prod_{\alpha \in I} A_{\alpha}$ into a graded commutative algebra. Namely, if $a = (a_{\alpha})_{\alpha \in I}$ and $b = (b_{\alpha})_{\alpha \in I}$ are two elements of $A$, we define
\begin{equation}
a \cdot b := ( a_{\alpha} \cdot b_{\alpha})_{\alpha \in I}.
\end{equation}
Clearly $|a \cdot b| = |a| + |b|$ and the unit element of $A$ is $1 = (1_{\alpha})_{\alpha \in I}$, where $1_{\alpha} \in A_{\alpha}$ is the algebra unit of $A_{\alpha}$, for all $\alpha \in I$. Equivalently, for each $\alpha \in I$ let $\mu_{\alpha}: A_{\alpha} \otimes_{\R} A_{\alpha} \rightarrow A_{\alpha}$ denote the multiplication and let $p_{\alpha}: A \rightarrow A_{\alpha}$ denote the canonical projection. Then the multiplication $\mu: A \otimes_{\R} A \rightarrow A$ is uniquely determined by the equation
\begin{equation}
p_{\alpha} \otimes \mu := \mu_{\alpha} \circ (p_{\alpha} \otimes p_{\alpha}), 
\end{equation}
for all $\alpha \in I$. One can easily argue that $A$ has the universal property of a product over the set $I$, that is for any $B \in \gcAs$ and any collection $\{ \varphi_{\alpha} \}_{\alpha \in I}$ of graded algebra morphisms $\varphi_{\alpha} \in \gcAs(B,A_{\alpha})$, there is a unique graded algebra morphism $\varphi: B \rightarrow A$ such that $p_{\alpha} \circ \varphi = \varphi_{\alpha}$ for all $\alpha \in I$. 
\item Every $A \in \gcAs$ may be viewed as a graded commutative ring. Formally, if $\mu: A \otimes_{\R} A \rightarrow A$ is its multiplication, we obtain a graded ring multiplication $\hat{\mu}: A \otimes_{\Z} A \rightarrow A$ by composing $\mu$ with the canonical morphism $A \otimes_{\Z} A \rightarrow A \otimes_{\R} A$ of graded abelian groups. Informally, one simply ``forgets'' the scalar multiplication and the corresponding homogeneity of $\mu$. 
\end{enumerate}
\end{rem}
\begin{tvrz} \label{tvrz_idealpower}
Let $A \in \gcAs$ and let $J \subseteq A$ be its ideal. For any $r \in \N$, define a graded subset $J^{r} = \{ J^{r}_{k} \}_{k \in \Z}$ of $A$, where we set
\begin{equation}
J^{r}_{k} := \R \{ a \in A_{k} \; | \; a = a_{1} \cdots a_{r}, \; a_{i} \in J \text{ for all } i \in \{1,\dots,r\} \}.
\end{equation}
Then $J^{r} \subseteq A$ is an ideal and $J^{r} \subseteq J^{s}$ whenever $r \geq s$. 
\end{tvrz}

\subsection{Tensor and symmetric algebra}
In this subsection, we will construct two graded associative algebras associated canonically to every graded vector space. In fact, we will obtain two functors $T: \gVect \rightarrow \gAs$ and $S: \gVect \rightarrow \gcAs$, respectively. For the needs of this paper, we will then modify this construction to obtain a so called extended symmetric algebra, a canonical functor $\bar{S}: \gVect \times \gcAs \rightarrow \gcAs$. 

Let $V \in \gVect$. For $p \in \N$, let $T^{p}(V) := V \otimes_{\R} \dots \otimes_{\R} V$ be a $p$-fold tensor product of $V$ with itself. For each $k \in \Z$, one thus has
\begin{equation}
(T^{p}(V))_{k} = \hspace{-3mm} \bigoplus_{k_{1} + \dots + k_{p} = k} \hspace{-3mm} V_{k_{1}} \otimes_{\R} \dots \otimes_{\R} V_{k_{p}}.
\end{equation}
As a convention, define $T^{0}(V) := \R$. 

For $V,X \in \gVect$ and $p \in \N$, by a \textbf{graded $p$-linear map} from $V$ to $X$, we mean a collection $\beta = \{ \beta_{k_{1} \dots k_{p}} \}_{(k_{1},\dots,k_{p}) \in \Z^{p}}$, where $\beta_{k_{1} \dots k_{p}}: V_{k_{1}} \times \dots \times V_{k_{p}} \rightarrow X_{k_{1} + \dots + k_{p}}$ is an ordinary $p$-linear map. The space of $0$-linear maps is by definition identified with the vector space $X_{0}$. We will write $\beta: V \times \dots \times V \rightarrow X$ and for $v_{1},\dots,v_{p} \in V$ use just $\beta(v_{1},\dots,v_{p})$ instead of $\beta_{|v_{1}| \dots |v_{p}|}(v_{1},\dots,v_{p})$. There is a canonical $p$-linear map $\alpha$ from $V$ to $T^{p}(V)$ given by 
\begin{equation}
\alpha(v_{1}, \dots, v_{p}) := v_{1} \otimes \dots \otimes v_{p},
\end{equation}
for any $p \in \N$ and $v_{1},\dots,v_{p} \in V$. For $p = 0$, $\alpha$ is set to correspond to $1 \in (T^{0}(V))_{0} = \R$. $T^{p}(V)$ now allows one to interpret graded $p$-linear maps as graded linear maps:

\begin{tvrz} \label{tvrz_plinearaslinear}
Let $V,X \in \gVect$ and $p \in \N_{0}$. For any graded $p$-linear map $\beta$ from $V$ to $X$, there exists a unique graded linear map $\hat{\beta}: T^{p}(V) \rightarrow X$, such that $\hat{\beta} \circ \alpha = \beta$, where $\alpha$ is the canonical $p$-linear map from $V$ to $T^{p}(V)$ defined above. 
\end{tvrz}

For each $p,q \in \N_{0}$, one can now consider a graded linear map $\otimes_{p,q}: T^{p}(V) \otimes_{\R} T^{q}(V) \rightarrow T^{p+q}(V)$ for all $v_{1},\dots,v_{p},w_{1},\dots,w_{q} \in V$ by the formula
\begin{equation}
\otimes_{p,q}( (v_{1} \otimes \dots \otimes v_{p}) \otimes (w_{1} \otimes \dots \otimes w_{q})) := v_{1} \otimes \dots \otimes v_{p} \otimes w_{1} \otimes \dots \otimes w_{q},
\end{equation}
It is an easy check that it is well-defined. It follows that this collection can be used to construct a graded linear map $\otimes: T(V) \otimes_{\R} T(V) \rightarrow T(V)$, where 
\begin{equation}
T(V) = \bigoplus_{p = 0}^{\infty} T^{p}(V).
\end{equation}
Let $1 \in T(V)_{0}$ be the number $1$ viewed as an element of $(T^{0}(V))_{0} = \R$. It follows that $(T(V),\otimes,1)$ becomes a graded associative algebra called the \textbf{tensor algebra of the graded vector space $V$}. Up to an isomorphism, $T(V)$ can be uniquely characterized by the following universal property:
\begin{tvrz} \label{tvrz_TVuniversal}
Suppose $V \in \gVect$ and let $i_{V}: V \rightarrow T(V)$ denote the inclusion of $V$ as a graded subspace $T^{1}(V) \subseteq T(V)$. Let $A \in \gAs$ be any graded associative algebra together with a graded linear map $\varphi: V \rightarrow A$. Then there is a unique graded algebra morphism $\hat{\varphi}: T(V) \rightarrow A$ fitting into the commutative diagram
\begin{equation}
\begin{tikzcd}
V \arrow{r}{i_{V}} \arrow{rd}{\varphi} & T(V) \arrow[dashed]{d}{\hat{\varphi}} \\
& A 
\end{tikzcd}.
\end{equation}
The assignment $V \mapsto T(V)$ defines a functor $T: \gVect \rightarrow \gAs$ and the universal property can be viewed as the adjunction
\begin{equation}
\gAs( T(V), A) \cong \gVect(V, \square(A)),
\end{equation}
where $\square: \gAs \rightarrow \gVect$ is the obvious forgetful functor.
\end{tvrz}
\begin{proof}
The proof is completely analogous to the one for ordinary vector spaces and we leave it for the reader. Note that for every graded linear map $\phi: V \rightarrow W$, the induced graded algebra morphism $T(\phi): T(V) \rightarrow T(W)$ can be constructed using the universal property by taking $A = T(W)$ and $\varphi = i_{W} \circ \phi$. 
\end{proof}

In general, $T(V)$ is not graded commutative. This can be fixed by considering a different graded algebra. First, define a graded set $M \subseteq T(V)$ for each $k \in \Z$ as 
\begin{equation} \label{eq_Mkset}
M_{k} := \{ v \otimes w - (-1)^{|v||w|} w \otimes v \; | \; v,w \in V, \; |v| + |w| = k \}.
\end{equation}
By Proposition \ref{tvrz_idealgenerated}, we may then consider an ideal $\< M \> \subseteq T(V)$ generated by this set and thus endow the corresponding quotient
\begin{equation}
S(V) := T(V) / \< M \>.
\end{equation}
with a structure of a graded associative algebra, called the \textbf{symmetric algebra of the graded vector space $V$}. Note that we may define $\<M\>^{p} := \<M\> \cap T^{p}(V)$ and write 
\begin{equation}
S(V) = \bigoplus_{p=0}^{\infty} S^{p}(V), \text{ where } S^{p}(V) = T^{p}(V) / \<M\>^{p}.
\end{equation}
Note that $\<M\>^{0} = \<M\>^{1} = 0$, which shows that $S^{0}(V) = \R$ and $S^{1}(V) = V$, so that we have the inclusion $j_{V}: V \rightarrow S(V)$. It follows that for each $k \in \Z$ and $p \in \N$, $(S^{p}(V))_{k}$ is generated by elements of the form 
\begin{equation}
v_{1} \dots v_{p} := q( v_{1} \otimes \dots \otimes v_{p}), 
\end{equation}
where $v_{1},\dots,v_{p} \in V$ such that $|v_{1}| + \dots + |v_{p}| = k$ and $q: T(V) \rightarrow S(V)$ denotes the quotient map. Besides the linearity in each $v_{i}$ inherited from the generators of $T^{p}(V)$, they are subject to relations following from the definition of $\<M\>$. For each $j \in \{1,\dots,p-1\}$, one has 
\begin{equation} \label{eq_gradedsymmetry}
v_{1} \dots v_{j}v_{j+1} \dots v_{p} = (-1)^{|v_{j}||v_{j+1}|} v_{1} \dots v_{j+1} v_{j} \dots v_{p}.
\end{equation}
As $q$ is declared to be a graded algebra morphism, it follows that the multiplication reads
\begin{equation}
(v_{1} \dots v_{p}) \cdot (w_{1} \dots w_{q}) = v_{1} \dots v_{p} w_{1} \dots w_{q},
\end{equation}
for all $v_{1},\dots,v_{p},w_{1},\dots,w_{q} \in V$. The algebra unit is still $1 \in (S^{0}(V))_{0} = \R$. Using (\ref{eq_gradedsymmetry}), it is now easy to prove that $S(V)$ is a  graded commutative associative algebra. 

Now, a graded $p$-linear map $\beta$ from $V$ to $X$ is called \textbf{completely symmetric}, if for all $v_{1},\dots,v_{p} \in V$ and $j \in \{1, \dots, p-1\}$, one has 
\begin{equation}
\beta(v_{1},\dots,v_{j},v_{j+1}, \dots, v_{p}) = (-1)^{|v_{j}||v_{j+1}|} \beta(v_{1}, \dots, v_{j+1},v_{j}, \dots, v_{p}).
\end{equation}
Every graded $0$-linear map (an element of $X_{0}$) is declared completely symmetric. For each $p \in \N_{0}$, there is a canonical completely symmetric graded $p$-linear map $\alpha$ from $V$ to $S^{p}(V)$ given for each $p \in \N$ by the formula
\begin{equation}
\alpha(v_{1},\dots,v_{p}) := v_{1} \dots v_{p},
\end{equation}
for all $v_{1},\dots,v_{p} \in V$ and for $p = 0$ identified with the unit $1 \in S(V)_{0}$. There are straightforward analogues of Proposition \ref{tvrz_plinearaslinear} and Proposition \ref{tvrz_TVuniversal} for the symmetric algebra:
\begin{tvrz} \label{tvrz_skewplinearaslinear}
Let $V,X \in \gVect$ and $p \in \N_{0}$. For any completely symmetric graded $p$-linear map $\beta$ from $V$ to $X$, there exists a unique graded linear map $\hat{\beta}: S^{p}(V) \rightarrow X$, such that $\hat{\beta} \circ \alpha = \beta$, where $\alpha$ is the canonical $p$-linear map from $V$ to $S^{p}(V)$ defined above. 
\end{tvrz}
\begin{tvrz} \label{tvrz_SVuniversal}
Suppose $V \in \gVect$ and let $j_{V}: V \rightarrow T(V)$ denote the inclusion of $V$ as a graded subspace $S^{1}(V) \subseteq S(V)$. Let $A \in \gcAs$ be any graded commutative associative algebra together with a graded linear map $\varphi: V \rightarrow A$. Then there is a unique graded algebra morphism $\hat{\varphi}: S(V) \rightarrow A$ fitting into the commutative diagram
\begin{equation}
\begin{tikzcd}
V \arrow{r}{j_{V}} \arrow{rd}{\varphi} & S(V) \arrow[dashed]{d}{\hat{\varphi}} \\
& A 
\end{tikzcd}.
\end{equation}
The assignment $V \mapsto S(V)$ defines a functor $S: \gVect \rightarrow \gcAs$ and the universal property can be viewed the adjunction
\begin{equation}
\gAs( S(V), A) \cong \gVect(V, \square(A)),
\end{equation}
where $\square: \gcAs \rightarrow \gVect$ is the obvious forgetful functor.
\end{tvrz}
In the following, we will need a more complicated example of a graded commutative associative algebra. Let $V \in \gVect$ and $A \in \gcAs$ be arbitrary. Define a graded vector space 
\begin{equation}
\bar{S}(V,A) := \prod_{p = 0}^{\infty} A \otimes_{\R} S^{p}(V).
\end{equation}
\begin{rem}
The presence and the position of the direct product instead of a direct sum is essential as typically $S^{p}(V) \neq 0$ for all $p \in \N_{0}$ and the direct product and the tensor product are in general not distributive, that is 
\begin{equation}
\prod_{p = 0}^{\infty} A \otimes_{\R} S^{p}(V) \neq A \otimes_{\R} \prod_{p=0}^{\infty} S^{p}(V).
\end{equation}
\end{rem}
To make $\bar{S}(V,A)$ into a graded associative algebra, for each $p,q \in \N_{0}$ define 
\begin{equation} \label{eq_mupqproduct}
\mu_{p,q}( (a \otimes \sigma) \otimes (a' \otimes \sigma')) := (-1)^{|\sigma||a'|} (a \cdot a') \otimes (\sigma \cdot \sigma') \in A \otimes_{\R} S^{p+q}(V),
\end{equation}
for all $a,a' \in A$, $\sigma \in S^{p}(V)$ and $\sigma' \in S^{q}(V)$, and extend it to a graded linear map 
\begin{equation} \mu_{p,q}: (A \otimes_{\R} S^{p}(V)) \otimes_{\R} (A \otimes_{\R} S^{q}(V)) \rightarrow A \otimes_{\R} S^{p+q}(V).
\end{equation}
Now, each element of $\bar{S}(V,A)$ is a sequence $x = (x_{p})_{p=0}^{\infty}$, where $x_{p} \in A \otimes_{\R} S^{p}(V)$. Note that $|x| = |x_{p}|$ for all $p \in \N_{0}$. The graded linear map $\mu: \bar{S}(V,A) \otimes_{\R} \bar{S}(V,A) \rightarrow \bar{S}(V,A)$ is defined as
\begin{equation}
\mu(x \otimes y) := ( \sum_{q=0}^{p} \mu_{q,p-q}( x_{q} \otimes y_{p-q}))_{p=0}^{\infty},
\end{equation}
for all $x,y \in \bar{S}(V,A)$. The algebra unit is $1 := 1 \otimes 1 \in A_{0} \otimes (S^{0}(V))_{0} \subseteq \bar{S}(V,A)_{0}$. It is now a straightforward exercise to prove that $(\bar{S}(V,A),\mu, 1)$ forms a graded commutative associative algebra called the \textbf{extended symmetric algebra of $V$ with coefficients in $A$}.

\begin{rem} \label{rem_simplerextsymmetric}
There are notable cases where $\bar{S}(V,A)$ coincides with a simpler graded commutative associative algebra $S(V,A)$, defined as
\begin{equation}
S(V,A) := A \otimes_{\R} S(V) = \bigoplus_{p = 0}^{\infty} A \otimes_{\R} S^{p}(V),
\end{equation}
see Remark \ref{rem_tensorproductofgcAs}-(i). This happens for example if $A$ is non-negatively graded ($A_{k} = 0$ for $k < 0$) and $V$ is strictly positively graded ($V_{k} = 0$  for $k \leq 0$). This forces $\bar{S}(V,A)_{k} = 0$ for all $k < 0$ and for any given $k \in \N_{0}$, one has $(A \otimes_{\R} S^{p}(V))_{k} = 0$ whenever $p > k$, whence we obtain
\begin{equation}
\bar{S}(V,A)_{k} = \prod_{p=0}^{k} (A \otimes_{\R} S^{p}(V))_{k} = \bigoplus_{p=0}^{k} (A \otimes_{\R} S^{p}(V))_{k} = S(V,A)_{k}.
\end{equation}
The fact that also the graded algebra structures coincide is a direct verification. The same statement holds also for $A$ non-positively graded with $V$ strictly negatively graded. 
\end{rem}

\begin{tvrz} \label{tvrz_SVAisfunctor}
The assignment $(V,A) \mapsto \bar{S}(V,A)$ defines a functor $\bar{S}: \gVect \times \gcAs \rightarrow \gcAs$. 
\end{tvrz}
\begin{proof}
Let $\varphi \in \gVect(V,W)$ and $\psi \in \gcAs(A,B)$ for $V,W \in \gVect$ and $A,B \in \gcAs$. We have to define a graded algebra morphism $\bar{S}(\varphi,\psi): \bar{S}(V,A) \rightarrow \bar{S}(W,B)$. By Proposition \ref{tvrz_SVuniversal}, we have the induced graded algebra morphism $S(\varphi): S(V) \rightarrow S(W)$. It follows from its construction that for each $p \in \N_{0}$, it restricts to a graded linear morphism $S^{p}(\varphi): S^{p}(V) \rightarrow S^{p}(W)$. We define
\begin{equation}
\bar{S}(\varphi,\psi) := \prod_{p = 0}^{\infty} \psi \otimes S^{p}(\varphi).
\end{equation}
It is straightforward to check that this is a graded algebra morphism and the assignment $(\varphi,\psi) \mapsto \bar{S}(\varphi,\psi)$ makes $\bar{S}$ into a functor.  
\end{proof}
Now, suppose that $V$ is finite-dimensional with a fixed total basis $(\xi_{\mu})_{\mu=1}^{n}$, where $n := \dim(V)$. To each $k \in \Z$, we assign the following set:
\begin{equation}
\N^{n}_{k} := \{ (p_{1},\dots,p_{n}) \in (\N_{0})^{n} \; | \; \sum_{\mu=1}^{n} p_{\mu} |\xi_{\mu}| = k, \; \; p_{\mu} \in \{0,1\} \text{ whenever } |\xi_{\mu}| \text{ is odd} \}.
\end{equation}
Note that for a general $V$, this set is infinite. It can be written as a disjoint union $\N^{n}_{k} = \sqcup_{p=0}^{\infty} \N^{n}_{k}(p)$ of finite sets $\N^{n}_{k}(p)$ defined by
\begin{equation} \label{eq_Nnkpset}
\N^{n}_{k}(p) = \{ (p_{1},\dots,p_{n}) \in \N^{n}_{k} \; | \; w(p_{1},\dots,p_{n}) := \sum_{\mu=1}^{n} p_{\mu} = p \}.
\end{equation}
Finally, let $\N^{n}(p) = \cup_{k \in \Z} \N^{n}_{k}(p)$ and note that this also a finite subset. We will use the shorthand notation $\fp = (p_{1},\dots,p_{n})$. We assume that all these sets inherit a lexicographical ordering from $(\N_{0})^{n}$, but this choice is not really important in anything what follows. For each $\fp \in (\N_{0})^{n}$, write
\begin{equation}
\xi^{\fp} := (\xi_{1})^{p_{1}} \dots (\xi_{n})^{p_{n}}, 
\end{equation}
where $(\xi_{\mu})^{0} := 1$ and for each $q > 0$ one sets \begin{equation}
(\xi_{\mu})^{q} := \underbrace{\xi_{\mu} \dots \xi_{\mu}}_{q \text{ times}}.
\end{equation}

\begin{tvrz} \label{tvrz_totalbasisSV}
Let $V \in \gVect$ with $n := \dim(V) < \infty$. Suppose $(\xi_{\mu})_{\mu=1}^{n}$ is its fixed total basis. 

Then for each $p \in \N_{0}$ and $k \in \Z$, the collection $( \xi^{\fp} )_{\fp \in \N^{n}_{k}(p)}$ forms a basis of $(S^{p}(V))_{k}$. Consequently, $( \xi^{\fp} )_{\fp \in \N^{n}(p)}$ forms a total basis of the graded vector space $S^{p}(V)$ for each $p \in \N_{0}$. In particular, $S^{p}(V)$ is a finite-dimensional graded vector space. 
\end{tvrz}
\begin{proof}
For $\fp \in \N^{n}_{k}(p)$, the element $\xi^{\fp}$ lies in $S^{p}(V)$ as the sum of powers is assumed to be $p$. Moreover, one has 
\begin{equation}
|\xi^{\fp}| = | (\xi_{1})^{p_{1}} \dots (\xi_{n})^{p_{n}}| = \sum_{\mu = 1}^{n} p_{\mu} |\xi_{\mu}| = k,
\end{equation}
hence it lies in the vector space $(S^{p}(V))_{k}$. The restriction on the values of $p_{\mu}$ for odd $|\xi_{\mu}|$ ensures that the corresponding monomial $\xi^{\fp}$ does not vanish. It is now straightforward to see that they actually form a basis. This finishes the proof.
\end{proof}
\begin{example} \label{ex_formalpower}
Let us now examine the graded commutative associative algebra $\bar{S}(V,A)$ for a finite-dimensional $V$ and suppose $A$ is an ordinary commutative associative algebra, that is $A_{k} = 0$ for $k \neq 0$. Let $(\xi_{\mu})_{\mu=1}^{n}$ be a fixed total basis for $V$, where $n = \dim(V)$. 

We claim that elements $\bar{S}(V,A)$ can be naturally interpreted as formal power series in $n$ variables $(\xi_{\mu})_{\mu=1}^{n}$ with coefficients in the algebra $A$. The variables are assumed to commute according to the rule 
\begin{equation} \label{eq_variablesrelation}
\xi_{\mu} \xi_{\nu} = (-1)^{|\xi_{\mu}||\xi_{\nu}|} \xi_{\nu} \xi_{\mu}.
\end{equation}
Every $x \in \bar{S}(V,A)_{k}$ can be written as a formal sum
\begin{equation} \label{eq_formalsum}
x = \sum_{\fp \in \N^{n}_{k}} x_{\fp} \xi^{\fp},
\end{equation}
for the unique sequence $(x_{\fp})_{\fp \in \N^{n}_{k}}$, where $x_{\fp} \in A$. Indeed, $x$ was originally a sequence $x = (x_{q})_{q = 0}^{\infty}$, where $x_{q} \in (A \otimes_{\R} S^{q}(V))_{k} = A \otimes_{\R} (S^{q}(V))_{k}$. By Proposition \ref{tvrz_totalbasisSV}, one can write each $x_{q}$ as 
\begin{equation}
x_{q} = \sum_{\fp \in \N^{n}_{k}(q)} x_{\fp} \otimes \xi^{\fp},
\end{equation} 
for the unique (finite) sequence $( x_{\fp} )_{\fp \in \N^{n}_{k}(q)}$ of elements of $A$. Gathering them into a single sequence, we can form the formal sum (\ref{eq_formalsum}). 

Next, observe that for any $\fp \in \N^{n}_{k}$ and $\fq \in \N^{n}_{\ell}$, such that $\fp + \fq \in \N^{n}_{k+\ell}$, there is a sign $\epsilon^{\fp,\fq}$ determined uniquely by the equation
\begin{equation} 
\xi^{\fp} \xi^{\fq} = \epsilon^{\fp, \fq} \xi^{\fp + \fq},
\end{equation}
It can happen that $\fp + \fq \notin \N^{n}_{k+\ell}$. This is the case precisely when there is $\mu \in \{1,\dots,n\}$ such that $|\xi_{\mu}|$ is odd and $p_{\mu} + q_{\mu} = 2$. But then $\xi_{\mu}^{2} = 0$ and thus $\xi^{\fp + \fq} = 0$. In such a scenario, we declare $\epsilon^{\fp, \fq} := 0$. In fact, it is not difficult to obtain $\epsilon^{\fp,\fq}$ explicitly from (\ref{eq_variablesrelation}), finding
\begin{equation} \label{eq_epsilonpqexplicit}
\epsilon^{\fp, \fq} = (-1)^{\sum_{\mu=1}^{n-1} \{  q_{\mu} |\xi_{\mu}| \cdot \sum_{\nu=\mu+1}^{n} p_{\nu} |\xi_{\nu}| \}},
\end{equation}
whenever $\fp + \fq \in \N^{n}_{k+\ell}$. However, we will seldom need this formula in what follows. Let 
\begin{equation}
x = \sum_{\fp \in \N^{n}_{k}} x_{\fp} \xi^{\fp}, \; \; y = \sum_{\fp \in \N^{n}_{\ell}} y_{\fp} \xi^{\fp}
\end{equation}
be two formal power series. Their product is defined as one would expect for the product of formal power series, that is 
\begin{equation} \label{eq_formalseriesproduct}
x \cdot y = \sum_{\fp \in \N^{n}_{k+\ell}} (x \cdot y)_{\fp} \xi^{\fp}, \text{ where } (x \cdot y)_{\fp} = \sum_{\substack{\fq \in \N^{n}_{k} \\ \fq \leq \fp}} \epsilon^{\fq,\fp-\fq}  x_{\fq} \cdot y_{\fp - \fq},
\end{equation}
where $\fq \leq \fp$, iff $q_{\mu} \leq p_{\mu}$ for all $\mu \in \{1, \dots, n\}$. Note that for $\fp \in \N^{n}_{k+\ell}$ and $\fq \in \N^{n}_{k}$ such that $\fq \leq \fp$, one always has $\fp - \fq \in \N^{n}_{\ell}$. It is not difficult to see that this product coincides with the one of $\bar{S}(V,A)$. In the rest of the paper, we will exclusively use this formalism. 
\end{example}
\begin{rem}
Note that $\bar{S}(V,A)$ is our main motivation to \textit{not work} with graded vector spaces as direct sums of its homogeneous subspaces. Indeed, suppose $V \in \gVect$ and let $V_{\oplus} = \bigoplus_{k \in \Z} V_{k}$. In the literature (see e.g. one of the appendices in \cite{eisenbud2013commutative}), the \textbf{graded symmetric algebra} $S_{\Z}(V_{\oplus})$ is usually defined as the quotient of the ordinary tensor algebra $T(V_{\oplus})$ by the ideal generated by the set $\cup_{k \in \Z} M_{k} \subseteq T(V_{\oplus})$, see (\ref{eq_Mkset}). For an associative commutative algebra $A \in \cAs$, we would then need to consider the associative algebra
\begin{equation}
\bar{S}_{\Z}(V_{\oplus},A) := \prod_{p = 0}^{\infty} A \otimes_{\R} S^{p}_{\Z}(V_{\oplus})
\end{equation}
and write it as a direct sum of its homogeneous subspaces. But this does not work. 
\end{rem}
\subsection{Graded modules and derivations} \label{subsec_gmods}
Let $A \in \gcAs$ be a graded associative commutative algebra. Let $V \in \gVect$ and let $\lambda: A \otimes_{\R} V \rightarrow V$ be a graded linear map, such that
\begin{equation} \label{eq_moduleaxioms}
(a \cdot a') \tr v = a \tr (a' \tr v), \; \; 1 \triangleright v = v,
\end{equation} 
for all $v \in V$ and $a,a' \in A$, where we use the shorthand notation $a \tr v := \lambda(a \otimes v)$. Then $(V,\lambda)$ is called a \textbf{graded $A$-module}. For each $v \in V$ and $a \in A$, let us also write $v \tl a := (-1)^{|a||v|} a \tr v$. It follows that $v \tl (a \cdot a') = (v \tl a) \tl a'$ and $v \tl 1 = v$ for all $v \in V$ and $a,a' \in A$. 

\begin{example} \label{ex_AisAmodule}
Every $A \in \gcAs$ is also a graded $A$-module, if we set $a \tr a' := a \cdot a'$ for all $a,a' \in A$. 
\end{example}

\begin{rem} \label{rem_shiftedmodule}
Let $\ell \in \Z$ be fixed and suppose $V \in \gVect$ is a graded $A$-module. Let us equip the degree shifted vector space $V[\ell]$ with a graded $A$-module structure. 
Observe that there is a canonical graded linear map $\delta[\ell]: V[\ell] \rightarrow V$ of degree $\ell$, where $(\delta[\ell])_{k} := \1_{V_{k+\ell}}: (V[\ell])_{k} \rightarrow V_{k+\ell}$ for each $k \in \Z$. $\delta[\ell]$ is called the \textbf{degree shifting operator}. Now, for each $a \in A$ and $v \in V[\ell]$, define $a \tr' v$ to fit into the equation
\begin{equation} \label{eq_shiftedmodule}
\delta[\ell]( a \tr' v) := (-1)^{|a|\ell} a \tr \delta[\ell](v).
\end{equation}
The sign seems to be unnecessary and quite strange. However, it is in accordance with the ``Koszul convention'' as we are swapping the positions of objects with non-zero degrees $|\delta[\ell]| = \ell$ and $|a|$. It is sometimes a difficult decision \textit{where} to put the signs - but they always appear somewhere. It is not difficult to verify that $\tr'$ makes $V[\ell]$ into a graded $A$-module. 
\end{rem}

\begin{definice} \label{def_Alinearmaps}
Let $V,W \in \gVect$ be two graded $A$-modules. Let $\ell \in \Z$. We say that $\varphi \in \Lin_{\ell}(V,W)$ is \textbf{graded $A$-linear of degree $\ell$}, if for every $a \in A$ and $v \in V$, one has
\begin{equation} \label{eq_Alinearity}
\varphi(a \tr v) = (-1)^{|a|\ell} a \tr \varphi(v). 
\end{equation}
The space of \textbf{graded $A$-linear maps of degree $\ell$} is denoted as $\Lin^{A}_{\ell}(V,W)$. Altogether, these maps form a graded subspace $\Lin^{A}(V,W) \subseteq \Lin(V,W)$. Graded $A$-modules together with graded $A$-linear maps of degree $0$ form a subcategory $\AgMod \subseteq \gVect$. 
\end{definice}

Note that the sign in (\ref{eq_Alinearity}) is important for the following proposition:
\begin{tvrz} \label{tvrz_LinAmodule}
Let $V,W \in \AgMod$. Then there is a canonical graded $A$-module structure on $\Lin^{A}(V,W)$. In particular, for any $V \in \AgMod$, there is its \textbf{dual $A$-module} $V^{\ast} := \Lin^{A}(V,A)$.
\end{tvrz}
\begin{proof}
For each $\varphi \in \Lin^{A}(V,W)$, $a \in A$ and $v \in V$, define $(a \tr \varphi)(v) := a \tr \varphi(v)$. The rest is a straightforward verification. 
\end{proof}

\begin{rem} \label{rem_LinAofshifted}
For every $\ell \in \Z$, and any $V,W \in \AgMod$, there are obvious canonical graded $A$-module isomorphisms
\begin{equation}
\Lin^{A}(V,W[\ell]) \cong \Lin^{A}(V[-\ell],W) \cong \Lin^{A}(V,W)[\ell].
\end{equation}
In particular, for each $k \in \Z$, there are vector space isomorphisms 
\begin{equation}
\Lin_{k}^{A}(V,W[\ell]) \cong \Lin_{k}^{A}(V[-\ell],W) \cong \Lin^{A}_{k+\ell}(V,W).
\end{equation}
Note that graded $A$-module structures on degree shifted graded vector spaces are those described in Remark \ref{rem_shiftedmodule}.
\end{rem}

\begin{definice}
Let $A \in \gcAs$ and let $V$ be a graded $A$-module. Let $\varphi \in \Lin_{\ell}(A,V)$ be a graded linear map of degree $\ell \in \Z$. We say that $\varphi$ is a \textbf{graded derivation of degree $\ell$}, if 
\begin{equation} \label{eq_derlleibniz}
\varphi(a \cdot a') = \varphi(a) \tl a' + (-1)^{\ell |a|} a \tr \varphi(a'),
\end{equation}
for all $a,a' \in A$. Graded derivations of degree $\ell$ form a subspace $\gDer_{\ell}(A,V) \subseteq \Lin_{\ell}(A,V)$, whence $\gDer(A,V) = \{ \gDer_{\ell}(A,V) \}_{\ell \in \Z}$ forms a graded linear subspace of $\Lin(A,V)$. For $V = A$, see Example \ref{ex_AisAmodule}, we write $\gDer(A,A) = \gDer(A)$. 
\end{definice}
Let us recall some elementary properties of graded derivations. 
\begin{tvrz}\label{tvrz_derivations}
Let $A \in \gcAs$ and let $V \in \AgMod$. Then:
\begin{enumerate}[(i)]
\item For any $\varphi \in \gDer(A,V)$, one has $\varphi(1) = 0$.
\item Let $A' \in \gcAs$ and let $\psi: A \rightarrow A'$ be a graded algebra morphism. Suppose $V$ is also a graded $A'$-module (with the action denoted as $\tr'$), such that $a \tr v = \psi(a) \tr' v$ for all $v \in V$ and $a \in A$. 

Then for any $\varphi' \in \gDer(A',V)$, the composition $\varphi' \circ \psi$ is in $\gDer(A,V)$. 
\item Let $V'$ be another graded $A$-module. Suppose that $\psi \in \Lin^{A}_{0}(V,V')$, see Definition \ref{def_Alinearmaps}. 

Then for any $\varphi \in \gDer(A,V)$, the composition $\psi \circ \varphi$ is in $\gDer(A,V')$.
\item For any $\varphi,\varphi' \in \gDer(A)$, one can define their \textbf{graded commutator} as 
\begin{equation}
[\varphi,\varphi'] = \varphi \circ \varphi' - (-1)^{|\varphi||\varphi'|} \varphi' \circ \varphi.
\end{equation}
Then $[\varphi,\varphi'] \in \gDer(A)$, and one has $|[\varphi,\varphi']| = |\varphi| + |\varphi'|$. 
\end{enumerate}
\end{tvrz}
\begin{proof}
For any $\varphi \in \gDer(A,V)$ and any $a \in A$, one finds $\varphi(a \cdot 1) = \varphi(a) \tl 1 + (-1)^{|a||\varphi|}a \tr \varphi(1)$. Hence $a \tr \varphi(1) = 0$. By choosing $a = 1$, one obtains the claim $(i)$. The claims $(ii)$ - $(iv)$ are easily verified directly from the respective definitions. 
\end{proof}

\begin{tvrz} \label{tvrz_gDerisAmodule}
Let $V$ be a graded $A$-module. For each $\varphi \in \gDer(A,V)$ and all $a,b \in A$, define
\begin{equation}
(a \tr \varphi)(b) := a \tr \varphi(b),
\end{equation}
Then $a \tr \varphi \in \gDer(V,A)$ and $|a \tr \varphi| = |a| + |\varphi|$. This makes $\gDer(V,A)$ into a graded $A$-module. 
\end{tvrz}
\begin{proof}
This is an easy verification. 
\end{proof}
\section{Graded sheaf theory} \label{sec_grsheaves}
\subsection{Graded presheaves and sheaves}
In this subsection, we will consider mostly presheaves valued in the category $\gcAs$. Completely analogous definitions work for other ``graded categories'' as $\gSet$, $\gVect$ or $\gAb$.

Let $X$ be a topological space. By $\Op(X)$, we denote the category, objects of which are open subsets of $X$. For any $U,V \in \Op(X)$ there is a single arrow $i_{V}^{U}: V \rightarrow U$, iff $V \subseteq U$ and none otherwise. The composition of arrows works since $W \subseteq V$ and $V \subseteq U$ implies $W \subseteq U$ and there are identity arrows as $U \subseteq U$. For any collection $\{ U_{\alpha} \}_{\alpha \in I} \subseteq \Op(X)$, any $p \in \N$ and $\alpha_{1},\dots,\alpha_{p} \in I$, we will henceforth write
\begin{equation}
U_{\alpha_{1} \dots \alpha_{p}} := U_{\alpha_{1}} \cap \dots \cap U_{\alpha_{p}}.
\end{equation}

\begin{definice}
Let $X$ be a topological space. A \textbf{presheaf of graded commutative associative algebras on $X$} is a functor $\F: \Op(X)^{\op} \rightarrow \gcAs$.

In other words, $\F$ is equivalent to the following data:
\begin{enumerate}[(i)]
\item For each $U \in \Op(X)$, one has $\F(U) \in \gcAs$. Elements of $\F(U)$ are usually called \textbf{sections of the presheaf $\F$ over $U$}.
\item For every $U,V \in \Op(X)$ such that $V \subseteq U$, there is a graded algebra morphism $\F^{U}_{V} \equiv \F(i^{U}_{V}): \F(U) \rightarrow \F(V)$, called the \textbf{restriction from $U$ to $V$}. Whenever possible, we will write $s|_{V} := \F^{U}_{V}(s)$ for any $s \in \F(U)$. 
\item For all $U,V,W \in \Op(X)$, such that $W \subseteq V \subseteq U$, one has $\F^{U}_{U} = \1_{\F(U)}$ and $\F^{V}_{W} \circ \F^{U}_{V} = \F^{U}_{W}$. 
\end{enumerate}
\end{definice}
For each $k \in \Z$, we define the \textbf{$k$-th component presheaf} $\F_{k}: \Op(X)^{\op} \rightarrow \Vect$ as 
\begin{equation}
\F_{k}(U) := \F(U)_{k},
\end{equation}
for all $U \in \Op(X)$. Restriction morphisms are linear maps $(\F_{k})^{U}_{V} := (\F^{U}_{V})_{k}$. Let $\F,\G$ be two presheaves of graded commutative associative algebras. A \textbf{presheaf morphism} $\eta: \F \rightarrow \G$ is a natural transformation of the two functors, that is a collection $\eta = \{ \eta_{U} \}_{U \in \Op(X)}$, where $\eta_{U}: \F(U) \rightarrow \G(U)$ is a graded algebra morphism, such that the diagram
\begin{equation}
\begin{tikzcd}
\F(U) \arrow{r}{\eta_{U}} \arrow{d}{\F^{U}_{V}} & \G(U) \arrow{d}{\G^{U}_{V}} \\
\F(V) \arrow{r}{\eta_{V}} & \G(V)
\end{tikzcd}
\end{equation}
commutes for all $U,V \in \Op(X)$ such that $V \subseteq U$. Presheaves of graded commutative associative algebras on $X$ together with presheaf morphisms form a category $\PSh(X, \gcAs)$. 

Although $\gcAs$ is not a concrete category, we can still define sheaves in a usual wording.
\begin{definice} \label{def_sheaf}
Let $\F \in \PSh(X,\gcAs)$. We say that $\F$ is a \textbf{sheaf of graded commutative associative algebras on $X$}, if for any $U \in \Op(X)$ and any its open cover $\{ U_{\alpha} \}_{\alpha \in I}$, it has the following two properties:
\begin{enumerate}[(s1)]
\item Let $s,t \in \F(U)$, such that $s|_{U_{\alpha}} = t|_{U_{\alpha}}$ for all $\alpha \in I$. Then $s = t$. This is called the \textbf{monopresheaf property of $\F$}.
\item Consider a collection $\{ s_{\alpha} \}_{\alpha \in I}$, where $s_{\alpha} \in \F(U_{\alpha})$ for each $\alpha \in I$, such that $s_{\alpha}|_{U_{\alpha \beta}} = s_{\beta}|_{U_{ \alpha \beta}}$ for each $(\alpha,\beta) \in I^{2}$. Then there exists $s \in \F(U)$, such that $s|_{U_{\alpha}} = s_{\alpha}$ for all $\alpha \in I$. By (s1), it is unique. This is called the \textbf{gluing property of $\F$}. 
\end{enumerate}
Sheaves of graded commutative associative algebras on $X$ form a full subcategory $\Sh(X,\gcAs)$ of the category $\PSh(X,\gcAs)$. 
\end{definice}
\begin{rem}
There is one important detail hidden in our conventions, namely that elements involved in (s1) and (s2) necessarily satisfy $|s| = |t|$ and $|s_{\alpha}| = |s_{\beta}|$ for all $(\alpha,\beta) \in I^{2}$. In other words, we may say that $\F$ is a sheaf of graded commutative associative algebras, if its $k$-th component presheaf $\F_{k}$ is a sheaf of vector spaces for all $k \in \Z$. 
\end{rem}
Now, the category $\gcAs$ has products over all indexing sets, see Remark \ref{rem_tensorproductofgcAs}. One can thus define sheaves using equalizers, see e.g. \cite{maclane2012sheaves}. These two definitions agree:
\begin{tvrz}
Let $\F \in \PSh(X,\gcAs)$. For any $U \in \Op(X)$ and any its open cover $\{ U_{\alpha} \}_{\alpha \in I}$, we have a commutative diagram
\begin{equation} \label{eq_sheafdiagram}
\begin{tikzcd}
\F(U) \arrow{r}{\varphi} & \prod_{\alpha \in I} \F(U_{\alpha}) \arrow[shift left=1mm]{r}{\psi_{1}} \arrow[shift right=1mm]{r}[swap]{\psi_{2}} & \prod_{(\alpha,\beta) \in I^{2}} \F(U_{\alpha \beta}),
\end{tikzcd}
\end{equation}
where $\varphi$, $\psi_{1}$ and $\psi_{2}$ are graded algebra morphisms uniquely determined by equations
\begin{equation}
p_{\alpha} \circ \varphi = \F^{U}_{U_{\alpha}}, \; \; p_{\alpha \beta} \circ \psi_{1} = \F^{U_{\alpha}}_{U_{\alpha \beta}} \circ p_{\alpha}, \; \; p_{\alpha \beta} \circ \psi_{2} = \F^{U_{\beta}}_{U_{\alpha \beta}} \circ p_{\beta}, 
\end{equation}
for all $(\alpha,\beta) \in I^{2}$, where $p_{\alpha}$ and $p_{\alpha \beta}$ are the obvious projections.

Then $\F$ is a sheaf, if and only if (\ref{eq_sheafdiagram}) is an equalizer diagram in $\gcAs$ for any $U \in \Op(X)$ and any its open cover $\{ U_{\alpha} \}_{\alpha \in I}$. 
\end{tvrz}
\begin{proof}
Recall that (\ref{eq_sheafdiagram}) is an equalizer diagram if it commutes, and for any $A \in \gcAs$ together with a graded algebra morphism $\chi: A \rightarrow \prod_{\alpha \in I} \F(U_{\alpha})$ satisfying $\psi_{1} \circ \chi = \psi_{2} \circ \chi$, there is a unique graded algebra morphism $\psi: A \rightarrow \F(U)$, such that $\varphi \circ \psi = \chi$. This is called the \textbf{universal property of the equalizer}. It is not difficult to see that the assertion about the uniqueness of $\psi$ corresponds to the monopresheaf property, whereas the statement about its existence is equivalent to the gluing property. 
\end{proof}

\begin{rem} \label{rem_ShgcAs}
\begin{enumerate}[(i)]
\item Sheaves valued in a given category exist only if it has a terminal object. In the case of $\gcAs$, it is the trivial graded algebra $A = 0$. Then choose $U = \emptyset$ and its empty cover, $I = \emptyset$. The empty product is always equal to the terminal object, whence in the diagram (\ref{eq_sheafdiagram}), we obtain a monomorphism $\varphi: \F(\emptyset) \rightarrow 0$. Hence necessarily $\F(\emptyset) = 0$. 

\item On the other hand, suppose $\F \in \PSh(X,\gcAs)$ satisfies $\F(U) \neq 0$ for some $U \in \Op(X)$. Then $\F(V) \neq 0$ for all $V \in \Op(X)$ such that $U \subseteq V$. In particular, for any non-trivial presheaf $\F \in \PSh(X,\gcAs)$, one has $\F(X) \neq 0$. 

Indeed, if $\F(V) = 0$, then the algebra unit $1$ is equal to the zero vector $0 \in \F(V)_{0}$. But as $\F^{V}_{U}$ is a graded algebra morphism, this proves that the algebra unit of $\F(U)$ is equal to the zero vector in $\F(U)_{0}$. This contradicts the assumption $\F(U) \neq 0$. 
\end{enumerate}
\end{rem}

Let $\F \in \PSh(X,\gcAs)$. For any $U \in \Op(X)$, we may restrict $\F$ to the full subcategory $\Op(U) \subseteq \Op(X)$ to obtain a \textbf{restriction of $\F$ to $U$}, denoted as $\F|_{U} \in \PSh(U,\gcAs)$. If $\F$ is a sheaf, then so is $\F|_{U}$. For any presheaf morphism $\eta: \F \rightarrow \G$, we obtain the obvious induced presheaf morphism $\eta|_{U}: \F|_{U} \rightarrow \G|_{U}$. There is one immediate consequence of the universal property of equalizers:

\begin{tvrz} \label{tvrz_gluingmorphisms}
Let $\F \in \PSh(X,\gcAs)$ and $\G \in \Sh(X,\gcAs)$. Let $U \in \Op(X)$ and $\{U_{\alpha} \}_{\alpha \in I}$ any its open cover. Then the following statements are true:
\begin{enumerate}[(i)]
\item Let $\eta,\eta': \F|_{U} \rightarrow \G|_{U}$ be a pair of presheaf morphisms, such that $\eta|_{U_{\alpha}} = \eta'|_{U_{\alpha}}$ for all $\alpha \in I$. Then they must coincide globally, $\eta = \eta'$. 
\item Let $\{ \eta_{\alpha} \}_{\alpha \in I}$ be a collection of presheaf morphisms $\eta_{\alpha}: \F|_{U_{\alpha}} \rightarrow \G|_{U_{\alpha}}$, satisfying
$\eta_{\alpha}|_{U_{\alpha \beta}} = \eta_{\beta}|_{U_{\alpha \beta}}$ for all $\alpha,\beta \in I$. Then there is a (unique) presheaf morphism $\eta: \F|_{U} \rightarrow \G|_{U}$, such that $\eta|_{U_{\alpha}} = \eta_{\alpha}$ for all $\alpha \in I$. 
\end{enumerate}
\end{tvrz}
This is a vital proposition allowing us to compare and define presheaf morphisms locally. 

\begin{example} \label{ex_thesheaf}
Let $\F \in \Sh(X,\As)$ by any sheaf of commutative associative algebras on $X$. Let $K \in \gVect$ be a finite-dimensional graded vector space. For each $U \in \Op(X)$, define
\begin{equation}
\O_{X}(U) := \bar{S}(K, \F(U)).
\end{equation}
For $V \subseteq U$, let $(\O_{X})^{U}_{V} := \bar{S}(\1_{K}, \F^{U}_{V})$. We claim that $\O_{X} \in \Sh(X,\gcAs)$. 

It follows immediately from Proposition \ref{tvrz_SVAisfunctor} that $\O_{X} \in \PSh(X,\gcAs)$. To prove that it is a sheaf, note that every section $s \in \O_{X}(U)$ can be written as 
\begin{equation}
s = \sum_{\fp \in \N^{n}_{|s|}} s_{\fp} \xi^{\fp},
\end{equation}
for a unique sequence $(s_{\fp})_{\fp \in \N^{n}_{|s|}} \subseteq \F(U)$, where we have fixed a total basis $(\xi_{\mu})_{\mu=1}^{n}$ of $K$. See Example \ref{ex_formalpower} for details. It follows that for each open subset $V \subseteq U$, the restriction morphism just restricts the coefficient sections, that is 
\begin{equation}
s|_{V} = \sum_{\fp \in \N^{n}_{|s|}} s_{\fp}|_{V} \xi^{\fp}.
\end{equation}
Now, let $\{ U_{\alpha} \}_{\alpha \in I}$ be an open cover of $U \in \Op(X)$ and suppose we have $s,t \in \O_{X}(Y)$ such that $s|_{U_{\alpha}} = t|_{U_{\alpha}}$ for all $\alpha \in I$. But then for each $\fp \in \N^{n}_{|s|}$, one finds that $s_{\fp}|_{U_{\alpha}} = t_{\fp}|_{U_{\alpha}}$. The monopresheaf property of $\F$ then implies that $s_{\fp} = t_{\fp}$ for all $\fp \in \N^{n}_{|s|}$, hence $s = t$. This shows that $\O_{X}$ has the monopresheaf property. The gluing property is proved similarly and we conclude that $\O_{X}$ is the sheaf of graded commutative associative algebras on $X$. 
\end{example}
\begin{example}  \label{ex_notasheaf}
Let us now offer an example of a similarly constructed presheaf, which in general fails to be a sheaf. Again, let $\F \in \Sh(X,\As)$ and $K$ a finite-dimensional graded vector space. For each $U \in \Op(X)$, one may define
\begin{equation}
\O^{\text{pol}}_{X}(U) := S(K,\F(U)) \equiv \F(U) \otimes_{\R} S(K),
\end{equation}
and $(\O^{\text{pol}}_{X})^{U}_{V} := \F^{U}_{V} \otimes S(\1_{K})$. Clearly $\O_{X}^{\text{pol}} \in \PSh(X,\gcAs)$. We claim that this (in general) fails to be a sheaf.

Indeed, consider $K \in \gVect$ with $K_{-2} = K_{2} = \R$ and all other components trivial. Let $(\xi_{1},\xi_{2})$ be a total basis of $K$ with $|\xi_{1}| = -2$ and $|\xi_{2}| = 2$. Let $X = \R$ and consider a sheaf $\F = \C^{0}_{\R}$ of continuous functions on $\R$. Elements of $\O^{\text{pol}}_{\R}(U)$ are \textit{polynomials} in $(\xi_{1},\xi_{2})$ with coefficients in the algebra $\C^{0}_{\R}(U)$. Let $U$ be a disjoint union of open subsets $U_{j} = (j,j+1)$ over $j \in \N_{0}$, and let 
\begin{equation}
s_{j} := (\xi_{1})^{j}(\xi_{2})^{j} \in (\O^{\text{pol}}_{\R}(U_{j}))_{0},
\end{equation}
for each $j \in \Z$. Those sections agree on the overlaps (as they are empty). However, there is no polynomial in $(\xi_{1},\xi_{2})$ with coefficients in $\C^{0}_{\R}(U)$ restricting to $s_{j}$ on $U_{j}$ for every $j \in \N_{0}$. Hence $\O_{\R}^{\text{pol}}$ is not a sheaf of graded commutative associative algebras. 
\end{example}
\begin{rem}
To any $\F \in \PSh(X,\gcAs)$, we may assign a presheaf $\F_{\oplus} \in \PSh(X, \As)$ of associative algebras, defined as $\F_{\oplus}(U) := \F(U)_{\oplus} \equiv \bigoplus_{k \in \Z} \F_{k}(U)$ for all $U \in \Op(X)$. This can be again viewed as a functor from $\PSh(X,\gcAs)$ to $\PSh(X,\As)$. However, $\F \in \Sh(X,\gcAs)$ \textit{does not imply} that $\F_{\oplus}$ is a sheaf of associative algebras. Note that this is an issues also for positively graded vector space $K$. 

Indeed, let $\O_{\R}$ be the sheaf from Example \ref{ex_thesheaf} for $X = \R$. Consider $K \in \gVect$ with $K_{2} = \R$ and all other components trivial, and $\F = \C^{0}_{\R}$ be the sheaf of continuous functions on $\R$. Let $\xi \in K_{2}$ be a fixed non-zero element. It is now easy to see that $(\O_{\R})_{\oplus}(U)$ is just the algebra of polynomials in a single variable $\xi$ with coefficients in $\C^{0}_{\R}(U)$. Using the arguments analogous to the previous example, one can show that $(\O_{\R})_{\oplus}$ is not a sheaf of associative algebras on $\R$. 
\end{rem}
\begin{definice}
Let $\F \in \PSh(X,\gcAs)$. We say that $\J \in \PSh(X,\gcAs)$ is a \textbf{presheaf of ideals in $\F$}, if the following two properties are satisfied:
\begin{enumerate}
\item $\J(U) \leq \F(U)$ is an ideal for every $U \in \Op(X)$.
\item Restriction morphisms of $\J$ are obtained from those of $\F$, that is $\J^{U}_{V} = \F^{U}_{V}|_{\J(U)}$ for all $U,V \in \Op(X)$, such that $V \subseteq U$. 
\end{enumerate}
If $\F \in \Sh(X,\gcAs)$, a \textbf{sheaf of ideals in $\F$} is $\J \in \Sh(X,\gcAs)$ with the same properties.
\end{definice}
\begin{rem}
To construct a presheaf of ideals, it suffices to define an ideal $\J(U) \subseteq \F(U)$ for each $U \in \Op(X)$ and verify that $\F^{U}_{V}(\J(U)) \subseteq \J(V)$ so one can \textit{define} $\J^{U}_{V} := \F^{U}_{V}|_{\J(U)}$. To construct a \textit{sheaf} of ideals, one has to verify the gluing property of $\J$. Let $\{ U_{\alpha} \}_{\alpha \in I}$ be an open cover of $U \in \Op(X)$ together with a collection $\{ s_{\alpha} \}_{\alpha \in I}$ of sections $s_{\alpha} \in \J(U_{\alpha})$ agreeing on the overlaps. By the gluing property of $\F$, there is a unique section $s \in \F(U)$ with $s|_{U_{\alpha}} = s_{\alpha}$. One only has to verify that this section is actually in $\J(U)$. 
\end{rem}
\begin{example} \label{ex_kernelsheaf}
Let $\eta: \F \rightarrow \G$ be a morphism of two sheaves $\F,\G \in \Sh(X,\gcAs)$. Let
\begin{equation}
\ker(\eta)(U) := \ker( \eta_{U}),
\end{equation}
for all $U \in \Op(X)$. Then $\ker(\eta)$ is a sheaf of ideals in $\F$ (with inherited restriction morphisms). This is can be shown easily using the naturality of $\eta$. 
\end{example}
\begin{rem} \label{rem_quotientsheaf}
Let $\J$ be a presheaf of ideals in $\F \in \PSh(X,\gcAs)$. For each $U \in \Op(X)$, define 
\begin{equation} (\F/\J)(U) := \F(U) / \J(U). \end{equation}
For each $V \subseteq U$, let $(\F / \J)^{U}_{V}: (\F/\J)(U) \rightarrow (\F/\J)(V)$ be the graded algebra morphism induced on the quotients by the graded algebra morphism $\F^{U}_{V}: \F(U) \rightarrow \F(V)$. This defines a \textbf{quotient presheaf} $\F / \J \in \PSh(X,\gcAs)$. However, even if $\F \in \Sh(X,\gcAs)$ and $\J$ is a \textit{sheaf} of ideals of $\F$, in general it is \textit{not true} that $\F/\J$ is a sheaf. 
\end{rem}
Let us now recall the following two general theorems for sheaves. Their proofs are a bit complicated diagram chasing using the equalizer diagrams (\ref{eq_sheafdiagram}) and we omit them here. See Exercise 8 in Chapter II of \cite{maclane2012sheaves} or Chapter I of \cite{serre1955faisceaux}.

\begin{tvrz}[\textbf{Collation of sheaves I}] \label{tvrz_shcollation1}
Let $X$ be a topological space and let $\{ U_{\alpha} \}_{\alpha \in I}$ be its open cover. Suppose we are given the following data:
\begin{enumerate}[(i)]
\item a collection $\{ \F_{\alpha} \}_{\alpha \in I}$, where $\F_{\alpha} \in \Sh(U_{\alpha}, \gcAs)$;
\item a collection $\{ \phi_{\alpha \beta} \}_{(\alpha,\beta) \in I^{2}}$, where $\phi_{\alpha \beta}: \F_{\alpha}|_{U_{\alpha \beta}} \rightarrow \F_{\beta}|_{U_{\alpha \beta}}$ are sheaf isomorphisms satisfying for each $(\alpha,\beta,\gamma) \in I^{3}$ the cocycle condition
\begin{equation} \label{eq_shcollcocycle}
\phi_{\beta \gamma} \circ \phi_{\alpha \beta} = \phi_{\alpha \gamma}, 
\end{equation}
where all morphisms are assumed to be restricted to the open subset $U_{\alpha \beta \gamma}$. 
\end{enumerate}
Then there exists $\F \in \Sh(X,\gcAs)$ together with a collection $\{ \lambda_{\alpha} \}_{\alpha \in I}$, where $\lambda_{\alpha}: \F|_{U_{\alpha}} \rightarrow \F_{\alpha}$ are sheaf isomorphisms satisfying $\phi_{\alpha \beta} \circ \lambda_{\alpha}|_{U_{\alpha \beta}} = \lambda_{\beta}|_{U_{\alpha \beta}}$ for all $(\alpha,\beta) \in I^{2}$. 

If $\F'$ and $\{\lambda'_{\alpha} \}_{\alpha \in I}$ are another data having these properties, there exists a unique sheaf isomorphism $\varphi: \F \rightarrow \F'$, such that $\lambda'_{\alpha} \circ \varphi|_{U_{\alpha}} = \lambda_{\alpha}$ for every $\alpha \in I$. 
\end{tvrz}
This proposition can be further refined, regarding refinements and different choices of the collation data.
\begin{tvrz}[\textbf{Collation of sheaves II}] \label{tvrz_shcollation2}
Let $X$ be a topological space and let $\{U_{\alpha}\}_{\alpha \in I}$ be its open cover. Suppose $\{\F_{\alpha} \}_{\alpha \in I}$ and $\{ \phi_{\alpha\beta} \}_{(\alpha,\beta)\in I^{2}}$ are the data $(i)$ and $(ii)$ as in the previous proposition. Suppose $\F$ and $\{ \lambda_{\alpha} \}_{\alpha \in I}$ are obtained by the collation of those data. Then the following observations are true:
\begin{enumerate}[(i)]
\item Let an open cover $\{ V_{\mu} \}_{\mu \in J}$ of $X$ be a refinement of $\{ U_{\alpha} \}_{\alpha \in I}$, that is there is a map $\zeta: J \rightarrow I$, such that $V_{\mu} \subseteq U_{\zeta(\mu)}$ for every $\mu \in J$. For every $\mu,\nu \in J$, define
\begin{equation} \F'_{\mu} := \F_{\zeta(\mu)}|_{V_{\mu}}, \; \; \phi'_{\mu \nu} := \phi_{\zeta(\mu)\zeta(\nu)}|_{V_{\mu \nu}}: \F'_{\mu}|_{V_{\mu\nu}} \rightarrow \F'_{\nu}|_{V_{\mu \nu}}.
\end{equation}
Then $\{ \F'_{\mu} \}_{\mu \in J}$ and $\{ \phi'_{\mu \nu} \}_{(\mu,\nu) \in J^{2}}$ form data $(i)$ and $(ii)$ as in the previous proposition, corresponding to the open cover $\{V_{\mu}\}_{\mu \in J}$. 

Suppose $\F'$ and $\{ \lambda'_{\mu} \}_{\mu \in J}$ are obtained by the collation of those data. Then there is a unique sheaf isomorphism $\varphi: \F \rightarrow \F'$, such that for every $\mu \in J$, one has $\lambda'_{\mu} \circ \varphi|_{V_{\mu}} = \lambda_{\zeta(\mu)}|_{V_{\mu}}$.

\item Let $\{ \F'_{\alpha} \}_{\alpha \in I}$ and $\{\phi'_{\alpha \beta} \}_{(\alpha,\beta) \in I^{2}}$ be another data (i) and (ii) in the previous proposition corresponding to the same open cover $\{ U_{\alpha} \}_{\alpha \in I}$ of $M$. Suppose $\F'$ and $\{ \lambda'_{\alpha} \}_{\alpha \in I}$ are obtained by the collation of those data. 

Let $\{ \psi_{\alpha} \}_{\alpha \in I}$ be a collection of sheaf isomorphisms $\psi_{\alpha}: \F_{\alpha} \rightarrow \F'_{\alpha}$, such that for all $(\alpha,\beta) \in I^{2}$
\begin{equation} \label{eq_collationisos} \phi'_{\alpha \beta} \circ \psi_{\alpha}|_{U_{\alpha \beta}} = \psi_{\beta}|_{U_{\alpha \beta}} \circ \phi_{\alpha \beta}. \end{equation}
Then there exists a unique sheaf isomorphism $\psi: \F \rightarrow \F'$, such that $\lambda'_{\alpha} \circ \psi|_{U_{\alpha}} = \psi_{\alpha} \circ \lambda_{\alpha}$ for all $\alpha \in I$. Conversely, every sheaf isomorphism induces a collection $\{\psi_{\alpha}\}_{\alpha \in I}$  satisfying (\ref{eq_collationisos}). 
\end{enumerate}
\end{tvrz}
\subsection{Stalks and sheafification}
Let $x \in X$ be a fixed point of a topological space $X$. By $\Op_{x}(X)$, we denote the set of all open neighborhoods of $x$. The category $\Op_{x}(X)^{\op}$ can be then easily seen to be a filtered category. The restriction of $\F \in \PSh(X,\gcAs)$ to $\Op_{x}(X)^{\op}$ then defines a functor over a filtered category, thus allowing for a definition of the corresponding filtered colimit\footnote{Historically, this is also called a direct limit and denoted as $\dlim$.}. See Chapter IX of \cite{mac2013categories} for details. 

\begin{tvrz} \label{tvrz_stalk}
For each $\F \in \PSh(X,\gcAs)$ and $x \in X$, there exists a filtered colimit 
\begin{equation}
\F_{x} := \colim_{U \in \Op_{x}(X)} \F(U).
\end{equation}
called the \textbf{stalk of the presheaf $\F$ at $x$}. For each $U \in \Op_{x}(X)$, let $\pi_{U,x}: \F(U) \rightarrow \F_{x}$ denote the graded algebra morphism of its \textbf{universal colimiting cone}. For each $s \in \F(U)$, the element $[s]_{x} := \pi_{U,x}(s) \in \F_{x}$ is called the \textbf{germ of $s$ at $x$}.
\end{tvrz}
\begin{proof}
By unfolding the definition of a filtered colimit, we must find a graded commutative associative algebra $\F_{x}$ together with a collection $\{ \pi_{U,x} \}_{U \in \Op_{x}(X)}$ of graded algebra morphisms $\pi_{U,x}: \F(U) \rightarrow \F_{x}$ having the following two properties:
\begin{enumerate}[(i)]
\item For any $U,V \in \Op_{x}(X)$, such that $V \subseteq U$, one has $\pi_{U,x} = \pi_{V,x} \circ \F^{U}_{V}$.
\item For any $A \in \gcAs$ and any other collection $\{ \tau_{U} \}_{U \in \Op_{x}(X)}$ of graded algebra morphisms $\tau_{U}: \F(U) \rightarrow A$ having the property (i), there exists a unique graded algebra morphism $\varphi: \F_{x} \rightarrow A$, such that $\varphi \circ \pi_{U,x} = \tau_{U}$ for all $U \in \Op_{x}(X)$. 
\end{enumerate}
First, let us construct the graded associative commutative algebra $\F_{x}$. For each $k \in \Z$, one defines
\begin{equation}
(\F_{x})_{k} := ( \hspace{-3mm} \bigsqcup_{U \in \Op_{x}(X)} \hspace{-3mm} \F(U)_{k} ) / \sim_{k},
\end{equation}
where $\sim_{k}$ is the equivalence relation defined as follows: for all $U,V \in \Op_{x}(X)$ and $s \in \F(U)_{k}$, $t \in \F(V)_{k}$, define
$s \sim_{k} t \Leftrightarrow$ there exists  $W \in \Op_{x}(X)$ such that $W \subseteq U \cap V$ and $s|_{W} = t|_{W}$. It is an easy exercise to prove that $\sim_{k}$ is reflexive, symmetric and transitive. Let $[s]_{x}$ denote the equivalence class of $s \in \F(U)_{k}$. The structure of a vector space on $(\F_{x})_{k}$ is defined by
\begin{equation}
[s]_{x} + \lambda [t]_{x} := [s|_{U \cap V} + \lambda t|_{U \cap V}]_{x}, 
\end{equation}
for all $s \in \F(U)_{k}$ and $t \in \F(V)_{k}$. Finally the multiplication $\mu: \F_{x} \otimes_{\R} \F_{x} \rightarrow \F_{x}$ is given as
\begin{equation}
\mu( [s]_{x} \otimes [t]_{x}) := [ \mu_{U \cap V}( s|_{U \cap V} \otimes t|_{U \cap V}) ]_{x},
\end{equation}
for all $s \in \F(U)_{k}$ and $t \in \F(V)_{k}$, where $\mu_{U}$ denotes the multiplication in $\F(U)$ for every $U \in \Op_{x}(X)$. The role of the algebra unit is played by the germ $[1_{U}]_{x}$ of the algebra unit $1_{U} \in \F(U)$, for any $U \in \Op_{x}(X)$. It is straightforward to verify that these operations are well defined and make $\F_{x}$ into a graded commutative associative algebra. 

Next, we set $\pi_{U,x}(s) := [s]_{x}$. Obviously, it is a graded algebra morphism and $\pi_{V,x} \circ \F^{U}_{V} = \pi_{U,x}$. This proves the property (i). Finally, for any $A \in \gcAs$ and a collection $\{ \tau_{U} \}_{U \in \Op_{x}(X)}$ of graded morphisms $\tau_{U}: \F(U) \rightarrow A$ having the property (i), define 
\begin{equation}
\varphi([s]_{x}) = \tau_{U}(s), \end{equation}
for all $s \in \F(U)$. Clearly, this is the unique graded algebra morphism $\varphi: \F_{x} \rightarrow A$ such that $\varphi \circ \pi_{U,x} = \tau_{U}$. This verifies the property (ii) and the proof is finished.
\end{proof}
A filtered colimit $\F_{x}$ and its universal colimiting cone $\{ \pi_{U,x} \}_{U \in \Op_{x}(X)}$ are uniquely defined by properties (i) and (ii), up to a graded algebra isomorphism. Note that for each $k \in \Z$, one has $(\F_{x})_{k} = (\F_{k})_{x}$, where $\F_{k} \in \PSh(X,\Vect)$ is the $k$-th component presheaf. Observe that for a general $\F$, $\pi_{U,x}$ is not necessarily an epimorphism. 
\begin{cor} \label{cor_inducedstalkmap}
Let $\F,\G \in \PSh(X,\gcAs)$ and $\varphi: \F \rightarrow \G$ a presheaf morphism. For any $x \in X$, there exists a unique graded algebra morphism $\varphi_{x}: \F_{x} \rightarrow \G_{x}$ fitting for every $U \in \Op_{x}(X)$ into the commutative diagram 
\begin{equation}
\begin{tikzcd}
\F(U) \arrow{d}{\pi^{\F}_{U,x}} \arrow{r}{\varphi_{U}} & \G(U) \arrow{d}{\pi^{\G}_{U,x}} \\
\F_{x} \arrow[dashed]{r}{\varphi_{x}} & \G_{x}
\end{tikzcd}.
\end{equation}
In particular, if $\psi: \G \rightarrow \H$ is another presheaf morphism, then $(\psi \circ \varphi)_{x} = \psi_{x} \circ \varphi_{x}$. If $\varphi$ is a presheaf isomorphism, then $\varphi_{x}$ is a graded algebra isomorphism. 
\end{cor}
\begin{proof}
For each $U \in \Op_{x}(X)$, choose $\tau_{U} := \pi^{\G}_{U,x} \circ \eta_{U}$ and use the universality of the colimiting cone $\{ \pi^{\F}_{U,x} \}_{U \in \Op_{x}(X)}$, that is the property (ii) in the proof above, to define $\eta_{x}$. In other words, $\eta_{x}$ is uniquely determined by the requirement $\eta_{x}([s]_{x}) = [ \eta_{U}(s)]_{x}$ for all $s \in \F(U)$. 
\end{proof}

\begin{rem}\label{rem_inducedmapsheafiso}
Let $\F,\G \in \Sh(X,\gcAs)$ and let $\varphi: \F \rightarrow \G$ be a sheaf morphism. Then $\varphi$ is an isomorphism, \textit{if and only if} $\varphi_{x}: \F_{x} \rightarrow \G_{x}$ is an isomorphism for all $x \in X$. One direction is a part of the above corollary. The proof of the other direction uses both the monopresheaf and gluing properties of $\F$ and the monopresheaf property of $\G$. However, it is straightforward and we leave the details to the reader. 
\end{rem}
\begin{example} \label{ex_bodyofasection}
Let $\O_{X} \in \Sh(X,\gcAs)$ be the sheaf from Example \ref{ex_thesheaf} for a given $K \in \gVect$ and $\F \in \Sh(X,\As)$. Let $n$ be the total dimension of $K$. One can view $\F$ as a sheaf of graded commutative associative algebras. There is then a canonical sheaf morphism $\beta: \O_{X} \rightarrow \F$, where for each $U \in \Op(X)$, one defines 
\begin{equation}
\beta_{U}(s) := s_{\mathbf{0}} \text{ for $|s| = 0$}, \; \; \beta_{U}(s) = 0 \text{ for $|s| \neq 0$},
\end{equation}
where $\mathbf{0} = (0,\dots,0) \in \N^{n}_{0}$. It follows from (\ref{eq_formalseriesproduct}) that this is a graded algebra morphism. For any $s \in \O_{X}(U)$, we will write just $\ul{s} := \beta_{U}(s)$ and call it the \textbf{body of the section $s$}. It follows from the previous corollary that for each $x \in X$, there is a well-defined graded algebra morphism $\beta_{x}: \O_{X,x} \rightarrow \F_{x}$ satisfying $\beta_{x}([s]_{x}) := [\ul{s}]_{x}$ for all $s \in \O_{X}(U)$. 
\end{example}

\begin{tvrz} \label{tvrz_sheafification}
There exists a \textbf{sheafification functor} $\Sff: \PSh(X,\gcAs) \rightarrow \Sh(X,\gcAs)$ with the following properties:
\begin{enumerate}[(i)]
\item For each $\F \in \PSh(X,\gcAs)$, there is a presheaf morphism $\eta_{\F}: \F \rightarrow \Sff(\F)$. It is a presheaf isomorphism, iff $\F$ is a sheaf. 
\item There exists a canonical identification $\Sff(\F)_{x} \cong \F_{x}$ for every $x \in X$. 
\item For every $x \in X$ and with respect to this identification, the induced graded algebra morphism $(\eta_{\F})_{x}: \F_{x} \rightarrow \F_{x}$ is the identity. 
\item For $\G \in \Sh(X,\gcAs)$ and any presheaf morphism $\psi: \F \rightarrow \G$, there exists a unique sheaf morphism $\hat{\psi}: \Sff(\F) \rightarrow \G$ such that $\hat{\psi} \circ \eta_{\F} = \psi$. In fact, $\Sff$ is left adjoint to the embedding of the full subcategory $\Sh(X,\gcAs)$ into $\PSh(X,\gcAs)$ and $\eta := \{ \eta_{\F} \}_{\F}$ becomes the unit of this adjunction. 
\end{enumerate}
\end{tvrz}
\begin{proof}
We will only make a quick sketch of the proof. It is a direct generalization of the procedure described e.g. in Chapter II of \cite{maclane2012sheaves}. One constructs a graded set $E = \{ E_{k} \}_{k \in \Z}$ by collecting all the stalks, $E_{k} := \sqcup_{x \in X} (\F_{k})_{x}$. For each $k \in \Z$, there is also a canonical projection $p_{k}: E_{k} \rightarrow X$ and a topology on $E_{k}$ making it into a local homeomorphism. One than defines
\begin{equation}
\Sff(\F)_{k}(U) \equiv \Gamma_{U}(E,p)_{k} := \{ \sigma: U \rightarrow E_{k} \; | \; \sigma \text{ is continuous}, \; p_{k} \circ \sigma = 1_{U} \},
\end{equation}
for all $k \in \Z$ and $U \in \Op(X)$. As $E_{x} := \{ p_{k}^{-1}(x) \}_{k \in \Z}$ is equal to $\F_{x} \in \gcAs$, one can make $\Sff(\F)(U)$ into a graded commutative associative algebra. By declaring $\Sff(\F)^{U}_{V}(\sigma) := \sigma|_{V}$ for every open $V \subseteq U$ and $\sigma \in \Gamma_{U}(E,p)$, we make $\Sff(\F)$ into a sheaf of graded commutative associative algebras on $X$. The presheaf morphism $\eta_{\F}$ is for any $U \in \Op(X)$ and $s \in \F(U)$ given by
\begin{equation}
((\eta_{\F})_{U}(s))(x) := [s]_{x},
\end{equation}
defining a graded algebra morphism $(\eta_{\F})_{U}: \F(U) \rightarrow \Sff(\F)(U)$. 
\end{proof}
\begin{rem}
In fact, as the above proof suggests, one can construct a category $\Et(X,\gcAs)$ of \textbf{étalé spaces\footnote{Often called also \textit{stalk spaces} or \textit{sheaf spaces}.} of graded commutative associative algebras over $X$}, equivalent to the category $\Sh(X,\gcAs)$. Its objects are pairs $(E,p)$, where $E = \{ E_{k} \}_{k \in \Z}$ is a sequence of topological spaces and $p = \{ p_{k} \}_{k \in \Z}$ is a sequence of local homeomorphisms $p_{k}: E_{k} \rightarrow X$. Moreover, each fiber $E_{x} := \{ p_{k}^{-1}(x) \}_{k \in \Z}$ must be an object in $\gcAs$ with some continuity conditions imposed on the algebraic operations. See Chapter 10 of excellent lecture notes \cite{gallier2016gentle}.
\end{rem}
\begin{rem}
In general, the sheaf $\O_{X}$ in Example \ref{ex_thesheaf} is \textit{not} a sheafification of the presheaf $\O_{X}^{\text{pol}}$ in Example \ref{ex_notasheaf}. We will show in Example \ref{ex_twogLRS} that $\O_{X,x}$ is a local graded ring for every $x \in X$. On the other hand, consider $\O^{\text{pol}}_{\R}$ for $\F = \C^{0}_{\R}$ and $K$ as in Example \ref{ex_notasheaf}. For any $x \in \R$, the germs $[1 \pm \xi_{1} \xi_{2}]_{x}$ are not invertible and their sum $[2]_{x}$ is. It thus follows from Proposition \ref{tvrz_grlocal}-$(v)$ that $\O^{\text{pol}}_{\R,x}$ is not a local graded ring. Hence $\O_{\R} \neq \Sff(\O_{\R}^{\text{pol}})$ by Proposition \ref{tvrz_sheafification}-$(ii)$ and Corollary \ref{cor_isoringtolocalislocal}. 
\end{rem}
%\begin{example}
%Let $\eta: \F \rightarrow \G$ be a morphism of two sheaves $\F,\G \in \gcAs$. Then we have a kernel sheaf $\ker(\eta) \subseteq \F$ of ideals. Let us now argue that for each $x \in X$, there is a canonical isomorphism $\ker(\eta)_{x} \cong \ker( \eta_{x})$, where $\eta_{x}: \F_{x} \rightarrow \G_{x}$ is the induced algebra morphism of stalks. 
%
%Let $\iota: \ker(\eta) \rightarrow \F$ denote the obvious inclusion. This is a sheaf morphism, whence for each $x \in X$, there is an induced graded algebra morphism $\iota_{x}: \ker(\eta)_{x} \rightarrow \F_{x}$. It is easy to see that it is injective and its image coincides with $\ker(\eta_{x})$. We can thus view it as a graded algebra isomorphism $\iota_{x}: \ker(\eta_{x}) \rightarrow \ker(\eta_{x})$.
%\end{example}
\subsection{Graded locally ringed spaces}
Let $X,Y$ be two topological spaces and let $\ul{\varphi}: X \rightarrow Y$ be a continuous map. 

For any $\F \in \PSh(X,\gcAs)$, the \textbf{pushforward presheaf} $\ul{\varphi}_{\ast}\F \in \PSh(Y,\gcAs)$ is for each $U \in \Op(Y)$ defined as 
\begin{equation} (\ul{\varphi}_{\ast} \F)(U) := \F(\ul{\varphi}^{-1}(U)). \end{equation}
The restriction morphisms are for any $V \subseteq U$ obtained as 
\begin{equation}
(\ul{\varphi}_{\ast} \F)^{U}_{V} := \F^{\ul{\varphi}^{-1}(U)}_{\ul{\varphi}^{-1}(V)}.
\end{equation}
It is easy to see that if $\F \in \Sh(X,\gcAs)$, then $\ul{\varphi}_{\ast}\F \in \Sh(Y,\gcAs)$. 

\begin{tvrz} \label{tvrz_inducedstalkpushsheaf}
Let $\F \in \PSh(X,\gcAs)$ and let $\ul{\varphi}: X \rightarrow Y$ be a continuous map. For each $x \in X$, there is an induced graded algebra morphism $\ul{\varphi}^{\ast}_{x}: (\ul{\varphi}_{\ast}\F)_{\ul{\varphi}(x)} \rightarrow \F_{x}$. If $\ul{\varphi}$ is a homeomorphism, $\ul{\varphi}^{\ast}_{x}$ is a graded algebra isomorphism. 
\end{tvrz}
\begin{proof}
Let $x \in X$ and consider $s \in (\ul{\varphi}_{\ast}\F)(V) \equiv \F(\ul{\varphi}^{-1}(V))$ for $V \in \Op_{\ul{\varphi}(x)}(X)$. The universality property of colimits then ensures that the formula
\begin{equation}
\ul{\varphi}^{\ast}_{x}([s]_{\ul{\varphi}(x)}) := [s]_{x}
\end{equation}
defines a unique graded algebra morphism from $(\ul{\varphi}_{\ast}\F)_{\ul{\varphi}(x)}$ to $\F_{x}$, where on the right hand side, $s$ is viewed as a section of $\F$ over $\ul{\varphi}^{-1}(V) \in \Op_{x}(X)$. When $\ul{\varphi}$ is a homeomorphism, each $t \in \F(U)$ for $U \in \Op_{x}(X)$ can be viewed as a section of $\ul{\varphi}_{\ast}\F$ over $\ul{\varphi}(U) \in \Op_{\ul{\varphi}(x)}(Y)$. Hence it makes sense to map $[t]_{x}$ to $[t]_{\ul{\varphi}(x)}$. It is easy to show that this is the inverse to $\ul{\varphi}^{\ast}_{x}$. 
\end{proof}

\begin{definice} \label{def_gLRS}
A pair $(X,\O_{X})$ of a topological space $X$ and a sheaf $\O_{X} \in \Sh(X,\gcAs)$ is called a \textbf{graded locally ringed space} with the \textbf{structure sheaf $\O_{X}$}, if for each $x \in X$, the stalk $\O_{X,x}$ is a local graded ring (see Remark \ref{rem_tensorproductofgcAs}-(iii)). 

Let $(X,\O_{X})$ and $(Y,\O_{Y})$ be a pair of graded locally ringed spaces. A pair $\varphi = (\ul{\varphi}, \varphi^{\ast})$ is called a \textbf{morphism of graded locally ringed spaces}, if 
\begin{enumerate}[(i)]
\item $\ul{\varphi}: X \rightarrow Y$ is a continuous map;
\item $\varphi^{\ast}: \O_{Y} \rightarrow \ul{\varphi}_{\ast}\O_{X}$ is a sheaf morphism;
\item for each $x \in X$, the graded algebra morphism $\varphi_{(x)}: \O_{Y,\ul{\varphi}(x)} \rightarrow \O_{X,x}$ defined as a composition 
\begin{equation} \label{eq_vartphixmap} \varphi_{(x)} := \ul{\varphi}^{\ast}_{x} \circ (\varphi^{\ast})_{\ul{\varphi}(x)}, \end{equation}
is a local graded ring morphism (see Corollary \ref{cor_inducedstalkmap} and Proposition \ref{tvrz_inducedstalkpushsheaf}).

In other words, the map defined for each $V \in \Op_{\ul{\varphi}(x)}(Y)$ and $s \in \O_{Y}(V)$ by 
\begin{equation} \label{eq_vartphixmap2} \varphi_{(x)}([s]_{\ul{\varphi}(x)}) := [\varphi^{\ast}_{V}(s)]_{x} \end{equation}
has to satisfy $\varphi_{(x)}( \frJ( \O_{Y,\ul{\varphi}(x)})) \subseteq \frJ( \O_{X,x})$. 
\end{enumerate}
We write $\varphi: (X,\O_{X}) \rightarrow (Y,\O_{Y})$. If $\psi: (Y,\O_{Y}) \rightarrow (Z,\O_{Z})$ is another morphism, their composition $\psi \circ \varphi$ is defined as a pair $(\ul{\psi} \circ \ul{\varphi}, (\psi \circ \varphi)^{\ast})$, where we declare
\begin{equation}
(\psi \circ \varphi)^{\ast}_{W} := \varphi^{\ast}_{\ul{\psi}^{-1}(W)} \circ \psi^{\ast}_{W}: \O_{Z}(W) \rightarrow \O_{X}( (\ul{\psi} \circ \ul{\varphi})^{-1}(W))
\end{equation}
for all $W \in \Op(Z)$. It follows that this defines a sheaf morphism $(\psi \circ \varphi)^{\ast}: \O_{Z} \rightarrow (\ul{\psi} \circ \ul{\varphi})_{\ast}\O_{X}$. Consequently, graded locally ringed spaces together with their morphisms form the category $\gLRS$. 
\end{definice}
\begin{tvrz} \label{tvrz_isoingLRS}
Let $(X,\O_{X})$ and $(Y,\O_{Y})$ be a pair of graded locally ringed spaces. Let $\varphi = (\ul{\varphi},\varphi^{\ast})$ be a pair having the properties (i) and (ii) in Definition \ref{def_gLRS}. 

Then $\varphi$ is an isomorphism in $\gLRS$, iff $\ul{\varphi}$ is a homeomorphism and $\varphi^{\ast} $ is a sheaf isomorphism. 
\end{tvrz}
\begin{proof}
First, let $\varphi$ be an isomorphism $\gLRS$. Let $\psi = (\ul{\psi}, \psi^{\ast})$ be its inverse. One has $\ul{\psi} \circ \ul{\varphi} = \1_{X}$ and $\ul{\varphi} \circ \ul{\psi} = \1_{Y}$, hence $\ul{\varphi}$ is a homeomorphism. The inverse to $\varphi^{\ast}: \O_{Y} \rightarrow \ul{\varphi}_{\ast}\O_{X}$ is for each $V \in \Op(Y)$ given by $(\varphi^{\ast})^{-1}_{V} := \psi^{\ast}_{\ul{\psi}(V)}$, hence $\varphi^{\ast}$ is a sheaf isomorphism. 

Conversely, one must verify that $\varphi = (\ul{\varphi}, \varphi^{\ast})$ has the property (iii) of Definition \ref{def_gLRS}. By Corollary \ref{cor_inducedstalkmap} and Proposition \ref{tvrz_inducedstalkpushsheaf}, the graded algebra morphism (\ref{eq_vartphixmap}) is an isomorphism for all $x \in X$. Hence by Corollary \ref{cor_isomorphismarelocal}, it is a \textit{local} graded ring morphism. Hence $\varphi$ is a morphism of graded locally ringed spaces. Letting $\ul{\psi} := \ul{\varphi}^{-1}$ and $\psi^{\ast}_{U} := (\varphi^{\ast})^{-1}_{\ul{\varphi}(U)}$ for all $U \in \Op(X)$, we obtain its inverse $\psi = (\ul{\psi}, \psi^{\ast})$. Hence $\varphi$ is an isomorphism in $\gLRS$. 
\end{proof}
\begin{tvrz} \label{tvrz_restrictedgLRS}
Let $(X,\O_{X}) \in \gLRS$. Let $U \in \Op(X)$. Then $(U,\O_{X}|_{U}) \in \gLRS$ and there is a canonical morphism $i: (U,\O_{X}|_{U}) \rightarrow (X, \O_{X})$ in $\gLRS$, where $\ul{i}: U \rightarrow X$ is the obvious inclusion. For each $x \in U$, the induced map $i_{(x)}: \O_{X,\ul{i}(x)} \rightarrow (\O_{X}|_{U})_{x}$ is a graded ring isomorphism.
\end{tvrz}
\begin{proof}
For each $V \in \Op(X)$, we have to define a graded algebra morphism 
\begin{equation}
i^{\ast}_{V}: \O_{X}(V) \rightarrow (\ul{i}_{\ast} \O_{X}|_{U})(V) = \O_{X}|_{U}( \ul{i}^{\ast}(V)) = \O_{X}(U \cap V).
\end{equation}
There is thus a canonical choice $i^{\ast}_{V} := (\O_{X})^{V}_{U \cap V}$. Its naturality in $V$ is straightforward, whence $i^{\ast}: \O_{X} \rightarrow \ul{i}_{\ast} \O_{X}|_{U}$ is a sheaf morphism. For each $x \in U$, $V \in \Op_{\ul{i}(x)}(X)$ and $s \in \O_{X}(V)$, one has
\begin{equation}
i_{(x)}( [s]_{\ul{i}(x)}) = [ s|_{U \cap V} ]_{x}.
\end{equation}
For each $V \in \Op_{x}(U)$, we have $\O_{X}|_{U}(V) = \O_{X}(V)$. For each $t \in \O_{X}|_{U}(V)$, we may thus send $[t]_{x}$ to $[t]_{\ul{i}(x)}$. This defines the inverse to $i_{(x)}$, so it is a graded algebra isomorphism. Moreover, by Corollary \ref{cor_isoringtolocalislocal}, this implies that $(\O_{X}|_{U})_{x}$ is local for every $x \in U$. Thus $(U,\O_{X}|_{U}) \in \gLRS$. 
\end{proof}
\begin{tvrz} 
Let $\varphi: (X,\O_{X}) \rightarrow (Y,\O_{Y})$ be a morphism in $\gLRS$. Let $U \in \Op(X)$. Then by $\varphi|_{U}: (U,\O_{X}|_{U}) \rightarrow (Y,\O_{Y})$, we denote the composition $\varphi \circ i$, where $i: (U,\O_{X}|_{U}) \rightarrow (X,\O_{X})$ is the inclusion from the previous proposition. 

If $\ul{\varphi}(U) \subseteq V$ for some $V \in \Op(Y)$, there is a unique morphism $\varphi': (U,\O_{X}|_{U}) \rightarrow (V,\O_{Y}|_{V})$, such that $j \circ \varphi' = \varphi|_{U}$, where $j: (V,\O_{Y}|_{V}) \rightarrow (Y,\O_{Y})$ is the inclusion as above. Usually, we will not distinguish between $\varphi'$ and $\varphi|_{U}$. 
\end{tvrz}
\begin{proof}
This is an easy exercise. 
\end{proof}
Sometimes it is useful to glue together a collection of locally defined morphisms.
\begin{tvrz} \label{tvrz_gLRSgluing}
Let $(X,\O_{X})$ and $(Y,\O_{Y})$ be a pair of graded locally ringed spaces. Let $\{ U_{\alpha} \}_{\alpha \in I}$ be an open cover of $X$. Suppose one has a collection $\{ \phi_{\alpha} \}_{\alpha \in I}$, where $\phi_{\alpha}: (U_{\alpha},\O_{X}|_{U_{\alpha}}) \rightarrow (Y,\O_{Y})$ is a $\gLRS$ morphism for each $\alpha \in I$. Moreover, for every $(\alpha,\beta) \in I^{2}$, one has $\phi_{\alpha}|_{U_{\alpha \beta}} = \phi_{\beta}|_{U_{\alpha \beta}}$. 

Then there is a unique morphism $\phi: (X,\O_{X}) \rightarrow (Y,\O_{Y})$, such that $\phi|_{U_{\alpha}} = \phi_{\alpha}$ for all $\alpha \in I$. 
\end{tvrz}
\begin{proof}
For each $\alpha \in I$, we have $\ul{\phi_{\alpha}}: U_{\alpha} \rightarrow Y$, such that $\ul{\phi_{\alpha}}|_{U_{\alpha \beta}} = \ul{\phi_{\beta}}|_{U_{\alpha \beta}}$ for every $(\alpha,\beta) \in I^{2}$. There is thus a unique continuous map $\ul{\phi}: X \rightarrow Y$, such that $\ul{\phi}|_{U_{\alpha}} = \ul{\phi_{\alpha}}$ for each $\alpha \in I$. Next, for each $V \in \Op(Y)$ and $\alpha \in I$, one has a graded algebra morphism 
\begin{equation}
(\phi_{\alpha})^{\ast}_{V}: \O_{Y}(V) \rightarrow \O_{X}( \ul{\phi_{\alpha}}^{-1}(V)) = O_{X}(\ul{\phi}^{-1}(V) \cap U_{\alpha}).
\end{equation}
Recall that one can now look at the equalizer diagram (\ref{eq_sheafdiagram}) for the sheaf $\O_{X}$, the open set $\ul{\phi}^{-1}(V)$ and its open cover $\{ \ul{\phi}^{-1}(V) \cap U_{\alpha} \}_{\alpha \in I}$. Define a graded algebra morphism 
\begin{equation}
\chi: \O_{Y}(V) \rightarrow \prod_{\alpha \in I} \O_{X}(\ul{\phi}^{-1}(V) \cap U_{\alpha}), \; \; p_{\alpha} \circ \chi := (\phi_{\alpha})^{\ast}_{V} \text{ for all } \alpha \in I.
\end{equation}
The assumption $\phi_{\alpha}|_{U_{\alpha \beta}} = \phi_{\beta}|_{U_{\alpha \beta}}$ ensures that $\psi_{1} \circ \chi = \psi_{2} \circ \chi$. By the universal property of equalizers, there is a unique graded algebra morphism $\phi^{\ast}_{V}: \O_{Y}(V) \rightarrow \O_{X}(\ul{\phi}^{-1}(V))$ such that $\phi^{\ast}_{V}(f)|_{U_{\alpha}} = (\phi_{\alpha})^{\ast}_{V}(f)$ for all $f \in \O_{Y}(V)$ and $\alpha \in I$. Uniqueness assertion of the universal property of equalizers can be also used to prove the naturality, whence $\phi^{\ast}: \O_{Y} \rightarrow \ul{\phi}_{\ast}\O_{X}$ is a sheaf morphism. Finally, it follows that $\phi = (\ul{\phi}, \phi^{\ast})$ has the property (iii) of Definition \ref{def_gLRS}, as $\phi|_{U_{\alpha}} = \phi_{\alpha}$ has it for all $\alpha \in I$. Hence $\phi$ is a unique morphism in $\gLRS$, such that $\phi|_{U_{\alpha}} = \phi_{\alpha}$. 
\end{proof}
\begin{example} \label{ex_twogLRS}
Let $\F \in \Sh(X,\As)$ be a sheaf of commutative associative algebras. Let $\O_{X} \in \Sh(X,\gcAs)$ be the sheaf constructed in Example \ref{ex_thesheaf} using $\F$ and a finite-dimensional graded vector space $K \in \gVect$. It is now useful to view $\F$ as a sheaf of (trivially) graded commutative associative algebras. Let us argue that $(X,\O_{X}) \in \gLRS$, if and only if $(X,\F) \in \gLRS$. 

We will use the formalism of Example \ref{ex_formalpower}. Let $(\xi_{\mu})_{\mu=1}^{n}$ be a fixed total basis of $K$. Let $x \in X$ be fixed. Let us examine the group of units $\frU( \O_{X,x})$ of the graded ring $\O_{X,x}$. We claim that it can be written as an inverse image (defined degree-wise)
\begin{equation} \label{eq_unitsOXxFxrel}
\frU(\O_{X,x}) = \beta_{x}^{-1}( \frU(\F_{x})),
\end{equation}
where $\beta_{x}: \O_{X,x} \rightarrow \F_{x}$ is the graded algebra morphism from Example \ref{ex_bodyofasection} and we view $\F$ as a sheaf of (trivially) graded commutative  associative algebras. The inclusion $\subseteq$ is true for any graded algebra morphism. To prove the inclusion $\supseteq$, first note that $\frU(\F_{x})_{k} = \emptyset$ for any $k \neq 0$, so it suffices to deal with elements of degree zero. Let $[s]_{x} \in (\O_{X,x})_{0}$ be an element represented by some $s \in \O_{X}(U)_{0}$, such that $\beta_{x}([s]_{x}) \equiv [s_{\mathbf{0}}]_{x} \in \frU(\F_{x})$. By restricting $s$ to some subset of $U$ if necessary, we may assume that $s_{\mathbf{0}} \in \F(U)$ is invertible. This allows one to write
\begin{equation} \label{eq_suminverse1}
s = s_{\mathbf{0}}( 1 + s'), 
\end{equation}
for the unique $s' \in \O_{X}(U)_{0}$. Note that $s'_{\mathbf{0}} = 0$. Let us now define $t \in \O_{X}(U)_{0}$ by
\begin{equation} \label{eq_suminverse2}
t = s_{\mathbf{0}}^{-1} \cdot \sum_{q=0}^{\infty} (-1)^{q} s'^{q}.
\end{equation}
To be more precise, one has $t = \sum_{\fp \in \N^{n}_{0}} t_{\fp} \xi^{\fp}$, where $t_{\fp}$ is given by the \textit{finite sum} 
\begin{equation} \label{eq_suminverse3}
t_{\fp} = s_{\mathbf{0}}^{-1} \cdot \sum_{q=0}^{w(\fp)} (-1)^{q} (s'^{q})_{\fp},
\end{equation}
where $w(\fp) = \sum_{\mu =1}^{n} p_{\mu}$. We can replace the upper bound by $\infty$ as $(s'^{q})_{\fp} = 0$ for $q > w(\fp)$. The section $t$ is easily seen to be the inverse to $s$, hence $s \in \frU(\O_{X}(U))_{0}$ and in particular, $[s]_{x} \in \frU(\O_{X,x})_{0}$. This concludes the proof of the relation (\ref{eq_unitsOXxFxrel}). As $\beta_{x}$ is a graded algebra epimorphism, it now follows from Proposition \ref{tvrz_grlocal}-$(v)$ that $\O_{X,x}$ is a graded local ring, if and only if $\F_{x}$ is a (trivially) graded local ring. As $x \in X$ was arbitrary, the statement follows. 
\end{example}
\subsection{Sheaves of graded modules} \label{subsec_sheavesofgradedmodules}
Following paragraphs can be viewed as a sheaf side of Subsection \ref{subsec_gmods}. 
\begin{definice} \label{def_sheavesofAmodules}
Let $X$ be a topological space, $\A \in \Sh(X,\gcAs)$ and $\F \in \Sh(X,\gVect)$. We say that $\F$ is a \textbf{sheaf of graded $\A$-modules}, if 
\end{definice}
\begin{enumerate}[(i)]
\item for each $U \in \Op(X)$, $\F(U) \in \gVect$ is a graded $\A(U)$-module;
\item restrictions of $\F$ and $\A$ are compatible with graded module structures, that is for all $U,V \in \Op(X)$ such that $V \subseteq U$, one has $\F^{U}_{V}( a \tr s) = \A^{U}_{V}(s) \tr \F^{U}_{V}(s)$ for all $a \in \A(U)$ and $s \in \F(U)$. 
\end{enumerate}
If $\G \in \Sh(X,\gVect)$ is another sheaf of graded $\A$-modules, a sheaf morphism $\eta: \F \rightarrow \G$ is called \textbf{$\A$-linear}, if for each $U \in \Op(X)$, the graded linear map $\eta_{U}: \F(U) \rightarrow \G(U)$ is \textbf{$\A(U)$-linear}, that is $\eta_{U}(a \tr s) = a \tr \eta_{U}(s)$ for all $a \in \A(U)$ and $s \in \F(U)$. Sheaves of graded $\A$-modules and their $\A$-linear morphisms form a subcategory $\Sh^{\A}(X,\gVect) \subseteq \Sh(X,\gVect)$. 

\begin{example}
$\A$ itself can be viewed as a sheaf of graded $\A$-modules. See Example \ref{ex_AisAmodule}. 
\end{example}
\begin{rem} \label{rem_degreeshiftedsheafofmodules}
For every $\F \in \Sh^{\A}(X,\gVect)$ and any $\ell \in \Z$, one can define the \textbf{degree shifted sheaf} $\F[\ell] \in \Sh^{\A}(X,\gVect)$ by declaring $(\F[\ell])(U) := \F(U)[\ell]$ for all $U \in \Op(X)$, see Definition \ref{def_degreeshifted}. The graded $\A(U)$-module structure on $\F(U)[\ell]$ is the one described in Remark \ref{rem_shiftedmodule}. Restriction morphisms are inherited from $\F$, making $\F[\ell]$ into a sheaf of graded $\A$-modules. 
\end{rem}

We would like to obtain a sheaf analogue of Proposition \ref{tvrz_LinAmodule}. In particular, to any $\F,\G \in \Sh^{\A}(X,\gVect)$, one would like to assign a sheaf of ``maps of all degrees from $\F$ to $\G$''. It turns out that this is not so straightforward.
\begin{definice} \label{def_ShAFG}
For each $\ell \in \Z$, we say that a sheaf morphism $\eta: \F \rightarrow \G[\ell]$ is an \textbf{$\A$-linear sheaf morphism of degree $\ell$}, if for every $U \in \Op(X)$, $\eta_{U} \in \gVect(\F(U),\G(U)[\ell]) \cong \Lin_{\ell}(\F(U),\G(U))$ is a graded $\A(U)$-linear map of degree $\ell$, see Definition \ref{def_Alinearmaps}. $\A$-linear sheaf morphisms of degree $\ell$ form a vector space which we denote as $\Sh_{\ell}^{\A}(\F,\G)$. Altogether, we obtain a graded vector space $\Sh^{\A}(\F,\G) = \{ \Sh^{\A}_{\ell}(\F,\G) \}_{\ell \in \Z}$. $\Sh_{0}^{\A}(\F,\G)$ is the set or morphisms from $\F$ to $\G$ in $\Sh^{\A}(X,\gVect)$. 
\end{definice}
We are now ready to make the following observation:
\begin{tvrz} \label{tvrz_uSh}
Let $\F,\G \in \Sh^{\A}(X,\gVect)$. For each $U \in \Op(X)$, define 
\begin{equation}
\ul{\Sh}^{\A}(\F,\G)(U) := \Sh^{\A|_{U}}( \F|_{U}, \G|_{U}) \in \gVect.
\end{equation}
Then there is a canonical way making $\ul{\Sh}^{\A}(\F,\G)$ into a sheaf of graded $\A$-modules. 

In particular, as $\A \in \Sh^{\A}(X,\gVect)$, to each $\F \in \Sh^{\A}(X,\gVect)$ one can associate its \textbf{dual sheaf of $\A$-modules} $\F^{\ast} := \ul{\Sh}^{\A}(\F,\A)$. 
\end{tvrz}
\begin{proof}
First, for any $U \in \Op(X)$, we must equip $\Sh^{\A|_{U}}(\F|_{U},\G|_{U})$ with a structure of a graded $\A(U)$-module. Let $a \in \A(U)$ and $\eta \in \Sh^{\A|_{U}}( \F|_{U}, \G|_{U})$. For every $V \in \Op(U)$, we thus have $\A(V)$-linear $\eta_{V}: \F(V) \rightarrow \G(V)$. By Proposition \ref{tvrz_LinAmodule}, it makes sense to define
\begin{equation}
(a \tr \eta)_{V} := \A^{U}_{V}(a) \tr \eta_{V}.
\end{equation}
Let us check the naturality. Let $W \subseteq V$ be an open subset and let $s \in \F(V)$. Then
\begin{equation}
\begin{split}
\G^{V}_{W}( (a \tr \eta)_{V}(s)) = & \ \G^{V}_{W}( \A^{U}_{V}(a) \tr \eta_{V}(s)) = \A^{U}_{W}(a) \tr \G^{V}_{W}( \eta_{V}(s)) \\
= & \ \A^{U}_{W}(a) \tr \eta_{W}( \F^{V}_{W}(s)) = (a \tr \eta)_{W}( \F^{V}_{W}(s)). 
\end{split}
\end{equation}
Hence $a \tr \eta \in \Sh^{\A|_{U}}(\F|_{U},\G|_{U})$. It is easy to check the axioms (\ref{eq_moduleaxioms}). 

Next, for each $U,V \in \Op(X)$ with $V \subseteq U$, let $(\ul{\Sh}^{\A})^{U}_{V}(\eta) := \eta|_{V}$ for any $\eta \in \Sh^{\A|_{U}}(\F|_{U},\G|_{U})$. It easy to check that $\eta|_{V} \in \Sh^{\A|_{V}}(\F|_{V},\G|_{V})$ and $(a \tr \eta)|_{V} = \A^{U}_{V}(a) \tr \eta|_{V}$. This makes $\ul{\Sh}^{\A}(\F,\G)$ into a presheaf of $\A$-modules. Finally, it follows immediately from Proposition \ref{tvrz_gluingmorphisms} that $\ul{\Sh}^{\A}(\F,\G)$ has the monopresheaf and gluing properties, and the proof is finished. 
\end{proof}

\begin{rem} \label{rem_ulShcat}
The sheaf of modules $\ul{\Sh}^{\A}(\F,\G)$ is interesting from the categorical point of view. Indeed, one can turn $\Sh^{\A}(X,\gVect)$ into a symmetric monoidal cateogory with a tensor product $\otimes_{\A}$ and $\A$ playing the role of its unit. The definition of $\otimes_{\A}$ is a bit complicated, as it requires one to use the sheafification. It turns out that $\ul{\Sh}^{\A}$ becomes the internal hom in this monoidal category, that is for all $\F,\G,\H \in \Sh^{\A}(X,\gVect)$, there is a bijection
\begin{equation} \label{eq_tensorinternalhom}
\Sh^{\A}_{0}(\F \otimes_{\A} \G, \H) \cong \Sh^{\A}_{0}(\F, \ul{\Sh}^{\A}(\G,\H)),
\end{equation}
natural in all three variables. See e.g. \cite{kelly1982basic} for details. 
\end{rem}

\begin{rem} \label{rem_uShdegshifted} The functor $\ul{\Sh}^{\A}$ behaves naturally with respect to the degree shifts. For any $\F,\G \in \Sh^{\A}(X, \gVect)$ and any $\ell \in \Z$, there are canonical $\A$-linear sheaf isomorphisms 
\begin{equation}
\ul{\Sh}^{\A}(\F, \G[\ell]) \cong \ul{\Sh}^{\A}( \F[-\ell], \G) \cong \ul{\Sh}^{\A}(\F,\G)[\ell].
\end{equation}
In particular, this proves that $(\F[\ell])^{\ast} \cong \F^{\ast}[-\ell]$. One can easily prove this ``point-wise'' by using the isomorphisms obtained in Remark \ref{rem_LinAofshifted}. 
\end{rem}
\begin{example} \label{ex_modelfreelygensheaf}
Let $\A \in \Sh(X,\gcAs)$ and let $K \in \gVect$ be any graded vector space. For each $U \in \Op(X)$, consider the graded vector space
\begin{equation}
\A[K](U) := \A(U) \otimes_{\R} K.
\end{equation}
For each $U,V \in \Op(X)$ such that $V \subseteq X$, one defines a graded linear map $\A[K]^{U}_{V} := \A^{U}_{V} \otimes \1_{K}$. This makes $\A[K]$ into a \textit{presheaf} of graded vector spaces. For general $K$, this is not a sheaf. However, suppose that $m = \dim(K) < \infty$ and let $(\vartheta_{\lambda})_{\lambda=1}^{m}$ be its total basis. Then any $s \in \A[K](U)$ can be written as a finite sum 
\begin{equation} \label{eq_suniquesumAK}
s = a^{\lambda} \otimes \vartheta_{\lambda}
\end{equation}
for unique sections $a^{\lambda} \in \A(U)$ satisfying $|s| = |a^{\lambda}| + |\vartheta_{\lambda}|$ for every $\lambda \in \{1,\dots,m\}$. Then
\begin{equation}
\A[K]^{U}_{V}(s) = \A^{U}_{V}(a^{\lambda}) \otimes \vartheta_{\lambda},
\end{equation}
and it is now easy to see that $\A[K]$ is a sheaf due to the fact that $\A$ is a sheaf. 

For each $U \in \Op(X)$, one can make $\A(U) \otimes_{\R} K$ into an $\A(U)$-module by declaring $a \tr (b \otimes x) := (a \cdot b) \otimes x$ for all $a,b \in \A(U)$ and $x \in K$. It is easy to see that these are compatible with the restrictions, hence $\A[K] \in \Sh^{\A}(X,\gVect)$.
\end{example}

\begin{definice}
Let $\F \in \Sh^{\A}(X,\gVect)$. We say that $\F$ is \textbf{freely and finitely generated}, if it is isomorphic to $\A[K]$ in the category $\Sh^{\A}(X,\gVect)$ for some finite-dimensional $K \in \gVect$. Freely and finitely generated sheaves of $\A$-modules form a full subcategory $\Sh^{\A}_{\fin}(X,\gVect)$. 
\end{definice}

\begin{definice}
Let $\F \in \Sh^{\A}(X,\gVect)$ be a non-trivial sheaf. Let $( \Phi_{\lambda} )_{\lambda = 1}^{m}$ be a collection of sections of $\F(X)$. We say that $( \Phi_{\lambda} )_{\lambda = 1}^{m}$ is a \textbf{global frame for $\F$}, if for every $U \in \Op(X)$, every section $s \in \F(U)$ can be written as 
\begin{equation}
s = a^{\lambda} \tr \Phi_{\lambda}|_{U},
\end{equation}
for \textit{unique} sections $a^{\lambda} \in \A(U)$ with $|s| = |a^{\lambda}| + |\Phi_{\lambda}|$, $\lambda \in \{1,\dots,m\}$. For the trivial $\F$, a global frame for $\F$ is \textit{declared} to be an empty collection of elements of $\F(X)$. 
\end{definice} 
\begin{tvrz} \label{tvrz_globalframe}
Let $\F \in \Sh^{\A}(X,\gVect)$. Then $\F$ is freely and finitely generated, if and only if there exists some global frame for $\F$. 
\end{tvrz}
\begin{proof}
The trivial $\F$ is always freely and finitely generated, since it is isomorphic to $\A[0]$. By definition, there also always exists an (empty) global frame for $\F$. Hence assume that $\F$ is non-trivial. Let $\F$ be freely and finitely generated. Let $\varphi: \F \rightarrow \A[K]$ be an isomorphism for some finite-dimensional $K \in \gVect$. Fix a total basis $( \vartheta_{\lambda})_{\lambda=1}^{m}$ of $K$. Note that $K \neq 0$, otherwise $\F$ would be trivial. Hence $m \geq 1$, and one can define
\begin{equation}
\Phi_{\lambda} := \varphi^{-1}_{X}( 1 \otimes \vartheta_{\lambda}),
\end{equation}
for each $\lambda \in \{1,\dots,m\}$. Let $s \in \F(U)$. We can write $\varphi_{U}(s) = a^{\lambda} \otimes \vartheta_{\lambda} = a^{\lambda} \tr (1 \otimes \vartheta_{\lambda})$ for unique $a^{\lambda} \in \A(U)$ with $|s| = |a^{\lambda}| + |\vartheta_{\lambda}|$, for each $\lambda \in \{1,\dots,m\}$. Whence
\begin{equation}
s = \varphi_{U}^{-1}( a^{\lambda} \tr (1 \otimes \vartheta_{\lambda})) = a^{\lambda} \tr \varphi_{U}^{-1}(1 \otimes \vartheta_{\lambda}) = a^{\lambda} \tr \Phi_{\lambda}|_{U},
\end{equation}
where we have used the $\A$-linearity and the naturality of $\varphi$. The uniqueness of the decomposition follows from the one of the decomposition of $\varphi_{U}(s)$. Hence $(\Phi_{\lambda})_{\lambda=1}^{m}$ is a global frame for $\F$. 

Conversely, if $(\Phi_{\lambda})_{\lambda=1}^{m}$ is a global frame, consider any graded vector space $K$ with a total basis $(\vartheta_{\lambda})_{\lambda=1}^{m}$ satisfying $|\vartheta_{\lambda}| = |\Phi_{\lambda}|$ for all $\lambda \in \{1,\dots,m\}$. By assumption, each $s \in \F(U)$ can be uniquely decomposed as $s = a^{\lambda} \tr \Phi_{\lambda}|_{U}$. For each $U \in \Op(X)$, define
\begin{equation} \varphi_{U}(s) := a^{\lambda} \otimes \vartheta_{\lambda}.
\end{equation}
It is easy to see that this gives an $\A$-linear sheaf isomorphism $\varphi: \F \rightarrow \A[K]$.
\end{proof}

\begin{rem}
For general $\A$ and $\F \in \Sh^{\A}_{\fin}(X,\gVect)$, the graded dimension of $K$ is not determined uniquely. Indeed, it suffices to find an invertible element $a \in \A(X)$ with $|a| \neq 0$. Let $(\Phi_{\lambda})_{\lambda=1}^{m}$ be a global frame for $\F$. Then $(a \tr \Phi_{\lambda})_{\lambda=1}^{m}$ is also a global frame for $\F$. As in the proof of the above proposition, one can use them to find isomorphisms of $\F$ with $\A[K]$ and $\A[K']$, where $K$ and $K'$ are of different graded dimensions. In other words, for general $\A$, there is no well-defined notion of a \textit{graded rank} of $\F$. 
\end{rem}

\begin{example} \label{ex_divnejsheaf}
Let us argue that there actually exists some sheaf of graded commutative associative algebras with an invertible global section of a non-zero degree. 

First, consider $V \in \gVect$ with $V_{-2} = V_{2} = \R$ and $V_{k} = \{0\}$ for $k \notin \{-2,2\}$. Let $(\xi_{1},\xi_{2})$ be its total basis with $|\xi_{1}| = -2$ and $|\xi_{2}| = 2$. Let $P \subseteq S(V)$ be a graded subset of the corresponding symmetric algebra, where $P_{0} := \{ 1 - \xi_{1}\xi_{2} \}$ and $P_{k} := \emptyset$ for $k \neq 0$. It follows that the graded commutative associative algebra $A = S(V) / \< P \>$ has invertible elements of a non-zero degree, e.g. $\hat{\xi}_{1} := q(\xi_{1}) \in A_{-2}$, where $q: S(V) \rightarrow A$ is the quotient map. Indeed, if $\hat{\xi}_{2} := q(\xi_{2}) \in A_{2}$, one finds 
\begin{equation} \hat{\xi}_{1} \cdot \hat{\xi}_{2} = q( \xi_{1} \xi_{2}) = q(1) = 1_{A}. \end{equation}
Next, to each $A \in \gcAs$, one can assign the corresponding \textbf{constant sheaf} $A_{X}$: For each $U \in \Op(X)$, one defines $A_{X}(U) = \{ \C^{\lc}(U, A_{k}) \}_{k \in \Z}$, where $\C^{\lc}(U,A_{k})$ denotes the set of locally constant functions from $U$ to $A_{k}$. All operations on $A_{X}(U)$ are defined point-wise using the graded algebra structure on $A$. Restriction morphisms $(A_{X})^{U}_{V}: A_{X}(U) \rightarrow A_{X}(V)$ are defined to be the actual restrictions. Whence $A_{X} \in \Sh(X,\gcAs)$. 

Finally, each element of $A$ can be viewed as a constant function, hence a section of $A_{X}(U)$ for any $U \in \Op(X)$. In particular, the invertible element $\hat{\xi}_{1} \in A_{-2}$ defines an invertible section of degree $-2$ in $A_{X}(X)$. This shows that finitely and freely generated sheaves of graded $A_{X}$-modules do not have a well-defined graded rank. 
\end{example}

Fortunately, there is an important class of examples, where sheaves of graded $\A$-modules \textit{do have} a well-defined graded rank. 
\begin{tvrz} \label{tvrz_thereisrank}
Let $\A \in \Sh(X,\gcAs)$. Suppose there is $U \in \Op(X)$ and an ideal $J \subseteq \A(U)$, such that $\A(U) / J \cong \R$. Suppose $\A[K] \cong \A[K']$ for some finite-dimensional $K,K' \in \gVect$. Then $K$ and $K'$ are isomorphic graded vector spaces. 

In particular, for any $\F \in \Sh^{\A}_{\fin}(X,\gVect)$, one can define its \textbf{graded rank} as $\grk(\F) := \gdim(K)$, where $K \in \gVect$ is any of the graded vector spaces such that $\F \cong \A[K]$. Two finitely and freely generated sheaves of $\A$-modules of the same graded rank are isomorphic. 
\end{tvrz}

\begin{proof}
Let $\varphi: \A[K] \rightarrow \A[K']$. For any ideal $J \subseteq \A(U)$, we have a graded submodule $J \otimes_{\R} K \subseteq \A(U) \otimes_{\R} K$. It is then easy to see that that the $\A$-linearity of $\varphi$ implies $\varphi_{U}(J \otimes_{\R} K) \subseteq J \otimes_{\R} K'$. There is also a graded vector space isomorphism $(\A(U) \otimes_{\R} K) / (J \otimes_{\R} K) \cong (\A(U) / J) \otimes_{\R} K$. It follows there is an induced isomorphism $\hat{\varphi}_{U}: (\A(U) / J) \otimes_{\R} K \rightarrow (\A(U) / J) \otimes_{\R} K'$. The rest follows from the fact that $\R$ is the unit in the monoidal category $(\gVect,\otimes_{\R})$. 
\end{proof}

\begin{rem}
It can happen that $\A$ itself is a trivial sheaf. It certainly cannot satisfy the assumptions of the above proposition. However, all $\F \in \Sh^{\A}_{\fin}(X,\gVect)$ are then also trivial sheaves and we can simply \textit{declare} $\grk(\F) := (0)_{j \in \Z}$. 
\end{rem}

\begin{example}
One can view any $\A \in \gcAs$ as a sheaf of graded $\A$-modules. If $\A$ is non-trivial, one can define a global frame to be the unit $1 \in \A(X)$, hence $\A \cong \A[\R]$. If the graded rank is well-defined, one has $\grk(\A) = (m_{j})_{j \in \Z}$, where $m_{0} = 1$ and $m_{j} = 0$ for $j \neq 0$. If $\A$ is trivial, we have $\grk(\A) = (0)_{j \in \Z}$ by definition. 
\end{example}

\begin{tvrz} \label{tvrz_dualfinfree}
Let $\F \in \Sh^{\A}_{\fin}(X,\gVect)$. Then its dual $\F^{\ast}$ is also freely and finitely generated. 
\end{tvrz}
\begin{proof}
Let $\F \in \Sh^{\A}_{\fin}(X,\gVect)$. The case of trivial $\F$ is obvious, hence let $\F$ be non-trivial. Fix an isomorphism $\varphi: \F \rightarrow \A[K]$ and a total basis $(\vartheta_{\lambda})_{\lambda=1}^{m}$ of $K$, so we have the corresponding global frame $( \Phi_{\lambda} )_{\lambda=1}^{m}$ for $\F$. 

Let us construct a sheaf isomorphism $\varphi^{\vee}: \F^{\ast} \rightarrow \A[K^{\ast}]$. Let $(\vartheta^{\lambda})_{\lambda =1}^{m}$ be the total basis of $K^{\ast}$ dual to $(\vartheta_{\lambda})_{\lambda=1}^{m}$. Note that $|\vartheta^{\lambda}| = - |\vartheta_{\lambda}|$ for each $\lambda \in \{1,\dots,m\}$. Let $U \in \Op(X)$ and $\alpha \in \F^{\ast}(U) \equiv \Sh^{\A|_{U}}(\F|_{U},\A|_{U})$. Then define $\varphi^{\vee}_{U}(\alpha) \in \A(U) \otimes_{\R} K^{\ast}$ as 
\begin{equation} \label{eq_varphivee}
\varphi^{\vee}_{U}( \alpha) := \alpha_{U}( \Phi_{\lambda}|_{U}) \otimes \vartheta^{\lambda}. 
\end{equation}
It is easy to see that $\varphi^{\vee}_{U}: \F^{\ast}(U) \rightarrow \A[K^{\ast}](U)$ defines a graded $\A(U)$-linear map. The naturality in $U$ is also clear. Hence $\varphi^{\vee}: \F^{\ast} \rightarrow \A[K^{\ast}]$ is a morphism of sheaves of graded $\A$-modules. 

Let us prove that each $\varphi^{\vee}_{U}$ is injective. Let $\varphi^{\vee}_{U}(\alpha) = 0$. But then $\alpha_{U}(\Phi_{\lambda}|_{U}) = 0$ for every $\lambda$. The naturality of $\alpha$ implies $\alpha_{V}(\Phi_{\lambda}|_{V}) = \alpha_{U}(\Phi_{\lambda}|_{U})|_{V} = 0$ for all $V \in \Op(U)$ and every $\lambda$. The fact that each $\alpha_{V}$ is $\A(V)$-linear and $(\Phi_{\lambda})_{\lambda=1}^{m}$ is a global frame implies $\alpha_{V} = 0$, for each $V \in \Op(X)$. Thus $\alpha = 0$ and the injectivity of $\varphi^{\vee}_{U}$ follows. 

Let us prove the surjectivity of $\varphi^{\vee}_{U}$. A given element $t \in \A(U) \otimes K^{\ast}$ can be decomposed as $t = t_{\lambda} \otimes \vartheta^{\lambda}$ for unique sections $t_{\lambda} \in \A(U)$ satisfying $|t_{\lambda}| + |\vartheta^{\lambda}| = |t|$ for each $\lambda$. For any $V \in \Op(U)$ and any $s = a^{\lambda} \tr \Phi_{\lambda}|_{V} \in \F(V)$, define 
\begin{equation}
\alpha_{V}(s) = \alpha_{V}( a^{\lambda} \tr \Phi_{\lambda}|_{V}) = (-1)^{|a^{\lambda}||t|} a^{\lambda} \cdot \alpha_{V}( \Phi_{\lambda}|_{V}) := (-1)^{|a^{\lambda}||t|} a^{\lambda} \cdot t_{\lambda}|_{V}.
\end{equation}
It is easy to verify that $\alpha_{V}$ are graded $\A(V)$-linear maps of degree $|t|$, natural in $V$. They thus define an element $\alpha \in \Sh^{\A|_{U}}(\F|_{U},\A|_{U}) \equiv \F^{\ast}(U)$ of degree $|t|$. One finds that
\begin{equation}
\varphi^{\vee}_{U}(\alpha) = \alpha_{U}(\Phi_{\lambda}|_{U}) \otimes \vartheta^{\lambda} = t_{\lambda} \otimes \vartheta^{\lambda} = t,
\end{equation}
proving that $\varphi^{\vee}_{U}$ is surjective. Hence $\varphi^{\vee}: \F^{\ast} \rightarrow \A[K^{\ast}]$ is an isomorphism. 
\end{proof}
Let $(\Phi^{\lambda})_{\lambda = 1}^{m}$ denote the global frame for $\F^{\ast}$ corresponding to the isomorphism $\varphi^{\vee}$ and the dual total basis $(\vartheta^{\lambda})_{\lambda=1}^{m}$ of $K^{\ast}$. It is called the \textbf{dual global frame to $(\Phi_{\lambda})_{\lambda = 1}^{m}$}. For each $\lambda \in \{1, \dots, m\}$ and $U \in \Op(X)$, the action of $\Phi^{\lambda}_{U}: \F(U) \rightarrow \A(U)$ on $s = a^{\kappa} \tr \Phi_{\kappa}|_{U} \in \F(U)$ is given by
\begin{equation}
\Phi^{\lambda}_{U}(s) = \Phi^{\lambda}_{U}( a^{\kappa} \tr \Phi_{\kappa}|_{U} ) = (-1)^{|a^{\kappa}||\vartheta^{\lambda}|} a^{\kappa} \delta_{\kappa}^{\lambda} = (-1)^{|a^{\lambda}||\vartheta^{\lambda}|} a^{\lambda}.
\end{equation}
Note that $|a^{\lambda}| = |s| - |\vartheta_{\lambda}| = |s| + |\vartheta^{\lambda}|$, so one can write $\Phi^{\lambda}_{U}(s) = (-1)^{(|s| + 1)(|\vartheta^{\lambda})} a^{\lambda}$. In particular,
\begin{equation}
\Phi^{\lambda}_{U}( \Phi_{\kappa}|_{U}) = \delta^{\lambda}_{\kappa}.
\end{equation}
One can use this notation to prove the following proposition:
\begin{tvrz} \label{tvrz_doubledual}
For $\F \in \Sh^{\A}_{\fin}(X,\gVect)$, there exists a canonical isomorphism $\F \cong (\F^{\ast})^{\ast}$. Moreover, if the graded rank is well-defined and $\grk(\F) = (m_{j})_{j \in \Z}$, then $\grk(\F^{\ast}) = (m_{-j})_{j \in \Z}$. 
\end{tvrz}
\begin{proof}
We will construct an $\A$-linear isomorphism $\eta: \F \rightarrow (\F^{\ast})^{\ast}$. For each $U \in \Op(X)$, we need a graded $\A(U)$-linear map $\eta_{U}: \F(U) \rightarrow (\F^{\ast})^{\ast}(U) \equiv \Sh^{\A|_{U}}(\F^{\ast}|_{U}, \A|_{U})$. Let $s \in \F(U)$. 

For each $V \in \Op(U)$, we must thus define an $\A(V)$-linear graded map $(\eta_{U}(s))_{V}: \F^{\ast}(V) \rightarrow \A(V)$ of degree $|s|$. Let $\alpha \in \F^{\ast}(V) \equiv \Sh^{\A|_{V}}(\F|_{V}, \A|_{V})$. Set 
\begin{equation}
(\eta_{U}(s))_{V}(\alpha) := (-1)^{|s||\alpha|} \alpha_{V}(s|_{V}).
\end{equation}
The sign is in accordance with the ``Koszul convention''. Moreover, it ensures that everything fits together. First, one has $|\alpha_{V}(s)| = |\alpha| + |s|$ and for any $b \in \A(V)$, one finds 
\begin{equation} \begin{split}
(\eta_{U}(s))_{V}(b \tr \alpha) = & \ (-1)^{|s|(|\alpha| + |b|)} (b \tr \alpha)_{V}(s|_{V}) = (-1)^{|s||b|} b \tr (-1)^{|s||\alpha|} \alpha_{V}(s|_{V}) \\
= & \ (-1)^{|s||b|} b \tr (\eta_{U}(s))_{V}(\alpha).
\end{split}
\end{equation}
This proves that $(\eta_{U}(s))_{V}$ is graded $\A(V)$-linear of degree $|s|$. The naturality of $\alpha$ then implies that the collection $\{ (\eta_{U}(s))_{V} \}_{V \in \Op(U)}$ defines $\eta_{U}(s) \in \Sh^{\A|_{U}}(\F^{\ast}|_{U}, \A|_{U})$. Next, one has
\begin{equation}
\begin{split}
(\eta_{U}(b \tr s))_{V}(\alpha) = & \ (-1)^{|b \tr s||\alpha|} \alpha_{V}( (b \tr s)|_{V}) = (-1)^{(|b| + |s|)|\alpha|} \alpha_{V}( b|_{V} \tr s|_{V}) \\
= & \  b|_{V} \tr (-1)^{|s||\alpha|} \alpha_{V}(s|_{V}) = b|_{V} \tr (\eta_{U}(s))_{V}(\alpha) = (b \tr \eta_{U}(s))_{V}(\alpha),
\end{split}
\end{equation}
for all $b \in \A(U)$, $V \in \Op(U)$, and $\alpha \in \F^{\ast}(V)$. This proves that $\eta_{U}$ is $\A(U)$-linear. The naturality in $U$ is trivial, whence $\eta: \F \rightarrow (\F^{\ast})^{\ast}$ is an $\A$-linear sheaf morphism. 

To prove that it is actually an isomorphism, we shall use the global frames. Pick a global frame $(\Phi_{\lambda})_{\lambda=1}^{m}$ for $\F$, its dual $(\Phi^{\lambda})_{\lambda=1}^{m}$ for $\F^{\ast}$, and the dual of the dual $(\Phi_{\lambda}^{\sharp})_{\lambda=1}^{m}$ for $(\F^{\ast})^{\ast}$. Then for each $\lambda, \kappa \in \{1,\dots,m\}$, one can write 
\begin{equation}
(\eta_{U}(\Phi_{\lambda}|_{U}))_{V}( \Phi^{\kappa}|_{V}) = (-1)^{|\vartheta_{\lambda}||\vartheta^{\kappa}|} \Phi^{\kappa}_{V}( \Phi_{\lambda}|_{V}) = (-1)^{|\vartheta_{\lambda}||\vartheta^{\kappa}|} \delta^{\kappa}_{\lambda} = (-1)^{|\vartheta_{\lambda}|} \delta_{\lambda}^{\kappa}. 
\end{equation}
But this means that $\eta_{U}( \Phi_{\lambda}|_{U}) = (-1)^{|\vartheta_{\lambda}|} \Phi_{\lambda}^{\sharp}|_{U}$. This certainly proves that $\eta_{U}$ is an isomorphism. The claim about graded ranks follows from the fact that $|\vartheta_{\lambda}| = -|\vartheta^{\lambda}|$ for all $\lambda \in \{1,\dots,m\}$. 
\end{proof}
\begin{rem} \label{rem_peculiardoubledual}
In the above proof, we have discovered a certain peculiarity. Under the correspondence $\F \cong (\F^{\ast})^{\ast}$, a given global frame $(\Phi_{\lambda})_{\lambda=1}^{m}$ is not mapped precisely onto the dual of the dual frame $(\Phi_{\lambda}^{\sharp})_{\lambda=1}^{m}$, but some signs appear. One has to be a bit careful about that. 
\end{rem}
A lot of the above properties are in some sense compatible with sheaf restrictions.
\begin{tvrz} \label{tvrz_finfreegenlocal}
Let $\F \in \Sh^{\A}_{\fin}(X,\gVect)$ be a non-trivial freely and finitely generated sheaf of graded $\A$-modules and let $U \in \Op(X)$ be arbitrary. Then the following observations are true:
\begin{enumerate}[(i)]
\item $\F|_{U} \in \Sh^{\A|_{U}}_{\fin}(U,\gVect)$.
\item Let $(\Phi_{\lambda})_{\lambda=1}^{m}$ be a global frame for $\F$. Suppose $\F|_{U}$ is non-trivial. Then $(\Phi_{\lambda}|_{U})_{\lambda=1}^{m}$ forms a global frame for $\F|_{U}$. 
\item If the graded rank is well-defined and $\F|_{U}$ is non-trivial, one has $\grk(\F|_{U}) = \grk(\F)$.
\item One has $\ul{\Sh}^{\A}(\F,\G)|_{U} = \ul{\Sh}^{\A|_{U}}(\F|_{U},\G|_{U})$. In particular, one has $(\F|_{U})^{\ast} = \F^{\ast}|_{U}$. 
\end{enumerate}
\end{tvrz}
\begin{proof}
Let $\varphi: \F \rightarrow \A[K]$ be an isomorphism of sheaves of graded $\A$-modules. It is obvious that $\varphi|_{U}: \F|_{U} \rightarrow \A[K]|_{U} = \A|_{U}[K]$ is an isomorphism for any $U \in \Op(X)$. This proves $(i)$. The rest of the statements is obvious, except that in $(ii)$ and $(iii)$, one has to exclude the case when $\F|_{U}$ is trivial. Indeed, in this case $(\Phi_{\lambda}|_{U})_{\lambda = 1}^{m}$ no longer is a global frame for $\F|_{U}$ and, by definition, one has $\grk(\F|_{U}) = 0$. 
\end{proof}
\begin{definice}
We say that $\F \in \Sh^{\A}(X,\gVect)$ is \textbf{locally freely and finitely generated}, if for each $x \in X$, there is $U \in \Op_{x}(X)$, such that $\F|_{U} \in \Sh^{\A|_{U}}_{\fin}(U,\gVect)$. Any global frame $(\Phi_{\lambda})_{\lambda=1}^{m}$ for $\F|_{U}$ is called a \textbf{local frame for $\F$ over $U$}. Locally freely and finitely generated sheaves of graded $\A$-modules form a full subcategory $\Sh^{\A}_{\lfin}(X,\gVect) \subseteq \Sh^{\A}_{\fin}(X,\gVect)$. 
\end{definice}
\begin{rem} \label{rem_nontrivialmodulesif}
There is always a little nuisance with trivial graded modules. Let $\F \in \Sh^{\A}_{\fin}(X,\gVect)$ be non-trivial. As noted in the proof of Proposition \ref{tvrz_finfreegenlocal}, it can happen that $\F|_{U} = 0$ for some $U \in \Op(X)$. However, this can only be the case when also $\A|_{U} = 0$. As $\F$ is non-trivial, we have $\F|_{U} \cong \A|_{U}[K]$ for $K \neq 0$. For $\A|_{U} \neq 0$, one has $A(U) \neq 0$ by Remark \ref{rem_ShgcAs}-(ii). But then for any $k \neq 0$, the element $1 \otimes k \in \A(U) \otimes_{\R} K$ is non-zero and $\F|_{U}$ has to be non-trivial. 
\end{rem}

\begin{tvrz} \label{tvrz_locfinfreegensheaves}
Let $\F \in \Sh^{\A}_{\lfin}(X,\gVect)$.  Let $x \in X$. Suppose $\A(U) \neq 0$ for all $U \in \Op_{x}(X)$. Assume that the graded rank of all sheaves of graded $\A|_{U}$-modules is well-defined for any $U \in \Op_{x}(X)$. Then one can define the \textbf{graded rank of $\F$ at $x$} as $\grk_{x}(\F) := \grk(\F|_{U})$, where $U \in \Op_{x}(X)$ is any neighborhood such that $\F|_{U} \in \Sh_{\fin}^{\A|_{U}}(U,\gVect)$. 

Suppose that the graded rank of $\F$ at $x$ is well defined for all $x \in X$. Then it is constant on every path-connected component of $X$.
\end{tvrz}
\begin{proof}
First, note that for each $x \in X$, $\grk_{x}(\F)$ is well-defined. Indeed, let $U' \in \Op_{x}(X)$ be another neighborhood of $x$, such that $\F|_{U'}$ is freely and finitely generated. By Proposition \ref{tvrz_finfreegenlocal}-(i), $\F|_{U \cap U'}$ is also finitely and finitely generated. If $\F|_{U} \neq 0$, then also $\F|_{U \cap U'} \neq 0$ by Remark \ref{rem_nontrivialmodulesif} and the fact that $U \cap U' \neq \emptyset$ by assumption implies $\A(U \cap U') \neq 0$. It then follows from Proposition \ref{tvrz_finfreegenlocal}-(iii) that $\grk(\F|_{U}) = \grk(\F|_{U \cap U'}) = \grk(\F|_{U'})$. If $\F|_{U} = 0$, then also $\F|_{U \cap U'} = 0$. Suppose $\F|_{U'} \cong \A|_{U'}(K)$ for some finite-dimensional $K \in \gVect$. But then $0 = \F|_{U \cap U'} \cong \A|_{U \cap U'}[K]$ implies $K = 0$, whence $\F|_{U'} = 0$.  In this case thus $\grk(\F|_{U}) = \grk(\F|_{U'}) = (0)_{j \in \Z}$. 

Next, let $x$ and $y$ be in the same path connected component. There is thus some continuous path $\gamma: [0,1] \rightarrow X$ connecting $x$ to $y$. For each $t \in [0,1]$, one can find $U_{t} \in \Op_{\gamma(t)}(X)$, such that $\F|_{U_{t}}$ is freely and finitely generated. Using the usual compactness argument, one can find a finite subdivision $0 = t_{0} < \dots < t_{n} = 1$ such that $U_{t_{k}} \cap U_{t_{k+1}} \neq \emptyset$ for all $k \in \{0,\dots,n-1\}$. Using the same argument as in the previous paragraph, we have $\grk(\F|_{U_{t_k}}) = \grk(\F|_{U_{t_{k+1}}})$ for all $k \in \{0,\dots,n-1\}$. Hence $\grk_{x}(\F) = \grk_{y}(\F)$ and the proof is finished. 
\end{proof}

Let us continue this subsection by looking at the stalks of sheaves of graded $\A$-modules. 
\begin{tvrz} \label{tvrz_shofAmodulesstalks}
Let $\F \in \Sh^{\A}(X,\gVect)$ and fix $x \in X$. 
\begin{enumerate}[(i)]
\item Its stalk $\F_{x}$ at $x$ has a natural structure of a graded $\A_{x}$-module, such that for each $U \in \Op_{x}(X)$, $a \in \A(U)$ and $s \in \F(U)$, one has $[a]_{x} \tr [s]_{x} = [a \tr s]_{x}$. 

Moreover, the universal property can be stated as follows: Let $K$ be any graded $\A_{x}$-module and suppose $\{ \tau_{U} \}_{U \in \Op_{x}(X)}$ is a collection of graded linear maps $\tau_{U}: \F(U) \rightarrow K$ satisfying $\tau_{U} = \tau_{V} \circ \F^{U}_{V}$ for all $U,V \in \Op_{x}(X)$, such that $V \subseteq U$. Moreover, suppose that for any $U \in \Op_{x}(X)$, $s \in \F(U)$ and $a \in \F(U)$, one has $\tau_{U}(a \tr s) = [a]_{x} \tr \tau_{U}(s)$. 

Then there is a unique graded $\A_{x}$-linear map $\varphi: \F_{x} \rightarrow K$ such that $\varphi \circ \pi_{U,x} = \tau_{U}$. 
\item Suppose $\F$ is locally freely and finitely generated. Then $\F_{x}$ is freely and finitely generated. If the graded rank is well defined and $\A_{x} \neq 0$, then $\grk(\F_{x}) = \grk_{x}(\F)$. 

Moreover, for any graded local frame $(\Phi_{\lambda})_{\lambda=1}^{m}$ for $\F$ over $U \in \Op_{x}(X)$, the collection $([\Phi_{\lambda}]_{x})_{\lambda=1}^{m}$ forms a frame for $\F_{x}$. 
\end{enumerate}
\end{tvrz}
\begin{proof}
Let $[a]_{x} \in \A_{x}$ and $[s]_{x} \in \F_{x}$. We may assume that $a \in \A(U)$ and $s \in \F(U)$ for the same $U \in \Op_{x}(X)$. We define $[a]_{x} \tr [s]_{x} := [a \tr s]_{x}$, where we use the $\A(U)$-module structure of $\F(U)$. It is easy to see that the properties of $\F$ described in Definition \ref{def_sheavesofAmodules} ensure that this is well-defined and makes $\F_{x}$ into a graded $\A_{x}$-module. The universal property of $\F_{x}$ can be obtained as an easy modification of the colimit property (ii) in the proof of Proposition \ref{tvrz_stalk}. This proves $(i)$. 

To prove $(ii)$, note that for any $K \in \gVect$, one can identify the $\A_{x}$-module $(\A[K])_{x}$ with $\A_{x}[K] := \A_{x} \otimes_{\R} K$. Indeed, it is easy to see that the collection $\{ \pi_{U,x}^{\A} \otimes \1_{K} \}_{U \in \Op_{x}(X)}$ is the universal colimiting cone, where $\pi^{\A}_{U,x}: \A(U) \rightarrow \A_{x}$ are the graded algebra morphisms of the universal colimiting cone for $\A$. In other words, one has 
\begin{equation}
[a \otimes k]_{x} := [a]_{x} \otimes k,
\end{equation}
for all $a \in \A(U)$, $k \in K$ and $U \in \Op_{x}(X)$. The graded $\A_{x}$-module structure obtained in the part $(i)$ coincides with the obvious one on $\A_{x}[K]$. 

Now, we say that $V \in \AgMod$ is freely and finitely generated, if it is isomorphic to $A[K] = A \otimes_{\R} K$ in $\AgMod$ for some finite-dimensional $K \in \gVect$.
The frame for a freely and finitely generated $A$-module $V$ is a collection $(\phi_{\lambda})_{\lambda = 1}^{m}$ of elements of $V$, such that every $v \in V$ can be written as $v = a^{\lambda} \tr \phi_{\lambda}$ for unique elements $a^{\lambda} \in A$ satisfying $|a^{\lambda}| + |\phi_{\lambda}| = |v|$. Trivial $A$-modules have a frame consisting of no elements. Similarly to the proof of Proposition \ref{tvrz_globalframe}, every isomorphism $V \cong A[K]$ and a choice of a total basis $(\vartheta_{\lambda})_{\lambda = 1}^{m}$ of $K$ gives raise to a unique frame $(\phi_{\lambda})_{\lambda = 1}^{m}$, where $\phi_{\lambda}$ are isomorphic images of $1 \otimes \vartheta_{\lambda} \in A[K]$. 

 Now, suppose that $\F \in \Sh^{\A}_{\lfin}(X,\gVect)$. There is thus some $U \in \Op_{x}(X)$, such that $\F|_{U}$ is finitely and freely generated. In other words, one finds a finite-dimensional $K \in \gVect$ and an $\A|_{U}$-linear sheaf morphism $\varphi: \F|_{U} \rightarrow \A|_{U}[K]$. Let $\varphi_{x}: (\F|_{U})_{x} \rightarrow (\A|_{U}[K])_{x}$ be the induced isomorphism of stalks. But $(\F|_{U})_{x} = \F_{x}$, and $(\A|_{U}[K])_{x} = \A_{x}[K]$ by the previous paragraph. It is easy to see that $\varphi_{x}$ is $\A_{x}$-linear, hence $\F_{x}$ is freely and finitely generated. 

Next, observe that the assumption $\A_{x} \neq 0$ is equivalent to $\A(U) \neq 0$ for all $U \in \Op_{x}(X)$. Indeed, one has $\A_{x} \neq 0$, iff $[1_{U}]_{x} \neq 0$ for all $U \in \Op_{x}(X)$, which is equivalent to $1_{U} \neq 0$ for all $U \in \Op_{x}(X)$. Note that similarly to $\Sh^{\A}_{\fin}(X,\gVect)$, the notion of graded rank is in general not well-defined in $\AgMod$. However, suppose that it is well-defined. Let $U \in \Op_{x}(X)$, such that $\F|_{U}$ is finitely and freely generated, and fix an isomorphism $\varphi: \F|_{U} \rightarrow \A|_{U}[K]$. Note that $\grk_{x}$ is well-defined (see the proof of Proposition \ref{tvrz_locfinfreegensheaves}) and since $\A_{x} \neq 0$, it follows immediately from the previous paragraph that $\grk_{x}(\F) = \grk(\F_{x}) := \dim(K)$. 

Finally, let $(\Phi_{\lambda})_{\lambda = 1}^{m}$ be a local frame for $\F$ over $U$. Let us first assume that $\F|_{U} \neq 0$. There is thus an induced $\A|_{U}$-linear sheaf isomorphism $\varphi: \F|_{U} \rightarrow \A|_{U}[K]$, where $K$ is some \textit{non-zero} finite-dimensional graded vector space. By construction, one has $\varphi_{U}( \Phi_{\lambda}) = 1 \otimes \vartheta_{\lambda}$, where $(\vartheta_{\lambda})_{\lambda = 1}^{m}$ is some total basis of $K$. Let $\varphi_{x}: \F_{x} \rightarrow \A_{x}[K]$ be the isomorphism of \textit{non-trivial} $\A_{x}$-modules as constructed above. We can then construct a frame for $\F_{x}$ as $\phi_{\lambda} := \varphi_{x}^{-1}(1 \otimes \vartheta_{\lambda}) \equiv [\varphi_{U}^{-1}(1 \otimes \vartheta_{\lambda})]_{x} = [\Phi_{\lambda}]_{x}$, for all $\lambda \in \{1,\dots,m\}$. For $\F|_{U} = 0$, the statement is trivial. This concludes the proof.\end{proof}

\begin{tvrz}
Let $\F \in \Sh^{\A}_{\lfin}(X,\gVect)$. Then $(\F^{\ast})_{x} \cong (\F_{x})^{\ast}$ for all $x \in X$. Moreover, this implies that $\eta: \F \rightarrow (\F^{\ast})^{\ast}$ constructed in the proof of Proposition \ref{tvrz_doubledual} is a sheaf isomorphism even for locally freely and finitely generated sheaves of graded $\A$-modules. 
\end{tvrz}
\begin{proof}
Let $x \in X$ by fixed, but otherwise arbitrary. Let $U \in \Op_{x}(X)$ and $\alpha \in \F^{\ast}(U) \equiv \Sh^{\A|_{U}}(\F|_{U},\A|_{U})$. We thus have the induced graded linear map $\alpha_{x}: \F_{x} \rightarrow \A_{x}$ of degree $|\alpha|$, obtained in the same way as in Corollary \ref{cor_inducedstalkmap}. It is easy to see that it is $\A_{x}$-linear, that is $\alpha_{x} \in \Lin^{\A_{x}}(\F_{x},\A_{x}) \equiv (\F_{x})^{\ast}$. We thus have a graded linear map $\tau_{U}: \F^{\ast}(U) \rightarrow (\F_{x})^{\ast}$ given by $\tau_{U}(\alpha) := \alpha_{x}$. For any $a \in \A(U)$, one has $\tau_{U}(a \tr \alpha) = [a]_{x} \tr \tau_{U}(\alpha)$ and for any $V \in \Op_{x}(U)$, we find $\tau_{V}(\alpha|_{V}) = \tau_{U}(\alpha)$. By the universal property in Proposition \ref{tvrz_shofAmodulesstalks}-$(i)$, there is thus a unique graded $\A_{x}$-linear map $\varphi: (\F^{\ast})_{x} \rightarrow (\F_{x})^{\ast}$ such that $\varphi([\alpha]_{x}) = \tau_{U}(\alpha) \equiv \alpha_{x}$. 

Now, let $(\Phi_{\lambda})_{\lambda = 1}^{m}$ be a local frame for $\F$ over some $U \in \Op_{x}(X)$. Let $(\Phi^{\lambda})_{\lambda=1}^{m}$ be its dual. For each $\lambda \in \{1,\dots,m\}$, one has $\varphi([\Phi^{\lambda}]_{x})( [\Phi_{\kappa}]_{x}) = \Phi^{\lambda}_{x}([\Phi_{\kappa}]_{x}) = [\Phi^{\lambda}_{U}(\Phi_{\kappa})]_{x} = \delta^{\lambda}_{\kappa}$. But this proves that $(\varphi([\Phi^{\lambda}]_{x}))_{\lambda=1}^{m}$ is a frame for $(\F_{x})^{\ast}$ dual to $([\Phi_{\lambda}]_{x})_{\lambda=1}^{m}$, see Proposition \ref{tvrz_shofAmodulesstalks}-$(ii)$. As $([\Phi^{\lambda}]_{x})_{\lambda=1}^{m}$ is a frame for $(\F^{\ast})_{x}$, this proves that $\varphi: (\F^{\ast})_{x} \rightarrow (\F_{x})^{\ast}$ is an isomorphism of graded $\A_{x}$-modules. This proves the first part of the statement and shows that one can canonically identify the germ $[\alpha]_{x} \in (\F^{\ast})_{x}$ of $\alpha \in \F^{\ast}(U)$ with the induced map $\alpha_{x} \in (\F_{x})^{\ast}$. 

Now, let $\eta: \F \rightarrow (\F^{\ast})^{\ast}$ be an $\A$-linear sheaf morphism from Proposition \ref{tvrz_doubledual}. Choose $s \in \F(U)$ and $\alpha \in \F^{\ast}(V)$ for $U \in \Op_{x}(X)$ and $V \in \Op_{x}(U)$. Then 
\begin{equation}
\begin{split}
(\eta_{x}([s]_{x}))([\alpha]_{x}) = & \ ( \eta_{U}(s))_{x}([\alpha]_{x}) = [(\eta_{U}(s))_{V}(\alpha)]_{x} =  [ (-1)^{|s||\alpha|} \alpha_{V}(s|_{V})]_{x} \\
= & \ (-1)^{|s||\alpha|} \alpha_{x}([s]_{x}).
\end{split}
\end{equation}
But this proves that the induced map $\eta_{x}: \F_{x} \rightarrow ((\F_{x})^{\ast})^{\ast}$ is precisely the canonical map into the graded double dual. As $\F_{x}$ is finitely and freely generated, it is an isomorphism of graded $\A_{x}$-modules by the same arguments as in the proof of Proposition \ref{tvrz_doubledual}. As $x \in X$ was arbitrary, this proves that $\eta$ is an isomorphism of sheaves of graded $\A$-modules, see Remark \ref{rem_inducedmapsheafiso}. 
\end{proof}
Having the canonical isomorphism with the double dual, one has a notion of a transpose sheaf morphism with all the usual properties. 
\begin{tvrz} \label{tvrz_transpose}
Let $\F,\G \in \Sh^{\A}_{\lfin}(X,\gVect)$. Let $\varphi \in \Sh^{\A}(\F,\G)$, see Definition \ref{def_ShAFG}. Then there is its \textbf{transpose} $\varphi^{T} \in \Sh^{\A}(\G^{\ast},\F^{\ast})$ having the following properties: 
\begin{enumerate}[(i)]
\item If $\psi \in \Sh^{\A}(\G,\H)$ for some $\H \in \Sh^{\A}(X,\gVect)$. Then $(\psi \circ \varphi)^{T}  = \varphi^{T} \circ \psi^{T}$ and $(\1_{\F})^{T} = \1_{\F^{\ast}}$.
\item Let $\eta^{\F}: \F \rightarrow (\F^{\ast})^{\ast}$ denote the canonical sheaf isomorphism (see Proposition \ref{tvrz_doubledual}). Then for any $\varphi \in \Sh^{\A}(\F,\G)$, the double transpose $(\varphi^{T})^{T}$ fits into the commutative diagram
\begin{equation}
\begin{tikzcd}
(\F^{\ast})^{\ast} \arrow{r}{(\varphi^{T})^{T}} & (\G^{\ast})^{\ast} \\
\F \arrow{r}{\varphi} \arrow{u}{\eta^{\F}} & \G \arrow{u}{\eta^{\G}}
\end{tikzcd}
\end{equation}
In other words, with respect to the identification $(\F^{\ast})^{\ast} \cong \F$, the double transpose $(\varphi^{T})^{T}$ corresponds to the original morphism $\varphi$. 
\item Let $U \in \Op(X)$. Then $(\varphi|_{U})^{T} = \varphi^{T}|_{U}$. 
\item The assignment $\varphi \mapsto \varphi^{T}$ defines a graded vector space isomorphism. 
\end{enumerate}
\end{tvrz}
\begin{proof}
Let $\varphi \in \Sh^{\A}(\F,\G)$. For each $U \in \Op(X)$ and $\alpha \in \G^{\ast}(U)$, we have to define an $\A|_{U}$-linear sheaf morphism $\varphi^{T}_{U}(\alpha): \F|_{U} \rightarrow \A|_{U}$ of degree $|\varphi| + |\alpha|$. For all $V \in \Op(U)$ and $s \in \F(V)$, define
\begin{equation}
(\varphi^{T}_{U}(\alpha))_{V}(s) := (-1)^{|\varphi||\alpha|} \alpha_{V}( \varphi_{V}(s)). 
\end{equation}
It is straightforward to verify the following properties:
\begin{enumerate}[(a)]
\item The map $(\varphi^{T}_{U}(\alpha))_{V}: \F(V) \rightarrow \A(V)$ is graded $\A(V)$-linear of degree $|\varphi| + |\alpha|$ and natural in $V$. Hence it defines $\varphi^{T}_{U}(\alpha) \in \Sh^{\A|_{U}}_{|\varphi|+|\alpha|}(\F|_{U},\A|_{U}) \equiv \F^{\ast}_{|\varphi|+|\alpha|}(U)$. 
\item $\varphi^{T}_{U}(\alpha)$ is graded $\A(U)$-linear in $\alpha$, hence it defines a graded $\A(U)$-linear map $\varphi^{T}_{U}: \G^{\ast}(U) \rightarrow \F^{\ast}(U)$ of degree $|\varphi|$. It is natural in $U$, hence it defines a $\A$-linear sheaf morphism $\varphi^{T}: \G^{\ast} \rightarrow \F^{\ast}$ of degree $|\varphi|$, $\varphi^{T} \in \Sh^{\A}_{|\varphi|}(\G^{\ast},\F^{\ast})$. 
\end{enumerate}
It is now easy to verify the claims $(i) - (iii)$. It remains to prove $(iv)$. Clearly, the assignment $\varphi \mapsto \varphi^{T}$ is graded linear. We only have to show that it is bijective. Let $\psi \in \Sh^{\A}(\G^{\ast},\F^{\ast})$. Then $\psi^{T} \in \Sh^{\A}( (\F^{\ast})^{\ast}, (\G^{\ast})^{\ast})$, so one can define $\varphi := (\eta^{\G})^{-1} \circ \psi^{T} \circ \eta^{\F}$. We claim that $\varphi^{T} = \psi$. To show that, apply $(i)$ and $(ii)$, together with the observation that $\eta^{\F^{\ast}} = (\eta^{\F})^{-T}$.
\end{proof}
\section{Graded manifolds} \label{sec_grman}
\subsection{Graded domains}
Let $(n_{j})_{j \in \Z}$ be a sequence of non-negative integers, such that $\sum_{j \in \Z} n_{j} < \infty$. Using the notation of Example \ref{ex_Rtosequence}, let us define a presheaf $\C^{\infty}_{(n_{j})} \in \PSh( \R^{n_{0}}, \gcAs)$ for each $U \subseteq \R^{n_{0}}$ as 
\begin{equation} \label{eq_grdomain}
\C^{\infty}_{(n_{j})}(U) := \bar{S}( \R^{(n_{j})}_{\ast}, \C^{\infty}_{n_{0}}(U)),
\end{equation}
where $\C^{\infty}_{n_{0}} \in \Sh(\R^{n_{0}}, \As)$ denotes the sheaf of ordinary smooth functions on $\R^{n_{0}}$. Assume that we have fixed a total basis $(\xi_{\mu})_{\mu=1}^{n_{\ast}}$ of $\R^{(n_{j})}_{\ast}$.

\begin{tvrz} \label{tvrz_gradeddomainisgLRS}
$(\R^{n_{0}}, \C^{\infty}_{(n_{j})}) \in \gLRS$. For each $U \subseteq \R^{n_{0}}$, its restriction $U^{(n_{j})} := (U, \C^{\infty}_{(n_{j})}|_{U})$ is called a \textbf{graded domain over $U$}. We will henceforth write just $\C^{\infty}_{(n_{j})}$ for the restricted sheaf. 
\end{tvrz}
\begin{proof}
We have $\C^{\infty}_{(n_{j})} \in \Sh(\R^{n_{0}}, \gcAs)$ as it is precisely the sheaf from Example \ref{ex_thesheaf} for $\F = \C^{\infty}_{n_{0}}$ and $K = \R^{(n_{j})}_{\ast}$. Moreover, we have shown in Example \ref{ex_twogLRS} that $(\R^{n_{0}}, \C^{\infty}_{(n_{j})}) \in \gLRS$, if and only if $(\R^{n_{0}}, \C^{\infty}_{n_{0}}) \in \gLRS$. It is easy to see that for each $x \in \R^{n_{0}}$, one has 
\begin{equation}
\C^{\infty}_{n_{0},x} - \frU( \C^{\infty}_{n_{0},x}) = \{ [f]_{x} \in \C^{\infty}_{n_{0},x} \; | \; f(x) = 0 \},
\end{equation}
which clearly forms an ideal in in $\C^{\infty}_{n_{0},x}$. The claim then follows from Proposition \ref{tvrz_grlocal}-$(iv)$.  
\end{proof}
Before we proceed, let us establish some nomenclature. Sections of $\C^{\infty}_{(n_{j})}(U)$ are usually called \textbf{functions on a graded domain}. To each function $f \in \C^{\infty}_{(n_{j})}(U)$, there is the ordinary smooth function $\ul{f} := \beta_{U}(f) \in \C^{\infty}_{n_{0}}(U)$ called the \textbf{body of the function $f$}. See Example \ref{ex_bodyofasection} for details. Note that $\ul{f} = 0$ for $|f| \neq 0$. For any $f \in \C^{\infty}_{(n_{j})}(U)$, the \textbf{value of $f$ at $x \in U$} is defined as $f(x) := \ul{f}(x)$. Hence always $f(x) = 0$ for $|f| \neq 0$. 
\begin{cor} \label{cor_JRgrdomain}
For each $x \in \R^{n_{0}}$, the Jacobson radical of the graded ring $\C^{\infty}_{(n_{j}),x}$ has the form 
\begin{equation}
\frJ( \C^{\infty}_{(n_{j}),x}) = \{ [f]_{x} \in \C^{\infty}_{(n_{j}),x}  \; | \; f(x) = 0 \}.
\end{equation}
\end{cor}
\begin{proof}
This follows immediately from the equation (\ref{eq_unitsOXxFxrel}) discussed in Example \ref{ex_twogLRS}. 
\end{proof}
Let us now introduce some two important shaves of ideals. First, recall that we have a sheaf morphism $\beta: \C^{\infty}_{(n_{j})} \rightarrow \C^{\infty}_{n_{0}}$, see Example \ref{ex_bodyofasection}. We can thus form its kernel sheaf
\begin{equation}
\J^{\pg}_{(n_{j})} := \ker(\beta) \subseteq \C^{\infty}_{(n_{j})},
\end{equation}
called the sheaf of ideals of \textbf{purely graded functions} of $\C^{\infty}_{(n_{j})}$. Note that there is an induced sheaf isomorphism $\hat{\beta}: \C^{\infty}_{(n_{j})} / \J^{\pg}_{(n_{j})} \rightarrow \C^{\infty}_{n_{0}}$. 
%For each $x \in \R^{n_{0}}$, we find
%\begin{equation}
%\J^{\pg}_{(n_{j}),x} = \ker(\beta_{x}) \equiv \{ [f]_{x} \in \C^{\infty}_{(n_{j}),x} \; | \; [\ul{f}]_{x} = 0 \},
%\end{equation}

Next, for each $x \in \R^{n_{0}}$, define the sheaf of ideals of \textbf{functions vanishing at $x$}. For $U \in \Op(X)$ such that $x \notin U$, let $\J^{x}_{(n_{j})}(U) := \C^{\infty}_{(n_{j})}(U)$. For $U \in \Op_{x}(\R^{n_{0}})$, set
\begin{equation}
\J^{x}_{(n_{j})}(U) := \{ f \in \C^{\infty}_{(n_{j})}(U) \; | \; f(x) = 0 \}. 
\end{equation}

It is straightforward to prove that $\J^{x}_{(n_{j})}$ forms a sheaf of ideals of $\C^{\infty}_{(n_{j})}$. Moreover, one has $\J^{\pg}_{(n_{j})} \subseteq \J^{x}_{(n_{j})}$, and  Corollary \ref{cor_JRgrdomain} implies that $\J^{x}_{(n_{j})}(U) = \pi_{U,x}^{-1}( \frJ( \C^{\infty}_{(n_{j}),x}))$ for every $U \in \Op_{x}(\R^{n_{0}})$.

We will now prove an important graded version of a similar statement from classical calculus, allowing us in some sense to approximate the functions (infinite formal power series) in the algebra $\C^{\infty}_{(n_{j})}(U)$ by finite polynomials. Let us first recall the classical result. 
\begin{lemma}[\textbf{Classical Hadamard}] \label{lem_Hadimrdclassic}
Let $f \in \C^{\infty}_{n_{0}}(U)$ for a given $U \in \Op(\R^{n_{0}})$. Then to any $a \in U$ and any $q \in \N_{0}$, there is a smooth map $h_{a}^{(q)}: U \rightarrow \Lin^{(q+1)}(\R^{n},\R)$, such that 
\begin{equation}
f(x) = f(a) + \sum_{k=1}^{q} \frac{1}{k!} [\fD^{k}f(a)](x-a)^{k} + [h_{a}^{(q)}(x)](x-a)^{q+1},
\end{equation} 
for all $x \in U$, where $\fD^{k}f: U \rightarrow \Lin^{(k)}(\R^{n},\R)$ is the $k$-th total derivative of $f$ and we use the shorthand notation $[h](x)^{k} := h(x,\dots,x)$ for any $x \in \R^{n}$ and any $k$-linear map $h \in \Lin^{(k)}(\R^{n},\R)$. 

We can thus write $f = t^{q}_{a}(f) + r^{q}_{a}(f)$, where $t^{q}_{a}(f)(x) := f(a) + \sum_{k=1}^{q} \frac{1}{k!} [\fD^{k}f(a)](x-a)^{k}$ is called the \textbf{degree $q$ Taylor polynomial of $f$ at $a$}.
\end{lemma}
\begin{proof}
Note that unlike in many calculus textbooks, $U$ is not assumed to be star-shaped. There is a version for arbitrary $U$ proved in \cite{duistermaat2004multidimensional}, but their $h_{a}^{(q)}$ is in general not differentiable at $a$. The version presented here is used without any reference in \cite{leites1980introduction}, \cite{carmeli2011mathematical} and \cite{fairon2017introduction}. To fill in the gap, the interested reader can find the proof of this lemma in the appendix, see Lemma \ref{lem_ap_Hadimrdclassic}. 
\end{proof}

One can now find a similar statement for functions on a graded domain. Its proof is a subtle modification of the one in \cite{fairon2017introduction}. 
\begin{lemma}[\textbf{Graded Hadamard}] \label{lem_Hadamard} Let $f \in \C^{\infty}_{(n_{j})}(U)$ for a given $U \in \Op(\R^{n_{0}})$. Let $(x^{1},\dots,x^{n_{0}})$ be the coordinates on $U$. Then to any $a \in U$ and any $q \in \N_{0}$, there is a decomposition 
\begin{equation}
f = T^{q}_{a}(f) + R^{q}_{a}(f),
\end{equation}
where $T^{q}_{a}(f)$ is a polynomial in variables $\{ (x^{i} - x^{i}(a)) \}_{i=1}^{n_{0}}$ and $\{ \xi_{\mu} \}_{\mu=1}^{n_{\ast}}$ of degree $q$, and $R_{a}^{q}(f) \in (\J^{a}_{(n_{j})}(U))^{q+1} \subseteq \C^{\infty}_{(n_{j})}(U)$. $T^{q}_{a}(f)$ is called a \textbf{degree $q$ Taylor polynomial of $f$ at $a$}. 
\end{lemma}
\begin{proof}
We have $f = \sum_{\fp \in \N^{n_{\ast}}_{|f|}} f_{\fp} \xi^{\fp}$. Observe that $\xi^{\fp}$ is a polynomial in variables $\{ \xi_{\mu} \}_{\mu=1}^{n_{\ast}}$ of degree $w(\fp) = \sum_{\mu=1}^{n_{\ast}} p_{\mu}$. For $\fp \in \N^{n_{\ast}}_{|f|}$ with $w(\fp) \leq q$, one use the classical Hadamard lemma to write each function $f_{\fp} \in \C^{\infty}_{n_{0}}(U)$ as 
\begin{equation}
f_{\fp} = t_{a}^{q - w(\fp)}(f_{\fp}) + r_{a}^{q - w(\fp)}(f_{\fp}),
\end{equation}
where $t^{q-w(\fp)}_{q}(f_{\fp})$ is a polynomial in $\{ (x^{i} - x^{i}(a)) \}_{i=1}^{n_{0}}$ of degree $q - w(\fp)$. Hence define
\begin{equation}
T_{a}^{q}(f) := \sum_{\substack{\fp \in \N_{|f|}^{n_{\ast}} \\ w(\fp) \leq q}} t_{a}^{q - w(\fp)}(f_{\fp}) \xi^{\fp}, \; \; R^{q}_{a}(f) := \sum_{\substack{\fp \in \N_{|f|}^{n_{\ast}} \\ w(\fp) \leq q}} r^{q-w(\fp)}_{a}(f_{\fp}) \xi^{\fp} + \sum_{\substack{\fp \in \N_{|f|}^{n_{\ast}} \\ w(\fp) > q}} f_{\fp} \xi^{\fp}.
\end{equation}
Observe that the sum defining $T^{q}_{a}(f)$ is finite, so it is indeed a polynomial of degree $q$ in variables $\{ (x^{i} - x^{i}(a)) \}_{i=1}^{n_{0}}$ and $\{ \xi_{\mu} \}_{\mu=1}^{n_{\ast}}$. Next, observe that functions $x^{i} - x^{i}(a)$ and $\xi_{\mu}$ are elements of $\J^{a}_{(n_{j})}(U)$ and each summand of the first (finite) sum defining $R^{q}_{a}(f)$ is a product of at least $q+1$ of them, so it is an element of $(\J^{a}_{(n_{j})}(U))^{q+1}$. The second sum can be rearranged as
\begin{equation}
\sum_{\substack{\fp \in \N_{|f|}^{n_{\ast}} \\ w(\fp) > q}} f_{\fp} \xi^{\fp} = \hspace{-2mm} \sum_{\fq \in \N^{n_{\ast}}(q+1)} \hspace{-2mm} g_{(\fq)} \cdot \xi^{\fq},
\end{equation}
where $\N^{n_{\ast}}(q+1)$ is a finite set defined under (\ref{eq_Nnkpset}) and $g_{(\fq)} \in (\C^{\infty}_{(n_{j})}(U))_{|f|-|\xi^{\fq}|}$. By construction, $\xi^{\fq} \in (\J^{a}_{(n_{j})}(U))^{q+1}$. As the sum over $\fq$ is finite, it is also the element of the same ideal.
\end{proof}
Hadamard lemma is most useful in cooperation with the following observation:
\begin{tvrz} \label{tvrz_inallvanishingideals}
Let $h \in \C^{\infty}_{(n_{j})}(U)$. Then $h = 0$, if and only if $h \in (\J^{a}_{(n_{j})}(U))^{q}$ for all $a \in U$ and $q \in \N$. Equivalently, one has
\begin{equation}
\bigcap_{q \in \N} \bigcap_{a \in U} (\J^{a}_{(n_{j})}(U))^{q} = 0.
\end{equation}
\end{tvrz}
\begin{proof}
Suppose $h \in (\J^{a}_{(n_{j})}(U))^{q}$ for a given $a \in U$ and $q \in \N$. Write $h = \sum_{\fp \in \N^{n_{\ast}}_{|h|}} h_{\fp} \xi^{\fp}$. We claim that $h_{\fp}(a) = 0$ whenever $w(\fp) < q$. It suffices to consider $h$ in the form of  a $q$-fold product
\begin{equation}
h = f^{(1)} \cdots f^{(q)}, \; \; f^{(s)} \in \J^{a}_{(n_{j})}(U) \text{ for all } s \in \{1,\dots,q\}.
\end{equation}
For each $\fp$ with $w(\fp) < q$, $h_{\fp}$ is a sum of $q$-fold products (up to a sign) of functions in the form $f^{(1)}_{\fq^{(1)}} \cdots f^{(q)}_{\fq^{(q)}}$ where $\fq^{(s)} \in \N^{n_{\ast}}_{|f^{(s)}|}$ for all $s \in \{1,\dots,q\}$, and $\sum_{s=1}^{q} \fq^{(s)} = \fp$. But $\sum_{s=1}^{q} w(\fq^{(s)}) = w(\fp) < q$. There is thus necessarily some $s \in \{1,\dots,q\}$ with $w(\fq^{(s)}) = 0$, that is $f^{(s)}_{\mathbf{0}}$ appears in the product. But $f^{(s)} \in \J^{a}_{(n_{j})}(U)$, whence $f^{(s)}_{\mathbf{0}}(a) = 0$. This shows that $h_{\fp}(a) = 0$. 

If $h \in (\J^{a}_{(n_{j})}(U))^{q}$ for \textit{all} $a \in U$ and $q \in \N$, we can use the previous paragraph to argue that $h = 0$. This proves the if part of the statement, the only if part is trivial. 
\end{proof}
To conclude this subsection, let us observe that the aforementioned ideals can be written also using their generating graded sets. Let us start with the ideal of purely graded functions:
\begin{tvrz} \label{tvrz_Jpggeneratingset}
For every $U \in \Op(X)$, the ideal $\J^{\pg}_{(n_{j})}(U)$ is generated by the graded set $S^{\pg}_{(n_{j})}(U)$, where for each $k \in \Z$, we define 
\begin{equation}
S^{\pg}_{(n_{j})}(U)_{k} := \{ \xi_{\mu} \; | \; \mu \in \{1,\dots,n_{\ast}\} \text{ such that } |\xi_{\mu}| = k \}.
\end{equation}
Note that we view each $\xi_{\mu}$ as a function in $\C^{\infty}_{(n_{j})}(U)$. 
\end{tvrz}
\begin{proof}
Obviously $\< S^{\pg}_{(n_{j})}(U)\> \subseteq \J^{\pg}_{(n_{j})}(U)$. For the other inclusion, let $f \in \J^{\pg}_{(n_{j})}(U)$. Since $\ul{f} = 0$, it can be decomposed as a finite sum $f = f^{\mu} \cdot \xi_{\mu}$, for some $f^{\mu} \in \C^{\infty}_{(n_{j})}(U)$. Thus $f \in \< S^{\pg}_{(n_{j})}(U)\>$.
\end{proof}
To obtain the ideal of functions vanishing at $a \in \R^{n_{0}}$, we have to throw in more generators:
\begin{tvrz} \label{tvrz_vanishingidealgen}
For every $a \in \R^{n_{0}}$ and $U \in \Op_{a}(X)$, the ideal $\J^{a}_{(n_{j})}(U)$ is generated by the graded set $S^{a}_{(n_{j})}(U)$, where we define
\begin{equation}
S^{a}_{(n_{j})}(U)_{0} := \{ x^{i} - x^{i}(a) \; | \; i \in \{1,\dots,n_{0}\} \}, \; \; S^{a}_{(n_{j})}(U)_{k} := S^{\pg}_{(n_{j})}(U)_{k} \text{ for } k \neq 0.
\end{equation}
Moreover, there is a canonical decomposition $\C^{\infty}_{(n_{j})}(U) = \R \oplus \J^{a}_{(n_{j})}(U)$ of graded vector spaces, where $\R$ corresponds to the graded subspace of all scalar multiples of the algebra unit in $\C^{\infty}_{(n_{j})}(U)$.
\end{tvrz}
\begin{proof}
The inclusion $\< S^{a}_{(n_{j})}(U) \> \subseteq \J^{a}_{(n_{j})}(U)$ is obvious. Let $f \in \J^{a}_{(n_{j})}(U)$. For $|f| \neq 0$, this reduces to the scenario in the previous proof. It thus suffices to consider $|f| = 0$. Write $f = f_{\mathbf{0}} + f^{\mu} \cdot \xi_{\mu}$ for some $f^{\mu} \in \C^{\infty}_{(n_{j})}(U)$. We must only argue that $f_{\mathbf{0}} \in \< S^{a}_{(n_{j})}(U) \>$. As $f_{\mathbf{0}}(a) = 0$, the classical Hadamard lemma \ref{lem_Hadimrdclassic} allows us to find $\{ h_{i} \}_{i=1}^{n_{0}} \subseteq \C^{\infty}_{n_{0}}(U)$ such that
\begin{equation}
f_{\mathbf{0}} = \sum_{i=1}^{n_{0}} h_{i} \cdot (x^{i} - x^{i}(a)),
\end{equation}
which is obviously an element of $\< S^{a}_{(n_{j})}(U) \>$. Let us prove the decomposition. For $k \neq 0$, one has $\R_{k} = \{0\}$ and $\J^{a}_{(n_{j})}(U)_{k} = \C^{\infty}_{(n_{j})}(U)_{k}$, so the statement is trivial. For any $f \in \C^{\infty}_{(n_{j})}(U)_{0}$, one can write $f = f_{\mathbf{0}}(a) + (f - f_{\mathbf{0}}(a))$. This is the required unique decomposition into $\R \oplus \J^{a}_{(n_{j})}(U)_{0}$.
\end{proof}
\subsection{Morphisms of graded domains}
Let $U^{(n_{j})} = (U, \C^{\infty}_{(n_{j})})$ and $V^{(m_{j})} = (V, \C^{\infty}_{(m_{j})})$ be two graded domains. Let $\varphi: U^{(n_{j})} \rightarrow V^{(m_{j})}$ be a morphism in $\gLRS$, see Definition \ref{def_gLRS}. In this subsection, we will show that every such $\varphi = (\ul{\varphi}, \varphi^{\ast})$ has a relatively simple structure. 

Let us start by describing the morphisms of ordinary domains.
\begin{lemma} \label{lem_itispullback}
$U^{n_{0}} = (U, \C^{\infty}_{n_{0}})$ and $V^{m_{0}} = (V, \C^{\infty}_{m_{0}})$ be two ordinary domains, viewed as (trivially) graded locally ringed spaces. Let $\varphi = (\ul{\varphi}, \varphi^{\ast}): U^{n_{0}} \rightarrow V^{m_{0}}$ be a morphism in $\gLRS$. 

Then $\ul{\varphi}: U \rightarrow V$ is a smooth map and the sheaf morphism $\varphi^{\ast}$ is just a pullback by $\ul{\varphi}$, that is for any $W \in \Op(V)$ and $f \in \C^{\infty}_{m_{0}}(W)$, one has $\varphi^{\ast}_{W}(f) = f \circ \ul{\varphi}$. 
\end{lemma}
\begin{proof}
Let $f \in \C^{\infty}_{m_{0}}(W)$. Let $x \in \ul{\varphi}^{-1}(W)$ be an arbitrary point. It follows that $[f - f(\ul{\varphi}(x))]_{\ul{\varphi}(x)} \in \frJ( \C^{\infty}_{m_{0}, \ul{\varphi}(x)})$, see the proof of Proposition \ref{tvrz_gradeddomainisgLRS}. As $\varphi$ is a morphism in $\gLRS$, we see that
\begin{equation}
\varphi_{(x)}( [f - f(\ul{\varphi}(x))]_{\ul{\varphi}(x)}) = [ \varphi^{\ast}_{W}( f - f(\ul{\varphi}(x))) ]_{x} = [\varphi^{\ast}_{W}(f) - f(\ul{\varphi}(x))]_{x} \in \frJ( \C^{\infty}_{n_{0},x}).
\end{equation}
But this implies $(\varphi^{\ast}_{W}(f))(x) - f(\ul{\varphi}(x)) = 0$. As $x \in \ul{\varphi}^{-1}(W)$ was arbitrary, this proves the claim about $\varphi^{\ast}$. Moreover, if a pullback of every smooth function on $V$ by $\ul{\varphi}$ is to be a smooth function on $U$, $\ul{\varphi}$ has to be a smooth map. 
\end{proof}
\begin{lemma} \label{lem_inducedandsmooth}
Let $\varphi = (\ul{\varphi},\varphi^{\ast}): (U,\C^{\infty}_{(n_{j})}) \rightarrow (V, \C^{\infty}_{(m_{j})})$ be a morphism of graded domains. Then 
\begin{enumerate}[(i)]
\item $\varphi^{\ast}$ preserves the sheaf of ideals of purely graded functions. For any $W \in \Op(V)$, one has 
\begin{equation}
\varphi^{\ast}_{W}(\J^{\pg}_{(m_{j})}(W)) \subseteq \J^{\pg}_{(n_{j})}(\ul{\varphi}^{-1}(W)).
\end{equation}
\item $\varphi^{\ast}$ preserves the ideals of functions vanishing at a point. For any $W \in \Op(V)$ and $x \in \ul{\varphi}^{-1}(W)$, there holds the the inclusion
\begin{equation}
\varphi^{\ast}_{W}(\J^{\ul{\varphi}(x)}_{(m_{j})}(W)) \subseteq \J^{x}_{(n_{j})}( \ul{\varphi}^{-1}(W)). 
\end{equation}
\end{enumerate}There is thus an induced morphism $\hat{\varphi} = (\ul{\varphi}, \hat{\varphi}^{\ast}): (U, \C^{\infty}_{n_{0}}) \rightarrow (V, \C^{\infty}_{m_{0}})$. In particular, $\ul{\varphi}: U \rightarrow V$ is smooth. 
\end{lemma}
\begin{proof}
Let us prove $(i)$ first. As $\varphi^{\ast}_{W}$ is a graded algebra morphism, it suffices to prove that 
\begin{equation}
\varphi^{\ast}_{W}(S^{\pg}_{(m_{j})}(W)) \subseteq \J^{\pg}_{(n_{j})}(\ul{\varphi}^{-1}(W)),
\end{equation}
 see Proposition \ref{tvrz_Jpggeneratingset}. But that is obvious, as the statement is non-trivial only in degree zero and $(S^{\pg}_{(m_{j})}(W))_{0} = \emptyset$. The claim $(ii)$ is a just a restatement of the the property (iii) of Definition \ref{def_gLRS} for $\varphi$. We have already constructed a sheaf isomorphism $\hat{\beta}: \C^{\infty}_{(n_{j})} / \J^{\pg}_{(n_{j})} \rightarrow \C^{\infty}_{n_{0}}$. There is thus the induced sheaf morphism $\hat{\varphi}^{\ast}: \C^{\infty}_{m_{0}} \rightarrow \ul{\varphi}_{\ast}\C^{\infty}_{n_{0}}$. It is easy to see that $\hat{\varphi} = (\ul{\varphi}, \hat{\varphi}^{\ast})$ has the property (iii) of Definition \ref{def_gLRS}. It follows from Lemma \ref{lem_itispullback} that $\ul{\varphi}: U \rightarrow V$ is smooth. 
\end{proof}
\begin{cor} \label{cor_domainmapsbodymaps}
Every morphism $\varphi: U^{(n_{j})} \rightarrow V^{(m_{j})}$ of graded domains behaves naturally with respect to the respective body maps. For any $W \in \Op(V)$ and any $f \in \C^{\infty}_{(m_{j})}(W)$, one has
\begin{equation}
\ul{\varphi^{\ast}_{W}(f)} = \ul{f} \circ \ul{\varphi}.
\end{equation}
\end{cor}
\begin{proof}
Lemma \ref{lem_itispullback} claims that $\hat{\varphi}^{\ast}_{W}(g) = g \circ \ul{\varphi}$ for any $g \in \C^{\infty}_{m_{0}}(W)$. But $\hat{\varphi}^{\ast}$ was defined to satisfy the equation $\ul{\varphi^{\ast}_{W}(f)} = \hat{\varphi}^{\ast}_{W}(\ul{f})$ for every $f \in \C^{\infty}_{(m_{j})}(W)$. This finishes the proof. 
\end{proof}

Next, let us construct a particular example of a morphism from a graded domain $U^{(n_{j})}$ to the ordinary domain $V^{m_{0}}$. It will be important for the main statement of this subsection. 
\begin{lemma} \label{lem_olphumorphism}
Let $\ul{\varphi}: U \rightarrow V$ be a smooth map. Let $\{ \bar{y}^{j} \}_{j=1}^{m_{0}}$ be a collection of functions in $\J^{\pg}_{(n_{j})}(U)_{0}$, that is $\bar{y}^{j}_{\mathbf{0}} = 0$ for every $j \in \{1,\dots,m_{0}\}$. For $W \in \Op(V)$ and $f \in \C^{\infty}_{m_{0}}(W)$, define 
\begin{equation}
\ol{\varphi}^{\ast}_{W}(f) := \sum_{k=0}^{\infty} \frac{1}{k!} (\frac{\partial^{k} f}{\partial y^{j_{1}} \dots \partial y^{j_{k}}} \circ \ul{\varphi}) \cdot (\bar{y}^{j_{1}} \cdots \bar{y}^{j_{k}})|_{\ul{\varphi}^{-1}(W)},
\end{equation}
where $(y^{1},\dots,y^{m_{0}})$ are the standard coordinates on $\R^{m_{0}}$. Then $\ol{\varphi} := (\ul{\varphi}, \ol{\varphi}^{\ast})$ defines a morphism $\ol{\varphi}: U^{(n_{j})} \rightarrow V^{m_{0}}$ of graded locally ringed spaces. Moreover, $\{ \bar{y}^{j} \}_{j=1}^{m_{0}}$ can be obtained as
\begin{equation} \label{eq_baryjobtained}
\bar{y}^{j} = \ol{\varphi}^{\ast}_{V}(y^{j}) - y^{j} \circ \ul{\varphi},
\end{equation}
where $y^{j} \in \C^{\infty}_{m_{0}}(V)$ are the coordinate functions corresponding to $(y^{1},\dots,y^{m_{0}})$.
\end{lemma}
\begin{proof}
Let us first make sense of the infinite sum over $k$. For every $\fp \in \N^{n_{\ast}}_{0}$ we find $(\bar{y}^{j_{1}} \dots \bar{y}^{j_{k}})_{\fp} = 0$ for $k > w(\fp)$, since $\bar{y}^{j}_{\mathbf{0}} = 0$. Each component function of $\ol{\varphi}^{\ast}_{W}(f)$ is thus a finite sum
\begin{equation}
(\ol{\varphi}^{\ast}_{W}(f))_{\fp} = \sum_{k=0}^{w(\fp)} \frac{1}{k!} (\frac{\partial^{k} f}{\partial y^{j_{1}} \dots \partial y^{j_{k}}} \circ \ul{\varphi}) (\bar{y}^{j_{1}} \cdots \bar{y}^{j_{k}})_{\fp}|_{\ul{\varphi}^{-1}(W)}.
\end{equation}
It is now easy to see that $\ol{\varphi}^{\ast}_{W}$ is linear in $f$, $\ol{\varphi}_{W}^{\ast}(1) = 1$ and it is natural in $W$. We only have to prove that it preserves the products. This is straightforward but a bit tedious, so we have moved the proof into the appendix, see Lemma \ref{lem_ap_olphumorphism}. Hence $\ol{\varphi}^{\ast}_{W}$ is the algebra morphism. The validity of the equation (\ref{eq_baryjobtained}) is obvious. 

To finish the proof, we must argue that $\ol{\varphi} = (\ul{\varphi}, \ol{\varphi}^{\ast})$ has the property (iii) of Definition \ref{def_gLRS}. Since we know the explicit form of the Jacobson radical, see Corollary \ref{cor_JRgrdomain}, it suffices to prove that for any $x \in U$ and $W \in \Op_{\ul{\varphi}(x)}(V)$, the equation $f(\ul{\varphi}(x)) = 0$ implies $(\ol{\varphi}^{\ast}_{W}(f))(x) = 0$. But that is obvious as $(\ol{\varphi}^{\ast}_{W}(f))_{\mathbf{0}} = f \circ \ul{\varphi}$. 
\end{proof}
We can now use this result to construct a particular example of a morphism from the graded domain $U^{(n_{j})}$ to the graded domain $V^{(m_{j})}$. 
\begin{lemma} \label{lem_themorphism}
Let $\ul{\varphi}: U \rightarrow V$ be a smooth map. Let $\{ \bar{y}^{j} \}_{j=1}^{m_{0}}$ be a collection of functions in $\J^{\pg}_{(n_{j})}(U)_{0}$. Suppose $( \theta_{\nu} )_{\nu=1}^{m_{\ast}}$ is a fixed total basis of a graded vector space $\R^{(m_{j})}_{\ast}$ and let $\{ \bar{\theta}_{\nu} \}_{\nu=1}^{m_{\ast}}$ be a collection of functions in $\C^{\infty}_{(n_{j})}(U)$, such that $|\bar{\theta}_{\nu}| = |\theta_{\nu}|$. For every $W \in \Op(V)$ and $f \in \C^{\infty}_{(m_{j})}(W)$, one can write $f = \sum_{\fr \in \N^{m_{\ast}}_{|f|}} f_{\fr} \theta^{\fr}$ for unique functions $f_{\fr} \in \C^{\infty}_{m_{0}}(W)$. Define
\begin{equation} \label{eq_themorphismformula}
\varphi^{\ast}_{W}(f) := \sum_{\fr \in \N^{m_{\ast}}_{|f|}} \ol{\varphi}^{\ast}_{W}(f_{\fr}) \cdot \bar{\theta}^{\fr}|_{\ul{\varphi}^{-1}(W)}, 
\end{equation}
where $\ol{\varphi} = (\ul{\varphi}, \ol{\varphi}^{\ast})$ is defined in Lemma \ref{lem_olphumorphism}, and $\bar{\theta}^{\fr}$ is a shorthand for the expression
\begin{equation}
\bar{\theta}^{\fr} := (\bar{\theta}_{1})^{r_{1}} \cdots (\bar{\theta}_{m_{\ast}})^{r_{m_{\ast}}}. 
\end{equation}

Then $\varphi = (\ul{\varphi}, \varphi^{\ast})$ defines a morphism $\varphi: U^{(n_{j})} \rightarrow V^{(m_{j})}$ of graded locally ringed spaces.  Moreover, the collections $\{\bar{y}^{j} \}_{j=1}^{m_{0}}$ and $ \{\bar{\theta}_{\nu}\}_{\nu=1}^{m_{\ast}}$ can be obtained as 
\begin{equation} \label{eq_barvariablesobtained}
\bar{y}^{j} = \varphi^{\ast}_{V}(y^{j}) - y^{j} \circ \ul{\varphi}, \;  \; \bar{\theta}_{\nu} = \varphi^{\ast}_{V}(\theta_{\nu}),
\end{equation}
where $y^{j} \in \C^{\infty}_{(m_{j})}(V)$ and $\theta_{\nu} \in \C^{\infty}_{(m_{j})}(V)$ are the functions corresponding to the coordinates $(y^{j})_{j=1}^{m_{0}}$ and the total basis $(\theta_{\nu})_{\nu=1}^{m_{\ast}}$. 
\end{lemma}
\begin{proof}
First, let us make sense of the infinite sum. Let $\fp \in \N^{n_{\ast}}_{|f|}$ be arbitrary. Observe that $(\bar{\theta}^{\fr})_{\fp} = 0$ whenever $w(\fr) > w(\fp)$. Indeed, $(\bar{\theta}^{\fr})_{\fp}$ is an alternating sum of products of the form
\begin{equation}
(\bar{\theta}_{1})_{\fq_{1}^{(1)}} \dots (\bar{\theta}_{1})_{\fq_{r_{1}}^{(1)}} \dots (\bar{\theta}_{m_{\ast}})_{\fq_{1}^{(m_{\ast})}} \dots (\bar{\theta}_{m_{\ast}})_{\fq_{r_{m_{\ast}}}^{(m_{\ast})}},
\end{equation}
where $\sum_{\nu=1}^{m_{\ast}} \sum_{a=1}^{r_{\nu}} \fq^{(\nu)}_{a} = \fp$. Since $|\theta_{\nu}| \neq 0$, we have $w(\fq_{a}^{(\nu)}) \geq 1$, so the number of the summands in this sum cannot exceed $w(\fp)$. But this number is precisely $w(\fr)$. Consequently, 
\begin{equation}
( \ul{\varphi}^{\ast}_{W}(f_{\fr}) \cdot \bar{\theta}^{\fr}|_{\ul{\varphi}^{-1}(W)})_{\fp} = 0, \text{ whenever } w(\fr) > w(\fp).
\end{equation}
In other words, we should have \textit{defined} $\varphi^{\ast}_{W}(f)$ for each $\fp \in \N^{n_{\ast}}_{|f|}$ by 
\begin{equation}
(\varphi^{\ast}_{W}(f))_{\fp} := \hspace{-3mm} \sum_{\substack{\fr \in \N^{m_{\ast}}_{|f|} \\ w(\fr) \leq w(\fp)}} \hspace{-3mm} ( \ol{\varphi}^{\ast}_{W}(f_{\fr}) \cdot \bar{\theta}^{\fr}|_{\ul{\varphi}^{-1}(W)})_{\fp}.
\end{equation}
It is obvious that $\varphi^{\ast}_{W}$ defines a linear map natural in $W$, and $\varphi^{\ast}_{W}(1) = 1$. One only has to show that it also preserves the products. Again, we have moved the slightly tedious verification into the appendix, see Lemma \ref{lem_ap_themorphism}. Hence $\varphi^{\ast}_{W}$ is the graded algebra morphism. 

It is easy to obtain (\ref{eq_barvariablesobtained}). It remains to prove that the induced graded algebra morphism $\varphi_{(x)}: \C^{\infty}_{(m_{j}), \ul{\varphi}(x)} \rightarrow \C^{\infty}_{(n_{j}),x}$ preserves the Jacobson radicals. To this amount, one has to show that the equation $f(\ul{\varphi}(x)) = 0$ implies $(\varphi^{\ast}_{W}(f))(x) = 0$. This statement is non-trivial only for $|f| = 0$ and it follows immediately from the fact  that $(\varphi^{\ast}_{W}(f))_{\mathbf{0}} = f_{\mathbf{0}} \circ \ul{\varphi}|_{\ul{\varphi}^{-1}(W)}$. 
\end{proof}

Let us now prove the main statement of this subsection. It shows that every morphism of graded domains is in fact one of those described in the previous lemma. 
\begin{theorem}[\textbf{Graded domain theorem}] \label{thm_gradedomaintheorem}
Let $U^{(n_{j})}$ and $V^{(m_{j})}$ be two graded domains. Let $(\theta_{\nu})_{\nu=1}^{m_{\ast}}$ be a total basis of $\R^{(m_{j})}_{\ast}$, and let $(y^{1},\dots,y^{m_{0}})$ be the standard coordinates on $V \subseteq \R^{m_{0}}$. 

Then any morphism of graded domains $\varphi: U^{(n_{j})} \rightarrow V^{(m_{j})}$ uniquely determines and is uniquely determined by the following data:
\begin{enumerate}[(i)]
\item A smooth map $\ul{\varphi}: U \rightarrow V$.
\item A collection $\{ \bar{y}^{j} \}_{j=1}^{m_{0}}$, where $\bar{y}^{j} \in \J^{\pg}_{(n_{j})}(U)_{0}$ for each $j \in \{1,\dots,m_{0}\}$.
\item A collection $\{ \bar{\theta}_{\nu} \}_{\nu=1}^{m_{\ast}}$, where $\bar{\theta}_{\nu} \in \C^{\infty}_{(n_{j})}(U)_{|\theta_{\nu}|}$ for each $\nu \in \{1,\dots,m_{\ast} \}$.
\end{enumerate}
The relation of those data to $\varphi$ is the one described in Lemma \ref{lem_themorphism}.
\end{theorem}
\begin{proof}
Let $\varphi = (\ul{\varphi}, \varphi^{\ast}): U^{(n_{j})} \rightarrow V^{(m_{j})}$ be a morphism of graded domains. By Lemma \ref{lem_inducedandsmooth}, the continuous map $\ul{\varphi}: U \rightarrow V$ is smooth. Let us define the collections $(ii)$ and $(iii)$ using the formulas (\ref{eq_barvariablesobtained}). We claim that $\varphi^{\ast}$ is then necessarily given by the formula (\ref{eq_themorphismformula}).

Let $W \in \Op(V)$ and let $\psi_{W}^{\ast}: \C^{\infty}_{(m_{j})}(W) \rightarrow \C^{\infty}_{(n_{j})}(\ul{\varphi}^{-1}(W))$ denote the graded algebra morphism defined by (\ref{eq_themorphismformula}). Let $a \in \ul{\varphi}^{-1}(W)$ be an arbitrary point. Suppose $f \in \C^{\infty}_{(m_{j})}(W)$ is a polynomial in variables $\{ (y^{j} - y^{j}(\ul{\varphi}(a)) \}_{j=1}^{m_{0}}$ and $\{ \theta_{\nu} \}_{\nu=1}^{m_{\ast}}$. Then certainly $\varphi^{\ast}_{W}(f) = \psi^{\ast}_{W}(f)$. 

Let $f \in \C^{\infty}_{(m_{j})}(W)$ is an arbitrary function. For any $a \in \ul{\varphi}^{-1}(W)$ and any $q \in \N_{0}$, we may write 
\begin{equation}
f = T^{q}_{\ul{\varphi}(a)}(f) + R^{q}_{\ul{\varphi}(a)}(f),
\end{equation}
where $T^{q}_{\ul{\varphi}(a)}(f)$ is a polynomial of degree $q$ in variables $\{ (y^{j} - y^{j}(\ul{\varphi}(a)) \}_{j=1}^{m_{0}}$ and $\{ \theta_{\nu} \}_{\nu=1}^{m_{\ast}}$, and 
\begin{equation}
R^{q}_{\ul{\varphi}(a)}(f) \in (\J^{\ul{\varphi}(a)}_{(m_{j})}(W))^{q+1}.
\end{equation}
This follows from Lemma \ref{lem_Hadamard}. It follows from Lemma \ref{lem_inducedandsmooth}-(ii) that
\begin{equation}
(\varphi^{\ast}_{W} - \psi^{\ast}_{W})(f) = (\varphi^{\ast}_{W} - \psi^{\ast}_{W})( R^{q}_{\ul{\varphi}(a)}(f)) \in (\J^{a}_{(n_{j})}(\ul{\varphi}^{-1}(W)))^{q+1}.
\end{equation}
As $a \in \ul{\varphi}^{-1}(W)$ and $q \in \N_{0}$ were arbitrary, Proposition \ref{tvrz_inallvanishingideals} shows that $(\varphi^{\ast}_{W} - \psi^{\ast}_{W})(f) = 0$. 

This shows that every morphism of graded domains $\varphi: U^{(n_{j})} \rightarrow V^{(m_{j})}$ determines and it is uniquely determined by the data $(i)$-$(iii)$, as was to be proved. 
\end{proof}
\begin{rem} \label{rem_itsufficestocalculatepullbacks}
The graded domain theorem is absolutely vital for the rest of this paper. It shows that every morphism of graded domains can be uniquely defined just by declaring the pullbacks of global coordinate functions. Note that the underlying map $\ul{\varphi}: U \rightarrow V$ can itself by recovered from the pullbacks, as we have $y^{j} \circ \ul{\varphi} = \varphi^{\ast}_{V}(y^{j})_{\mathbf{0}}$ for all $j \in \{1,\dots,m_{0}\}$. 
\end{rem}

\subsection{Definition and basic properties}
\begin{definice}
Let $\M = (M,\C^{\infty}_{\M})$ be a graded locally ringed space over a topological space $M$. A pair $(U,\varphi)$ is called a \textbf{graded local chart for $\M$}, if
\begin{enumerate}[(i)]
\item $U \in \Op(M)$;
\item $\varphi: (U,\C^{\infty}_{\M}|_{U}) \rightarrow \hat{U}^{(n_{j})}$ is an isomorphism in $\gLRS$, where $(n_{j})_{j \in \Z}$ is a sequence of non-negative integers satisfying $\sum_{j \in \Z} n_{j} < \infty$, and $\hat{U} \in \Op(\R^{n_{0}})$.
\end{enumerate}
In other words, $\varphi = (\ul{\varphi}, \varphi^{\ast})$, where $\ul{\varphi}: U \rightarrow \hat{U}$ is a homeomorphism and $\varphi^{\ast}: \C^{\infty}_{(n_{j})} \rightarrow \ul{\varphi}_{\ast} \C^{\infty}_{\M}$ is a sheaf isomorphism. See Proposition \ref{tvrz_isoingLRS} for details. 
\end{definice}

\begin{definice}
Let $\M = (M,\C^{\infty}_{\M})$ be a graded locally ringed space over a topological space $M$.
A collection $\A = \{ (U_{\alpha}, \varphi_{\alpha}) \}_{\alpha \in I}$ of graded local charts for $\M$ is called a \textbf{graded smooth atlas for $\M$}, if $M = \cup_{\alpha \in I} U_{\alpha}$ and $\varphi_{\alpha}: (U_{\alpha},\C^{\infty}_{\M}|_{U_{\alpha}}) \rightarrow \hat{U}^{(n_{j})}_{\alpha}$ for all $\alpha \in I$ and a \textit{single sequence} $(n_{j})_{j \in \Z}$. This sequence is called the \textbf{graded dimension} of the atlas $\A$.
\end{definice}

We can now finally define our main object of interest in this paper. 
 
\begin{definice}
Let $\M = (M, \C^{\infty}_{\M})$ be a graded locally ringed space over a \textit{second countable Hausdorff} topological space $M$. 

We say that $\M$ is a \textbf{graded (smooth) manifold} of a \textbf{graded dimension $\gdim(\M) = (n_{j})_{j \in \Z}$}, if there exists a graded smooth atlas $\A$ for $\M$ of a graded dimension $(n_{j})_{j \in \Z}$. Sections of the \textbf{structure sheaf} $\C^{\infty}_{\M}$ are called \textbf{functions on the graded manifold $\M$}.
\end{definice}

\begin{rem}
Note that the particular choice of the atlas for $\M$ is not important. In fact, it is not even necessary to specify the graded dimension as there are never two atlases of a different graded dimension (we will show this later on). For any two graded smooth atlases $\A$ and $\A'$ for $\M$, their union $\A \cup \A'$ is again a graded smooth atlas for $\M$. In particular, there is a unique maximal graded smooth atlas for $\M$ containing all the others.
\end{rem}

\begin{example} \label{ex_coordfunctions}
Let $(n_{j})_{j \in \Z}$ be a graded dimension of $\M$. Let $\varphi: \M|_{U} \rightarrow \hat{U}$ be a graded local chart for $\M$. In the rest of the paper, let us slightly abuse the notation and use the same symbol for the functions $\xi_{\mu} \in \C^{\infty}_{(n_{j})}(\hat{U})$, induced by the choice of the total basis $(\xi_{\mu})_{\mu =1}^{n_{\ast}}$ for $\R^{(n_{j})}_{\ast}$, and their pullbacks $\xi_{\mu} := \varphi^{\ast}_{\hat{U}}(\xi_{\mu}) \in \C^{\infty}_{\M}(U)$. In the same way, we will write $x^{i} := \varphi^{\ast}_{\hat{U}}(x^{i}) \in \C^{\infty}_{\M}(U)$ for the pullbacks of coordinate functions corresponding to the standard coordinates $(x^{1},\dots,x^{n_{0}})$ on $\hat{U} \subseteq \R^{n_{0}}$. In accordance with the usual terminology, we shall call them \textbf{coordinate functions corresponding to the graded local chart $(U,\varphi)$}. 
\end{example}

\begin{definice}
Let $\M$ and $\cN$ be two graded manifolds. A morphism $\varphi: \M \rightarrow \cN$ in $\gLRS$ is usually called a \textbf{graded smooth map}. Graded manifolds and graded smooth maps form a full subcategory $\gMan^{\infty}$. Isomorphisms in $\gMan^{\infty}$ are called \textbf{graded diffeomorphisms}.
\end{definice}
It follows that every graded manifold $\M = (M,\C^{\infty}_{\M})$ induces a canonical smooth structure on its underlying topological space $M$.
\begin{tvrz} \label{tvrz_transitionmapsunderlyingmanifold}
Let $\M = (M, \C^{\infty}_{\M})$ be a graded manifold and let $\A = \{ (U_{\alpha},\varphi_{\alpha}) \}_{\alpha \in I}$ be a graded smooth atlas for $\M$. For each $(\alpha,\beta) \in I^{2}$, there is a graded diffeomorphism
\begin{equation}
\varphi_{\alpha \beta}: \ul{\varphi_{\beta}}(U_{\alpha \beta})^{(n_{j})} \rightarrow \ul{\varphi_{\alpha}}(U_{\alpha \beta})^{(n_{j})}, 
\end{equation}
obtained as a composition $\varphi_{\alpha \beta} = \varphi_{\alpha}|_{U_{\alpha \beta}} \circ \varphi_{\beta}^{-1}|_{\ul{\varphi_{\beta}}(U_{\alpha \beta})}$, where $U_{\alpha \beta} := U_{\alpha} \cap U_{\beta}$.  

In particular, the maps $\ul{\varphi_{\alpha \beta}}= \ul{\varphi_{\alpha}} \circ \ul{\varphi_{\beta}}^{-1}: \ul{\varphi_{\beta}}(U_{\alpha \beta}) \rightarrow \ul{\varphi_{\alpha}}(U_{\alpha \beta})$ are smooth. It follows that $\ul{\A} := \{ (U_{\alpha}, \ul{\varphi_{\alpha}}) \}_{\alpha \in I}$ forms an ordinary smooth atlas for $M$, making it into an ordinary smooth manifold of dimension $n_{0}$. Together with the smooth structure defined by $\ul{\A}$, one calls $M$ the \textbf{underlying manifold of $\M$}. 
\end{tvrz}
\begin{proof}
The smoothness of $\ul{\varphi_{\alpha \beta}}$ follows from Lemma \ref{lem_inducedandsmooth}, the rest is straightforward. 
\end{proof}
\begin{rem}
Let us introduce the following convention. Graded manifolds are denoted by capital calligraphic letters ($\M, \cN, \cS, \dots$). Their respective underlying manifolds are always denoted by \textit{the same} letter, but non-calligraphic ($M, N, S, \dots$). 
\end{rem}
\begin{example} \label{ex_gVectasGM}
Every graded domain $U^{(n_{j})} = (U, \C^{\infty}_{(n_{j})})$ is a graded manifold of a graded dimension $(n_{j})_{j \in \Z}$. Its underlying manifold is $U := (U,\C^{\infty}_{n_{0}})$. 

Similarly, for every finite-dimensional graded vector space $X \in \gVect$, let $Y \in \gVect$ be a graded vector space $(Y)_{0} := \{0\}$ and $(Y)_{k} := (X^{\ast})_{k} = (X_{-k})^{\ast}$ for $k \neq 0$. As $X_{0}$ is a finite-dimensional vector space, it can be viewed as a smooth manifold with a sheaf $\C^{\infty}_{X_{0}}$ of its smooth functions. For each $U \in \Op(X_{0})$, define $\C^{\infty}_{X}(U) := \bar{S}(Y, \C_{X_{0}}^{\infty}(U))$. As discussed in Example \ref{ex_thesheaf} and Example \ref{ex_twogLRS}, we find $\M_{X} := (X_{0}, \C^{\infty}_{X}) \in \gLRS$. 

Let $(n_{j})_{j \in \Z}$ be the graded dimension of $X$. Fix a basis $(t_{i})_{i=1}^{n_{0}}$ of $X_{0}$. By using its dual $(t^{i})_{i=1}^{n_{0}}$ as coordinate functions, we obtain a global coordinate chart $\ul{\varphi}: X_{0} \rightarrow \R^{n_{0}}$. Now, fix a total basis $(\xi_{\mu})_{\mu=1}^{n_{\ast}}$ of $Y$. This choice corresponds to the unique isomorphism $\psi: \R^{(n_{-j})}_{\ast} \rightarrow Y$ in $\gVect$. For each $\hat{U} \in \Op(\R^{n_{0}})$, let us consider a graded algebra morphism
\begin{equation}
\varphi^{\ast}_{\hat{U}} := \bar{S}(\psi, \ul{\varphi}^{\ast}): \C^{\infty}_{(n_{-j})}(U) \rightarrow \C^{\infty}_{X}( \ul{\varphi}^{-1}(\hat{U})). 
\end{equation}
See Proposition \ref{tvrz_SVAisfunctor} for details. It follows that $\varphi := (\ul{\varphi}, \varphi^{\ast}): \M_{X} \rightarrow (\R^{n_{0}})^{(n_{-j})}$ defines a graded global chart $(X_{0},\varphi)$ for $\M_{X}$, thus making it into a graded manifold of a graded dimension $(n_{-j})_{j \in \Z}$. Its underlying manifold is the vector space $X_{0}$. 

There are several remarks in order:
\begin{enumerate}[(i)]
\item The intuitive idea behind the construction is the following. We choose a total basis for $X$, its degree zero elements form a basis $(t_{i})_{i=1}^{n_{0}}$ of $X_{0}$, the remaining ones form a total basis $(\xi^{\mu})_{\mu=1}^{n_{\ast}}$ of $Y^{\ast}$. We want to declare elements of the corresponding \textit{dual total basis} to be coordinate functions on $\M_{X}$ of the same degree.
\item The most common example of this construction is $X = V[k]$, where $V \in \Vect$ is an ordinary finite-dimensional vector space and $k \in \Z - \{0\}$. See Definition \ref{def_degreeshifted}. Then $X_{-k} = V$ and all the other components are trivial. In particular, $X_{0} = \{ 0 \}$. A total basis $(\xi_{\mu})_{\mu=1}^{n_{\ast}}$ is then just a basis of $(X_{-k})^{\ast} = V^{\ast}$ and one finds $|\xi_{\mu}| = k$. $\M_{X}$ thus becomes a graded manifold over the singleton manifold. 
\item The construction of $\M_{X}$ does not depend on any particular choices. In fact, the assignment $X \mapsto \M_{X}$ can be viewed as a functor from the category $\gVect_{\fin}$ of finite-dimensional graded vector spaces to the category $\gMan^{\infty}$. 
\end{enumerate}
In the following, we shall usually write just $X$ for $\M_{X}$. 
\end{example}
Let $\M = (M,\C^{\infty}_{\M})$ be a graded manifold. For each $U \in \Op(M)$, we write $\M|_{U} := (U, \C^{\infty}_{\M}|_{U})$. By Proposition \ref{tvrz_restrictedgLRS}, this is a graded locally ringed space. It is not difficult to see that it is also a graded manifold. There is a canonical graded smooth map $i: \M|_{U} \rightarrow \M$. 

For any $f \in \C^{\infty}_{\M}(U)$, its \textbf{local representative} with respect to the graded local chart $(U_{\alpha},\varphi_{\alpha})$ is a function $\hat{f}_{\alpha} \in \C^{\infty}_{(n_{j})}( \ul{\varphi_{\alpha}}(U \cap U_{\alpha}))$ defined as $\hat{f}_{\alpha} := (\varphi^{-1}_{\alpha})^{\ast}_{U \cap U_{\alpha}}(f|_{U \cap U_{\alpha}})$. 

Similarly, let $\phi: \M \rightarrow \cN$ be a graded smooth map over a continuous map $\ul{\phi}: M \rightarrow N$. For each $m \in M$, one can find a graded local chart $(U_{\alpha},\varphi_{\alpha})$ for $\M$ and $(V_{\mu},\psi_{\mu})$ for $\cN$, such that $m \in U_{\alpha}$ and $\ul{\phi}(U_{\alpha}) \subseteq V_{\mu}$. One can thus consider a composition 
\begin{equation} \label{eq_hatphialphamu}
\hat{\phi}_{\alpha \mu} := \psi_{\mu} \circ \phi|_{U_{\alpha}} \circ \varphi_{\alpha}^{-1}: \hat{U}_{\alpha}^{(n_{j})} \rightarrow \hat{V}_{\mu}^{(m_{j})}. \end{equation}
This is a morphism of graded domains. In particular, its underlying map is smooth by Lemma \ref{lem_inducedandsmooth}. We have just proved the following important observation:
\begin{tvrz}
Let $\M = (M,\C^{\infty}_{\M})$ and $\cN = (N,\C^{\infty}_{\cN})$ be two graded manifolds. Let $\phi: \M \rightarrow \cN$ be a graded smooth map. Then the underlying continuous map $\ul{\phi}: M \rightarrow N$ is smooth. 

Consequently, the assignment $\M \mapsto M$ can be viewed as a functor from $\gMan^{\infty}$ to $\Man^{\infty}$. 
\end{tvrz}
\begin{tvrz} \label{tvrz_locrepgluing}
Let $f \in \C^{\infty}_{\M}(U)$. Let $\A = \{ (U_{\alpha},\varphi_{\alpha}) \}_{\alpha \in I}$ be a graded smooth atlas for $\M$. Then for each $(\alpha,\beta) \in I^{2}$, there holds a relation
\begin{equation} \label{eq_locreprelation}
\hat{f}_{\beta}|_{\ul{\varphi_{\beta}}(U \cap U_{\alpha \beta})} = (\varphi_{\alpha \beta})^{\ast}_{\ul{\varphi_{\alpha}}(U \cap U_{\alpha \beta})}( \hat{f}_{\alpha}|_{\ul{\varphi_{\alpha}}(U \cap U_{\alpha \beta})}).
\end{equation}
Conversely, given a collection $\{ \hat{f}_{\alpha} \}_{\alpha \in I}$, where $\hat{f}_{\alpha} \in \C^{\infty}_{(n_{j})}( \ul{\varphi_{\alpha}}(U \cap U_{\alpha}))$ satisfy the relation (\ref{eq_locreprelation}) for each $(\alpha,\beta) \in I^{2}$, there exists a unique $f \in \C^{\infty}_{\M}(U)$, such that $\hat{f}_{\alpha}$ is its local representative with respect to $(U_{\alpha},\varphi_{\alpha})$ for every $\alpha \in I$. 
\end{tvrz}
\begin{proof}
The first part follows immediately from definitions. The second part is just the gluing property of the sheaf $\C^{\infty}_{\M}$ in disguise. 
\end{proof}

Recall that for graded domains, we have constructed the body map, a sheaf morphism $\beta: \C^{\infty}_{(n_{j})} \rightarrow \C^{\infty}_{n_{0}}$. It turns out that a similar map can be constructed for a general graded manifold. 
\begin{tvrz} \label{tvrz_bodymap}
Let $\M = (M,\C^{\infty}_{\M})$ be a graded manifold. Then there exists a canonical sheaf morphism $i^{\ast}_{M}: \C^{\infty}_{\M} \rightarrow \C^{\infty}_{M}$, called the \textbf{body map}. For each $f \in \C^{\infty}_{\M}(U)$, $\ul{f} := (i_{M})^{\ast}_{U}(f)$ is called the \textbf{body of the function $f$}. Note that $\ul{f} = 0$ for $|f| \neq 0$. 

In fact, we may use it do define a graded smooth map $i_{M}: M \rightarrow \M$, where $i_{M} = (\1_{M}, i^{\ast}_{M})$. Let $\cN = (N, \C^{\infty}_{\cN})$ be another graded manifold. For each graded smooth map $\phi: \M \rightarrow \cN$, we find a commutative diagram
\begin{equation} \label{eq_cdbodymaps}
\begin{tikzcd}
\M \arrow{r}{\phi} & \cN \\
M \arrow{u}{i_{M}} \arrow{r}{\ul{\phi}} & N \arrow{u}{i_{N}} 
\end{tikzcd}
\end{equation}
In other words, for each $V \in \Op(N)$ and $f \in \C^{\infty}_{\cN}(V)$, we find $\ul{ \phi^{\ast}_{V}(f)} = \ul{f} \circ \ul{\phi}$. 
\end{tvrz}
\begin{proof}
Let $\A = \{ (U_{\alpha},\varphi_{\alpha}) \}_{\alpha \in I}$ be an atlas for $\M$. We will define a collection of sheaf morphisms $i^{\ast}_{\alpha}: \C^{\infty}_{\M}|_{U_{\alpha}} \rightarrow \C^{\infty}_{M}|_{U_{\alpha}}$ such that $i^{\ast}_{\alpha}|_{U_{\alpha \beta}} = i^{\ast}_{\beta}|_{U_{\alpha \beta}}$. The sheaf morphism $i^{\ast}_{M}: \C^{\infty}_{\M} \rightarrow \C^{\infty}_{M}$ is then obtained by Proposition \ref{tvrz_gluingmorphisms}-(ii). For every $\alpha \in I$, $V \in \Op(U_{\alpha})$ and $f \in \C^{\infty}_{\M}(V)$, define
\begin{equation}
(i^{\ast}_{\alpha})_{V}(f) := \ul{\hat{f}_{\alpha}} \circ \ul{\varphi_{\alpha}},
\end{equation}
where $\hat{f}_{\alpha} \in \C^{\infty}_{(n_{j})}( \ul{\varphi_{\alpha}}(V))$ is the local representative of $f$ with respect to the graded local chart $(U_{\alpha},\varphi_{\alpha})$. It is obvious that $i^{\ast}_{\alpha}$ is a sheaf morphism. In fact, $i_{\alpha} := (\1_{U_{\alpha}}, i^{\ast}_{\alpha})$ defines a graded smooth map $U_{\alpha} \rightarrow \M|_{U_{\alpha}}$. To see that they agree on overlaps, apply the body map to the equation (\ref{eq_locreprelation}) and use Corollary \ref{cor_domainmapsbodymaps} to obtain the relation
\begin{equation}
\ul{\hat{f}_{\beta}}|_{\ul{\varphi_{\beta}}(V)} = \ul{\hat{f}_{\alpha}}|_{\ul{\varphi_{\alpha}}(V)} \circ \ul{\varphi_{\alpha \beta}},
\end{equation}
for all $V \in \Op(U_{\alpha \beta})$ and $f \in \C^{\infty}_{\M}(V)$. But this can be rewritten as $\ul{\hat{f}_{\alpha}} \circ \ul{\varphi_{\alpha}} = \ul{\hat{f}_{\beta}} \circ \ul{\varphi_{\beta}}$. 

By Proposition \ref{tvrz_gluingmorphisms}-(ii), there is thus a unique sheaf morphism $i^{\ast}_{M}: \C^{\infty}_{\M} \rightarrow \C^{\infty}_{M}$, such that $i^{\ast}_{M}|_{U_{\alpha}} = i^{\ast}_{\alpha}$. Moreover, note that $i_{M} = (\1_{M}, i^{\ast}_{M})$ has the property (iii) of Definition \ref{def_gLRS}, as $i_{\alpha} = (\1_{U_{\alpha}}, i^{\ast}_{\alpha})$ does have it for each $\alpha \in I$. It is easy to see that this construction does not depend on the particular choice of the atlas $\A$, and for any graded local chart $(U,\varphi)$ and any $V \in \Op(U)$, we have the formula 
\begin{equation} (i_{M})^{\ast}_{V}(f) = \ul{\hat{f}} \circ \ul{\varphi}, \end{equation}
where $\hat{f} \in \C^{\infty}_{(n_{j})}( \ul{\varphi}(V))$ is the corresponding local representative. 

It remains to verify the commutativity of (\ref{eq_cdbodymaps}). Let $V \in \Op(N)$ and $f \in \C^{\infty}_{\cN}(V)$ be arbitrary. We may assume that $U := \ul{\varphi}^{-1}(V) \neq \emptyset$, otherwise the statement is trivial. For any $m \in U$, one can find a graded local chart $(U_{\alpha},\varphi_{\alpha})$ for $\M$ and $(V_{\mu},\psi_{\mu})$ for $\cN$, such that $U_{\alpha} \in \Op_{m}(U)$, $V_{\mu} \in \Op(V)$ and $\ul{\phi}(U_{\alpha}) \subseteq V_{\mu}$. Let $\hat{\phi}_{\alpha \mu}: \hat{U}_{\alpha}^{(n_{j})} \rightarrow \hat{V}_{\mu}^{(m_{j})}$ be the composition (\ref{eq_hatphialphamu}). One finds
\begin{equation}
\ul{ \phi^{\ast}_{V}(f)}|_{U_{\alpha}} = \ul{ (\hat{\phi}_{\alpha \mu})^{\ast}_{\hat{V}_{\mu}}( \hat{f}_{\mu})} \circ \ul{\varphi_{\alpha}} = (\ul{\hat{f}_{\mu}} \circ \ul{\hat{\phi}_{\alpha \mu}}) \circ \ul{\varphi_{\alpha}} = (\ul{\hat{f}_{\mu}} \circ \ul{\psi_{\mu}}) \circ \ul{\phi}|_{U_{\alpha}} = (\ul{f} \circ \ul{\phi})|_{U_{\alpha}}.
\end{equation}
We have used Corollary \ref{cor_domainmapsbodymaps} for $\hat{\phi}_{\alpha \mu}$ in the process. As $m \in U$ was arbitrary, this shows that $\ul{\phi^{\ast}_{V}(f)} = \ul{f} \circ \ul{\phi}$. This is true for any $V \in \Op(N)$ and $f \in \C^{\infty}_{\cN}(V)$, which is equivalent to the commutativity of (\ref{eq_cdbodymaps}). This finishes the proof.
\end{proof}

As we have just constructed a sheaf morphism $i^{\ast}_{M}: \C^{\infty}_{\M} \rightarrow \C^{\infty}_{M}$, we can form its kernel sheaf
\begin{equation} \label{eq_sheaofidealspurelygraded}
\J^{\pg}_{\M} := \ker(i^{\ast}_{M}) \subseteq \C^{\infty}_{\M}
\end{equation}
called the sheaf of ideals of \textbf{purely graded functions} of $\C^{\infty}_{\M}$. Note that at this moment, we cannot identify the quotient $\C^{\infty}_{\M} / \J^{\pg}_{\M}$ with the sheaf $\C^{\infty}_{M}$. To do so, we need to argue that $i^{\ast}_{M}$ is surjective. Although it is true, the necessary tools will be introduced later on. Let $f \in \C^{\infty}_{\M}(U)$. It is not difficult to obtain the following local criterion:
\begin{equation}
f \in \J^{\pg}_{\M}(U) \Leftrightarrow \hat{f}_{\alpha} \in \J^{\pg}_{(n_{j})}(\ul{\varphi_{\alpha}}(U \cap U_{\alpha})) \text{ for all } \alpha \in I.
\end{equation}

For any $U \in \Op(X)$, any $f \in \C^{\infty}_{\M}(U)$ and any $m \in U$, the \textbf{value of $f$ at $m \in U$} is defined as $f(m) := \ul{f}(m)$. Note that always $f(m) = 0$ for $|f| \neq 0$. 

For each $m \in M$, one can define the sheaf of ideals of \textbf{functions vanishing at $m$}. For $U \in \Op(M)$ such that $m \notin U$, let $\J^{m}_{\M}(U) := \C^{\infty}_{\M}(U)$. For $U \in \Op_{m}(M)$, let
\begin{equation} \label{eq_sheaofidealsvanishing}
\J^{m}_{\M}(U) := \{ f \in \C^{\infty}_{\M}(U) \; | \; f(m) = 0 \}. 
\end{equation}
It is straightforward to see that this is a sheaf of ideals and $\J^{\pg}_{\M} \subseteq \J^{m}_{\M}$. Let $f \in \C^{\infty}_{\M}(U)$ and let $(U_{\alpha},\varphi_{\alpha})$ be any graded local chart for $\M$, $m \in U_{\alpha}$. Then
\begin{equation}
f \in \J^{m}_{\M}(U) \Leftrightarrow \hat{f}_{\alpha} \in \J^{\ul{\varphi_{\alpha}}(m)}_{(n_{j})}( \ul{\varphi_{\alpha}}(U \cap U_{\alpha})).
\end{equation}
Similarly to the graded domain case (see Proposition \ref{tvrz_vanishingidealgen}), for each $m \in M$ and $U \in \Op_{m}(M)$, there is a canonical decomposition
\begin{equation} \label{eq_CinftyMUdecomp}
\C^{\infty}_{\M}(U) = \R \oplus \J^{m}_{\M}(U),
\end{equation}
where $\R$ corresponds to the scalar multiples of the algebra unit in $\C^{\infty}_{\M}(U)$. 

Functions on a graded manifold $\M$ have some properties very similar to the ordinary case. We can also find a simple description of Jacobson radicals of the stalks of the structure sheaf.
\begin{tvrz} \label{tvrz_invertibility}
Let $f \in \C^{\infty}_{\M}(U)$. Then the following statements are true:
\begin{enumerate}[(i)]
\item $f$ has a multiplicative inverse, if and only if $f(m) \neq 0$ for all $m \in U$. 
\item If $f(m) \neq 0$ at some $m \in U$, there is $V \in \Op_{m}(U)$ such that $f(m') \neq 0$ for all $m' \in V$. 
\item For each $m \in M$, the Jacobson radical $\frJ(\C^{\infty}_{\M,m}) \subseteq \C^{\infty}_{\M,m}$ can be written as 
\begin{equation} \label{eq_Jacobsonradicalgradedmafolds}
\frJ(\C^{\infty}_{\M,m}) = \{ [f]_{m} \; | \; f(m) = 0 \}.
\end{equation}
\end{enumerate}
\end{tvrz}
\begin{proof}
Let us start by proving (i). If $f$ has a multiplicative inverse $f^{-1} \in \C^{\infty}_{\M}(U)$, we have $1 = (f \cdot f^{-1})(m) = f(m) \cdot f^{-1}(m)$, hence $f(m) \neq 0$ for all $m \in U$. 

Conversely, suppose that $f(m) \neq 0$ for all $m \in U$. Let $\A = \{ (U_{\alpha},\varphi_{\alpha}) \}_{\alpha \in I}$ be a graded smooth atlas for $\M$. We claim that each local representative $\hat{f}_{\alpha}$ of $f$ is invertible. We may assume $U \cap U_{\alpha} \neq \emptyset$, otherwise the statement is trivial. By assumption, each ordinary function $(\hat{f}_{\alpha})_{\mathbf{0}}$ is non-zero everywhere on $\ul{\varphi_{\alpha}}(U \cap U_{\alpha})$, hence invertible in $\C^{\infty}_{n_{0}}( \ul{\varphi_{\alpha}}(U \cap U_{\alpha}))$. As in Example \ref{ex_twogLRS}, one may now use the formulas (\ref{eq_suminverse1} - \ref{eq_suminverse3}) to construct the inverse $\hat{f}_{\alpha}^{-1}$. Clearly, these must be the local representatives of the inverse $f^{-1}$. One thus only has to check that they obey (\ref{eq_locreprelation}). But as $\varphi^{\ast}_{\alpha \beta}$ is a sheaf morphism, it follows easily that  $(\varphi_{\alpha \beta})^{\ast}_{\ul{\varphi_{\alpha}}(U \cap U_{\alpha \beta})}( \hat{f}_{\alpha}^{-1}|_{\ul{\varphi_{\alpha}}(U \cap U_{\alpha \beta})})$ inverts $\hat{f}_{\beta}|_{\ul{\varphi_{\beta}}(U \cap U_{\alpha \beta})}$. By Proposition \ref{tvrz_locrepgluing}, there is thus a unique $f^{-1} \in \C^{\infty}_{\M}(U)$, whose local representatives are $\hat{f}_{\alpha}^{-1}$. It is easy to prove that $f \cdot f^{-1} = 1$. 

The claim (ii) follows trivially from the same property of ordinary continuous functions. 

To prove (iii), suppose $[f]_{m} \in \frU( \C^{\infty}_{\M,m})$. We may choose the representative $f \in \C^{\infty}_{\M}(U)$ having a multiplicative inverse. It follows from (i) that $f(m) \neq 0$. Conversely, if $[f]_{m}$ is represented by some $f \in \C^{\infty}_{\M}(U)$ with $f(m) \neq 0$, we may use (ii) to find $V \in \Op_{m}(U)$, such that $f(m') \neq 0$ for all $m' \in V$. Hence $f|_{V}$ has a multiplicative inverse by (i). Thus $[f]_{m} \in \frU(\C^{\infty}_{\M,m})$. The equation (\ref{eq_Jacobsonradicalgradedmafolds}) follows from the fact that $\frJ(\C^{\infty}_{\M,m}) = \C^{\infty}_{\M,m} - \frU(\C^{\infty}_{\M,m})$, see Proposition \ref{tvrz_grlocal}.
\end{proof}
\begin{example} \label{ex_constantmapping}
Let $\M = (M,\C^{\infty}_{\M})$ and $\cN = (N, \C^{\infty}_{\cN})$ be two graded manifolds. For each $n \in N$, let $\ul{\kappa_{n}}: M \rightarrow N$ to be the corresponding constant mapping, that is $\ul{\kappa_{n}}(m) = n$ for all $m \in M$. Let us now define a graded smooth map $\kappa_{n}: \M \rightarrow \cN$ as follows. 

For each $V \in \Op_{n}(N)$ and $f \in \C^{\infty}_{\cN}(V)$, we set
\begin{equation}
(\kappa_{n})^{\ast}_{V}(f) := \ul{f}(n) \in \C^{\infty}_{\M}(M) = (\ul{\kappa_{n}}_{\ast} \C^{\infty}_{\M})(V),
\end{equation}
where we have used the fact that $\ul{\kappa_{n}}^{-1}(V) = M$. For $V \in \Op(N)$ such that $n \notin V$, one has $\ul{\kappa_{n}}^{-1}(V) = \emptyset$, whence $(\ul{\kappa_{n}}_{\ast} \C^{\infty}_{\M})(V) = 0$. We can thus set $(\kappa_{n})^{\ast}_{V} := 0$. 

It is easy to verify that this defines a sheaf morphism $\kappa_{n}^{\ast}: \C^{\infty}_{\cN} \rightarrow \ul{\kappa_{n}}_{\ast} \C^{\infty}_{\M}$. One only has to verify that for each $m \in M$, the induced map $(\kappa_{n})_{(m)}: \C^{\infty}_{\cN,n} \rightarrow \C^{\infty}_{\M,m}$ preserves the Jacobson radicals. We have $[f]_{n} \in \frJ(\C^{\infty}_{\cN,n})$, iff $\ul{f}(n) = 0$. Whence $(\kappa_{n})_{(m)}([f]_{n}) = [0]_{m} \in \frJ(\C^{\infty}_{\M,m})$. We conclude that $\kappa_{n}: \M \rightarrow \cN$ is a graded smooth map, called the \textbf{constant mapping valued at $n$}.
\end{example}

We conclude this subsection with a global analogue of Theorem \ref{thm_gradedomaintheorem}.
\begin{theorem} \label{thm_globaldomain}.
Let $\M = (M,\C^{\infty}_{\M})$ be a graded manifold. Let $\hat{V}^{(m_{j})}$ be a graded domain. Let $(\theta_{\nu})_{\nu=1}^{m_{\ast}}$ be a total basis of $\R^{(m_{j})}_{\ast}$, and let $(y^{1},\dots,y^{m_{0}})$ be the standard coordinates on $\hat{V} \subseteq \R^{m_{0}}$. 

Then any graded smooth map $\phi: \M \rightarrow \hat{V}^{(m_{j})}$ uniquely determines and is uniquely determined by the following data:
\begin{enumerate}[(i)]
\item A collection $\{ \hat{y}^{j} \}_{j=1}^{m_{0}}$, where $\hat{y}^{j} \in \C^{\infty}_{\M}(M)_{0}$ for each $j \in \{1,\dots,m_{0}\}$, where the smooth map $(\ul{\hat{y}}^{1}, \dots, \ul{\hat{y}}^{m_{0}}): M \rightarrow \R^{m_{0}}$ takes values in $\hat{V}$. 
\item A collection $\{ \hat{\theta}_{\nu} \}_{\nu=1}^{m_{\ast}}$, where $\hat{\theta}_{\nu} \in \C^{\infty}_{\M}(M)_{|\theta_{\nu}|}$ for each $\nu \in \{1,\dots,m_{\ast} \}$. 
\end{enumerate}
The relation of these data to $\phi = (\ul{\phi}, \phi^{\ast})$ is 
\begin{equation} \label{eq_hatvariablesglobaldomainthm}
(\ul{\hat{y}}^{1}, \dots, \ul{\hat{y}}^{m_{0}}) = \ul{\phi}, \; \; 
\hat{y}^{j} = \phi^{\ast}_{\hat{V}}(y^{j}), \; \; \hat{\theta}_{\nu} = \phi^{\ast}_{\hat{V}}(\theta_{\nu}),\end{equation}
 where $y^{j} \in \C^{\infty}_{(m_{j})}(\hat{V})$ and $\theta_{\nu} \in \C^{\infty}_{(m_{j})}(\hat{V})$ are the functions corresponding to the coordinates $(y^{j})_{j=1}^{m_{0}}$ and the total basis $(\theta_{\nu})_{\nu=1}^{m_{\ast}}$. 
\end{theorem}
\begin{proof}
The interested reader can find the detailed proof in the appendix, see Theorem \ref{thm_ap_globaldomain}.
\end{proof}
This theorem allows one to identify the global sections of the structure sheaf with actual graded smooth maps to a graded manifold.
\begin{cor} \label{cor_functionsasmaps}
For each $k \in \Z$, we have a pair of functors from $(\gMan^{\infty})^{\op}$ to $\Set$, namely
\begin{equation}
\M \mapsto \C^{\infty}_{\M}(M)_{k}, \; \; \M \mapsto \gMan^{\infty}(\M,\R[k])
\end{equation}
Then there is a natural isomorphism of these functors. 

For $f \in \C^{\infty}_{\M}(M)_{k}$, let $\phi_{f} \in \gMan^{\infty}(\M,\R[k])$ denote the corresponding graded smooth map. 
\end{cor}
\begin{proof}
Recall that $\R[k]$ is a graded manifold from Example \ref{ex_gVectasGM}. It has a graded dimension $(m_{j})_{j \in \Z}$ where $m_{k} = 1$ and zero otherwise. It coincides with the graded domain $\R^{(m_{j})}$ for $k = 0$ and with $\{ 0 \}^{(m_{j})}$ for $k \neq 0$. The rest follows easily from the previous theorem. 
\end{proof}
\subsection{Collation and gluing theorems}
In this subsection, we will formulate and prove the sequence of technical propositions of extreme importance for the actual construction of graded manifolds. 

\begin{tvrz}[\textbf{Collation of graded manifolds I}] \label{tvrz_gmcollation1}
Let $M$ be a second countable Hausdorff topological space and let $\{ U_{\alpha} \}_{\alpha \in I}$ be its open cover. Suppose we are given the following data:
\begin{enumerate}[(i)]
\item a collection $\{ \M_{\alpha} \}_{\alpha \in I}$, where each $\M_{\alpha} = (U_{\alpha}, \C^{\infty}_{\M_{\alpha}})$ is a graded manifold of the same given graded dimension $(n_{j})_{j \in \Z}$;
\item a collection $\{ \phi_{\alpha \beta} \}_{(\alpha,\beta) \in I^{2}}$, where $\phi_{\alpha \beta}: \M_{\beta}|_{U_{\alpha \beta}} \rightarrow \M_{\alpha}|_{U_{\alpha \beta}}$ are graded diffeomorphisms such that $\ul{\phi_{\alpha \beta}} = \1_{U_{\alpha \beta}}$ and for each $(\alpha,\beta,\gamma) \in I^{3}$, there holds the cocycle condition
\begin{equation} \label{eq_gmcollationcocycle}
\phi_{\alpha \gamma} = \phi_{\alpha \beta} \circ \phi_{\beta \gamma}
\end{equation}
where all maps are assumed to be restricted to the open subset $U_{\alpha \beta \gamma}$. 
\end{enumerate}
Then there exists a graded manifold $\M = (M, \C^{\infty}_{\M})$ of a graded dimension $(n_{j})_{j \in \Z}$ together with a collection $\{ \lambda_{\alpha} \}_{\alpha \in I}$, where each $\lambda_{\alpha}: \M_{\alpha} \rightarrow \M|_{U_{\alpha}}$ is a graded diffeomorphism such that $\ul{\lambda_{\alpha}} = \1_{U_{\alpha}}$, and for all $(\alpha,\beta) \in I^{2}$, one has $\lambda_{\alpha}|_{U_{\alpha \beta}} \circ \phi_{\alpha \beta} = \lambda_{\beta}|_{U_{\alpha \beta}}$. 

If $\M'$ and $\{\lambda'_{\alpha} \}_{\alpha \in I}$ are another data having these properties, there exists a unique graded diffeomorphism $\varphi: \M \rightarrow \M'$, such that $\ul{\varphi} = \1_{M}$ and $\varphi|_{U_{\alpha}} \circ \lambda_{\alpha} = \lambda'_{\alpha}$ for every $\alpha \in I$. 
\end{tvrz}
\begin{proof}
By definition, we have a collection of sheaves $\{ \C^{\infty}_{\M_{\alpha}} \}_{\alpha \in I}$ where $\C^{\infty}_{\M_{\alpha}} \in \Sh(U_{\alpha},\gcAs)$. Moreover, we have a collection of sheaf morphisms $\{ \phi^{\ast}_{\alpha \beta} \}_{(\alpha,\beta)\in I^{2}}$, where $\phi^{\ast}_{\alpha \beta}: \C^{\infty}_{\M_{\alpha}}|_{U_{\alpha \beta}} \rightarrow \C^{\infty}_{\M_{\beta}}|_{U_{\alpha \beta}}$. Together, they form a collation data $(i)$ and $(ii)$ of Proposition \ref{tvrz_shcollation1}. There is thus a sheaf $\C^{\infty}_{\M} \in \Sh(M,\gcAs)$ together with a collection of sheaf isomorphisms $\{ \lambda^{\ast}_{\alpha} \}_{\alpha \in I}$ where $\lambda^{\ast}_{\alpha}: \C^{\infty}_{\M}|_{U_{\alpha}} \rightarrow \C^{\infty}_{\M_{\alpha}}$ satisfy $\phi^{\ast}_{\alpha \beta} \circ \lambda^{\ast}_{\alpha}|_{U_{\alpha \beta}} = \lambda^{\ast}_{\beta}|_{U_{\alpha \beta}}$ for every $(\alpha,\beta) \in I^{2}$. 

It now follows from Corollary \ref{cor_isoringtolocalislocal} and Proposition \ref{tvrz_restrictedgLRS} that $\M = (M,\C^{\infty}_{\M})$ is a graded locally ringed space and $\lambda_{\alpha} := (\1_{U_{\alpha}}, \lambda^{\ast}_{\alpha})$ become isomorphisms in $\gLRS$, $\lambda_{\alpha}: \M_{\alpha} \rightarrow \M|_{U_{\alpha}}$. 
 Moreover, for any graded local chart $(U,\varphi)$ for $\M_{\alpha}$, the pair $(U, \varphi \circ \lambda_{\alpha}^{-1}|_{U})$ becomes a graded local chart for $\M$. In this manner, we may define a graded smooth atlas for $\M$ of a graded dimension $(n_{j})_{j \in \Z}$, thus making it into a graded smooth manifold. By construction, graded smooth maps $\lambda_{\alpha}$ for all $(\alpha,\beta) \in I^{2}$ satisfy $\lambda_{\alpha}|_{U_{\alpha \beta}} \circ \phi_{\alpha \beta} = \lambda_{\beta}|_{U_{\alpha \beta}}$. 
 
Finally, if $\M'$ and $\{ \lambda'_{\alpha} \}_{\alpha \in I}$ are another data having the required properties, Proposition \ref{tvrz_shcollation1} gives us a unique sheaf isomorphism $\varphi^{\ast}: \C^{\infty}_{\M'} \rightarrow \C^{\infty}_{\M}$, such that $\lambda^{\ast}_{\alpha} \circ \varphi^{\ast}|_{U_{\alpha}} = \lambda'^{\ast}_{\alpha}$ for each $\alpha \in I$. It follows that $\varphi := (\1_{M}, \varphi^{\ast})$ is the required unique graded diffeomorphism. 
\end{proof}
\begin{tvrz}[\textbf{Collation of graded manifolds II}] \label{tvrz_gmcollation2}
Let $M$ be a second countable Hausdorff topological space and let $\{U_{\alpha}\}_{\alpha \in I}$ be its open cover. Suppose $\{\M_{\alpha} \}_{\alpha \in I}$ and $\{ \phi_{\alpha\beta} \}_{(\alpha,\beta)\in I^{2}}$ are the data $(i)$ and $(ii)$ as in the previous proposition. Suppose $\M$ and $\{ \lambda_{\alpha} \}_{\alpha \in I}$ are obtained by the collation of those data. Then the following observations are true:
\begin{enumerate}[(i)]
\item Let an open cover $\{ V_{\mu} \}_{\mu \in J}$ of $M$ be a refinement of $\{ U_{\alpha} \}_{\alpha \in I}$, that is there is a map $\zeta: J \rightarrow I$, such that $V_{\mu} \subseteq U_{\zeta(\mu)}$ for every $\mu \in J$. For every $\mu,\nu \in J$, define
\begin{equation} \M'_{\mu} := \M_{\zeta(\mu)}|_{V_{\mu}}, \; \; \phi'_{\mu \nu} := \phi_{\zeta(\mu)\zeta(\nu)}|_{V_{\mu \nu}}: \M'_{\nu}|_{V_{\mu\nu}} \rightarrow \M'_{\mu}|_{V_{\mu \nu}}.
\end{equation}
Then $\{ \M'_{\mu} \}_{\mu \in J}$ and $\{ \phi'_{\mu \nu} \}_{(\mu,\nu) \in J^{2}}$ form data $(i)$ and $(ii)$ as in the previous proposition, corresponding to the open cover $\{V_{\mu}\}_{\mu \in J}$. 

Suppose $\M'$ and $\{ \lambda'_{\mu} \}_{\mu \in J}$ are obtained by the collation of those data. Then there is a unique graded diffeomorphism $\varphi: \M \rightarrow \M'$, such that $\ul{\varphi} = \1_{M}$ and for every $\mu \in J$, one has $\varphi|_{V_{\mu}} \circ \lambda_{\zeta(\mu)}|_{V_{\mu}} = \lambda'_{\mu}$.

\item Let $\{ \M'_{\alpha} \}_{\alpha \in I}$ and $\{\phi'_{\alpha \beta} \}_{(\alpha,\beta) \in I^{2}}$ be another data (i) and (ii) in the previous proposition corresponding to the same open cover $\{ U_{\alpha} \}_{\alpha \in I}$ of $M$. Suppose $\M'$ and $\{ \lambda'_{\alpha} \}_{\alpha \in I}$ are obtained by the collation of those data. 

Let $\{ \psi_{\alpha} \}_{\alpha \in I}$ be a collection of graded diffeomorphisms $\psi_{\alpha}: \M_{\alpha} \rightarrow \M'_{\alpha}$, such that $\ul{\psi_{\alpha}} = \1_{U_{\alpha}}$ and for all $(\alpha,\beta) \in I^{2}$, one has 
\begin{equation} \label{eq_collationisos2} \phi'_{\alpha \beta} \circ \psi_{\beta}|_{U_{\alpha \beta}} = \psi_{\alpha}|_{U_{\alpha \beta}} \circ \phi_{\alpha \beta}. \end{equation}
Then there exists a unique graded diffeomorphism $\psi: \M \rightarrow \M'$, such that $\lambda'_{\alpha} \circ \psi_{\alpha} = \psi|_{U_{\alpha}} \circ \lambda_{\alpha}$ for all $\alpha \in I$. Conversely, every graded diffeomorphism induces a collection $\{\psi_{\alpha}\}_{\alpha \in I}$ of graded diffeomorphisms satisfying (\ref{eq_collationisos}). 
\end{enumerate}
\end{tvrz}
\begin{proof}
The proof is very similar to the one of the above proposition. This time, everything is a straightforward consequence of Proposition \ref{tvrz_shcollation2}.
\end{proof}
We now arrive to a very important statement for the actual construction of graded manifolds. It shows that one can start with a smooth manifold $M$ and specify graded smooth transition maps between graded domains.
\begin{tvrz}[\textbf{Gluing of graded manifolds I}] \label{tvrz_gmgluing1}
Let $(n_{j})_{j \in \Z}$ be a given sequence. Let $M$ be a given $n_{0}$-dimensional smooth manifold with a smooth atlas $\ul{\A} = \{ (U_{\alpha}, \ul{\varphi_{\alpha}}) \}_{\alpha \in I}$. 

Suppose that for each $(\alpha,\beta) \in I^{2}$, one has a graded smooth diffeomorphism 
\begin{equation}
\varphi_{\alpha \beta}: \ul{\varphi_{\beta}}(U_{\alpha \beta})^{(n_{j})} \rightarrow \ul{\varphi_{\alpha}}( U_{\alpha \beta})^{(n_{j})},
\end{equation}
such that $\ul{\varphi_{\alpha \beta}}$ is the smooth transition map between the ordinary local charts $(U_{\alpha},\ul{\varphi_{\alpha}})$ and $(U_{\beta}, \ul{\varphi_{\beta}})$. Moreover, suppose that for each $(\alpha,\beta,\gamma) \in I^{3}$, there holds the cocycle condition
\begin{equation} \label{eq_cocyclegmgluing}
\varphi_{\alpha \gamma} = \varphi_{\alpha \beta} \circ \varphi_{\beta \gamma},
\end{equation}
where all maps are assumed to be restricted to the homeomorphic images of the open subset $U_{\alpha \beta \gamma}$. 

Then there is a graded manifold $\M = (M, \C^{\infty}_{\M})$ together with a graded smooth atlas $\A = \{ (U_{\alpha},\varphi_{\alpha}) \}_{\alpha \in I}$, such that the smooth atlas $\ul{\A}$ is induced from $\A$ as in Proposition \ref{tvrz_transitionmapsunderlyingmanifold}, and $\varphi_{\alpha \beta}$ become its transition maps.

If $\M'$ and $\A' = \{(U_{\alpha},\varphi'_{\alpha})\}_{\alpha \in I}$ are another such data, there exists a unique graded diffeomorphism $\phi: \M \rightarrow \M'$, such that $\ul{\phi} = \1_{M}$ and $\varphi'_{\alpha} \circ \phi|_{U_{\alpha}} = \varphi_{\alpha}$ for each $\alpha \in I$. 
\end{tvrz} 
\begin{proof}
For each $\alpha \in I$, let $\C^{\infty}_{\M_{\alpha}} := (\ul{\varphi_{\alpha}})^{-1}_{\ast} \C^{\infty}_{(n_{j})} \in \Sh(U_{\alpha},\gcAs)$. 

Then $\M_{\alpha} := (U_{\alpha}, \C^{\infty}_{\M_{\alpha}})$ is a graded locally ringed space. We also have the obvious isomorphism $
\psi_{\alpha}: \M_{\alpha} \rightarrow \hat{U}_{\alpha}^{(n_{j})}$ in $\gLRS$, such that $\ul{\psi_{\alpha}} = \ul{\varphi_{\alpha}}$. It follows that $(U_{\alpha},\psi_{\alpha})$ forms a graded global chart for $\M_{\alpha}$, hence making it into a graded manifold of a graded dimension $(n_{j})_{j \in \Z}$. 

For each $(\alpha,\beta) \in I^{2}$, let $\phi_{\alpha \beta} := (\psi_{\alpha}|_{U_{\alpha \beta}})^{-1} \circ \varphi_{\alpha \beta} \circ \psi_{\beta	}|_{U_{\alpha \beta}}$. It now follows easily that $\{ \M_{\alpha} \}_{\alpha \in I}$ and $\{ \phi_{\alpha \beta} \}_{(\alpha,\beta) \in I^{2}}$ form the collation data $(i)$ and $(ii)$ of Proposition \ref{tvrz_gmcollation1}. There is thus s graded manifold $\M = (M, \C^{\infty}_{\M})$ together with a collection of graded diffeomorphisms $\{ \lambda_{\alpha} \}_{\alpha \in I}$, where $\lambda_{\alpha}: \M_{\alpha} \rightarrow \M|_{U_{\alpha}}$ satisfy $\ul{\lambda_{\alpha}} = \1_{U_{\alpha}}$ and $\lambda_{\alpha}|_{U_{\alpha \beta}} \circ \phi_{\alpha \beta} = \lambda_{\beta}|_{U_{\alpha \beta}}$ for all $(\alpha,\beta) \in I^{2}$. 

Let $\varphi_{\alpha} := \psi_{\alpha} \circ \lambda_{\alpha}^{-1}: \M|_{U_{\alpha}} \rightarrow \hat{U}_{\alpha}^{(n_{j})}$. These are graded diffeomorphisms, thus making $\A = \{ (U_{\alpha}, \varphi_{\alpha}) \}_{\alpha \in I}$ into the graded smooth atlas for $\M$. It follows that $\ul{\A}$ is indeed induced by $\A$ as in Proposition \ref{tvrz_transitionmapsunderlyingmanifold}. Finally, for each $(\alpha,\beta) \in I^{2}$, one finds 
\begin{equation}
\varphi_{\alpha}|_{U_{\alpha \beta}} \circ \varphi_{\beta}^{-1}|_{\ul{\varphi_{\beta}}(U_{\alpha \beta})} = \psi_{\alpha}|_{U_{\alpha \beta}} \circ \phi_{\alpha \beta} \circ (\psi_{\beta}|_{U_{\alpha \beta}})^{-1} = \varphi_{\alpha \beta}.
\end{equation}
This proves that $\varphi_{\alpha \beta}$ form the transition maps of $\A = \{ (U_{\alpha},\varphi_{\alpha}) \}_{\alpha \in I}$. 

For the second part of the statement, let $\M'$ and $\A' = \{ (U_{\alpha}, \varphi'_{\alpha}) \}_{\alpha \in I}$ by another such data. For each $\alpha \in I$, set $\lambda'_{\alpha} := \varphi'^{-1}_{\alpha} \circ \psi_{\alpha}: \M_{\alpha} \rightarrow \M'|_{U_{\alpha}}$. They are graded diffeomorphisms such that $\ul{\lambda'_{\alpha}} = \1_{U_{\alpha}}$ and $\lambda'_{\alpha}|_{U_{\alpha \beta}} \circ \phi_{\alpha \beta} = \lambda'_{\beta}|_{U_{\alpha \beta}}$ for all $(\alpha,\beta) \in I^{2}$. By Proposition \ref{tvrz_shcollation1}, there exists a unique graded smooth diffeomorphism $\phi: \M \rightarrow \M'$, such that $\ul{\phi} = \1_{M}$ and $\phi|_{U_{\alpha}} \circ \lambda_{\alpha} = \lambda'_{\alpha}$ for each $\alpha \in I$. This is equivalent to $\varphi'_{\alpha} \circ \phi|_{U_{\alpha}} = \varphi_{\alpha}$ for each $\alpha \in I$. This finishes the proof.
\end{proof}
Finally, we obtain the proposition dealing with refinements and equivalent transition maps.
\begin{tvrz}[\textbf{Gluing of graded manifolds II}] \label{tvrz_gmgluing2}
Let $(n_{j})_{j \in \Z}$ be a given sequence. Let $M$ be a given $n_{0}$-dimensional smooth manifold with a smooth atlas $\ul{\A} = \{ (U_{\alpha},\ul{\varphi_{\alpha}}) \}_{\alpha \in I}$. Suppose $\{ \varphi_{\alpha \beta} \}_{(\alpha,\beta) \in I^{2}}$ is a collection as in the previous proposition. Suppose $\M$ and $\A = \{ (U_{\alpha},\varphi_{\alpha}) \}_{\alpha \in I}$ are obtained as in above proposition. Then the following observations are true:
\begin{enumerate}[(i)]
\item Let an open cover $\{ V_{\mu} \}_{\mu \in J}$ of $M$ be a refinement of $\{ U_{\alpha} \}_{\alpha \in I}$, that is there is a map $\zeta: J \rightarrow I$, such that $V_{\mu} \subseteq U_{\zeta(\mu)}$ for every $\mu \in J$. Let $\ul{\A}' = \{ (V_{\mu}, \varphi'_{\mu}) \}_{\mu \in J}$ be the corresponding refinement of $\ul{\A}$, that is $\ul{\varphi'_{\mu}} := \ul{\varphi_{\zeta(\mu)}}|_{V_{\mu}}$ for each $\mu \in J$. Let $\varphi'_{\mu \nu} := \varphi_{\zeta(\mu)\zeta(\nu)}|_{\ul{\varphi'_{\nu}}(V_{\mu \nu})}$. 

Then $\ul{\A}'$ and $\{ \varphi'_{\mu \nu} \}_{(\mu,\nu) \in J^{2}}$ satisfy the assumptions of the previous proposition. Suppose $\M'$ and $\A' = \{ (V_{\mu},\varphi'_{\mu}) \}_{\mu \in J}$ are obtained by the above gluing procedure. Then there exists a unique graded diffeomorphism $\phi: \M \rightarrow \M'$, such that $\varphi'_{\mu} \circ \phi|_{V_{\mu}} = \varphi_{\zeta(\mu)}|_{V_{\mu}}$ for all $\mu \in J$. 

\item Let $\ul{\A}' = \{ (U_{\alpha}, \ul{\varphi'_{\alpha}}) \}_{\alpha \in I}$ and $\{ \varphi'_{\alpha \beta} \}_{(\alpha,\beta) \in J^{2}}$ satisfy the assumptions of the previous proposition for the same open cover $\{ U_{\alpha} \}_{\alpha \in I}$ of $M$. Suppose $\M'$ and $\A' = \{ (U_{\alpha}, \varphi'_{\alpha}) \}_{\alpha \in I}$ are obtained by the above gluing procedure. 

Let $\{ \eta_{\alpha} \}_{\alpha \in I}$ be a collection of graded diffeomorphisms $\eta_{\alpha}: \ul{\varphi_{\alpha}}(U_{\alpha})^{(n_{j})} \rightarrow \ul{\varphi'_{\alpha}}(U_{\alpha})^{(n_{j})}$, such that $\ul{\eta_{\alpha}} \circ \ul{\varphi_{\alpha}} = \ul{\varphi'_{\alpha}}$ for all $\alpha \in I$, and for all $(\alpha,\beta) \in I^{2}$, one has 
\begin{equation} \label{eq_gluingisos2} \varphi'_{\alpha \beta} \circ \eta_{\beta}|_{\ul{\varphi_{\beta}}(U_{\alpha \beta})} = \eta_{\alpha}|_{\ul{\varphi_{\alpha}}(U_{\alpha\beta})} \circ \varphi_{\alpha \beta}. \end{equation}
Then there exists a unique graded diffeomorphism $\eta: \M \rightarrow \M'$, such that $\ul{\eta} = \1_{M}$ and $\varphi'_{\alpha} \circ \eta|_{U_{\alpha}} = \eta_{\alpha} \circ \varphi_{\alpha}$ for all $\alpha \in I$. Conversely, every graded diffeomorphism induces a collection $\{\eta_{\alpha}\}_{\alpha \in I}$ of graded diffeomorphisms satisfying (\ref{eq_gluingisos2}). 
\end{enumerate}
\end{tvrz}
\begin{proof}
This follows in a straightforward manner from Proposition \ref{tvrz_gmcollation2}.
\end{proof}

\begin{example}  \label{ex_degreeshiftedordvect}
Let us now demonstrate the importance of gluing theorems. Let $q: E \rightarrow M$ be an ordinary vector bundle of a finite rank and denote $n_{0} := \dim(M)$, $n_{\ast} := \rk(E)$. 

Let us consider a local trivialization $\{ (U_{\alpha}, \phi_{\alpha}^{\vee}) \}_{\alpha \in I}$ of the \textit{dual} vector bundle $E^{\ast}$ and assume that we have a smooth atlas $\ul{\A} = \{ (U_{\alpha}, \ul{\varphi_{\alpha}}) \}_{\alpha \in I}$ for $M$. 

We have a collection of transition maps $g^{\vee}_{\alpha \beta}: U_{\alpha \beta} \rightarrow \GL(n_{\ast},\R)$ given by the equation 
\begin{equation}
\phi_{\beta}^{\vee}(m,x) = \phi_{\alpha}^{\vee}(m, g_{\alpha \beta}^{\vee}(m)x),
\end{equation}
for all $(m,x) \in U_{\alpha \beta} \times \R^{n_{\ast}}$. When restricted to $U_{\alpha \beta \gamma}$, they satisfy the cocycle condition 
\begin{equation} \label{eq_Estarcocycletrivialization}
g^{\vee}_{\alpha \gamma} = g^{\vee}_{\alpha \beta} \circ g^{\vee}_{\beta \gamma}.
\end{equation}
Now, let us fix a \textit{non-zero} integer $k \in \Z$, and let $(n_{j})_{j \in \Z}$ be a sequence where $n_{k} := n_{\ast}$ and $n_{j} = 0$ for $j \notin \{0,k\}$. We will now define graded diffeomorphisms $\varphi_{\alpha \beta}: \ul{\varphi_{\beta}}(U_{\alpha \beta})^{(n_{j})} \rightarrow \ul{\varphi_{\alpha}}(U_{\alpha \beta})^{(n_{j})}$. 

Let us fix a total basis $(\xi_{\mu})_{\mu=1}^{n_{\ast}}$ of $\R^{(n_{j})}_{\ast}$. Note that $|\xi_{\mu}| = k$ and $(\xi_{\mu})_{\mu=1}^{n_{\ast}}$ is actually an ordinary basis of $(\R^{(n_{j})}_{\ast})_{k} = \R^{n_{\ast}}$. Let $(x^{1},\dots,x^{n_{0}})$ be the coordinates on $\ul{\varphi_{\alpha}}(U_{\alpha \beta}) \subseteq \R^{n_{0}}$. Let $(\alpha,\beta) \in I^{2}$ be fixed. By Theorem \ref{thm_gradedomaintheorem}, we must specify the underlying smooth map $\ul{\varphi_{\alpha \beta}}$ and two collections $\{ \bar{x}^{i} \}_{i=1}^{n_{0}}$ and $\{ \bar{\xi}_{\mu} \}_{\mu=1}^{n_{\ast}}$. Let us set
\begin{equation} \ul{\varphi_{\alpha \beta}} := \ul{\varphi_{\alpha}} \circ \ul{\varphi_{\beta}}^{-1}|_{\ul{\varphi_{\beta}}(U_{\alpha \beta})}, \end{equation}
that is precisely the transition map between the graded local charts $(U_{\alpha},\varphi_{\alpha})$ and $(U_{\beta}, \varphi_{\beta})$. For each $i \in \{1,\dots,n_{0}\}$, define $\bar{x}^{i} := 0$ and for each $\mu \in \{1, \dots, n_{\ast} \}$, define 
\begin{equation}
\bar{\xi}_{\mu} \equiv (\varphi_{\alpha \beta})^{\ast}_{\ul{\varphi_{\alpha}}(U_{\alpha \beta})}(\xi_{\mu}) := ((g_{\beta \alpha}^{\vee})^{\nu}{}_{\mu} \circ \ul{\varphi_{\beta}}^{-1}) \xi_{\nu},
\end{equation}
where $(g^{\vee}_{\beta \alpha})^{\nu}{}_{\mu}$ are the matrix elements of the linear map $g^{\vee}_{\alpha \beta}$ with respect to the above basis of $\R^{n_{\ast}}$. This defines a graded diffeomorphism $\varphi_{\alpha \beta}$, as $\varphi_{\beta \alpha}$ can be easily shown to be its inverse. Moreover, by comparing the pullbacks of the coordinate functions $x^{i}$ and $\xi_{\mu}$, it is straightforward to show that the cocycle condition (\ref{eq_cocyclegmgluing}) follows from (\ref{eq_Estarcocycletrivialization}).

We thus have the data $\ul{\A}$ and $\{ \varphi_{\alpha \beta} \}_{(\alpha,\beta) \in I^{2}}$ as in Proposition \ref{tvrz_gmgluing1}. There is then a graded manifold $E[k] := (M, \C^{\infty}_{E[k]})$ together with a graded smooth atlas $\A = \{ (U_{\alpha}, \varphi_{\alpha}) \}$ for $E[k]$ making $\varphi_{\alpha \beta}$ into its transition maps. $E[k]$ is called the \textbf{degree $k$ shifted vector bundle}.

Let us tackle the question of the independence of this construction on the chosen local trivialization of $E$. First, if we refine both the atlas $\ul{\A}$ and the local trivialization $\{ (U_{\alpha}, \phi_{\alpha}^{\vee}) \}_{\alpha \in I}$, the resulting graded manifold will be diffeomorphic to the original one by Proposition \ref{tvrz_gmgluing2}-(ii). By the common refinement argument, it thus suffices to consider another atlas $\ul{\A}' = \{ (U_{\alpha}, \ul{\varphi'_{\alpha}}) \}_{\alpha \in I}$ and a different local trivialization $\{ (U_{\alpha}, \phi'^{\vee}_{\alpha}) \}_{\alpha \in I}$ of $E^{\ast}$ corresponding to the same open cover $\{ U_{\alpha} \}_{\alpha \in I}$. For each $\alpha \in I$, there is thus a unique $\lambda_{\alpha}: U_{\alpha} \rightarrow \GL(n_{\ast},\R)$ such that
\begin{equation}
\phi'^{\vee}_{\alpha}(m,x) = \phi^{\vee}_{\alpha}(m, \lambda_{\alpha}(m)x),
\end{equation}
for all $m \in U_{\alpha}$ and $x \in \R^{m_{\ast}}$. Let us use it to define a graded diffeomorphism $\eta_{\alpha}: \ul{\varphi_{\alpha}}(U_{\alpha})^{(n_{j})} \rightarrow \ul{\varphi'_{\alpha}}(U_{\alpha})^{(n_{j})}$. Set $\ul{\eta_{\alpha}} := \ul{\varphi'_{\alpha}} \circ \ul{\varphi_{\alpha}}^{-1}$, $\bar{x}^{i} := 0$ for all $i \in \{1,\dots,n_{0}\}$, and let 
\begin{equation}
\bar{\xi}_{\mu} \equiv (\eta_{\alpha})^{\ast}_{\ul{\varphi'_{\alpha}}(U_{\alpha})}(\xi_{\mu}) := ((\lambda_{\alpha})^{\nu}{}_{\mu} \circ \ul{\varphi_{\alpha}}^{-1}) \xi_{\nu},
\end{equation}
for all $\mu \in \{1,\dots,n_{\ast}\}$. By Theorem \ref{thm_gradedomaintheorem}, this defines a graded smooth map $\eta_{\alpha}$. By comparing the pullbacks of coordinate functions, it is now straightforward to verify (\ref{eq_gluingisos2}). Proposition \ref{tvrz_gmgluing2}-(ii) then shows that the resulting graded manifolds are diffeomorphic. 

The intuitive idea behind the construction of $E[k]$ is the following. Each local trivialization chart $(U_{\alpha},\phi_{\alpha}^{\vee})$ for $E^{\ast}$ and a local chart $(U_{\alpha},\ul{\varphi_{\alpha}})$ for $M$ can be used to construct a set of local coordinates for the total space manifold $E$ on $q^{-1}(U_{\alpha})$. We obtain the ``base'' coordinates $(x^{1},\dots,x^{n})$ corresponding to the chart $\ul{\varphi_{\alpha}}$ and the ``fiber coordinates'' $(\xi_{1},\dots,\xi_{n_{\ast}})$ defined by $\phi_{\alpha}^{\vee}$. Sections of the sheaf $\C^{\infty}_{E[k]}|_{U_{\alpha}}$ now correspond to functions on $q^{-1}(U_{\alpha}) \subseteq E$, polynomial in the ``fiber coordinates''. But now, each variable $\xi_{\mu}$ has its \textit{degree shifted by $k$}, that is $|\xi_{\mu}| = k$. Transition functions of the graded smooth atlas for $E[k]$ are defined so that all the coordinate functions transform in the same way as they did on the original manifold $E$. 
\end{example}
\subsection{Partition of unity and consequences}
In the realm of smooth manifolds, the existence of partitions of unity is without a doubt one of the most important properties with far-reaching implications. It is thus very desirable to have such a tool in  the shed of graded geometry. First, let us introduce some terminology. In the entire subsection, let $\M = (M,\C^{\infty}_{\M})$ be a graded manifold. First, one needs a suitable definition of a support of a function.
\begin{definice}
Let $f \in \C^{\infty}_{\M}(U)$ for some $U \in \Op(M)$. The \textbf{support $\supp(f)$ of the function $f$} is a subset 
\begin{equation}
\supp(f) := \{ m \in U \; | \; [f]_{m} \neq 0 \} \subseteq U.
\end{equation}
\end{definice}
It is easy to see that $\supp(f)$ is a closed subset of $U$ and $f$ vanishes when restricted to its complement $\supp(f)^{c} \subseteq U$. 
\begin{lemma}
Let $M$ be an ordinary manifold and let $f \in \C^{\infty}_{M}(U)$ for some $U \in \Op(M)$. Then 
\begin{equation}
\supp(f) = \ol{\{ m \in U \; | \; f(m) \neq 0 \}}.
\end{equation}
For any graded manifold $\M = (M, \C^{\infty}_{\M})$ and $f \in \C^{\infty}_{\M}(U)$ for $U \in \Op(M)$, one has $\supp(\ul{f}) \subseteq \supp(f)$. The inclusion $\supseteq$ is in general not true.  
\end{lemma}
\begin{proof}
The first claim is an easy exercise. For the second one, note that for every $m \in U$, one has $[\ul{f}]_{m} = (i_{M})_{(m)}([f]_{m})$, where $(i_{M})_{(m)}$ is the graded algebra morphism (\ref{eq_vartphixmap}) induced by the graded smooth map $i_{M}$ obtained in Proposition \ref{tvrz_bodymap}. Hence $[\ul{f}]_{m} \neq 0$ implies $[f]_{m} = 0$. However, on every graded domain $\hat{U}^{(n_{j})}$, each coordinate function $\xi_{\mu}$ satisfies $\supp(\xi_{\mu}) = \hat{U}$ and $\ul{\xi_{\mu}} = 0$. 
\end{proof}
Next, recall that any collection $\{ X_{\mu} \}_{\mu \in J}$ of subsets of $M$ is called \textbf{locally finite}, if for every point $m \in M$, there exists a \textit{finite} subset $J_{0} \subseteq J$ and a neighborhood $U \in \Op_{m}(M)$, such that $X_{\mu} \cap U \neq \emptyset$ only if $\mu \in J_{0}$. 

\begin{rem}
If $M$ is locally compact, $\{ X_{\mu} \}_{\mu \in J}$ is locally finite, iff for every compact subset $K \subseteq M$, there exists a finite subset $J_{0} \subseteq J$, such that $X_{\mu} \cap K \neq \emptyset$ only if $\mu \in J_{0}$. Note that every smooth manifold is locally compact. 
\end{rem}

\begin{rem} \label{rem_sumoflocallyfinitesupportedfctions}
Let $U \in \Op(M)$ and let $\{ f_{\mu} \}_{\mu \in J}$ be any collection of functions in $\C^{\infty}_{\M}(U)$ of the same degree. Suppose that the collection $\{ \supp(f_{\mu}) \}_{\mu \in J}$ is locally finite. Then we can define a single function $f \in \C^{\infty}_{\M}(U)$ as follows: 

For each $m \in U$, there is $V_{m} \in \Op_{m}(U)$ and a finite subset $J_{m} \subseteq J$, such that $\supp(f_{\mu}) \cap V_{m} \neq \emptyset$ only if $\mu \in J_{m}$. Let $f_{m} := \sum_{\mu \in J_{m}} f_{\mu}|_{V_{m}}$. In this way, we obtain an open cover $\{ V_{m} \}_{m \in U}$ of $U$, and a collection of sections $\{ f_{m} \}_{m \in M}$ where $f_{m} \in \C^{\infty}_{\M}(V_{m})$ agree on the overlaps. By the gluing and monopresheaf property of $\C^{\infty}_{\M}$, there is thus a unique $f \in \C^{\infty}_{\M}(U)$ such that $f|_{V_{m}} = f_{m}$. 

We shall henceforth denote $f$ simply as $\sum_{\mu \in J} f_{\mu}$. 
\end{rem}
We are now ready to formulate the most important statement of this subsection. Its proof follows the lines of its version for ordinary smooth manifolds, see e.g. Theorem 2.23 of \cite{lee2012introduction}. For the sake of completeness, one can find the proof in the appendix, see Proposition \ref{tvrz_ap_partition}. 
\begin{tvrz}[\textbf{Partitions of unity}] \label{tvrz_partition}
Let $\{ U_{\mu} \}_{\mu \in J}$ be any open cover of $M$. Then there exists a family of functions $\{ \lambda_{\mu} \}_{\mu \in J}$, having the following properties:
\begin{enumerate}[(i)]
\item $\lambda_{\mu} \in \C^{\infty}_{\M}(M)_{0}$ for each $\mu \in J$;
\item the collection $\{ \supp(\lambda_{\mu}) \}_{\mu \in J}$ is locally finite and $\supp(\lambda_{\mu}) \subseteq U_{\mu}$ for each $\mu \in J$;
\item $\sum_{\mu \in J} \lambda_{\mu} = 1$ and $\ul{\lambda_{\mu}} \geq 0$ for each $\mu \in J$.
\end{enumerate}
We say that $\{ \lambda_{\mu} \}_{\mu \in J}$ is a \textbf{partition of unity subordinate to} $\{ U_{\mu} \}_{\mu \in J}$. Alternatively, there exists a family of functions $\{ \lambda'_{\nu} \}_{\nu \in J'}$, having the following properties:
\begin{enumerate}[(i)]
\item $\lambda'_{\nu} \in \C^{\infty}_{\M}(M)_{0}$ for each $\nu \in J'$;
\item $\{ \supp(\lambda'_{\nu}) \}_{\nu \in J'}$ is a locally finite collection of compact sets, and for each $\nu \in J'$, $\supp(\lambda'_{\nu}) \subseteq U_{\mu}$ for some $\mu \in J$. 
\item $\sum_{\nu \in J'} \lambda'_{\nu} = 1$ and $\ul{\lambda'_{\nu}} \geq 0$ for each $\nu \in J'$.
\end{enumerate}
We say that $\{ \lambda'_{\nu} \}_{\nu \in J'}$ is a \textbf{compactly supported partition of unity subordinate to $\{U_{\mu}\}_{\mu \in J}$}.
\end{tvrz}

In the spirit of Remark \ref{rem_sumoflocallyfinitesupportedfctions}, let us now introduce the following notation. Let $\{ U_{\mu} \}_{\mu \in J}$ be an open cover together with a collection $\{ f_{\mu} \}_{\mu \in J}$ of functions of the same degree, where $f_{\mu} \in \C^{\infty}_{\M}(U_{\mu})$ for each $\mu \in J$. Let $\{ \lambda_{\mu} \}_{\mu \in J}$ be a partition of unity subordinate to $\{ U_{\mu} \}_{\mu \in J}$. First, for each $\mu \in J$, it makes sense to define a function $\lambda_{\mu} \cdot f_{\mu} \in \C^{\infty}_{\M}(M)$ by its restrictions
\begin{equation} \label{eq_skoroproduct}
(\lambda_{\mu} \cdot f_{\mu})|_{U_{\mu}} := \lambda_{\mu}|_{U_{\mu}} \cdot f_{\mu}, \; \; (\lambda_{\mu} \cdot f_{\mu})|_{\supp(\lambda_{\mu})^{c}} := 0
\end{equation}
to the sets of the open cover $\{ U_{\mu}, \supp(\lambda_{\mu})^{c} \}$ of $M$. As $\C^{\infty}_{\M}$ is a sheaf, this determines $\lambda_{\mu} \cdot f_{\mu}$ uniquely. Note that $|\lambda_{\mu} \cdot f_{\mu}| = |f_{\mu}|$ for all $\mu \in J$. Next, using the notation of Remark \ref{rem_sumoflocallyfinitesupportedfctions}, one can then define a global function $\sum_{\mu \in J} \lambda_{\mu} \cdot f_{\mu} \in \C^{\infty}_{\M}(M)$. This shows that similarly to the ordinary differential geometry, partitions of unity are useful for the gluing of local sections which do not necessarily agree on the overlaps. Let us now summarize some basic properties of this procedure.
\begin{tvrz} \label{tvrz_partitionsnaturalops}
Let $\M = (M,\C^{\infty}_{\M})$ be a graded manifold, $\{ U_{\mu} \}_{\mu \in J}$ some open cover of $M$ and $\{ \lambda_{\mu} \}_{\mu \in J}$ a partition of unity subordinate to this open cover. Using the notation introduced in the above paragraph, one observes the following facts:
\begin{enumerate}[(i)]
\item Let $f \in \C^{\infty}_{\M}(M)$. Then $\sum_{\mu \in J} \lambda_{\mu} \cdot f|_{U_{\mu}} = f$. One has $\lambda_{\mu} \cdot f|_{U_{\mu}} = \lambda_{\mu} \cdot f$ for each $\mu \in J$. This legalizes the infamous ``insertion of the unit'' written usually with a lot of courage as 
\begin{equation} f = 1 \cdot f = (\sum_{\mu \in J} \lambda_{\mu}) \cdot f = \sum_{\mu \in J} \lambda_{\mu} \cdot f. \end{equation}
\item Let $\phi: \cN \rightarrow \M$ be a graded smooth map, where $\cN = (N,\C^{\infty}_{\cN})$ is an arbitrary graded manifold. Then $\{ \ul{\phi}^{-1}(U_{\mu}) \}_{\mu \in J}$ forms an open cover of $N$ and $\{ \phi^{\ast}_{M}(\lambda_{\mu}) \}_{\mu \in J}$ forms a partition of unity subordinate to this open cover. Moreover, for any collection $\{ f_{\mu} \}_{\mu \in J}$ of functions of the same degree, where $f_{\mu} \in \C^{\infty}_{\M}(U_{\mu})$ for each $\mu \in J$, one has the formula
\begin{equation} \label{eq_pullbackofasumpartition}
\phi^{\ast}_{M}( \sum_{\mu \in J} \lambda_{\mu} \cdot f_{\mu}) = \sum_{\mu \in J} \phi^{\ast}_{M}(\lambda_{\mu}) \cdot \phi^{\ast}_{U_{\mu}}(f_{\mu}). 
\end{equation}
This shows that the gluing procedure behaves naturally with respect to pullbacks. 
\end{enumerate}
\end{tvrz}
\begin{proof}
The proof of $(i)$ follows easily from the definitions. Let us prove $(ii)$. First, note that for any locally finite collection $\{ X_{\mu} \}_{\mu \in J}$ of subsets of $M$, the collection $\{ \ul{\phi}^{-1}(X_{\mu}) \}_{\mu \in J}$ is also locally finite. For each $\mu \in J$ and $n \in N$, one has $[\phi^{\ast}_{M}(\lambda_{\mu})]_{n} = \phi_{(n)}( [\lambda_{\mu}]_{\ul{\phi}(n)})$, see (\ref{eq_vartphixmap}). As $\phi_{(n)}$ is a graded algebra morphism, this proves that $\supp( \phi^{\ast}_{M}(\lambda_{\mu})) \subseteq \ul{\phi}^{-1}( \supp(\lambda_{\mu})) \subseteq \ul{\phi}^{-1}(U_{\mu})$. It also follows from the preceding remark that $\{ \supp(\phi^{\ast}_{M}(\lambda_{\mu})) \}_{\mu \in J}$ is a locally finite collection. Now, for each $n \in N$, one can choose $V_{n} := \ul{\phi}^{-1}(U_{\ul{\phi}(n)})$, where $U_{\ul{\phi}(n)} \cap \supp(\lambda_{\mu}) \neq \emptyset$ only if $\mu \in J_{0}$, where $J_{0} \subseteq J$ is a finite subset. Then
\begin{equation}
(\sum_{\mu \in J} \phi^{\ast}_{M}(\lambda_{\mu}))|_{V_{n}} \equiv \sum_{\mu \in J_{0}} \phi^{\ast}_{M}(\lambda_{\mu})|_{V_{n}} = \phi^{\ast}_{U_{\ul{\phi}(n)}}( \sum_{\mu \in J_{0}} \lambda_{\mu}|_{U_{\ul{\phi}(n)}}) = \phi^{\ast}_{U_{\ul{\phi}(n)}}( 1) = 1.
\end{equation}
Hence $\sum_{\mu \in J} \phi^{\ast}_{M}(\lambda_{\mu}) = 1$. Moreover, one has $\ul{ \phi^{\ast}_{M}(\lambda_{\mu})} = \ul{\lambda_{\mu}} \circ \ul{\phi} \geq 0$ for all $\mu \in J$. This proves that $\{ \phi^{\ast}_{M}(\lambda_{\mu}) \}_{\mu \in J}$ is a partition of unity subordinate to $\{ \ul{\phi}^{-1}(U_{\mu}) \}_{\mu \in J}$. To prove (\ref{eq_pullbackofasumpartition}), note that 
\begin{equation}
\phi^{\ast}_{M}( \lambda_{\mu} \cdot f_{\mu}) = \phi^{\ast}_{M}(\lambda_{\mu}) \cdot \phi^{\ast}_{U_{\mu}}(f_{\mu}),
\end{equation}
for all $\mu \in J$. This can be shown easily by comparing the restrictions of both sides to the open subsets $\ul{\phi}^{-1}(U_{\mu})$ and $\ul{\phi}^{-1}(\supp(\lambda_{\mu})^{c})$ covering $N$. Let $V_{n} = \ul{\phi}^{-1}(U_{\ul{\phi}(n)})$ be as above. Then
\begin{equation}
\begin{split}
(\phi^{\ast}_{M}( \sum_{\mu \in J} \lambda_{\mu} \cdot f_{\mu}))|_{V_{n}} = & \ \phi^{\ast}_{U_{\ul{\phi}(n)}}( \sum_{\mu \in J_{0}} (\lambda_{\mu} \cdot f_{\mu})|_{U_{\ul{\phi}(n)}}) = \sum_{\mu \in J_{0}} \phi^{\ast}_{U_{\ul{\phi}(n)}}( (\lambda \cdot f_{\mu})|_{U_{\ul{\phi}(n)}}) \\
= & \ \sum_{\mu \in J_{0}} (\phi^{\ast}_{M}(\lambda_{\mu} \cdot f_{\mu}))|_{V_{n}} = \sum_{\mu \in J_{0}} ( \phi^{\ast}_{M}(\lambda_{\mu}) \cdot \phi^{\ast}_{U_{\mu}}(f_{\mu}))|_{V_{n}} \\
= & \ ( \sum_{\mu \in J} \phi^{\ast}_{M}(\lambda_{\mu}) \cdot \phi^{\ast}_{U_{\mu}}(f_{\mu}))|_{V_{n}}.
\end{split}
\end{equation}
We have used the naturality $\phi^{\ast}$. As $n \in N$ was arbitrary, the equation (\ref{eq_pullbackofasumpartition}) follows. 
\end{proof}

\begin{tvrz}[\textbf{Graded bump functions}] \label{tvrz_bumpfunctions}
Let $U,V \in \Op(M)$ such that $V \subseteq \ol{V} \subseteq U$. 

Then there exists $\lambda \in \C^{\infty}_{\M}(M)_{0}$, such that $0 \leq \ul{\lambda} \leq 1$, $\supp(\lambda) \subseteq U$ and $\lambda|_{V} = 1$. Such $\lambda$ is called a \textbf{graded bump function supported on $U$}.
\end{tvrz}
\begin{proof}
Consider a partition of unity $\{\lambda,\lambda'\}$ subordinate to the open cover $\{U, \ol{V}^{c} \}$. Clearly $0 \leq \ul{\lambda} \leq 1$ and $\supp(\lambda) \subseteq U$. Moreover, one has $V \subseteq \ol{V} \subseteq \supp(\lambda')^{c}$. Whence $\lambda|_{V} = (\lambda + \lambda')|_{V} = 1$. 
\end{proof}
\begin{tvrz}[\textbf{Extension lemma}] \label{tvrz_extensionlemma}
Let $U \in \Op(M)$ and $f \in \C^{\infty}_{\M}(U)$. 

Then for any $V \in \Op(M)$, such that $V \subseteq \ol{V} \subseteq U$, there exists $g \in \C^{\infty}_{\M}(M)$, such that $f|_{V} = g|_{V}$. 
\end{tvrz}
\begin{proof}
Let $\lambda \in \C^{\infty}_{\M}(M)$ be a graded bump function supported on $U$, such that $\lambda|_{V} = 1$. Define $g \in \C^{\infty}_{\M}(M)$ by its restrictions to the sets of the open cover $\{U, \supp(\lambda)^{c} \}$ as 
\begin{equation} \label{eq_lambdabumpf}
g|_{U} := \lambda|_{U} \cdot f, \; \; g|_{\supp(\lambda)^{c}} := 0. 
\end{equation}
We usually write just $g = \lambda \cdot f$. Then $g|_{V} = \lambda|_{V} \cdot f|_{V} = f|_{V}$ and the proof is finished. 
\end{proof}
This observation has vital consequences for the structure sheaf $\C^{\infty}_{\M}$. The following propositions should be considered as one the most important statements in the entire paper. 
\begin{cor} \label{cor_piUmsurjective}
For every $m \in M$ and $U \in \Op_{m}(M)$, the canonical map $\pi_{U,m}: \C^{\infty}_{\M}(U) \rightarrow \C^{\infty}_{\M,m}$ is a graded algebra epimorphism. 
In other words, for every element $[f]_{m} \in \C^{\infty}_{\M,m}$, one can pick any $U \in \Op_{m}(M)$ and assume that $f \in \C^{\infty}_{\M}(U)$. 
\end{cor}
\begin{proof}
Let $[f]_{m} \in \C^{\infty}_{\M,m}$, represented by some $f \in \C^{\infty}_{\M}(W)$ where $W \in \Op_{m}(M)$. We may assume that $W \subseteq U$. As $M$ is locally Euclidean, one may find $V \in \Op_{m}(M)$ such that $\ol{V} \subseteq W$. By Proposition \ref{tvrz_extensionlemma}, there is $g \in \C^{\infty}_{\M}(U)$ such that $g|_{V} = f|_{V}$. Thus $[f]_{m} = [g]_{m} \equiv \pi_{U,m}(g)$. 
\end{proof}
We can now tie one of the loose ends regarding the body map.
\begin{tvrz} \label{tvrz_bodymapsurj}
The body map $i^{\ast}_{M}: \C^{\infty}_{\M} \rightarrow \C^{\infty}_{M}$ is surjective. Consequently, there is a canonical sheaf isomorphism $\C^{\infty}_{\M} / \J^{\pg}_{\M} \cong \C^{\infty}_{M}$. 
\end{tvrz}
\begin{proof}
Let $\A = \{ (U_{\alpha}, \varphi_{\alpha}) \}_{\alpha \in I}$ be a graded smooth atlas for $\M$. Let $f_{0} \in \C^{\infty}_{M}(U)$ be a given smooth function. For each $\alpha \in I$, one can use the graded local chart $\varphi_{\alpha}$ to define $f_{\alpha} \in \C^{\infty}_{\M}(U \cap U_{\alpha})$ such that $\ul{f_{\alpha}} = f_{0}|_{U \cap U_{\alpha}}$. Now, let $\{ \lambda_{\alpha} \}_{\alpha \in I}$ be the partition of unity on $\M|_{U}$ subordinate to the open cover $\{ U \cap U_{\alpha} \}_{\alpha \in I}$. Set $f := \sum_{\alpha \in I} \lambda_{\alpha} \cdot f_{\alpha}$. It follows from Proposition \ref{tvrz_partitionsnaturalops}-$(ii)$ that $\{ \ul{\lambda_{\alpha}} \}_{\alpha \in I}$ is the partition of unity on $M|_{U}$ subordinate to $\{ U \cap U_{\alpha} \}_{\alpha \in I}$. Moreover, the formula (\ref{eq_pullbackofasumpartition}) reads
\begin{equation}
\ul{f} \equiv (i^{\ast}_{M})_{U}( \sum_{\alpha \in I} \lambda_{\alpha} \cdot f_{\alpha}) =  \sum_{\alpha \in I} \ul{\lambda_{\alpha}} \cdot \ul{f_{\alpha}} = \sum_{\alpha \in I} \ul{\lambda_{\alpha}} \cdot f_{0}|_{U \cap U_{\alpha}} = f_{0},
\end{equation}
where the last equality follows from Proposition \ref{tvrz_partitionsnaturalops}-$(i)$. This finishes the proof.
\end{proof}
Let $\J \subseteq \C^{\infty}_{\M}$ be a given sheaf of ideals of $\C^{\infty}_{\M}$. As we have noted in Remark \ref{rem_quotientsheaf}, there is the quotient presheaf $\C^{\infty}_{\M} / \J$. It turns out that in this case, one can say a bit more.
\begin{tvrz} \label{eq_quotientpresheafissheaf}
For any sheaf of ideals $\J \subseteq \C^{\infty}_{\M}$, the quotient presheaf $\C^{\infty}_{\M} / \J$ is a sheaf. 
\end{tvrz}
\begin{proof}
Let $U \in \Op(M)$ and let $\{ U_{\alpha} \}_{\alpha \in I}$ be its open cover. Let $f,g \in \C^{\infty}_{\M}(U)$ be a pair of functions such that the sections $[f],[g] \in \C^{\infty}_{\M}(U) / \J(U)$ satisfy $[f]|_{U_{\alpha}} = [g]|_{U_{\alpha}}$ for each $\alpha \in I$. We thus have $h_{\alpha} := (f-g)|_{U_{\alpha}} \in \J(U_{\alpha})$ for each $\alpha \in I$. As $\J$ is a sheaf of ideals, there is a unique function $h \in \J(U)$, such that $h|_{U_{\alpha}} = h_{\alpha}$ for all $\alpha \in I$. By the monopresheaf property of $\C^{\infty}_{\M}$, we have $h = f - g$, and thus $[f] = [g]$. This proves that $\C^{\infty}_{\M} / \J$ has the monopresheaf property. Observe that this is true for any sheaf of ideals in \textit{any} sheaf.

Next, suppose that for each $\alpha \in I$, one has $[f_{\alpha}] \in \C^{\infty}_{\M}(U_{\alpha}) / \J(U_{\alpha})$, such that $[f_{\alpha}]|_{U_{\alpha \beta}} = [f_{\beta}]|_{U_{\alpha \beta}}$ for every $(\alpha,\beta) \in I^{2}$. For each $(\alpha,\beta) \in I^{2}$, there is thus $h_{\alpha \beta} := f_{\beta}|_{U_{\alpha \beta}} - f_{\alpha}|_{U_{\alpha \beta}} \in \J(U_{\alpha \beta})$. Let $\{ \lambda_{\alpha} \}_{\alpha \in I}$ be a partition of unity subordinate to $\{ U_{\alpha} \}_{\alpha \in I}$ and define $f := \sum_{\beta \in I} \lambda_{\beta} \cdot f_{\beta} \in \C^{\infty}_{\M}(U)$. 

We claim that $[f]|_{U_{\alpha}} = [f_{\alpha}]$ for every $\alpha \in I$. For each $m \in U_{\alpha}$, pick $V \in \Op_{m}(U_{\alpha})$, such that $V \cap \supp(\lambda_{\beta}) \neq \emptyset$ only if $\beta \in I_{0}$, where $I_{0} \subseteq I$ is a finite subset. For each $\beta \in I$, one can write
\begin{equation}
(\lambda_{\beta} \cdot f_{\beta})|_{V \cap U_{\beta}} = \lambda_{\beta}|_{V \cap U_{\beta}} \cdot f_{\beta}|_{V \cap U_{\beta}} = \lambda_{\beta}|_{V \cap U_{\beta}} \cdot ( f_{\alpha}|_{V \cap U_{\beta}} + h_{\alpha \beta}|_{V \cap U_{\beta}}).
\end{equation}
Next, observe that $\supp(\lambda_{\beta}|_{V})^{c} = \supp(\lambda_{\beta})^{c} \cap V$ and we have $(\lambda_{\beta} \cdot f_{\beta})|_{\supp(\lambda_{\beta}|_{V})^{c}} = 0$. This proves that the restriction $(\lambda_{\beta} \cdot f_{\beta})|_{V} \in \C^{\infty}_{\M}(V)$ can be written as 
\begin{equation} (\lambda_{\beta} \cdot f_{\beta})|_{V} = \lambda_{\beta}|_{V} \cdot (f_{\alpha}|_{V \cap U_{\beta}} + h_{\alpha \beta}|_{V \cap U_{\beta}}). \end{equation}
Observe that on the right-hand side there is not an actual product, but a function defined by (\ref{eq_skoroproduct}). Also note that $\{ \lambda_{\alpha}|_{V} \}_{\alpha \in I}$ is a partition of unity on $\M|_{V}$ subordinate to $\{ V \cap U_{\alpha} \}_{\alpha \in I}$. Then 
\begin{equation}
\begin{split}
f|_{V} = & \  \sum_{\beta \in I_{0}} (\lambda_{\beta} \cdot f_{\beta})|_{V} = \sum_{\beta \in I_{0}} \lambda_{\beta}|_{V} \cdot (f_{\alpha}|_{V \cap U_{\beta}} + h_{\alpha \beta}|_{V \cap U_{\beta}}) \\
= & \ \sum_{\beta \in I} \lambda_{\beta}|_{V} \cdot f_{\alpha}|_{V \cap U_{\beta}} + \sum_{\beta \in I_{0}} \lambda_{\beta}|_{V} \cdot h_{\alpha \beta}|_{V \cap U_{\beta}} \\
= & \ f_{\alpha}|_{V} + \sum_{\beta \in I_{0}} \lambda_{\beta}|_{V} \cdot h_{\alpha \beta}|_{V \cap U_{\beta}}.
\end{split}
\end{equation}
Now, each summand of the finite sum on the right-hand side is an element of $\J(V)$. This can be seen by restricting it to $V \cap U_{\beta}$ and $\supp(\lambda_{\beta}|_{V})^{c}$ and using the fact that $\J$ is a \textit{sheaf} of ideals. Therefore, we have just proved that $[f]|_{V} = [f_{\alpha}|_{V}]$. As $m \in U_{\alpha}$ was arbitrary, this proves that $[f]|_{U_{\alpha}} = [f_{\alpha}]$ and the proof is finished. 
\end{proof}
\subsection{Products of graded manifolds}
One of the most important constructions in differential geometry is the ability to equip a Cartesian product of two manifolds with an essentially unique smooth structure making it it into a product in the category of smooth manifolds. We will now repeat this construction in the graded setting. 

\begin{tvrz} \label{tvrz_products}
Let $\M = (M,\C^{\infty}_{\M})$ and $\cN = (N, \C^{\infty}_{\cN})$ be a pair of graded manifolds. Then:
\begin{enumerate}[(i)]
\item There exists a graded manifold $\M \times \cN = (M \times N, \C^{\infty}_{\M \times \cN})$, such that 
\begin{equation} \gdim(\M \times \cN) = \gdim(\M) + \gdim(\cN).\end{equation}
\item There are two graded smooth maps $\pi_{\M}: \M \times \cN \rightarrow \M$ and $\pi_{\cN}: \M \times \cN \rightarrow \cN$ such that $\ul{\pi_{\M}} = \pi_{M}$ and $\ul{\pi_{\cN}} = \pi_{N}$. Here $\pi_{M}$ and $\pi_{N}$ denote the usual projections.
\item $\M \times \cN$ together with $\pi_{\M}$ and $\pi_{\cN}$ becomes a binary product in $\gMan^{\infty}$. This property determines $\M \times \cN$ uniquely up to a graded diffeomorphism. 
\end{enumerate}
\end{tvrz}
\begin{proof}
Let $\A = \{ (U_{\alpha},\varphi_{\alpha}) \}_{\alpha \in I}$ and $\B = \{ (V_{\rho}, \psi_{\rho}) \}_{\rho \in J}$ be a graded smooth atlas for $\M$ and $\cN$, respectively. Then $\ul{\A \times \B} := \{ (U_{\alpha} \times V_{\rho}, \ul{\varphi_{\alpha}} \times \ul{\psi_{\rho}}) \}_{(\alpha,\rho) \in I \times J}$ forms an ordinary smooth atlas for $M \times N$ equipped with the product topology. Let $(n_{j})_{j \in \Z} := \gdim(\M)$ and $(m_{j})_{j \in \Z} := \gdim(\cN)$.  

The idea is to use Proposition \ref{tvrz_gmgluing1} to construct a sheaf $\C^{\infty}_{\M \times \cN}$ together with a graded smooth atlas $\A \times \B = \{ (U_{\alpha} \times V_{\rho}, \varphi_{\alpha} \times \psi_{\rho}) \}_{(\alpha,\rho) \in I \times J}$. One only needs to construct the respective transition maps. For any $\hat{U} \in \Op(\R^{n_{0}})$ and $\hat{V} \in \Op(\R^{m_{0}})$, one has canonical graded smooth maps
\begin{equation}
\hat{\pi}_{1}: (\hat{U} \times \hat{V})^{(n_{j} + m_{j})} \rightarrow \hat{U}^{(n_{j})}, \; \; \hat{\pi}_{2}: (\hat{U} \times \hat{V})^{(n_{j} + m_{j})} \rightarrow \hat{V}^{(m_{j})},
\end{equation}
such that $\ul{\hat{\pi}_{1}}: \hat{U} \times \hat{V} \rightarrow \hat{U}$ is the canonical projection, and similarly for $\hat{\pi}_{2}$. Next, note that
\begin{equation}
(\ul{\varphi_{\alpha}} \times \ul{\psi_{\rho}})((U_{\alpha} \times V_{\rho}) \cap (U_{\beta} \times V_{\sigma})) = \ul{\varphi_{\alpha}}(U_{\alpha \beta}) \times \ul{\psi_{\rho}}(V_{\rho \sigma}),
\end{equation}
for all $\alpha,\beta \in I$ and $\rho,\sigma \in J$. Each transition map $(\varphi \times \psi)_{(\alpha \rho)(\beta \sigma)}$ can be then constructed as the unique graded smooth map fitting into the pair of commutative diagrams  
\begin{equation}
\begin{tikzcd}[column sep=huge]
(\ul{\varphi_{\beta}}(U_{\alpha \beta}) \times \ul{\psi_{\sigma}}(V_{\rho \sigma}))^{(n_{j} + m_{j})} \arrow{d}{\hat{\pi}_{1}} \arrow[dashed]{r}{(\varphi \times \psi)_{(\alpha \rho)(\beta \sigma)}}& (\ul{\varphi_{\alpha}}(U_{\alpha \beta}) \times \ul{\psi_{\rho}}(V_{\rho \sigma}))^{(n_{j} + m_{j})} \arrow{d}{\hat{\pi}_{1}} \\
\ul{\varphi_{\beta}}(U_{\alpha \beta})^{(n_{j})} \arrow{r}{\varphi_{\alpha \beta}} & \ul{\varphi_{\alpha}}(U_{\alpha \beta})^{(n_{j})}
\end{tikzcd},
\end{equation}
\begin{equation}
\begin{tikzcd}[column sep=huge]
(\ul{\varphi_{\beta}}(U_{\alpha \beta}) \times \ul{\psi_{\sigma}}(V_{\rho \sigma}))^{(n_{j} + m_{j})} \arrow{d}{\hat{\pi}_{2}} \arrow[dashed]{r}{(\varphi \times \psi)_{(\alpha \rho)(\beta \sigma)}}& (\ul{\varphi_{\alpha}}(U_{\alpha \beta}) \times \ul{\psi_{\rho}}(V_{\rho \sigma}))^{(n_{j} + m_{j})} \arrow{d}{\hat{\pi}_{2}} \\
\ul{\psi_{\sigma}}(V_{\rho \sigma})^{(m_{j})} \arrow{r}{\psi_{\rho \sigma}} & \ul{\psi_{\rho}}(V_{\rho \sigma})^{(m_{j})}
\end{tikzcd}.
\end{equation}
These diagrams can also be used to prove the cocycle condition (\ref{eq_cocyclegmgluing}). Hence by Proposition \ref{tvrz_gmgluing1}, we obtain a graded manifold $\M \times \cN$ and the claim $(i)$ is proved. 

Next, for each $(\alpha,\rho) \in I \times J$, define $\pi_{\M}^{(\alpha,\rho)}: (\M \times \cN)|_{U_{\alpha} \times V_{\rho}} \rightarrow \M$ as 
\begin{equation} \pi^{(\alpha,\rho)}_{\M} := \varphi_{\alpha}^{-1} \circ \hat{\pi}_{1} \circ (\varphi_{\alpha} \times \psi_{\rho}).
\end{equation}
Using the above commutative diagrams, it is easy to check that restrictions of those maps agree on the overlaps of the open cover $\{ U_{\alpha} \times V_{\rho} \}_{(\alpha,\rho) \in I \times J}$ of $M \times N$. By Proposition \ref{tvrz_gLRSgluing}, there is then the unique graded smooth map $\pi_{\M}: \M \times \cN \rightarrow \M$, such that $\pi_{\M}|_{U_{\alpha} \times V_{\rho}} = \pi_{\M}^{(\alpha,\rho)}$. Obviously $\ul{\pi_{\M}} = \pi_{M}$. The construction of $\pi_{\cN}$ is analogous and the claim $(ii)$ is proved. 

To show that $\M \times \cN$ together with $\pi_{\M}$ and $\pi_{\cN}$ forms a binary product in $\gMan^{\infty}$, we have to prove that it has the \textit{universal property}: to any graded manifold $\cS = (S, \C^{\infty}_{\cS})$ and a pair of graded smooth maps $\phi: \cS \rightarrow \M$, $\chi: \cS \rightarrow \cN$, there exists a \textit{unique} graded smooth map $(\phi,\chi): \cS \rightarrow \M \times \cN$, such that $\pi_{\M} \circ (\phi,\psi) = \phi$ and $\pi_{\cN} \circ (\phi,\psi) = \chi$. 

Let $(k_{j})_{j \in \Z} := \gdim(\cS)$. Let $\hat{U} \subseteq \R^{n_{0}}$ and $\hat{V} \subseteq \R^{m_{0}}$. It straightforward to show that to any $\hat{W} \subseteq \R^{k_{0}}$ and a pair of graded smooth maps $\hat{\phi}: \hat{W}^{(k_{j})} \rightarrow \hat{U}^{(n_{j})}$ and $\hat{\chi}: \hat{W}^{(k_{j})} \rightarrow \hat{V}^{(m_{j})}$, there exists a unique graded smooth map $(\hat{\phi},\hat{\chi}): \hat{W}^{(k_{j})} \rightarrow (\hat{U} \times \hat{V})^{(n_{j} + m_{j})}$, such that $\hat{\pi}_{1} \circ (\hat{\phi},\hat{\chi}) = \hat{\phi}$ and $\hat{\pi}_{2} \circ (\hat{\phi},\hat{\psi}) = \hat{\chi}$. 

Now, for each $s \in S$, there is a graded local chart $(U_{\alpha}, \varphi_{\alpha})$ for $\M$ containing $\ul{\phi}(s)$, a graded local chart $(V_{\rho}, \psi_{\rho})$ for $\cN$ containing $\ul{\chi}(s)$, and a graded local chart $(W,\eta)$ for $\cS$ containing $s$, such that $\ul{\phi}(W) \subseteq U_{\alpha}$ and $\ul{\chi}(W) \subseteq V_{\rho}$. If $(\phi,\chi)$ exists, its restriction to $W$ must satisfy the formula
\begin{equation} \label{eq_(phi,chi)restriction}
(\phi,\chi)|_{W} = (\varphi_{\alpha} \circ \psi_{\rho})^{-1} \circ ( \varphi_{\alpha} \circ \phi|_{W} \circ \eta^{-1}, \psi_{\rho} \circ \chi|_{W} \circ \eta^{-1}) \circ \eta.
\end{equation}
This follows immediately from the uniqueness claim in the previous paragraph. Denote the graded smooth map on the right-hand side as $(\phi,\chi)_{(W)}: \cS|_{W} \rightarrow \M \times \cN$. By repeating this procedure for every $s \in S$, we obtain an open cover $\{ W_{s} \}_{s \in S}$ of $S$ together with a collection $\{ (\phi,\psi)_{(W_{s})} \}_{s \in S}$ of graded smooth maps. It is not difficult to see that those maps agree on the overlaps, hence by Proposition \ref{tvrz_gLRSgluing}, they define a unique graded smooth map $(\phi,\chi): \cS \rightarrow \M \times \cN$, such that $(\phi,\chi)|_{W_{s}} = (\phi,\chi)_{(W_{s})}$ for every $s \in S$. Clearly $\pi_{\M} \circ (\phi,\chi) = \phi$, $\pi_{\cN} \circ (\phi,\chi) = \chi$, and it is the unique such graded smooth map. This proves the claim $(iii)$. 
\end{proof}

\begin{cor}
Let $\phi: \M \rightarrow \M'$ and $\psi: \cN \rightarrow \cN'$ be two graded smooth maps of graded smooth manifolds. Then there exists a unique map $\phi \times \psi: \M \times \cN \rightarrow \M' \times \cN'$ satisfying 
\begin{equation}
\pi_{\M'} \circ (\phi \times \psi) = \phi \circ \pi_{\M}, \; \; \pi_{\cN'} \circ (\phi \times \psi) = \psi \circ \pi_{\cN}.
\end{equation}
\end{cor}
\begin{proof}
Let $\phi \times \psi := (\phi \circ \pi_{\M}, \psi \circ \pi_{\cN})$.
\end{proof}
\begin{example} \label{ex_productgdomains}
Let $\hat{U}^{(n_{j})}$ and $\hat{V}^{(m_{j})}$ be a pair of graded domains. 

It follows from Theorem \ref{thm_globaldomain} that $(\hat{U} \times \hat{V})^{(n_{j} + m_{j})}$ together with the canonical projections $\hat{\pi}_{1}$ and $\hat{\pi}_{2}$ also form a binary product of $\hat{U}^{(n_{j})}$ and $\hat{V}^{(n_{j})}$ in $\gMan^{\infty}$. There is thus a unique graded diffeomorphism
\begin{equation}
(\hat{U} \times \hat{V})^{(n_{j}+m_{j})} \cong \hat{U}^{(n_{j})} \times \hat{V}^{(n_{j})}. 
\end{equation}
In particular, for \textit{any} graded local charts $(U,\varphi)$ and $(V,\psi)$ for $\M$ and $\cN$, respectively, the graded smooth map $\varphi \times \psi: \M \times \cN|_{U \times V} \rightarrow \hat{U}^{(n_{j})} \times \hat{V}^{(m_{j})} \cong (\hat{U} \times \hat{V})^{(n_{j} + m_{j})}$ defines a graded chart $(U \times V, \varphi \times \psi)$ for $\M \times \cN$. In fact, all the graded local charts we have obtained in the proof of Proposition \ref{tvrz_products} are precisely of this form. 
\end{example}
\section{Infinitesimal theory} \label{sec_infinitesimal}
\subsection{Tangent spaces and induced maps}
Let $\M = (M, \C^{\infty}_{\M})$ be a graded manifold. We shall now define a graded analogue of a tangent space at a given point $m \in M$. 

First, observe that a (trivially) graded vector space $\R$ has a natural structure of a graded $\C^{\infty}_{\M,m}$-module. For $[f]_{m} \in \C^{\infty}_{\M,m}$ represented by $f \in \C^{\infty}_{\M}(U)$ for some $U \in \Op_{m}(M)$, define 
\begin{equation} \label{eq_StalkactiononR}
[f]_{m} \tr \lambda := f(m) \cdot \lambda,
\end{equation}
for all $\lambda \in \R$. This is well-defined as $f(m)$ does not depend on the function representing $[f]_{m}$. It is trivial to verify the properties (\ref{eq_moduleaxioms}). The following definition thus makes sense.

\begin{definice}
Let $\M = (M, \C^{\infty}_{\M})$ be a graded manifold and let $m \in M$. The \textbf{tangent space at $m$ to $\M$} is the graded vector space
\begin{equation}
T_{m}\M := \gDer( \C^{\infty}_{\M,m}, \R),
\end{equation}
where the $\C^{\infty}_{\M,m}$-module structure on $\R$ is described above. 
\end{definice}

\begin{tvrz} \label{tvrz_differential}
Let $\M = (M, \C^{\infty}_{\M})$ and $\cN = (N, \C^{\infty}_{\cN})$ be a pair of graded manifolds. Let $\phi: \M \rightarrow \cN$ be a graded smooth map. Let $m \in M$ be a given point.

Then there is a graded linear map $T_{m}\phi: T_{m}\M \rightarrow T_{\ul{\phi}(m)} \cN$ defined for each  $v \in T_{m}\M$ by 
\begin{equation} \label{eq_differentialdef} (T_{m}\phi)(v) := v \circ \phi_{(m)}, \end{equation}
where $\phi_{(m)}: \C^{\infty}_{\cN,\ul{\phi}(m)} \rightarrow \C^{\infty}_{\M,m}$ is the graded algebra morphism (\ref{eq_vartphixmap}). If $\psi: \cN \rightarrow \cS$ is another graded smooth map for a graded manifold $\cS = (S, \C^{\infty}_{\cS})$, one finds
\begin{equation} \label{tvrz_differentialfunctorial}
T_{m}(\psi \circ \phi) = T_{\ul{\phi}(m)}\psi \circ T_{m}\phi, \; \; T_{m}\1_{\M} = \1_{T_{m}\M}.
\end{equation}
The map $T_{m}\phi$ is called the \textbf{differential of $\phi$ at $m$}. 
\end{tvrz}
\begin{proof}
One has to argue that $v \circ \phi_{(m)} \in \gDer(\C^{\infty}_{\M,m}, \R)$. But this follows immediately from the commutativity of (\ref{eq_cdbodymaps}), the equation (\ref{eq_vartphixmap2}) and Proposition \ref{tvrz_derivations}-$(ii)$. The composition rule in (\ref{tvrz_differentialfunctorial}) follows from the fact that $(\psi \circ \varphi)_{(m)} = \varphi_{(m)} \circ \psi_{\ul{\varphi}(m)}$. This can be show directly from (\ref{eq_vartphixmap2}). Finally, one clearly has $(\1_{\M})_{(m)} = \1_{\C^{\infty}_{\M,m}}$. This proves the final claim. 
\end{proof}
In the following paragraphs, we are going to prove that $T_{m}\M$ is a finite-dimensional graded vector space and find its convenient total basis. To do so, we need several technical statements. Let $m \in M$ be a fixed point a let $(x^{1},\dots,x^{n},\xi_{1},\dots,\xi_{n_{\ast}})$ be the coordinate functions corresponding to some graded local chart $(U,\varphi)$ around $m$, see Example \ref{ex_coordfunctions}. We can thus consider the germs $[x^{i} - x^{i}(m)]_{m} \in \C^{\infty}_{\M,m}$ and $[\xi_{\mu}]_{m} \in \C^{\infty}_{\M,m}$. They play an important role in the following ``stalk analogue'' of Lemma \ref{lem_Hadamard} and Proposition \ref{tvrz_inallvanishingideals}.
\begin{lemma}[\textbf{Local graded Hadamard}] \label{lem_lochadamard}
With the notation established in the above paragraph, the following statements are true:
\begin{enumerate}[(i)]
\item Let $[f]_{m} \in \C^{\infty}_{\M,m}$. Then to any $q \in \N_{0}$, there is a decomposition
\begin{equation}
[f]_{m} = \hat{T}^{q}[f]_{m} + \hat{R}^{q}[f]_{m},
\end{equation}
where $\hat{T}^{q}[f]_{m}$ is a polynomial in variables $\{ [x^{i} - x^{i}(m)]_{m} \}_{i=1}^{n_{0}}$ and $\{ [\xi_{\mu}]_{m} \}_{\mu=1}^{n_{\ast}}$ of degree $q$, and $\hat{R}^{q}[f]_{m} \in \frJ(\C^{\infty}_{\M,m})^{q+1}$. 
\item Let $[h]_{m} \in \C^{\infty}_{\M,m}$. Then $[h]_{m} = 0$, if and only if there is $V \in \Op_{m}(M)$ and $h \in \C^{\infty}_{\M}(V)$ representing $[h]_{m}$, such that $[h]_{n} \in \frJ( \C^{\infty}_{\M,n})^{q}$ for all $n \in V$ and $q \in \N$.
\end{enumerate}
\end{lemma}
\begin{proof}
Let $[f]_{m} \in \C^{\infty}_{\M,m}$. By Corollary \ref{cor_piUmsurjective}, we may assume that $f \in \C^{\infty}_{\M}(U)$, where $(U,\varphi)$ is any given graded local chart around $m$. Let $\hat{f} \in \C^{\infty}_{(n_{j})}(\hat{U})$ be the local representative of $f$ with respect to $(U,\varphi)$. According to Lemma \ref{lem_Hadamard}, one may write $\hat{f} = T^{q}_{\ul{\varphi}(m)}(\hat{f}) + R^{q}_{\ul{\varphi}(m)}(\hat{f})$. Let
\begin{equation}
\hat{T}^{q}[f]_{m} := \varphi_{(m)}( [ T^{q}_{\ul{\varphi}(m)}(\hat{f})]_{\ul{\varphi}(m)}), \; \; \hat{R}^{q}[f]_{m} := \varphi_{(m)}( [R^{q}_{\ul{\varphi}(m)}(\hat{f})]_{\ul{\varphi}(m)}).
\end{equation}
It follows from Proposition \ref{tvrz_invertibility}-(iii) that $\pi_{\hat{U},\ul{\varphi}(m)}( \J^{\ul{\varphi}(m)}_{(n_{j})}(\hat{U})) \subseteq \frJ( \C^{\infty}_{(n_{j}), \ul{\varphi}(m)})$. Consequently, one has $[R^{q}_{\ul{\varphi}(m)}(\hat{f})]_{\ul{\varphi}(m)} \in (\frJ( \C^{\infty}_{(n_{j}), \ul{\varphi}(m)}))^{q+1}$. As $\varphi_{(m)}: \C^{\infty}_{(n_{j}), \ul{\varphi}(m)} \rightarrow \C^{\infty}_{\M,m}$ is a graded local ring morphism, this proves that $\hat{R}^{q}[f]_{m} \in \frJ(\C^{\infty}_{\M,m})^{q+1}$. 

The rest of the claim $(i)$ follows from the equations
\begin{equation}
\varphi_{(m)}( [x^{i} - x^{i}(\ul{\varphi}(m))]_{\ul{\varphi(m)}}) = [x^{i} - x^{i}(m)]_{m}, \; \; \varphi_{(m)}( [\xi_{\mu}]_{\ul{\varphi}(m)}) = [\xi_{\mu}]_{m}.
\end{equation}
Note that on the left-hand sides, there are coordinate functions in $\C^{\infty}_{(n_{j})}(\hat{U})$, whereas on the right-hand sides, there are their pullbacks to $\C^{\infty}_{\M}(U)$, see Example \ref{ex_coordfunctions}. 

Let us proceed with the proof of the claim $(ii)$. The statement is local, so it suffices to prove it for graded domains. But with the help of Proposition \ref{tvrz_invertibility}-(iii), the proof is then completely analogous to the one of Proposition \ref{tvrz_inallvanishingideals}.
\end{proof}
Moreover, we obtain the following local analogue of Proposition \ref{tvrz_vanishingidealgen}:
\begin{tvrz} \label{tvrz_Jacobsradicalgenset}
For any $m \in U$, the Jacobson radical $\frJ( \C^{\infty}_{\M,m})$ is generated by the graded set $\frS_{\M,m} \subseteq \C^{\infty}_{\M,m}$, where we define 
\begin{align}
(\frS_{\M,m})_{0} := & \ \{ [x^{i} - x^{i}(m)]_{m} \; | \; i \in \{1,\dots,n_{0}\} \}, \\
(\frS_{\M,m})_{k} = & \ \{ [\xi_{\mu}]_{m} \; | \; \mu \in \{1, \dots, n_{\ast} \} \text{ such that } |\xi_{\mu}| = k \} \text{ for } k \neq 0.
\end{align}
Moreover, there is a canonical decomposition $\C^{\infty}_{\M,m} = \R \oplus \frJ(\C^{\infty}_{\M,m})$ of graded vector spaces, where $\R$ corresponds to the graded subspace of all scalar multiples of the algebra unit in $\C^{\infty}_{\M,m}$. 
\end{tvrz}
\begin{proof}
It suffices to prove it for graded domains. Use Proposition \ref{tvrz_invertibility}-$(iii)$ and Proposition \ref{tvrz_vanishingidealgen}. 
\end{proof}

\begin{example} \label{ex_coordtangentvectors}
Let us now construct the two most important examples of tangent vectors at a given point $m \in M$. Let $(x^{1},\dots,x^{n_{0}},\xi_{1},\dots,\xi_{n_{\ast}})$ be the coordinate functions corresponding to a given graded local chart $(U,\varphi)$ around $m$. 

For every $i \in \{1,\dots,n_{0}\}$, define a tangent vector $\frac{\partial}{\partial x^{i}}|_{m} \in (T_{m}\M)_{0}$ for each $[f]_{m} \in \C^{\infty}_{\M,m}$ by
\begin{equation}
\frac{\partial}{\partial x^{i}}|_{m}( [f]_{m}) := \frac{\partial}{\partial x^{i}}|_{m}( \ul{f}),
\end{equation}
where we can assume that $f \in \C^{\infty}_{\M}(U)$, and on the right-hand side, there is the ordinary tangent vector $\frac{\partial}{\partial x^{i}}|_{m} \in T_{m}M$ acting on the smooth function $\ul{f} \in \C^{\infty}_{M}(U)$. Recall that 
\begin{equation}
\frac{\partial}{\partial x^{i}}|_{m}( \ul{f}) := \frac{\partial( \ul{f} \circ \ul{\varphi})}{\partial x^{i}}(\ul{\varphi}(m))
\end{equation}
It is obvious that $\frac{\partial}{\partial x^{i}}|_{m}$ is linear and satisfies the Leibniz rule (\ref{eq_derlleibniz}). 

Next, for every $\mu \in \{1, \dots, n_{\ast} \}$, let us define $\frac{\partial}{\partial \xi_{\mu}}|_{m} \in (T_{m}\M)_{-|\xi_{\mu}|}$. Its only non-trivial component is the one acting on $(\C^{\infty}_{\M,m})_{|\xi_{\mu}|}$, hence let $[f]_{m} \in (\C^{\infty}_{\M,m})_{|\xi_{\mu}|}$. We may always represent it by some $f \in \C^{\infty}_{\M}(U)_{|\xi_{\mu}|}$. Let $\hat{f} \in \C^{\infty}_{(n_{j})}(\hat{U})$ be its local representative with respect to $(U,\varphi)$. The formal power series expansion of $\hat{f}$ contains a linear term $\hat{f}_{\mu} \xi_{\mu}$ (no summation assumed) for a unique function $\hat{f}_{\mu} \in \C^{\infty}_{n_{0}}(\hat{U})$. Set 
\begin{equation}
\frac{\partial}{\partial \xi_{\mu}}|_{m}([f]_{m}) := \hat{f}_{\mu}( \ul{\varphi}(m)).
\end{equation}
This does not depend on the particular choice of $f$ and it defines a graded linear map of degree $-|\xi_{\mu}|$. The Leibniz rule (\ref{eq_derlleibniz}) has to be verified only for the case where $|[f]_{m}| = |\xi_{\mu}|$ and $|[g]_{m}| = 0$. All the other cases are trivial or follow from the graded symmetry. We can write $[f]_{m} \cdot [g]_{m} = [f \cdot g]_{m}$ for $f,g \in \C^{\infty}_{\M}(U)$. Then $(\hat{f} \cdot \hat{g})_{\mu} = \hat{f}_{\mu} \cdot \hat{g}_{\mathbf{0}}$, whence 
\begin{equation}
\frac{\partial}{\partial \xi_{\mu}}|_{m}( [f]_{m} \cdot [g]_{m}) = \hat{f}_{\mu}(\ul{\varphi}(m)) \cdot \hat{g}_{\mathbf{0}}(\ul{\varphi}(m)) = \frac{\partial}{\partial \xi_{\mu}}|_{m}([f]_{m}) \tl [g]_{m}.
\end{equation}
This is the Leibniz rule (\ref{eq_derlleibniz}) as $\frac{\partial}{\partial \xi_{\mu}}|_{m}([g]_{m}) = 0$. Hence $\frac{\partial}{\partial \xi_{\mu}}|_{m} \in (T_{m}\M)_{-|\xi_{\mu}|}$. Let us summarize the values of these tangent vectors on germs of coordinate functions:
\begin{equation} \label{eq_basisvectorsongerms}
\frac{\partial}{\partial x^{i}}|_{m}( [x^{j}]_{m}) = \delta^{j}_{i}, \; \; \frac{\partial}{\partial x^{i}}|_{m}( [\xi_{\mu}]_{m}) = 0, \; \; \frac{\partial}{\partial \xi_{\mu}}|_{m}( [\xi_{\nu}]_{m}) = \delta^{\mu}_{\nu}, \; \; \frac{\partial}{\partial \xi_{\mu}}|_{m}([x^{i}]_{m}) = 0,
\end{equation}
for all $i,j \in \{1,\dots,n_{0}\}$ and $\mu,\nu \in \{1,\dots,n_{\ast} \}$.
\end{example}
We are now ready to prove the main result of this subsection.
\begin{tvrz} \label{tvrz_coordtangentvectors}
Let $\M = (M, \C^{\infty}_{\M})$ be a graded manifold of a graded dimension $(n_{j})_{j \in \Z}$. Fix $m \in M$ and a graded coordinate chart $(U,\varphi)$ around $m$. Let $(x^{1},\dots,x^{n_{0}},\xi_{1},\dots,\xi_{n_{\ast}})$ be the coordinate functions corresponding to $(U,\varphi)$. 

Then every tangent vector $v \in T_{m}\M$ is uniquely determined by its values $v([x^{i}]_{m})$ and $v([\xi_{\mu}]_{m})$ on germs of coordinate functions, $i \in \{1,\dots,n_{0} \}$ and $\mu \in \{1,\dots,n_{\ast} \}$. 

Moreover, the collections $\{ \frac{\partial}{\partial x^{i}}|_{m} \}_{i = 1}^{n_{0}}$ and $ \{ \frac{\partial}{\partial \xi_{\mu}}|_{m} \}_{\mu=1}^{n_{\ast}}$ constructed in the above example form a total basis of $T_{m}\M$. In particular, the graded dimension of $T_{m}\M$ is $(n_{-j})_{j \in \Z}$.
\end{tvrz}
\begin{proof}
Let $v \in T_{m}\M$. Define a vector field $\hat{v} \in T_{m}\M$ to be the linear combination
\begin{equation}
\hat{v} := v([x^{i}]_{m}) \frac{\partial}{\partial x^{i}}|_{m} + v([\xi_{\mu}]_{m}) \frac{\partial}{\partial \xi_{\mu}}|_{m}.
\end{equation}
Note that $|\hat{v}| = |v|$, as $v([f]_{m}) \neq 0$ only for $|f| + |v| = 0$. We have to argue that $v = \hat{v}$. Let $[f]_{m} \in \C^{\infty}_{\M,m}$. Using Lemma \ref{lem_lochadamard}, one can write $[f]_{m} = \hat{T}^{1}[f]_{m} + \hat{R}^{1}[f]_{m}$, where $\hat{T}^{1}[f]_{m}$ is a polynomial of degree $1$ in variables $[x^{i} - x^{i}(m)]_{m}$ and $[\xi_{\mu}]_{m}$ and $\hat{R}^{1}[f]_{m} \in \frJ(\C^{\infty}_{\M,m})^{2}$. As $v([x^{i} - x^{i}(m)]_{m}) = v([x^{i}]_{m})$ thanks to Proposition \ref{tvrz_derivations}-$(i)$, it follows from (\ref{eq_basisvectorsongerms}) that $v(\hat{T}^{1}[f]_{m}) = \hat{v}(\hat{T}^{1}[f]_{m})$. Moreover, the Leibniz rule (\ref{eq_derlleibniz}) and Proposition \ref{tvrz_invertibility}-$(iii)$ imply that $\frJ(\C^{\infty}_{\M,m})^{2}$ is in kernel of any tangent vector, whence $v([f]_{m}) = \hat{v}([f]_{m})$. 

Finally, the collections $\{ \frac{\partial}{\partial x^{i}}|_{m} \}_{i=1}^{n_{0}}$ and $\{ \frac{\partial}{\partial \xi_{\mu}}|_{m} \}_{\mu=1}^{n_{\ast}}$ form a total basis of $T_{m}\M$ since the corresponding coefficients can be recovered by acting on germs of coordinate functions. 
\end{proof}
There is one crucial consequence of this proposition.
\begin{tvrz}[\textbf{Graded dimension is an invariant}] \label{tvrz_gdiminvariant}
Let $\phi: \M \rightarrow \cN$ be a graded smooth map. Then $\phi$ is a graded diffeomorphism, only if $\gdim(\M) = \gdim(\cN)$. 
\end{tvrz}
\begin{proof}
Suppose $\phi$ is a graded diffeomorphism. It follows from (\ref{tvrz_differentialfunctorial}) that $T_{m}\phi: T_{m}\M \rightarrow T_{\ul{\phi}(m)}\cN$ is an isomorphism  in $\gVect$ for each $m \in M$. This is only possible if $\gdim(T_{m}\M) = \gdim(T_{\ul{\phi}(m)} \cN)$ for all $m \in M$. The rest follows from the previous proposition. 
\end{proof}
In ordinary differential geometry, tangent vectors can be defined as linear maps from the algebra of (local) functions to real numbers. This is the case also for the graded manifolds. 
\begin{tvrz} 
Let $\M = (M,\C^{\infty}_{\M})$ be a graded manifold. Let $m \in M$ and $U \in \Op_{m}(M)$ be arbitrary. Then there is a canonical graded vector space isomorphism $T_{m}\M \cong \gDer(\C^{\infty}_{\M}(U),\R)$. We make $\R$ into a graded $\C^{\infty}_{\M}(U)$-module by setting $f \tr \lambda := f(m) \cdot \lambda$ for all $\lambda \in \R$ and $f \in \C^{\infty}_{\M}(U)$. 

Let $\phi: \M \rightarrow \cN$ be a graded smooth map. Then for any $m \in M$, $v \in T_{m}\M$, and $f \in \C^{\infty}_{\cN}(V)$, one finds that (with respect to this identification)
\begin{equation} \label{eq_differentialformula}
[(T_{m}\phi)(v)](f) = v( \phi^{\ast}_{V}(f)),
\end{equation}
where $V \in \Op_{\ul{\phi}(m)}(N)$ is arbitrary.
\end{tvrz}
\begin{proof}
Let $m \in M$ and $U \in \Op_{m}(M)$ be fixed but arbitrary. Recall that we have a graded algebra morphism $\pi_{U,m}: \C^{\infty}_{\M}(U) \rightarrow \C^{\infty}_{\M,m}$. It certainly intertwines the two involved module structures on $\R$. Hence for any $v \in T_{m}\M = \gDer(\C^{\infty}_{\M,m}, \R)$, the composition $\hat{v} := v \circ \pi_{U,m}$ is in $\gDer( \C^{\infty}_{\M}(U), \R)$ by Proposition \ref{tvrz_derivations}-$(ii)$. It is clear that $v \rightarrow \hat{v}$ defines a graded linear map. 

Let us argue that this is actually an isomorphism in $\gVect$. Let $\hat{v} \in \gDer( \C^{\infty}_{\M}(U), \R)$ be an arbitrary graded derivation. Let $[f]_{m} \in \C^{\infty}_{\M,m}$. By Proposition \ref{cor_piUmsurjective}, we may assume that $f \in \C^{\infty}_{\M}(U)$. We must only argue that the formula
\begin{equation}
v([f]_{m}) := \hat{v}(f)
\end{equation}
defines a tangent vector $v \in T_{m}\M$. This would prove the surjectivity of the map in question. Its injectivity follows immediately from the surjectivity of $\pi_{U,m}$. First, one has to prove that $v$ is well-defined. We have to show that if there is $f \in \C^{\infty}_{\M}(U)$ and $V \in \Op_{m}(U)$ such that $f|_{V} = 0$, then $\hat{v}(f) = 0$. Pick $W \in \Op_{m}(M)$, such that $\ol{W} \subseteq V$ and let $\lambda \in \C^{\infty}_{\M}(U)$ be a graded bump function supported on $V$, such that $\lambda|_{W} = 1$. Then $f = (1-\lambda) \cdot f$ and the Leibniz rule (\ref{eq_derlleibniz}) implies
\begin{equation}
\hat{v}(f) = \hat{v}( (1-\lambda) \cdot f) = \hat{v}(1 - \lambda) \cdot f(m) + (1 - \lambda)(m) \cdot \hat{v}(f) = 0,
\end{equation}
since $f(m) = 0$ and $\lambda(m) = 1$. It is now easy to prove that $v$ is in $\Der_{|\hat{v}|}(\C^{\infty}_{\M}(U),\R)$. The formula (\ref{eq_differentialformula}) follows immediately from the definition of $T_{m}\phi$. This finishes the proof. 
\end{proof}
\begin{example} \label{ex_diffgradeddomains}
Let $\phi: \M \rightarrow \cN$ be a graded smooth map and fix $m \in M$. It is useful to write down the explicit expression for $T_{m}\phi$ in terms of total bases of the respective tangent spaces. 

Let us pick a graded local chart $(U,\varphi)$ for $\M$ around $m$ and a graded local chart $(V,\psi)$ for $\cN$ around $\ul{\phi}(m)$, such that $\ul{\phi}(U) \subseteq V$. One has the corresponding local representative 
\begin{equation}
\hat{\phi}: \hat{U}^{(n_{j})} \rightarrow \hat{V}^{(m_{j})},
\end{equation} 
defined as in (\ref{eq_hatphialphamu}). Let $(x^{1},\dots,x^{n_{0}}, \xi_{1}, \dots, \xi_{n_{\ast}})$ and $(y^{1},\dots,y^{m_{0}}, \theta_{1}, \dots, \theta_{m_{\ast}})$ denote the coordinate functions with respect to the charts $(U,\varphi)$ and $(V,\psi)$, respectively. First, it is straightforward to obtain the following formula for every $i \in \{1, \dots, n_{0} \}$:
\begin{equation}
(T_{m}\phi)( \frac{\partial}{\partial x^{i}}|_{m}) = \frac{\partial (y^{j} \circ \ul{\hat{\phi}})}{\partial x^{i}}(\ul{\varphi}(m)) \frac{\partial}{\partial y^{j}}|_{\ul{\phi}(m)},
\end{equation}
where $\ul{\hat{\phi}}: \hat{U} \rightarrow \hat{V}$ is the local representative of the underlying smooth map $\ul{\phi}: M \rightarrow N$. Observe that this is in fact just the matrix of the linear map $T_{m} \ul{\phi}: T_{m}M \rightarrow T_{\ul{\phi}(m)}N$.

Next, fix $\mu \in \{1,\dots,n_{\ast} \}$. For each $\nu \in \{1,\dots,m_{\ast} \}$ such that $|\theta_{\nu}| = |\xi_{\mu}|$, the formal power series expansion of $\bar{\theta}_{\nu} := (\hat{\phi})^{\ast}_{\hat{V}}(\theta_{\nu}) \in \C^{\infty}_{(n_{j})}(\hat{U})$ then contains a linear term $(\bar{\theta}_{\nu})_{\mu} \xi_{\mu}$ for the unique smooth function $(\bar{\theta}_{\nu})_{\mu} \in \C^{\infty}_{n_{0}}(\hat{U})$. It now follows easily that 
\begin{equation}
(T_{m}\phi)( \frac{\partial}{\partial \xi_{\mu}}|_{m}) = \hspace{-2mm} \sum_{\nu, |\theta_{\nu}| = |\xi_{\mu}|} \hspace{-2mm} (\bar{\theta}_{\nu})_{\mu}( \ul{\varphi}(m)) \frac{\partial}{\partial \theta_{\nu}}|_{\ul{\phi}(m)}.
\end{equation} 
\end{example}
\subsection{Sheaf of vector fields}
For an ordinary manifold, vector fields can be viewed as a space of derivations of the algebra of smooth functions. A graded manifold $\M = (M,\C^{\infty}_{\M})$ is \textit{defined} by its sheaf of graded algebras of functions, hence it is easy to generalize this approach. More precisely, we will define vector fields as sections of a certain sheaf of graded $\C^{\infty}(\M)$-modules.

\begin{rem} \label{rem_gradedCinfmodulesaregood}
Let $m \in M$ and let $U \in \Op_{m}(M)$. As we have a decomposition (\ref{eq_CinftyMUdecomp}), one has $\C^{\infty}_{\M}(U) \neq 0$ and $\C^{\infty}_{\M}(U) / \J^{m}_{\M}(U) \cong \R$. This ensures that locally freely and finitely generated sheaves of $\C^{\infty}_{\M}$-modules have a well-defined graded rank at every $m \in M$. See Proposition \ref{tvrz_thereisrank} and Proposition \ref{tvrz_locfinfreegensheaves}. Moreover, by Proposition \ref{tvrz_Jacobsradicalgenset}, one has $\C^{\infty}_{\M,x} = \R \oplus \frJ(\C^{\infty}_{\M,m})$ . This implies that also all finitely generated graded $\C^{\infty}_{\M,x}$-modules have a well-defined graded rank. 
\end{rem}

\begin{tvrz} \label{tvrz_vf}
Let $\M = (M, \C^{\infty}_{\M})$ be a graded manifold. For each $U \in \Op(M)$, define 
\begin{equation}
\X_{\M}(U) := \gDer( \C^{\infty}_{\M}(U)).
\end{equation}
Then there is a canonical way to make $\X_{\M}$ into a sheaf of graded $\C^{\infty}_{\M}$-modules. Its sections are called \textbf{vector fields on a graded manifold $\M$}.
\end{tvrz}
\begin{proof}
For each $U \in \Op(M)$, $\gDer(\C^{\infty}_{\M}(U))$ has a structure of $\C^{\infty}_{\M}(U)$-module described in Proposition \ref{tvrz_gDerisAmodule}. For each $X \in \X_{\M}(U)$ and $f \in \C^{\infty}_{\M}(U)$, we will write just $fX$ for $f \tr X$. 

We have to make $\X_{\M}$ into a sheaf. Let $X \in \X_{\M}(U)$ be arbitrary. 

First, suppose that there is $V \in \Op(U)$ and $f \in \C^{\infty}_{\M}(U)$, such that $f|_{V} = 0$. Observe that $X(f)|_{V} = 0$. Indeed, pick any $m \in V$ and $W \in \Op_{m}(V)$ such that $\ol{W} \subseteq V$. Let $\lambda \in \C^{\infty}_{\M}(U)$ be a graded bump function supported on $V$, such that $\lambda|_{W} = 1$. Then one has $f = (1 - \lambda) \cdot f$ and 
\begin{equation}
X(f) = X((1-\lambda) \cdot f) = X(1-\lambda) \cdot f + (1-\lambda) \cdot X(f).
\end{equation}
By restricting the right-hand side to $W$, we find $X(f)|_{W} = 0$. As we may cover $V$ by such open subsets $W$, this proves that $X(f)|_{V} = 0$. 

Next, let $V \in \Op(U)$. We want to define a restricted vector field $X|_{V} \in \X_{\M}(V)$. Let $f \in \C^{\infty}_{\M}(V)$. For each $m \in M$, find $W_{(m)} \in \Op_{m}(V)$ such that $\ol{W}_{(m)} \subseteq V$. By Lemma \ref{tvrz_extensionlemma}, there exists a function $f_{(m)} \in \C^{\infty}_{\M}(U)$, such that $f_{(m)}|_{W_{(m)}} = f|_{W_{(m)}}$. We can now define a function $X|_{V}(f) \in \C^{\infty}_{\M}(V)$ by declaring 
\begin{equation}
X|_{V}(f)|_{W_{(m)}} := X(f_{(m)})|_{W_{(m)}},
\end{equation}
for all $m \in M$. As $\{ W_{(m)} \}_{m \in V}$ forms an open cover of $V$ and $\C^{\infty}_{\M}$ is a sheaf, it suffices to prove that the local sections on the right-hand side agree on the overlaps. But $f_{(m)}|_{W_{(m)} \cap W_{(m')}} = f|_{W_{(m)} \cap W_{(m')}} = f_{(m')}|_{W_{(m)} \cap W_{(m')}}$ and thus $X(f_{(m)})|_{W_{(m)} \cap W_{(m')}} = X(f_{(m')})|_{W_{(m)} \cap W_{(m')}}$ by the previous paragraph. Using a similar argument, one can prove that the definition of $X|_{V}(f)$ does not depend on any choices. 

It is then straightforward to prove that $X|_{V} \in \gDer(\C^{\infty}_{\M}(V))$, the restriction $X \mapsto X|_{V}$ is graded linear and $(fX)|_{V} = f|_{V} X|_{V}$. This proves that $\X_{\M}$ is a presheaf of graded $\C^{\infty}_{\M}$-modules. Note that for any $f \in \C^{\infty}_{\M}(U)$, one has the formula
\begin{equation} \label{eq_restrictedvfonrestrictedf}
X|_{V}(f|_{V}) = X(f)|_{V}. 
\end{equation}
Let us prove that $\X_{\M}$ is a sheaf. Let $\{ U_{\alpha} \}_{\alpha \in I}$ be an open cover of $U$. Suppose $X,Y \in \X_{\M}(U)$ satisfy $X|_{U_{\alpha}} = Y|_{U_{\alpha}}$ for all $\alpha \in I$. For any $f \in \C^{\infty}_{\M}(U)$, the equation (\ref{eq_restrictedvfonrestrictedf}) implies
\begin{equation}
X(f)|_{U_{\alpha}} = Y(f)|_{U_{\alpha}},
\end{equation}
for $\alpha \in I$. As $\C^{\infty}_{\M}$ is a sheaf, it follows that $X(f) = Y(f)$. Since $f$ was arbitrary, this proves that $X = Y$ and $\X_{\M}$ has the monopresheaf property. Let $\{ X_{\alpha} \}_{\alpha \in I}$ be a collection where $X_{\alpha} \in \X_{\M}(U_{\alpha})$ satisfy $X_{\alpha}|_{U_{\alpha \beta}} = X_{\beta}|_{U_{\alpha \beta}}$ for all $(\alpha,\beta) \in I^{2}$. Let $f \in \C^{\infty}_{\M}(U)$. For each $\alpha \in I$, define
\begin{equation}
X(f)|_{U_{\alpha}} := X_{\alpha}(f|_{U_{\alpha}}). 
\end{equation}
It follows from (\ref{eq_restrictedvfonrestrictedf}) that the local sections on the right-hand side agree on the overlaps. As $\C^{\infty}_{\M}$ is a sheaf, this defines a function $X(f) \in \C^{\infty}_{\M}(U)$. It is easy to see that $X \in \gDer(\C^{\infty}_{\M}(U))$ and $X|_{U_{\alpha}} = X_{\alpha}$ for all $\alpha \in I$. Hence $\X_{\M}$ has the gluing property and we conclude that $\X_{\M}$ is a sheaf of graded $\C^{\infty}_{\M}$-modules. 
\end{proof}
In the following paragraphs, we shall prove that $\X_{\M}$ is in fact locally freely and finitely generated. To do so, we will find a suitable local frame for $\X_{\M}$ over $U$, where $(U,\varphi)$ is some graded local chart for $\M$. 

\begin{example}
Consider a graded domain $\hat{U}^{(n_{j})} = (\hat{U}, \C^{\infty}_{(n_{j})})$. We will write $\X_{(n_{j})}$ for its sheaf of vector fields. Let $(\xi_{\mu})_{\mu=1}^{n_{\ast}}$ be a total basis of $\R^{(n_{j})}_{\ast}$ and let $(x^{1},\dots,x^{n_{0}})$ denote the standard coordinates on $\hat{U} \subseteq \R^{n_{0}}$. In this example, we will construct two collections of vector fields $\{ \frac{\partial}{\partial x^{i}} \}_{i=1}^{n_{0}}$ and $\{ \frac{\partial}{\partial \xi_{\mu}} \}_{\mu =1}^{n_{\ast}}$ in $\X_{(n_{j})}(\hat{U})$. It will later turn out that they form a global frame for $\X_{(n_{j})}$. 

Let $f \in \C^{\infty}_{(n_{j})}(\hat{U})$. Write $f$ as a formal power series $f = \sum_{\fp \in \N^{n_{\ast}}_{|f|}} f_{\fp} \xi^{\fp}$ for unique functions $f_{\fp} \in \C^{\infty}_{n_{0}}(\hat{U})$. For each $i \in \{1,\dots,n_{0}\}$, define 
\begin{equation} \label{eq_partialxiformula}
\frac{\partial f}{\partial x^{i}} := \sum_{\fp \in \N^{n_{\ast}}_{|f|}} \frac{\partial f_{\fp}}{\partial x^{i}} \xi^{\fp}.
\end{equation}
We claim that $\frac{\partial}{\partial x^{i}}(f) := \frac{\partial f}{\partial x^{i}}$ defines a vector field $\frac{\partial}{\partial x^{i}} \in \X_{(n_{j})}(\hat{U})$. It is a graded linear map of degree $0$. Let $g \in \C^{\infty}_{(n_{j})}(\hat{U})$. Recall that for each $\fp \in \N^{n_{\ast}}_{|f|+|g|}$, one has 
\begin{equation}
(f \cdot g)_{\fp} = \sum_{\substack{\fq \in \N^{n_{\ast}}_{|f|} \\ \fq \leq \fp}} \epsilon^{\fq, \fp-\fq} f_{\fq} \cdot  g_{\fp - \fq},
\end{equation}
see the equation (\ref{eq_formalseriesproduct}). Consequently, using the Leibniz rule for partial derivatives, one finds
\begin{equation}
\begin{split}
(\frac{\partial(f \cdot g)}{\partial x^{i}})_{\fp} = & \ \frac{\partial (f \cdot g)_{\fp}}{\partial x^{i}} = \sum_{\substack{\fq \in \N^{n_{\ast}}_{|f|} \\ \fq \leq \fp}} \epsilon^{\fq,\fp-\fq} \frac{\partial (f_{\fq} \cdot g_{\fp-\fq})}{\partial x^{i}} \\
= & \ \sum_{\substack{\fq \in \N^{n_{\ast}}_{|f|} \\ \fq \leq \fp}} \epsilon^{\fq,\fp-\fq} \{ \frac{\partial f_{\fq}}{\partial x^{i}} \cdot g_{\fp - \fq} + f_{\fq} \cdot \frac{\partial g_{\fp - \fq}}{\partial x^{i}} \} = ( \frac{\partial f}{\partial x^{i}} \cdot g + f \cdot \frac{\partial g}{\partial x^{i}} )_{\fp},
\end{split}
\end{equation}
for every $\fp \in \N^{n_{\ast}}_{|f|+|g|}$. This proves the Leibniz rule (\ref{eq_derlleibniz}) and the claim follows. 

Let $\mu \in \{1,\dots,n_{\ast}\}$ and $f \in \C^{\infty}_{(n_{j})}(\hat{U})$ be parametrized as above. For each $\fp \in \N^{n_{\ast}}_{|f| - |\xi^{\mu}|}$, define 
\begin{equation} \label{eq_partialximuformula}
(\frac{\partial f}{\partial \xi^{\mu}} )_{\fp} := (p_{\mu} + 1) (-1)^{|\xi_{\mu}|( p_{1} |\xi_{1}| + \dots + p_{\mu-1} |\xi_{\mu -1}|)} f_{p_{1} \dots (p_{\mu}+1) \dots p_{n_{\ast}}}. 
\end{equation}
We will now argue that the formula $\frac{\partial}{\partial \xi^{\mu}}(f) := \frac{\partial f}{\partial \xi_{\mu}} = \sum_{\fp \in \N^{n_{\ast}}_{|f| - |\xi^{\mu}|}} (\frac{\partial f}{\partial \xi^{\mu}})_{\fp} \xi^{\fp}$ defines the action of the vector field $\frac{\partial}{\partial \xi^{\mu}} \in \X_{(n_{j})}(\hat{U})$ of degree $-|\xi^{\mu}|$. 

To understand the formula, note that for each $\fq \in \N^{n_{\ast}}_{|f|}$, it acts on the monomial $\xi^{\fq} = (\xi_{1})^{q_{1}} \dots (\xi_{n_{\ast}})^{q_{n_{\ast}}}$ according to the expected Leibniz rule, that is 
\begin{equation}
\frac{\partial \xi^{\fq}}{\partial \xi^{\mu}} = q_{\mu} (-1)^{|\xi_{\mu}|(q_{1}|\xi_{1}| + \dots + q_{\mu-1} |\xi_{\mu-1}|)} (\xi_{1})^{q_{1}} \cdots (\xi_{\mu})^{q_{\mu}-1} \cdots (\xi_{n_{\ast}})^{q_{n_{\ast}}},
\end{equation}
where the sign is obtained by commuting the degree $-|\xi_{\mu}|$ vector field and $(\xi_{1})^{q_{1}} \dots (\xi_{\mu-1})^{q_{\mu-1}}$ of degree $q_{1}|\xi_{1}| + \dots + q_{\mu-1} |\xi_{\mu-1}|$. Its action on a general function $f = \sum_{\fq \in \N^{n_{\ast}}_{|f|}} f_{\fq} \xi^{\fq}$ can be then written as a term-by-term differentiation of this power series, that is
\begin{equation} 
\frac{\partial f}{\partial \xi_{\mu}} = \sum_{\fq \in \N^{n_{\ast}}_{|f|}} \frac{\partial( f_{\fq} \xi^{\fq})}{\partial \xi_{\mu}} = \sum_{\fq \in \N^{n_{\ast}}_{|f|}} f_{\fq} \frac{\partial \xi^{\fq}}{\partial \xi_{\mu}},
\end{equation}
where the formula (\ref{eq_partialximuformula}) is then obtained as the term proportional to $\xi^{\fp}$ for a given $\fp \in \N^{n_{\ast}}_{|f| - |\xi^{\mu}|}$. Clearly $\frac{\partial}{\partial \xi^{\mu}}$ is a graded linear map of degree $-|\xi_{\mu}|$. Let $g \in \C^{\infty}_{(n_{j})}(\hat{U})$ be another function. Let $\fq \in \N^{n_{\ast}}_{|f|}$ and $\fp \in \N^{n_{\ast}}_{|f|+|g|}$ satisfy $\fq \leq \fp$. We claim that there holds a formula
\begin{equation} \label{eq_partialofxixibyxi}
\frac{\partial( \xi^{\fq} \cdot \xi^{\fp - \fq})}{\partial \xi^{\mu}} = \frac{\partial \xi^{\fq}}{\partial \xi^{\mu}} \cdot \xi^{\fp - \fq} + (-1)^{|\xi_{\mu}||f|} \xi^{\fq} \cdot 
\frac{\partial \xi^{\fp - \fq}}{\partial \xi^{\mu}}. 
\end{equation}
The Leibniz rule (\ref{eq_derlleibniz}) can be then easily proved. Indeed, one can write
\begin{equation}
\begin{split}
\frac{\partial(f \cdot g)}{\partial \xi_{\mu}} = & \hspace{-2mm}\ \sum_{\fp \in \N^{n_{\ast}}_{|f|+|g|}} \sum_{\substack{\fq \in \N^{n_{\ast}}_{|f|} \\ \fq \leq \fp}} f_{\fq} g_{\fp - \fq} \frac{\partial( \xi^{\fq} \cdot \xi^{\fp - \fq})}{\partial \xi^{\mu}} \\
= & \hspace{-2mm} \ \sum_{\fp \in \N^{n_{\ast}}_{|f|+|g|}} \sum_{\substack{\fq \in \N^{n_{\ast}}_{|f|} \\ \fq \leq \fp}} f_{\fq} g_{\fp - \fq} \{  \frac{\partial \xi^{\fq}}{\partial \xi^{\mu}} \cdot \xi^{\fp - \fq} + (-1)^{|\xi_{\mu}||f|} \xi^{\fq} \cdot 
\frac{\partial \xi^{\fp - \fq}}{\partial \xi^{\mu}} \} \\
= & \ \frac{\partial f}{\partial \xi_{\mu}} \cdot g + (-1)^{|\xi_{\mu}||f|} f \cdot \frac{\partial g}{\partial \xi_{\mu}}.
\end{split}
\end{equation}
Both sides of (\ref{eq_partialofxixibyxi}) are proportional to a single monomial, namely to $(\xi_{1})^{p_{1}} \cdots (\xi_{\mu})^{p_{\mu} - 1} \cdots (\xi_{n_{\ast}})^{p_{n_{\ast}}}$. Everything is thus just a matter of scalar factors. For details, see Proposition \ref{tvrz_ap_partialofxixibyxi} in the appendix. 

To conclude this example, we have constructed a collection of degree zero vector fields $\{ \frac{\partial}{\partial x^{i}} \}_{i=1}^{n_{0}}$ in $\X_{(n_{j})}(\hat{U})$ and negative degree vector fields $\{ \frac{\partial}{\partial \xi_{\mu}} \}_{\mu=1}^{n_{\ast}}$ both acting on functions in $\C^{\infty}_{(n_{j})}(\hat{U})$ in a way resembling the partial differentiation. In particular, on coordinate functions, they give
\begin{equation} \label{eq_basisvfsoncoordinatefctions}
\frac{\partial x^{j}}{\partial x^{i}} = \delta^{j}_{i}, \; \; \frac{\partial \xi_{\mu}}{\partial x^{i}} = 0, \; \; \frac{\partial \xi_{\nu}}{\partial \xi_{\mu}} = \delta^{\mu}_{\nu}, \; \; \frac{\partial x^{i}}{\partial \xi_{\mu}} = 0,
\end{equation}
for all $i,j \in \{1,\dots,n_{0} \}$ and $\mu,\nu \in \{1,\dots,n_{\ast}\}$. We write their action as a partial differentiation. 
\end{example}

\begin{rem} \label{rem_coordinatevectorfields}
Now, let $\M = (M,\C^{\infty}_{\M})$ be a graded manifold and let $(U,\varphi)$ be a graded local chart for $\M$. Suppose $\varphi: \M|_{U} \rightarrow \hat{U}^{(n_{j})}$. We thus have the collections of vector fields $\{ \frac{\partial}{\partial x^{i}} \}_{i = 1}^{n_{0}}$ and $\{ \frac{\partial}{\partial \xi_{\mu}} \}_{\mu=1}^{n_{\ast}}$ in $\X_{(n_{j})}(\hat{U})$ constructed in the previous example. 

One slightly abuses the notation and for each $i \in \{1,\dots,n_{0}\}$ and $\mu \in \{1,\dots,n_{\ast} \}$ defines 
\begin{equation}
\frac{\partial}{\partial x^{i}} := \varphi_{\hat{U}}^{\ast} \circ \frac{\partial}{\partial x^{i}} \circ (\varphi^{\ast}_{\hat{U}})^{-1}, \; \; \frac{\partial}{\partial \xi_{\mu}} := \varphi_{\hat{U}}^{\ast} \circ \frac{\partial}{\partial \xi_{\mu}} \circ (\varphi^{\ast}_{\hat{U}})^{-1}
\end{equation}
As $\varphi^{\ast}_{\hat{U}}: \C^{\infty}_{(n_{j})}(\hat{U}) \rightarrow \C^{\infty}_{\M}(U)$ is a graded algebra isomorphism, it follows that this defines vector fields in $\X_{\M}(U)$. Their values on functions are again usually written as partial derivatives. Moreover, it is clear that on coordinate functions (see Example \ref{ex_coordfunctions}), they are given by (\ref{eq_basisvfsoncoordinatefctions}). 

$\{ \frac{\partial}{\partial x^{i}} \}_{i=1}^{n_{0}}$ and $\{ \frac{\partial}{\partial \xi_{\mu}} \}_{\mu=1}^{n_{\ast}}$ are called \textbf{coordinate vector fields} corresponding to $(U,\varphi)$. 
\end{rem}
We have now everything prepared to prove the main proposition of this subsection. 
\begin{tvrz} \label{tvrz_vfields}
Let $\M = (M,\C^{\infty}_{\M})$ be a graded manifold and let $(U,\varphi)$ be a graded local chart for $\M$. Let $(x^{1},\dots,x^{n_{0}}, \xi_{1},\dots,\xi_{n_{\ast}})$ be the corresponding coordinate functions.

Then every vector field $X \in \X_{\M}(U)$ is uniquely determined by values $X(y^{i})$ and $X(\xi_{\mu})$, for all $i \in \{1,\dots,n_{0}\}$ and $\mu \in \{1,\dots,n_{\ast} \}$. It can be uniquely decomposed as 
\begin{equation} \label{eq_vfdecomposition}
X = X(x^{i}) \frac{\partial}{\partial x^{i}} + X(\xi_{\mu}) \frac{\partial}{\partial \xi_{\mu}},
\end{equation}
using the coordinate vector fields defined in Remark \ref{rem_coordinatevectorfields}. 

In particular, $\{ \frac{\partial}{\partial x^{i}} \}_{i = 1}^{n_{0}}$ and $\{ \frac{\partial}{\partial \xi_{\mu}} \}_{\mu=1}^{n_{\ast}}$ form a local frame for $\X_{\M}$ over $U$. Hence $\X_{\M}$ is locally freely and finitely generated and $\grk_{m}( \X_{\M}) = (n_{-j})_{j \in \Z}$ for all $m \in M$, where $(n_{j})_{j \in \Z} := \gdim(\M)$. 
\end{tvrz}
\begin{proof}
As $\M|_{U} \cong \hat{U}^{(n_{j})}$ where $\hat{U} = \ul{\varphi}(U)$, it suffices to prove the claim for the graded domain $\hat{U}^{(n_{j})}$. Let $X \in \X_{(n_{j})}(\hat{U})$ be a given vector field. Suppose $\hat{X} \in \X_{(n_{j})}(\hat{U})$ is defined using the right-hand side of (\ref{eq_vfdecomposition}). We have to show that $\hat{X} = X$. 

Using (\ref{eq_basisvfsoncoordinatefctions}), it is easy to see that the actions of both vector fields coincide on polynomials in variables $x^{i} - x^{i}(a)$ and $\xi_{\mu}$ for each $a \in \hat{U}$. Moreover, as it is a graded derivation, one has 
\begin{equation}
X( (\J^{a}_{(n_{j})}(\hat{U}))^{q+1}) \subseteq (\J^{a}_{(n_{j})}(\hat{U}))^{q},
\end{equation}
for all $q \in \N$, and the same holds for $\hat{X}$. 

Now, let $f \in \C^{\infty}_{(n_{j})}(\hat{U})$ be arbitrary.  For every $a \in \hat{U}$ and every $q \in \N$, decompose $f$ as in Lemma \ref{lem_Hadamard}, that is  $f = T^{q}_{a}(f) + R^{q}_{a}(f)$. Using the previous paragraph, one finds
\begin{equation}
(\hat{X} - X)(f) = (\hat{X} - X)( R^{q}_{a}(f)) \in (\J^{a}_{(n_{j})}(\hat{U}))^{q}
\end{equation}
Since $a \in \hat{U}$ and $q \in \N$ were arbitrary, Proposition \ref{tvrz_inallvanishingideals} implies that $(\hat{X} - X)(f) = 0$. As $f \in \C^{\infty}_{(n_{j})}(\hat{U})$ was arbitrary, we find that $\hat{X} = X$. Remaining statements follow in the same way as in Proposition \ref{tvrz_coordtangentvectors} and the proof is finished. 
\end{proof}
\begin{example} \label{ex_Euler}
Let $\M = (M,\C^{\infty}_{\M})$. Then there is a canonical vector field $E \in \X_{\M}(M)$ defined for each $f \in \C^{\infty}_{\M}(M)$ as $E(f) = |f| f$. One has $|E| = 0$ and the Leibniz rule is easily verified. 

$E$ is called the \textbf{Euler vector field} on $\M$. For any graded local chart $(U,\varphi)$, one can write $E|_{U} = |\xi_{\mu}| \xi_{\mu} \frac{\partial}{\partial \xi_{\mu}}$. This explains the name. 
\end{example}

Similarly to the ordinary differential geometry, one expects that every vector field can be evaluated at each point of its domain, obtaining thus a tangent vector in the respective tangent space. This is indeed true also in the graded setting.
\begin{tvrz} \label{tvrz_valueofvf}
For each $m \in M$ and $U \in \Op_{m}(M)$, there is a canonical surjective graded linear map $\ev_{U,m}: \X_{\M}(U) \rightarrow T_{m}\M$ commuting with restrictions, that is $\ev_{V,m}(X|_{V}) = \ev_{U,m}(X)$ for all $X \in \X_{\M}(U)$ and $V \in \Op_{m}(U)$. For $X \in \X_{\M}(U)$, one usually writes $X_{m} := \ev_{U,x}(X)$ and calls it the \textbf{value of $X$ at $m$}.

For any graded local chart $(U,\varphi)$ and $X \in \X_{\M}(U)$, one has 
\begin{equation} \label{eq_evUondecomposition}
\ev_{U,x}( X^{i} \frac{\partial}{\partial x^{i}} + X_{\mu} \frac{\partial}{\partial \xi_{\mu}}) = X^{i}(m) \frac{\partial}{\partial x^{i}}|_{m} + X_{\mu}(m) \frac{\partial}{\partial \xi_{\mu}}|_{m},
\end{equation}
where $\frac{\partial}{\partial x^{i}}|_{m}$ and $\frac{\partial}{\partial \xi_{\mu}}|_{m}$ are the tangent vectors constructed in Example \ref{ex_coordtangentvectors}. 
\end{tvrz}
\begin{proof}
Fix $m \in M$. Let $X \in \X_{\M}(U)$. Each $[f]_{m} \in \C^{\infty}_{\M,m}$ can be represented by $f \in \C^{\infty}_{\M}(U)$ by Corollary \ref{cor_piUmsurjective}. Define $X_{m}([f]_{m}) := (X(f))(m)$. We have argued in the proof of Proposition \ref{tvrz_vf} that $f|_{V} = 0$ for some $V \in \Op_{m}(U)$ implies $X(f)|_{V} = 0$. This proves that $X_{m}([f]_{m})$ does not depend on the choice of $f$ representing $[f]_{m}$ and it is thus well-defined. The fact that $(X|_{V})_{m} = X_{m}$ for any $V \in \Op_{m}(M)$ follows easily from the property (\ref{eq_restrictedvfonrestrictedf}). It is straightforward to verify that $X_{m} \in T_{m}\M \equiv \gDer( \C^{\infty}_{\M,m}, \R)$ and $\ev_{U,x}(X) := X_{m}$ obviously defines a graded linear map.

Let $(U,\varphi)$ be a graded chart. One can directly verify that 
\begin{equation}
\ev_{U,m}( \frac{\partial}{\partial x^{i}}) = \frac{\partial}{\partial x^{i}}|_{m}, \; \; \ev_{U,m}( \frac{\partial}{\partial \xi_{\mu}}) = \frac{\partial}{\partial \xi_{\mu}}|_{m},
\end{equation}
for all $i \in \{1, \dots, n_{0} \}$ and $\mu \in \{1, \dots, n_{\ast} \}$. Moreover, clearly $(fX)_{m} = f(m) X_{m}$ for all $X \in \X_{\M}(U)$ and $f \in \C^{\infty}_{\M}(U)$. The equation (\ref{eq_evUondecomposition}) then follows immediately. 

It remains to prove that $\ev_{U,m}: \X_{\M}(U) \rightarrow T_{m}\M$ is surjective. It suffices to consider $U = M$. Let $v \in T_{m}\M$ be a given tangent vector. Pick any graded local chart $(U,\varphi)$ for $\M$ around $m$. One can then use (\ref{eq_evUondecomposition}) to find $X \in \X_{\M}(U)$ with $X_{m} = v$. Consider $V \in \Op_{m}(M)$ satisfying $\ol{V} \subseteq U$ and let $\lambda \in \C^{\infty}_{\M}(M)$ be a graded bump function supported in $U$ satisfying $\lambda|_{V} = 1$. Let $\lambda X \in \X_{\M}(M)$ be a vector field defined by its restrictions to the open cover $\{ U, \supp(\lambda)^{c} \}$:
\begin{equation}
(\lambda X)|_{U} := \lambda|_{U} X, \; \; (\lambda X)|_{\supp(\lambda)^{c}} := 0.
\end{equation}
Both agree on the overlap $U \cap \supp(\lambda)^{c}$ and we use the fact that $\X_{\M}$ is a sheaf. But then
\begin{equation}
\ev_{M,m}(\lambda X) = \ev_{V,m}( (\lambda X)|_{V}) = \ev_{V,m}(X|_{V}) = \ev_{U,m}(X) = v.
\end{equation} 
This shows the surjectivity of $\ev_{M,m}$ and the proof is finished. 
\end{proof}
\begin{rem} \label{rem_notbyvalues}
Unlike for ordinary manifolds, vector fields are \textit{not determined} by their values at all points. Consider for example $n_{\ast} > 0$ and the Euler vector field $E \in \X_{\M}(M)$. Although it is non-trivial, one has $E_{m} = 0$ for all $m \in M$. 
\end{rem}

\begin{rem} \label{rem_unifiedcoordfunct}
For the following paragraphs (and the rest of the paper), it is useful to introduce a unified notation for coordinate functions. Let $(U,\varphi)$ be a graded local chart and let $(x^{1},\dots,x^{n_{0}}, \xi_{1},\dots,\xi_{n_{\ast}})$ be the corresponding coordinate functions. We will write $(\bbz^{A})_{A=1}^{n}$ for functions defined as $\bbz^{i} = x^{i}$ for $i \in \{1,\dots,n_{0}\}$ and $\bbz^{\mu + n_{0}} := \xi_{\mu}$ for $\mu \in \{1,\dots,n_{\ast} \}$. 
\end{rem}

\begin{tvrz} \label{tvrz_coordinatevftransrule}
Now, let $(U_{\alpha},\varphi_{\alpha})$ and $(U_{\beta},\varphi_{\beta})$ be two graded local charts. Let $(\bbz^{A}_{\alpha})_{A=1}^{n}$ and $(\bbz^{A}_{\beta})_{A=1}^{n}$ denote the respective coordinate functions. We thus have the associated coordinate vector fields $\{ \frac{\partial}{\partial \bbz^{A}_{\alpha}} \}_{A=1}^{n}$ in $\X_{\M}(U_{\alpha})$ and $\{ \frac{\partial}{\partial \bbz^{B}_{\beta}} \}_{B = 1}^{n}$ in $\X_{\M}(U_{\beta})$. Then one has 
\begin{equation}
\frac{\partial}{\partial \bbz^{B}_{\beta}}|_{U_{\alpha \beta}} = \frac{\partial \bbz^{A}_{\alpha}}{ \partial \bbz^{B}_{\beta}} \frac{\partial}{\partial \bbz^{A}_{\alpha}}|_{U_{\alpha \beta}},
\end{equation}
for each $B \in \{1,\dots,n\}$. One uses the simplified notation 
\begin{equation}
\frac{\partial \bbz^{A}_{\alpha}}{ \partial \bbz^{B}_{\beta}} := \frac{\partial}{\partial \bbz^{B}_{\beta}}|_{U_{\alpha \beta}}( \bbz^{A}_{\alpha}|_{U_{\alpha \beta}}).
\end{equation}
Note that one usually omits the explicit writing of restrictions to $U_{\alpha \beta}$.
\end{tvrz}
\begin{proof}
This follows immediately from Proposition \ref{tvrz_vfields}. 
\end{proof}
\begin{tvrz}[\textbf{Chain rule}]
Let $\phi: \M \rightarrow \cN$ be a graded smooth map. Let $(U,\varphi)$ and $(V,\psi)$ be graded local charts for $\M$ and $\cN$, respectively, such that $\ul{\phi}(U) \subseteq V$. Let $(\bbz^{A})_{A=1}^{n}$ and $(\bby^{K})_{K=1}^{m}$ denote the corresponding coordinate functions in $\C^{\infty}_{\M}(U)$ and $\C^{\infty}_{\cN}(V)$, respectively. 

Then for any $f \in \C^{\infty}_{\cN}(V)$ and $A \in \{1,\dots,n\}$, one has 
\begin{equation} \label{eq_chainrule}
\frac{\partial( \phi^{\ast}_{V}(f)|_{U}) }{\partial \bbz^{A}} = \frac{\partial (\phi^{\ast}_{V}(\bby^{K})|_{U})}{\partial \bbz^{A}} \cdot \phi^{\ast}_{V}( \frac{\partial f}{\partial \bby^{K}})|_{U}.
\end{equation}
\end{tvrz} 
\begin{proof}
Since it is a local statement, it suffices to prove it for a graded smooth map $\phi: \hat{U}^{(n_{j})} \rightarrow \hat{V}^{(m_{j})}$ of two graded domains. Let us first prove (\ref{eq_chainrule}) for $f$ in the form 
\begin{equation}
f = (\bby^{1} - \bby^{1}(\ul{\phi}(a)))^{q_{1}} \cdots (\bby^{m} - \bby^{m}(\ul{\phi}(a)))^{q_{m}},
\end{equation}
where $m = \sum_{j \in \Z} m_{j}$, and $a \in \hat{U}$ and $(q_{1},\dots,q_{m}) \in (\N_{0})^{m}$ are arbitrary. Write $\hat{\bby}^{K} := \phi^{\ast}_{\hat{V}}(\bby^{K})$ for all $K \in \{1,\dots,m\}$. We can then use the Leibniz rule and write 
\begin{equation}
\begin{split}
\frac{\partial (\phi^{\ast}_{\hat{V}}(f))}{\partial \bbz^{A}} = & \ \frac{\partial}{\partial \bbz^{A}}( (\hat{\bby}^{1} - \bby^{1}(\ul{\phi}(a)))^{q_{1}} \cdots (\hat{\bby}^{m} - \bby^{m}(\ul{\phi}(a)))^{q_{m}}) \\
= & \ \sum_{K=1}^{m} (-1)^{|\bbz^{A}|(q_{1}|\bby^{1}| + \dots + q_{K-1} |\bby^{K-1}|)} q_{K} (\hat{\bby}^{1} - \bby^{1}(\ul{\phi}(a)))^{q_{1}} \cdots \\
& \cdots \frac{\partial \hat{\bby}^{K}}{\partial \bbz^{A}} \cdot  (\hat{\bby}^{K} - \bby^{K}(\ul{\phi}(a)))^{q_{K}-1} \cdots (\hat{\bby}^{m} - \bby^{m}(\ul{\phi}(a)))^{q_{m}}.
 \\
= & \ \sum_{K=1}^{m} \frac{\partial \hat{\bby}^{k}}{\partial \bbz^{A}} \cdot (-1)^{|\bby^{K}|(q_{1} | \bby^{1}| + \dots + q_{K-1} |\bby^{K-1}|)} q_{K}  (\hat{\bby}^{1} - \bby^{1}(\ul{\phi}(a)))^{q_{1}} \cdots \\
& \cdots (\hat{\bby}^{K} - \bby^{K}(\ul{\phi}(a)))^{q_{K}-1} \cdots (\hat{\bby}^{m} - \bby^{m}(\ul{\phi}(a)))^{q_{m}} \\
= & \ \sum_{K=1}^{m} \frac{\partial \hat{\bby}^{k}}{\partial \bbz^{A}} \cdot \phi^{\ast}_{\hat{V}}( \frac{\partial f}{\partial \bby^{K}} ). 
\end{split}
\end{equation}
This proves the formula (\ref{eq_chainrule}) for all monomials, and by linearity also for all polynomials, in variables $\bby^{K} - \bby^{K}(\ul{\phi}(a))$. 

Finally, observe that for all $a \in \hat{U}$ and $q \in \N$, both sides of (\ref{eq_chainrule}) take $f \in (\J_{(m_{j})}^{\ul{\phi}(a)}(\hat{V}))^{q+1}$ and map it to an element of $(\J^{a}_{(n_{j})}(\hat{U}))^{q}$. We can thus use the same argument as in the proof of Proposition \ref{tvrz_vfields} to show the validity of (\ref{eq_chainrule}) for a general $f \in \C^{\infty}_{(m_{j})}(\hat{V})$. 
\end{proof}

As vector fields are just graded derivations of graded algebras, one can very easily introduce a concept of a graded commutator of two vector fields, see Proposition \ref{tvrz_derivations}-$(iv)$. Moreover, it turns out that it is well-behaved with respect to restrictions. 
\begin{tvrz} \label{tvrz_gcommutator}
Let $\M = (M,\C^{\infty}_{\M})$ be a graded manifold. For every $U \in \Op(M)$ and $X,Y \in \X_{\M}(U)$, their \textbf{graded commutator} $[X,Y] \in \X_{\M}(U)$ is for each $f \in \C^{\infty}_{\M}(U)$ defined as
\begin{equation}
[X,Y](f) := X(Y(f)) - (-1)^{|X||Y|} Y(X(f)). 
\end{equation}
In fact, it makes $\X_{\M}$ into a \textbf{sheaf of graded Lie algebras of degree $0$}. In other words, it has the following properties:
\begin{enumerate}[(i)]
\item It is compatible with restrictions, i.e. for any $V \in \Op(U)$, one has $[X,Y]|_{V} = [X|_{V},Y|_{V}]$. 
\item For every $U \in \Op(M)$, it makes each graded vector space $\X_{\M}(U)$ into a graded Lie algebra of degree $0$, that is for all $X,Y,Z \in \X_{\M}(U)$, it satisfies
\begin{enumerate}[($\ell$1)]
\item $|[X,Y]| = |X| + |Y|$; 
\item $[X,Y] = - (-1)^{|X||Y|} [Y,X]$; 
\item $[X,[Y,Z]] = [[X,Y],Z] + (-1)^{|X||Y|} [Y,[X,Z]]$.
\end{enumerate}
\end{enumerate}
Moreover, for each $U \in \Op(X)$, $X,Y \in \X_{\M}(U)$ and $f \in \C^{\infty}_{\M}(U)$, it satisfies the Leibniz rule
\begin{equation}
[X,fY] = X(f)Y + (-1)^{|X||f|} f [X,Y]. 
\end{equation}
\end{tvrz}
\begin{proof}
All statements are verified directly, except for $(i)$. Recall the definition of a restricted vector field in the proof of Proposition \ref{tvrz_vf}, including the notation. For each $f \in \C^{\infty}_{\M}(V)$ and $m \in V$, one has 
\begin{equation}
\begin{split}
[X,Y]|_{V}(f)|_{W_{(m)}} = & \ [X,Y](f_{(m)})|_{W_{(m)}} = (X(Y(f_{(m)})) - (-1)^{|X||Y|} Y(X(f_{(m)})))|_{W_{(m)}} \\
= & \ (X|_{V}(Y|_{V}(f)) - (-1)^{|X||Y|} Y|_{V}(X|_{V}(f)))|_{W_{(m)}} \\
= & \ [X|_{V},Y|_{V}](f)|_{W_{(m)}}. 
\end{split}
\end{equation}
As $\{ W_{(m)} \}_{m \in V}$ forms an open cover of $V$ and $\C^{\infty}_{\M}$ is a sheaf, this proves that $[X,Y]|_{V}(f) = [X|_{V},Y|_{V}](f)$ for all $f \in \C^{\infty}_{\M}(V)$ and the statement $(i)$ follows. 
\end{proof}
\begin{definice} \label{def_relatedvf}
Let $\phi: \M \rightarrow \cN$ be a graded smooth map. Let $X \in \X_{\M}(M)$ and $Y \in \X_{\cN}(N)$. We say that $X$ and $Y$ are \textbf{$\phi$-related} (and write $X \sim_{\phi} Y$), if for all $f \in \C^{\infty}_{\cN}(N)$, one has
\begin{equation}
X( \phi^{\ast}_{N}(f)) = \phi^{\ast}_{N}(Y(f)).
\end{equation}
\end{definice}
Let us summarize some basic properties of this notion.
\begin{tvrz} \label{tvrz_relatedprops}
Let $\phi: \M \rightarrow \cN$ be a graded smooth map. 
\begin{enumerate}[(i)]
\item Let $X \in \X_{\M}(M)$ and $Y \in \X_{\cN}(N)$ and suppose $X \sim_{\varphi} Y$. Then for any $m \in M$, one has $(T_{m}\phi)(X_{m}) = Y_{\ul{\phi}(m)}$. The converse is not true for general graded manifolds.  
\item Suppose $\phi$ is a graded diffeomorphism. Then for every $X \in \X_{\M}(M)$, there is a unique vector field $\phi_{\ast}(X) \in \X_{\cN}(N)$, such that $X \sim_{\phi} \phi_{\ast}(X)$.
\item Let $X,X' \in \X_{\M}(M)$ and $Y,Y' \in \X_{\cN}(N)$ satisfy $X \sim_{\phi} Y$ and $X' \sim_{\varphi} Y'$. Then their graded commutators are also $\phi$-related, $[X,X'] \sim_{\phi} [Y,Y']$. 
\item Let $X \in \X_{\M}(M)$ and $Y \in \X_{\cN}(N)$. 

Suppose $X \sim_{\phi} Y$ and consider $U \in \Op(M)$ and $V \in \Op(N)$, such that $\ul{\phi}(U) \subseteq V$. Then $X|_{U}$ and $Y|_{V}$ are $\phi|_{U}$-related. 

Conversely, let there be an open cover $\{ U_{\alpha} \}_{\alpha \in I}$ of $M$ and a collection of open subsets $\{ V_{\alpha} \}_{\alpha \in I}$ of $N$, such that $\ul{\phi}(U_{\alpha}) \subseteq V_{\alpha}$ and $X|_{U_{\alpha}} \sim_{\phi|_{U_{\alpha}}} Y|_{V_{\alpha}}$ for all $\alpha \in I$. Then $X \sim_{\phi} Y$. 
\end{enumerate}
\end{tvrz}
\begin{proof}
To prove $(i)$, one can use (\ref{eq_differentialformula}) and definitions. Let $f \in \C^{\infty}_{\cN}(N)$. Then
\begin{equation}
\begin{split}
((T_{m}\phi)(X_{m}))(f) = & \ X_{m}( \phi^{\ast}_{N}(f)) = (X(\phi^{\ast}_{N}(f)))(m) = (\phi^{\ast}_{N}(Y(f)))(m) \\
= & \ (Y(f))(\ul{\phi}(m)) = Y_{\ul{\phi}(m)}(f).
\end{split}
\end{equation}
On the other hand, the Euler vector field $E \in \X_{\M}(M)$ satisfies $E_{m} = 0$ for all $m \in M$, whence $(T_{m}\phi)(E_{m}) = 0$. But this obviously does not imply that $E \sim_{\phi} 0_{\cN}$, where $0_{\cN}$ denotes the zero vector field (of degree $0$) on $\cN$. To prove $(ii)$, define $\phi_{\ast}(f) := (\phi^{\ast}_{N})^{-1}(X(\phi^{\ast}_{N}(f)))$ for all $f \in \C^{\infty}_{\cN}(N)$. The proof of $(iii)$ is a direct verification. It remains to show the property $(iv)$. 

Suppose $X \sim_{\phi} Y$ and consider $U \in \Op(M)$ and $V \in \Op(N)$, such that $\ul{\phi}(U) \subseteq V$. Let $f \in \C^{\infty}_{\cN}(V)$. For each $m \in U$, pick $W \in \Op_{\ul{\phi}(m)}(N)$ such that $\ol{W} \subseteq V$. There is thus $f' \in \C^{\infty}_{\cN}(N)$, such that $f'|_{W} = f|_{W}$ (see Proposition \ref{tvrz_extensionlemma}). Let $W' := U \cap \ul{\phi}^{-1}(W)$. Then one can write 
\begin{equation}
\begin{split}
X|_{U}( (\phi|_{U})^{\ast}_{V}(f))|_{W'} = & \ X|_{W'}( (\phi|_{U})^{\ast}_{V}(f)|_{W'}) = X|_{W'}( \phi^{\ast}_{V}(f)|_{W'}) = X|_{W'}( \phi^{\ast}_{W}(f|_{W})|_{W'}) \\
= & \ X|_{W'}( \phi^{\ast}_{W}(f'|_{W})|_{W'} ) = X|_{W'}( \phi^{\ast}_{N}(f')|_{W'}) = X(\phi^{\ast}_{N}(f'))|_{W'}.
\end{split}
\end{equation}
We have used the naturality of $\phi$ and the property (\ref{eq_restrictedvfonrestrictedf}). Similarly, one can derive the equation
\begin{equation}
(\phi|_{U})^{\ast}_{V}( Y|_{V}(f))|_{W'} = \phi^{\ast}_{N}(Y(f'))|_{W'}. 
\end{equation}
Using the assumption $X \sim_{\phi} Y$ thus gives us 
\begin{equation}
X|_{U}( (\phi|_{U})^{\ast}_{V}(f))|_{W'} = (\phi|_{U})^{\ast}_{V}( Y|_{V}(f))|_{W'}.
\end{equation}
As $m \in U$ was arbitrary, we may cover $U$ by open sets of the form $W'$ and the fact that $X|_{U}$ and $Y|_{V}$ are $\phi|_{U}$-related follows from the monopresheaf property of $\C^{\infty}_{\M}$. 

It remains to prove the last statement. For each $f \in \C^{\infty}_{\cN}(N)$ and every $\alpha \in I$, one has 
\begin{equation}
\begin{split}
X(\phi^{\ast}_{N}(f))|_{U_{\alpha}} = & \  X|_{U_{\alpha}}( \phi^{\ast}_{N}(f)|_{U_{\alpha}}) = X|_{U_{\alpha}}( \phi^{\ast}_{V_{\alpha}}( f|_{V_{\alpha}})|_{V_{\alpha}}) = (\phi|_{U_{\alpha}})^{\ast}_{V_{\alpha}}( Y|_{V_{\alpha}}(f|_{V_{\alpha}})) \\
= & \ (\phi|_{U_{\alpha}})^{\ast}_{V_{\alpha}}( Y(f)|_{V_{\alpha}}) = \phi^{\ast}_{N}( Y(f))|_{U_{\alpha}}. 
\end{split}
\end{equation} 
Since $\{ U_{\alpha} \}_{\alpha \in I}$ is an open cover of $M$ and $f$ was arbitrary, this proves the claim $X \sim_{\varphi} Y$. 
\end{proof}
\begin{example}
Let $\M$ and $\cN$ be two graded manifolds and let $E_{\M}$ and $E_{\cN}$ denote the respective Euler vector fields. Then for any graded smooth map $\phi: \M \rightarrow \cN$, one has $E_{\M} \sim_{\phi} E_{\cN}$. 
\end{example}
\begin{example}
Let $\M = (M, \C^{\infty}_{\M})$ and $\cN = (N, \C^{\infty}_{\cN})$ be a pair of graded manifolds. Let $X \in \X_{\M}(M)$ and $Y \in \X_{\cN}(N)$ be a pair of vector fields. 

Then there are vector fields $X \otimes 1$ and $1 \otimes Y$ on $\M \times \cN$, determined uniquely by equations
\begin{equation}
X \otimes 1 \sim_{\pi_{\M}} X, \; \; X \otimes 1 \sim_{\pi_{\cN}} 0_{\cN}, \; \; 1 \otimes Y \sim_{\pi_{\cN}} Y, \; \; 1 \otimes Y \sim_{\pi_{\M}} 0_{\M},
\end{equation}
where $0_{\M}$ is the zero vector field on $\M$ of degree $|Y|$, $0_{\cN}$ is the zero vector field on $\cN$ of degree $|X|$, and $\pi_{\M}$ and $\pi_{\cN}$ are the projections. 

Let $(U,\varphi)$ and $(V,\psi)$ be the graded local charts for $\M$ and $\cN$, respectively, inducing the corresponding coordinate functions $(\bbz^{A})_{A=1}^{n}$ and $(\bby^{K})_{K=1}^{m}$. We have the induced graded local chart $(U \times V, \varphi \times \psi)$ for $\M \times \cN$. One usually denotes the corresponding $(m+n)$-tuple of coordinate functions using the same symbols. If $X$ and $Y$ can be locally decomposed as 
\begin{equation}
X|_{U} = X^{A} \frac{\partial}{\partial \bbz^{A}}, \; \; Y|_{V} = Y^{K} \frac{\partial}{\partial \bby^{K}}, 
\end{equation}
then the vector fields $X \otimes 1$ and $1 \otimes Y$ have the local form
\begin{equation}
(X \otimes 1)|_{U \times V} = (\pi_{\M}|_{U \times V})^{\ast}_{U}( X^{A}) \frac{\partial}{\partial \bbz^{A}}, \; \; (1 \otimes Y)|_{U \times V} = (\pi_{\cN}|_{U \times V})^{\ast}_{V}(Y^{K}) \frac{\partial}{\partial \bby^{K}}.
\end{equation}
\end{example}
\subsection{Inverse function theorem, etc.}
In calculus and differential geometry, the inverse function theorem is without a doubt one of the most important statements, period. It is thus of vital importance to have its version also in the graded setting. Let $\M = (M, \C^{\infty}_{\M})$ and $\cN = (N, \C^{\infty}_{\cN})$ be graded manifolds of a graded dimension $(n_{j})_{j \in \Z}$ and $(m_{j})_{j \in \Z}$, respectively. Let $\phi: \M \rightarrow \cN$ be a graded smooth map. 

Now, fix $m \in M$ and let $(U,\varphi)$ and $(V,\psi)$ be the graded local charts for $\M$ and $\cN$, respectively, such that $m \in U$ and $\ul{\phi}(U) \subseteq V$. Then one can write 
\begin{equation} \label{eq_Tmphiaschain}
(T_{m}\phi)( \frac{\partial}{\partial \bbz^{A}}|_{m}) = \frac{\partial}{\partial \bbz^{A}}|_{m}( [ \phi^{\ast}_{V}(\bby^{K})|_{U}]_{m}) \frac{\partial}{\partial \bby^{K}}|_{\ul{\phi}(m)},
\end{equation}
where $(\bbz^{A})_{A=1}^{n}$ and $(\bby^{K})_{K=1}^{m}$ are the coordinate functions corresponding to $(U,\varphi)$ and $(V,\psi)$, respectively. We have introduced this unified notation in Remark \ref{rem_unifiedcoordfunct}. To obtain this equation, evaluate (\ref{eq_chainrule}) at $m$, use Proposition \ref{tvrz_valueofvf} and (\ref{eq_differentialformula}). 

Now, for the purposes in this subsection, let us subdivide the coordinate functions $(\bbz^{A})_{A=1}^{n}$ into classes according to their degree. For each $j \in \Z$, let $(\bbz^{A_{j}}_{(j)})_{A_{j}=1}^{n_{j}}$ denote the coordinate functions of degree $j$, that is $|\bbz^{A_{j}}_{(j)}| = j$. Similarly, we write $(\bby^{K_{j}}_{(j)})_{K_{j}=1}^{m_{j}}$. It turns out that for each $j \in \Z$ and $A_{j} \in \{1,\dots,n_{j}\}$, the above equation reads
\begin{equation}
(T_{m}\phi)( \frac{\partial}{\partial \bbz^{A_{j}}_{(j)}}|_{m}) = \frac{\partial}{\partial \bbz^{A_{j}}_{(j)}}|_{m}( [ \phi^{\ast}_{V}(\bby^{K_{j}}_{(j)})|_{U}]_{m}) \frac{\partial}{\partial \bby^{K_{j}}_{(j)}}|_{\ul{\phi}(m)},
\end{equation}
that is the sum over all $K$ reduces to the sum over degree $j$ coordinates, labeled by $K_{j} \in \{1,\dots,m_{j}\}$. This is because the scalar coefficient in front of $\frac{\partial}{\partial \bby^{K}}|_{\ul{\phi}(m)}$ in (\ref{eq_Tmphiaschain}) is zero unless $|\bby^{K}| = |\bbz^{A}|$. For each $j \in \Z$, let us define an $m_{j} \times n_{j}$ matrix $\fD^{(j)}_{m}(\phi)$, elements of which are
\begin{equation}
( \fD^{(j)}_{m}(\phi))^{K_{j}}{}_{A_{j}} := \frac{\partial}{\partial \bbz^{A_{j}}_{(j)}}|_{m}( [ \phi^{\ast}_{V}(\bby^{K_{j}}_{(j)})|_{U}]_{m}),
\end{equation}
for each $K_{j} \in \{1,\dots,m_{j} \}$ and $A_{j} \in \{1,\dots,n_{j}\}$. 

In other words, $\fD^{(j)}_{m}(\phi)$ is nothing but the the matrix of an ordinary linear map $(T_{m}\phi)_{-j}$ with respect to the basis $(\frac{\partial}{\partial \bbz^{A_{j}}_{(j)}}|_{m})_{A_{j}=1}^{n_{j}}$ of $(T_{m}\M)_{-j}$ and $(\frac{\partial}{\partial \bby^{K_{j}}_{(j)}}|_{\ul{\phi}(m)})_{K_{j}=1}^{m_{j}}$ of $(T_{\ul{\phi}(m)}\cN)_{-j}$. 

\begin{definice}
By the \textbf{graded rank of $\phi$ at $m \in M$}, we mean the sequence 
\begin{equation} \grk_{m}(\phi) := ( \rk( (T_{m}\phi)_{-j}) )_{j \in \Z} \equiv ( \rk( \fD^{(j)}_{m}(\phi)) )_{j \in \Z}. \end{equation}
Clearly, one has $\grk_{m}(\phi) \leq \min \{ \gdim(\M), \gdim(\cN) \}$. 
\end{definice}

\begin{definice}
A graded smooth map $\phi: \M \rightarrow \cN$ is called a \textbf{local graded diffeomorphism at $m \in M$}, if $T_{m}\phi: T_{m}\M \rightarrow T_{\ul{\phi}(m)} \cN$ is an isomorphism in $\gVect$. We say that it is a \textbf{local graded diffeomorphism}, if it is a local graded diffeomorphism at every $m \in M$.
\end{definice}

\begin{tvrz}
A graded smooth map $\phi: \M \rightarrow \cN$ is a local graded diffeomorphism at $m \in M$, if and only if $\grk_{m}(\phi) = \gdim(\M) = \gdim(\cN)$. 
\end{tvrz}
\begin{proof}
The statement follows immediately from definitions and the fact that $\gdim(T_{m}\M) = (n_{-j})_{j \in \Z}$, where $(n_{j})_{j \in \Z} := \gdim(\M)$, see Proposition \ref{tvrz_coordtangentvectors}. 
\end{proof}

\begin{lemma} \label{lem_ift}
Let $\phi: \M \rightarrow \cN$ be a local graded diffeomorphism at $m$. Then one can find $U \in \Op_{m}(M)$ and $V \in \Op_{\ul{\phi}(m)}(N)$, such that 
\begin{enumerate}[(i)]
\item $\ul{\phi}|_{U}: U \rightarrow V$ is a diffeomorphism;
\item there exist graded local charts $(U,\varphi)$ and $(V,\psi)$ for $\M$ and $\cN$. 
\end{enumerate}
\end{lemma}
\begin{proof}
It is easy to see that the $n_{0} \times n_{0}$ invertible matrix $\fD^{(0)}_{m}(\phi)$ (constructed using arbitrary graded local charts) is in fact just the matrix of the linear map $T_{m}\ul{\phi}$ with respect the basis of $T_{m}M$ and $T_{\ul{\phi}(m)}N$ obtained from the induced local charts for $M$. It follows that $\ul{\phi}: M \rightarrow N$ is a local diffeomorphism. By the standard inverse function theorem, there is thus $U \in \Op_{m}(M)$ and $V \in \Op_{\ul{\phi}(m)}(N)$, such that $\ul{\phi}|_{U}: U \rightarrow V$ is a diffeomorphism. See e.g. Theorem 4.5 in \cite{lee2012introduction}. By taking suitable intersections, one can assume that $U$ and $V$ are domains for graded local charts for $\M$ and $\cN$, respectively. 
\end{proof}

\begin{theorem}[\textbf{Inverse function theorem}] \label{thm_ift}
Let $\phi: \M \rightarrow \cN$ be be a local graded diffeomorphism at $m \in M$. Then there is $U \in \Op_{m}(M)$ such that 
\begin{enumerate}[(i)] 
\item $V := \ul{\phi}(U)$ is open and $\ul{\phi}: U \rightarrow V$ is a diffeomorphism;
\item $\phi|_{U}: \M|_{U} \rightarrow \cN|_{V}$ is a graded diffeomorphism. 
\end{enumerate}
\end{theorem}
\begin{proof}
Due to Lemma \ref{lem_ift}, it suffices to prove the theorem for a local graded diffeomorphism $\phi: U^{(n_{j})} \rightarrow U^{(n_{j})}$ at a given $m \in U$, where $U \in \Op(\R^{n_{0}})$ and $\ul{\phi} = \1_{U}$. 

For each $j \in \Z$, let $\fD^{(j)}: U \rightarrow \R^{n_{j},n_{j}}$ denote the matrix-valued map $\fD^{(j)}(x) := \fD^{(j)}_{x}(\phi)$ for all $x \in U$. By assumption, we have $\fD^{(j)}(m) \in \GL(n_{j},\R)$. For each $j \in \Z$, there is thus $U_{(j)} \in \Op_{m}(U)$, such that $\fD^{(j)}(U_{(j)}) \subseteq \GL(n_{j},\R)$. Now, note $n_{j} \neq 0$ only for finitely many $j \in \Z$, and for $n_{j} = 0$, one can choose $U_{(j)} = U$. It follows that $U' = \cap_{j \in \Z} U_{(j)} \in \Op_{m}(U)$. But this means that we could have started with $U'$ instead of $U$ and we may assume that $\fD^{(j)}: U \rightarrow \GL(n_{j},\R)$ for every $j \in \Z$. Let $(\bbz^{A_{j}}_{(j)})_{A_{j}=1}^{n_{j}}$ denote the (global) coordinate functions on $U^{(n_{j})}$ of degree $j$. 

It follows from the definition of $\fD^{(j)}$ and Example \ref{ex_diffgradeddomains} that for each $j \in \Z$ and $A_{j} \in \{1,\dots,n_{j}\}$, one can write the expression
\begin{equation}
\phi^{\ast}_{U}( \bbz^{A_{j}}_{(j)}) = (\fD^{(j)}){}^{A_{j}}{}_{B_{j}} \bbz^{B_{j}}_{(j)} + f^{A_{j}},
\end{equation}
where $f^{A_{j}} \in \prod_{p>1} \C^{\infty}_{n_{0}}(U) \otimes S^{p}(\R^{(n_{i})}_{\ast})_{j}$. Note that it is crucial that $\fD^{(0)}$ is the unit matrix since we assume $\ul{\phi} = \1_{U}$. Let us now construct a graded diffeomorphism $\psi: U^{n_{j}} \rightarrow U^{n_{j}}$ with $\ul{\psi} = \1_{U}$. By Remark \ref{rem_itsufficestocalculatepullbacks}, it suffices to declare the pullbacks of coordinate functions. Set
\begin{equation}
\psi^{\ast}_{U}(\bbz^{A_{j}}_{(j)}) := ((\fD^{(j)})^{-1}) {}^{A_{j}}{}_{B_{j}} \bbz^{B_{j}}_{(j)},
\end{equation}
for each $j \in \Z$ and every $A_{j} \in \{1,\dots,n_{j}\}$. It is a graded diffeomorphism as one can easily construct its graded smooth inverse. It follows that for each $j \in \Z$ and every $A_{j} \in \{1,\dots,n_{j}\}$, one has
\begin{equation}
(\phi \circ \psi)^{\ast}_{U}(\bbz^{A_{j}}_{(j)}) = \bbz^{A_{j}}_{(j)} + g^{A_{j}},
\end{equation}
where $g^{A_{j}} \in \prod_{p>1} \C^{\infty}_{(n_{0})}(U) \otimes S^{p}(\R^{(n_{i})}_{\ast})_{j}$. But this proves that $(\phi \circ \psi)^{\ast}_{U} = \1_{\C^{\infty}_{(n_{j})}(U)} + \rho$ where $\rho: \C^{\infty}_{(n_{j})}(U) \rightarrow \C^{\infty}_{(n_{j})}(U)$ is a graded linear map which strictly increases the ``polynomial degree''. In other words, for each $f \in \C^{\infty}_{(n_{j})}(U)$ and each $\fp \in \N^{n_{\ast}}_{|f|}$, one has $(\rho^{q}(f))_{\fp} = 0$ for $q > w(\fp)$. 

It thus makes sense to define a graded linear map $\chi := \sum_{q = 0}^{\infty} (-1)^{q} \rho^{q}$. More precisely, for every $f \in \C^{\infty}_{(n_{j})}(U)$ and $\fp \in \N^{n_{\ast}}_{|f|}$, one defines $(\chi(f))_{\fp} = \sum_{q=0}^{w(\fp)} (-1)^{q} (\rho^{q}(f))_{\fp}$. It is now straightforward to check that $\chi$ is precisely the two-sided inverse to $(\phi \circ \psi)^{\ast}_{U}$. Whence $(\phi \circ \psi)^{\ast}_{U}$ is a graded algebra isomorphism. One can repeat these arguments to prove that $(\phi \circ \psi)^{\ast}_{V}$ is a graded algebra isomorphism for every $V \in \Op(U)$ and we conclude that $(\phi \circ \psi)^{\ast}$ is a sheaf isomorphism. From Proposition \ref{tvrz_isoingLRS}, it follows that $\phi \circ \psi: U^{(n_{j})} \rightarrow U^{(n_{j})}$ is a graded diffeomorphism. 

This shows that $\phi: U^{(n_{j})} \rightarrow U^{(n_{j})}$ is a graded diffeomorphism and the proof is finished. 
\end{proof}
Recall that for ordinary manifolds, a bijective local diffeomorphism is a diffeomorphism. A similar (mildly stronger) statement holds for graded manifolds.
\begin{tvrz} \label{tvrz_diffiflocaldiffandbij}
Let $\phi: \M \rightarrow \cN$ be a graded smooth map. Then $\phi$ is a graded diffeomorphism, iff it is a local graded diffeomorphism and $\ul{\phi}$ is bijective. 
\end{tvrz}
\begin{proof}
We have already proved the only if direction in the proof of Proposition \ref{tvrz_gdiminvariant}. Hence suppose that $\phi$ is a local graded diffeomorphism with a bijective underlying map $\ul{\phi}: M \rightarrow N$. 

For each $m \in M$ one has $U_{m} \in \Op_{m}(M)$ and $V_{m} \in \Op_{\ul{\phi}(m)}(N)$, such that $V_{m} = \ul{\phi}(U_{m})$ and $\phi|_{U_{m}}: \M|_{U_{m}} \rightarrow \cN|_{V_{m}}$ is a graded diffeomorphism. This is just the assumption and Theorem \ref{thm_ift}. Let $\psi_{(m)} := (\phi|_{U_{m}})^{-1}$ and view it as a map $\psi_{(m)}: \cN|_{V_{m}} \rightarrow \M$. For any other $m' \in M$, one has
\begin{equation}
\psi_{(m)}|_{V_{m} \cap V_{m'}} = (\phi|_{U_{m} \cap U_{m'}})^{-1} = \psi_{(m')}|_{V_{m} \cap V_{m'}}. 
\end{equation}
It is important that $V_{m} \cap V_{m'} = \ul{\phi}( U_{m} \cap U_{m'})$, since $\ul{\phi}$ is assumed to be bijective. This also implies that $\{ V_{m} \}_{m \in M}$ is an open cover of $N$. Hence by Proposition \ref{tvrz_gLRSgluing}, there is a unique graded smooth map $\psi: \cN \rightarrow \M$, such that $\psi|_{V_{(m)}} = \psi_{(m)}$. It is clear that $\psi$ is the inverse to $\phi$.
\end{proof}

Observe that the graded rank of a graded smooth map is not a particularly useful notion. The same issue arises with supermanifolds and one has to come up with a more subtle notion of a (super)rank, see Definition 5.2.1 in \cite{carmeli2011mathematical}. 

\begin{example} \label{ex_zeroranknonconstant}
Indeed, consider the graded domain $\R^{(n_{j})}$ with $n_{0} = n_{2} = n_{-2} = 1$, and $n_{j} = 0$ otherwise. We thus have the coordinate functions $x, \xi_{1},\xi_{2} \in \C^{\infty}_{(n_{j})}(\R)$ with $|x| = 0$, $|\xi_{1}| = -2$ and $|\xi_{2}| = 2$. Now, the most general graded smooth map $\phi: \R^{(n_{j})} \rightarrow \R^{(n_{j})}$ is determined by pullbacks of the coordinate functions. We can parametrize them as 
\begin{equation}
\phi^{\ast}_{\R}(x) = \ul{\phi} + \sum_{n=1}^{\infty} a_{n} (\xi_{1})^{n} (\xi_{2})^{n} , \; \; \phi^{\ast}_{\R}(\xi_{1}) = \sum_{n=0}^{\infty} b_{n} (\xi_{1})^{n+1} (\xi_{2})^{n}, \; \; \phi^{\ast}_{\R}(\xi_{2}) = \sum_{n=0}^{\infty} c_{n} (\xi_{1})^{n} (\xi_{2})^{n+1},
\end{equation}
where $\ul{\phi}: \R \rightarrow \R$ is the underlying smooth map and $\{ a_{n} \}_{n=1}^{\infty}$, $\{ b_{n} \}_{n=0}^{\infty}$ a $\{ c_{n} \}_{n=0}^{\infty}$ are smooth functions on $\R$. For each $y \in \R$, the matrices $\fD^{(j)}_{y}(\phi)$ are non-zero only for $j \in \{-2,0,2\}$ and
\begin{equation}
\fD^{(0)}_{y}(\phi) = \begin{pmatrix} (\frac{d}{dx} \ul{\phi})(y) \end{pmatrix}, \; \; \fD^{(-2)}_{y}(\phi) = \begin{pmatrix} b_{0}(y) \end{pmatrix}, \; \; \fD^{(2)}_{y}(\phi) = \begin{pmatrix} c_{0}(y) \end{pmatrix}.
\end{equation}
In particular, whenever $\ul{\phi}$ is constant and $b_{0} = c_{0} = 0$, we find that $\grk_{y}(\phi) = (0)_{j \in \Z}$ for every $y \in \R$. The rest of the functions can be completely arbitrary. 
\end{example}

This suggests that there is no graded analogue of the rank theorem (see e.g. Theorem 4.12 in \cite{lee2012introduction}). Fortunately this is not the case of maps with \textit{maximal rank}. 

\begin{definice} \label{def_immersion}
Let $\phi: \M \rightarrow \cN$ be a graded smooth map. 

We say that $\phi$ is an \textbf{immersion at $m \in M$}, if $T_{m}\phi$ is an injective graded linear map. We say that $\phi$ is an \textbf{immersion}, if it is an immersion at every $m \in M$. 

We say that $\phi$ is an \textbf{submersion at $m \in M$}, if $T_{m}\phi$ is a surjective graded linear map. Such $m \in M$ is called a \textbf{regular point} of $\phi$. We say that $\phi$ is an \textbf{submersion}, if it is a submersion at every $m \in M$. A point $n \in N$ is called a \textbf{regular value} of $\phi$, if all $m \in \ul{\phi}^{-1}(n)$ are regular points.
\end{definice}

\begin{theorem}[\textbf{Local immersion theorem}] \label{thm_immersion}
Let $\phi: \M \rightarrow \cN$ be a graded smooth map. Let $(n_{j})_{j \in \Z} = \gdim(\M)$ and $(m_{j})_{j \in \Z} = \gdim(\cN)$. Let $m \in M$. Then the following statements are equivalent:
\begin{enumerate}[(i)]
\item $\phi$ is an immersion at $m$; 
\item $\grk_{m}(\phi) = (n_{j})_{j \in \Z}$;
\item There exist graded local charts $(U,\varphi)$ for $\M$ and $(V,\psi)$ for $\cN$, having the following properties:
\begin{enumerate}[(a)] 
\item $U \in \Op_{m}(M)$, $V \in \Op_{\ul{\phi}(m)}(N)$ and $\ul{\phi}(U) \subseteq V$.
\item $\ul{\varphi}(U) = \hat{U}$, $\ul{\psi}(V) = \hat{U} \times \hat{W}$, where $\hat{U} \subseteq \R^{n_{0}}$ and $\hat{W} \subseteq \R^{m_{0} - n_{0}}$ are open cubes and one has $\ul{\varphi}(m) = 0$, $\ul{\psi}( \ul{\phi}(m)) = (0,0)$.
\item Let $\hat{\phi}: \hat{U}^{(n_{i})} \rightarrow (\hat{U} \times \hat{W})^{(m_{i})}$ be the corresponding local representative. Then for each $j \in \Z$, we may choose the coordinate functions $(\bbz^{A_{j}}_{(j)})_{A_{j}=1}^{n_{j}}$ and $(\bby^{K_{j}}_{(j)})_{K_{j}=1}^{m_{j}}$ so that 
\begin{align}
\label{eq_locrepimm1} \hat{\phi}^{\ast}_{\hat{U} \times \hat{W}}( \bby^{K_{j}}_{(j)}) = & \ \bbz^{K_{j}}_{(j)}, \text{ for } K_{j} \in \{1,\dots,n_{j} \}, \\
\label{eq_locrepimm2} \hat{\phi}^{\ast}_{\hat{U} \times \hat{W}}( \bby^{K_{j}}_{(j)}) = & \ 0, \text{ for } K_{j} \in \{n_{j}+1, \dots, m_{j} \}.
\end{align}
\end{enumerate}
\end{enumerate}
\end{theorem}
\begin{proof}
$(i) \Leftrightarrow (ii)$ is just the rephrasing using the fact that $\gdim(T_{m}\M) = (n_{-j})_{j \in \Z}$. The implication $(iii) \Rightarrow (i)$ can be verified directly. It remains to prove $(i) \Rightarrow (iii)$. 

As this is a local statement, we may assume that $\phi: U^{(n_{i})} \rightarrow V^{(m_{i})}$ is an immersion at $0 \in U$, where $U \subseteq \R^{n_{0}}$ and $V \subseteq \R^{m_{0}}$ are open cubes and $\ul{\phi}(0) = 0$. 

Using any coordinate functions, for each $j \in \Z$, the $m_{j} \times n_{j}$ matrix $\fD^{(j)}_{0}(\phi)$ has a maximal rank $n_{j}$. We may thus always relabel the degree $j \in \Z$ coordinate functions so that its submatrix consisting of the first $n_{j}$ rows is non-singular. This gives us a decomposition 
\begin{equation}
V^{(m_{i})} = Y^{(n_{i})} \times W^{(m_{i} - n_{i})},
\end{equation}
where $Y \subseteq \R^{n_{0}}$ and $W \subseteq \R^{m_{0} - n_{0}}$ are open cubes. We can then write $\phi = (\phi_{1},\phi_{2})$, where $\phi_{1}: U^{(n_{i})} \rightarrow Y^{(n_{i})}$ is a local graded diffeomorphism at $0$ and $\phi_{2}: U^{(n_{i})} \rightarrow W^{(m_{i} - n_{i})}$. Let 
\begin{equation}
X := \{ (u,z) \in U \times \R^{m_{0} - n_{0}} \; | \; \ul{\alpha}(u,z) := (\ul{\phi_{1}}(u), \ul{\phi_{2}}(u) + z) \in Y \times W \}.
\end{equation}
This is an open subset of $\R^{m_{0}}$. Let $\pi_{1}: X^{(m_{i})} \rightarrow U^{(n_{j})}$ be a restriction of the canonical projection $U^{(n_{i})} \times (\R^{m_{0}-n_{0}})^{(m_{i}-n_{i})} \rightarrow U^{(n_{i})}$ to $X$. As $(u,0) \in X$ for all $u \in U$, the underlying map $\ul{\pi_{1}}: X \rightarrow U$ is surjective. Let us construct a graded smooth map $\alpha: X^{(m_{i})} \rightarrow V^{(m_{i})}$. Let us denote the graded coordinates of degree $j \in \Z$ on $X^{(m_{i})}$ as $(\bbu^{K_{j}}_{(j)})_{K_{j}=1}^{m_{j}}$ and those on $V^{(m_{i})}$ as $(\bby^{K_{j}}_{(j)})_{K_{j}=1}^{m_{j}}$. Set
\begin{align}
\alpha^{\ast}_{V}( \bby^{K_{j}}_{(j)}) := & \ (\phi \circ \pi_{1})^{\ast}_{V}( \bby^{K_{j}}_{(j)}) \text{ for } K_{j} \in \{1,\dots,n_{j} \}, \\
\alpha^{\ast}_{V}( \bby^{K_{j}}_{(j)}) := & \ (\phi \circ \pi_{1})^{\ast}_{V}( \bby^{K_{j}}_{(j)}) + \bbu^{K_{j}}_{(j)} \text{ for } K_{j} \in \{n_{j}+1, \dots, m_{j} \}.
\end{align}
By Remark \ref{rem_itsufficestocalculatepullbacks}, this determines a unique graded smooth map. Its underlying map is $\ul{\alpha}$ used in the definition of the open set $X$. It is a local graded diffeomorphism at $(0,0) \in X$ as 
\begin{equation}
\fD^{(j)}_{(0,0)}(\alpha) = \begin{pmatrix}
\fD^{(j)}_{0}(\phi_{1}) & 0 \\
\fD^{(j)}_{0}(\phi_{2}) & \mathbf{1}_{m_{j} - n_{j}}.
\end{pmatrix}
\end{equation}
By the inverse function theorem, there is thus an open neighborhood $X' \subseteq X$ of $(0,0)$ and an open neighborhood $V' \subseteq Y \times W$ of $\ul{\alpha}(0,0) = (0,0)$, such that $\alpha|_{X'}: X'^{(n_{i})} \rightarrow V'^{(n_{i})}$ is a graded diffeomorphism. By shrinking it when necessary, we may assume that $X' = U' \times W'$ for open cubes $U' \subseteq U$ and $W' \subseteq \R^{m_{0} - n_{0}}$ around the origin.  

We claim that $(U', \1_{U'^{(n_{j})}})$ and $(V',(\alpha|_{X'})^{-1})$ are the required graded local charts. Observe that $\ul{\phi}(U') \subseteq V'$. This follows from the fact that $\ul{\phi}(u') = \ul{\alpha}(u',0) \in V'$ for every $u' \in U'$. The local representative $\hat{\phi}: U'^{(n_{i})} \rightarrow X'^{(m_{i})}$ is uniquely determined by the equation $\alpha|_{X'} \circ \hat{\phi} = \phi|_{U'}$. It thus suffices to show that it is satisfied by $\hat{\phi}$ given by (\ref{eq_locrepimm1}, \ref{eq_locrepimm2}). For any $f \in \C^{\infty}_{(n_{i})}(U)$, one has 
\begin{equation}
\hat{\phi}^{\ast}_{X'}( (\pi_{1})^{\ast}_{U}(f)|_{X'}) = f|_{U'},
\end{equation}
due to (\ref{eq_locrepimm1}), and for any $j \in \Z$ and $K_{j} \in \{n_{j}+1,\dots,m_{j}\}$, one has $\hat{\phi}^{\ast}_{X'}( \bbu^{K_{j}}_{(j)}) = 0$ due to (\ref{eq_locrepimm2}). But then, for any $j \in \Z$ and $K_{j} \in \{1,\dots,m_{j}\}$, one finds
\begin{equation}
(\hat{\phi}^{\ast}_{X'} \circ (\alpha|_{X'})^{\ast}_{U'})( \bby^{K_{j}}_{(j)}) := \hat{\phi}^{\ast}_{X'}( ((\pi_{1})^{\ast}_{U} \circ \phi^{\ast}_{V})(\bby^{K_{j}}_{(j)})|_{X'}) = \phi^{\ast}_{V}( \bby^{K_{j}}_{(j)})|_{U'} = (\phi|_{U'})^{\ast}_{V}( \bby^{K_{j}}_{(j)}). 
\end{equation}
This proves the claim and the proof is finished. 
\end{proof}
\begin{rem}
The proof is just a graded analogue of the one of Theorem 2.5.12 in \cite{abraham2012manifolds}. See also Proposition 5.2.4 in \cite{carmeli2011mathematical} for its  supermanifold version. Also note that in the construction of the special graded local charts in $(iii)$, the graded local chart $(U,\varphi)$ for $\M$ can be obtained by merely shrinking (and possibly translating the origin) of any given graded local chart for $\M$ around $m$ 
\end{rem}
There is a similar version of this theorem for submersions. 
\begin{theorem}[\textbf{Local submersion theorem}] \label{thm_submersion}
Let $\phi: \M \rightarrow \cN$ be a graded smooth map. Let $(n_{j})_{j \in \Z} = \gdim(\M)$ and $(m_{j})_{j \in \Z} = \gdim(\cN)$. Let $m \in M$. Then the following statements are equivalent:
\begin{enumerate}[(i)]
\item $\phi$ is a submersion at $m$; 
\item $\grk_{m}(\phi) = (m_{j})_{j \in \Z}$;
\item There exist graded local charts $(U,\varphi)$ for $\M$ and $(V,\psi)$ for $\cN$, having the following properties:
\begin{enumerate}[(a)] 
\item $U \in \Op_{m}(M)$, $V \in \Op_{\ul{\phi}(m)}(N)$ and $\ul{\phi}(U) \subseteq V$.
\item $\ul{\varphi}(U) = \hat{V} \times \hat{W}$, $\ul{\psi}(V) = \hat{V}$, where $\hat{V} \subseteq \R^{m_{0}}$ and $\hat{W} \subseteq \R^{n_{0} - m_{0}}$ are open cubes and $\ul{\varphi}(m) = 0$, $\ul{\psi}( \ul{\phi}(m)) = (0,0)$.
\item Let $\hat{\phi}: (\hat{V} \times \hat{W})^{(n_{i})} \rightarrow \hat{V}^{(m_{i})}$ be the corresponding local representative. Then for each $j \in \Z$, we may choose the coordinate functions $(\bbz^{A_{j}}_{(j)})_{A_{j}=1}^{n_{j}}$ and $(\bby^{K_{j}}_{(j)})_{K_{j}=1}^{m_{j}}$ so that 
\begin{equation} \label{eq_locrepsub} \hat{\phi}^{\ast}_{\hat{V}}( \bby^{K_{j}}_{(j)}) = \bbz^{K_{j}}_{(j)}, \end{equation}
for all $K_{j} \in \{1,\dots,m_{j}\}$.  
\end{enumerate}
\end{enumerate}
\end{theorem}
\begin{proof}
The proof is very similar to the one of the immersion theorem. Again, the only non-trivial statement is $(i) \Rightarrow (iii)$. 

As it is a local statement, one can assume that $\phi: U^{(n_{i})} \rightarrow V^{(m_{i})}$ is a submersion at $0 \in U$, where $U \subseteq \R^{n_{0}}$ and $V \subseteq \R^{m_{0}}$ are open cubes and $\ul{\phi}(0) = 0$. 

Using any coordinate functions, for each $j \in \Z$, the $m_{j} \times n_{j}$ matrix $\fD^{(j)}_{0}(\phi)$ has a maximal rank $m_{j}$. We may thus always relabel the degree $j \in \Z$ coordinate functions so that its submatrix consisting of the first $m_{j}$ columns is non-singular. This gives us a decomposition 
\begin{equation}
U^{(n_{i})} = Y^{(m_{i})} \times W^{(n_{i} - m_{i})},
\end{equation}
where $Y \subseteq \R^{m_{0}}$ and $W \subseteq \R^{n_{0} - m_{0}}$ are open cubes. This time, construct a graded smooth map $\beta: U^{(n_{i})} \rightarrow V^{(m_{i})} \times W^{(n_{i} - m_{i})}$. Let us denote the coordinate functions of degree $j \in \Z$ on $U^{(n_{i})}$ as $(\bbz^{A_{j}}_{(j)})_{A_{j}=1}^{n_{j}}$ and those on $V^{(m_{i})} \times W^{(n_{i}-m_{i})}$ as $(\bby^{K_{j}}_{(j)})_{K_{j}=1}^{m_{j}}$ and $(\bbz^{A_{j}}_{(j)})_{A_{j}=m_{j}+1}^{n_{j}}$, corresponding to the equally labeled coordinate functions on $V^{(m_{i})}$ and $W^{(n_{i}-m_{i})}$, respectively. Set 
\begin{align}
\beta^{\ast}_{V \times W}( \bby^{K_{j}}_{(j)}) := & \ \phi^{\ast}_{V}( \bby^{K_{j}}_{(j)}), \text{ for all } K_{j} \in \{1,\dots,m_{j}\}, \\
\beta^{\ast}_{V \times W}( \bbz^{A_{j}}_{(j)}) := & \ \bbz^{A_{j}}_{(j)}, \text{ for all } A_{j} \in \{m_{j}+1, \dots, n_{j} \}.
\end{align}
It follows that for each $j \in \Z$, one finds 
\begin{equation}
\fD_{(0,0)}^{(j)}(\beta) = \begin{pmatrix}
\fD_{(0,0)}^{(j)}(\phi) \\
\begin{array}{cc} 0 & \f1_{n_{j}-m_{j}} \end{array}
\end{pmatrix}.
\end{equation}
By assumption, this is a non-singular matrix, whence $\beta$ is a local diffeomorphism at $(0,0)$. There is thus an open neighborhood $U' \subseteq U$ of $(0,0)$ and an open neighborhood $X' \subseteq V \times Y$, such that $\beta|_{U'}: U'^{(n_{i})} \rightarrow X'^{(n_{i})}$ is a graded diffeomorphism. By shrinking it if necessary, we may assume that $X' = V' \times Y'$ for open cubes $V' \subseteq V \subseteq \R^{m_{0}}$ and $Y' \subseteq Y \subseteq \R^{n_{0}-m_{0}}$. 

We claim that $(U', (\beta|_{U'})^{-1})$ and $(V', \1_{V'^{(m_{i})}})$ are the required graded local charts. Observe that $\ul{\phi}(U') \subseteq V'$, this follows from the definition of $\beta$. It suffices to prove that $\hat{\phi}: (V' \times Y')^{(n_{i})} \rightarrow V'^{(m_{i})}$ given by (\ref{eq_locrepsub}) satisfies the equation $\hat{\phi} \circ \beta|_{U'} = \phi|_{U'}$. But that is an easy check.
\end{proof}

\begin{rem}
The proof is just a graded analogue of the one of Theorem 2.5.13 in \cite{abraham2012manifolds}. See also Proposition 5.2.5 in \cite{carmeli2011mathematical} for its supermanifold version. Observe that the graded local chart $(V,\psi)$ for $\cN$ in $(iii)$ can be in fact obtained just by shrinking (and possibly translating the origin) of any given graded local chart around $\ul{\phi}(m)$. 
\end{rem}

Now, consider a pair of graded smooth manifolds $\M = (M,\C^{\infty}_{\M})$ and $\cN = (N,\C^{\infty}_{\cN})$. For each $(m,n) \in M \times N$, one has a canonical isomorphism
\begin{equation}
T_{(m,n)}(\M \times \cN) \cong T_{m}\M \oplus T_{n}\cN,
\end{equation}
given by the map $(T_{(m,n)}\pi_{\M}, T_{(m,n)} \pi_{\cN})$, where $\pi_{\M}: \M \times \cN \rightarrow \M$ and $\pi_{\cN}: \M \times \cN \rightarrow \cN$ are the canonical projections. With this in mind, one can formulate the following theorem:
\begin{theorem}[\textbf{Implicit function theorem}]
Let $\phi: \M \times \cN \rightarrow \cS$ be a graded smooth map. Let $(m,n) \in M \times N$ and write $s := \ul{\phi}(m,n)$. Suppose that the restriction $(T_{(m,n)}\phi)|_{T_{n}\cN} : T_{n}\cN \rightarrow T_{s}\cS$ is an isomorphism of graded vector spaces.

Then there exists $U \in \Op_{m}(M)$ and a graded smooth map $\psi: \M|_{U} \rightarrow \cN$, such that 
\begin{enumerate}[(i)]
\item $\ul{\psi}(m) = n$;
\item $\phi \circ (i,\psi) = \kappa_{s}$, where $i: \M|_{U} \rightarrow \M$ is the canonical inclusion, and $\kappa_{s}: \M|_{U} \rightarrow \cS$ is the constant mapping valued at $s = \ul{\phi}(m,n)$, see Example \ref{ex_constantmapping}. 
\item For every $m' \in U$, the graded linear map $(T_{(m',\ul{\psi}(m'))}\phi)|_{T_{\ul{\psi}(m')}\cN}$ is an isomorphism, and 
\begin{equation}
(T_{m'}\psi)(v) = - [(T_{(m',\ul{\psi}(m'))}\phi)|_{T_{\ul{\psi}(m')}\cN}]^{-1}( (T_{(m',\ul{\psi}(m'))}\phi)(v,0)),
\end{equation}
for every $v \in T_{m'}\M$. 
\end{enumerate}
\end{theorem}
\begin{proof}
Let $\alpha := (\pi_{\M}, \phi)$, where $\pi_{\M}: \M \times \cN \rightarrow \M$ is the projection. This is a graded smooth map from $\M \times \cN$ to $\M \times \cS$, and $\ul{\alpha}(m',n') = (m', \ul{\phi}(m',n'))$ for all $(m',n') \in M \times N$. With respect to the canonical decompositions
\begin{equation}
T_{(m,n)}(\M \times \cN) = T_{m}\M \oplus T_{n}\cN, \; \; T_{(m,s)}(\M \times \cS) = T_{m}\M \oplus T_{s}\cS,
\end{equation}
the differential $T_{(m,n)} \alpha $ has the block form 
\begin{equation} \label{eq_TmnalphaimplFT}
T_{(m,n)}(\alpha) = \begin{pmatrix}
\1_{T_{m}\M} & 0 \\
(T_{(m,n)}\phi)|_{T_{m}\M} & (T_{(m,n)}\phi)|_{T_{n}\cN} 
\end{pmatrix}
\end{equation}

This shows that $\alpha$ is a local graded diffeomorphism at $(m,n)$. By the inverse function theorem, there are open neighborhoods $X \in \Op_{(m,n)}(M \times N)$ and $U \times V \in \Op_{(m,s)}(M \times S)$, such that $\alpha|_{X}: (\M \times \cN)|_{X} \rightarrow \M|_{U} \times \cS|_{V}$ is a graded diffeomorphism. 

As $s \in V$, we may consider a graded smooth map $\varphi := (\1_{\M|_{U}}, \kappa_{s}): \M|_{U} \rightarrow \M|_{U} \times \cS|_{V}$, where $\kappa_{s}: \M|_{U} \rightarrow \cS|_{V}$ is the constant mapping valued at $s$. Note that $\ul{\varphi}(u) = (u,s)$ for all $u \in U$. Finally, we can define
\begin{equation}
\psi := \pi_{\cN}|_{X} \circ (\alpha|_{X})^{-1} \circ \varphi: \M|_{U} \rightarrow \cN
\end{equation}
We claim that this is a required graded smooth map. First, note that 
\begin{equation}
\ul{\psi}(m) = \ul{\pi_{\cN}}|_{X}( (\ul{\alpha}|_{X})^{-1} (\ul{\varphi}(m))) = \ul{\pi_{\cN}}|_{X}( (\ul{\alpha}|_{X})^{-1} (m,s)) = \ul{\pi_{\cN}}|_{X}(m,n) = n.
\end{equation}
This shows that $\psi$ has the property $(i)$. Next, we claim that 
\begin{equation} \pi_{\M}|_{X} \circ (\alpha|_{X})^{-1} = i \circ \pi_{\M}|_{U \times V}, \end{equation}
where $i: \M|_{U} \rightarrow \M$ is the inclusion. By composing both sides with $\alpha|_{X}$, this is equivalent to the formula $\pi_{\M}|_{X} = i \circ \pi_{\M}|_{U \times V} \circ \alpha|_{X}$. This follows from the definition of $\alpha$. Moreover, we claim that
\begin{equation}
(i,\psi) = j \circ (\alpha|_{X})^{-1} \circ \varphi,
\end{equation}
where $j: (\M \times \cN)|_{X} \rightarrow \M \times \cN$ is the canonical inclusion. Indeed, one finds
\begin{align}
\pi_{\M} \circ j \circ (\alpha|_{X})^{-1} \circ \varphi = & \ \pi_{\M}|_{X} \circ (\alpha|_{X})^{-1} \circ \varphi = i \circ \pi_{\M}|_{U \times V} \circ \varphi = i, \\
\pi_{\cN} \circ j \circ (\alpha|_{X})^{-1} \circ \varphi = & \ \pi_{\cN}|_{X} \circ (\alpha|_{X})^{-1} \circ \varphi = \psi.
\end{align}
Using this equation, we can thus write
\begin{equation}
\phi \circ (i,\psi) = (\pi_{\cS} \circ \alpha) \circ (j \circ (\alpha|_{X})^{-1} \circ \varphi) = \pi_{\cS} \circ \alpha|_{X} \circ (\alpha|_{X})^{-1} \circ \varphi = \kappa_{s}.
\end{equation}
This shows that $\psi$ has the property $(ii)$. 

Finally, note that for each $m' \in U$, one has $\ul{\alpha}(m',\ul{\psi}(m')) = (m',s) \in U \times V$. In particular, one has $(m',\ul{\psi}(m')) \in X$. Consequently $T_{(m',\ul{\psi}(m'))}(\alpha)$ must be a graded vector space isomorphism for every $m' \in U$. It has the block form (\ref{eq_TmnalphaimplFT}) with $(m,n)$ replaced by $(m',\ul{\psi}(m'))$. This implies that $(T_{m',\ul{\psi}(m'))} \phi)|_{T_{\ul{\psi}(m')} \cN}$ is an isomorphism for all $m' \in U$. For any $v \in T_{m'}\M$, one obtains 
\begin{equation}
\begin{split}
0 = & \ (T_{m'} \kappa_{s})(v) = T_{m'}( \phi \circ (i, \psi)) = (T_{(m',\ul{\psi}(m'))} \phi)( v, (T_{m'}\psi)(v)) \\
= & \ (T_{(m',\ul{\psi}(m'))} \phi)(v,0) + (T_{(m',\ul{\psi}(m'))} \phi)|_{T_{\ul{\psi}(m')}\cN}( (T_{m'}\psi)(v)).
\end{split}
\end{equation}
This finishes the proof. 
\end{proof}
\section{Graded vector bundles} \label{sec_vb}
\subsection{Definition and basic properties}
Our intention in this section is to generalize the notion of a vector bundle over a manifold. It is convenient to take a rather algebraic approach. In ordinary differential geometry, every vector bundle $\pi: E \rightarrow M$ is (up to an isomorphism) uniquely determined by its sheaf of sections $\Gamma_{E}$. It is a locally freely and finitely generated  sheaf of graded $C^{\infty}_{M}$-modules.  This is the reason we have in detail explored a theory of sheaves of graded modules in Subsection \ref{subsec_sheavesofgradedmodules}. 

Let $\M = (M,\C^{\infty}_{\M})$ be a graded manifold. As already noted in Remark \ref{rem_gradedCinfmodulesaregood}, locally freely and finitely generated sheaves of $\C^{\infty}_{\M}$-modules have a well-defined graded rank at each $m \in M$. 

\begin{definice} \label{def_vb}
By a \textbf{graded vector bundle} $\E$ over a graded manifold $\M$, we mean a locally freely and finitely generated sheaf of $\C^{\infty}_{\M}$-modules $\Gamma_{\E} \in \Sh^{\C^{\infty}_{\M}}_{\lfin}(M,\gVect)$ of a constant graded rank. 

We can thus define the \textbf{graded rank of $\E$} as $\grk(\E) := \grk_{m}(\Gamma_{\E})$ for any $m \in M$. 
We will call $\Gamma_{\E}$ the \textbf{sheaf of sections of the graded vector bundle $\E$}.
\end{definice}

\begin{rem}
Recall that by Proposition \ref{tvrz_locfinfreegensheaves}, a graded rank of every locally freely and finitely generated sheaf of $\C^{\infty}_{\M}$-modules is automatically constant on each (path) connected component of $M$. The requirement of a constant graded rank is thus just a matter of taste (the same dilemma appears already in the ordinary case).

Let us also apologize for a rather strange nomenclature, as a vector bundle $\E$ is at the same time its own sheaf of sections $\Gamma_{\E}$. This results from our philosophy to mimic the terminology and the notation of ordinary differential geometry. However, we will show in the following that there actually exists a graded manifold $\E$ and a projection $\pi: \E \rightarrow \M$ associated uniquely (up to a diffeomorphism) to the sheaf $\Gamma_{\E}$. 
\end{rem}

\begin{example} \label{ex_trivialVB}
For any finite-dimensional $K \in \gVect$, let $\Gamma_{\E} := \C^{\infty}_{\M}[K]$. See Example \ref{ex_modelfreelygensheaf} for details. We say that $\E$ is a \textbf{trivial graded vector bundle} and write $\E = \M \times K$. Clearly $\grk(\M \times K) = \gdim(K)$. 
\end{example} 

\begin{example}
For any graded manifold $\M = (M,\C^{\infty}_{\M})$, the \textbf{tangent bundle $T\M$ of the graded manifold $\M$} is defined by $\Gamma_{T\M} := \X_{\M}$. Here $\grk(T\M) = (n_{-j})_{j \in \Z}$, where $(n_{j})_{j \in \Z} := \gdim(\M)$.  
\end{example}

For any section $\sigma \in \Gamma_{\E}(U)$ and any $f \in \C^{\infty}_{\M}(U)$, we write just $f \sigma$ for $f \tr \sigma$. 
One can take the advantage of additional tools available on the sheaf $\C^{\infty}_{\M}$. Notably, we have partitions of unity and bump functions. This gives us the following useful tool. 

\begin{tvrz}[\textbf{Extension lemma for sections}]
Let $\E$ be a vector bundle over $\M$. Let $U \in \Op(M)$ and $\sigma \in \Gamma_{\E}(U)$ be arbitrary. 

Then for any $V \in \Op(U)$, such that $\ol{V} \subseteq U$, there is a section $\ol{\sigma} \in \Gamma_{\E}(M)$ satisfying $\ol{\sigma}|_{V} = \sigma|_{V}$. 
\end{tvrz}
\begin{proof}
Pick a graded bump function $\lambda \in \C^{\infty}_{\M}(V)$ supported on $U$, such that $\lambda|_{V} = 1$. We can then define a section $\ol{\sigma}$ by its restrictions to the open cover $\{ U, \supp(\lambda)^{c} \}$, namely 
\begin{equation}
\ol{\sigma}|_{U} := \lambda|_{U} \sigma, \; \; \ol{\sigma}|_{\supp(\lambda)^{c}} := 0. 
\end{equation}
Both definitions agree on $U \cap \supp(\lambda)^{c}$, hence they define a unique section $\ol{\sigma} \in \Gamma_{\E}(M)$, and $\ol{\sigma}|_{V} = \lambda|_{V} \sigma|_{V} = \sigma|_{V}$. Note that it is crucial that $\Gamma_{\E}$ has the property (ii) in Definition \ref{def_sheavesofAmodules}. 
\end{proof}
\begin{definice}
For a given vector bundle $\E$, one can choose a single $K \in \gVect$, such that for each $m \in M$, one can find $U \in \Op_{m}(M)$ together with a $\C^{\infty}_{\M}|_{U}$-linear sheaf isomorphism $\varphi: \Gamma_{\E}|_{U} \rightarrow \C^{\infty}_{\M}|_{U}[K]$. Note that $\grk(\E) = \gdim(K)$. $K$ is called the \textbf{typical fiber of $\E$} and $(U,\varphi)$ is called a \textbf{local trivialization chart}.
\end{definice}

\begin{example}
Let $\E$ be a graded vector bundle over $\M$. It follows from Proposition \ref{tvrz_dualfinfree} and Proposition \ref{tvrz_finfreegenlocal} that the dual $\Gamma_{\E}^{\ast}$ is locally freely and finitely generated of a constant graded rank, hence it forms a sheaf of sections of a graded vector bundle. 

We denote it as $\E^{\ast}$ and call it the \textbf{dual graded vector bundle to $\E$}. If the typical fiber of $\E$ is $K$, we can choose its dual $K^{\ast}$ to be the typical fiber of $\E^{\ast}$. 
\end{example}
Having the extension lemma at hand, one can find an easier description of sections of $\E^{\ast}$.
\begin{tvrz} \label{tvrz_dualsectionseasy}
Let $\E$ be a graded vector bundle over $\M$. For each $U \in \Op(M)$, define
\begin{equation}
\Omega^{1}_{\E}(U) := \Lin^{\C^{\infty}_{\M}(U)}( \Gamma_{\E}(U), \C^{\infty}_{\M}(U)), 
\end{equation}
see Proposition \ref{tvrz_LinAmodule}. Then there is a way to make $\Omega^{1}_{\E}$ into a sheaf of graded $\C^{\infty}_{\M}$-modules. Moreover, there is a canonical $\C^{\infty}_{\M}$-linear sheaf isomorphism of $\Gamma_{\E^{\ast}} \equiv \Gamma^{\ast}_{\E}$ and $\Omega^{1}_{\E}$.
\end{tvrz}
\begin{proof}
$\Omega^{1}_{\E}$ is made into a sheaf of graded $\C^{\infty}_{\M}$-modules using completely the same procedure as the sheaf of vector fields $\X_{M}$, see the proof of Proposition \ref{tvrz_vf}. The $\C^{\infty}_{\M}(U)$-linearity of each $\alpha \in \Omega^{1}_{\E}(U)$ ensures that $\alpha(\sigma)|_{V} = 0$ whenever $\sigma|_{V} = 0$ for some $V \in \Op(U)$. The extension lemma is then used to construct $\alpha|_{V} \in \Omega^{1}_{\E}(V)$ for each $V \in \Op(U)$, satisfying 
\begin{equation} \label{eq_alpharestonrest}
\alpha|_{V'}(\sigma|_{V}) = \alpha|_{V}(\sigma)|_{V'}
\end{equation} 
for any $\sigma \in \Gamma_{\E}(V)$ and $V' \subseteq V \subseteq U$. Let us construct $\varphi: \Gamma^{\ast}_{\E} \rightarrow \Omega^{1}_{\E}$. For each $U \in \Op(M)$ and $\alpha \in \Gamma^{\ast}_{\E}(U)$ set $\varphi_{U}(\alpha) := \alpha_{U}$. It is easy to check that $\varphi$ is a $\C^{\infty}_{\M}$-linear sheaf morphism. Let us find its inverse. For any $\beta \in \Omega^{1}_{\E}(U)$ and $V \in \Op(U)$, let $(\varphi_{U}^{-1}(\beta))_{V} := \beta|_{V}$. The naturality in $V$ follows immediately from (\ref{eq_alpharestonrest}), whence $\varphi_{U}^{-1}(\beta) \in \Gamma^{\ast}_{\E}(U)$. $\varphi_{U}^{-1}$ is obviously the inverse to $\varphi_{U}$, hence it is automatically $\C^{\infty}_{\M}(U)$-linear and natural in $U$. Hence $\varphi$ is a $\C^{\infty}_{\M}$-linear sheaf isomorphism. 
\end{proof}

It is easy to define a category of graded vector bundles over a given graded manifold $\M$. However, one can consider a bigger category of graded vector bundles over \textit{all} graded manifolds. This requires some preliminary work to be done. 
\begin{rem} \label{rem_pushforwardmodules}
For any graded smooth map $\varphi: \M \rightarrow \cN$ and any sheaf of graded $\C^{\infty}_{\M}$-modules $\F$, the pushforward sheaf $\ul{\varphi}_{\ast} \F \in \Sh(N, \gVect)$ can be made into a sheaf of graded $\C^{\infty}_{\cN}$-modules. Indeed, for any $V \in \Op(N)$ and any $g \in \C^{\infty}_{\cN}(V)$, and $s \in (\ul{\varphi}_{\ast} \F)(V) \equiv \F(\ul{\varphi}^{-1}(V))$, set
\begin{equation}
g \tr s := \varphi^{\ast}_{V}(g) \tr s,
\end{equation} 
where on the right-hand side, there is the original graded $\C^{\infty}_{\M}( \ul{\varphi}^{-1}(V))$-module action. We will denote this sheaf of graded $\C^{\infty}_{\cN}$-modules simply as $\varphi_{\ast} \F$. 
\end{rem}
As in the case of ordinary vector bundles, to define vector bundle morphisms over arbitrary smooth maps in terms of sections, one has to ``go dual''. 
\begin{definice} 
Let $\E$ be a graded vector bundle over a graded manifold $\M$ and $\E'$ a graded vector bundle over a graded manifold $\cN$. By a \textbf{graded vector bundle morphism $F: \E \rightarrow \E'$ over $\varphi$}, we mean a pair $F = (\varphi, F^{\dagger})$, where $\varphi: \M \rightarrow \cN$ is a graded smooth map and
\begin{equation}
F^{\dagger} : \Gamma_{\E'}^{\ast} \rightarrow \varphi_{\ast} \Gamma_{\E}^{\ast}
\end{equation}
is a $\C^{\infty}_{\cN}$-linear sheaf morphism. 
\end{definice}
If $\E''$ is another vector bundle over a graded manifold $\cS = (S, \C^{\infty}_{\cS})$ and $G = (\psi,G^{\dagger})$ is a graded vector bundle morphism $G: \E' \rightarrow \E''$, set $G \circ F := (\psi \circ \varphi, (G \circ F)^{\dagger})$, where 
\begin{equation} \label{eq_gvbmorfcomp}
(G \circ F)^{\dagger}_{V} := F^{\dagger}_{\ul{\psi}^{-1}(V)} \circ G^{\dagger}_{V},
\end{equation}
for every $V \in \Op(S)$. It is easy to check that this defines a $\C^{\infty}_{\cN}$-linear sheaf morphism from $\Gamma^{\ast}_{\E''}$ to $(\psi \circ \varphi)_{\ast} \Gamma^{\ast}_{\E}$, hence a graded vector bundle morphism $G \circ F: \E \rightarrow \E''$ over $\psi \circ \varphi$. It is not difficult to see that graded vector bundles together with graded vector bundle morphisms form the \textbf{category of graded vector bundles} denoted as $\gVBun$. 

\begin{rem} \label{rem_VBmorphismsaremorphisms}
Let $\E$ and $\E'$ be two graded vector bundles over the same graded manifold $\M$. Then every graded vector bundle morphism $F: \E \rightarrow \E'$ over the identity $\1_{\M}$ corresponds to the unique $\C^{\infty}_{\M}$-linear sheaf morphism $F: \Gamma_{\E} \rightarrow \Gamma_{\E'}$ by means or Proposition \ref{tvrz_transpose}, that is $F^{\dagger} = F^{T}$.  
\end{rem}
\begin{example} \label{ex_trivasvbunmorf}
Let $\E$ be a graded vector bundle over a graded manifold $\M = (M,\C^{\infty}_{\M})$. For any $U \in \Op(M)$, one can consider a restricted vector bundle $\E|_{U}$ with $\Gamma_{\E|_{U}} := \Gamma_{\E}|_{U}$. It follows from the preceding remark that each local trivialization chart $(U,\varphi)$ for $\E$ can be viewed as a graded vector bundle isomorphism $\varphi: \E|_{U} \rightarrow \M|_{U} \times K$ over $\1_{\M|_{U}}$.
\end{example}
\subsection{Subbundles, quotients and pullbacks}
\begin{definice}
Let $\E$ be a graded vector bundle over $M$. Consider a sheaf $\Gamma_{\cL}$ of graded $\C^{\infty}_{\M}$-submodules of $\Gamma_{\E}$, that is
\begin{enumerate}[(i)]
\item for each $U \in \Op(M)$, $\Gamma_{\cL}(U) \subseteq \Gamma_{\E}(U)$ is a graded $\C^{\infty}_{\M}(U)$-submodule;
\item the restriction morphisms of $\Gamma_{\cL}$ are obtained from those of $\Gamma_{\E}$. 
\end{enumerate}
Finally, assume that for each $m \in M$, there exists a local trivialization chart $(U,\varphi)$ for $\E$ around $m$, such that $\varphi( \Gamma_{\cL}|_{U}) = \C^{\infty}_{\M}|_{U}[L]$ for some fixed graded subspace $L \subseteq K$. Then $\Gamma_{\cL}$ defines a \textbf{graded subbundle} $\cL \subseteq \E$. It is obvious that $\cL$ is itself a graded vector bundle and $\grk(\cL) = \gdim(L)$. 
\end{definice}

Graded subbundles can be used to obtain new graded vector bundles. 
\begin{tvrz} \label{tvrz_quotient}
Let $\E$ be a graded vector bundle and $\cL \subseteq \E$ its graded subbundle. Then the quotient $\Gamma_{\E} / \Gamma_{\cL}$ defines a locally freely and finitely generated sheaf of graded $\C^{\infty}_{\M}$-modules of a constant graded rank. It thus defines a \textbf{quotient graded vector bundle} $\E / \cL$ with $\Gamma_{\E/\cL} := \Gamma_{\E} / \Gamma_{\cL}$. 

If the typical fiber of $\E$ and $\cL$ is $K$ and $L$, respectively, the typical fiber of $\E / \cL$ is the quotient $K / L$. The canonical quotient map $Q: \Gamma_{\E} \rightarrow \Gamma_{\E} / \Gamma_{\cL}$ can be viewed as a graded vector bundle morphism $Q: \E \rightarrow \E/\cL$ over the identity $\1_{\M}$. 
\end{tvrz}

\begin{proof}
Similarly to Remark \ref{rem_quotientsheaf}, we obtain a quotient presheaf $\Gamma_{\E} / \Gamma_{\cL}$ of graded $\C^{\infty}_{\M}$-modules. Moreover, using the same arguments as in the proof of Proposition \ref{eq_quotientpresheafissheaf}, one can show that $\Gamma_{\E} / \Gamma_{\cL}$ is actually a sheaf. 

Now, let $m \in M$ be arbitrary and let $(U,\varphi)$ be the local trivialization chart for $\E$ around $m$ obtained from the definition of a subbundle $\cL$, that is $\varphi(\Gamma_{\cL}|_{U}) = \C^{\infty}_{\M}|_{U}[L]$, where $L$ is the typical fiber of $\cL$, a graded subspace of the typical fiber $K$ of $\E$. Let $q: K \rightarrow K/L$ and $Q: \Gamma_{\E} \rightarrow \Gamma_{\E} / \Gamma_{\cL}$ denote the canonical quotient maps. It follows that there is a unique $\C^{\infty}_{\M}$-linear sheaf isomorphism $\hat{\varphi}: (\Gamma_{\E} / \Gamma_{\cL})|_{U} \rightarrow \C^{\infty}_{\M}|_{U}[K/L]$ fitting into the commutative diagram 
\begin{equation}
\begin{tikzcd}
\Gamma_{\E}|_{U} \arrow{d}{Q|_{U}} \arrow{r}{\varphi} & \C^{\infty}_{\M}|_{U}[K] \arrow{d}{\1 \otimes q}\\
(\Gamma_{\E} / \Gamma_{\cL})|_{U} \arrow[dashed]{r}{\hat{\varphi}} & \C^{\infty}_{\M}|_{U}[K/L]
\end{tikzcd},
\end{equation}
where $(\1 \otimes q)_{V} := \1_{\C^{\infty}_{\M}(V)} \otimes q$ for all $V \in \Op(U)$. This proves that $\Gamma_{\E} / \Gamma_{\cL}$ is locally freely and finitely generated. Its graded rank is constant and equal to $\grk(\E) - \grk(\cL)$, hence it defines a graded vector bundle $\E/ \cL$ with $\Gamma_{\E/\cL} := \Gamma_{\E} / \Gamma_{\cL}$ with a typical fiber $K/L$. Finally, as $Q: \Gamma_{\E} \rightarrow \Gamma_{\E}/\Gamma_{\cL}$ is a $\C^{\infty}_{\M}$-linear sheaf morphism, it corresponds to a unique sheaf morphism $Q: \E \rightarrow \E/\cL$ over $\1_{\M}$ due to Remark \ref{rem_VBmorphismsaremorphisms}. 
\end{proof}
Now, one of the most vital constructions in ordinary theory of vector bundles is a pullback vector bundle. It turns out that everything works without any hiccups in the graded setting. 

\begin{tvrz} \label{tvrz_pullbackVB}
Let $\E$ be a graded vector bundle over $\M$. Let $\cN$ be any graded manifold together with a graded smooth map $\phi: \cN \rightarrow \M$. 

Then there exists a graded vector bundle $\phi^{\ast} \E$ over $\cN$ having the following properties:
\begin{enumerate}[(i)]
\item It has the same typical fiber (hence also the graded rank) as $\E$.
\item There exists a canonical graded vector bundle morphism $\hat{\phi}: \phi^{\ast}\E \rightarrow \E$ over $\phi$.
\item Let $\E'$ be any graded vector bundle over $\cN$ and let $F: \E' \rightarrow \E$ be any graded vector bundle morphism over $\phi$. Then there exists a unique graded vector bundle morphism $\hat{F}: \E' \rightarrow \phi^{\ast}\E$ over $\1_{\cN}$, such that $F = \hat{\phi} \circ \hat{F}$. 
\item For any $U \in \Op(M)$ and $\sigma \in \Gamma_{\E}(U)$, there exists a \textbf{pullback section} $\sigma^{!} \in \Gamma_{\phi^{\ast}\E}( \ul{\varphi}^{-1}(U))$. If $(\Phi_{\lambda})_{\lambda=1}^{m}$ is a local frame for $\E$ over $U$, then $(\Phi^{!}_{\lambda})_{\lambda=1}^{m}$ is a local frame for $\phi^{\ast}\E$ over $\ul{\phi}^{-1}(U)$. 
\item If $\cS$ is another graded manifold together with a graded smooth map $\psi: \cS \rightarrow \cN$, there is a canonical graded vector bundle isomorphism
\begin{equation} \label{eq_pullbackcomp}
(\phi \circ \psi)^{\ast}\E \cong \psi^{\ast}(\phi^{\ast} \E).
\end{equation}
\end{enumerate}
$\phi^{\ast}\E$ is called the \textbf{pullback graded vector bundle}.
\end{tvrz}
\begin{proof}
First, note that $(ii)$ and $(iii)$ determine $\phi^{\ast}\E$ up to an isomorphism in $\gVBun$ over $\1_{\cN}$. In other words, let $\G$ be another graded vector bundle over $\cN$ together with a graded vector bundle morphism $\hat{\phi}': \G \rightarrow \E$ over $\phi$ having the universal property $(iii)$. Then there is a unique graded vector bundle isomorphism $F: \G \rightarrow \phi^{\ast} \E$ over $\1_{\cN}$ such that $\hat{\phi} \circ F = \hat{\phi}'$. This is an easy exercise. 

One only has to construct the suitable sheaf of sections of $\phi^{\ast} \E$. One defines $\Gamma_{\phi^{\ast}\E} := \phi^{\ast} \Gamma_{\E}$, where $\phi^{\ast} \Gamma_{\E}$ is the so called \textit{pullback sheaf}. This is a standard construction in algebraic geometry, see e.g. Chapter II, Section 5 of \cite{hartshorne2013algebraic}. However, it is not so straightforward and we have decided to move the detailed discussion into the appendix, see Proposition \ref{tvrz_ap_pullbackVB}.
\end{proof}

\begin{example} \label{ex_pullbacktrivial}
For any finite-dimensional graded vector space $K$ and any graded smooth map $\phi: \cN \rightarrow \M$, one has $\phi^{\ast}(\M \times K) = \cN \times K$, see Example \ref{ex_trivialVB}. One can see this directly using the actual construction of the pullback sheaf in the proof of Proposition \ref{tvrz_ap_pullbackVB}. 

Note that this example also shows that $\C^{\infty}_{\cN} = \phi^{\ast} \C^{\infty}_{\M}$, choose $K = \R$. 
\end{example}
The rather abstract notion of a pullback graded vector bundle can be utilized to easily construct some interesting new graded vector bundles. 
\begin{example} \label{ex_fiberofgVB}
Let $\E$ be any graded vector bundle. For a given $m \in M$, consider the constant mapping $\kappa_{m}: \{ \ast \} \rightarrow \M$ valued at $m$. Here $\{ \ast \}$ is a singleton manifold, viewed as a graded manifold. We define the \textbf{fiber of $\E$ over $m$} as the pullback 
\begin{equation}
\E_{m} := \kappa_{m}^{\ast} \E. 
\end{equation}
This is a graded vector bundle over $\{ \ast \}$, hence it can be canonically identified with the graded vector space. Any local trivialization chart $(U,\varphi)$ for $\E$ around $m$ then induces a graded vector space isomorphism $\E_{m} \cong K$. This justifies the name ``typical fiber'' for $K$. 

The graded vector space $\E_{m}$ can be written as a tensor product of two graded $\C^{\infty}_{\M,m}$-modules:
\begin{equation}
\E_{m} = \R \otimes_{\C^{\infty}_{\M,m}} \Gamma_{\E,m},
\end{equation}
where the action of $\C^{\infty}_{\M,m}$ on $\R$ is defined by (\ref{eq_StalkactiononR}), whereas the $\Gamma_{\E,m}$ is a graded $\C^{\infty}_{\M,m}$-module by Proposition \ref{tvrz_shofAmodulesstalks}. For any $U \in \Op_{m}(M)$ and any $\sigma \in \Gamma_{\E}(U)$, the pullback section $\sigma^{!}$ corresponds to the element $\sigma_{m} := 1 \otimes [\sigma]_{m} \in \E_{m}$, called the \textbf{value of the section $\sigma$ at $m$}. Note that for each element $e \in \E_{m}$, and any $U \in \Op_{m}(M)$, there exists some section $\sigma \in \Gamma_{U}(E)$, such that $e = \sigma_{m}$.

Let us employ the universal property. We have the induced graded vector bundle morphism $\hat{\kappa}_{m}: \E_{m} \rightarrow \E$ over $\kappa_{m}$. Now, suppose that $\E'$ is a graded vector bundle over a graded manifold $\cN$ and we are given a graded vector bundle morphism $F: \E \rightarrow \E'$ over a graded smooth map $\phi: \M \rightarrow \cN$. Consider the composition $F \circ \hat{\kappa}_{m}: \E_{m} \rightarrow \E'$. This is a graded vector bundle morphism over a constant mapping $\kappa_{\ul{\phi}(m)}: \{ \ast \} \rightarrow N$ valued at $\ul{\phi}(m)$. 

By property $(iii)$ in Proposition \ref{tvrz_pullbackVB}, there is a unique graded vector bundle morphism $F_{m}: \E_{m} \rightarrow \E'_{\ul{\phi}(m)}$, such that $\hat{\kappa}_{\ul{\phi}(m)} \circ F_{m} = F \circ \hat{\kappa}_{m}$. We have just shown that for every $m \in M$, any graded vector bundle morphism induces the graded linear map of the fibers $\E_{m}$ and $\E'_{\ul{\phi}(m)}$. 
\end{example}
\begin{example} \label{ex_fiberoftangentistangent}
Let $\M$ be any graded manifold. For each $m \in M$, the fiber $(T\M)_{m}$ of the tangent bundle $T\M$ can be canonically identified with the tangent space $T_{m}\M$.  

Indeed, for each $\lambda \in \R$ and $[X]_{m} \in \X_{M,m}$ represented by $X \in \X_{\M}(U)$ for some $U \in \Op_{m}(M)$, define a graded linear map $\psi: (T\M)_{m} \rightarrow T_{m}\M$ by
\begin{equation}
\psi(\lambda \otimes [X]_{m}) := \lambda X_{m},
\end{equation}
where $X_{m}$ is the value of a vector field $X$ at $m$ defined in Proposition \ref{tvrz_valueofvf}. It is easy to check that this is a well-defined graded linear map from $\R \otimes_{\C^{\infty}_{\M,m}} \X_{\M,m}$ to $T_{m}\M$. Let us prove that $\psi$ is surjective. Let $v \in T_{m} \M$ be a given tangent vector. Let $U \in \Op_{m}(M)$ be arbitrary. By Proposition \ref{tvrz_valueofvf}, there is a vector field $X \in \X_{\M}(U)$ such that $X_{m} = v$. Thus $\psi(1 \otimes [X]_{m}) = v$. 
As $(T\M)_{m}$ and $T_{m}\M$ have the same graded dimension, this proves that $\psi$ is a graded isomorphism. 

Finally, note that with respect to this isomorphism, the value of a section $X \in \X_{\M}(U)$ at $m \in U$ coincides with the value of the vector field $X_{m}$ defined in Proposition \ref{tvrz_valueofvf}. 
\end{example}
Next, every graded vector bundle $\E$ over a graded manifold $\M = (M,\C^{\infty}_{\M})$ induces a canonical graded vector bundle over its underlying manifold $M$. 
\begin{example} \label{ex_pullbackbyiM}
Let $i_{M}: M \rightarrow \M$ be a graded smooth map constructed in Proposition \ref{tvrz_bodymap}. We can thus form a pullback $\ul{\E} := i^{\ast}_{M} \E$, a graded vector bundle over $M$ of the same graded rank as $\E$. 

For each $k \in \Z$, we may view the $k$-th component sheaf $(\Gamma_{\ul{\E}})_{k}$ as a sheaf of graded $\C^{\infty}_{M}$-submodules of $\Gamma_{\ul{\E}}$. For any local trivialization chart $(U,\varphi)$ for $\ul{\E}$, one has $\varphi( (\Gamma_{\ul{\E}})_{k}|_{U}) = \C^{\infty}_{M}|_{U}[K_{k}]$. This proves that $(\Gamma_{\ul{\E}})_{k}$ defines a subbundle $\ul{\E}_{k} \subseteq \ul{\E}$ with a typical fiber $K_{k}$. One can view each of these subbundles as an ordinary vector bundle. In particular, the total space of the vector bundle $\ul{\E}_{0}$ will serve as a underlying manifold for the total space of $\E$, which will be constructed later. 
\end{example}

\begin{example} 
Let $\M = (M,\C^{\infty}_{\M})$ be any graded manifold. Then there is a vector bundle isomorphism $\ul{T\M}_{0} \cong TM$ over $\1_{M}$.

Let $U \in \Op(M)$ and let $X \in \X_{\M}(U)_{0}$ be any degree zero vector field. We claim that there is an associated \textbf{underlying vector field} $\ul{X} \in \X_{M}(U)$. We define it by the formula
\begin{equation}
\ul{X}( \ul{f}) := \ul{X(f)},
\end{equation}
for any $f \in \C^{\infty}_{\M}(U)_{0}$. As the body map is surjective by Proposition \ref{tvrz_bodymapsurj}, this determines $\ul{X}$. It is well defined, only if $\ul{f} = 0$ implies $\ul{X(f)} = 0$. It suffices to prove this locally, which is straightforward. It is obvious that it satisfies the Leibniz rule, hence it defines a vector field $\ul{X} \in \X_{M}(U)$. One can also argue that for any $U' \in \Op(U)$, one has $\ul{X}|_{U'} = \ul{X|_{U'}}$. Moreover, for every $f \in \C^{\infty}_{\M}(U)_{0}$, one finds $\ul{f X} = \ul{f} \ul{X}$. 

To finish the proof, we will construct a $\C^{\infty}_{M}$-linear sheaf isomorphism $\psi: \Gamma_{\ul{T\M}_{0}} \rightarrow \X_{M} \equiv \Gamma_{TM}$. For each $U \in \Op(U)$ and any $X \in \X_{\M}(U)_{0}$, we have the pullback section $X^{!} \in \Gamma_{\ul{T\M}}(U)_{0} \equiv \Gamma_{\ul{T\M}_{0}}(U)$. Define $\psi_{U}(X^{!}) := \ul{X}$. It is a straightforward exercise to prove that this defines a $\C^{\infty}_{M}$-linear sheaf isomorphism $\psi$. 
\end{example}

\begin{rem} \label{rem_VBunjinak}
Pullback graded vector bundle provides another useful viewpoint on the category $\gVBun^{\infty}$. Let $F: \E \rightarrow \E'$ and $G: \E' \rightarrow \E''$ be a pair of graded vector bundle morphisms over $\phi: \M \rightarrow \cN$ and $\psi: \cN \rightarrow \cS$, respectively. 

By Proposition \ref{tvrz_pullbackVB}-$(iii)$, they correspond uniquely to a pair of graded vector bundle morphisms $\hat{F}: \E \rightarrow \phi^{\ast} \E'$ and $\hat{G}: \E' \rightarrow \psi^{\ast} \E''$ over $\1_{\M}$ and $\1_{\cN}$, respectively. By Remark \ref{rem_VBmorphismsaremorphisms}, they can be viewed as $\hat{F} \in \Sh_{0}^{\C^{\infty}_{\M}}( \Gamma_{\E}, \phi^{\ast} \Gamma_{\E'})$ and $\hat{G} \in \Sh_{0}^{\C^{\infty}_{\cN}}(\E', \psi^{\ast} \E'')$. We have the functor $\phi^{\ast}: \Sh^{\C^{\infty}_{\cN}}(N,\gVect) \rightarrow \Sh^{\C^{\infty}_{\M}}(M, \gVect)$, see (\ref{eq_pullbackfunctor}). In particular, one finds
\begin{equation}
\phi^{\ast}\hat{G}: \Sh_{0}^{\C^{\infty}_{\M}}( \phi^{\ast}\Gamma_{\E'}, \phi^{\ast}(\psi^{\ast} \Gamma_{\E''}) \cong \Sh_{0}^{\C^{\infty}_{\M}}( \phi^{\ast} \Gamma_{\E'}, (\psi \circ \phi)^{\ast} \Gamma_{\E''}),
\end{equation} 
where we have used the identification in Proposition \ref{tvrz_pullbackVB}-$(v)$. 

One can thus form the composition $\phi^{\ast} \hat{G} \circ \hat{F} \in \Sh^{\C^{\infty}_{\M}}_{0}( \Gamma_{\E}, (\psi \circ \phi)^{\ast} \Gamma_{\E''})$, which by Remark \ref{rem_VBmorphismsaremorphisms} and Proposition \ref{tvrz_pullbackVB}-$(iii)$ corresponds to the unique graded vector bundle morphism from $\E$ to $\E''$ over $\psi \circ \varphi$. Not surprisingly, this is precisely the composition $G \circ F$, defined by (\ref{eq_gvbmorfcomp}). The actual proof of this claim just requires one to track the proof of Proposition \ref{tvrz_ap_pullbackVB} and we omit it here. 
\end{rem}
\subsection{Sums, products, and degree shifts}
In this subsection, we will provide a few standard constructions of new graded vector bundles. 
\begin{tvrz}
Let $\E,\E'$ be a pair of graded vector bundles over a graded manifold $\M = (M,\C^{\infty}_{\M})$. Let $K$ and $K'$ be a typical fiber of $\E$ and $\E'$, respectively.

Then the direct sum $\Gamma_{\E} \oplus \Gamma_{\E'}$ defines a locally freely and finitely generated sheaf of graded $\C^{\infty}_{\M}$-modules of a constant graded rank $\grk(\E) + \grk(\E')$. It thus defines a graded vector bundle $\E \oplus \E'$ over $\M$ with $\Gamma_{\E \oplus \E'} := \Gamma_{\E} \oplus \Gamma_{\E'}$, called the \textbf{Whitney sum of $\E$ and $\E'$}. Its typical fiber is $K \oplus K'$ and for any graded smooth map $\phi: \cN \rightarrow \M$, one has the identifications
\begin{equation} \label{eq_dualpullbacksum}
(\E \oplus \E)^{\ast} \cong \E^{\ast} \oplus \E'^{\ast}, \; \; \phi^{\ast}(\E \oplus \E') \cong \phi^{\ast}\E \oplus \phi^{\ast}\E'.
\end{equation}
In particular, for every $m \in M$, one has $(\E \oplus \E')_{m} = \E_{m} \oplus \E'_{m}$. 
\end{tvrz}
\begin{proof}
For each $m \in M$, one can choose a local trivialization chart $(U,\varphi)$ and $(U,\varphi')$ for $\E$ and $\E'$ around $m$, respectively. An obvious $\C^{\infty}_{\M}|_{U}$-linear sheaf isomorphism 
\begin{equation}
\varphi \oplus \varphi' : \Gamma_{\E \oplus \E'}|_{U} \rightarrow \C^{\infty}_{\M}|_{U}[K \oplus K'] \cong \C^{\infty}_{\M}|_{U}[K] \oplus \C^{\infty}_{\M}|_{U}[K']
\end{equation}
defines a local trivialization chart $(U, \varphi \oplus \varphi')$ for $\E \oplus \E'$ around $m$. The claim about duals in (\ref{eq_dualpullbacksum}) is obvious. The claim about pullbacks can be easily proved using the universal property.
\end{proof}

\begin{tvrz}
Let $\E,\E'$ be a pair of graded vector bundles over a graded manifold $\M = (M,\C^{\infty}_{\M})$. Let $K$ and $K'$ be a typical fiber of $\E$ and $\E'$, respectively.

Then the tensor product $\Gamma_{\E} \otimes_{\C^{\infty}_{\M}} \Gamma_{\E'}$ defines a locally freely and finitely generated sheaf of graded $\C^{\infty}_{\M}$-modules of a constant graded rank. It thus defines a graded vector bundle $\E \otimes \E'$ over $\M$ with $\Gamma_{\E \otimes \E'} := \Gamma_{\E} \otimes_{\C^{\infty}_{\M}} \Gamma_{\E'}$, called the \textbf{tensor product of $\E$ and $\E'$}. Its typical fiber is $K \otimes_{\R} K'$ and for any graded smooth map $\phi: \cN \rightarrow \M$, one has the identifications 
\begin{equation} \label{eq_dualpullbacktproduct}
(\E \otimes \E')^{\ast} \cong \E^{\ast} \otimes \E'^{\ast}, \; \; \phi^{\ast}(\E \otimes \E') \cong \phi^{\ast} \E \otimes \phi^{\ast}\E'. 
\end{equation}
In particular, for every $m \in M$, one has $(\E \otimes \E')_{m} = \E_{m} \otimes_{\R} \E'_{m}$.
\end{tvrz}
\begin{proof}
First, let us recall how the tensor product $\Gamma_{\E} \otimes_{\C^{\infty}_{\M}} \Gamma_{\E'}$ is defined. This is a general construction, see also Remark \ref{rem_ulShcat}. For each $U \in \Op(M)$, consider a graded $\C^{\infty}_{\M}(U)$-module
\begin{equation} 
(\Gamma_{\E} \otimes^{P}_{\C^{\infty}_{\M}} \Gamma_{\E'})(U) := \Gamma_{\E}(U) \otimes_{\C^{\infty}_{\M}(U)} \Gamma_{\E'}(U).
\end{equation} 
It is easy to see that defines a presheaf $\Gamma_{\E} \otimes_{\C^{\infty}_{\M}}^{P} \Gamma_{\E'}$ of graded $\C^{\infty}_{\M}$-modules. The actual tensor product is then defined as a sheafification of this presheaf:
\begin{equation}
\Gamma_{\E} \otimes_{\C^{\infty}_{\M}} \Gamma_{\E'} := \Sff( \Gamma_{\E} \otimes^{P}_{\C^{\infty}_{\M}} \Gamma_{\E'}).  
\end{equation}
Observe that $\C^{\infty}_{\M}[K] \otimes_{\C^{\infty}_{\M}} \C^{\infty}_{\M}[K'] \cong \C^{\infty}_{\M}[K \otimes_{\R} K']$. This is true already before the sheafification, that is there is acanonical $\C^{\infty}_{\M}$-linear presheaf isomorphism 
\begin{equation} \label{eq_pretensoroffreemodules}
\C^{\infty}_{\M}[K] \otimes_{\C^{\infty}_{\M}}^{P} \C^{\infty}_{\M}[K'] \cong \C^{\infty}_{\M}[K \otimes_{\R} K'].
\end{equation}
We leave the details for the reader. Now, let $m \in M$ be arbitrary and let $(U,\varphi)$ and $(U,\varphi')$ be local trivialization charts for $\E$ and $\E'$ around $m$, respectively. Let 
\begin{equation}
\varphi \otimes^{P} \varphi': \Gamma_{\E}|_{U} \otimes_{\C^{\infty}_{\M}|_{U}}^{P} \Gamma_{\E'}|_{U} \rightarrow \C^{\infty}_{\M}|_{U}[K] \otimes_{\C^{\infty}_{\M}|_{U}}^{P} \C^{\infty}_{\M}|_{U}[K']
\end{equation}
be the $\C^{\infty}_{\M}|_{U}$-linear presheaf isomorphism given for each $V \in \Op(U)$, $\sigma \in \Gamma_{\E}(V)$, $\sigma' \in \Gamma_{\E'}(V)$ by 
\begin{equation}
(\varphi \otimes^{P} \varphi')_{V}( \sigma \otimes \sigma') := (-1)^{|\varphi'||\sigma|} \varphi_{V}(\sigma) \otimes \varphi'_{V}(\sigma').
\end{equation}
One has to verify that it is well-defined and natural in $V$, but that is straightforward. By the universal property of the sheafification, it induces a $\C^{\infty}_{\M}|_{U}$-linear sheaf isomorphism
\begin{equation}
\varphi \otimes \varphi': \Gamma_{\E}|_{U} \otimes_{\C^{\infty}_{\M}|_{U}} \Gamma_{\E'}|_{U} \rightarrow \C^{\infty}_{\M}|_{U}[K] \otimes_{\C^{\infty}_{\M}|_{U}}^{P} \C^{\infty}_{\M}|_{U}[K'].
\end{equation}
Using (\ref{eq_pretensoroffreemodules}) and the canonical identification $(\Gamma_{\E} \otimes_{\C^{\infty}_{\M}} \Gamma_{\E'})|_{U} \cong \Gamma_{\E}|_{U} \otimes_{\C^{\infty}_{\M}|_{U}} \Gamma_{\E'}|_{U}$, we can view it as $\C^{\infty}_{\M}|_{U}$-linear sheaf isomorphism $\varphi \otimes \varphi': \Gamma_{\E \otimes \E'}|_{U} \rightarrow \C^{\infty}_{\M}|_{U}[K \otimes_{\R} K']$. We have just proved that $\Gamma_{\E \otimes \E'}$ is locally freely and finitely generated and it has a constant graded rank $\gdim(K \otimes K')$. It thus defines a graded vector bundle $\E \otimes \E'$. 

Next, the proof of the identification of duals in (\ref{eq_dualpullbacktproduct}) is a bit tedious as one has to work with sheafifications and their universal property. However, it is straightforward and we omit the proof. The identification of pullbacks in (\ref{eq_dualpullbacktproduct}) is an easy application of the Yoneda lemma\footnote{In an unlikely case of the reader's dissatisfaction with this hint, see Proposition \ref{tvrz_ap_pullbacktensor} in the appendix.}.
\end{proof}

It turns out that the category of graded vector bundles has binary products.

\begin{tvrz}
Let $\M$ and $\cN$ be a pair of graded manifolds. Let $\E$ be a graded vector bundle over $\M$ and $\E'$ a graded vector bundle over $\cN$. Let $K$ and $K'$ be a typical fiber of $\E$ and $\E'$, respectively. Let $\pi_{\M}: \M \times \cN \rightarrow \M$ and $\pi_{\cN}: \M \times \cN \rightarrow \cN$ be the canonical projections.

Then the graded vector bundle $\E \times \E' := \pi_{\M}^{\ast}\E \oplus \pi_{\cN}^{\ast}\E'$ over $\M \times \cN$ defines a product of $\E$ and $\E'$ in the category $\gVBun^{\infty}$. For each $(m,n) \in M \times N$, there is a canonical identification
\begin{equation} \label{eq_productfibers}
(\E \times \E')_{(m,n)} \cong \E_{m} \oplus \E'_{n}.
\end{equation}
\end{tvrz}
\begin{proof}
First, we need to construct a pair of graded vector bundle morphisms $P_{\E}: \E \times \E' \rightarrow \E$ and $P_{\E'}: \E \times \E' \rightarrow \E'$ over $\pi_{\M}$ and $\pi_{\cN}$, respectively. In view of Remark \ref{rem_VBunjinak}, it suffices to define two $\C^{\infty}_{\M \times \cN}$-linear sheaf morphisms 
\begin{equation} \hat{P}_{\E}: \pi_{\M}^{\ast} \Gamma_{\E} \oplus \pi_{\cN}^{\ast} \Gamma_{\E'} \rightarrow \pi_{\M}^{\ast} \Gamma_{\E}, \; \; \hat{P}_{\E'}: \pi_{\M}^{\ast} \Gamma_{\E} \oplus \pi_{\cN}^{\ast} \Gamma_{\E'} \rightarrow \pi_{\cN} \Gamma_{E'}. 
\end{equation}
Clearly, we choose $\hat{P}_{\E}$ and $\hat{P}_{\E'}$ as projections on the first and second summand, respectively. Let us prove the universal property of the product. Let $\G$ be any graded vector bundle over $\cS$ and let $F_{1}: \G \rightarrow \E$ and $F_{2}: \G \rightarrow \E'$ be a pair of graded vector bundle morphisms over $\phi_{1}: \cS \rightarrow \M$ and $\phi_{2}: \cS \rightarrow \cN$, respectively. We must argue that there is a unique graded vector bundle morphism $F: \G \rightarrow \E \times \E'$, such that $P_{\E} \circ F = F_{1}$ and $P_{\E'} \circ F = F_{2}$. 

By the universality of the product $\M \times \cN$, $F$ must be a graded vector bundle morphism over $(\phi_{1},\phi_{2}): \cS \rightarrow \M \times \cN$, see Proposition \ref{tvrz_products}. In view of Remark \ref{rem_VBunjinak}, it suffices to define a $\C^{\infty}_{\cS}$-linear sheaf morphism 
\begin{equation}
\hat{F}: \Gamma_{\G} \rightarrow (\phi_{1},\phi_{2})^{\ast} \Gamma_{\E \times \E'} \equiv (\phi_{1},\phi_{2})^{\ast}( \pi_{\M}^{\ast} \Gamma_{\E} \oplus \pi_{\cN}^{\ast} \Gamma_{\E'}) \cong \phi_{1}^{\ast} \Gamma_{\E} \oplus \phi_{2}^{\ast} \Gamma_{\E'},
\end{equation}
such that $(\phi_{1},\phi_{2})^{\ast} \hat{P}_{\E} \circ \hat{F} = \hat{F}_{1}$ and $(\phi_{1},\phi_{2})^{\ast} \hat{\P}_{\E'} \circ \hat{F} = \hat{F}_{2}$. It is not difficult to see that $(\phi_{1},\phi_{2})^{\ast} \hat{P}_{\E}: \phi_{1}^{\ast} \Gamma_{\E} \oplus \phi_{2}^{\ast} \Gamma_{\E'} \rightarrow (\phi_{1},\phi_{2})^{\ast} (\pi_{\M}^{\ast} \Gamma_{\E}) \cong \phi_{1}^{\ast} \Gamma_{\E}$ is just the projection onto the first summand. Similarly, $(\phi_{1},\phi_{2})^{\ast} \hat{P}_{\E'}$ is the projection onto the second summand. This means that for each $U \in \Op(S)$ and $\sigma \in \Gamma_{\G}(U)$, $\hat{F}_{U}$ has to be given by the formula
\begin{equation}
\hat{F}_{U}(\sigma) = ( (\hat{F}_{1})_{U}(\sigma), (\hat{F}_{2})_{U}(\sigma)).
\end{equation}
Clearly, this defines a $\C^{\infty}_{\cS}$-linear sheaf morphism $\hat{F}: \Gamma_{\G} \rightarrow \phi_{1}^{\ast} \Gamma_{\E} \oplus \phi^{\ast}_{2} \Gamma_{\E'}$. The statement (\ref{eq_productfibers}) follows from the fact that the constant mapping $\kappa_{(m,n)}: \{ \ast \} \rightarrow \M \times \cN$ satisfies $\pi_{\M} \circ \kappa_{(m,n)} = \kappa_{m}$ and $\pi_{\cN} \circ \kappa_{(m,n)} = \kappa_{n}$. Indeed, (\ref{eq_dualpullbacksum}) together with Proposition \ref{tvrz_pullbackVB}-$(v)$ imply
\begin{equation}
\begin{split}
(\E \times \E')_{(m,n)} \equiv & \ \kappa_{(m,n)}^{\ast}( \pi_{\M}^{\ast} \E \oplus \pi_{\cN}^{\ast} \E') \cong \kappa_{(m,n)}^{\ast}( \pi_{\M}^{\ast} \E) \oplus \kappa_{(m,n)}^{\ast}( \pi_{\cN}^{\ast} \E') \\
\cong & \  \kappa_{m}^{\ast}\E \oplus \kappa_{n}^{\ast} \E' \equiv \E_{m} \oplus \E'_{n}. 
\end{split}
\end{equation}
This finishes the proof. 
\end{proof}
Furthermore, in the graded setting, one has the following algebraic construction with no analogue in ordinary differential geometry. 
\begin{tvrz} \label{tvrz_degreeshiftedvector}
Let $\E$ be a graded vector bundle over a graded manifold $\M = (M, \C^{\infty}_{\M})$. Let $K$ be a typical fiber of $\E$. Let $\ell \in \Z$ be an arbitrary integer. 

Then the degree shifted sheaf $\Gamma_{\E}[\ell]$ (see Remark \ref{rem_shiftedmodule} and Remark \ref{rem_degreeshiftedsheafofmodules}) is locally freely and finitely generated sheaf of graded $\C^{\infty}_{\M}$-modules of a constant graded rank. It thus defines a graded vector bundle $\E[\ell]$ over $\M$ with $\Gamma_{\E[\ell]} := \Gamma_{\E}[\ell]$, called the \textbf{degree shift of a graded vector bundle $\E$}. Its typical fiber is the degree shifted graded vector space $K[\ell]$ and for any graded smooth map $\phi: \cN \rightarrow \M$, one has the identifications
\begin{equation} \label{eq_shiftedifentifications}
(\E[\ell])^{\ast} \cong \E^{\ast}[-\ell], \; \; (\phi^{\ast} \E)[\ell] \cong \phi^{\ast}(\E[\ell]). 
\end{equation}
In particular, for every $m \in M$, one has $(\E[\ell])_{m} \cong \E_{m}[\ell]$. 
\end{tvrz}
\begin{proof}
For each $\F \in \Sh^{\C^{\infty}_{\M}}(M,\gVect)$, we have a canonical \textbf{degree shifting operator} $\delta_{\F}[\ell] \in \Sh_{\ell}^{\C^{\infty}_{\M}}( \F[\ell], \F)$, defined for each $U \in \Op(M)$ as the degree shifting operator $(\delta_{\F}[\ell])_{U}: \F(U)[\ell] \rightarrow \F(U)$ defined in Remark \ref{rem_shiftedmodule}. By construction, it is $\C^{\infty}_{\M}(U)$-linear and natural in $U$. 

Let $\F,\G \in \Sh^{\C^{\infty}_{\M}}(M, \gVect)$. Let $\varphi \in \Sh^{\C^{\infty}_{\M}}_{k}(\F,\G)$. There is always a unique map $\varphi[\ell] \in \Sh^{\C^{\infty}_{\M}}_{k}( \F[\ell],\G[\ell])$ making the following diagram commutative:
\begin{equation} \label{eq_varphiellshifted}
\begin{tikzcd}
\F[\ell] \arrow{d}{\delta_{\F}[\ell]} \arrow[dashed]{r}{\varphi[\ell]} & \G[\ell] \arrow{d}{\delta_{\G}[\ell]} \\
\F \arrow{r}{\varphi}& \G.
\end{tikzcd}
\end{equation}
In particular, this shows that one can view the assignment $\F \mapsto \F[\ell]$ as a functor. Note that for each $j \in \Z$, the component $\varphi[\ell]_{j}: \F[\ell]_{j} \rightarrow \G[\ell]_{j}$ is just the component $\varphi_{j+\ell}: \F_{j+\ell} \rightarrow \G_{j+\ell}$.

Next, observe that there is a canonical $\C^{\infty}_{\M}$-linear sheaf isomorphism 
\begin{equation} \label{eq_psiKiso} \psi^{K}: (\C^{\infty}_{\M}[K])[\ell] \rightarrow \C^{\infty}_{\M}[K[\ell]]. \end{equation}
Let $(\vartheta_{\lambda})_{\lambda =1}^{m}$ be a total basis for $K$. For each $U \in \Op(X)$ and every section $s = f^{\lambda} \otimes \vartheta_{\lambda} \in ((\C^{\infty}_{\M}[K])[\ell])(U) \equiv (\C^{\infty}_{\M}(U) \otimes_{\R} K)[\ell]$, set $\psi^{K}_{U}(s) := (-1)^{|s|'\ell} f^{\lambda} \otimes \vartheta_{\lambda} \in (\C^{\infty}_{\M}(U) \otimes_{\R} K[\ell])$, where $|s|'$ denotes the degree of $s$ in the degree shifted graded $\C^{\infty}_{\M}(U)$-module, and on the right-hand side, each $k \in K$ of degree $|k|$ is viewed as an element of $K[\ell]$ of degree $|k'| = |k|-\ell$. For each $\lambda \in \{1,\dots,m\}$, one has  $|f^{\lambda}| + |\vartheta_{\lambda}| = |s|' + \ell$. Also note that $|\psi_{U}^{K}(s)| = |f^{\lambda}| + |\vartheta_{\lambda}|' = |s|'$.

The sign is important for $\psi^{K}_{U}$ to be $\C^{\infty}_{\M}$-linear, as the action $\tr'$ of $\C^{\infty}_{\M}(U)$ on the degree shifted graded module is given by (\ref{eq_shiftedmodule}). For each $g \in \C^{\infty}_{\M}(U)$, we have
\begin{equation}
\begin{split}
\psi_{U}^{K}( g \tr' (f^{\lambda} \otimes \vartheta_{\lambda})) = & \  \psi^{K}_{U}( (-1)^{|g|\ell} (g \cdot f^{\lambda}) \otimes \vartheta_{\lambda}) \\
= & \ (-1)^{|g|\ell} (-1)^{(|g| + |s|')\ell} (g \cdot f^{\lambda}) \otimes \vartheta_{\lambda}  \\
= & \ g \tr ( (-1)^{|s|'\ell} f^{\lambda} \otimes \vartheta_{\lambda} ) \\
= & \ g \tr \psi_{U}^{K}( f^{\lambda} \otimes \vartheta_{\lambda}). 
\end{split}
\end{equation}
$\psi_{U}^{K}$ is clearly invertible and natural in $U$, hence it defines a $\C^{\infty}_{\M}$-linear sheaf isomorphism $\psi^{K}$. 

Let us get back to graded vector bundles. Let $(U,\varphi)$ be a local trivialization chart for $\E$. One can consider the composition $\varphi' := \psi^{K}|_{U} \circ \varphi[\ell]: \Gamma_{\E}[\ell]|_{U} \rightarrow \C^{\infty}_{\M}|_{U}[K[\ell]]$. This proves that $\Gamma_{\E}[\ell]$ is locally finitely and freely generated of a constant graded rank $\gdim(K[\ell])$, hence it defines a graded vector bundle $\E[\ell]$ with $\Gamma_{\E[\ell]} := \Gamma_{\E}[\ell]$, and $(U,\varphi')$ becomes the local trivialization chart for $\E[\ell]$. Its typical fiber is indeed $K[\ell]$. 

The first of the two identifications in (\ref{eq_shiftedifentifications}) is an immediate consequence of Remark \ref{rem_uShdegshifted}. The second identification is true for any $\F \in \Sh^{\C^{\infty}_{\M}}(M, \gVect)$, that is there is a canonical isomorphism $(\phi^{\ast} \F)[\ell] = \phi^{\ast}( \F[\ell])$. Indeed, for any $\G \in \Sh^{\C^{\infty}_{\cN}}(N, \gVect)$, one can write 
\begin{equation}
\begin{split}
\Sh_{0}^{\C^{\infty}_{\cN}}( (\phi^{\ast} \F)[\ell], \G) \cong & \ \Sh_{-\ell}^{\C^{\infty}_{\cN}}( \phi^{\ast} \F, \G) \cong \Sh^{\C^{\infty}_{\cN}}_{-\ell}( \F, \phi_{\ast} \G) \\
\cong & \ \Sh^{\C^{\infty}_{\cN}}_{0}( \F[\ell], \phi_{\ast} \G) \cong \Sh^{\C^{\infty}_{\cN}}_{0}( \phi^{\ast}( \F[\ell]), \G),
\end{split}
\end{equation}
where all bijections are natural in $\G$ (in fact, they are natural also in $\F$). We have used Remark \ref{rem_uShdegshifted} together with (\ref{eq_uShpulllpush}). By applying the Yoneda lemma in the same way as in the proof of Proposition \ref{tvrz_ap_pullbacktensor}, we obtain the canonical $\C^{\infty}_{\cN}$-linear sheaf isomorphism $(\phi^{\ast} \F)[\ell] \cong \phi^{\ast}(\F[\ell])$. For $\F = \Gamma_{\E}$, we obtain precisely the second identification in (\ref{eq_shiftedifentifications}), and the proof is finished. 
\end{proof}
\subsection{Transition maps and functions}
Let $\E$ be a graded vector bundle over a graded manifold $\M = (M,\C^{\infty}_{\M})$ with a typical fiber $K$. 

By definition, we can find a collection $\{ (U_{\alpha}, \varphi_{\alpha}) \}_{\alpha \in I}$ of local trivialization charts, where $\{ U_{\alpha} \}_{\alpha \in I}$ is an open cover of $M$ and for each $\alpha \in I$, $\varphi_{\alpha}: \Gamma_{\E}|_{U_{\alpha}} \rightarrow \C^{\infty}_{\M}|_{U_{\alpha}}[K]$ is a $\C^{\infty}_{\M}|_{U_{\alpha}}$-linear sheaf isomorphism. Every such collection is called a \textbf{local trivialization of $\E$}.

For each $(\alpha,\beta) \in I^{2}$, define the $\C^{\infty}_{\M}|_{U_{\alpha \beta}}$-linear sheaf isomorphism 
\begin{equation}
G_{\alpha \beta} := \varphi_{\alpha}|_{U_{\alpha \beta}} \circ \varphi_{\beta}^{-1}|_{U_{\alpha \beta}}: \C^{\infty}_{\M}|_{U_{\alpha \beta}}[K] \rightarrow \C^{\infty}_{\M}|_{U_{\alpha \beta}}[K].
\end{equation}
$G_{\alpha \beta}$ are called the \textbf{transition maps} corresponding to the local trivialization $\{ (U_{\alpha},\varphi_{\alpha}) \}_{\alpha \in I}$. By construction, one finds the condition
\begin{equation} \label{eq_VBtransmapcocycle}
G_{\alpha \beta} \circ G_{\beta \gamma} = G_{\alpha \gamma},
\end{equation}
when restricted to $U_{\alpha \beta \gamma}$, for all $(\alpha,\beta,\gamma) \in I^{3}$. In particular, one has $G_{\alpha \alpha} = \1_{\C^{\infty}_{\M}|_{U_{\alpha}}[K]}$ and $G_{\beta \alpha} = G_{\alpha \beta}^{-1}$ for all $(\alpha,\beta) \in I^{2}$. 

Fix a total basis $(\vartheta_{\lambda})_{\lambda=1}^{m}$. For each $\lambda \in \{1,\dots,m\}$ One can now write 
\begin{equation}
(G_{\alpha \beta})_{U_{\alpha \beta}}(1 \otimes \vartheta_{\lambda}) = (G_{\alpha \beta})^{\kappa}{}_{\lambda} \otimes \vartheta_{\kappa},
\end{equation}
for unique \textbf{transition functions} $(G_{\alpha \beta})^{\kappa}{}_{\lambda} \in \C^{\infty}_{\M}(U_{\alpha \beta})$. They completely determine $G_{\alpha \beta}$. Indeed, for any $V \in \Op(U_{\alpha \beta})$ and any $f^{\lambda} \otimes \vartheta_{\lambda} \in \C^{\infty}_{\M}(V) \otimes_{\R} K$, one has 
\begin{equation} \label{eq_transmapusingtransff}
\begin{split}
(G_{\alpha \beta})_{V}(f^{\lambda} \otimes \vartheta_{\lambda}) = & \  f^{\lambda} \tr (G_{\alpha \beta})_{V}( 1 \otimes \vartheta_{\lambda}) = f^{\lambda} \tr (G_{\alpha \beta})_{U_{\alpha \beta}}(1 \otimes \vartheta_{\lambda})|_{V} \\
= & \ f^{\lambda} \tr ((G_{\alpha \beta})^{\kappa}{}_{\lambda}|_{V} \otimes \vartheta_{\kappa}) \\
= & \ (f^{\lambda} \cdot (G_{\alpha \beta})^{\kappa}{}_{\lambda}|_{V}) \otimes \vartheta_{\kappa}.
\end{split}
\end{equation}
When restricted to $U_{\alpha \beta \gamma}$, for all $(\alpha,\beta,\gamma) \in I^{3}$, they satisfy the \textbf{cocycle condition}:
\begin{equation} \label{eq_VBtransfunctcocycle}
(G_{\beta \gamma})^{\rho}{}_{\lambda} \cdot (G_{\alpha \beta})^{\kappa}{}_{\rho} = (G_{\alpha \gamma})^{\kappa}{}_{\lambda},
\end{equation}
for all $\lambda,\kappa \in \{1,\dots,m\}$. Note that the order is important as the functions do not commute. Note that one has for each $(\alpha,\beta) \in I^{2}$ and $\lambda,\kappa \in \{1,\dots,m\}$, one has 
\begin{equation} \label{eq_VBtransfunctother}
|(G_{\alpha \beta})^{\kappa}{}_{\lambda}| = |\vartheta_{\lambda}| - |\vartheta_{\kappa}|, \; \; (G_{\alpha \alpha})^{\kappa}{}_{\lambda} = \delta^{\kappa}_{\lambda}. 
\end{equation}
As in the ordinary case, graded vector bundles can be uniquely (up to the isomorphism) recovered from their transition functions. 

\begin{tvrz}[\textbf{Collation of graded vector bundles}] \label{tvrz_collationVB}
Let $\M = (M, \C^{\infty}_{\M})$ be a graded manifold. Let $K$ be a finite-dimensional graded vector space and choose its total basis $(\vartheta_{\lambda})_{\lambda=1}^{m}$. 

Let $\{ U_{\alpha} \}_{\alpha \in I}$ be an open cover of $M$. Suppose that for each $(\alpha,\beta) \in I^{2}$ and $\lambda,\kappa \in \{1,\dots,m\}$, one has a function $(G_{\alpha \beta})^{\kappa}{}_{\lambda} \in \C^{\infty}_{\M}(U_{\alpha \beta})$, such that (\ref{eq_VBtransfunctcocycle}) and (\ref{eq_VBtransfunctother}) are satisfied.  

Then there is a graded vector bundle $\E$ together with a local trivialization $\{ (U_{\alpha}, \varphi_{\alpha}) \}_{\alpha \in I}$, such that the given collection of functions is the collections of its transition functions. If $\E'$ is another graded vector bundle together with a local trivialization $\{ (U'_{\alpha}, \varphi'_{\alpha}) \}_{\alpha \in I}$ having this property, there is a unique graded vector bundle isomorphism $F: \E \rightarrow \E'$ over $\1_{\M}$, such that $\varphi'_{\alpha} \circ F|_{U_{\alpha}} = \varphi_{\alpha}$. We interpret the local trivialization maps as morphisms of graded vector bundles, see Example \ref{ex_trivasvbunmorf}.
\end{tvrz}
\begin{proof}
For each $(\alpha,\beta) \in I^{2}$, one can use (\ref{eq_transmapusingtransff}) to construct a $\C^{\infty}_{\M}|_{U_{\alpha \beta}}$-linear sheaf isomorphism 
\begin{equation}
G_{\alpha \beta}: \C^{\infty}_{\M}|_{U_{\alpha \beta}}[K] \rightarrow \C^{\infty}_{\M}|_{U_{\alpha \beta}}[K].
\end{equation}
By assumption, their collection satisfies the cocycle condition (\ref{eq_VBtransmapcocycle}). 

For each $\alpha \in I$, consider a sheaf $\F_{\alpha} := \C^{\infty}_{\M}|_{U_{\alpha}}[K]$ of graded $\C^{\infty}_{\M}|_{U_{\alpha}}$-modules. For each $(\alpha,\beta) \in I^{2}$, we have a $\C^{\infty}_{\M}|_{U_{\alpha \beta}}$-linear sheaf isomorphism $\phi_{\alpha \beta} := G_{\beta \alpha}: \F_{\alpha}|_{U_{\alpha \beta}} \rightarrow \F_{\beta}|_{U_{\alpha \beta}}$. It follows from (\ref{eq_VBtransmapcocycle}) that their collection $\{ \phi_{\alpha \beta} \}_{(\alpha,\beta) \in I^{2}}$ satisfies the cocycle condition (\ref{eq_shcollcocycle}). Using the analogue of Proposition \ref{tvrz_shcollation1} for sheaves valued in the category $\gVect$, there is a sheaf $\F \in \Sh(M, \gVect)$ together with a collection $\{ \varphi_{\alpha} \}_{\alpha \in I}$, where $\varphi_{\alpha}: \F|_{U_{\alpha}} \rightarrow \F_{\alpha}$ are sheaf isomorphisms satisfying $\phi_{\alpha \beta} \circ \varphi_{\alpha}|_{U_{\alpha \beta}} = \varphi_{\beta}|_{U_{\alpha \beta}}$ for all $(\alpha,\beta) \in I^{2}$. In fact, one can show that $\F$ can be made into a sheaf of graded $\C^{\infty}_{\M}$-modules, such that for each $\alpha \in I$, $\varphi_{\alpha}: \F|_{U_{\alpha}} \rightarrow \F_{\alpha} \equiv \C^{\infty}_{\M}|_{U_{\alpha}}[K]$ is $\C^{\infty}_{\M}|_{U_{\alpha}}$-linear. 

In other words, it follows that $\F$ is locally finitely and freely generated of a constant graded rank, hence it defines a graded vector bundle $\E$ over $M$ with $\Gamma_{\E} := \F$. The collection $\{ (U_{\alpha}, \varphi_{\alpha}) \}_{\alpha \in I}$ becomes its local trivialization and $\varphi_{\alpha}|_{U_{\alpha \beta}} \circ \varphi_{\beta}^{-1}|_{U_{\alpha \beta}} = G_{\alpha \beta}$. 

Let $\E'$ be another graded vector together with a local trivialization $\{ (U_{\alpha}, \varphi'_{\alpha})\}_{\alpha \in I}$ having this property. By Proposition \ref{tvrz_shcollation1}, there is then a unique sheaf morphism $F: \Gamma_{\E} \rightarrow \Gamma_{\E'}$, such that $\varphi'_{\alpha} \circ F|_{U_{\alpha}} = \varphi_{\alpha}$ for all $\alpha \in I$. By Remark \ref{rem_VBmorphismsaremorphisms}, this corresponds to the unique graded vector bundle morphism $F: \E \rightarrow \E'$ over $\1_{\M}$, such that $\varphi'_{\alpha} \circ \F|_{U_{\alpha}} = \varphi_{\alpha}$ for all $\alpha \in I$. 
\end{proof}

If $( \Phi_{\lambda}^{(\alpha)} )_{\lambda =1}^{m}$ is the local frame corresponding to the local trivialization chart $(U_{\alpha}, \varphi_{\alpha})$ and the total basis $(\vartheta_{\lambda})_{\lambda=1}^{m}$ of $K$, transition functions can be used (or equivalently defined) to relate the local frames corresponding to different local trivialization charts and the same total basis:
\begin{equation} \label{eq_locframerel}
\Phi^{(\beta)}_{\lambda}|_{U_{\alpha \beta}} = (G_{\alpha \beta})^{\kappa}{}_{\lambda} \Phi_{\kappa}^{(\alpha)}|_{U_{\alpha \beta}},
\end{equation}
for all $(\alpha,\beta) \in I^{2}$. In the following subsection, we will make use of the transition functions corresponding to the induced local trivialization of the dual bundle.
\begin{tvrz} \label{tvrz_dualtrans}
Let $\E$ be a graded vector bundle over a graded manifold $\M = (M, \C^{\infty}_{\M})$ with a typical fiber $K$. Let $\{ (U_{\alpha},\varphi_{\alpha}) \}_{\alpha \in I}$ be a local trivialization of $\E$. Fix a total basis $(\vartheta_{\lambda})_{\lambda=1}^{m}$ of $K$. 

Let $\{ (U_{\alpha}, \varphi_{\alpha}^{\vee}) \}_{\alpha \in I}$ denote the induced local trivialization of the dual vector bundle $\E^{\ast}$, see Proposition \ref{tvrz_dualfinfree}. Let $G_{\alpha \beta}^{\vee}$ be the corresponding transition maps. Define the transition functions corresponding to the dual total basis $(\vartheta^{\lambda})_{\lambda=1}^{m}$ of $K^{\ast}$ as 
\begin{equation}
(G^{\vee}_{\alpha \beta})_{U_{\alpha \beta}}(1 \otimes \vartheta^{\lambda}) =: (G^{\vee}_{\alpha \beta})_{\kappa}{}^{\lambda} \otimes \vartheta^{\kappa},
\end{equation}
for all $(\alpha,\beta) \in I^{2}$ and $\lambda,\kappa \in \{1,\dots,m\}$. 

Then, with respect to the identification $\C^{\infty}_{\M}|_{U_{\alpha \beta}}[K^{\ast}] \cong (\C^{\infty}_{\M}|_{U_{\alpha \beta}}[K])^{\ast}$, one has $G_{\alpha \beta}^{\vee} = G_{\alpha \beta}^{-T}$. Moreover, the transition functions have the form
\begin{equation} \label{eq_GalphabetaVeecomps}
(G^{\vee}_{\alpha \beta})_{\kappa}{}^{\lambda} = (-1)^{|\vartheta_{\lambda}|( |\vartheta_{\kappa}| - |\vartheta_{\lambda}|)} (G_{\beta \alpha})^{\lambda}{}_{\kappa},
\end{equation}
where on the right-hand side, there are the transition functions corresponding to the local trivialization $\{ (U_{\alpha},\varphi_{\alpha}) \}_{\alpha \in I}$ of $\E$ and the total basis $(\vartheta_{\lambda})_{\lambda=1}^{m}$. 
\end{tvrz}
\begin{proof}
First, note that for any $U \in \Op(M)$, the identification $\C^{\infty}_{\M}|_{U}[K^{\ast}] \cong (\C^{\infty}_{\M}|_{U}[K])^{\ast}$ goes as follows. For any $V \in \Op(M)$, any element $f_{\lambda} \otimes \vartheta^{\lambda} \in \C^{\infty}_{\M}(V) \otimes_{\R} K^{\ast}$ can be viewed as graded $\C^{\infty}_{\M}(V)$-linear map from $\C^{\infty}_{\M}(V) \otimes K$ to $\C^{\infty}_{\M}(V)$, defined as
\begin{equation} \label{eq_CKastisCKast}
[f_{\lambda} \otimes \vartheta^{\lambda}](g^{\kappa} \otimes \vartheta_{\kappa}) := (-1)^{|\vartheta^{\lambda}||g^{\lambda}|} f_{\lambda} \cdot g^{\lambda}. 
\end{equation}
In this way, we define an element of $(\C^{\infty}_{\M}|_{U}[K])^{\ast}(V) = \Omega^{1}_{\M|_{U} \times K}(U)$ by Proposition \ref{tvrz_dualsectionseasy}. It is not difficult to see that this actually defines a $\C^{\infty}_{\M}|_{U}$-linear sheaf isomorphism $\C^{\infty}_{\M}|_{U}[K^{\ast}] \cong (\C^{\infty}_{\M}|_{U}[K])^{\ast}$.

Now, for each local trivialization chart $(U,\varphi)$, recall the definition (\ref{eq_varphivee}) of $\varphi^{\vee}: \Gamma^{\ast}_{\E}|_{U} \rightarrow \C^{\infty}_{\M}|_{U}[K^{\ast}]$. For each $V \in \Op(U)$ and $\alpha \in \Omega^{1}_{\E}(V)$, the action (\ref{eq_CKastisCKast}) of $\varphi^{\vee}_{V}(\alpha)$ on an element $g^{\kappa} \otimes \vartheta_{\kappa} \in \C^{\infty}_{\M}(V) \otimes K$ takes the form 
\begin{equation}
\begin{split}
[\varphi^{\vee}_{V}(\alpha)]( g^{\kappa} \otimes \vartheta_{\kappa}) = & \  [\alpha( \Phi_{\lambda}|_{V}) \otimes \vartheta^{\lambda}]( g^{\kappa} \otimes \vartheta_{\kappa}) = (-1)^{|\vartheta^{\lambda}||g^{\lambda}|} \alpha(\Phi_{\lambda}|_{V}) \cdot g^{\lambda} \\
= & \ (-1)^{|\alpha||g^{\lambda}|} g^{\lambda} \cdot \alpha(\Phi_{\lambda}|_{V}) = \alpha( g^{\lambda} \Phi_{\lambda}|_{V}) \\
= & \ \alpha( \varphi_{V}^{-1}( g^{\kappa} \otimes \vartheta_{\kappa})). 
\end{split}
\end{equation}
But this proves with respect to the identification $\C^{\infty}_{\M}|_{U}[K^{\ast}] \cong (\C^{\infty}_{\M}|_{U}[K])^{\ast}$, $\varphi^{\vee}$ can be identified with the transpose sheaf morphism $(\varphi^{-1})^{T} =: \varphi^{-T}$, see Proposition \ref{tvrz_transpose}. The claim $G^{\vee}_{\alpha \beta} := G_{\alpha \beta}^{-T}$ now follows easily from the definitions and the properties of transpose maps in Proposition \ref{tvrz_transpose}. 

The equation (\ref{eq_GalphabetaVeecomps}) can be obtained immediately from (\ref{eq_locframerel}). Indeed, one finds 
\begin{equation}
\begin{split}
\Phi^{\lambda}_{(\beta)}|_{U_{\alpha \beta}}( \Phi_{\kappa}^{(\alpha)}|_{U_{\alpha \beta}}) = &\ \Phi^{\lambda}_{(\beta)}|_{U_{\alpha \beta}}( (G_{\beta \alpha})^{\rho}{}_{\kappa} \Phi^{(\beta)}_{\rho}|_{U_{\alpha \beta}}) \\
= & \ (-1)^{|\vartheta^{\lambda}|(|\vartheta_{\kappa}| - |\vartheta_{\rho}|)} (G_{\beta \alpha})^{\rho}{}_{\kappa} \Phi^{\lambda}_{(\beta)}( \Phi^{(\beta)}_{\rho})|_{U_{\alpha \beta}} \\
= & \ (-1)^{|\vartheta_{\lambda}|( |\vartheta_{\kappa}| - |\vartheta_{\lambda}|)} (G_{\beta \alpha})^{\lambda}{}_{\kappa}.
\end{split}
\end{equation}
We have also used the fact that $|\vartheta_{\lambda}| = -|\vartheta^{\lambda}|$ in the last step. This proves that
\begin{equation}
\Phi^{\lambda}_{(\beta)}|_{U_{\alpha \beta}} = (-1)^{|\vartheta_{\lambda}|( |\vartheta_{\kappa}| - |\vartheta_{\lambda}|)} (G_{\beta \alpha})^{\lambda}{}_{\kappa} \Phi^{\kappa}_{(\alpha)}|_{U_{\alpha \beta}}.
\end{equation}
The claim (\ref{eq_GalphabetaVeecomps}) now follows from (\ref{eq_locframerel}). 

Note that the peculiar signs are essential, otherwise the cocycle conditions (\ref{eq_VBtransfunctcocycle}) would contradict each other. They are also consistent with Remark \ref{rem_peculiardoubledual}. 
\end{proof}
Moreover, for the purpose of the next section, we will also need the transition functions for degree shifted graded vector bundle. 
\begin{tvrz} \label{tvrz_shiftedtrans}
Let $\E$ be a graded vector bundle over a graded manifold $\M = (M, \C^{\infty}_{\M})$ with a typical fiber $K$. Let $\{ (U_{\alpha},\varphi_{\alpha}) \}_{\alpha \in I}$ be a local trivialization of $\E$. Fix a total basis $(\vartheta_{\lambda})_{\lambda=1}^{m}$ of $K$. Let $\ell \in \Z$ be a fixed integer. 

Let $\{ (U_{\alpha}, \varphi'_{\alpha}) \}_{\alpha \in I}$ be the induced local trivialization of the degree shift $\E[\ell]$ of $\E$, see Proposition \ref{tvrz_degreeshiftedvector}. Let $G'_{\alpha \beta}$ be the corresponding transition maps. As $(\vartheta_{\lambda})_{\lambda=1}^{m}$ can be viewed as a total basis of $K[\ell]$, we may define the respective transition functions as 

\begin{equation}
(G'_{\alpha \beta})_{U_{\alpha \beta}}(1 \otimes \vartheta^{\lambda}) =: (G'_{\alpha \beta})^{\kappa}{}_{\lambda} \otimes \vartheta^{\kappa},
\end{equation}
for all $(\alpha,\beta) \in I^{2}$ and $\lambda,\kappa \in \{1,\dots,m\}$. 

Then each $G'_{\alpha \beta}$ fits into the commutative diagram 
\begin{equation}  \label{eq_commdiagsihftedtrans}
\begin{tikzcd}[column sep=large]
(\C^{\infty}_{\M}|_{U_{\alpha \beta}}[K])[\ell] \arrow{d}{\psi^{K}|_{U_{\alpha \beta}}} \arrow{r}{G_{\alpha \beta}[\ell]} & (\C^{\infty}_{\M}|_{U_{\alpha \beta}}[K])[\ell] \arrow{d}{\psi^{K}|_{U_{\alpha \beta}}} \\
\C^{\infty}_{\M}|_{U_{\alpha \beta}}[K[\ell]] \arrow[dashed]{r}{G'_{\alpha \beta}} & \C^{\infty}_{\M}|_{U_{\alpha \beta}}[K[\ell]]
\end{tikzcd},
\end{equation}
where $G_{\alpha \beta}[\ell]$ is defined by (\ref{eq_varphiellshifted}) and $\psi^{K}$ is the sheaf isomorphism  defined under (\ref{eq_psiKiso}). Moreover, the transition functions take the form
\begin{equation} \label{eq_shiftedtrans}
(G'_{\alpha \beta})^{\kappa}{}_{\lambda} = (G_{\alpha \beta})^{\kappa}{}_{\lambda},
\end{equation}
where on the right-hand side, there are the transition functions corresponding to the local trivialization $\{ (U_{\alpha},\varphi_{\alpha})\}_{\alpha \in I}$ of $\E$ and the total basis $(\vartheta_{\lambda})_{\lambda=1}^{m}$. 
\end{tvrz}
\begin{proof}
The commutative diagram (\ref{eq_commdiagsihftedtrans}) follows immediately from the definition of the induced local trivialization $\{ (U_{\alpha}, \varphi'_{\alpha}) \}_{\alpha \in I}$. Now, for each $\alpha \in I$, we have the local frame $( \Phi_{\lambda}^{(\alpha)})_{\lambda=1}^{m}$ corresponding to the local trivialization chart $(U_{\alpha},\varphi_{\alpha})$ for $\E$ and the total basis $(\vartheta_{\lambda})_{\lambda=1}^{m}$ of $K$. Let $(\Phi'^{(\alpha)}_{\lambda})_{\lambda=1}^{m}$ be the local frame for $\E[\ell]$ corresponding to $(U_{\alpha},\varphi'_{\alpha})$ and $(\vartheta_{\lambda})_{\lambda=1}^{m}$, viewed as a total basis of $K[\ell]$. It follows from the definitions that $\Phi'^{(\alpha)}_{\lambda} = (-1)^{\ell(|\vartheta_{\lambda}| + \ell)} \Phi^{(\alpha)}_{\lambda}$. Using (\ref{eq_locframerel}), we thus find
\begin{equation}
\begin{split}
\Phi'^{(\beta)}_{\lambda} = & \ (-1)^{\ell(|\vartheta_{\lambda}| + \ell)} \Phi^{(\beta)}_{\lambda} = (-1)^{\ell(|\vartheta_{\lambda}| + \ell)} (G_{\alpha \beta})^{\kappa}{}_{\lambda} \Phi_{\kappa}^{(\alpha)} \\
= & \ (-1)^{\ell( |\vartheta_{\lambda}| - |\vartheta_{\kappa}|)} (G_{\alpha \beta})^{\kappa}{}_{\lambda} \Phi'^{(\alpha)}_{\kappa} \\
= & \ (G_{\alpha \beta})^{\kappa}{}_{\lambda} \tr' \Phi'^{(\alpha)}_{\kappa},
\end{split}
\end{equation}
where $\tr'$ denotes the action of $\C^{\infty}_{\M}$ on $\Gamma_{\E}[\ell]$ which differs by a sign from the original one, see Remark \ref{rem_shiftedmodule}. The claim (\ref{eq_shiftedtrans}) now follows from (\ref{eq_locframerel}). 
\end{proof}
Finally, let us examine the transition functions of the pullback graded vector bundle. We omit the proof, but one can easily deduce the statements using part (b) of the proof of Proposition \ref{tvrz_ap_pullbackVB}. 

\begin{tvrz} \label{tvrz_transffunctionspullback}
Let $\E$ be a graded vector bundle over a graded manifold $\M = (M, \C^{\infty}_{\M})$ with a typical fiber $K$. Let $\{ (U_{\alpha},\varphi_{\alpha}) \}_{\alpha \in I}$ be a local trivialization of $\E$. Fix a total basis $(\vartheta_{\lambda})_{\lambda=1}^{m}$ of $K$. Let $\phi: \cN \rightarrow \M$ be a graded smooth map. 

Let $\{ (\ul{\phi}^{-1}(U_{\alpha}), \phi'^{\ast}_{\alpha}(\varphi_{\alpha})) \}_{\alpha \in I}$ be the induced local trivialization of the pullback $\phi^{\ast} \E$, see the part $(b)$ of the proof of Proposition \ref{tvrz_ap_pullbackVB}. Let $G'_{\alpha \beta}$ be the corresponding transition maps. Let $(G'_{\alpha \beta})^{\kappa}{}_{\lambda}$ be the transition functions corresponding to the total basis $(\vartheta_{\lambda})_{\lambda=1}^{m}$. For each $(\alpha,\beta) \in I^{2}$, define
\begin{equation}
\phi'_{\alpha \beta} := \phi|_{\ul{\phi}^{-1}(U_{\alpha \beta})}: \cN|_{\ul{\phi}^{-1}(U_{\alpha \beta})} \rightarrow \M|_{U_{\alpha \beta}}.
\end{equation}
Then, with respect to the identification $\C^{\infty}_{\cN}|_{\ul{\phi}^{-1}(U_{\alpha \beta})}[K] \cong \phi'^{\ast}_{\alpha \beta}( \C^{\infty}_{\M}|_{U_{\alpha \beta}}[K])$, one has 
\begin{equation}
G'_{\alpha \beta} = \phi'^{\ast}_{\alpha \beta}( G_{\alpha \beta}),
\end{equation}
where on the right-hand side, there is the pullback functor $\phi'^{\ast}_{\alpha \beta}$ acting on the morphism in the category $\Sh^{\C^{\infty}_{\M}|_{U_{\alpha \beta}}}(U_{\alpha \beta}, \gVect)$, see the part $(a)$ of the proof of Proposition \ref{tvrz_ap_pullbackVB}. Moreover, the transition functions have the form
\begin{equation} \label{eq_transffunctionspullback}
(G'_{\alpha \beta})^{\kappa}{}_{\lambda} = \phi^{\ast}_{U_{\alpha \beta}}( (G_{\alpha \beta})^{\kappa}{}_{\lambda}),
\end{equation}
where on the right-hand side, there are the transition functions corresponding to the local trivialization $\{ (U_{\alpha},\varphi_{\alpha})\}_{\alpha \in I}$ of $\E$ and the total basis $(\vartheta_{\lambda})_{\lambda=1}^{m}$. 
\end{tvrz}
\begin{rem}
Note that the above proposition gives an alternative way to construct the pullback vector bundle $\phi^{\ast} \E$. One can simply \textit{define} its transition functions by (\ref{eq_transffunctionspullback}) and then use Proposition \ref{tvrz_collationVB}. Using the modification of Proposition \ref{tvrz_shcollation2}, it can be proved that this construction does not (up to an isomorphism) depend on the used local trivialization of $\E$. 
\end{rem}
\subsection{Constructing the total space} \label{subsec_totalspace}
Let $\E$ be a graded vector bundle over a graded manifold $\M = (M,\C^{\infty}_{\M})$. As promised under Definition \ref{tvrz_shcollation2}, we will now construct an actual graded manifold $\E$ together with a graded smooth projection $\pi: \E \rightarrow \M$. We will do so using the gluing procedure of Proposition \ref{tvrz_gmgluing1}. One proceeds in several steps. Note that the procedure is in fact very similar to the ordinary case. An uninterested reader can thus skip directly to the resulting Proposition \ref{tvrz_totalspace}.

\begin{enumerate}[(a)]
\item \textbf{Preliminaries, the underlying manifold}:
One has to start with the underlying manifold $E$. We have already given a hint in Example \ref{ex_pullbackbyiM}. Namely, one constructs the pullback $\ul{\E} := i^{\ast}_{M} \E$ by the canonical inclusion $i_{M}: M \rightarrow \M$, obtaining thus a graded vector bundle over $M$. For each $k \in \Z$, one obtains a graded subbundle $\ul{\E}_{k} \subseteq \ul{\E}$ with a typical fiber $K_{k}$, where $\Gamma_{\ul{\E}_{k}} := (\Gamma_{\ul{\E}})_{k}$ In particular, $E := \ul{\E}_{0}$ is an ordinary vector bundle over $M$, described by its sheaf of sections $\Gamma_{E} := (\Gamma_{\ul{\E}})_{0} \in \Sh^{\C^{\infty}_{M}}(M,\Vect)$.  

Next, let $\{ (U_{\alpha}, \varphi_{\alpha}) \}_{\alpha \in I}$ be a local trivialization for $\E$. Without a loss of generality, we may assume that we have a graded smooth atlas $\A = \{ (U_{\alpha}, \phi_{\alpha}) \}$ for $\M$. Moreover, we may assume that the open cover $\{ U_{\alpha} \}_{\alpha \in I}$ is at most countable. 

Let $(\vartheta_{\lambda})_{\lambda = 1}^{m}$ be a total basis of $K$. Let $(\vartheta^{(0)}_{\lambda_{0}} )_{\lambda_{0}=1}^{m_{0}}$ denote its elements of degree zero, forming the basis of the vector space $K_{0}$. Let $(G_{\alpha \beta})^{\kappa}{}_{\lambda}$ be the transition functions corresponding to $\{ (U_{\alpha}, \varphi_{\alpha}) \}_{\alpha \in I}$ and $(\vartheta_{\lambda})_{\lambda=1}^{m}$. Let $\{ (U_{\alpha}, \ul{\varphi_{\alpha}}) \}_{\alpha \in I}$ denote the induced local trivialization of $E$. Each of the $\C^{\infty}_{M}|_{U_{\alpha}}$-linear sheaf morphisms $\ul{\varphi_{\alpha}}: \Gamma_{\E}|_{U_{\alpha}} \rightarrow \C^{\infty}_{M}|_{U_{\alpha}}[K_{0}]$ is obtained as the zero component of the induced local trivialization map of the pullback vector bundle $\ul{\E}$. Its transition functions with respect to the basis $(\vartheta_{\lambda_{0}}^{(0)})_{\lambda_{0}=1}^{m}$ take the form 
\begin{equation}
(g_{\alpha \beta})^{\kappa_{0}}{}_{\lambda_{0}} := (i_{M})^{\ast}_{U_{\alpha \beta}}( (G_{\alpha \beta})^{\kappa_{0}}{}_{\lambda_{0}}) = \ul{ (G_{\alpha \beta})^{\kappa}{}_{\lambda_{0}}}, 
\end{equation}
where we have used Proposition \ref{tvrz_transffunctionspullback} and the definition of $i_{M}$, see Proposition \ref{tvrz_bodymap}. In other words, the transition functions of $E$ are the bodies of transition functions of $\E$ corresponding to the basis elements of a degree zero. It follows that each transition map $g_{\alpha \beta}$ can be viewed as a smooth function $g_{\alpha \beta}: U_{\alpha \beta} \rightarrow \GL(K_{0})$. We shall now employ the following standard theorem to construct the total space for the vector bundle $E$. 
\begin{tvrz} \label{tvrz_totalspaceE}
Let $E$ be a vector bundle over $M$ with a typical fiber $K_{0}$, defined by its sheaf of sections $\Gamma_{E}$, that is as in Definition \ref{def_vb}. Let $\{ (U_{\alpha}, \ul{\varphi_{\alpha}}) \}_{\alpha \in I}$ be its local trivialization and suppose we have a smooth atlas $\ul{\A} = \{ (U_{\alpha}, \ul{\phi_{\alpha}}) \}_{\alpha \in I}$ for $M$, defined on the same and at most countable cover $\{ U_{\alpha} \}_{\alpha \in I}$ of $M$. 

Then there exists a smooth manifold $E$ and a surjective smooth map $\ul{\pi}: E \rightarrow M$, such that:
\begin{enumerate}[(i)] 
\item For each $\alpha \in I$, there is a diffeomorphism $\ul{\varphi_{\alpha}}: \ul{\pi}^{-1}(U_{\alpha}) \rightarrow U_{\alpha} \times K_{0}$, such that $\pi_{1} \circ \ul{\varphi_{\alpha}} = \ul{\pi}$. For each $(\alpha,\beta) \in I^{2}$ and $(m,k) \in U_{\alpha \beta} \times K$, one has $\ul{\varphi_{\alpha}} \circ \ul{\varphi_{\beta}}^{-1}(m,k) = (m, g_{\alpha \beta}(m)k)$, where $g_{\alpha \beta}$ are the transition maps corresponding to the given local trivialization of $E$. 
\item $\ul{\pi}: E \rightarrow M$ becomes a usual vector bundle with a local trivialization $\{ (U_{\alpha}, \ul{\varphi_{\alpha}}) \}_{\alpha \in I}$ obtained in $(i)$. Its sheaf of smooth sections is isomorphic to $\Gamma_{E}$. 
\item There is an induced smooth atlas $\ul{\B} = \{ ( \ul{\pi}^{-1}(U_{\alpha}), \ul{\rho_{\alpha}}) \}$ for $E$, where $\ul{\rho_{\alpha}}: \pi^{-1}(U_{\alpha}) \rightarrow \ul{\phi_{\alpha}}(U_{\alpha}) \times \R^{m_{0}}$ is defined by $\ul{\rho_{\alpha}}(e) := ( \ul{\phi_{\alpha}}( \ul{\pi}(e)), \psi_{0}^{-1}( \pi_{2}(\ul{\varphi_{\alpha}}(e)))$, where $\psi_{0}: \R^{m_{0}} \rightarrow K_{0}$ is the isomorphism induced by the choice of the basis $(\vartheta_{\lambda_{0}})_{\lambda_{0}=1}^{m_{0}}$. 
\end{enumerate}
\end{tvrz}
There are certain uniqueness claims about $E$, which we will only discuss in the appendix, see the proof of Proposition \ref{tvrz_ap_totalspaceunique}. 

\item \textbf{The total space:} Now, let us construct the graded manifold $\E$. For each $\alpha \in I$, write $\U_{\alpha} := \ul{\pi}^{-1}(U_{\alpha})$. Let $(n_{j})_{j \in \Z} := \gdim(\M)$ and $(m_{j})_{j \in \Z} := \grk(\E)$. The graded dimension of $\E$ will be $(N_{j})_{j \in \Z}$, where $N_{j} := n_{j} + m_{-j}$. Its underlying manifold will be $E$ provided to us by Proposition \ref{tvrz_totalspaceE}. Note that $N_{0} = n_{0} + m_{0} = \dim(E)$. We also have a smooth atlas $\ul{\B} = \{ ( \U_{\alpha}, \ul{\rho_{\alpha}}) \}_{\alpha \in I}$ for $E$. To satisfy the assumptions of Proposition \ref{tvrz_gmgluing1}, for each $(\alpha,\beta) \in I^{2}$, one has to construct a graded smooth diffeomorphism 
\begin{equation}
\rho_{\alpha \beta}: \ul{\rho_{\beta}}( \U_{\alpha \beta})^{(N_{j})} \rightarrow \ul{\rho_{\alpha}}(\U_{\alpha \beta})^{(N_{j})},
\end{equation}
such that $\ul{\rho_{\alpha \beta}}$ are the transition maps of $\ul{\B}$, and there holds the cocycle condition (\ref{eq_cocyclegmgluing}). 

Since $\ul{\rho_{\beta}}( \U_{\alpha \beta}) = \ul{\phi_{\beta}}(U_{\alpha \beta}) \times \R^{m_{0}}$, there is a canonical isomorphism (see Example \ref{ex_productgdomains}):
\begin{equation}
\ul{\rho_{\beta}}( \U_{\alpha \beta})^{(N_{j})} \cong \ul{\phi_{\beta}}(U_{\alpha \beta})^{(n_{j})} \times (\R^{m_{0}})^{(m_{-j})}.
\end{equation}
Let $(\bbz^{A})_{A=1}^{n}$ denote the graded coordinate functions corresponding to the first factor (the ``\textit{base manifold coordinates}''), and $(\Xi^{\lambda})_{\lambda=1}^{m}$ those corresponding to the second factor (the ``\textit{fiber coordinates}''). We choose them so that $|\Xi^{\lambda}| = |\vartheta^{\lambda}|$, where $(\vartheta^{\lambda})_{\lambda =1}^{m}$ is the total basis of $K^{\ast}$, dual to the total basis $(\vartheta_{\lambda})_{\lambda=1}^{m}$ we have used in part $(a)$ to construct the transition functions. We also have a graded smooth projection $\hat{\pi}_{1}: \ul{\rho_{\beta}}(\U_{\alpha \beta})^{(N_{j})} \rightarrow \ul{\phi_{\beta}}(U_{\alpha \beta})^{(n_{j})}$. 

In order to define $\rho_{\alpha \beta}$, it suffices to define the pullbacks of coordinate functions, see Remark \ref{rem_itsufficestocalculatepullbacks} together with Theorem \ref{thm_gradedomaintheorem}. First, we declare 
\begin{equation} \label{eq_transmaptot1}
(\rho_{\alpha \beta})^{\ast}_{\ul{\rho_{\alpha}}(\U_{\alpha \beta})}( \bbz^{A}) := ((\hat{\pi}_{1})^{\ast}_{\ul{\phi_{\beta}}(U_{\alpha \beta})} \circ (\phi_{\alpha \beta})^{\ast}_{\ul{\phi_{\alpha}}(U_{\alpha \beta})})( \bbz^{A}),
\end{equation}
where on the right-hand side, $\bbz^{A}$ are viewed as coordinate functions on $\ul{\phi_{\alpha}}(U_{\alpha})^{(n_{j})}$. For the fiber coordinates, we declare 
\begin{equation} \label{eq_transmaptot2}
(\rho_{\alpha \beta})^{\ast}_{\ul{\rho_{\alpha}}(\U_{\alpha \beta})}( \Xi^{\lambda}) := ((\hat{\pi}_{1})^{\ast}_{\ul{\phi_{\beta}}(U_{\alpha \beta})} \circ (\phi^{-1}_{\beta})^{\ast}_{U_{\alpha \beta}})( (G^{\vee}_{\beta \alpha})_{\kappa}{}^{\lambda}) \cdot \Xi^{\kappa},
\end{equation}
where $(G^{\vee}_{\beta \alpha})_{\kappa}{}^{\lambda}$ are the transition functions of the dual $\E^{\ast}$, see Proposition \ref{tvrz_dualtrans}. To understand this expression, we just find the local representatives of  $(G^{\vee}_{\beta \alpha})_{\kappa}{}^{\lambda} \in \C^{\infty}_{\M}(U_{\alpha \beta})$ with respect to the graded local chart $(U_{\beta}, \phi_{\beta})$, pull them upstairs to the graded domain  $\ul{\rho_{\beta}}(\U_{\alpha \beta})^{(N_{j})}$, and multiply them with coordinate functions $\Xi^{\kappa}$ already living there. 

First, one has to argue that the underlying smooth map of $\rho_{\alpha \beta}$ is indeed the transition map $\ul{\rho_{\alpha \beta}}$ between the  two ordinary charts $(U_{\alpha}, \ul{\rho_{\alpha}})$ and $(U_{\beta}, \ul{\rho_{\beta}})$ for $E$. By Remark \ref{rem_itsufficestocalculatepullbacks}, it can be obtained from the bodies of the pullbacks of degree zero coordinate functions. 

For $|\bbz^{A}| = 0$, applying the body map on both sides of (\ref{eq_transmaptot1}) and using (\ref{eq_cdbodymaps}), one finds
\begin{equation}
\ul{(\rho_{\alpha \beta})^{\ast}_{\ul{\rho_{\alpha}}(\U_{\alpha \beta})}( \bbz^{A})} = \bbz^{A} \circ \ul{\phi_{\alpha \beta}} \circ \ul{\pi_{1}},
\end{equation}
where on the right-hand side, $\bbz^{A}$ is viewed as an ordinary coordinate function on $\ul{\phi_{\alpha}}(U_{\alpha \beta}) \subseteq \R^{n_{0}}$, and $\ul{\pi_{1}}: \ul{\phi_{\beta}}(U_{\alpha \beta}) \times \R^{m_{0}} \rightarrow \ul{\phi_{\beta}}(U_{\alpha\beta})$ is the projection. It follows from Proposition \ref{tvrz_totalspaceE}-$(iii)$ that the right-hand side is indeed $\ul{\rho_{\alpha \beta}}^{\ast}(\bbz^{A})$. 

The degree zero fiber coordinates can be relabeled as $(\Xi^{\lambda_{0}})_{\lambda_{0}=1}^{m_{0}}$, as they correspond to the basis $(\vartheta^{\lambda_{0}})_{\lambda_{0}=1}^{m_{0}}$, dual to the one $(\vartheta_{\lambda_{0}})_{\lambda_{0}=1}^{m_{0}}$ of $K_{0}$ (which we have used to construct $E$). They can be also viewed as ordinary coordinate functions on $\ul{\rho_{\alpha}}(\U_{\alpha \beta}) = \ul{\phi_{\alpha}}(U_{\alpha \beta}) \times \R^{m_{0}} \subseteq \R^{N_{0}}$. Applying the body map on both sides of (\ref{eq_transmaptot2}) and using (\ref{eq_cdbodymaps}), one finds
\begin{equation}
\ul{(\rho_{\alpha \beta})^{\ast}_{\ul{\rho_{\alpha}}(\U_{\alpha \beta})}( \Xi^{\lambda_{0}})} = ( (g^{\vee}_{\beta \alpha})_{\kappa_{0}}{}^{\lambda_{0}} \circ \ul{\phi_{\beta}^{-1}} \circ \ul{\pi_{1}}) \cdot \Xi^{\kappa_{0}},
\end{equation}
where $(g^{\vee}_{\beta \alpha})_{\kappa_{0}}{}^{\lambda_{0}} = (g_{\alpha \beta})^{\lambda_{0}}{}_{\kappa_{0}}$ are the transition functions of the dual vector bundle $E^{\ast}$. Again, it follows from Proposition \ref{tvrz_totalspaceE}-$(iii)$ that this is exactly how the transition functions $\ul{\rho_{\alpha \beta}}$ pull back the fiber coordinate functions $\Xi^{\lambda_{0}}$. Note that $\rho_{\alpha \beta}$ defines a graded diffeomorphism as its inverse is $\rho_{\beta \alpha}$. 

It remains to verify the cocycle condition (\ref{eq_cocyclegmgluing}). For each $(\alpha,\beta,\gamma) \in I^{3}$, let us use the shorthand notation $V := U_{\alpha \beta \gamma}$, $\hat{V}_{\alpha} := \ul{\phi_{\alpha}}(U_{\alpha \beta \gamma})$ and define $\hat{V}_{\beta}$ and $\hat{V}_{\gamma}$ similarly. It suffices to compare the pullbacks of coordinate functions. For each $A \in \{1,\dots,n\}$, one obtains 
\begin{equation}
\begin{split}
((\rho_{\beta \gamma})^{\ast}_{\hat{V}_{\beta}} \circ (\rho_{\alpha \beta})^{\ast}_{\hat{V}_{\alpha}})(\bbz^{A}) = & \ (\rho_{\beta \gamma})^{\ast}_{\hat{V}_{\beta}}[((\hat{\pi}_{1})^{\ast}_{\hat{V}_{\beta}} \circ (\phi_{\alpha \beta})^{\ast}_{\hat{V}_{\alpha}})(\bbz^{A})]  \\
= & \ (\hat{\pi}_{1})^{\ast}_{\hat{V}_{\gamma}}[ ((\phi_{\beta \gamma})^{\ast}_{\hat{V}_{\beta}} \circ (\phi_{\alpha \beta})^{\ast}_{\hat{V}_{\alpha}})(\bbz^{A})] \\
= & \ ((\hat{\pi}_{1})^{\ast}_{\hat{V}_{\gamma}} \circ (\phi_{\alpha \gamma})^{\ast}_{\hat{V}_{\alpha}})(\bbz^{A}) = (\rho_{\alpha \gamma})^{\ast}_{\hat{V}_{\alpha}}( \bbz^{A}). 
\end{split}
\end{equation}
We have used the cocycle condition (\ref{eq_cocyclegmgluing}) for the transition maps $\phi_{\alpha \beta}$. Next, for each $\lambda \in \{1, \dots, m\}$, one finds the expression 
\begin{equation}
\begin{split}
((\rho_{\beta \gamma})^{\ast}_{\hat{V}_{\beta}} \circ (\rho_{\alpha \beta})^{\ast}_{\hat{V}_{\alpha}})(\Xi^{\lambda}) = & (\rho_{\beta \gamma})^{\ast}_{\hat{V}_{\beta}} [ ( (\hat{\pi}_{1})^{\ast}_{\hat{V}_{\beta}} \circ (\phi^{-1}_{\beta})^{\ast}_{V})( (G^{\vee}_{\beta \alpha})_{\kappa}{}^{\lambda}|_{V}) \cdot \Xi^{\kappa} ] \\
= & \ (\hat{\pi}_{1})^{\ast}_{\hat{V}_{\gamma}}[ ((\phi_{\beta \gamma})^{\ast}_{\hat{V}_{\beta}} \circ (\phi_{\beta}^{-1})^{\ast}_{V})( (G^{\vee}_{\beta \alpha})_{\kappa}{}^{\lambda}|_{V})] \cdot (\rho_{\beta \gamma})^{\ast}_{\hat{V}_{\beta}}( \Xi^{\kappa}) \\
= & \ (\hat{\pi}_{1})^{\ast}_{\hat{V}_{\gamma}}[ (\phi_{\gamma}^{-1})^{\ast}_{V}( (G^{\vee}_{\beta \alpha})_{\kappa}{}^{\lambda}|_{V})] \cdot (\rho_{\beta \gamma})^{\ast}_{\hat{V}_{\beta}}( \Xi^{\kappa}) \\
= & \ ((\hat{\pi}_{1})^{\ast}_{\hat{V}_{\gamma}} \circ (\phi_{\gamma}^{-1})^{\ast}_{V})( (G^{\vee}_{\beta \alpha})_{\kappa}{}^{\lambda}|_{V} \cdot (G^{\vee}_{\gamma \beta})_{\chi}{}^{\kappa}|_{V}) \cdot \Xi^{\chi} \\
= & \ ((\hat{\pi}_{1})^{\ast}_{\hat{V}_{\gamma}} \circ (\phi_{\gamma}^{-1})^{\ast}_{V})( (G^{\vee}_{\gamma \alpha})_{\chi}{}^{\lambda}|_{V}) \cdot \Xi^{\chi} \\
= & \ (\rho_{\alpha \gamma})^{\ast}_{\hat{V}_{\alpha}}(\Xi^{\lambda}).
\end{split}
\end{equation}
As expected, we have used the cocycle condition (\ref{eq_VBtransfunctcocycle}) for the transition functions $(G_{\alpha \beta}^{\vee})_{\kappa}{}^{\lambda}$. 

By Proposition \ref{tvrz_gmgluing1}, there is thus a graded manifold $\E = (E, \C^{\infty}_{\E})$ together with a graded smooth atlas $\B = \{ (\U_{\alpha}, \rho_{\alpha}) \}_{\alpha \in I}$, such that the smooth atlas $\ul{\B}$ is induced from $\B$ and $\rho_{\alpha \beta}$ become its transition maps. Moreover, such $\E$ is unique up to an isomorphism. 
\item \textbf{The projection:} Let us now construct a graded smooth map $\pi: \E \rightarrow \M$ over the ordinary vector bundle projection $\ul{\pi}: E \rightarrow M$ provided by Proposition \ref{tvrz_totalspaceE}. 

Let $f \in \C^{\infty}_{\M}(U)$ for a given $U \in \Op(M)$. For each $\alpha \in I$, we have its local representative $\hat{f}_{\alpha} \in \C^{\infty}_{(n_{j})}( \ul{\phi_{\alpha}}(U \cap U_{\alpha}))$ with respect to the graded chart $(U_{\alpha}, \phi_{\alpha})$. It follows that functions
\begin{equation}
\hat{F}_{\alpha} := (\hat{\pi}_{1})^{\ast}_{\ul{\phi_{\alpha}}(U \cap U_{\alpha)}}( \hat{f}_{\alpha}) \in \C^{\infty}_{(N_{j})}( \ul{\rho_{\alpha}}( \ul{\pi}^{-1}(U) \cap \U_{\alpha}))
\end{equation}
for each $(\alpha,\beta) \in I^{2}$ satisfy the condition (\ref{eq_locreprelation}), if we consider the open set $\U := \ul{\pi}^{-1}(U) \subseteq E$ and the graded smooth atlas $\B = \{ (\U_{\alpha}, \rho_{\alpha}) \}_{\alpha \in I}$ for $\E$. Indeed, one finds 
\begin{equation}
\begin{split}
(\rho_{\alpha \beta})^{\ast}_{\ul{\rho_{\alpha}}(\U \cap \U_{\alpha \beta})}( \hat{F}_{\alpha}|_{\ul{\rho_{\alpha}}(\U \cap \U_{\alpha \beta})}) = & \  [(\rho_{\alpha \beta})^{\ast}_{\ul{\phi_{\alpha}}(U \cap U_{\alpha \beta}) \times \R^{m_{0}}} \circ (\hat{\pi}_{1})^{\ast}_{\ul{\phi_{\alpha}}(U \cap U_{\alpha \beta}}]( \hat{f}_{\alpha}|_{\ul{\phi_{\alpha}}(U \cap U_{\alpha \beta})}) \\
= & \ [(\hat{\pi}_{1})^{\ast}_{\ul{\phi_{\beta}}(U \cap U_{\alpha \beta})} \circ (\phi_{\alpha \beta})^{\ast}_{\ul{\phi_{\alpha}}(U \cap U_{\alpha \beta})}]( \hat{f}_{\alpha}|_{\ul{\phi_{\alpha}}(U \cap U_{\alpha \beta})}) \\
= & \ (\hat{\pi}_{1})^{\ast}_{\ul{\phi_{\beta}}(U \cap U_{\alpha \beta})}( \hat{f}_{\beta}|_{\ul{\phi_{\beta}}(U \cap U_{\alpha \beta})}) \\
= & \ \hat{F}_{\beta}|_{\ul{\rho_{\beta}}(\U \cap \U_{\alpha \beta})}.
\end{split}
\end{equation}
By Proposition \ref{tvrz_locrepgluing}, there is a unique function $\pi^{\ast}_{U}(f) \in \C^{\infty}_{\E}(\ul{\pi}^{-1}(U))$, such that each $\hat{F}_{\alpha}$ becomes its local representative with respect to the graded chart $(\U_{\alpha}, \rho_{\alpha})$. It is easy to check that $\pi^{\ast}_{U}$ defines a graded algebra morphism natural in $U$, hence a sheaf morphism $\pi^{\ast}: \C^{\infty}_{\M} \rightarrow \ul{\pi}_{\ast} \C^{\infty}_{\E}$. One can check that for each $U \in \Op(M)$, $f \in \C^{\infty}_{\M}(U)$ and $e \in \ul{\phi}^{-1}(U)$, one has 
\begin{equation}
(\pi^{\ast}_{U}(f))(e) = f( \ul{\pi}(e)).
\end{equation}
Consequently, $f( \ul{\pi}(e)) = 0$ implies $(\pi^{\ast}_{U}(f))(e) = 0$, hence due to (\ref{eq_Jacobsonradicalgradedmafolds}), $\pi = (\ul{\phi}, \pi^{\ast})$ defines a graded smooth map $\pi: \E \rightarrow \M$. 
\item \textbf{$\E$ is locally trivial}: Recall that we have started with a local trivialization $(U_{\alpha}, \varphi_{\alpha})$ of $\E$, a collection of sheaf isomorphisms $\varphi_{\alpha}: \Gamma_{\E}|_{U_{\alpha}} \rightarrow \C^{\infty}_{\M}|_{U_{\alpha}}$. 

Having a graded manifold $\E$, let us now construct a collection of ``actual'' graded diffeomorphisms $\varphi_{\alpha}: \E|_{\U_{\alpha}} \rightarrow \M|_{U_{\alpha}} \times K$, where $K$ is viewed as a graded manifold over $K_{0}$, as discussed in Example \ref{ex_gVectasGM}. Let $\pi_{1}: \M|_{U_{\alpha}} \times K \rightarrow \M|_{U_{\alpha}}$ and $\pi_{2}: \M|_{U_{\alpha}} \times K \rightarrow K$ be the canonical projections. We impose the following requirements:
\begin{enumerate}[(i)]
\item For each $\alpha \in I$, the underlying smooth map of $\varphi_{\alpha}$ has to be $\ul{\varphi_{\alpha}}: \U_{\alpha} \rightarrow U_{\alpha} \times K_{0}$, obtained in Proposition \ref{tvrz_totalspaceE}-$(i)$.
\item For each $\alpha \in I$, one has $\pi_{1} \circ \varphi_{\alpha} = \pi|_{\U_{\alpha}}: \E|_{\U_{\alpha}} \rightarrow \M|_{U_{\alpha}}$.
\end{enumerate}
By the universal property of products, $\varphi_{\alpha}$ is uniquely determined by its compositions with the projections $\pi_{1}$ and $\pi_{2}$. The first composition is already fixed by the property (ii). We thus only have to define a graded smooth map $\pi_{2} \circ \varphi_{\alpha}: \E|_{\U_{\alpha}} \rightarrow K$. As we have the graded local chart $\rho_{\alpha}: \E|_{\U_{\alpha}} \rightarrow \ul{\rho_{\alpha}}(\U_{\alpha})^{(N_{j})}$, it suffices to specify the composition
\begin{equation}
\hat{\varphi}_{\alpha} := \pi_{2} \circ \varphi_{\alpha} \circ \rho_{\alpha}^{-1}: \ul{\rho_{\alpha}}(\U_{\alpha})^{(N_{j})} \rightarrow K.
\end{equation}
It follows from the requirement (i) and the construction of $\ul{\rho_{\alpha}}$ in Proposition \ref{tvrz_totalspaceE}-$(iii)$ that the underlying smooth map $\ul{\hat{\varphi}_{\alpha}}: \ul{\phi_{\alpha}}(U_{\alpha}) \times \R^{m_{0}} \rightarrow K$ must take the form $\ul{\hat{\varphi}_{\alpha}} = \psi_{0} \circ \ul{\hat{\pi}_{2}}$, where $\ul{\hat{\pi}_{2}}: \ul{\phi_{\alpha}}(U_{\alpha}) \times \R^{m_{0}} \rightarrow \R^{m_{0}}$ is the projection and $\psi_{0}: \R^{m_{0}} \rightarrow K_{0}$ is the diffeomorphism induced by the choice of the basis $(\vartheta_{\lambda_{0}})_{\lambda_{0}=1}^{m_{0}}$. 

Now, as noted in Example \ref{ex_gVectasGM}-(i), we may regard the dual total basis $(\vartheta^{\lambda})_{\lambda=1}^{m}$ of $K^{\ast}$ as global coordinate functions on $K$. To define $\hat{\varphi}_{\alpha}$, it thus suffices to define their pullbacks. We declare
\begin{equation}
(\hat{\varphi}_{\alpha})^{\ast}_{K_{0}}( \vartheta^{\lambda}) := \Xi^{\lambda} \in \C^{\infty}_{(N_{j})}( \ul{\hat{\varphi}_{\alpha}}^{-1}(K_{0})). 
\end{equation}
Degree zero coordinate functions can be relabeled as $(\vartheta^{\lambda_{0}})_{\lambda_{0} = 1}^{m_{0}}$ and it follows from the definition of the coordinate functions $\Xi^{\lambda_{0}}$ on $\ul{\rho_{\alpha}}(\U_{\alpha})$ that 
\begin{equation}
\ul{(\hat{\varphi}_{\alpha})^{\ast}_{K_{0}}( \vartheta^{\lambda})} = \vartheta^{\lambda_{0}} \circ (\psi_{0} \circ \ul{\pi_{2}}),
\end{equation}
proving that one indeed has $\ul{\hat{\varphi}_{\alpha}} = \psi_{0} \circ \ul{\pi_{2}}$.

Finally, note that it is easy to see that $\varphi_{\alpha}$ is a local graded diffeomorphism, e.g. using the local coordinates. Its underlying map $\ul{\varphi_{\alpha}}: E|_{\U_{\alpha}} \rightarrow U_{\alpha} \times K_{0}$ is a bijection, hence by Proposition \ref{tvrz_diffiflocaldiffandbij}, $\varphi_{\alpha}$ is a graded diffeomorphism. 

Finally, if we denote the global graded chart for $K$ as $\varphi_{K}$, the graded chart $\rho_{\alpha}$ can be written as a composition $\rho_{\alpha} := (\phi_{\alpha} \times \varphi_{K}) \circ \varphi_{\alpha}$, if we identify $\ul{\rho_{\alpha}}(\U_{\alpha})^{(N_{j})} \cong \ul{\phi_{\alpha}}(U_{\alpha})^{(n_{j})} \times (\R^{m_{0}})^{(m_{-j})}$. 
\end{enumerate}
Let us summarize the results obtained in parts (a) - (d) in the form of a proposition. It should be viewed as a graded version of Proposition \ref{tvrz_totalspaceE}.
\begin{tvrz}[\textbf{The total space}] \label{tvrz_totalspace}
Let $\E$ be a graded vector bundle over $\M$ with a typical fiber $K$. Let $\{ (U_{\alpha}, \varphi_{\alpha}) \}_{\alpha \in I}$ be its local trivialization and suppose we have a graded smooth atlas $\A = \{ (U_{\alpha}, \phi_{\alpha}) \}_{\alpha \in I}$ for $\M$, defined on the same and at most countable cover $\{ U_{\alpha} \}_{\alpha \in I}$ of $M$.

Then there exists a graded manifold $\E = (E,\C^{\infty}_{\E})$, where $\ul{\pi}: E \rightarrow M$ is an ordinary vector bundle, and a graded smooth map $\pi = (\ul{\pi},\pi^{\ast}): \E \rightarrow \M$, such that:
\begin{enumerate}[(i)]
\item For each $\alpha \in I$, there is a diffeomorphism $\varphi_{\alpha}: \E|_{\U_{\alpha}} \rightarrow \M|_{U_{\alpha}} \times K$, such that $\pi_{1} \circ \varphi_{\alpha} = \pi$, where $\pi_{1}$ denotes the canonical projection and $\U_{\alpha} := \ul{\pi}^{-1}(U_{\alpha})$. Moreover, the collection $\{ (U_{\alpha}, \ul{\varphi_{\alpha}}) \}_{\alpha \in I}$ is a local trivialization of $\ul{\pi}: E \rightarrow M$. 
\item There is an induced graded smooth atlas $\B = \{ (\U_{\alpha}, \rho_{\alpha}) \}_{\alpha \in I}$ for $\E$, where each $\rho_{\alpha}: \E|_{\U_{\alpha}} \rightarrow \ul{\rho_{\alpha}}(\U_{\alpha})^{(N_{j})} \cong \ul{\phi_{\alpha}}(U_{\alpha})^{(n_{j})} \times (\R^{m_{0}})^{(m_{-j})}$ can be written as a composition
\begin{equation}
\rho_{\alpha} = (\phi_{\alpha} \times \varphi_{K}) \circ \varphi_{\alpha},
\end{equation}
where $\varphi_{K}: K \rightarrow (\R^{m_{0}})^{(m_{-j})}$ is the global graded chart for $K$ induced by a choice of a total basis $(\vartheta_{\lambda})_{\lambda=1}^{m}$ of $K$, see Example \ref{ex_gVectasGM}.  
\end{enumerate}
This construction does not depend on the particular choice of a local trivialization $\{(U_{\alpha},\varphi_{\alpha})\}_{\alpha \in I}$ or of a graded smooth atlas $\A$. See Proposition \ref{tvrz_ap_totalspaceunique} for details.

The graded manifold $\E$ is called the \textbf{total space} of the graded vector bundle $\E$.
\end{tvrz}
\begin{example}
We can finally do some justice to the trivial graded vector bundle defined in Example \ref{ex_trivialVB}. Its sheaf of sections is $\C^{\infty}_{\M}[K]$, hence $M$ can be covered by a single ``local'' trivialization chart $(M, \1_{\C^{\infty}_{\M}[K]})$. It follows from Proposition \ref{tvrz_totalspace}-$(i)$ that its total space is diffeomorphic to the product manifold $\M \times K$. This explains the notation. 
\end{example}

\begin{example}
Proposition \ref{tvrz_totalspace} is by far the most common way to construct graded manifolds. Indeed, one often starts with an ordinary vector bundle $q: E \rightarrow M$, viewing it as a (trivially) graded vector bundle. The graded manifold $E[k]$ constructed in Example \ref{ex_degreeshiftedordvect} is then precisely the total space corresponding to the degree shift of $E$. One can then construct a tangent bundle of $E[k]$ followed by another degree shift, that is $(T E[k])[\ell]$, which is usually denoted as $T[\ell]E[k]$, and look at the corresponding total space to obtain a new graded manifold over $M$ (if $k \neq 0$, $\ell \neq 0$ and $k + \ell \neq 0$). Similarly, one write $T^{\ast}[\ell]E[k]$ for $(T^{\ast} E[k])[\ell]$ (if $k \neq 0$, $\ell \neq 0$ and $k - \ell \neq 0$). 

Let us briefly discuss the degrees of local coordinates. Let $n := \dim(M)$ and $m := \rk(E)$. On $E[k]$, we have the local coordinates $(x^{i})_{i=1}^{n}$ of degree $0$ and $(\xi_{\mu})_{\mu=1}^{m}$ of degree $k$. On $T[\ell]E[k]$ we obtain the additional ``conjugate velocities'' $(v^{i})_{i=1}^{n}$ and $(w_{\mu})_{\mu=1}^{m}$ of degree $\ell$ and $k + \ell$, respectively. For the shifted cotangent case $T^{\ast}[\ell]E[k]$, we obtain the additional ``conjugate momenta'' $(p_{i})_{i=1}^{n}$ and $(r^{\mu})_{\mu=1}^{k}$ of degree $\ell$ and $\ell - k$. 

For example, the graded manifold corresponding to the exact Courant algebroid is $\M = T^{\ast}[2]T[1]M$, where $M$ is an ordinary manifold, see \cite{roytenberg2002structure}. Here we obtain the coordinates $(x^{i})_{i=1}^{n}$ of degree $0$, $(\xi_{i})_{i=1}^{n}$ of degree $1$ and the corresponding conjugate momenta $(p_{i})_{i=1}^{n}$ of degree $2$ and $(r^{i})_{i=1}^{n}$ of degree $1$. 
\end{example}
\begin{rem} \label{rem_gradedvectorbundlesdifferently}
Let us very briefly elaborate on the following two questions:
\begin{enumerate}[(1)] 
\item Can the sheaf $\Gamma_{\E}$ be obtained from the sheaf $\C^{\infty}_{\E}$?
\item Is there a way to define graded vector bundles in a ``standard way'', that is using graded manifolds and graded smooth maps?
\end{enumerate}
In both cases, the answer is affirmative. First, observe that $\C^{\infty}_{\E}$ can be viewed as a sheaf of graded $\C^{\infty}_{\E}$-modules. Hence by Remark \ref{rem_pushforwardmodules}, we have a pushforward sheaf $\pi_{\ast} \C^{\infty}_{\E}$ of graded $\C^{\infty}_{\M}$-modules. Let $\{ (U_{\alpha},\varphi_{\alpha}) \}_{\alpha \in I}$ be a local trivialization of $\E$. In particular, we have the corresponding graded diffeomorphisms $\varphi_{\alpha}: \E|_{\U_{\alpha}} \rightarrow \M|_{U_{\alpha}} \times K$, where $\U_{\alpha} := \ul{\pi}^{-1}(U_{\alpha})$. 

Let $f \in (\pi_{\ast} \C^{\infty}_{\E})(U) \equiv \C^{\infty}_{\E}(\ul{\pi}^{-1}(U))$. Write $\U := \ul{\pi}^{-1}(U)$. We say that $f$ \textbf{linear in the fiber}, if for every $\alpha \in I$, the function 
\begin{equation} 
(\varphi^{-1}_{\alpha})^{\ast}_{\U \cap \U_{\alpha}}( f|_{\U \cap \U_{\alpha}}) \in \C^{\infty}_{\M|_{U_{\alpha}} \times K}( (U \cap U_{\alpha}) \times K_{0})
\end{equation}
can be written as a linear combination $(\pi_{1})^{\ast}_{U \cap U_{\alpha}}(f_{\lambda}^{(\alpha)}) \cdot \vartheta^{\lambda}$ for functions $f_{\lambda}^{(\alpha)} \in \C^{\infty}_{\M}(U \cap U_{\alpha})$, where $\pi_{1}: \M|_{U_{\alpha}} \times K \rightarrow \M|_{U_{\alpha}}$ is the projection and $\vartheta^{\lambda}$ are global coordinate functions on $\M|_{U_{\alpha}} \times K$ induced from a total basis $(\vartheta^{\lambda})_{\lambda=1}^{m}$ of $K^{\ast}$. 

One can show that functions linear in fiber form a graded $\C^{\infty}_{\M}(U)$-submodule 
\begin{equation}
\C^{\lin}_{\E}(U) \subseteq (\pi_{\ast} \C^{\infty}_{\E})(U).
\end{equation}
In fact, they form a subsheaf $\C^{\lin}_{\E} \subseteq \pi_{\ast} \C^{\infty}_{\E}$ of graded $\C^{\infty}_{\M}$-modules. This sheaf can be seen isomorphic to the dual sheaf $\Gamma_{\E}^{\ast}$. More precisely, every $f \in \C^{\lin}_{\E}(U)$ is mapped to a section in $\Gamma_{\E}^{\ast}(U)$ which can be on $U \cap U_{\alpha}$ decomposed as $f_{\lambda}^{(\alpha)} \Phi^{\lambda}_{(\alpha)}|_{U \cap U_{\alpha}}$, where $(\Phi^{\lambda}_{(\alpha)})_{\lambda=1}^{m}$ is the local frame corresponding to the local trivialization chart $(U, \varphi_{\alpha}^{\vee})$ of $\E^{\ast}$ and $(\vartheta^{\lambda})_{\lambda=1}^{m}$. This answers the question (1). 

Now, let $\E$ and $\M$ be a pair of graded manifolds together with a graded smooth map $\pi: \E \rightarrow \M$. Assume that $\ul{\pi}: E \rightarrow M$ is surjective. Fix a finite-dimensional graded vector space $K$ and suppose we are given a collection $\{ (U_{\alpha}, \varphi_{\alpha}) \}_{\alpha \in I}$, such that $\{ U_{\alpha} \}_{\alpha \in I}$ is an open cover of $M$, and each $\varphi_{\alpha}: \E|_{\U_{\alpha}} \rightarrow \M|_{U_{\alpha}} \times K$ is a graded diffeomorphism satisfying $\pi_{1} \circ \varphi_{\alpha} = \pi|_{\U_{\alpha}}$. We write $\U_{\alpha} := \ul{\pi}^{-1}(U_{\alpha})$. In the same way as above, one would like to define a subsheaf $\C^{\lin}_{\E} \subseteq \pi_{\ast} \C^{\infty}_{\E}$ of functions linear in the fiber. To do so, form the transition maps 
\begin{equation} \label{eq_transdiffeos1}
\varphi_{\alpha \beta} := \varphi_{\alpha}|_{\U_{\alpha \beta}} \circ (\varphi_{\beta}|_{\U_{\alpha \beta}})^{-1}: \M|_{U_{\alpha \beta}} \times K \rightarrow \M|_{U_{\alpha \beta}} \times K.
\end{equation}
It turns out $\C^{\lin}_{\E}$ is well-defined, iff for each $(\alpha,\beta) \in I^{2}$ and $\lambda \in \{1,\dots,m\}$, one has 
\begin{equation} \label{eq_transdiffeos2}
(\pi_{2} \circ \varphi_{\alpha \beta})^{\ast}_{K_{0}}(\vartheta^{\lambda}) = (\pi_{1})^{\ast}_{U_{\alpha \beta}}( (G_{\beta \alpha}^{\vee})_{\kappa}{}^{\lambda} ) \cdot \vartheta^{\kappa},
\end{equation}
for unique functions $(G^{\vee}_{\beta \alpha})_{\kappa}{}^{\lambda} \in \C^{\infty}_{\M}(U_{\alpha \beta})$. Not surprisingly, they satisfy the cocycle condition (\ref{eq_VBtransfunctcocycle}). It follows that $\Gamma_{\E} := (\C_{\E}^{\lin})^{\ast}$ is locally freely and finitely generated sheaf of $\C^{\infty}_{\M}$-modules of a constant rank, hence it defines a graded vector bundle over $\M$ in the sense of Definition \ref{def_vb}. 

This shows graded vector bundles can be equivalently defined in a more ``traditional way'' as triples $(\E,\M,\pi)$ together with a local trivialization $\{ (U_{\alpha}, \varphi_{\alpha}) \}_{\alpha \in I}$, where each $\varphi_{\alpha}: \E|_{\U_{\alpha}} \rightarrow \M|_{U_{\alpha}} \times K$ is a graded diffeomorphism and the transition maps (\ref{eq_transdiffeos1}) have the form (\ref{eq_transdiffeos2}). This is the approach to graded vector bundles which was taken e.g. in \cite{mehta2006supergroupoids}. 
\end{rem}
\section{Exterior algebra} \label{sec_exterior}
\subsection{Sheaf of differential forms} \label{subsec_diffforms}
Let $\M = (M,\C^{\infty}_{\M})$ be a graded manifold. In this section, our main goal is to construct a \textbf{sheaf $\Omega_{\M}$ of differential forms on $\M$}, having the following properties:
\begin{enumerate}[(i)]
\item It is a sheaf of graded commutative algebras over $M$, $\Omega_{\M} \in \Sh(M,\gcAs)$.
\item It is a sheaf of graded $\C^{\infty}_{\M}$-modules over $M$, $\Omega_{\M} \in \Sh^{\C^{\infty}_{\M}}(M, \gVect)$. 
\item These two structures are mutually compatible.
\item One obtains a graded analogue of Cartan calculus, that is one has a Lie derivative, exterior derivative and interior product satisfying ``Cartan magic formulas''. 
\end{enumerate}
We use the approach that appeared already in \cite{bernshtein1977integration} for supermanifolds. The idea is to find $\Omega_{\M}$ as a sheaf of functions on a certain graded manifold over $M$. The requirements $(i)$-$(iii)$ will be satisfied automatically by construction. Finally, all operations in the Cartan calculus are in fact graded derivations, which implies that they can be defined as vector fields. In the $\Z$-graded setting, this was used to some extent in \cite{mehta2006supergroupoids} and it is a common knowledge in the area. 

In Proposition \ref{tvrz_vf}, we have constructed a sheaf $\X_{\M}$ of vector fields, which in turn defines the tangent bundle $T\M$ with $\Gamma_{T\M} := \X_{\M}$. Its typical fiber is $\R^{(n_{-j})}$. For each integer $\ell \in \Z$, one can consider its degree shift $T\M[\ell]$ (see Proposition \ref{tvrz_degreeshiftedvector}), which is usually denoted as $T[\ell]\M$. Let $(n_{j})_{j \in \Z} := \gdim(\M)$. It follows that the graded rank of $T[\ell]\M$ is $(n_{\ell - j})_{j \in \Z}$. 

By Proposition \ref{tvrz_totalspace}, one can view $T[\ell]\M$ as its total space graded manifold of the graded dimension $(n_{j} + n_{j-\ell})_{j \in \Z}$, whose underlying manifold is an ordinary vector bundle $E$, where
\begin{equation}
E := \ul{T[\ell]\M}_{0} \cong \ul{T\M}_{\ell},
\end{equation}
see Example \ref{ex_pullbackbyiM} for the definition. Its rank is $n_{-\ell}$. 

In the literature, people have usually considered the case $\ell = 1$. See e.g. \cite{mehta2006supergroupoids, Kontsevich:1997vb, 2011RvMaP..23..669C}. This works flawlessly for non-negatively graded manifolds. However, in general, $\Omega_{\M} := \C^{\infty}_{T[1]\M}$ is not a sheaf over manifold $M$, as $E$ can have a non-zero rank. To get out of the trouble, we resort to a dirty little trick: simply choose $\ell$ large enough to avoid it entirely. 

\begin{definice}  \label{def_diffforms}
Let $\M = (M,\C^{\infty}_{\M})$ be a graded manifold of the graded dimension $(n_{j})_{j \in \Z}$. Let $s \in \N_{0}$ be the smallest non-negative \textit{even} integer, such that $n_{-s'} = 0$ for all $s' > s$. 
Since $\sum_{j \in \Z} n_{j} < \infty$, $s$ is well-defined and unique. Define the \textbf{sheaf $\Omega_{\M}$ of differential forms on $\M$} as the algebra of functions on the degree shifted tangent bundle $T[1+s]\M$, that is 
\begin{equation}
\Omega_{\M} := \C^{\infty}_{T[1+s]\M}.
\end{equation}
\end{definice}
\begin{tvrz}
$\Omega_{\M}$ has the properties (i) - (iii) listed above.
\end{tvrz}
\begin{proof}
The underlying vector bundle of $T[1+s]\M$ is canonically identified with $\1_{M}: M \rightarrow M$, whence $\Omega_{\M} \in \Sh(M,\gcAs)$. This fulfills the requirement (i). 

Next, we have a graded smooth map $\pi: T[1+s]\M \rightarrow \M$. As noted in Remark \ref{rem_gradedvectorbundlesdifferently}, the pushforward sheaf $\pi_{\ast} \Omega_{\M}$ has a natural structure of a sheaf of graded $\C^{\infty}_{\M}$-modules. But as $\ul{\pi} = \1_{M}$, we have $\pi_{\ast} \Omega_{\M} = \Omega_{\M}$. This shows that $\Omega_{\M}$ has the required property (ii). More precisely, the action of $f \in \C^{\infty}_{\M}(U)$ on $\omega \in \Omega_{\M}(U)$ is $f \tr \omega := \pi^{\ast}_{U}(f) \cdot \omega$. 

Finally, for any $f \in \C^{\infty}_{\M}(U)$ and $\omega,\omega' \in \Omega_{\M}(U)$, one finds

\begin{equation}
f \tr (\omega \cdot \omega') = (f \tr \omega) \cdot \omega' = (-1)^{|f||\omega|} \omega \cdot (f \tr \omega').
\end{equation}
This shows that the the structure of graded commutative associative algebra on $\Omega_{\M}(U)$ is compatible with the structure of a graded $\C^{\infty}_{\M}(U)$-module and $\Omega_{\M}$ has the property (iii). 
\end{proof}

Next, let us talk about suitable coordinate functions on $T[1+s]\M$. Let $\A = \{ (U_{\alpha}, \varphi_{\alpha}) \}$ be a graded smooth atlas for $\M$. For each $\alpha \in I$, it is convenient to use the unified notation $(\bbz^{A}_{\alpha})_{A=1}^{n}$ for all the coordinate functions in $\C^{\infty}_{\M}(U_{\alpha})$ corresponding to the standard coordinate functions $(\bbz^{A})_{A=1}^{n}$ in the graded domain, see Remark \ref{rem_unifiedcoordfunct}. We have shown in Proposition \ref{tvrz_vfields} that the collection of coordinate vector fields $( \frac{\partial}{\partial \bbz^{A}_{\alpha}} )_{A=1}^{n}$ forms a local frame for $T\M$ over $U_{\alpha}$. Let $(\vartheta_{A})_{A=1}^{n}$ be the standard total basis of $\R^{(n_{-j})}$, ordered so that $|\vartheta_{A}| = | \frac{\partial}{\partial \bbz^{A}_{\alpha}}| = -|\bbz^{A}|$. 

Define the local trivialization $\{ (U_{\alpha}, \phi_{\alpha}) \}_{\alpha \in I}$ of $T\M$ to satisfy the equation
\begin{equation} (-1)^{|\bbz^{A}|} \frac{\partial}{\partial \bbz^{A}_{\alpha}} = (\phi_{\alpha})^{-1}_{U_{\alpha}}(1 \otimes \vartheta_{A}),
\end{equation}
for each $A \in \{1,\dots,n\}$. Note that this local trivialization corresponds to the slightly modified local frame $((-1)^{|\bbz^{A}|} \frac{\partial}{\partial \bbz^{A}_{\alpha}})_{A=1}^{n}$, see Proposition \ref{tvrz_globalframe}. This choice will be justified later. What are its transition functions? We have answered this question in Proposition \ref{tvrz_coordinatevftransrule}, hence
\begin{equation}
(G_{\alpha \beta})^{A}{}_{B} = (-1)^{|\bbz^{B}| - |\bbz^{A}|} \frac{\partial \bbz^{A}_{\alpha}}{\partial \bbz^{B}_{\beta}}.
\end{equation}
One can view these as transition functions corresponding to the local trivialization $\{ (U_{\alpha}, \phi'_{\alpha}) \}_{\alpha \in I}$ of $T[1+s]\M$ and $(\vartheta_{A})_{A=1}^{n}$, viewed as a total basis of $\R^{(n_{-j})}[1+s] \equiv \R^{(n_{- (j + 1 + s)})}$, see Proposition \ref{tvrz_shiftedtrans}. Now, we have the induced graded smooth atlas $\B = \{ (U_{\alpha}, \rho_{\alpha}) \}_{\alpha \in I}$ on $T[1+s]\M$ (recall that $\ul{\pi} = \1_{M}$). For each $\alpha \in I$, it gives us the base manifold coordinate functions $(\bbz^{A}_{\alpha})_{A=1}^{n}$ and the fiber coordinate functions $(\Xi^{A}_{\alpha})_{A=1}^{n}$, see Subsection \ref{subsec_totalspace}. It follows from (\ref{eq_transmaptot2}) that they transform as 
\begin{equation}
\Xi^{A}_{\beta}|_{U_{\alpha \beta}} = \pi^{\ast}_{U_{\alpha \beta}}( (G^{\vee}_{\alpha \beta})_{B}{}^{A}) \cdot \Xi^{B}_{\alpha}|_{U_{\alpha \beta}},
\end{equation}
where $(G^{\vee}_{\alpha \beta})_{B}{}^{A}$ are the transition functions corresponding to the induced local trivialization $\{ (U_{\alpha}, \phi'^{\vee}_{\alpha})\}_{\alpha \in I}$ of $(T[1+s]\M)^{\ast}$. It follows from (\ref{eq_GalphabetaVeecomps}) that 
\begin{equation}
(G^{\vee}_{\alpha \beta})_{B}{}^{A} = (-1)^{|\vartheta_{A}|'( |\vartheta_{B}|' - |\vartheta_{A}|')} (G_{\beta \alpha})^{A}{}_{B},
\end{equation}
where $|\vartheta_{A}|' = |\vartheta_{A}| - (1 + s) = -(|\bbz^{A}| + 1 + s)$ is the degree of the element $\vartheta_{A}$ in the shifted graded vector space $\R^{(n_{-j})}[1+s]$. Altogether, we find 
\begin{equation} \label{eq_Xistransformshiftedtangnet}
\Xi^{A}_{\beta}|_{U_{\alpha \beta}} = (-1)^{|\bbz^{A}|( |\bbz^{B}| - |\bbz^{A}|)} \frac{\partial \bbz^{A}_{\beta}}{\partial \bbz^{B}_{\alpha}} \tr \Xi^{B}_{\alpha}|_{U_{\alpha \beta}},
\end{equation}
where $\tr$ is the action of $\C^{\infty}_{\M}(U_{\alpha \beta})$ on $\Omega_{\M}(U_{\alpha \beta})$. In other words, we have fiddled with the signs so that $(\Xi^{A}_{\alpha})_{A=1}^{n}$ transforms as the local frame of $T^{\ast}\M$ dual to $( \frac{\partial}{\partial \bbz^{A}_{\alpha}})_{A=1}^{n}$. Consequently, we introduce the notation $\dr \bbz^{A}_{\alpha} := \Xi^{A}_{\alpha}$. It is also convenient to rewrite (\ref{eq_Xistransformshiftedtangnet}) using the corresponding right graded module action $\tl$, see the text under (\ref{eq_moduleaxioms}), to obtain 
\begin{equation} \label{eq_ddztransform}
\dr \bbz^{A}_{\beta}|_{U_{\alpha \beta}} = \dr \bbz^{B}_{\alpha}|_{U_{\alpha \beta}} \tl \frac{\partial \bbz^{A}_{\beta}}{\partial \bbz^{B}_{\alpha}}.
\end{equation} 
This allows us to get rid of the annoying signs. For each $\alpha \in I$, the coordinate functions $\dr \bbz^{A}_{\alpha}$ are called the \textbf{coordinate $1$-forms on $\M$} corresponding to the graded local chart $(U_{\alpha},\varphi_{\alpha})$. In the following, we will usually replace both $\tr$ and $\tl$ by the ordinary multiplication sign $\cdot$ (this will be justified later) and omit the explicit writing of the restrictions of the coordinate $1$-forms. Observe that the coordinate $1$-forms commute as it is expected of them (this explains the degree shift):
\begin{equation}
\dr \bbz^{A}_{\alpha} \cdot \dr \bbz^{B}_{\alpha} = (-1)^{(|\bbz^{A}|+1)(|\bbz^{B}|+1)} \dr \bbz^{B}_{\alpha} \cdot \dr \bbz^{A}_{\alpha}.
\end{equation}
Note that the integer $s$ was chosen to ensure that $|\dr{\bbz}^{A}_{\alpha}| = |\bbz_{\alpha}^{A}| + 1 + s> 0$ for all $A \in \{1,\dots,n\}$. 
\subsection{Subsheaves of p-forms}
For the purpose of the next definition, let us consider the following set:
\begin{equation}
\ul{\N}^{n}_{p} := \{ (q_{1},\dots,q_{n}) \in (\N_{0})^{n} \; | \; q_{1} + \dots + q_{n} = p \text{ and } q_{A} \in \{0,1\} \text{ for } |\bbz^{A}| \text{ even }\}.
\end{equation}
It is defined so that for each $(q_{1},\dots,q_{n}) \in \ul{\N}^{n}_{p}$, the expression $(\dr \bbz^{1}_{\alpha})^{q_{1}} \cdots (\dr \bbz^{n}_{\alpha})^{q_{n}}$ contains precisely $p$ coordinate $1$-forms and it is non-zero. 
\begin{definice}
Let $\M = (M, \C^{\infty}_{\M})$ be a graded manifold with a graded smooth atlas $\A = \{ (U_{\alpha}, \varphi_{\alpha}) \}_{\alpha \in I}$ inducing the coordinate $1$-forms $(\dr \bbz^{A}_{\alpha})_{A=1}^{n}$. Let $U \in \Op(M)$ and $p \in \N_{0}$ be arbitrary. We say that $\omega \in \Omega_{\M}(U)$ is a \textbf{differential $p$-form on $\M$}, if for every $\alpha \in I$, $\omega|_{U \cap U_{\alpha}}$ can be written as a finite sum
\begin{equation} \label{eq_pformdefinition}
\omega|_{U \cap U_{\alpha}} = \hspace{-3mm} \sum_{(q_{1},\dots,q_{n}) \in \ul{\N}^{n}_{p}} \hspace{-3mm} \omega^{(\alpha)}_{q_{1} \dots q_{n}} \cdot  (\dr \bbz^{1}_{\alpha})^{q_{1}} \cdots (\dr \bbz^{n}_{\alpha})^{q_{n}},
\end{equation}
where $\omega^{(\alpha)}_{q_{1} \dots q_{n}} \in \C^{\infty}_{\M}(U \cap U_{\alpha})$ are (unique) \textbf{component functions} of $\omega$. We write $\omega \in \Omega^{p}_{\M}(U)$. 
\end{definice}
It follows from (\ref{eq_ddztransform}) that the concept of a $p$-form is well-defined. Let us summarize some expected properties of $\Omega^{p}_{\M}$ in the form of a proposition.
\begin{tvrz} \label{tvrz_pforms}
For any $p \in \N_{0}$, $\Omega^{p}_{\M} \subseteq \Omega_{\M}$ is a subsheaf of graded $\C^{\infty}_{\M}$-submodules, called the \textbf{sheaf of differential $p$-forms}.
Moreover, it has the following properties:
\begin{enumerate}[(i)]
\item For $\omega \in \Omega^{p}_{\M}(U)$ and $\omega' \in \Omega^{q}_{\M}(U)$, one has $\omega \cdot \omega' \in \Omega^{p+q}_{\M}(U)$. 
\item There is a convenient alternative grading on $\Omega^{p}_{\M}$. For any $\omega \in \Omega^{p}_{\M}(U)$, define
\begin{equation} \label{eq_formdegree}
\deg(\omega) := |\omega| - p(1+s).
\end{equation}
If $\omega' \in \Omega^{q}_{\M}(U)$, one has $\deg(\omega \cdot \omega') = \deg(\omega) + \deg(\omega')$, and 
\begin{equation}
\omega \cdot \omega' = (-1)^{(\deg(\omega)+p)(\deg(\omega')+q)} \omega' \cdot \omega. 
\end{equation} 
\item There are canonical $\C^{\infty}_{\M}$-linear sheaf isomorphisms $\Omega^{0}_{\M} \cong \C^{\infty}_{\M}$ and $\Omega^{1}_{\M} \cong \X_{\M}^{\ast} ( \cong \Omega^{1}_{T\M})$. In the latter case, we assume the alternative grading from $(ii)$.
\item There is a canonical $\C^{\infty}_{\M}$-linear sheaf isomorphism $\Omega_{\M} \cong \prod_{p=0}^{\infty} \Omega_{\M}^{p}$. 

In other words, for each $U \in \Op(M)$, $\omega \in \Omega_{\M}(U)$ can be written as a unique sequence $\omega = (\omega_{p})_{p=0}^{\infty}$, where $\omega_{p} \in \Omega^{p}_{\M}(U)$. Moreover, if $f \in \C^{\infty}_{\M}(U)$, then $f \cdot \omega = (f \cdot \omega_{p})_{p=0}^{\infty}$. For $V \subseteq U$, one has $\omega|_{V} = (\omega_{p}|_{V})$, and if $\omega' \in \Omega_{\M}(U)$ is another differential form, one has $\omega \cdot \omega' = ( \sum_{k=0}^{p} \omega_{k} \cdot \omega'_{p-k})_{p=0}^{\infty}$ 
\end{enumerate}
\end{tvrz}
\begin{proof}
It is easy to see that $\Omega^{p}_{\M}$ is a subpresheaf of graded $\C^{\infty}_{\M}$-submodules. To prove that it is a subsheaf, we must prove that for any open cover $\{ V_{\mu} \}_{\mu \in J}$ of $U \in \Op(M)$ and $\omega \in \Omega_{\M}(U)$, the property $\omega|_{V_{\mu}} \in \Omega^{p}_{\M}(V_{\mu})$ for all $\mu \in J$ implies $\omega \in \Omega^{p}_{\M}(U)$. For each $\alpha \in I$, one thus has 
\begin{equation}
\omega|_{V_{\mu} \cap U_{\alpha}} = \hspace{-3mm} \sum_{(q_{1},\dots,q_{n}) \in \ul{\N}^{n}_{p}} \hspace{-3mm} \omega^{(\alpha,\mu)}_{q_{1} \dots q_{n}} \cdot (\dr \bbz^{1}_{\alpha})^{q_{1}} \cdots (\dr \bbz^{n}_{\alpha})^{q_{n}}
\end{equation}
It follows that for each $(q_{1},\dots,q_{n}) \in \ul{\N}^{n}_{p}$, the of functions $\{ \omega^{(\alpha,\mu)}_{q_{1},\dots,q_{n}} \}_{\mu \in J}$ agree on the overlaps of the open cover $\{ V_{\mu} \cap U_{\alpha} \}_{\mu \in J}$ of $U_{\alpha}$, Consequently, there exists a unique function $\omega^{(\alpha)}_{q_{1},\dots,q_{n}} \in \C^{\infty}_{\M}(U_{\alpha})$, such that $\omega^{(\alpha)}_{q_{1} \dots q_{n}}|_{V_{\mu} \cap U_{\alpha}} = \omega^{(\alpha,\mu)}_{q_{1} \dots q_{n}}$ for each $\mu \in J$. It now follows from the monopresheaf property of $\Omega_{\M}$ that $\omega|_{U \cap U_{\alpha}}$ must be given by the sum
\begin{equation}
\omega|_{U \cap U_{\alpha}} = \hspace{-3mm} \sum_{(q_{1},\dots,q_{n}) \in \ul{\N}^{n}_{p}} \hspace{-3mm} \omega^{(\alpha)}_{q_{1} \dots q_{n}} \cdot (\dr \bbz^{1}_{\alpha})^{q_{1}} \cdots (\dr \bbz^{n}_{\alpha})^{q_{n}}.
\end{equation}
As $\alpha \in I$ was arbitrary, this proves that $\omega \in \Omega^{p}_{\M}(U)$. Hence $\Omega^{p}_{\M}$ is a subsheaf of graded $\C^{\infty}_{\M}$-submodules. The proof of properties $(i)$ and $(ii)$ is just a direct verification. 

To see (iii), note that $\Omega^{0}_{\M}$ is precisely the image of the injective $\C^{\infty}_{\M}$-linear sheaf morphism $\pi^{\ast}: \C^{\infty}_{\M} \rightarrow \Omega_{\M}$, defined by the projection $\pi: T[1+s]\M \rightarrow \M$. This also justifies why we can treat functions on $\M$ as $0$-forms and identify the actions $\tr$ and $\tl$ with the multiplication in $\Omega_{\M}$. Next, recall that we have already noted that $( \dr \bbz^{A}_{\alpha})_{A=1}^{n}$ transform on the overlaps in the same way as the local frame $(\Phi^{A}_{(\alpha)})_{A=1}^{n}$ for $T^{\ast}\M$ dual to $(\frac{\partial}{\partial \bbz^{A}_{\alpha}})_{A=1}^{n}$. Let $\omega \in \Omega^{1}_{\M}(U)$. We can write
\begin{equation}
\omega|_{U \cap U_{\alpha}} := \omega^{(\alpha)}_{A} \cdot \dr \bbz^{A}_{\alpha},
\end{equation}
for each $\alpha \in I$, where $\omega_{A}^{(\alpha)} \in \C^{\infty}_{\M}(U \cap U_{\alpha})$ are unique. The idea presents itself, we define 
\begin{equation} \label{eq_fPsi1formsdualiso}
\fPsi_{U}(\omega)|_{U \cap U_{\alpha}} := \omega^{(\alpha)}_{A} \Phi^{A}_{(\alpha)}|_{U \cap U_{\alpha}} \in \X^{\ast}_{\M}(U \cap U_{\alpha}). 
\end{equation}
Note that $| \fPsi_{U}(\omega)| = |\omega_{A}^{(\alpha)}| + |\Phi^{A}_{(\alpha)}| = |\omega_{A}^{(\alpha)}| + |\bbz^{A}| = \deg(\omega)$. It is easy to see that the sections on the right-hand side agree on the overlaps, whence they define a unique element $\fPsi_{U}(\omega) \in \X^{\ast}_{\M}(U)$. The $\C^{\infty}_{\M}$-linearity and naturality in $U$ is obvious, and we can define the inverse $\fPsi_{U}^{-1}$ in the same way, this time using the gluing property of the sheaf $\Omega^{1}_{\M}$. Note that one can also consider the original grading on $\Omega^{1}_{\M}$, which makes $\fPsi_{U}$ into a sheaf morphism of degree $-(1+s)$. However, one then has to add a sign $(-1)^{|\omega|}$ to the formula (\ref{eq_fPsi1formsdualiso}) to make it $\C^{\infty}_{\M}$-linear. 

Finally, let $\omega \in \Omega_{\M}(U)$. We will now assign a unique sequence $(\omega_{p})_{p=0}^{\infty}$ to it, where $\omega_{p} \in \Omega^{p}_{\M}(U)$ for each $p \in \N_{0}$. Let $\B = \{ (U_{\alpha},\rho_{\alpha}) \}_{\alpha \in I}$ be the graded smooth atlas on $T[1+s]\M$ induced by the graded smooth atlas $( U_{\alpha}, \varphi_{\alpha})_{\alpha \in I}$ on $\M$. It follows that each local representative $\hat{\omega}_{\alpha} \in \C^{\infty}_{(N_{j})}(\ul{\rho_{\alpha}}(U \cap U_{\alpha}))$ is a formal power series in graded coordinates which can certainly be reordered as 
\begin{equation}
\hat{\omega}_{\alpha} = \sum_{p=0}^{\infty} \sum_{(q_{1},\dots,q_{n}) \in \ul{\N}^{n}_{p}} \hspace{-3mm} \hat{\omega}^{(\alpha,p)}_{q_{1} \dots q_{n}} \cdot (\dr \bbz^{1})^{q_{1}} \cdots (\dr \bbz^{n})^{q_{n}},
\end{equation}
for unique functions $\hat{\omega}^{(\alpha,p)}_{q_{1} \dots q_{n}} \in \C^{\infty}_{(n_{j})}( \ul{\phi_{\alpha}}(U \cap U_{\alpha}))$. We write $\dr \bbz^{A}$ for the fiber coordinates and again identify functions on the base with their pullback upstairs by the projection. Here $N_{j} = n_{j} + n_{j-(1+s)}$. It follows from the definition of the transition maps of $\B$ that the functions
\begin{equation}
\omega_{p}|_{U \cap U_{\alpha}} := \hspace{-3mm} \sum_{(q_{1},\dots,q_{n}) \in \ul{\N}^{n}_{p}} \hspace{-3mm} (\phi_{\alpha})^{\ast}_{\ul{\phi_{\alpha}}(U \cap U_{\alpha})}(\hat{\omega}^{(\alpha,p)}_{q_{1} \dots q_{n}}) \cdot ( \dr \bbz^{1})^{q_{1}} \cdots  (\dr \bbz^{n})^{q_{n}} \in \Omega^{p}_{\M}(U \cap U_{\alpha}).
\end{equation}
agree on the overlaps, hence they define a unique $p$-form $\omega_{p} \in \Omega^{p}_{\M}(U)$. It is now straightforward to verify that $\omega \mapsto (\omega_{p})_{p \in \Z}$ is a bijection having all the properties described in $(iv)$. 
\end{proof}
\subsection{Cartan calculus}
The definition of $\Omega_{\M}$ as a sheaf of functions on the graded manifold $T[1+s]\M$ allows one to easily introduce all the standard operations on differential forms. In the entire subsection, we will use the graded smooth atlas $\B = \{ (U_{\alpha}, \rho_{\alpha}) \}_{\alpha \in I}$ on $T[1+s]\M$ induced by a graded smooth atlas $\A = \{ (U_{\alpha}, \varphi_{\alpha}) \}_{\alpha \in I}$ on $\M$. For each $\alpha \in I$, we thus have the base manifold coordinate functions $(\bbz^{A}_{\alpha})_{A=1}^{n}$ and the fiber coordinate functions $(\dr \bbz^{A}_{\alpha})_{A=1}^{n}$, see Subsection \ref{subsec_diffforms}. 

\begin{tvrz} \label{tvrz_deRham}
There exists a canonical degree $1 + s$ vector field $\dr \in \X_{T[1+s]\M}$. It is homological, that is $\dr^{2} = \frac{1}{2}[\dr, \dr] = 0$. For each $p \in \N_{0}$, it restricts to a map $\dr: \Omega^{p}_{\M}(M) \rightarrow \Omega^{p+1}_{\M}(M)$ and for each $\omega \in \Omega^{p}_{\M}(M)$, one has $\deg( \dr \omega) = \deg(\omega)$. $\dr$ is called the \textbf{exterior derivative} or the \textbf{de Rham differential}. We will use the same symbol $\dr$ for any restriction $\dr|_{U}$, $U \in \Op(M)$. 
\end{tvrz}
\begin{proof}
For each $\alpha \in I$, we will define $\dr$ by its restriction to $U_{\alpha}$, namely set 
\begin{equation}
\dr|_{U_{\alpha}} := \dr \bbz^{A}_{\alpha} \cdot \frac{\partial}{\partial \bbz^{A}_{\alpha}}. 
\end{equation}
It follows immediately from Proposition \ref{tvrz_coordinatevftransrule} and (\ref{eq_ddztransform}) that the vector fields on the right-hand side agree on the overlaps, hence they define a global vector field $\dr \in \X_{T[1+s]\M}(M)$. All the other statements are trivial to verify. 
\end{proof}

\begin{cor}
For each $U \in \Op(M)$ and $f \in \C^{\infty}_{\M}(U)$, one can define its \textbf{differential} $\dr{f} \in \Omega^{1}_{\M}(U)$. One has $\deg(\dr{f}) = |f|$. Explicitly, for each $\alpha \in I$, one has 
\begin{equation} \label{eq_differentiallocal}
\dr{f}|_{U \cap U_{\alpha}} = \dr \bbz^{A}_{\alpha} \cdot \frac{\partial f}{\partial \bbz^{A}_{\alpha}}.
\end{equation}
Note that the coordinate $1$-forms $(\dr \bbz^{A}_{\alpha})_{A=1}^{n}$ are then precisely the differentials of the coordinate functions $(\bbz^{A}_{\alpha})_{A=1}^{n}$ (although this statement is admittedly a bit tautological).
\end{cor}
Functions with vanishing differential have to be quite special in ordinary differential geometry. The same is true in the graded setting. 
\begin{tvrz} \label{tvrz_dfis0consequences}
Let $f \in \C^{\infty}_{\M}(U)$ satisfy $\dr{f} = 0$. Then for any connected component $U_{0}$ of $U$, $f|_{U_{0}} = \lambda$ for some constant $\lambda \in \R$. In particular, for $|f| \neq 0$, one has $f = 0$. 
\end{tvrz}
\begin{proof}
For each $\alpha \in I$, let $\hat{f}_{\alpha} \in \C^{\infty}_{(n_{j})}( \ul{\phi_{\alpha}}(U \cap U_{\alpha}))$ be the local representative of $f$. By assumption, if $(x^{1},\dots,x^{n_{0}})$ and $(\xi_{1},\dots,\xi_{n_{\ast}})$ are the coordinate functions on the graded domain, we have
\begin{equation}
\frac{\partial \hat{f}_{\alpha}}{\partial x_{i}} = 0, \; \; \frac{\partial \hat{f}_{\alpha}}{\partial \xi_{\mu}} = 0, 
\end{equation}
for all $i \in \{1,\dots,n_{0}\}$ and all $\mu \in \{1, \dots, n_{\ast} \}$. The formula (\ref{eq_partialximuformula}) now implies that $(\hat{f}_{\alpha})_{\fp} = 0$ for all $\fp \in \N^{n_{\ast}}_{|f|} - \{ \mathbf{0} \}$, whereas (\ref{eq_partialxiformula}) shows that $(\hat{f}_{\alpha})_{\mathbf{0}}$ is locally constant. For each $m \in U_{\alpha}$, there is thus a connected $V_{m} \in \Op_{m}(U \cap U_{\alpha})$, such that $\hat{f}_{\alpha}|_{\ul{\phi_{\alpha}}(V_{m})} = \lambda_{m}$ for some constant $\lambda_{m} \in \R$. As $\alpha \in I$ was arbitrary, we conclude that each $m \in U$ has a connected neighborhood $V_{m} \in \Op_{m}(U)$, such that $f|_{V_{m}} = \lambda_{m}$ for some constant $\lambda_{m} \in \R$. If $U_{0}$ is any connected component of $U$, $\lambda_{m} =: \lambda$ is the same for all $m \in U_{0}$ (using the standard ``connecting path is compact'' argument). As $\{ V_{m} \}_{m \in U_{0}}$ forms the open cover of $U_{0}$, we have $f|_{U_{0}} = \lambda$ by the monopresheaf property. 
\end{proof}
Let us continue with the next standard operation.
\begin{tvrz} \label{tvrz_interior}
For every vector field $X \in \X_{\M}(M)$, there exists a canonical a canonical degree $|X|-(1+s)$ vector field $i_{X} \in \X_{T[1+s]\M}(M)$. For any other $Y \in \X_{\M}(M)$, it satisfies $[i_{X},i_{Y}] = 0$. For each $p \in \N_{0}$, it restricts to a map $i_{X}: \Omega^{p}_{\M}(M) \rightarrow \Omega^{p-1}_{\M}(M)$ and for each $\omega \in \Omega^{p}_{\M}(M)$, one has $\deg(i_{X}\omega) = \deg(\omega) + |X|$. For each $f \in \C^{\infty}_{\M}(M)$, one has $i_{f X} = f i_{X}$ and for any $U \in \Op(X)$, one has $i_{X}|_{U} = i_{X|_{U}}$. The vector field $i_{X}$ is called the \textbf{interior product (with $X$)}.
\end{tvrz}
\begin{proof}
For each $\alpha \in I$, one can write $X|_{U_{\alpha}} = X^{A}_{(\alpha)} \frac{\partial}{\partial \bbz^{A}_{\alpha}}$ for unique functions $X^{A}_{(\alpha)} \in \C^{\infty}_{\M}(U_{\alpha})$. Let 
\begin{equation} \label{eq_interiorproduct}
i_{X}|_{U_{\alpha}} := X^{A}_{(\alpha)} \frac{\partial}{\partial (\dr \bbz^{A}_{\alpha})}.
\end{equation}
Note that due to (\ref{eq_ddztransform}), the coordinate vector fields on the right-hand side transform as 
\begin{equation}
\frac{\partial}{\partial( \dr \bbz^{B}_{\beta})}|_{U_{\alpha \beta}} = \frac{\partial \bbz^{A}_{\alpha}}{\partial \bbz^{B}_{\beta}} \frac{\partial}{\partial( \dr \bbz^{A}_{\alpha})}|_{U_{\alpha \beta}}.
\end{equation}
It follows that the vector fields on the right-hand side of (\ref{eq_interiorproduct}) agree on the overlaps, hence they define a unique vector field $i_{X} \in \X_{T[1+s]\M}(M)$. It is straightforward to verify the remaining properties and the proof is finished. 
\end{proof}
\begin{cor}
The sheaf isomorphism $\fPsi: \Omega^{1}_{\M} \rightarrow \X^{\ast}_{\M}$ obtained in Proposition \ref{tvrz_pforms}-$(iii)$ can be equivalently described as follows. For each $\omega \in \Omega^{1}_{\M}(U)$, $\fPsi_{U}(\omega) \in \X^{\ast}_{\M}(U) \cong \Omega^{1}_{T\M}(U)$ is a $\C^{\infty}_{\M}(U)$-linear map from $\X_{\M}(U)$ to $\C^{\infty}_{\M}(U)$. Then for each $X \in \X_{\M}(U)$, we have
\begin{equation}
[ \fPsi_{U}(\omega)](X) = (-1)^{\deg(\omega)(|X|-1)} i_{X}\omega
\end{equation}
\end{cor}
\begin{proof}
For each $\alpha \in I$, restrict both sides to $U \cap U_{\alpha}$, and compute. 
\end{proof}
Finally, let us introduce the last player of the Cartan calculus ensemble. 
\begin{tvrz} \label{tvrz_Lieder}
For every vector field $X \in \X_{\M}(M)$, there exists a canonical degree $|X|$ vector field $\Li{X} \in \X_{T[1+s]\M}(M)$. For any other $Y \in \X_{\M}(M)$, it satisfies the relations
\begin{equation} \label{eq_cartanrels}
\Li{X} = [i_{X},\dr], \; \; [\Li{X},\Li{Y}] = \Li{[X,Y]}, \; \; [\Li{X},\dr] = 0, \; \; [\Li{X},i_{Y}] = i_{[X,Y]}.
\end{equation}
For each $p \in \N_{0}$, it restricts to a map $\Li{X}: \Omega^{p}_{\M}(M) \rightarrow \Omega^{p}_{\M}(M)$ and for each $\omega \in \Omega^{p}_{\M}(M)$, one has $\deg(\Li{X}\omega) = \deg(\omega) + |X|$. For any $U \in \Op(X)$, one has $\Li{X}|_{U} = \Li{X|_{U}}$. The vector field $\Li{X}$ is called the \textbf{Lie derivative (along $X$)}.
\end{tvrz}
\begin{proof}
We define $\Li{X}$ be the first of the relations in (\ref{eq_cartanrels}), that is $\Li{X} := [i_{X}, \dr]$. We choose the order in the graded commutator to ensure that for any $f \in \C^{\infty}_{\M}(M) \cong \Omega^{0}_{\M}(M)$, one has 
\begin{equation}
\Li{X}(f) := i_{X}(\dr{f}) = X(f).
\end{equation}
This can be easily verified in local coordinates. Except for the relations (\ref{eq_cartanrels}), all of the properties of $\Li{X}$ follow immediately from the properties of $\dr$ and $i_{X}$, together with Proposition \ref{tvrz_gcommutator}. The coordinate expression for $\Li{X}$ can be calculated from the local expressions for  $\dr$ and $i_{X}$, giving
\begin{equation}
\Li{X}|_{U_{\alpha}} = X^{A}_{(\alpha)} \frac{\partial}{\partial \bbz^{A}_{\alpha}} + (-1)^{|X|} \dr{ \bbz^{B}_{\alpha}} \frac{\partial X^{A}_{(\alpha)}}{\partial \bbz^{B}_{\alpha}} \frac{\partial}{\partial( \dr \bbz^{A}_{\alpha})}. 
\end{equation}
It is easy to calculate the commutator of this vector field with $i_{Y}|_{U_{\alpha}}$, finding 
\begin{equation}
[\Li{X},i_{Y}]|_{U_{\alpha}} = (X^{B}_{(\alpha)} \frac{\partial Y^{A}_{(\alpha)}}{\partial \bbz^{B}_{\alpha}} - (-1)^{|X||Y|} Y^{B}_{(\alpha)} \frac{\partial X^{A}_{(\alpha)}}{\partial \bbz^{B}_{\alpha}}) \frac{\partial}{\partial( \dr \bbz^{A}_{\alpha})} = i_{[X,Y]}|_{U_{\alpha}}
\end{equation}
The remaining equations can be now easily proved using the graded Jacobi identity for the graded commutator, see Proposition \ref{tvrz_gcommutator}-$(ii)$-\textit{($\ell$3)}. Indeed, one has 
\begin{equation}
[\Li{X},\dr] = [[i_{X},\dr], \dr] = [i_{X}, [\dr,\dr]] - [[i_{X},\dr],\dr] = -[\Li{X},\dr],
\end{equation}
where we have used the fact that $\dr^{2} = \frac{1}{2}[\dr,\dr] = 0$. Hence $[\Li{X},\dr] = 0$. Finally, one finds
\begin{equation}
[\Li{X},\Li{Y}] = [\Li{X},[i_{Y},\dr]] = [[\Li{X},i_{Y}],\dr] + (-1)^{|X|(|Y|-1)} [i_{Y}, [\Li{X},\dr]] = [i_{[X,Y]}, \dr] = \Li{[X,Y]}. 
\end{equation}
This concludes the proof. 
\end{proof}
Note that using the alternative grading $\deg(\omega)$, one eliminates the ``auxiliary'' shift $s$ from all of the formulas. There is another way to introduce this grading.
\begin{example} \label{ex_Eulervfonforms}
Recall that on every graded manifold $\M$, we have the canonical Euler vector field $E \in \X_{\M}(M)$, see Example \ref{ex_Euler}. The corresponding Lie derivative has the local form
\begin{equation}
\Li{E}|_{U_{\alpha}} = |\bbz^{A}| \bbz^{A}_{\alpha} \frac{\partial}{\partial \bbz^{A}_{\alpha}} + |\bbz^{A}| \dr{ \bbz^{A}_{\alpha}} \frac{\partial}{\partial( \dr \bbz^{A}_{\alpha})}.
\end{equation}
This is \textit{not} the Euler field on $T[1+s]\M$. Instead, as $\deg(\dr \bbz^{A}_{\alpha}) = |\bbz^{A}|$, it is easy to see that 
\begin{equation} \label{eq_Eulervfonforms}
\Li{E} \omega = \deg(\omega) \omega,
\end{equation}
for any $\omega \in \Omega^{p}_{\M}(M)$. This can be viewed as an alternate definition of $\deg(\omega)$. 
\end{example}
\subsection{Morphisms of forms}
One of the main features of differential forms on ordinary manifolds is the notion of their pullback along smooth maps. In this subsection, we will make it possible also in the graded setting.
\begin{tvrz} \label{tvrz_formspullback}
Let $\phi: \cN \rightarrow \M$ be a graded smooth map. 

Then for each $p \in \N_{0}$, there is a canonical $\C^{\infty}_{\M}$-linear sheaf morphism $\phi^{\ast}: \Omega^{p}_{\M} \rightarrow \phi_{\ast} \Omega^{p}_{\cN}$, having the following properties:
\begin{enumerate}[(i)]
\item It preserves the degrees, that is $\deg( \phi^{\ast}_{U}(\omega)) = \deg(\omega)$ for any $U \in \Op(M)$ and $\omega \in \Omega^{p}_{\M}(U)$.
\item For $p = 0$, it coincides with the original pullback $\phi^{\ast}: \C^{\infty}_{\M} \rightarrow \ul{\phi}_{\ast} \C^{\infty}_{\cN}$.
\item It preserves the products. For each $U \in \Op(M)$, $\omega \in \Omega^{p}_{\M}(U)$ and $\omega' \in \Omega^{q}_{\M}(U)$, one has 
\begin{equation}
\phi^{\ast}_{U}(\omega \cdot \omega') = \phi^{\ast}_{U}(\omega) \cdot \phi^{\ast}_{U}(\omega'). 
\end{equation}
\item If $\psi: \cS \rightarrow \cN$ is another graded smooth map. Let $U \in \Op(M)$ be arbitrary. Then 
\begin{equation} (\phi \circ \psi)^{\ast}_{U} = \psi^{\ast}_{\ul{\phi}^{-1}(U)} \circ \phi^{\ast}_{U}, \; \; \1_{\M}^{\ast} = \1_{\Omega^{p}_{\M}}. \end{equation}
\item It commutes with $\dr$, that is for any $U \in \Op(M)$ and $\omega \in \Omega^{p}_{\M}(U)$, one has 
\begin{equation} \label{eq_pullbackcommuteswithd}
\dr( \phi^{\ast}_{U}(\omega)) = \phi^{\ast}_{U}( \dr \omega).
\end{equation}
\end{enumerate}
\end{tvrz}
\begin{proof}
Let $\A = \{ (U_{\alpha}, \varphi_{\alpha}) \}_{\alpha \in I}$ be a graded smooth atlas on $\M$, inducing for each $\alpha \in I$ the coordinate functions $(\bbz^{A}_{\alpha})_{A=1}^{n}$ and coordinate $1$-forms $\dr{ \bbz^{A}_{\alpha}}$ Let $U \in \Op(M)$ and $\omega \in \Omega^{p}_{\M}(U)$. For each $\alpha \in I$, we can thus write
\begin{equation}
\omega|_{U \cap U_{\alpha}} = \hspace{-3mm} \sum_{(q_{1},\dots,q_{n}) \in \ul{\N}^{n}_{p}} \hspace{-3mm} \omega^{(\alpha)}_{q_{1} \dots q_{p}} \cdot (\dr \bbz^{1}_{\alpha})^{q_{1}} \cdots (\dr \bbz^{n}_{\alpha})^{q_{n}},
\end{equation}
for unique functions $\omega^{(\alpha)}_{q_{1} \dots q_{n}} \in \C^{\infty}_{\M}(U \cap U_{\alpha})$. Let us write $\ol{\bbz}^{A}_{\alpha} := \phi^{\ast}_{U \cap U_{\alpha}}( \bbz^{A}_{\alpha}) \in \C^{\infty}_{\cN}( V_{\alpha})$, where $V_{\alpha} := \ul{\phi}^{-1}(U \cap U_{\alpha})$. For each $\alpha \in I$, we now define
\begin{equation} \label{eq_pullbackofform}
\phi^{\ast}_{U}(\omega)|_{V_{\alpha}} := \hspace{-3mm} \sum_{(q_{1},\dots,q_{n}) \in \ul{\N}^{n}_{p}} \hspace{-3mm} \phi^{\ast}_{U \cap U_{\alpha}}(\omega^{(\alpha)}_{q_{1} \dots q_{n}}) \cdot (\dr \ol{\bbz}^{1}_{\alpha})^{q_{1}} \cdots (\dr \ol{\bbz}^{n}_{\alpha})^{q_{n}} \in \Omega_{\cN}(V_{\alpha})
\end{equation} 
We claim that for each $(\alpha,\beta) \in I^{2}$, there holds the equation
\begin{equation}
\dr{\ol{\bbz}^{A}_{\beta}}|_{V_{\alpha \beta}} = \dr{ \ol{\bbz}^{B}_{\alpha}}|_{V_{\alpha \beta}} \cdot \phi^{\ast}_{U \cap U_{\alpha \beta}}( \frac{\partial \bbz^{A}_{\beta}}{\partial \bbz^{B}_{\alpha}}|_{U \cap U_{\alpha \beta}}). 
\end{equation}
But this can be proved easily by writing the differentials locally as (\ref{eq_differentiallocal}) using a graded smooth atlas on $\cN$, and then employing the chain rule (\ref{eq_chainrule}). As the functions $\dr \ol{\bbz}^{A}_{\alpha}$ commute in the same way as $\dr \bbz^{A}_{\alpha}$, it is now easy to see that the functions on the right-hand side of (\ref{eq_pullbackofform}) agree on the overlaps $V_{\alpha \beta}$, hence they define the function $\phi^{\ast}_{U}(\omega) \in \Omega_{\cN}(\ul{\phi}^{-1}(U))$. Note that it is obviously a $p$-form. It is easy to verify that $\phi^{\ast}_{U}$ is $\C^{\infty}_{\M}(U)$-linear (with respect to the induced action of $\C^{\infty}_{\M}$ on $\phi_{\ast} \Omega^{p}_{\cN}$, see Remark \ref{rem_pushforwardmodules}) and natural in $U$, hence it defines a $\C^{\infty}_{\M}$-linear sheaf morphism $\phi^{\ast}: \Omega^{p}_{\M} \rightarrow \phi_{\ast} \Omega^{p}_{\cN}$. It is straightforward to verify the properties $(i) - (iii)$. Finally, one can again use the chain rule (\ref{eq_chainrule}) to show that every $f \in \C^{\infty}_{\M}(U \cap U_{\alpha})$ satisfies
\begin{equation}
\dr( \phi^{\ast}_{U \cap U_{\alpha}}(f)) = \dr{\ol{\bbz}^{A}_{\alpha}} \cdot \phi^{\ast}_{U \cap U_{\alpha}}( \frac{\partial f}{\partial \bbz^{A}_{\alpha}}).
\end{equation}
Using this and (\ref{eq_pullbackofform}), for each $\alpha \in I$ we find
\begin{equation}
\begin{split}
\dr( \phi^{\ast}_{U}(\omega))|_{V_{\alpha}} = & \ \hspace{-3mm} \sum_{(q_{1},\dots,q_{n}) \in \ul{\N}^{n}_{p}} \hspace{-3mm} \dr(\phi^{\ast}_{U \cap U_{\alpha}}(\omega^{(\alpha)}_{q_{1} \dots q_{n}})) \cdot (\dr \ol{\bbz}^{1}_{\alpha})^{q_{1}} \cdots (\dr \ol{\bbz}^{n}_{\alpha})^{q_{n}} \\
= & \ \hspace{-3mm} \sum_{(q_{1},\dots,q_{n}) \in \ul{\N}^{n}_{p}} \hspace{-3mm} \dr{\ol{\bbz}^{A}_{\alpha}} \cdot \phi^{\ast}_{U \cap U_{\alpha}}( \frac{\partial \omega^{(\alpha)}_{q_{1} \dots q_{n}}}{\partial \bbz^{A}_{\alpha}}) \cdot (\dr \ol{\bbz}^{1}_{\alpha})^{q_{1}} \cdots (\dr \ol{\bbz}^{n}_{\alpha})^{q_{n}} \\
= & \ \phi^{\ast}_{U \cap U_{\alpha}}( \hspace{-3mm} \sum_{(q_{1},\dots,q_{n}) \in \ul{\N}^{n}_{p}} \hspace{-3mm} \dr \bbz^{A}_{\alpha} \cdot \frac{\partial \omega^{(\alpha)}_{q_{1} \dots q_{n}}}{\partial \bbz^{A}_{\alpha}} \cdot  (\dr \bbz^{1}_{\alpha})^{q_{1}} \cdots (\dr \bbz^{n}_{\alpha})^{q_{n}} ) \\
= & \ \phi^{\ast}_{U \cap U_{\alpha}}( \dr{ \omega|_{U \cap U_{\alpha}} }) = \phi^{\ast}_{U}( \dr \omega)|_{V_{\alpha}}.
\end{split}
\end{equation}
This proves the claim $(v)$. Finally, let $\psi: \cS \rightarrow \cN$ be another graded smooth map. For each $\alpha \in I$, let $W_{\alpha} := \ul{\psi}^{-1}(V_{\alpha}) = (\ul{\phi} \circ \ul{\psi})^{-1}(U \cap U_{\alpha})$. Using the naturality of $\psi^{\ast}$ and (\ref{eq_pullbackofform}), we find
\begin{equation}
\begin{split}
\psi^{\ast}_{\ul{\psi}^{-1}(U)}( \phi^{\ast}_{U}(\omega))|_{W_{\alpha}} = & \ \psi^{\ast}_{V_{\alpha}}( \hspace{-3mm} \sum_{(q_{1},\dots,q_{n}) \in \ul{\N}^{n}_{p}} \hspace{-3mm} \phi^{\ast}_{U \cap U_{\alpha}}(\omega^{(\alpha)}_{q_{1} \dots q_{n}}) \cdot (\dr \ol{\bbz}^{1}_{\alpha})^{q_{1}} \cdots (\dr \ol{\bbz}^{n}_{\alpha})^{q_{n}}) \\
= & \ \hspace{-3mm} \sum_{(q_{1},\dots,q_{n}) \in \ul{\N}^{n}_{p}} \hspace{-3mm} (\phi \circ \psi)^{\ast}_{U \cap U_{\alpha}}( \omega^{(\alpha)}_{q_{1} \dots q_{n}}) \cdot ( \psi^{\ast}_{V_{\alpha}}( \dr \ol{\bbz}^{1}_{\alpha}))^{q_{1}} \cdots ( \psi^{\ast}_{V_{\alpha}}( \dr \ol{\bbz}^{n}_{\alpha}))^{q_{n}} \\
= & \ \hspace{-3mm} \sum_{(q_{1},\dots,q_{n}) \in \ul{\N}^{n}_{p}} \hspace{-3mm} (\phi \circ \psi)^{\ast}_{U \cap U_{\alpha}}( \omega^{(\alpha)}_{q_{1} \dots q_{n}}) \cdot  ( \dr (\psi^{\ast}_{V_{\alpha}}( \ol{\bbz}^{1}_{\alpha})))^{q_{1}} \cdots ( \dr (\psi^{\ast}_{V_{\alpha}}( \ol{\bbz}^{n}_{\alpha})))^{q_{n}} \\
= & \ (\phi \circ \psi)^{\ast}_{U}(\omega)|_{W_{\alpha}}. 
\end{split}
\end{equation}
We have used the properties $(i) - (iii)$ and $(v)$. In the very last step, observe that $\psi^{\ast}_{V_{\alpha}}( \ol{\bbz}^{A}_{\alpha}) = (\phi \circ \psi)^{\ast}_{U \cap U_{\alpha}}( \bbz^{A}_{\alpha})$, so we have just used the definition (\ref{eq_pullbackofform}). This verifies the first claim of $(v)$. The fact that $\1_{\M}^{\ast} = \1_{\Omega^{p}_{\M}}$ is obvious and the proof is finished. 
\end{proof}
\begin{rem}
Note that in general, $\phi^{\ast}$ does not define a morphism of sheaves of graded algebras $\phi^{\ast}: \Omega_{\M} \rightarrow \ul{\phi}_{\ast} \Omega_{\cN}$, as it does not preserve the ``functional degree''. This is because $\Omega_{\M} = \C^{\infty}_{T[1+s]\M}$ and $\Omega_{\cN} = \C^{\infty}_{T[1+s'] \cN}$, where in general $s \neq s'$. However, the definition of the auxiliary degree shift in Definition \ref{def_diffforms} was only a matter of taste. For a given pair of manifolds $\M$ and $\cN$, we may always make choose $s = s'$ for all practical purposes. It then follows that $T[1+s]\phi := (\ul{\phi}, \phi^{\ast}): T[1+s]\cN \rightarrow T[1+s]\M$ defines a graded smooth map. The condition (\ref{eq_pullbackcommuteswithd}) then simply means that the two de Rham differentials (as vector fields) are $T[1+s]\phi$-related, see Definition \ref{def_relatedvf}.
\end{rem}

\begin{tvrz} \label{tvrz_inducedtangent}
\begin{enumerate}[(i)]
\item For every graded smooth map $\phi: \cN \rightarrow \M$, we obtain a canonical graded vector bundle morphism $T\phi: T\cN \rightarrow T\M$ over $\phi$. 
\item Let $n \in N$ be arbitrary. By $(i)$ and Example \ref{ex_fiberofgVB}, there is the induced graded linear map 
\begin{equation}
(T\phi)_{n}: (T\cN)_{n} \rightarrow (T\M)_{\ul{\phi}(n)}.
\end{equation}
Then $(T\phi)_{n}$ coincides with the differential $T_{n}\phi: T_{n}\cN \rightarrow T_{\ul{\phi}(n)}\M$, see Proposition \ref{tvrz_differential}, where we assume the identifications $(T\cN)_{n} \cong T_{n} \cN$ and $(T\M)_{\ul{\phi}(n)} \cong T_{\ul{\phi}(n)}\M$ as in Example \ref{ex_fiberoftangentistangent}.
\end{enumerate}
\end{tvrz}
\begin{proof}
Recall that we have a canonical sheaf isomorphism $\Omega^{1}_{\M} \cong \Gamma^{\ast}_{T\M}$, see Proposition \ref{tvrz_pforms}-$(iii)$. By Proposition \ref{tvrz_formspullback}, we thus have a $\C^{\infty}_{\M}$-linear sheaf morphism $\phi^{\ast}: \Gamma^{\ast}_{T\M} \rightarrow \phi_{\ast} \Gamma^{\ast}_{T\cN}$. Consequently, $T \phi := (\phi, \phi^{\ast})$ then defines a graded vector bundle morphism $T\phi: T\cN \rightarrow T\M$ (over $\phi)$. 

To prove the second claim, pick graded local charts $(V,\varphi)$ for $\cN$ and $(U, \psi)$ for $\M$, such that $n \in V$ and $\ul{\varphi}(V) \subseteq U$. Let $(\bby^{K})_{k=1}^{m}$ and $(\bbz^{A})_{A=1}^{n}$ be the corresponding coordinate functions. By definition and (\ref{eq_differentiallocal}), one obtains
\begin{equation} \label{eq_phiastonccord1forms}
\phi^{\ast}_{U}( \dr{\bbz}^{A})|_{V} = \dr{\bby^{K}} \cdot \frac{\partial( \phi^{\ast}_{U}(\bbz^{A})|_{V})}{\partial \bby^{K}} = (-1)^{|\bbz^{A}|( |\bby^{K}| - |\bbz^{A}|)} \frac{\partial( \phi^{\ast}_{U}(\bbz^{A})|_{V})}{\partial \bby^{K}} \cdot \dr{\bby}^{K}. 
\end{equation}
As noted in the proof of Proposition \ref{tvrz_pforms}-$(iii)$, the isomorphism $\Omega^{1}_{\M} \cong \Gamma^{\ast}_{T\M}$ identifies the coordinate $1$-forms $(\dr{\bbz}^{A})_{A=1}^{n}$ with the local frame for $\Gamma^{\ast}_{TM}$ dual to $( \frac{\partial}{\partial \bbz^{A}})_{A=1}^{n}$, hence (\ref{eq_phiastonccord1forms}) is precisely the local expression of the sheaf morphism $\phi^{\ast}: \Gamma^{\ast}_{T\M} \rightarrow \phi_{\ast} \Gamma^{\ast}_{T\cN}$ using the local frames. 

By tracking the definition of the fiber in Example \ref{ex_fiberofgVB} and the identifications $(T\cN)_{n}  \cong  T_{n}\cN$ and $(T\M)_{\ul{\phi}(n)} \cong T_{\ul{\phi}(n)} \M$ as in Example \ref{ex_fiberoftangentistangent}, one can use (\ref{eq_phiastonccord1forms}) to prove that 
\begin{equation}
(T\phi)_{n}( \frac{\partial}{\partial \bby^{K}}|_{n}) = \frac{\partial( \phi^{\ast}_{U}(\bbz^{A})|_{V})}{\partial \bby^{K}}(n) \frac{\partial}{\partial \bbz^{A}}|_{\ul{\phi}(n)},
\end{equation}
for all $J \in \{1,\dots,n\}$. But this is precisely the expression (\ref{eq_Tmphiaschain}) for $T_{n}\phi$. 
\end{proof}
\subsection{De Rham cohomology, Poincaré lemma}
For each $p \in \N_{0}$ and $k \in \Z$, one can consider a sheaf $\Omega^{p(k)}_{\M} \in \Sh(M,\Vect)$ of $p$-forms of degree $k$, defined for each $U \in \Op(M)$ as 
\begin{equation}
\Omega^{p(k)}_{\M}(U) := \{ \omega \in \Omega^{p}_{\M}(U) \; | \; \deg(\omega) = k \}. 
\end{equation}
With the original grading, this is nothing but the component sheaf $(\Omega^{p}_{\M})_{k+p(1+s)}$. It follows from Proposition \ref{tvrz_deRham} that for each $p \in \N_{0}$, the exterior derivative restricts to the linear map 
\begin{equation}
\dr: \Omega^{p(k)}_{\M}(M) \rightarrow \Omega^{p+1(k)}_{\M}(M).
\end{equation}
In this way, we obtain a cochain complex $(\Omega^{\bullet(k)}_{\M}(M), \dr)$ for each $k \in \Z$. 
\begin{definice}
For each $k \in \Z$, by the \textbf{degree $k$ de Rham cohomology $H^{\bullet(k)}_{\M}$ of $\M$}, we mean the cohomology of the cochain complex $(\Omega^{\bullet(k)}_{\M}(M), \dr)$. One  can thus form a single graded vector space $H^{\bullet}_{\M} := \{ H^{\bullet(k)}_{\M} \}_{k \in \Z}$ called the \textbf{graded de Rham cohomology of $\M$}.
\end{definice}

It follows immediately from Proposition \ref{tvrz_formspullback} that the assignment $\M \mapsto H^{\bullet}_{\M}$ defines a functor from $(\gMan^{\infty})^{\op}$ to $\gVect$. In particular, graded de Rham cohomology is possibly a good invariant. However, in \cite{roytenberg2002structure} it was observed that for $k \neq 0$, $H^{\bullet(k)}_{\M}$ happens to be ``rather'' trivial. 
\begin{tvrz} \label{tvrz_degreekdeRhamtrivial}
For every $k \neq 0$ and $p \in \N_{0}$, one has $H^{p(k)}_{\M} = 0$. 
\end{tvrz}
\begin{proof}
Let $k \neq 0$ be fixed. For $f \in \Omega^{0(k)}_{\M}(M) \cong \C^{\infty}_{\M}(M)_{k}$, the condition $\dr{f} = 0$ implies $f = 0$, see Proposition \ref{tvrz_dfis0consequences}. This shows that $H^{0(k)}_{\M} = 0$. Next, let $p > 0$ and suppose that $\omega \in \Omega^{p(k)}_{\M}(M)$ satisfies $\dr{\omega} = 0$. Using (\ref{eq_Eulervfonforms}), we can write
\begin{equation}
\omega = \frac{1}{\deg(\omega)} \Li{E}\omega = \frac{1}{\deg(\omega)}(i_{E}(\dr{\omega}) + \dr( i_{E}\omega)) = \dr( \frac{1}{\deg(\omega)} i_{E}\omega).
\end{equation}
But this shows that every closed $p$-form of degree $k$ is exact, hence $H^{p(k)}_{\M} = 0$. 
\end{proof}
It follows that $H^{\bullet(0)}_{\M}$ remains the only candidate for a non-trivial cohomology. However, as we shall now argue, it is also not particularly interesting. First, we obtain a statement for graded domains. Its proof follows the standard one, see e.g. Chapter I, $\mathsection$4 of \cite{bott2013differential}. 
\begin{tvrz}[\textbf{graded Poincaré lemma}] \label{tvrz_Poincare}
Let $(n_{j})_{j \in \Z}$ be any sequence, such that $\sum_{j \in \Z} n_{j} < \infty$. Let $U \in \Op(\R^{n_{0}})$ be arbitrary. Then there is a vector space isomorphism
\begin{equation} \label{eq_Poincare}
H^{\bullet(0)}_{U^{(n_{j})}} \cong H^{\bullet}_{U},
\end{equation}
where on the right-hand side, there is the ordinary de Rham cohomology of the open subset $U \subseteq \R^{n_{0}}$. In particular, if $U$ is contractible, then $H^{\bullet(0)}_{U^{(n_{j})}} = 0$. 
\end{tvrz}
\begin{proof}
Let $(x^{1},\dots,x^{n_{0}})$ and $(\xi_{1},\dots,\xi_{n_{\ast}})$ be the standard coordinates on $U^{(n_{j})}$. We may assume $n_{\ast} > 0$, otherwise the proof is already finished. And consider the graded domain $U^{(m_{j})}$, where $m_{|\xi_{1}|} = n_{|\xi_{1}|} - 1$ and $m_{j} = n_{j}$ for $j \neq |\xi_{1}|$. Let us write its standard coordinates as $(x^{1},\dots,x^{n_{0}})$ and $(\xi_{2},\dots,\xi_{n_{\ast}})$, that is we ``threw away'' the degree $|\xi_{1}|$ coordinate $\xi_{1}$. 

We can construct a pair of graded smooth maps $\pi: U^{(n_{j})} \rightarrow U^{(m_{j})}$ and $s: U^{(m_{j})} \rightarrow U^{(n_{j})}$, such that $\pi \circ s = \1_{U^{(m_{j})}}$. $\pi$ is the projection and $s$ is the ``zero section'', that is $s = (\1_{U}, s^{\ast})$, where
\begin{equation}
s^{\ast}_{U}(\xi_{1}) = 0, \; \; s^{\ast}_{U}(\xi_{\mu}) = \xi_{\mu} \text{ for } \mu \in \{2,\dots,n_{\ast} \}. 
\end{equation}
Let us write $\Omega_{(n_{j})}(U) := \Omega_{U^{(n_{j})}}(U)$. By Proposition \ref{tvrz_formspullback}, they induce a pair of $\C^{\infty}_{\M}(U)$-linear maps 
\begin{equation}
\pi^{\ast}: \Omega^{p(0)}_{(m_{j})}(U) \rightarrow \Omega^{p(0)}_{(n_{j})}(U), \; \; s^{\ast}: \Omega^{p(0)}_{(n_{j})}(U) \rightarrow \Omega^{p(0)}_{(m_{j})}(U),
\end{equation}
such that $s^{\ast} \circ \pi^{\ast} = \1$ and they commute with the exterior derivatives. The idea is to find a collection of linear maps $K_{p}: \Omega^{p(0)}_{(n_{j})}(U) \rightarrow \Omega^{p-1(0)}_{(n_{j})}(U)$, such that for each $p \in \N_{0}$, one has 
\begin{equation} \label{eq_Kpchainhomotopy}
\1 - \pi^{\ast} \circ s^{\ast} = K_{p+1} \circ \dr + \dr \circ K_{p}
\end{equation}
The construction of $K_{p}$ now depends quite significantly on the parity of $|\xi_{1}|$. 
\begin{enumerate}[(i)]
\item \textbf{$|\xi_{1}|$ is even}: First, let us establish the following notation. For any $f \in \C^{\infty}_{(n_{j})}(U)$, there is always a function $F \in \C^{\infty}_{(n_{j})}(U)$, such that $f = \frac{\partial F}{\partial \xi_{1}}$. Note that this is true only for even $|\xi_{1}|$. $F$ is not determined uniquely, but we can define the unique ``definite integral'' 
\begin{equation} \label{eq_integralofoverxi1}
\int_{0}^{\xi_{1}} f \dr{\xi}_{1} := (\1 - \pi^{\ast} \circ s^{\ast})(F). 
\end{equation}
By construction, one has $s^{\ast}( \int_{0}^{\xi_{1}} f \dr{\xi}_{1})  = 0$, $\frac{\partial}{\partial \xi_{1}}( \int_{0}^{\xi_{1}} f \dr{\xi}_{1}) = f$ and for $\mu > 1$, one has
\begin{equation}
\frac{\partial}{\partial \xi_{\mu}} \int_{0}^{\xi_{1}} f \dr{\xi}_{1} = (-1)^{|\xi_{\mu}||\xi_{1}|} \int_{0}^{\xi_{1}} \frac{\partial f}{\partial \xi_{\mu}} \dr{\xi}_{1}. 
\end{equation}
Every $p$-form in $\Omega^{p(0)}_{(n_{j})}(U)$ can be written as a finite sum of forms in the following two classes:
\begin{enumerate}[(t1)]
\item $\omega = f \cdot \pi^{\ast}(\hat{\omega})$ for $f \in \C^{\infty}_{(n_{j})}(U)$ and $\hat{\omega} \in \Omega^{p}_{(m_{j})}(U)$. Note that $|f| + \deg(\hat{\omega}) = 0$. 
\item $\omega = \dr{\xi}_{1} \cdot f \cdot \pi^{\ast}(\hat{\omega})$ for $f \in \C^{\infty}_{(n_{j})}(U)$ and $\hat{\omega} \in \Omega^{p-1}_{(m_{j})}(U)$. Note that $|f| + \deg(\hat{\omega}) + |\xi_{1}| = 0$. 
\end{enumerate}
It is vital that $|\dr{\xi}_{1}|$ is odd. One declares $K_{p}(\omega) = 0$ for $\omega$ of type $(\text{t}1)$. For type $(\text{t}2)$, set 
\begin{equation}
K_{p}( \dr{\xi}_{1} \cdot f \cdot \pi^{\ast}(\hat{\omega})) := (\int_{0}^{\xi_{1}} f \dr{\xi}_{1}) \cdot \pi^{\ast}(\hat{\omega}). 
\end{equation}
One now has to verify the condition (\ref{eq_Kpchainhomotopy}). It suffices to consider forms of type (t$1$) and (t$2$). The rest is a straightforward verification using the properties of the exterior derivative $\dr$ and those of the definite integral listed below (\ref{eq_integralofoverxi1}). 
\item \textbf{$|\xi_{1}|$ is odd}: First, note that every $p$-form in $\Omega^{p(0)}_{(n_{j})}(U)$ can be written as a finite sum of forms in the following two classes:
\begin{enumerate}[(t1)]
\item $\omega = (\xi_{1})^{r} \cdot \pi^{\ast}(\hat{\omega})$ for $r \in \{0,1\}$ and $\hat{\omega} \in \Omega^{p}_{(m_{j})}(U)$. Note that $r|\xi_{1}| + \deg(\hat{\omega}) = 0$.
\item $\omega = (\dr{\xi}_{1})^{q} \cdot (\xi_{1})^{r} \cdot \pi^{\ast}(\hat{\omega})$ for $q \in \{1,\dots,p\}$, $r \in \{0,1\}$ and $\hat{\omega} \in \Omega^{p-q}_{(m_{j})}(U)$. Note that the degree of $\hat{\omega}$ has to satisfy $(q+r)|\xi_{1}| + \deg(\hat{\omega}) = 0$.
\end{enumerate}
Now, on forms of type (t$1$), we declare $K_{p}(\omega) = 0$. On forms of type (t$2$), set 
\begin{equation}
K_{p}( (\dr{\xi}_{1})^{q} \cdot (\xi_{1})^{r} \cdot \pi^{\ast}(\hat{\omega})) := (\dr{\xi}_{1})^{q-1} \cdot (\xi_{1})^{r+1} \cdot \pi^{\ast}(\hat{\omega}). 
\end{equation}
It is straightforward to verify (\ref{eq_Kpchainhomotopy}). 
\end{enumerate}
The existence of a cochain homotopy (\ref{eq_Kpchainhomotopy}) ensures that the induced linear map $\hat{\pi}^{\ast}: H^{p(0)}_{U^{(m_{j})}} \rightarrow H^{p(0)}_{U^{(n_{j})}}$ defined by $\hat{\pi}^{\ast}([\omega]) := [ \pi^{\ast}(\omega)]$ is an isomorphism. In other words, we have just proved that 
\begin{equation}
H^{\bullet(0)}_{U^{(n_{j})}} \cong H^{\bullet(0)}_{U^{(m_{j})}}.
\end{equation}
By iterating this procedure, we obtain the isomorphism (\ref{eq_Poincare}). The claim for contractible $U$ follows from the Poincaré lemma for ordinary manifolds. 
\end{proof}
\begin{cor}[\textbf{Closed forms are locally exact}]
Let $\omega \in \Omega^{p(0)}_{\M}(M)$, such that $\dr{\omega} = 0$. Then for any $m \in M$, there exists $U \in \Op_{m}(M)$ and $\alpha \in \Omega^{p-1(0)}_{\M}(U)$, such that $\omega|_{U} = \dr{\alpha}$. 
\end{cor}
\begin{proof}
Pick a local chart $(U,\varphi)$ with contractible $U$. It induces the isomorphism $H^{\bullet(0)}_{\M|_{U}} \cong H^{\bullet(0)}_{\hat{U}^{(n_{j})}}$, where $\hat{U} := \ul{\varphi}(U)$. By the graded Poincaré lemma, one has $H^{\bullet(0)}_{\M|_{U}} \cong H^{\bullet}_{\hat{U}} = 0$. 
\end{proof}
In fact, we have proved that for any graded local chart $(U,\varphi)$, there is a vector space isomorphism $H^{\bullet(0)}_{\M|_{U}} \cong H^{\bullet}_{U}$. One suspects that globally, this is also true.
\begin{theorem} \label{thm_degreezerogrdeRhamisdeRham}
For any graded manifold $\M$, one has $H^{\bullet(0)}_{\M} \cong H^{\bullet}_{M}$. On the right-hand side, there is the ordinary de Rham cohomology of the smooth manifold $M$. 
\end{theorem}
\begin{proof}
Fix an open cover $\U := \{U_{\alpha} \}_{\alpha \in I}$ of $M$. For each $p,q \in \N_{0}$, one defines the vector space 
\begin{equation} \label{eq_Czechcocycles}
\czC^{q}( \Omega^{p(0)}_{\M}, \U) := \hspace{-5mm} \prod_{(\alpha_{0},\dots,\alpha_{q}) \in I^{q+1}} \hspace{-3mm} \Omega^{p(0)}_{\M}(U_{\alpha_{0} \dots \alpha_{q}}). 
\end{equation}
For each $p \in \N_{0}$, one obtains the \textit{Čech cochain complex} $( \czC^{\bullet}(\Omega^{p(0)}_{\M}, \U), \delta)$ corresponding to the sheaf $\Omega^{p(0)}_{\M}$ and the open cover $\U$. The Čech differential $\delta: \czC^{q}(\Omega^{p(0)}_{\M},\U) \rightarrow \czC^{q+1}( \Omega^{p(0)}_{\M},\U)$ is given by
\begin{equation} \label{eq_Cechdifferential}
(\delta \omega)_{\alpha_{0} \dots \alpha_{q+1}} := \sum_{i=0}^{q} (-1)^{i+1} {\omega_{\alpha_{0} \dots \hat{\alpha}_{i} \dots \alpha_{q}}}|_{U_{\alpha_{0} \dots \alpha_{q+1}}},
\end{equation} 
for all $(\alpha_{0},\dots,\alpha_{q+1}) \in I^{q+1}$ and $\omega \in \czC^{q}(\Omega^{p(0)}_{\M},\U)$. One also has a linear map $r: \Omega^{p(0)}_{\M}(M) \rightarrow \czC^{0}(\Omega^{p(0)}_{\M}, \U)$, given by $(r(\omega))_{\alpha} := r|_{U_{\alpha}}$, for each $\omega \in \Omega^{p(0)}_{\M}(M)$ and $\alpha \in I$. Since $\Omega^{p(0)}_{\M}$ is a sheaf and we have a partition of unity subordinate to $\U$, one can show that the sequence
\begin{equation}
\begin{tikzcd} \label{eq_firstSESMayerVietoris}
0 \arrow{r} & \Omega^{p(0)}_{\M}(M) \arrow{r}{r} & \czC^{0}( \Omega^{p(0)}_{\M}, \U) \arrow{r}{\delta} & \czC^{1}( \Omega^{p(0)}_{\M}, \U) \arrow{r}{\delta} & \cdots 
\end{tikzcd}
\end{equation}
is exact for each $p \in \N_{0}$. The proof is completely the same as the one of Proposition 8.5 in \cite{bott2013differential}. On the other hand, for each $q \in \N_{0}$, one can construct the sequence
\begin{equation} \label{eq_secondSESMayerVietoris}
\begin{tikzcd}
0 \arrow{r} & \czC^{q}( \R_{M}, \U) \arrow{r}{i} & \czC^{q}( \Omega^{0(0)}_{\M}, \U) \arrow{r}{\dr} & \czC^{q}( \Omega^{1(0)}_{\M}, \U) \arrow{r}{\dr} & \cdots,
\end{tikzcd}
\end{equation}
where $\dr$ is just the exterior derivative acting component-wise, and $\czC^{q}(\R_{M},\U)$ is the space of Čech $q$-cochains corresponding to the constant sheaf $\R_{M}$ and the open cover $\U$. Recall that the constant sheaf $\R_{M} \in \Sh(M,\Vect)$ is for each $U \in \Op(M)$ defined as $\R_{M}(U) := \C^{\lc}(U,\R)$, a space of locally constant functions on $U$, see Example \ref{ex_divnejsheaf}. $\czC^{q}(\R_{\M},\U)$ is then given by (\ref{eq_Czechcocycles}) with $\Omega_{\M}^{p(0)}$ replaced by $\R_{M}$. Due to Proposition \ref{tvrz_dfis0consequences}, we have the obvious identification
\begin{equation}
\czC^{q}(\R_{M},\U) \cong \ker( \dr) \subseteq \czC^{q}( \Omega^{0(0)}_{\M}, \U), 
\end{equation}
and $i$ in the sequence (\ref{eq_secondSESMayerVietoris}) is just the inclusion. 

Now, $\U = \{ U_{\alpha} \}_{\alpha \in I}$ is called a \textit{good open cover}, if for all $q \in \N_{0}$ and all $(\alpha_{0}, \dots, \alpha_{q}) \in I^{q+1}$, the open set $U_{\alpha_{0} \dots \alpha_{q}}$ is diffeomorphic to $\R^{n_{0}}$. Every cover of a smooth manifold $M$ can be refined to obtain a good open cover, see Theorem 5.1 in \cite{bott2013differential}. We may thus assume that we have a graded smooth atlas $\A = \{ (U_{\alpha},\varphi_{\alpha}) \}_{\alpha \in I}$ for $\M$, such that $\U = \{ U_{\alpha} \}_{\alpha \in I}$ is a good open cover of $M$. It follows from Proposition \ref{tvrz_Poincare} that for such a $\U$, the sequence (\ref{eq_secondSESMayerVietoris}) is exact for every $q \in \N_{0}$. 

Finally, whenever (\ref{eq_firstSESMayerVietoris}) is exact for every $p \in \N_{0}$ and (\ref{eq_secondSESMayerVietoris}) is exact for every $q \in \N_{0}$, there is a canonical isomorphism $H^{\bullet(0)}_{\M} \cong \czH^{\bullet}( \R_{M}, \U)$, where $\czH^{\bullet}(\R_{\M}, \U)$ is the cohomology of the cochain complex $(\czC^{\bullet}(\R_{M},\U), \delta)$ with $\delta$ given by the formula (\ref{eq_Cechdifferential}). This observation follows using the same arguments as Theorem 8.9 in \cite{bott2013differential}. In fact, the very same theorem says that $\czH^{\bullet}(\R_{M}, \U) \cong H^{\bullet}_{M}$. The combination of these two statements finishes the proof. 
\end{proof}

\begin{rem} This shows that the graded de Rham cohomology of graded manifolds is not particularly interesting. In fact, it cannot distinguish two graded manifolds with the same underlying manifold. The same issue appears for (Berezin--Leites) supermanifolds, see Chapter V of \cite{bartocci2012geometry}. 
\end{rem}
\section{Submanifolds} \label{sec_submanifolds}
\subsection{Definitions and basic properties}
Recall that we have already introduced the concept of an immersion, see Definition \ref{def_immersion}.
\begin{definice}
Let $\M$ be a graded manifold. A pair $(\cS,\iota)$ is called an \textbf{immersed submanifold of $\M$}, if $\cS$ is graded manifold, $\iota: \cS \rightarrow \M$ is an immersion, and the underlying map $\ul{\iota}: S \rightarrow M$ is injective. Note that $(S, \ul{\iota})$ is then an immersed submanifold of $M$. 

We say that $(\cS, \iota)$ is an \textbf{embedded submanifold of $\M$}, if it is an immersed submanifold and the underlying map $\ul{\iota}: S \rightarrow M$ is an embedding. We say that it is a \textbf{closed embedded submanifold}, if $\ul{\iota}(S) \subseteq M$ is a closed subset. 
\end{definice}
\begin{example}
For any $U \in \Op(M)$, the pair $(\M|_{U}, i)$, where $i: \M|_{U} \rightarrow \M$ is the obvious inclusion, is an embedded submanifold of $\M$, called an \textbf{open submanifold of $\M$}.
\end{example}
The following example shows that the underlying manifold of any graded manifold can be viewed as a closed embedded submanifold. 
\begin{example} \label{ex_underlyingissubmanifold}
For every graded manifold $\M$, we have a graded smooth map $i_{M}: M \rightarrow \M$, see Proposition \ref{tvrz_bodymap}. It is easy to see that $i_{M}$ is an immersion and as $\ul{i_{M}} = \1_{M}$, it follows that $(M,i_{M})$ is a closed embedded submanifold of $\M$. 
\end{example}
Now, there is a useful characterization of (closed) embedded submanifolds in terms of the sheaf morphism $\iota^{\ast}: \C^{\infty}_{\M} \rightarrow \ul{\iota}_{\ast} \C^{\infty}_{\cS}$. First, the existence of partitions of unity implies the following:
\begin{lemma} \label{lem_surjectivity}
Let $\iota: \cS \rightarrow \M$ be any graded smooth map. Let $U \in \Op(M)$. 

Then the graded algebra morphism $\iota^{\ast}_{U}: \C^{\infty}_{\M}(U) \rightarrow \C^{\infty}_{\cS}(\ul{\iota}^{-1}(U))$ is surjective, iff $\iota^{\ast}_{V}: \C^{\infty}_{\M}(V) \rightarrow \C^{\infty}_{\cS}(\ul{\iota}^{-1}(V))$ is surjective for all $V \in \Op(U)$. Moreover, let $\{ U_{\alpha} \}_{\alpha \in I}$ be any open cover of $U$, such that $\iota^{\ast}_{U_{\alpha}}$ is surjective for all $\alpha \in I$. Then $\iota^{\ast}_{U}$ is surjective. 
\end{lemma}
\begin{proof}
One only has to prove the only if part of the first statement. Hence assume that $i^{\ast}_{U}$ is surjective and let $V \in \Op(U)$ be arbitrary. Let $f \in \C^{\infty}_{\cS}( \ul{\iota}^{-1}(V))$. For each $s \in \ul{\iota}^{-1}(V)$, choose some $V_{s} \in \Op_{s}(U)$, such that $\ol{V}_{s} \subseteq V$. Let $\eta_{s} \in \C^{\infty}_{\M}(U)$ be a graded bump function supported on $V$, such that $\eta_{s}|_{V_{s}} = 1$. It follows that $\eta'_{s} := \iota^{\ast}_{U}(\lambda_{s}) \in \C^{\infty}_{\cS}(\ul{\iota}^{-1}(U))$ is a graded bump function supported on $\ul{\iota}^{-1}(V)$, such that $\eta'_{s}|_{\ul{\iota}^{-1}(V_{s})} = 1$. Let $f_{s} := \eta'_{s} \cdot f \in \C^{\infty}_{\cS}(\ul{\iota}^{-1}(U))$, see (\ref{eq_lambdabumpf}). By the assumed surjectivity of $i^{\ast}_{U}$, there is thus $g_{s} \in \C^{\infty}_{\M}(U)$ such that $f_{s} = i^{\ast}_{U}(g_{s})$. 

Let $\{ \lambda_{s} \}_{s \in S}$ be the partition of unity subordinate to the open cover $\{ V_{s} \}_{s \in S}$ of $V$, and define $g := \sum_{s \in S} \lambda_{s} \cdot g_{s}|_{V_{s}}$. By Proposition \ref{tvrz_partitionsnaturalops}-$(ii)$, $\{ \iota^{\ast}_{V}(\lambda_{s}) \}_{s \in S}$ is the partition of unity subordinate to the open cover $\{ \ul{\iota}^{-1}(V_{s}) \}_{s \in S}$ of $\ul{\iota}^{-1}(V)$, and one can use the formula (\ref{eq_pullbackofasumpartition}) to write 
\begin{equation}
\begin{split}
\iota^{\ast}_{V}(g) = & \ \iota^{\ast}_{V}( \sum_{s \in S} \lambda_{s} \cdot g_{s}|_{V_{s}}) = \sum_{s \in S} \iota^{\ast}_{V}(\lambda_{s}) \cdot \iota^{\ast}_{V_{s}}(g_{s}|_{V_{s}}) = \sum_{s \in S} \iota^{\ast}_{V}(\lambda_{s}) \cdot \iota^{\ast}_{U}(g_{s})|_{V_{s}} \\
= & \ \sum_{s \in S} \iota^{\ast}_{V}(\lambda_{s}) \cdot f_{s}|_{V_{s}} = \sum_{s \in S} \iota^{\ast}_{V}(\lambda_{s}) \cdot f|_{V_{s}} = f,
\end{split}
\end{equation}
where in the last step, we have used Proposition \ref{tvrz_partitionsnaturalops}-$(i)$. This finishes the first part of the proof. Next, let $\{ U_{\alpha} \}_{\alpha \in I}$ be an open cover of $U$, such that $\iota^{\ast}_{U_{\alpha}}$ is surjective for all $\alpha \in I$. Let $f \in \C^{\infty}_{\cS}(\ul{\iota}^{-1}(U))$ be arbitrary. For each $\alpha \in I$, there is $g_{\alpha} \in \C^{\infty}_{\M}(U_{\alpha})$, such that $f|_{\ul{\iota}^{-1}(U_{\alpha})} = i^{\ast}_{U_{\alpha}}(g_{\alpha})$. Let $\{ \lambda_{\alpha} \}_{\alpha \in I}$ be the partition of unity subordinate to $\{ U_{\alpha} \}_{\alpha \in I}$ and define $g := \sum_{\alpha \in I} \lambda_{\alpha} \cdot g_{\alpha}$. Using the formula (\ref{eq_pullbackofasumpartition}), one obtains 
\begin{equation}
\iota^{\ast}_{U}(g) = \sum_{\alpha \in I} \iota^{\ast}_{U}(\lambda_{\alpha}) \cdot i^{\ast}_{U_{\alpha}}(g_{\alpha}) = \sum_{\alpha \in I} i^{\ast}_{U}(\lambda_{\alpha}) \cdot f|_{\ul{\iota}^{-1}(U_{\alpha})} = f, 
\end{equation}
where we have used Proposition \ref{tvrz_partitionsnaturalops}-$(i)$. This finishes the proof. 
\end{proof}
We now obtain the following simple characterization of (closed) embeddings:
\begin{tvrz} \label{tvrz_embeddinghassurjectivepullback}
Let $\iota: \cS \rightarrow \M$ be a graded smooth map. 

Then $(\cS,\iota)$ is an embedded submanifold of $\cN$, iff there exists $Y \in \Op(M)$, such that $\ul{\iota}(S) \subseteq Y$ and $\iota^{\ast}_{Y}: \C^{\infty}_{\M}(Y) \rightarrow \C^{\infty}_{\cS}(S)$ is surjective. One can choose $Y = M$, if and only if $(\cS,\iota)$ is a closed embedded submanifold of $\M$. 
\end{tvrz}
\begin{proof}
Suppose $(\cS, \iota)$ is an embedded submanifold of $\cS$. In particular, $\iota$ is an immersion. For each $s \in S$, we have graded local charts $(U,\varphi)$ for $\cS$ and $(V',\psi)$ for $\M$ having the properties as in Theorem \ref{thm_immersion}-$(iii)$. Let $(n_{j})_{j \in \Z} := \gdim(\M)$ and $(m_{j})_{j \in \Z} := \gdim(\cS)$. Let $\hat{\iota}: \ul{\varphi}(U)^{(m_{j})} \rightarrow \ul{\psi}(V')^{(n_{j})}$ be the corresponding local representative of $\iota$. It is easy to see that the graded algebra morphism $\hat{i}^{\ast}_{\ul{\psi}(V')}: \C^{\infty}_{(n_{j})}(\ul{\psi}(V')) \rightarrow \C^{\infty}_{(m_{j})}(\ul{\varphi}(U))$ is surjective. 

Since $\ul{\iota}$ is an embedding, there is $W \in \Op(M)$ such that $\ul{\iota}(U) = W \cap \ul{\iota}(S)$. As $\ul{\iota}(U) \subseteq V'$, one can write $\ul{\iota}(U) = V \cap \ul{\iota}(S)$, where $V := V' \cap W$. Moreover, one has $\ul{\iota}(U) \subseteq V$. It follows that 
\begin{equation}
\iota^{\ast}_{V} = \varphi^{\ast}_{\ul{\varphi}(U)} \circ \hat{i}^{\ast}_{\ul{\psi}(V)} \circ (\psi^{-1})^{\ast}_{V}
\end{equation}
As $\ul{\psi}(V) \subseteq \ul{\psi}(V')$, $\hat{i}^{\ast}_{\ul{\psi}(V)}$ is surjective by Lemma \ref{lem_surjectivity}. Consequently, also $\iota^{\ast}_{V}: \C^{\infty}_{\M}(V) \rightarrow \C^{\infty}_{\cS}(U)$ is surjective. Note that by construction, one has $U = \ul{\iota}^{-1}(V)$. 

To summarize, to each point $s \in S$, one can find $U_{s} \in \Op_{s}(S)$ and $V_{s} \in \Op_{\ul{\iota}(s)}(M)$, such that $\ul{\iota}(U_{s}) = V_{s} \cap \ul{\iota}(S)$ and $i^{\ast}_{V_{s}}: \C^{\infty}_{\cS}(V_{s}) \rightarrow \C^{\infty}_{\cS}(U_{s})$ is surjective. Let $Y := \cup_{s \in S} V_{s}$. Clearly $\ul{\iota}(S) \subseteq Y$ and $\{ V_{s} \}_{s \in S}$ is its open cover. By Lemma \ref{lem_surjectivity}, this proves that $i^{\ast}_{Y}: \C^{\infty}_{\M}(U) \rightarrow \C^{\infty}_{\cS}(S)$ is surjective. If $\ul{\iota}(S) \subseteq M$ is closed, we can consider another open subset $V_{0} := M - \ul{\iota}(S)$. Moreover, the graded algebra morphism $\iota^{\ast}_{V_{0}}: \C^{\infty}_{\M}(V_{0}) \rightarrow \C^{\infty}_{\cS}(\emptyset) = 0$ is (trivially) surjective. Applying Lemma \ref{lem_surjectivity} on the open cover $\{ V_{s} \}_{s \in S} \cup \{ V_{0} \}$ of $M$ proves that $\iota^{\ast}_{M}: \C^{\infty}_{\M}(M) \rightarrow \C^{\infty}_{\cS}(S)$ is surjective. 

Conversely, let $\iota: \cS \rightarrow \M$ be any graded smooth map and suppose that there is $Y \in \Op(M)$, such that $\ul{\iota}(S) \subseteq Y$ and $\iota^{\ast}_{Y}: \C^{\infty}_{\M}(Y) \rightarrow \C^{\infty}_{\cS}(S)$ is surjective. 

Consequently, the pullback map $\ul{\iota}^{\ast}: \C^{\infty}_{M}(Y) \rightarrow \C^{\infty}_{S}(S)$ is also surjective. Indeed, for any $f_{0} \in \C^{\infty}_{S}(S)$, there is a $f \in \C^{\infty}_{\cS}(S)$, such that $f_{0} = \ul{f}$, see Proposition \ref{tvrz_bodymapsurj} (in fact, this also follows from Example \ref{ex_underlyingissubmanifold} and already proved only if part). There is thus some $g \in \C^{\infty}_{\M}(Y)$, such that $f = \iota^{\ast}_{Y}(g)$. Hence $f_{0} = \ul{\iota^{\ast}_{Y}(g)} = \ul{\iota}^{\ast}(\ul{g})$, see Proposition \ref{tvrz_bodymap}. 

Let us argue that $\ul{\iota}: S \rightarrow M$ must be an embedding. First, suppose that $\ul{\iota}(s) = \ul{\iota}(s')$ for $s \neq s'$. There certainly exists a smooth function $f \in \C^{\infty}_{S}(S)$, such that $f(s) \neq f(s')$. By assumption, there is $g \in \C^{\infty}_{M}(Y)$, such that $\ul{i}^{\ast}(g) = f$. In particular, one has $f(s) = g(\ul{\iota}(s)) = g(\ul{\iota}(s')) = f(s')$, which is a contradiction. This proves that $\ul{\iota}$ is injective. Next, for any $s \in S$, let  $x \in T_{s}S$ satisfy $(T_{s} \ul{\iota})(x) = 0$. For arbitrary smooth function $f \in \C^{\infty}_{\M}(U)$, one can find $g \in \C^{\infty}_{\cS}(Y)$, such that $f = \ul{\iota}^{\ast}(g)$. But then
\begin{equation}
x(f) = x( \ul{\iota}^{\ast}(g)) = [(T_{s} \ul{\iota})(x)](g) = 0. 
\end{equation}
But this proves that $x = 0$ and we conclude that $\ul{\iota}$ is an immersion. It remains to prove that $\ul{\iota}$ is a homeomorphism onto its image. By Lemma \ref{lem_balls}, there exists a countable basis $\B$ for the topology of $S$ consisting of regular coordinate balls. For each $B \in \B$, there exists a smooth function $f_{B} \in \C^{\infty}_{S}(S)$ strictly positive on $B$ and zero everywhere else, see the discussion above (\ref{eq_hnufunkce}). Consequently, one has $f_{B}^{-1}( \R - \{0\}) = B$. By assumption, there exists $g_{B} \in \C^{\infty}_{\M}(Y)$, such that $f_{B} = \ul{\iota}^{\ast}(g_{B})$. Then
\begin{equation}
\ul{\iota}(B) = g_{B}^{-1}(\R - \{0\}) \cap \ul{\iota}(S).
\end{equation}
But this shows that $\ul{i}(B)$ is open in the subspace topology of $\ul{\iota}(S)$. As $B \in \B$ was arbitrary and $\B$ is the basis for the topology of $S$, this proves the claim.

Next, for each $s \in S$, the injectivity of $T_{s} \iota$ can be proved in completely the same way as the one of $T_{s}\ul{\iota}$ in the previous paragraph, this time using the formula (\ref{eq_differentialformula}). We conclude that $\iota$ is an immersion. Altogether, $(\cS, \iota)$ is an embedded manifold of $\M$. 

It remains to prove that if we can choose $Y = M$, $\ul{\iota}(S) \subseteq M$ must be closed. Let $m \in M$ be a limit point of $\ul{\iota}(S)$. Let $(U,\ul{\varphi})$ be any local chart around $m$. By definition, $U \cap \ul{\iota}(S) \neq \emptyset$. There is thus a non-empty $V \in \Op(S)$, such that $\ul{\iota}(V) = U \cap \ul{\iota}(S)$. Suppose $m \notin \ul{\iota}(S)$. Let $f \in \C^{\infty}_{S}(V)$ be a smooth function defined for each $s \in V$ as 
\begin{equation}
f(s) := \frac{1}{\| \ul{\varphi}(\ul{\iota}(s)) - \ul{\varphi}(m) \|^{2}}.
\end{equation}
As $\iota^{\ast}_{M}$ is surjective by assumption, so is $\iota^{\ast}_{U}$ by Lemma \ref{lem_surjectivity}. Consequently, also the pullback $\ul{\iota}^{\ast}: \C^{\infty}_{M}(U) \rightarrow \C^{\infty}_{S}(V)$ is surjective (we have proved this in one of the preceding paragraphs). There is thus some $g \in \C^{\infty}_{M}(U)$, such that $f = \ul{\iota}^{\ast}(g)$. But $m = \lim_{k \rightarrow \infty} \ul{\iota}(s_{k})$ for some sequence $\{ s_{k} \}_{k = 1}^{\infty} \subseteq V$. Hence $g(m) = \lim_{k \rightarrow \infty} g( \ul{\iota}(s_{k})) = \lim_{k \rightarrow \infty} f(s_{k})$. But this is a contradiction since the sequence $\{ f(s_{k}) \}_{k = 1}^{\infty}$ diverges horribly. This finishes the proof. 
\end{proof}
This proposition implies the following extension property:
\begin{cor} \label{cor_extensionsubmanifolds}
Let $(\cS,\iota)$ be an embedded submanifold of $\M$. Let $Y \in \Op(M)$ be the open subset, such that $\ul{\iota}(S) \subseteq Y$ and $\iota^{\ast}_{Y}$ is surjective. 

Let $U \in \Op(S)$ and $f \in \C^{\infty}_{\cS}(U)$. Then for any $V \in \Op(Y)$, such that $\ul{\iota}(U) = V \cap \ul{\iota}(S)$, there exists $g \in \C^{\infty}_{\M}(V)$, such that $f = \iota^{\ast}_{V}(g)$. 
\end{cor}
\begin{proof}
Use the previous proposition together with Lemma \ref{lem_surjectivity}. 
\end{proof}
\begin{rem}
It is a well-known fact from standard differential geometry that Proposition \ref{tvrz_embeddinghassurjectivepullback} fails for general immersed submanifolds. 
\end{rem}
\subsection{Tangent space and vector fields}
Let $(\M,\iota)$ be an immersed submanifold of $\M$. For each $s \in S$, we have the injective graded linear map $T_{s}\iota: T_{s}\cS \rightarrow T_{\ul{\iota}(s)} \M$. This allows one to view $T_{s} \cS$ as a graded linear subspace of $T_{\ul{\iota}(s)} \M$. It is a standard convention to identify $S$ with its image $\ul{\iota}(S)$ and write
\begin{equation}
T_{s} \cS \subseteq T_{s} \M,
\end{equation} 
whenever there is no possibility for a confusion. We thus say that a given tangent vector $v \in T_{s}\M$ is \textbf{tangent to $\cS$}, if it lies in the subspace $T_{s}\cS$. 

Next, let $U \in \Op(M)$ and $X \in \X_{\M}(U)$. We say that $X$ is \textbf{tangent to $\cS$}, if there exists a vector field $X' \in \X_{\cS}(U \cap S)$, such that for all $f \in \C^{\infty}_{\M}(U)$, one has 
\begin{equation} \label{eq_vftangenttoN}
\iota^{\ast}_{U}( X(f)) = X'( \iota^{\ast}_{U}(f)).
\end{equation}
This is equivalent to $X' \sim_{\iota|_{U \cap S}} X$. These two notions are compatible:
\begin{tvrz}
Let $U \in \Op(M)$ and let $X \in \X_{\M}(U)$ be tangent to $\cS$. 

Then for every $s \in S$, the value of $X$ at $s$ is tangent to $\cS$, $X_{s} \in T_{s}\cS$. In general, the converse is not true (in contrast to ordinary manifolds).
\end{tvrz}
\begin{proof}
This follows immediately from definitions and Proposition \ref{tvrz_relatedprops}-$(i)$. To find a counterexample to the converse statement, consider graded manifolds $\M := \R^{2}[1]$ and $\cS := \R[1]$. The respective underlying manifolds are $M = S = \{ \ast \}$. 

Let $\xi_{1},\xi_{2} \in \C^{\infty}_{\M}(\{ \ast \})$ be the coordinate functions of degree $1$, and denote the single coordinate function on $\cS$ as $\xi_{1} \in \C^{\infty}_{\cS}(\{\ast\})$. Let $\iota = (\1_{\{\ast \}}, \iota^{\ast})$ be defined by
\begin{equation}
\iota^{\ast}_{\{\ast\}}(\xi_{1}) := \xi_{1}, \; \; \iota^{\ast}_{\{ \ast \}}(\xi_{2}) := 0.
\end{equation}
This makes $(\cS,\iota)$ into a closed embedded submanifold of $\cS$. Consider $X \in \X_{\M}(\{\ast \})$ given by
\begin{equation}
X = \xi_{1} \frac{\partial}{\partial \xi_{2}}.
\end{equation}
We have $X_{\ast} = 0$, so $X_{\ast}$ is tangent to $\cS$. For $f := \xi_{2}$, the equation (\ref{eq_vftangenttoN}) reads $\xi_{1} = X'(0)$. As there is no such $X' \in \X_{\cS}(\{\ast \})$, we conclude that $X$ is \textit{not tangent} to $\cS$. 
\end{proof}
\begin{example}
Let $(M,i_{M})$ be the closed embedded submanifold of Example \ref{ex_underlyingissubmanifold}. Then the Euler vector field $E \in \X_{\M}(M)$ is tangent to $M$. Indeed, for any $f \in \C^{\infty}_{\M}(M)$, one has $(i_{M})^{\ast}_{M}(E(f)) = 0$. This is (\ref{eq_vftangenttoN}) for $X' = 0$ and the conclusion follows. 
\end{example}
For embedded submanifolds, there is the following useful criterion:
\begin{tvrz}
Let $(\cS,\iota)$ be an embedded submanifold of $\M$. Let $U \in \Op(M)$ be arbitrary and let $X \in \X_{\M}(U)$ be a vector field.

Then $X$ is tangent to $\cS$, iff for every $f \in \C^{\infty}_{\M}(U)$ satisfying $\iota^{\ast}_{U}(f) = 0$, one has $\iota^{\ast}_{U}(X(f)) = 0$. 
\end{tvrz}
\begin{proof}
One direction follows immediately from (\ref{eq_vftangenttoN}). Hence assume that for every $f \in \C^{\infty}_{\M}(U)$, the equation $\iota^{\ast}_{U}(f) = 0$ implies $\iota^{\ast}_{U}(X(f)) = 0$. 

First, let us argue that for any $V \in \Op(U)$, $X|_{V} \in \X_{\M}(V)$ has the same property. Suppose that $g \in \C^{\infty}_{\M}(V)$ satisfies $\iota^{\ast}_{V}(g) = 0$. Pick any $m \in V$ and $W \in \Op_{m}(V)$, such that $\ol{W} \subseteq V$. Let $\lambda \in \C^{\infty}_{\M}(U)$ be the graded bump function supported on $V$, such that $\lambda|_{W} = 1$. It is easy to see that $\iota^{\ast}_{U}( \lambda \cdot g) = 0$. By assumption, we have $\iota^{\ast}_{U}(X( \lambda \cdot g)) = 0$. Using the naturality and definitions, we find the expression
\begin{equation}
0 = \iota^{\ast}_{U}(X(\lambda \cdot g))|_{W \cap S} = \iota^{\ast}_{W}( X|_{W}(g|_{W})) = \iota^{\ast}_{V}( X|_{V}(g))|_{W \cap S}. 
\end{equation}
As $m \in V$ was arbitrary, the equation $\iota^{\ast}_{V}(X|_{V}(g)) = 0$ follows. 

Now, let us return to the actual proof. Let $Y \in \Op(M)$ be the open subset, such that $\ul{\iota}(S) \subseteq Y$ and $\iota^{\ast}_{Y}$ is surjective. As $(\cS,\iota)$ is embedded, it exists due to Proposition \ref{tvrz_embeddinghassurjectivepullback}. It follows from Corollary \ref{cor_extensionsubmanifolds} that for every $f \in \C^{\infty}_{\cS}(U \cap S)$, there exists a function $g \in \C^{\infty}_{\M}(U \cap Y)$, such that $f = \iota^{\ast}_{U \cap Y}(g)$. Let us thus define $X' \in \X_{\cS}(U \cap S)$ by 
\begin{equation}
X'(f) := \iota^{\ast}_{U \cap Y}( X|_{U \cap Y}(g)).
\end{equation}
The previous paragraph ensures that the right-hand side does not depend on the particular choice of $g$, hence $X'$ is well-defined. It is easy to check that it is a vector field. Finally, by choosing $f = i^{\ast}_{U}(f')$ for $f' \in \C^{\infty}_{\M}(U)$, we can choose $g := f'|_{U \cap Y}$ to obtain (\ref{eq_vftangenttoN}). We conclude that $X$ is tangent to $\cS$, as was to be proved. 
\end{proof}

Now, let $\E$ be a graded vector bundle over $\M$ and suppose $(\cS,\iota)$ is an immersed submanifold of $\M$. The \textbf{restricted vector bundle $\E_{\cS}$} is then defined as a pullback $\E_{\cS} := \iota^{\ast} \E$. For each $U \in \Op(M)$ and every section $\sigma \in \Gamma_{\E}(U)$, by $\sigma_{\cS} \in \Gamma_{\E_{\cS}}(U \cap S)$, we mean the corresponding pullback section $\sigma^{!}$, see Proposition \ref{tvrz_pullbackVB}-$(iv)$. For embedded submanifolds, we obtain the following extension property:

\begin{tvrz}
Let $(\cS,\iota)$ be an embedded submanifold. Let $Y \in \Op(M)$ be an open subset, such that $\ul{\iota}(S) \subseteq Y$ and $\iota^{\ast}_{Y}$ is surjective. Let $U \in \Op(Y)$ be arbitrary.

Then to any $\sigma' \in \Gamma_{\E_{\cS}}(U \cap S)$, there exists $\sigma \in \Gamma_{\E}(U)$, such that $\sigma' = \sigma_{\cS}$. 
\end{tvrz}
\begin{proof}
Let $m \in U$. There is thus a local frame $(\Phi_{\lambda})_{\lambda=1}^{k}$ for $\E$ over some $V_{m} \in \Op_{m}(U)$. It follows from Proposition \ref{tvrz_pullbackVB}-$(iv)$ that $\sigma'|_{V_{m} \cap N}$ can be decomposed as 
\begin{equation}
\sigma'|_{V_{m} \cap N} = f^{\lambda} \Phi_{\lambda}^{!},
\end{equation}
for unique functions $f^{\lambda} \in \C^{\infty}_{\cS}(V_{m} \cap S)$. By Corollary \ref{cor_extensionsubmanifolds}, there are functions $g^{\lambda} \in \C^{\infty}_{\M}(V_{m})$, such that $f^{\lambda} = \iota^{\ast}_{V_{m}}(g^{\lambda})$ for each $\lambda \in \{1,\dots,k\}$. Define $\sigma_{m} \in \Gamma_{\E}(V_{m})$ by $\sigma_{m} := g^{\lambda} \Phi_{\lambda}$. It follows that $(\sigma_{m})^{!} = \sigma'|_{V_{m} \cap S}$. Let $\{ \lambda_{m} \}_{m \in U}$ be the partition of unity subordinate to the open cover $\{ V_{m} \}_{m \in U}$ of $U$, and define $\sigma := \sum_{m \in U} \lambda_{m} \sigma_{m}$. Now, there holds the analogue of Proposition \ref{tvrz_partitionsnaturalops} and we find the expression
\begin{equation}
\sigma_{\cS} \equiv \sigma^{!} = \sum_{m \in U} \iota^{\ast}_{U}(\lambda_{m}) \cdot (\sigma_{m})^{!} = \sum_{m \in U} \iota^{\ast}_{U \cap Y}(\lambda_{m}) \cdot \sigma'|_{V_{m} \cap N} = \sigma'. 
\end{equation}
This finishes the proof. 
\end{proof}
The notion of restricted vector bundles allows one to relate the respective tangent bundles. 
\begin{tvrz}
Let $(\cS,\iota)$ be an immersed submanifold of $\M$. Then the tangent bundle $T\cS$ can be viewed as a subbundle of the restricted tangent bundle $T\M_{\cS}$. 
\end{tvrz}
\begin{proof}
It follows from Proposition \ref{tvrz_inducedtangent}, that we have the induced graded vector bundle morphism $T \iota: T\cS \rightarrow T\M$ over $\iota$. By Proposition \ref{tvrz_pullbackVB}-$(iii)$, there is then a unique graded vector bundle morphism $\hat{T} \iota: T\cS \rightarrow T\M_{\cS}$, such that $T\iota = \hat{\iota} \circ \hat{T}\iota$. For each $s \in S$, one has $(T\cS)_{s} \cong T_{s} \cS$, $(T\M_{\cS})_{s} \cong T_{\iota(s)}\M$, and the induced graded linear map $(\hat{T}\iota)_{s}$ corresponds to the injective graded linear map $T_{s} \iota: T_{s} \cS \rightarrow T_{\iota(s)} \M$. 

Finally, one can prove that for any ``fiber-wise injective'' graded vector bundle morphism over $\1_{\cS}$, the image sheaf $\im( \hat{T} \iota) \subseteq \Gamma_{T\M_{\cS}}$ defines a subbundle of $T\M_{\cS}$ isomorphic to $T\cS$. 
\end{proof}
\begin{cor}
One can define the \textbf{normal bundle} of an immersed submanifold $(\cS,\iota)$ as
\begin{equation}
N \cS := T \M_{\cS} / T \cS,
\end{equation}
see Proposition \ref{tvrz_quotient}. Note that in general, $N\cS$ cannot be identified with a subbundle of $T\M_{\cS}$.
\end{cor}
\subsection{Sheaves of regular ideals}
In this subsection, we will briefly examine a certain class of sheaves of ideals in the structure sheaf $\C^{\infty}_{\M}$ of a given graded manifold $\M$. It will play an crucial role in the more algebraic description of closed embedded submanifolds in the following subsections. 

First, observe that to each $f \in \C^{\infty}_{\M}(U)$, we may assign its differential $\dr{f} \in \Omega^{1}_{\M}(U) \cong \X_{\M}^{\ast}(U)$. Its value $(\dr{f})_{m}$ at $m \in U$ is thus an element of the fiber $(T^{\ast}\M)_{m}$, which can be canonically identified with the dual $T^{\ast}_{m}\M$ to the tangent space $T_{m}\M$. We call $(\dr{f})_{m} \in T^{\ast}_{m}\M$ a \textbf{differential of $f$ at $m$}. Let us summarize its basic properties:
\begin{tvrz} \label{tvrz_differentialatm}
Let $f \in \C^{\infty}_{\M}(U)$ and let $m \in U$. Then we have:
\begin{enumerate}[(i)]
\item For any $V \in \Op_{m}(U)$, one has $(\dr f|_{V})_{m} = (\dr{f})_{m}$. In other words, the differential of $f$ at $m$ depends only on the germ $[f]_{m} \in \C^{\infty}_{\M,m}$. 
\item It behaves naturally with respect to pullbacks, that is for any smooth map $\phi: \cN \rightarrow \M$ and any $n \in \ul{\phi}^{-1}(U)$, one has 
\begin{equation} \label{eq_differentialpullback}
(\dr{ [\phi^{\ast}_{U}(f)]})_{n} := (T_{n} \phi)^{T}( (\dr{f})_{\ul{\phi}(n)}).
\end{equation}
\item Suppose we have a graded local chart $(U,\varphi)$ inducing the coordinate functions $(\bbz^{A})_{A=1}^{n}$. Then $((\dr{\bbz}^{A})_{m})_{A=1}^{n}$ forms the total basis of $T^{\ast}_{m}\M$ and one can write 
\begin{equation}
(\dr{f})_{m} = \frac{\partial f}{\partial \bbz^{A}}(m) (\dr{\bbz}^{A})_{m}. 
\end{equation}
Moreover, any $v \in T_{m}\M$ can be uniquely decomposed as $v = v^{A} \frac{\partial}{\partial \bbz^{A}}|_{m}$. Then 
\begin{equation} \label{eq_diffonv}
(\dr{f})_{m}(v) = v^{A} \frac{\partial f}{\partial \bbz^{A}}(m) = v([f]_{m}).
\end{equation}
\end{enumerate}
\end{tvrz}
\begin{proof}
All properties follow easily from properties of a value of a section at the given point, see Example \ref{ex_fiberoftangentistangent}. For example, since $( \dr{\bbz}^{A})_{\lambda=1}^{n}$ form the local frame for $T^{\ast}\M$ over $U$, the collection of values $((\dr{\bbz}^{A})_{m})_{A=1}^{n}$ forms the total basis of the fiber. Note that the identification $(T^{\ast} \M)_{m} \cong T^{\ast}_{m} \M$ is precisely the canonical sheaf isomorphism $\nu: \kappa_{m}^{\ast} \Gamma^{\ast}_{T\M} \rightarrow \Gamma^{\ast}_{\kappa_{m}^{\ast}(T\M)}$, see (\ref{eq_pullbackofdual}). 
\end{proof}
\begin{rem}
Equivalently, to each $f \in \C^{\infty}_{\M}(U)$, we may assign a graded smooth map $\phi_{f}: \M|_{U} \rightarrow \R[|f|]$, see Corollary \ref{cor_functionsasmaps}. Its differential at $m \in M$ is $T_{m} \phi_{f}: T_{m} \M \rightarrow T_{\ul{\phi_{f}}(m)} \R[|f|] \cong \R[-|f|]$. But $T_{m}\phi_{f} \in \Lin_{0}( T_{m}\M, \R[-|f|]) \cong \Lin_{-|f|}(T_{m}\M, \R) \cong (T^{\ast}_{m} \M)_{|f|}$. 

It turns out that under these identifications, one has $T_{m} \phi_{f} = (\dr{f})_{m}$.  
\end{rem}
\begin{definice}
We say that $f_{1}, \dots, f_{k} \in \C^{\infty}_{\M}(U)$ are \textbf{functions independent at $m \in U$}, if their differentials $(\dr{f}_{1})_{m}, \dots, (\dr{f}_{k})_{m}$ are independent vectors in $T_{m}^{\ast} \M$. 
\end{definice}
\begin{rem}
Here, one should be a bit careful with the meaning of a linear independence for graded vector spaces. Let $V \in \gVect$ be a graded vector space. We say that $v_{1}, \dots, v_{k} \in V$ are \textbf{independent vectors} in $V$, if for each $j \in \Z$, the set $S_{j} := \{ v_{\mu} \in \{v_{1},\dots,v_{k} \} \; | \; |v_{\mu}| = j \}$ is linearly independent in $V_{j} \in \Vect$. 
\end{rem}

Now, for any $U \in \Op(M)$, let $J \subseteq \C^{\infty}_{\M}(U)$ be an ideal. For any $m \in U$, we can then define a subspace $J_{(m)} := \pi_{U,m}(J) \subseteq \C^{\infty}_{\M,m}$. As $\pi_{U,m}: \C^{\infty}_{\M}(U) \rightarrow \C^{\infty}_{\M,m}$ is surjective by Corollary \ref{cor_piUmsurjective}, $J_{(m)}$ forms an ideal in $\C^{\infty}_{\M,m}$. One can now make the following definition.
\begin{definice} \label{def_regularideal}
Let $U \in \Op(M)$ be arbitrary, and let $J \subseteq \C^{\infty}_{\M}(U)$ be an ideal. We say that $J$ is a \textbf{regular ideal}, if it has the following properties:
\begin{enumerate}[(i)]
\item For any $m \in U$, such that $J \subseteq \J^{m}_{\M}(U)$, there exists a finite graded subset $\G \subseteq J$ consisting of functions independent at $m$ and $V \in \Op_{m}(U)$, such that for each $m' \in V$, the graded subset $[\G]_{m'} := \{ [f]_{m'} \; | \; f \in \G \}$ generates the ideal $J_{(m')} \subseteq \C^{\infty}_{\M,m'}$. 
\item Let $\{ f_{\mu} \}_{\mu \in K}$ be any collection of elements of $J$, such that the collection $\{ \supp(f_{\mu}) \}_{\mu \in K}$ is locally finite. Then the sum $\sum_{\mu \in K} f_{\mu}$ is also an element of $J$. 
\end{enumerate}
\end{definice}
\begin{rem}
The condition (i) is stronger than the condition (1) in Definition 5.3.6 in \cite{carmeli2011mathematical}, where one only assumes that $\G_{(m)}$ generates $\J_{(m)}$. However, we believe that this is not enough for purposes in this section. Also compare it to definitions in Chapter III, $\mathsection$2 of \cite{leites1980introduction}. 
\end{rem}

In the following, we will exclusively deal with ideals $J$ which come as spaces of sections of some sheaf of ideals. Let us make some observations in this case.
\begin{lemma} \label{lem_regularideals}
Suppose there is a sheaf of ideals $\J \subseteq \C^{\infty}_{\M}$, such that $J = \J(U)$. Then the following statements are true:
\begin{enumerate}[1)]
\item For each $m \in M$, $J_{(m)}$ can be canonically identified with the stalk $\J_{m}$ of the sheaf $\J$.
\item The property $(i)$ is equivalent to the following one:

For any $m \in U$, such that $J \subseteq \J^{m}_{\M}(U)$, there exists a finite graded subset $\G \subseteq J$ consisting of functions independent at $m$ and $V \in \Op_{m}(U)$, such that for all $V' \in \Op(V)$, the graded subset $\G|_{V'} = \{ f|_{V'} \; | \; f \in \G \}$ generates the ideal $\J(V')$. 
\item The condition (ii) is satisfied automatically. 
\item For any $m \in U$, one has $J \subseteq \J^{m}_{\M}(U)$, if and only if $\J_{m} \subseteq \frJ(\C^{\infty}_{\M,m})$, if and only if $\J_{m} \neq \C^{\infty}_{\M,m}$. 
\end{enumerate}
\end{lemma}
\begin{proof}
Let us start by proving \textit{1)}. Since $\J$ is a sheaf of ideals, there is a canonical sheaf morphism $j: \J \rightarrow \C^{\infty}_{\M}$. For each $m \in M$, it induces an injective graded algebra morphism $j_{m}: \J_{m} \rightarrow \C^{\infty}_{\M,m}$ and we usually identify $\J_{m}$ with its image $j_{m}(\J_{m}) \subseteq \C^{\infty}_{\M,m}$. The inclusion $J_{(m)} \subseteq \J_{m}$ is obvious. The converse follows from the fact that any $[f]_{m} \in \J_{m}$ can be represented by $f \in J$. Indeed, every $[f]_{m} \in \J_{m}$ can be by definition represented by $f \in \C^{\infty}_{\M}(V)$ for some $V \in \Op_{m}(U)$. Find some $W \in \Op_{m}(V)$ with $\ol{W} \subseteq V$, and consider a graded bump function $\lambda \in \C^{\infty}_{M}(U)$ supported in $V$, such that $\lambda|_{W} = 1$. Then $\lambda \cdot f \in J$ since $\J$ is a sheaf of ideals, and $[f]_{m} = [\lambda \cdot f]_{m}$. 

Next, let us prove \textit{2)}. Let $m \in U$ and $\G = \{ f_{1}, \dots, f_{k} \} \subseteq \J(U)$ be any finite graded set. Let $V \in \Op_{m}(M)$ be arbitrary. It suffices to prove that $[\G]_{m'}$ generates $\J_{m'}$ for all $m' \in V$, if and only if $\G|_{V'}$ generates $\J(V')$ for any $V' \in \Op(V)$. 

Suppose $[\G]_{m'}$ generates $\J_{m'}$ for all $m' \in V$.  Let $V' \in \Op(V)$ and $f \in \J(V')$ be arbitrary. For every $m' \in V'$, we have $[f]_{m'} \in \J_{m'}$. As $V' \subseteq V$, $\J_{m'}$ is generated by $[\G]_{m'}$. Consequently, there exists $V_{(m')} \in \Op_{m'}(V')$ and a collection $\{ h^{i}_{(m')} \}_{i=1}^{k} \subseteq \C^{\infty}_{\M}(V_{(m')})$, such that $f|_{V_{(m')}}$ can be written as a finite sum
\begin{equation}
f|_{V_{(m')}} = \sum_{i=1}^{k} h^{i}_{(m')} \cdot f_{i}|_{V_{(m')}}. 
\end{equation}
In this way, we obtain an open cover $\{ V_{(m')} \}_{m' \in V'}$ of $V'$. Let $\{ \lambda_{(m')} \}_{m' \in V'}$ be a partition of unity subordinate to this open cover. For each $i \in \{1,\dots,k\}$, let 
\begin{equation}
h^{i} := \sum_{m' \in V'} h^{i}_{(m')} \in \C^{\infty}_{\M}(V').
\end{equation}
It is now straightforward to verify that $f = \sum_{i=1}^{k} h^{i} \cdot f_{i}|_{V'}$. This shows that $\G|_{V'}$ indeed generates $\J(V')$ for any $V' \in \Op(V)$. 

Conversely, suppose that $\G|_{V'}$ generates $\J(V')$ for any $V' \in \Op(V)$. Let $m' \in V$ be arbitrary. Any given $[f]_{m'} \in \J_{m'}$ can be represented by $f \in \J(V)$. Since $\G|_{V}$ generates $\J(V)$, there is a collection of functions $\{ h^{i} \}_{i=1}^{k} \subseteq \C^{\infty}_{\M}(V)$, such that $\sum_{i=1}^{k} f = h^{i} \cdot f_{i}|_{V}$. Thus 
\begin{equation}
[f]_{m'} = \sum_{i=1}^{k} [h^{i}]_{m'} \cdot [f_{i}]_{m'}.
\end{equation}
This proves that $[\G]_{m'}$ generates $\J_{m'}$ for any $m' \in V$. This finishes the proof of \textit{2)}. 

Let us proceed with the proof of \textit{3)}. Let $\{ f_{\mu} \}_{\mu \in K} \subseteq J$, such that $\{ \supp(f_{\mu}) \}_{\mu \in K}$ is locally finite. This means that for each $m \in U$, there is $V_{m} \in \Op_{m}(U)$, such that $\supp(f_{\mu}) \cap V_{m} \neq \emptyset$ only for $\mu$ in some finite subset $K_{m} \subseteq K$. The function $\sum_{\mu \in K} f_{\mu}$ is then defined by declaring
\begin{equation}
(\sum_{\mu \in K} f_{\mu})|_{V_{m}} := \sum_{\mu \in K_{m}} f_{\mu}|_{V_{m}},
\end{equation}
for each $m \in U$. But the right-hand side is in $\J(V_{m})$ for every $m \in U$ and as $\J$ is assumed to be a \textit{sheaf} of ideals, this implies that $\sum_{\mu \in K} f_{\mu} \in \J(U)$. 

Finally, to see \textit{4)}, let $m \in U$ be any point. The equivalence of the inclusions $J \subseteq \J^{m}_{\M}(U)$ and $\J_{m} \subseteq \frJ(\C^{\infty}_{\M,m})$ follows immediately from (\ref{eq_Jacobsonradicalgradedmafolds}) and the fact that every $[f]_{m} \in \J_{m}$ can be represented by $f \in J$. The equivalence of the inclusion $\J_{m} \subseteq \frJ(\C^{\infty}_{\M,m})$ with $\J_{m} \neq \C^{\infty}_{\M,m}$ follows from the fact that every proper ideal of a local graded ring is contained in the Jacobson radical. 
\end{proof}
\begin{rem} \label{rem_generatesonUonsmaller}
Let $\G \subseteq \J(U)$ be any finite graded set. Observe that in the proof of \textit{2)}, we have actually shown that if $\G|_{V}$ generates $\J(V)$ for some $V \in \Op(U)$, then $\G|_{W}$ generates $\J(W)$ for all $W \in \Op(V)$. 
\end{rem}

\begin{definice}
Let $\J \subseteq \C^{\infty}_{\M}$ be a sheaf of ideals. We call it a \textbf{sheaf of regular ideals}, if $\J(M)$ is a regular ideal. 
\end{definice}
\begin{rem}
This is an analogue of Definition 5.3.9 in \cite{carmeli2011mathematical}, where they call $\J$ a \textit{quasi-coherent sheaf of ideals}. Their intention is to emphasize the analogy with quasi-coherent sheaves of ideals in algebraic geometry. Nevertheless, we stick to the more straightforward nomenclature.
\end{rem}
\begin{tvrz}
Let $\J$ be a sheaf of regular ideals and let $U \in \Op(M)$ be arbitrary. Then $\J(U)$ is a regular ideal. 
\end{tvrz}
\begin{proof}
Let $m \in U$, such that $\J(U) \subseteq \J^{m}_{\M}(U)$. But then also $\J(M) \subseteq \J^{m}_{\M}(M)$ by Lemma \ref{lem_regularideals}-\textit{4)}. As $\J(M)$ is regular, there is thus a finite graded set $\G \subseteq \J(M)$ consisting of functions independent at $m \in M$, such that $[\G]_{m'}$ generates $\J_{m'}$ for all $m'$ in some $V \in \Op_{m}(M)$. 

It follows from Proposition \ref{tvrz_differentialatm}-$(i)$ that $\G|_{U} \subseteq \J(U)$ is a finite graded set consisting of functions independent at $m$, and $[\G|_{U}]_{m'} = [\G]_{m'}$ obviously generates $\J_{m'}$ for all $m' \in U \cap V$. Hence $\J(U)$ is a regular ideal. 
\end{proof}

\begin{example} \label{ex_regid1}
On any graded manifold $\M$, we have a sheaf of ideals $\J^{\pg}_{\M}$ of purely graded functions defined by (\ref{eq_sheaofidealspurelygraded}). We claim that it is a sheaf of regular ideals. Clearly $\J^{\pg}_{\M}(M) \subseteq \J^{m}_{\M}(M)$ for all $m \in M$. 

For a given $m \in M$, let $(U,\varphi)$ be a graded local chart around $m$. Then the graded set $\{ \xi_{\mu} \}_{\mu=1}^{n_{\ast}}$ consisting of the corresponding purely graded coordinate functions generates $\J^{\pg}_{\M}(U)$, see Proposition \ref{tvrz_Jpggeneratingset}. Choose $V \in \Op_{m}(U)$, such that $\ol{V} \subseteq U$, and find a graded bump function $\lambda \in \C^{\infty}_{\M}(M)$ supported in $U$, such that $\lambda|_{V} = 1$. Then $\G = \{ \lambda \cdot \xi_{\mu} \}_{\mu=1}^{n_{\ast}} \subseteq \J^{\pg}_{\M}(M)$ is a finite graded subset, such that $[\G]_{m'} = \{ [\xi_{\mu}]_{m'} \}_{\mu=1}^{n_{\ast}}$ generates $(\J^{\pg}_{\M})_{m'}$ for all $m' \in V$. By Proposition \ref{tvrz_differentialatm}-$(iii)$, $\G$ consists of functions independent at $m$. Hence $\J^{\pg}_{\M}(M)$ is a regular ideal. 
\end{example}

\begin{example} \label{ex_regid2}
On any graded manifold $\M$ and for each $a \in M$, we have a sheaf of ideals $\J^{a}_{\M}$ of functions vanishing at $a$. This is also a sheaf of regular ideals. 

We have $(\frJ^{a}_{\M})_{a} = \frJ(\C^{\infty}_{\M,a})$ and $(\J^{a}_{\M})_{m} = \C^{\infty}_{\M,m}$ for any $m \neq a$. Pick any graded local chart $(U,\varphi)$ around $a$, such that $\ul{\varphi}(a) = 0$. Then the graded set $\{ \bbz^{A} \}_{A=1}^{n}$ consisting of \textit{all} the corresponding coordinate functions generates $\J^{a}_{\M}(U)$, see Proposition \ref{tvrz_vanishingidealgen}. Pick $V \in \Op_{a}(U)$ and $\lambda \in \C^{\infty}_{\M}(M)$ as in the previous example. Let $\G := \{ \lambda \cdot \bbz^{A} \}_{A=1}^{n} \subseteq \J^{a}_{\M}(M)$. It follows from Remark \ref{rem_generatesonUonsmaller} that $\G|_{V'}$ generates $\J^{a}_{\M}(V')$ for any $V' \in \Op(V)$. 
\end{example}
\subsection{Algebraic viewpoint on submanifolds}
Suppose $(\cS,\iota)$ is an immersed submanifold of a graded manifold $\M$. We have a sheaf morphism $\iota^{\ast}: \C^{\infty}_{\M} \rightarrow \ul{\iota}_{\ast} \C^{\infty}_{\cS}$, hence we can consider its kernel sheaf
\begin{equation}
\J_{\cS} := \ker(\iota^{\ast}). 
\end{equation}
Let us now argue that for embedded submanifolds, a lot of interesting observations can be made using this sheaf. We shall start with a very important property of embedded manifolds.
\begin{definice} \label{def_weakembedding}
Let $(\cS,\iota)$ be be an immersed submanifold. Let $\phi: \cN \rightarrow \M$ be a graded smooth map, such that 
\begin{enumerate}[(i)]
\item $\ul{\phi}(N) \subseteq \ul{\iota}(S)$;
\item $\phi^{\ast}( \J_{\cS}) = 0$;
\end{enumerate}
We say that $(\cS,\iota)$ is \textbf{weakly embedded}, if for any such $\phi$ there exists a unique graded smooth map $\phi': \cN \rightarrow \cS$ satisfying $\iota \circ \phi' = \phi$. Note that the conditions $(i)$ and $(ii)$ on $\phi$ are necessary. 
\end{definice}
\begin{tvrz} \label{tvrz_weakembedd}
Let $(\cS,\iota)$ be an embedded submanifold. Then it is weakly embedded. 
\end{tvrz}
\begin{proof}
Let $\phi: \cN \rightarrow \M$ be any graded smooth map satisfying $(i)$ and $(ii)$ as in Definition \ref{def_weakembedding}. Since $\ul{\phi}(N) \subseteq \ul{\iota}(S)$ and $(S,\ul{\iota})$ is an embedded submanifold of $M$, there is a unique smooth map $\ul{\phi}': N \rightarrow S$, such that $\ul{\iota} \circ \ul{\phi}' = \ul{\phi}$. This is a standard statement, see e.g. Theorem 5.29 in \cite{lee2012introduction}. 

First, assume that $\iota^{\ast}_{U}: \C^{\infty}_{\M}(U) \rightarrow \C^{\infty}_{\cN}(U \cap S)$ is surjective for any $U \in \Op(M)$. Consequently, there is an induced sheaf isomorphism $\hat{\iota}^{\ast}: \C^{\infty}_{\M} / \J_{\cS} \rightarrow \ul{\iota}_{\ast} \C^{\infty}_{\cS}$. On the other hand, the assumption (ii) ensures that $\phi^{\ast}$ induces a sheaf morphism $\hat{\phi}^{\ast}: \C^{\infty}_{\M} / \J_{\cS} \rightarrow \ul{\phi}_{\ast} \C^{\infty}_{\cN}$. 

Now, for any $V \in \Op(S)$, one can find $U \in \Op(M)$, such that $V = U \cap S$. Define
\begin{equation}
\phi'^{\ast}_{V} := \hat{\phi}^{\ast}_{U} \circ ( \hat{\iota}^{\ast}_{U})^{-1}: \C^{\infty}_{\cS}(V) \rightarrow \C^{\infty}_{\cN}( \ul{\phi}'^{-1}(V)) \equiv (\ul{\phi}'_{\ast} \C^{\infty}_{\cN})(V)
\end{equation}
where one uses the fact $\ul{\phi}^{-1}(U) = \ul{\phi}'^{-1}(V)$. By construction, this is a graded algebra morphism. It is easy to see that the definition does not depend on the particular choice of $U$ and $\phi'^{\ast}_{V}$ is natural in $V$. To prove that $\phi' = (\ul{\phi}', \phi'^{\ast})$ defines a graded smooth map, it remains to verify the property (iii) of Definition \ref{def_gLRS}. For any $n \in N$, $V \in \Op_{\ul{\phi}'(n)}(S)$ and $f \in \C^{\infty}_{\cS}(V)$, the condition $f(\ul{\phi}'(n)) = 0$ must imply $(\phi'^{\ast}_{V}(f))(n) = 0$. By construction, we find $g \in \C^{\infty}_{\M}(U)$, such that $V = S \cap U$ and $f = \iota^{\ast}_{U}(g)$. In particular, one has $g(\ul{\phi}(n)) = g( \ul{\iota}( \ul{\phi}'(n))) = f( \ul{\phi}'(n)) = 0$. Consequently, one has 
\begin{equation}
(\phi'^{\ast}_{V}(f))(n) = (\phi^{\ast}_{U}(g))(n) = g( \ul{\phi}(n)) = 0.
\end{equation}
We conclude that $\phi' = (\ul{\phi}', \phi'^{\ast})$ defines a graded smooth map from $\cN$ to $\cS$. Moreover, it is obvious from the construction that $\phi'$ is the unique map satisfying $\phi = \iota \circ \phi'$. 

Now, let us go back to the general embedded submanifold $(\cS, \iota)$. By Proposition \ref{tvrz_embeddinghassurjectivepullback}, there is $Y \in \Op(M)$, such that $\ul{\iota}(S) \subseteq Y$ and $\iota^{\ast}_{Y}: \C^{\infty}_{\M}(Y) \rightarrow \C^{\infty}_{\cS}(S)$ is surjective. Since we may view $(\cS,\iota)$ as an embedded submanifold of $\M|_{Y}$ and $\phi$ as a graded smooth map $\phi: \cN \rightarrow \M|_{Y}$ satisfying the assumptions $(i)$ and $(ii)$, we obtain $\phi'$ from the previous paragraph together with Lemma \ref{lem_surjectivity}. This finishes the proof. 
\end{proof}
\begin{rem} \label{rem_weakweakembedd}
For any graded smooth map $\phi: \cN \rightarrow \M$, the properties $(i)+(ii)$ of Definition \ref{def_weakembedding} are equivalent to $(i)+(ii)'$, where:

$(ii)'$ There exists $U \in \Op(M)$, such that $\ul{\iota}(S) \subseteq U$ and $\phi^{\ast}_{U}( \J_{\cS}(U)) = 0$. 

Indeed, assume that $\phi: \cN \rightarrow \M$ satisfies $(i)$ and $(ii)'$. Using Proposition \ref{tvrz_extensionlemma}, one can show that $\phi^{\ast}_{V}( \J_{\cS}(V)) = 0$ for all $V \in \Op(U)$. But the assumption $(i)$ ensures that for \textit{any} $V \in \Op(M)$ and $f \in \C^{\infty}_{\M}(V)$, one has $\phi^{\ast}_{V}(f) = \phi^{\ast}_{V \cap U}(f|_{V \cap U})$. In particular, one has $\phi^{\ast}_{V}( \J_{\cS}(V)) = 0$ for any $V \in \Op(M)$ and we see that $\phi$ has the property $(ii)$. 
\end{rem}

We find the following characteristic properties of the sheaf of ideals $\J_{\cS}$ for closed embeddings.
\begin{tvrz} \label{tvrz_JSprops}
Let $(\cS, \iota)$ be an embedded submanifold of $\M$. Then the sheaf of ideals $\J_{\cS} := \ker(\iota)$ has the following properties:
\begin{enumerate}[(i)]
\item Let $(\cS,\iota)$ be a closed embedded submanifold. Then $\J_{\cS}(M) \subseteq \J_{\M}^{m}(M)$, if and only if $m \in \ul{\iota}(S)$. Recall that $\J_{\M}^{m}$ is the sheaf of ideals of functions vanishing at $m$, see (\ref{eq_sheaofidealsvanishing}). 
\item For each $m \in \ul{\iota}(S)$, there exists a finite graded set $\G \subseteq \J_{\cS}(M)$ consisting of functions independent at $m$ and $U \in \Op_{m}(M)$, such that for all $U' \in \Op(U)$, the graded subset $\G|_{U'} = \{ f|_{U'} \; | \; f \in \G \}$ generates the ideal $\J_{\cS}(U')$.
\item $\C^{\infty}_{\cS}$ is isomorphic to the inverse image sheaf $\ul{\iota}^{-1}( \C^{\infty}_{\M} / \J_{\cS})$. 
\end{enumerate}
\end{tvrz}
\begin{proof}
Let us start by proving $(i)$. Suppose that $\J_{\cS}(M) \subseteq \J^{m}_{\M}(M)$, such that $m \notin \ul{\iota}(S)$. Since $M - \ul{\iota}(S)$ is open, there exists $U \in \Op_{m}(M - \ul{\iota}(S))$. Choose $V \in \Op_{m}(U)$, such that $\ol{V} \subseteq U$, and pick a smooth bump function $\lambda \in \C^{\infty}_{\M}(M)$ supported in $U$, such that $\lambda|_{V} = 1$. One can write $S = \ul{\iota}^{-1}( \supp(\lambda)^{c})$, whence $\iota^{\ast}_{M}(\lambda) = \iota^{\ast}_{\supp(\lambda)^{c}}( \lambda|_{\supp(\lambda)^{c}}) = 0$. This shows that $\lambda \in \J_{\cS}(M) \subseteq \J^{m}_{\M}(M)$. But $\lambda(m) = \lambda|_{V}(m) = 1$, which is a contradiction.  

Conversely, let $m \in \ul{\iota}(S)$. Consider $f \in \J_{\cS}(M)$. We can write $m = \ul{\iota}(s)$ for a unique $s \in S$, whence $0 = (\iota^{\ast}_{M}f)(s) = f( \ul{\iota}(s)) = f(m)$. This proves that $f \in \J^{m}_{\M}(M)$ and the proof of $(i)$ is finished. 

Next, let us prove $(ii)$. Suppose $m = \ul{\iota}(s)$ for some $s \in S$. Since $\iota: \cS \rightarrow \M$ is an immersion at $s$, there exist graded local charts $(V,\psi)$ for $\cS$ around $s$ and $(W,\varphi)$ for $\M$ around $m$, as in Theorem \ref{thm_immersion}-$(iii)$. We can shrink $W$ so that $\ul{\iota}(V) = W \cap \ul{\iota}(S)$, see the proof of Proposition \ref{tvrz_embeddinghassurjectivepullback}. Let $(n_{j})_{j \in \Z} := \gdim(\M)$ and $(s_{j})_{j \in \Z} := \gdim(\cS)$. 

For each $j \in \Z$, we thus have the coordinate functions $(\bbz^{A_{j}}_{(j)})_{A_{j}=1}^{n_{j}}$, such that $|\bbz^{A_{j}}_{(j)}| = j$ and $\iota^{\ast}_{W}( \bbz^{A_{j}}_{(j)}) = 0$ for all $s_{j} < A_{j} \leq n_{j}$. We see that $\bbz^{A_{j}}_{(j)} \in \J_{\cS}(W)$ for any $j \in \Z$ and $s_{j} < A_{j} \leq n_{j}$. Pick any $U \in \Op_{m}(W)$, such that $\ol{U} \subseteq W$. Let $\lambda \in \C^{\infty}_{\M}(M)$ be a graded bump function supported in $W$, such that $\lambda|_{U} = 1$. For each $j \in \Z$ and $s_{j} < A_{j} \leq n_{j}$, let 
\begin{equation}
\bbu^{A_{j}}_{(j)} := \lambda \cdot \bbz^{A_{j}}_{(j)} \in \C^{\infty}_{\M}(M).
\end{equation}
As $\J_{\cS}$ is a sheaf of ideals, we see that $\bbu^{A_{j}}_{(j)} \in \J_{\cS}(M)$. Let $\G_{j} := \{ \bbu^{A_{j}}_{(j)} \}_{s_{j} < A_{j} \leq n_{j}}$. We have to prove that $\J_{\cS}(V)$ is for any $V \in \Op(U)$ generated by the graded set $\G|_{V}$, where $(\G|_{V})_{j} := \{ \bbu^{A_{j}}_{(j)}|_{V} \}_{s_{j} < A_{j} \leq n_{j}}$. Pick any $f \in \J_{\cS}(V)$. It follows that its local representative $\hat{f} \in \C^{\infty}_{(n_{j})}(\ul{\varphi}(V))$ satisfies the assumptions of Lemma \ref{lem_fdecompasgenerators}. We can thus write it as 
\begin{equation}
\hat{f} = \sum_{j \in \Z} \sum_{A_{j}=s_{j}+1}^{n_{j}} \hat{f}^{A_{j}}_{(j)} \cdot \bbz^{A_{j}}_{(j)},
\end{equation}
for some functions $\hat{f}^{A_{j}}_{(j)} \in \C^{\infty}_{(n_{j})}(\ul{\varphi}(V))$. It follows that $f$ can be now decomposed as 
\begin{equation}
f = \sum_{j \in \Z} \sum_{A_{j}=s_{j}+1}^{n_{j}} \varphi^{\ast}_{\ul{\varphi}(V)}(\hat{f}^{A_{j}}_{(j)}) \cdot \bbz^{A_{j}}_{(j)}|_{V} = \sum_{j \in \Z} \sum_{A_{j}=s_{j}+1}^{n_{j}} \varphi^{\ast}_{\ul{\varphi}(U)}(\hat{f}^{A_{j}}_{(j)}) \cdot \bbu^{A_{j}}_{(j)}|_{V}.
\end{equation} 

This proves that $\G|_{V}$ generates $\J_{\cS}(V)$ for any $V \in \Op(U)$. Finally, we have to argue that the functions in $\G$ are independent at $m$. For any $j \in \Z$ and $s_{j} < A_{j} \leq n_{j}$, one has $(\dr{\bbu}^{A_{j}}_{(j)})_{m} = ( \dr( \bbu^{A_{j}}_{(j)}|_{U}))_{m} = (\dr{\bbz}^{A_{j}}_{(j)})_{m}$, where we have used Proposition \ref{tvrz_differentialatm}-$(i)$. The claim then follows from Proposition \ref{tvrz_differentialatm}-$(iii)$. This finishes the proof of the property $(ii)$. 

Finally, let us prove $(iii)$. For the construction of an inverse image sheaf, see the part (a) of the proof of Proposition \ref{tvrz_ap_pullbackVB}. In fact, we will now define a presheaf isomorphism 
\begin{equation}
\psi: \iota_{P}^{-1}( \C^{\infty}_{\M} / \J_{\cS}) \rightarrow \C^{\infty}_{\cS}.
\end{equation}
Let $V \in \Op(S)$. Let $[\natural_{U}(f)]_{V} \in (\iota_{P}^{-1}( \C^{\infty}_{\M}/\J_{\cS}))(V)$, represented by a class $\natural_{U}(f)$ in $(\C^{\infty}_{\M} / \J_{\cS})(U)$, which is in turn represented by some $f \in \C^{\infty}_{\M}(U)$. $U \in \Op(M)$ satisfies $\ul{\iota}(V) \subseteq U$. As $\ul{\iota}$ is an embedding, we may shrink it so that $\ul{\iota}(V) = U \cap \ul{\iota}(S)$. In particular, we have $V = \ul{\iota}^{-1}(U)$. Define 
\begin{equation}
\psi_{V}( [\natural f]_{V}) := \iota^{\ast}_{U}(f) \in \C^{\infty}_{\M}(V).
\end{equation}
It is easy to check that $\psi_{V}$ is a well-defined graded algebra morphism natural in $V$. It is injective, as $\iota^{\ast}_{U}(f) = 0$ implies $\natural_{U}(f) = 0$ and thus also $[\natural_{U}(f)]_{V} = 0$. Finally, the surjectivity of $\psi_{V}$ follows immediately from Corollary \ref{cor_extensionsubmanifolds}. This proves that $\psi$ is a presheaf isomorphism. In particular, this shows that $\iota_{P}^{-1}(\C^{\infty}_{\M}/\J_{\cS})$ is already a sheaf, hence canonically isomorphic to $\iota^{-1}(\C^{\infty}_{\M}/\J_{\cS})$. 
\end{proof}
\begin{rem}
The if part of the claim $(i)$ is not true for general embedded submanifolds. Consider the ordinary manifold $\R$ and its open submanifold $S := \R - \{0\}$. If $f \in \J_{S}(U)$ for some $U \in \Op_{0}(\R)$, we have $f|_{U - \{0\}} = 0$. But by continuity, one has also $f(0) = 0$. Thus $f \in \J^{0}_{\R}(U)$ and we see that $\J_{S} \subseteq \J_{\R}^{0}$. But this certainly does not imply $0 \in S$. 

Moreover, it is interesting to observe where exactly the proof of the claim $(ii)$ fails for general immersed submanifolds. Although one can still choose the charts $(V,\psi)$ for $\cS$ and $(W,\varphi)$ for $\M$ as in Theorem \ref{thm_immersion}-$(iii)$, one cannot assume that $\ul{\iota}(V) = W \cap \ul{\iota}(S)$. We can thus only conclude that $\iota^{\ast}_{W}( \bbz^{A_{j}}_{(j)})|_{V} = 0$. This is not enough to prove that $\bbz^{A}_{(j)} \in \J_{\cS}(W)$.  
\end{rem}
Note that for closed embedded submanifolds, Proposition \ref{tvrz_weakembedd} can be modified a bit.
\begin{tvrz} \label{tvrz_clweakembedd}
Let $(\cS,\iota)$ be a closed embedded submanifold. Let $\phi: \cN \rightarrow \M$ be any graded smooth map, such that $\phi^{\ast}_{M}(\J_{\cS}(M)) = 0$. 

Then there exists a unique graded smooth map $\phi': \cN \rightarrow \cS$, such that $\iota \circ \phi' = \phi$. 

This statement is not true for general embedded submanifolds. 
\end{tvrz}
\begin{proof}
Thanks to Proposition \ref{tvrz_weakembedd} and Remark \ref{rem_weakweakembedd}, it suffices to prove that for a closed embedded submanifold $(\cS,\iota)$, the assumption on $\phi$ already implies the property (i) of Definition \ref{def_weakembedding}. 

Let $n \in N$ be arbitrary. By assumption, every $f \in \J_{\cS}(M)$ satisfies $\phi^{\ast}_{M}(f) = 0$. In particular, we have $f(\ul{\phi}(n)) = 0$. Hence $\J_{\cS}(M) \subseteq \J^{\ul{\phi}(n)}_{\M}(M)$. But then $\ul{\phi}(n) \in \ul{\iota}(S)$ by Proposition \ref{tvrz_JSprops}-$(i)$. We conclude that $\ul{\phi}(N) \subseteq \ul{\iota}(S)$ and the statement follows. 

To show that this is not true for general embedded submanifolds, consider the ordinary manifold $M = \R$ and its embedded submanifold $S = (-1,1)$. Let $N = \R$ and let $\phi: N \rightarrow M$ be a smooth map $\phi(x) = \sin(x)$. Suppose $f \in \C^{\infty}_{M}(\R)$ vanishes on $S$. Then $(\phi^{\ast}_{M}(f))(x) = 0$ for all $x \in \R - \{ (2k+1) \frac{\pi}{2}  \; | \; k \in \Z \}$. But $\phi^{\ast}_{M}(f)$ is continuous on entire $\R$, hence $(\phi^{\ast}_{M}(f))(x) = 0$ for all $x \in \R$. We see that $\phi^{\ast}_{M}( \J_{S}(M)) = 0$. On the other hand, we have $\phi(N) = [-1,1]$, which fails to be a subset of $S$. 
\end{proof}

\begin{lemma} \label{lem_equivalentsubmafolds}
Let $(\cS,\iota)$ and $(\cS',\iota')$ be two embedded submanifolds of $\M$. We say that they are \textbf{equivalent}, if there exists a unique graded diffeomorphism $\varphi: \cS \rightarrow \cS'$, such that $\iota' \circ \varphi = \iota$. Suppose $(\cS,\iota)$ and $(\cS',\iota')$ are two \textit{closed} embedded submanifolds.

Then $(\cS,\iota)$ and $(\cS',\iota')$ are equivalent, if and only if $\J_{\cS} = \J_{\cS'}$. 
\end{lemma}
\begin{proof}
One direction is trivial. Hence assume that $\J_{\cS} = \J_{\cS'}$. Clearly $\iota^{\ast}_{M}(\J_{\cS'}(M)) = 0$. By Proposition \ref{tvrz_clweakembedd}, there thus exists a unique graded smooth map $\varphi: \cS \rightarrow \cS'$, such that $\iota' \circ \varphi = \iota$. Replacing the role of $\cS$ and $\cS'$, one obtains $\psi: \cS' \rightarrow \cS$, such that $\iota \circ \psi = \iota'$. Finally, we find $\iota' \circ (\varphi \circ \psi) = \iota'$, and the uniqueness claim of Proposition \ref{tvrz_clweakembedd} implies $\varphi \circ \psi = \1_{\cS'}$. The claim $\psi \circ \varphi = \1_{\cS}$ is proved in the same way. This proves that $(\cS,\iota)$ and $(\cS',\iota')$ are equivalent. 
\end{proof}

One can ask which sheaves of ideals in $\C^{\infty}_{\M}$ correspond to kernel subsheaves of closed embedded submanifolds. Not surprisingly, as Proposition \ref{tvrz_JSprops} suggests, these are sheaves of regular ideals discussed in the previous subsection. 

\begin{rem} \label{rem_vardimnuisance}
There is a minor technical nuisance with the following statement. Throughout this paper, we usually only consider graded manifolds of a given graded dimension. In order to avoid unnecessary assumptions, we have to relax this requirement. It turns out that the resulting graded dimension of the submanifold $(\cS,\iota)$ can be different for each connected component of the underlying topological (sub)space $S$. 
\end{rem}
\begin{theorem} \label{thm_idealissubmanifold}
Let $(\cS,\iota)$ be a closed embedded submanifold. Then $\J_{\cS} = \ker(\iota^{\ast})$ is a sheaf of regular ideals. Conversely, to any sheaf of regular ideals $\I$, there exists a closed embedded submanifold $(\cS,\iota)$, such that $\I = \J_{\cS}$. By Lemma \ref{lem_equivalentsubmafolds}, it is unique up to an equivalence.
\end{theorem}
\begin{proof}
When $(\cS,\iota)$ is a closed embedded submanifold, $\J_{\cS}$ is regular by Proposition \ref{tvrz_JSprops}.

Hence assume  that $\I$ is a regular sheaf of ideals. We will prove the statement in several steps. Some technical details were moved to the Appendix, see Theorem \ref{thm_ap_idealissubmanifold}.
\begin{enumerate}[(a)]
\item \textbf{Underlying embedded submanifold}: Let us define the subset $S$ of $M$ as 
\begin{equation}
S := \{ m \in M \; | \; \I(M) \subseteq \J^{m}_{\M}(M) \} \subseteq M.
\end{equation}
This is a closed subset of $M$, since it can be equivalently defined as 
\begin{equation}
S = \hspace{-3mm} \bigcap_{f \in \I(M)} \ul{f}^{-1}(0). 
\end{equation}
Let $\ul{\iota}: S \rightarrow M$ be the canonical inclusion. Let $s \in S$. By assumption, there exists $U \in \Op_{s}(M)$ and a finite graded set $\G \subseteq \I(M)$, such that $\I(V) = \<\G|_{V}\>$ for all $V \in \Op(U)$, and $\G$ consists of functions independent at $s$. For each $j \in \Z$, write 
\begin{equation}
\G_{j}:= \{ \bbu^{A_{j}}_{(j)} \}_{A_{j} = s_{j}+1}^{n_{j}},
\end{equation}
where $(n_{j})_{j} := \gdim(\M)$ and $s_{j} := n_{j} - \# \G_{j}$. Consider the graded domain $(\R^{n_{0}-s_{0}})^{(n_{j}-s_{j})}$ with standard coordinate functions $(\bby^{K_{j}}_{(j)})_{K_{j}=1}^{n_{j}-s_{j}}$, for each degree $j \in \Z$. Set
\begin{equation}
\phi^{\ast}_{\R^{n_{0}- s_{0}}}( \bby^{K_{j}}_{(j)}) := \bbu^{K_{j} + s_{j}}_{(j)},
\end{equation}
for each $j \in \Z$ and $K_{j} \in \{1,\dots,n_{j}-s_{j}\}$. By Theorem \ref{thm_globaldomain}, this uniquely determines the graded smooth map $\phi: \M \rightarrow (\R^{n_{0}-s_{0}})^{(n_{j}-s_{j})}$. Note that its underlying map takes the form 
\begin{equation}
\ul{\phi}(m) = ( \bbu^{s_{0}+1}_{(0)}(m), \dots, \bbu^{n_{0}}_{(0)}(m)),
\end{equation}
for all $m \in M$. In particular, observe that $\ul{\phi}(s) = (0,\dots,0)$. This follows from the fact that $\I(M) \subseteq \J^{s}_{\M}(M)$. Using the formula (\ref{eq_differentialpullback}) and the assumed independence of functions in $\G$ at $s$, we see that $\phi$ is a submersion at $s$. By Proposition \ref{thm_submersion}, there exists a graded local chart $(U',\varphi)$ around $s$, together with open cubes $\hat{V} \subseteq \R^{n_{0}-s_{0}}$ and $\hat{W} \subseteq \R^{s_{0}}$, such that 
\begin{enumerate}[(i)]
\item $\ul{\phi}(U') \subseteq \hat{V}$, $\ul{\varphi}(U') = \hat{W} \times \hat{V}$, and $\ul{\varphi}(s) = (0,0)$.
\item For each $j \in \Z$, one can reorder the degree $j \in \Z$ coordinate functions $(\bbz^{A_{j}}_{(j)})_{A_{j}=1}^{n_{j}}$ induced by $(U',\varphi)$, so that for each $K_{j} \in \{1,\dots,n_{j}-s_{j}\}$, one has 
\begin{equation}
\phi^{\ast}_{\hat{V}}( \bby^{K_{j}}_{(j)})|_{U'} = \bbz^{K_{j}+s_{j}}_{(j)}.
\end{equation}
\end{enumerate}
Without the loss of generality, we may assume that $U = U'$. For each $V \in \Op(U)$, we have 
\begin{equation}
\bbz^{A_{j}}_{(j)}|_{V} = \phi^{\ast}_{\hat{V}}( \bby^{A_{j}-s_{j}}_{(j)})|_{V} = \bbu^{A_{j}}_{(j)}|_{V},
\end{equation}
for all $j \in \Z$ and $A_{j} \in \{s_{j}+1, \dots, n_{j}\}$. We claim that $(U,\ul{\varphi})$ is a $s_{0}$-slice chart for $S$, that is 
\begin{equation} \label{eq_SbyIslicechart}
\ul{\varphi}(U \cap S) = \ul{\varphi}(U) \cap (\R^{s_{0}} \times \{0\}).
\end{equation}
First, let $s' \in U \cap S$. Then $\I(M) \subseteq \J^{s'}_{\M}(M)$ and thus
\begin{equation}
\begin{split}
\ul{\varphi}(s') = & \ ( \bbz^{1}_{(0)}(s'), \dots, \bbz^{s_{0}}_{(0)}(s'), \bbz^{s_{0}+1}_{(0)}(s'), \dots, \bbz^{n_{0}}_{(0)}(s')) \\
= & \ ( \bbz^{1}_{(0)}(s'), \dots, \bbz^{s_{0}}_{(0)}(s'), \bbu^{s_{0}+1}_{(0)}(s'), \dots, \bbu^{n_{0}}_{(0)}(s')) \\
= & \ ( \bbz^{1}_{(0)}(s'), \dots, \bbz^{s_{0}}_{(0)}(s'), 0, \dots, 0).
\end{split}
\end{equation}
This proves the inclusion $\subseteq$. Conversely, suppose $s' \in U$ is a point, such that $\ul{\varphi}(s') \in \R^{s_{0}} \times \{0\}$, that is $\bbu^{A_{0}}_{(0)}(s') = 0$ for all $A_{0} \in \{s_{0}+1,\dots,n_{0}\}$. To show that $s' \in S$, we must argue that $\I(M) \subseteq \J^{s'}_{\M}(M)$. Let $f \in \I(M)$ be arbitrary. Then $f|_{U} \in \< \G|_{U} \>$, whence 
\begin{equation}
f|_{U} = \sum_{j \in \Z} \sum_{A_{j}=s_{j}+1}^{n_{j}} f^{A_{j}}_{(j)} \cdot \bbu^{A_{j}}_{(j)}|_{U},
\end{equation}
for some functions $f^{A_{j}}_{(j)} \in \C^{\infty}_{\M}(U)$. Evaluating this at $s'$ gives $f(s') = 0$, that is $f \in \J^{s'}_{\M}(M)$. This proves the inclusion $\supseteq$ of (\ref{eq_SbyIslicechart}). 

For each $s \in S$, we have thus found an $s_{0}$-slice chart $(U,\ul{\varphi})$ for $S$ around $s$. Note that the dimension $s_{0}$ can be different for each $s$, but it is the same for all $s$ from the single connected component of $S$. With this and Remark \ref{rem_vardimnuisance} in mind, we conclude that $(S,\ul{\iota})$ becomes a closed embedded submanifold of $M$. 
\item \textbf{Structure sheaf and the embedding}: Having the underlying closed embedded submanifold $(S,\ul{\iota})$, one may look at Proposition \ref{tvrz_JSprops}-$(iii)$. This leads us to \textit{define} $\C^{\infty}_{\cS}$ to be the sheaf
\begin{equation}
\C^{\infty}_{\cS} := \ul{\iota}^{-1}( \C^{\infty}_{\M} / \I),
\end{equation}
see the part (a) of the proof of Proposition \ref{tvrz_ap_pullbackVB}. In fact, the assignment $\F \mapsto \ul{\iota}^{-1}\F$ defines a functor from $\Sh(M,\gcAs)$ to $\Sh(S,\gcAs)$. This functor is adjoint to the pushforward sheaf functor $\G \mapsto \ul{\iota}_{\ast} \G$. In other words, there is a bijection 
\begin{equation}
\mu_{\F,\G}: \Sh_{0}( \ul{\iota}^{-1}\F, \G) \rightarrow \Sh_{0}(\F, \ul{\iota}_{\ast} \G),
\end{equation}
natural in $\F$ and $\G$, where $\Sh_{0}$ denotes the collection of all sheaf morphisms between the respective sheaves. Now, it can be shown that if $(M,\C^{\infty}_{\M})$ is a graded locally ringed space, then so is $(M, \C^{\infty}_{\M}/\I)$, and consequently also $(S, \C^{\infty}_{\cS})$. In fact, for each $s \in S$, there are canonical identifications $\C^{\infty}_{\cS,s} \cong (\C^{\infty}_{\M} / \I)_{s} \cong \C^{\infty}_{\M,s} / \I_{s}$ and the Jacobson radical $\frJ(\C^{\infty}_{\cS,s})$ then corresponds to $\frJ(\C^{\infty}_{\M,s}) / \I_{s}$. For details, see Theorem \ref{thm_ap_idealissubmanifold}. 

Recall that for any sheaf of ideals $\I$, the quotient $\C^{\infty}_{\M} / \I$ is a sheaf, see Proposition \ref{eq_quotientpresheafissheaf}. The canonical quotient map can be viewed as a sheaf morphism $\natural: \C^{\infty}_{\M} \rightarrow \C^{\infty}_{\M} / \I$. Applying the functor $\ul{\iota}^{-1}$, we obtain $\ul{\iota}^{-1}(\natural): \ul{\iota}^{-1}(\C^{\infty}_{\M}) \rightarrow \C^{\infty}_{\cS}$. Finally, let 
\begin{equation} \label{eq_iotaast}
\iota^{\ast} := \mu_{ \C^{\infty}_{\M}, \C^{\infty}_{\cS}}( \ul{\iota}^{-1}(\natural)): \C^{\infty}_{\M} \rightarrow \ul{\iota}_{\ast} \C^{\infty}_{\cS}.
\end{equation}
We claim that $\iota := (\ul{\iota}, \iota^{\ast})$ becomes a morphism of graded locally ringed spaces. Indeed, one only has to verify the property (iii) in Definition \ref{def_gLRS}. But for each $s \in S$, with respect to the above identifications, the graded algebra morphism $\iota_{(s)}: \C^{\infty}_{\M,s} \rightarrow \C^{\infty}_{\cS,s} \cong \C^{\infty}_{\M,s} / \I_{s}$ corresponds to the canonical quotient map, hence $\iota_{(s)}( \frJ(\C^{\infty}_{\M,s}) = \frJ( \C^{\infty}_{\M,s}) / \I_{s} \cong \frJ(\C^{\infty}_{\cS,s})$. Again, the details are discussed in Theorem \ref{thm_ap_idealissubmanifold}. 

In conclusion, we have constructed a graded locally ringed space $(S, \C^{\infty}_{\cS})$ together with a morphism of graded locally ringed spaces $\iota: (S, \C^{\infty}_{\cS}) \rightarrow (M, \C^{\infty}_{\M})$. 
\item \textbf{Graded smooth atlas for $\cS$}: Let $s \in S$ be fixed. In the part $(a)$, we have constructed a particular graded local chart $(U,\varphi)$ for $\M$ around $s$. We have 
\begin{equation}
\varphi: \M|_{U} \rightarrow (\hat{W} \times \hat{V})^{(n_{j})}.
\end{equation}
Let $\hat{\pi}_{1}: (\hat{W} \times \hat{V})^{(n_{j})} \rightarrow \hat{W}^{(s_{j})}$ be the canonical projection. Let us consider a composition
\begin{equation}
\psi := \hat{\pi}_{1} \circ \varphi \circ \iota|_{U \cap S}: \cS|_{U \cap S} \rightarrow \hat{W}^{(s_{j})}
\end{equation}
of morphisms of graded locally ringed spaces. We claim that this is an isomorphism, that is $(U \cap S, \psi)$ becomes a graded local chart for $\cS$ around $s$. 

We do so by constructing its inverse $\chi: \hat{W}^{(s_{j})} \rightarrow \cS|_{U \cap S}$. Let $\hat{\iota}_{0}: \hat{W}^{(s_{j})} \rightarrow (\hat{W} \times \hat{V})^{(n_{j})}$ denote the ``zero section'', a graded smooth map defined as in Lemma \ref{lem_fdecompasgenerators}. Let 
\begin{equation}
\chi_{0} := \varphi^{-1} \circ \hat{\iota}_{0}: \hat{W}^{(s_{j})} \rightarrow \M|_{U}.
\end{equation}
We claim that $\ul{\chi_{0}}( \hat{W}) \subseteq U \cap S$ and $\chi_{0}^{\ast}(\I|_{U}) = 0$. For each $w \in \hat{W}$, we have $\ul{\hat{\iota}_{0}}(w) = (w,0) \in \ul{\varphi}(U) \cap (\R^{s_{0}} \times \{0\}) = \ul{\varphi}(U \cap S)$. Hence $\ul{\chi_{0}}(w) \in U \cap S$. Next, for any $V \in \Op(U)$, we have $\I(V) = \< \G|_{V} \>$ and it follows from the construction of the graded local chart $(U,\varphi)$ that $(\varphi^{-1})^{\ast}_{V}(\<\G|_{V}\>) \subseteq \ker( (\hat{\iota}_{0})^{\ast}_{\ul{\varphi}(V)})$. Thus $(\chi_{0})^{\ast}_{V}( \I(V)) = 0$. 

Now, we shall employ the following observation:
\begin{lemma} \label{lem_factorthroughI}
Let $\cN = (N, \O_{N})$ be any graded locally ringed space. Suppose $\chi_{0}: \cN \rightarrow \M|_{U}$ is a morphism of graded locally ringed spaces, such that $\ul{\chi_{0}}(N) \subseteq U \cap S$ and $\chi_{0}^{\ast}(\I|_{U}) = 0$. 

Then there exists a unique morphism of graded locally ringed spaces $\chi: \cN \rightarrow \cS|_{U \cap S}$, such that 
\begin{equation} \iota|_{U \cap S} \circ \chi = \chi_{0}. \end{equation}
\end{lemma}
For the proof, see Theorem \ref{thm_ap_idealissubmanifold}. Applying the lemma to $\chi_{0}: \hat{W}^{(s_{j})} \rightarrow \M|_{U}$, we obtain a unique morphism of graded locally ringed spaces $\chi: \hat{W}^{(s_{j})} \rightarrow \cS|_{U \cap S}$ satisfying $\iota|_{U \cap S} \circ \chi = \chi_{0}$. It remains to prove that $\chi$ is the inverse to $\psi$. One direction is immediate, one finds
\begin{equation}
\psi \circ \chi = (\hat{\pi}_{1} \circ \varphi \circ \iota|_{U \cap S}) \circ \chi = \hat{\pi}_{1} \circ \varphi \circ (\varphi^{-1} \circ \hat{\iota}_{0}) = \hat{\pi}_{1} \circ \hat{\iota}_{0} = \1_{\hat{W}^{(s_{j})}}. 
\end{equation}
The other direction is a bit more involved. It suffices to prove that $\iota|_{U \cap S} \circ (\chi \circ \psi) = \iota|_{U \cap S}$. Indeed, the uniqueness claim of Lemma \ref{lem_factorthroughI} then ensures that $\chi \circ \psi = \1_{\cS|_{U \cap S}}$. We have
\begin{equation}
\iota|_{U \cap S} \circ (\chi \circ \psi) = \chi_{0} \circ \psi = \varphi^{-1} \circ \hat{\iota}_{0} \circ \hat{\pi}_{1} \circ \varphi \circ \iota|_{U \cap S}. 
\end{equation} 
Consequently, we must prove that $(\hat{\iota}_{0} \circ \hat{\pi}_{1}) \circ \varphi \circ \iota|_{U \cap S} = \varphi \circ \iota|_{U \cap S}$. 

Let $\hat{X} \in \Op(\hat{V} \times \hat{W})$ and $f \in \C^{\infty}_{(n_{j})}(\hat{X})$ be arbitrary. Let 
\begin{equation}
\hat{X}' := (\ul{\hat{\iota}_{0}} \circ \ul{ \hat{\pi}_{1}})^{-1}(\hat{X}) = \{ (w,v) \in \hat{W} \times \hat{V} \; | \; (w,0) \in \hat{X} \}.
\end{equation}
Write $X := \ul{\varphi}^{-1}(\hat{X})$ and $X' := \ul{\varphi}^{-1}(\hat{X}')$. We thus need to prove that 
\begin{equation}
\iota^{\ast}_{X}( \varphi^{\ast}_{\hat{X}}(f)) = \iota^{\ast}_{X'}( \varphi^{\ast}_{\hat{X}'}( (\hat{\iota}_{0} \circ \hat{\pi}_{1})^{\ast}_{\hat{X}}(f)). 
\end{equation}
Now, let $Y := X \cap X'$ and $\hat{Y} := \hat{X} \cap \hat{X}'$. Observe that $X \cap S = Y \cap S = X' \cap S$. This follows from the fact that $(U,\ul{\varphi})$ is an $s_{0}$-slice local chart for $S$. By (trivially) restricting the both sides to $Y \cap S$ and using the naturality of the involved morphisms, we obtain the equation 
\begin{equation}
(\iota^{\ast}_{Y} \circ \varphi^{\ast}_{\hat{Y}})( f|_{\hat{Y}} - (\hat{\iota}_{0} \circ \hat{\pi}_{1})^{\ast}_{\hat{X}}(f)|_{\hat{Y}}) = 0. 
\end{equation}
Let $g := f|_{\hat{Y}} - (\hat{\iota}_{0} \circ \hat{\pi}_{1})^{\ast}_{\hat{X}}(f)|_{\hat{Y}} \in \C^{\infty}_{(n_{j})}(\hat{Y})$. It is easy to see that $(\hat{\iota}_{0})^{\ast}_{\hat{Y}}(g) = 0$. But it then follows from Lemma \ref{lem_fdecompasgenerators} that $g$ can be written as a finite sum
\begin{equation}
g = \sum_{j \in \Z} \sum_{A_{j} = s_{j}+1}^{n_{j}} f^{A_{j}}_{(j)} \cdot \bbz^{A_{j}}_{(j)},
\end{equation}
for some functions $f^{A_{j}}_{(j)} \in \C^{\infty}_{(n_{j})}(\hat{Y})$. But this proves that $\varphi^{\ast}_{\hat{Y}}(g) \in \< \G|_{Y} \> = \I(Y)$ and thus $(\iota^{\ast}_{Y} \circ \varphi^{\ast}_{\hat{Y}})(g) = 0$, as was to  be proved. Hence $\chi \circ \psi = \1_{\cS|_{U \cap S}}$ and we have just constructed a graded local chart $(U \cap S, \psi)$ for $\cS$ around $s$. Taking Remark \ref{rem_vardimnuisance} into account, this proves that $\cS := (S, \C^{\infty}_{\cS})$ is a graded manifold. 

We can now prove that $\iota: \cS \rightarrow \M$ is an immersion. For each $s \in S$, we have the local charts $(U,\varphi)$ for $\M$ and $(U \cap S, \psi)$ for $\cS$, constructed as above. It follows that the local representative of $\iota$ with these two graded charts is precisely the zero section $\hat{\iota}_{0}: \hat{W}^{(s_{j})} \rightarrow (\hat{V} \times \hat{W})^{(n_{j})}$, which is obviously an immersion. 

\item \textbf{Comparing the two sheaves of ideals}: We have to show that $\I = \J_{\cS}$. The inclusion $\I \subseteq \J_{\cS}$ follows from the construction. As both $\I$ and $\J_{\cS}$ are sheaves of ideals, it suffices to prove that to each $m \in M$, there is $U \in \Op_{m}(M)$, such that $\J_{\cS}|_{U} \subseteq \I|_{U}$. 

First, let $m \in M - S$. This means that there is some $U \in \Op_{m}(M)$ and $f \in \I(U)$, such that $f(m) \neq 0$. By shrinking $U$ if necessary, we may assume that $f$ has the multiplicative inverse, see Proposition \ref{tvrz_invertibility}. But then $\I|_{U} = \C^{\infty}_{\M}|_{U}$ and the inclusion $\J_{\cS}|_{U} \subseteq \I|_{U}$ is obvious. 

Hence let $s \in S$. Let $(U,\varphi)$ be the graded local chart constructed as above. For every $V \in \Op(U)$ and any $f \in \J_{\cS}(V)$, its local representative $\hat{f} \in \C^{\infty}_{(n_{j})}(\ul{\varphi}(V))$ satisfies $(\hat{\iota}_{0})^{\ast}_{\ul{\varphi}(V)}(\hat{f}) = 0$. It follows from Lemma \ref{lem_fdecompasgenerators} that $\hat{f}$ can be written as a finite sum 
\begin{equation} \hat{f} = \sum_{j \in \Z} \sum_{A_{j}=s_{j}+1}^{n_{j}} f^{A_{j}}_{(j)} \cdot \bbz^{A_{j}}_{(j)},
\end{equation}
for some functions $f^{A_{j}}_{(j)} \in \C^{\infty}_{(n_{j})}( \ul{\varphi}(V))$. Consequently, $f = \varphi^{\ast}_{\ul{\varphi}(V)}(\hat{f}) \in \< \G|_{V} \> = \I(V)$. 
\end{enumerate}
This finishes the proof. 
\end{proof}
Let us now examine a few examples of closed embedded submanifolds and their corresponding regular sheaves of ideals. 
\begin{example} \label{ex_singlepointsubmafold}
In Example \ref{ex_underlyingissubmanifold}, we have argued that for every graded manifold $\M$, we may view the underlying manifold as a closed embedded submanifold $(M,i_{M})$. By definition, the corresponding sheaf of regular ideals is $\J^{\pg}_{\M}$, see (\ref{eq_sheaofidealspurelygraded}). See also Example \ref{ex_regid1}. 
\end{example}
\begin{example}
Let $\M$ be any graded manifold, and fix $a \in M$. In Example \ref{ex_regid2}, we have shown that $\J^{a}_{\M}$ is a sheaf of regular ideals. We have also shown that $S = \{ m \in M \; | \; \J^{a}_{\M}(M) \subseteq \J^{m}_{\M}(M) \} = \{a \}$. It is easy to see that the corresponding embedded submanifold is $(\{ a \}, \kappa_{a} )$, where $\{ a \}$ is a singleton manifold, viewed as a graded manifold, and $\kappa_{a}: \{ a \} \rightarrow \M$ is the constant mapping valued at $a \in M$, see Example \ref{ex_constantmapping}.
\end{example}

\begin{example} \label{ex_evenpart}
First, let us consider a graded domain $U^{(n_{j})}$, for some $U \in \Op(\R^{n_{0}})$. For every $V \in \Op(U)$, consider an ideal $\J^{\odd}_{(n_{j})}(V) := \< \{ \xi_{\mu} \; | \; |\xi_{\mu}| \text{ is odd} \} \>$. In other words, $f \in \J^{\odd}_{(n_{j})}(V)$, if each of its terms in formal power series in variables $(\xi_{\mu})_{\mu=1}^{n_{\ast}}$ contains some variable of odd degree. 

It is easy to see that $\J^{\odd}_{(n_{j})} \subseteq \C^{\infty}_{(n_{j})}$ is a sheaf of ideals. Moreover, for any graded morphism $\varphi: U^{(n_{j})} \rightarrow W^{(m_{j})}$, and any $V \in \Op(W)$, one has $\varphi^{\ast}_{V}( \J^{\odd}_{(m_{j})}(V)) \subseteq \J^{\odd}_{(n_{j})}( \ul{\varphi}^{-1}(V))$. 

Now, let $\M$ be any graded manifold. Let $\A = \{ (U_{\alpha}, \varphi_{\alpha}) \}_{\alpha \in I}$ be a graded smooth atlas on $\M$. For each $U \in \Op(M)$, define the graded subset
\begin{equation}
\J^{\odd}_{\M}(U) := \{ f \in \C^{\infty}_{\M}(U) \; | \; \hat{f}_{\alpha} \in \J^{\odd}_{(n_{j})}( \ul{\varphi}_{\alpha}(U \cap U_{\alpha})) \text{ for all } \alpha \in I \}.
\end{equation}
It is easy to see that this defines a sheaf of ideals $\J^{\odd}_{\M}$ in $\C^{\infty}_{\M}$.

Now, one has $\J^{\odd}_{\M} \subseteq \J^{m}_{\M}$ for any $m \in M$. For each $m \in M$, the open set $U \in \Op_{m}(M)$ and the generating set $\G$ are constructed as in Example \ref{ex_regid1}, except that we only use the coordinate functions of odd degree. This shows that $\J^{\odd}_{\M} \subseteq \C^{\infty}_{\M}$ is a regular sheaf of ideals. 

Consequently, we obtain a closed embedded submanifold $(\M_{0}, \iota_{0})$, called the \textbf{even part of $\M$}. Its underlying manifold is $M$, one has $\C^{\infty}_{\M_{0}} = \C^{\infty}_{\M} / \J^{\odd}_{\M}$ and $\iota_{0} = (\1_{M}, \natural)$, where $\natural: \C^{\infty}_{\M} \rightarrow \C^{\infty}_{\M}/\J^{\odd}_{\M}$ is the canonical quotient map. 
\end{example}
\subsection{Transversality and inverse images}
It is useful to introduce the concept of two transversal graded smooth maps, allowing us to construct ``inverse images'' of submanifolds. Most of the definitions mimic standard notions in differential geometry. 

\begin{definice}
Let $\M$, $\cN$ and $\cS$ be graded manifolds. Let $\phi: \cN \rightarrow \M$ and $\psi: \cS \rightarrow \M$ be graded smooth maps. We say that $\phi$ and $\psi$ are \textbf{transversal maps}, if for every $m \in \ul{\phi}(N) \cap \ul{\psi}(S)$ and every $(n,s) \in \ul{\phi}^{-1}(m) \times \ul{\psi}^{-1}(m)$, one has 
\begin{equation}
T_{m}\M = (T_{n}\phi)( T_{n}\cN) + (T_{s} \psi)(T_{s} \cS). 
\end{equation}
We write $\phi \pitchfork \psi$. If $(\cS,\iota)$ is an immersed submanifold of $\cN$, we say that $\phi: \cN \rightarrow \M$ is \textbf{transversal to the submanifold $\cS$} and write $\phi \pitchfork \cS$, if $\phi \pitchfork \iota$. Finally, we say that $(\cS,\iota)$ and $(\cS',\iota')$ are \textbf{transversal submanifolds} and write $\cS \pitchfork \cS'$, if $\iota \pitchfork \iota'$. 
\end{definice}

\begin{rem}
Note that $\phi \pitchfork \psi$ implies $|\phi| \pitchfork |\psi|$. This follows from the fact that $(T_{m}\M)_{0} \cong T_{m}M$ and the linear map $(T_{m}\phi)_{0}$ coincides with $T_{m} \ul{\phi}$, see Example \ref{ex_diffgradeddomains}. 
\end{rem}

\begin{tvrz} \label{tvrz_tangenttosubmafold}
Let $(\cS,\iota)$ be an embedded submanifold of $\M$. Let $\J_{\cS} = \ker(\iota^{\ast})$. Then 
\begin{equation}
T_{s} \cS = \{ v \in T_{s}\M \; | \; v([f]_{s}) = 0 \text{ for all } [f]_{s} \in (\J_{\cS})_{s} \},
\end{equation}
for every $s \in S$. 
\end{tvrz}
\begin{proof}
Let $v \in T_{s}\cS \subseteq T_{s}\M$. But this means that $v = (T_{s}\iota)(v')$ for $v' \in T_{s}\cS$. For any $[f]_{s} \in (\J_{\cS})_{s}$ represented by $f \in \J_{\cS}(U)$ for some $U \in \Op_{s}(M)$, we thus have 
\begin{equation}
v([f]_{s}) = ((T_{s}\iota)(v'))([f]_{s}) = v'( \iota_{(s)}([f]_{s})) = v'( [ \iota^{\ast}_{U}(f)]_{s}) = 0.
\end{equation}
Conversely, suppose $v \in T_{s}\M$ satisfies $v([f]_{s}) = 0$ for all $[f]_{s} \in (\J_{\cS})_{m}$. By Proposition \ref{tvrz_embeddinghassurjectivepullback}, there exists $Y \in \Op(M)$, such that $\ul{\iota}(S) \subseteq Y$ and $\iota^{\ast}_{Y}: \C^{\infty}_{\M}(Y) \rightarrow \C^{\infty}_{\cS}(S)$ is surjective. 

Let $[f']_{s} \in \C^{\infty}_{\cS,s}$. We may represent it by $f' \in \C^{\infty}_{\cS}(S)$. There is thus some $f \in \C^{\infty}_{\M}(Y)$, such that $f' = \iota^{\ast}_{Y}(f)$. Define $v'([f']_{s}) := v([f]_{s})$. 

Suppose $[f']_{s} = [g']_{s}$ for some other $g' \in \C^{\infty}_{\cS}(S)$. There is thus some $W \in \Op_{s}(N)$, such that $f'|_{W} = g'|_{W}$. As $\iota$ is an embedding, we may find $V \in \Op_{s}(Y)$, such that $\ul{\iota}(W) = V \cap \ul{\iota}(S)$. Let $g \in \C^{\infty}_{\M}(Y)$ satisfy $g' = \iota^{\ast}_{Y}(g)$. But then $\iota^{\ast}_{V}(f|_{V} - g|_{V}) = f'|_{W} - g'|_{W} = 0$ and we conclude that $[f]_{s} - [g]_{s} \in (\J_{\cS})_{s}$. This proves that $v'$ is well-defined and obviously $v' \in \gDer( \C^{\infty}_{\cS,s}, \R) \equiv T_{s} \cS$. By construction, we have $v = (T_{s} \iota)(v')$ and thus $v \in T_{s} \cS \subseteq T_{s} \M$. 
\end{proof}

We arrive to the main statement of this subsection.
\begin{theorem} \label{thm_inverseimage}
Let $(\cS,\iota)$ be a closed embedded submanifold of $\M$. Suppose $\phi: \cN \rightarrow \M$ is a graded smooth map transversal to $(\cS,\iota)$. 

Then there exists a closed embedded submanifold $(\phi^{-1}(\cS), \iota')$ over the subset $\ul{\phi}^{-1}(S)$, and a graded smooth map $\phi': \phi^{-1}(\cS) \rightarrow \cS$ fitting into the commutative diagram
\begin{equation} \label{eq_inverseimagediagrampullback}
\begin{tikzcd}
\phi^{-1}(\cS) \arrow{r}{\phi'} \arrow{d}{\iota'} & \cS \arrow{d}{\iota} \\
\cN \arrow{r}{\phi} & \M
\end{tikzcd}.
\end{equation}
Moreover, this diagram is a categorical pullback in $\gMan^{\infty}$: To any graded manifold $\Q$ and a pair of graded smooth maps $\mu: \Q \rightarrow \cN$ and $\chi: \Q \rightarrow \cS$ satisfying $\phi \circ \mu = \iota \circ \chi$, there exists a unique graded smooth map $\varphi: \Q \rightarrow \phi^{-1}(\cS)$, such that $\iota' \circ \varphi = \mu$ and $\phi' \circ \varphi = \chi$. In particular, this property determines $(\phi^{-1}(\cS),\iota')$ up to the equivalence, see Lemma \ref{lem_equivalentsubmafolds}. 

Finally, for each $n \in \ul{\phi}^{-1}(S)$, one has 
\begin{equation} \label{eq_inversetangent}
T_{n}(\phi^{-1}(\cS)) = (T_{n} \phi)^{-1}( T_{\ul{\phi}(n)} \cS). 
\end{equation}
\end{theorem}
\begin{proof}
Let $\J_{\cS}$ be the sheaf of regular ideals corresponding to $(\cS,\iota)$. The idea is to produce a sheaf of regular ideals $\I' \subseteq \C^{\infty}_{\cN}$ and use Theorem \ref{thm_idealissubmanifold} to construct $(\phi^{-1}(\cS), \iota')$. Let us divide the proof into several parts.
\begin{enumerate}[a)]
\item \textbf{Constructing the sheaf of ideals $\I'$}: The sheaf of ideals $\J_{\cS} \subseteq \C^{\infty}_{\M}$ can be viewed as a sheaf of graded $\C^{\infty}_{\M}$-submodules of $\C^{\infty}_{\M}$. Let $j: \J_{\cS} \rightarrow \C^{\infty}_{\M}$ be the inclusion viewed as a $\C^{\infty}_{\M}$-linear sheaf morphism. Having a graded smooth map $\phi: \cN \rightarrow \M$, we have the corresponding pullback functor $\phi^{\ast}: \Sh^{\C^{\infty}_{\M}}(M,\gVect) \rightarrow \Sh^{\C^{\infty}_{\cN}}(N,\gVect)$, see the part $(a)$ of the proof of Proposition \ref{tvrz_ap_pullbackVB}. We thus get a $\C^{\infty}_{\cN}$-linear sheaf morphism 
\begin{equation}
\phi^{\ast}(j): \phi^{\ast}\J_{\cS} \rightarrow \phi^{\ast} \C^{\infty}_{\M} \cong \C^{\infty}_{\cN},
\end{equation}
where we have used the observation from Example \ref{ex_pullbacktrivial}. Now, note that the presheaf of graded $\C^{\infty}_{\cN}$-submodules defined for each $V \in \Op(N)$ as 
\begin{equation}
\im( \phi^{\ast}(j))(V) := \im( (\phi^{\ast}(j))_{V}) \subseteq \C^{\infty}_{\cN}(V).
\end{equation}
actually forms a sheaf. This is because there a canonical $\C^{\infty}_{\cN}$-linear presheaf isomorphism 
\begin{equation}
\im( \phi^{\ast}(j)) \cong \phi^{\ast} \J_{\cS} / \ker( \phi^{\ast}(j)),
\end{equation}
and every quotient of a sheaf of graded $\C^{\infty}_{\cN}$-modules by a sheaf of graded $\C^{\infty}_{\cN}$-submodules is always a sheaf. This is proved using the same arguments as Proposition \ref{eq_quotientpresheafissheaf}. In other words, we have just constructed a sheaf of ideals
\begin{equation}
\I' := \im( \phi^{\ast}(j)) \subseteq \C^{\infty}_{\cN}. 
\end{equation}
Recall that $\phi^{\ast}(j): \phi^{\ast}\J_{\cS} \rightarrow \C^{\infty}_{\cN}$ is obtained by the universal property of sheafification from the $\C^{\infty}_{\cN}$-linear presheaf morphism $\phi^{\ast}_{P}(j): \phi^{\ast}_{P} \J_{\cS} \rightarrow \C^{\infty}_{\cN}$. Explicitly, one has 
\begin{equation} \label{eq_phiPjmap}
(\phi^{\ast}_{P}(j))_{V}(g \otimes [f]_{V}) = g \cdot \phi^{\ast}_{U}(f)|_{V},
\end{equation}
for any $g \in \C^{\infty}_{\cN}(V)$ and $[f]_{V} \in (\ul{\phi}^{-1}_{P} \J_{\cS})(V)$ represented by $f \in \J_{\cS}(U)$ for some $U \supseteq \ul{\phi}(V)$. 

Now, for each $n \in N$, we are interested in the stalk $\I'_{n} \subseteq \C^{\infty}_{\cN,n}$. Recall that we have a graded algebra morphism $\phi_{(n)}: \C^{\infty}_{\M,\ul{\phi}(n)} \rightarrow \C^{\infty}_{\cN,n}$ defined as in Definition \ref{def_gLRS}-(iii). We claim that 
\begin{equation} \label{eq_stalkinverseideal}
\I'_{n} = \< \phi_{(n)}( (\J_{\cS})_{\ul{\phi}(n)}) \>.
\end{equation}
In other words, $\I'_{n}$ is the ideal generated by the image of the ideal $(\J_{\cS})_{\ul{\phi}(n)}$ under $\phi_{(n)}$. We have moved a detailed argument into the appendix, see Theorem \ref{thm_ap_inverseimage}. 

\item \textbf{$\I'$ is a sheaf of regular ideals}: First let us us show that 
\begin{equation} \label{eq_inverseimagesubmafoldunderlying}
\ul{\phi}^{-1}(S) = \{ n \in N \; | \; \I'_{n} \subseteq \frJ(\C^{\infty}_{\cN,n}) \}.
\end{equation}
To see this, we use the expression (\ref{eq_stalkinverseideal}). For any $n \in \ul{\phi}^{-1}(S)$, we have $(\J_{\cS})_{\ul{\phi}(n)} \subseteq \frJ(\C^{\infty}_{\M,\ul{\phi}(n)})$ by Proposition \ref{tvrz_JSprops}-$(i)$ and Lemma \ref{lem_regularideals}-\textit{4)}. Since $\phi_{(n)}$ is a morphism of local graded rings, we have $\phi_{(n)}( (\J_{\cS})_{\ul{\phi}(n)}) \subseteq \frJ( \C^{\infty}_{\cN,n})$ and thus also $\I'_{n} \subseteq \frJ(\C^{\infty}_{\cN,n})$. 

Conversely, suppose $\I'_{n} \subseteq \frJ(\C^{\infty}_{\cN,n})$. Suppose $\ul{\phi}(n) \notin S$. Hence $(\J_{\cS})_{\ul{\phi}(n)} = \C^{\infty}_{\M,\ul{\phi}(n)}$ and thus $[1]_{\ul{\phi}(n)} \in (\J_{\cS})_{\ul{\phi}(n)}$. But then $[1]_{n} = \phi_{(n)}( [1]_{\ul{\phi}(n)}) \in \I'_{n}$, which contradicts the assumption. 

Now, let $n \in \ul{\phi}^{-1}(S)$ be fixed. As $\J_{\cS}(M)$ is a regular ideal and $\J_{\cS}(M) \subseteq \J^{\ul{\phi}(n)}_{\M}(M)$, there exists a finite graded set $\G \subseteq \J_{\cS}(M)$ consisting of functions independent at $\ul{\phi}(n)$ and $U \in \Op_{\ul{\phi}(n)}(M)$, such that $\G_{(m)} = \{ [f]_{m} \; | \; f \in \G \}$ generates $(\J_{\cS})_{m}$ for all $m \in U$. 

Recall that for each $f \in \J_{\cS}(M)$, we have the corresponding pullback section $f^{!} \in (\phi^{\ast} \J_{\cS})(N)$, defined as $f^{!} = (\eta_{\phi^{\ast}_{P} \J_{\cS}})_{N}( 1 \otimes [f]_{N})$, where $\eta_{ \phi^{\ast}_{P} \J_{\cS}}: \phi^{\ast}_{P} \J_{\cS} \rightarrow \phi^{\ast} \J_{\cS}$ is the canonical map obtained by the sheafification procedure. Consider the finite graded set
\begin{equation}
\G' = \{ (\phi^{\ast}(j))_{N}(f^{!}) \; | \; f \in \G \} \subseteq \I'(N).
\end{equation}
Observe that explicitly, one has $(\phi^{\ast}(j))_{N}(f^{!}) = (\phi^{\ast}_{P}(j))_{N}(1 \otimes [f]_{N}) = \phi^{\ast}_{M}(f)$. Hence
\begin{equation}
[ (\phi^{\ast}(j))_{N}(f^{!}) ]_{n'} = [ \phi^{\ast}_{M}(f)]_{n'} = \phi_{(n')}( [f]_{\ul{\phi}(n')}). 
\end{equation}
But this proves that $\G'_{(n')}$ generates $\I'_{n'}$ for all $n' \in \ul{\phi}^{-1}(U)$. 

Note that this also proves that for each $V \in \Op( \ul{\phi}^{-1}(U))$, the ideal $\I'(V)$ is generated by the finite graded set $\G'|_{V} = \{ \phi^{\ast}_{M}(f)|_{V} \; | \; f \in \G \}$, see the proof of Lemma \ref{lem_regularideals}-\textit{2)}. 

We only have to argue that $\G'$ consists of functions independent at $n$. We have 
\begin{equation} \label{eq_differentialsgensets}
(\dr (\phi^{\ast}_{M}(f)))_{n} = (T_{n} \phi)^{T}( (\dr{f})_{\ul{\phi}(n)}). 
\end{equation}
It follows from Proposition \ref{tvrz_tangenttosubmafold} that for each $f \in \G$ and $s \in S$, one has $(\dr{f})_{s} \in \an( T_{s} \cS)$. It is now easy to see that the assumed transversality $\phi \pitchfork \cS$ ensures that the restriction of $(T_{n}\phi)^{T}$ to $\an(T_{\ul{\phi}(n)} \cS)$ is injective. The conclusion then follows from (\ref{eq_differentialsgensets}). 

This proves that $\I'$ is a sheaf of regular ideals. 
\item \textbf{Pullback diagram and universality}:
We can thus define the the submanifold $(\phi^{-1}(\I), \iota')$ using the sheaf of regular ideals $\I'$ and Theorem \ref{thm_idealissubmanifold}. It follows from (\ref{eq_inverseimagesubmafoldunderlying}) that its underlying submanifold is indeed $\ul{\phi}^{-1}(S)$. 

To find a graded smooth map $\phi': \phi^{-1}(\cS) \rightarrow \cS$, we can employ Proposition \ref{tvrz_clweakembedd}. It suffices to prove that $(\phi \circ \iota')^{\ast}_{M}( \J_{\cS}(M)) = 0$. As $\I' = \ker( \iota'^{\ast})$, we only have to show that $\phi^{\ast}_{M}( \J_{\cS}(M)) \subseteq \I'(N)$. But by construction, we have
\begin{equation}
\I'(N) = \im( (\phi^{\ast}(j))_{N}) \supseteq \im( (\phi^{\ast}_{P}(j))_{N}) = \< \phi^{\ast}_{M}( \J_{\cS}(M)) \>, 
\end{equation}
where the last equation can be seen easily from (\ref{eq_phiPjmap}). By Proposition \ref{tvrz_clweakembedd}, there is thus the unique graded smooth map $\phi': \phi^{-1}(\cS) \rightarrow \cS$, such that $\iota \circ \phi' = \phi \circ \iota'$. 

Let us argue that (\ref{eq_inverseimagediagrampullback}) is in fact a pullback square in the category $\gMan^{\infty}$. To any graded manifold $\Q$ and a pair of graded smooth maps $\mu: \Q \rightarrow \cN$ and $\chi: \Q \rightarrow \cS$ satisfying $\phi \circ \mu = \iota \circ \chi$, we must find a unique graded smooth map $\varphi: \Q \rightarrow \phi^{-1}(\cS)$ fitting into the diagram
\begin{equation} \label{eq_inverseimagediagrampullback2}
\begin{tikzcd}
\Q \arrow[bend left=20]{rrd}{\chi} \arrow[bend right=20]{ddr}{\mu} \arrow[dashed]{rd}{\varphi} & & \\
& \phi^{-1}(\cS) \arrow{r}{\phi'} \arrow{d}{\iota'} & \cS \arrow{d}{\iota} \\
& \cN \arrow{r}{\phi} & \M
\end{tikzcd},
\end{equation}
and making both the triangles commute, that is $\phi' \circ \varphi = \chi$ and $\iota' \circ \varphi = \mu$. As $(\phi^{-1}(\cS'), \iota')$ is a closed embedded submanifold, we will construct $\varphi$ using the Proposition \ref{tvrz_clweakembedd}. This also ensures its uniqueness. We thus only have to argue that $\mu^{\ast}(\I') = 0$. 

Let $\hat{\mu}^{\ast} := \mu^{\ast}|_{\I'}$. We thus have to show that the sheaf morphism $\hat{\mu}^{\ast}: \I' \rightarrow \ul{\mu}_{\ast} \Q$ is trivial. For each $n \in N$, it suffices to find $V \in \Op_{n}(N)$, such that $\hat{\mu}^{\ast}|_{V}: \I'|_{V} \rightarrow (\ul{\mu}_{\ast} \Q)|_{V}$ is trivial. 
First observe that $\ul{\phi} \circ \ul{\mu} = \ul{\iota} \circ \ul{\chi}$, hence $\ul{\mu}(Q) \subseteq \ul{\phi}^{-1}(S)$. We now have two distinct cases:
\begin{enumerate}[(i)]
\item $n \in N - \ul{\phi}^{-1}(S)$: We may find a subset $V \in \Op_{n}(N - \ul{\phi}^{-1}(S))$. Then $\ul{\mu}^{-1}(V) = \emptyset$ and thus obviously $\hat{\mu}^{\ast}|_{V} = 0$.
\item $n \in \ul{\phi}^{-1}(S)$: We have shown in the part $(b)$ of this proof that there is a finite graded subset $\G \subseteq \J_{\cS}(M)$ and $U \in \Op_{\ul{\phi}(n)}(M)$, such that for any $V \in \Op(\ul{\phi}^{-1}(U))$, the ideal $\I'(V)$ is generated by $\G'|_{V} = \{ \phi^{\ast}_{M}(f)|_{V} \; | \; f \in \G \}$. But for any $f \in \G$, one has
\begin{equation}
\hat{\mu}^{\ast}_{V}( \phi^{\ast}_{M}(f)|_{V}) = \mu^{\ast}_{N}( \phi^{\ast}_{M}(f))|_{V} = \chi^{\ast}_{S}( \iota^{\ast}_{M}(f))|_{V} = 0,
\end{equation}
since $f \in \J_{S}(M) = \ker(\iota^{\ast}_{M})$. But this proves that $\hat{\mu}^{\ast}_{V} = 0$ for any $V \subseteq \ul{\phi}^{-1}(U)$ and we conclude that $\hat{\mu}^{\ast}|_{\ul{\phi}^{-1}(U)} = 0$.
\end{enumerate}
The fact that $\hat{\mu}^{\ast} = 0$ now follows from Proposition \ref{tvrz_gluingmorphisms}-(i) and have just proved that $\mu^{\ast}(\I') = 0$. By Proposition \ref{tvrz_clweakembedd}, there is thus a unique graded smooth map $\varphi: \Q \rightarrow \phi^{-1}(\cS)$, such that $\iota' \circ \varphi = \mu$. We only have to argue that $\phi' \circ \varphi = \chi$. But that is trivial, since we find
\begin{equation}
\iota \circ (\phi' \circ \varphi) = \phi \circ (\iota' \circ \varphi) = \phi \circ \mu = \iota \circ \chi.
\end{equation}
Since $(\cS,\iota)$ is a closed embedded submanifold, the equality $\phi' \circ \varphi = \chi$ follows from the uniqueness assertion of Proposition \ref{tvrz_clweakembedd}. This proves the universal property for the pullback square (\ref{eq_inverseimagediagrampullback}). 
\item \textbf{Tangent space}: To finish the proof, it remains to verify the formula (\ref{eq_inversetangent}). We will employ Proposition \ref{tvrz_tangenttosubmafold}. Also recall the equation (\ref{eq_stalkinverseideal}). Let $n \in \ul{\phi}^{-1}(S)$. 

We have $v \in T_{n}( \phi^{-1}(\cS))$, if and only if $v([f]_{n}) = 0$ for all $[f]_{n} \in \I'_{n}$. But since $\I'_{n}$ is generated by $\phi_{(n)}((\J_{\cS})_{\ul{\phi}(n)})$, and $v$ is a graded derivation, this is equivalent to $v( \phi_{(n)}([g]_{\ul{\phi}(n)})) = 0$ for all $[g]_{\ul{\phi}(n)} \in (\J_{\cS})_{\ul{\phi}(n)}$. It follows from the definition of $T_{n}\phi$, see (\ref{eq_differentialdef}), that this is equivalent to $((T_{n}\phi)(v))([g]_{\ul{\phi}(n)}) = 0$ for all $[g]_{\ul{\phi}(n)} \in (\J_{\cS})_{\ul{\phi}(n)}$. But by Proposition \ref{tvrz_tangenttosubmafold}, this is equivalent to $(T_{n}\phi)(v) \in T_{\ul{\phi}(n)} \cS$. By following this chain of equivalent statements, we arrive to (\ref{eq_inversetangent}). 
\end{enumerate}
This finishes the proof. 
\end{proof}
In fact, the assumption that $(\cS,\iota)$ is a \textit{closed} embedded submanifold is unnecessary. However, the arguments are less straightforward as we cannot employ Theorem \ref{thm_idealissubmanifold}. We have thus decided to move the complete proof into the appendix.
\begin{theorem} \label{thm_inverseimagegen}
All claims of Theorem \ref{thm_inverseimage} remain true for a general embedded submanifold $(\cS,\iota)$, except that $(\phi^{-1}(\cS), \iota')$ is not necessarily closed. 
\end{theorem}
\begin{proof}
See Theorem \ref{thm_ap_inverseimagegen} in the appendix. 
\end{proof}
\begin{rem}
If $(\cS,\iota)$ is a \textit{closed} embedded submanifold, the universality of the pullback diagram (\ref{eq_inverseimagediagrampullback}) ensures that the submanifolds  constructed in Theorem \ref{thm_inverseimage} and in Theorem \ref{thm_inverseimagegen} are equivalent. Also note that if $(n_{j})_{j \in \Z} := \gdim(\M)$, $(m_{j})_{j \in \Z} := \gdim(\cN)$ and $(s_{j})_{j \in \Z} := \gdim(\cS)$, then $\gdim( \phi^{-1}(\cS)) = (s_{j} + m_{j} - n_{j})_{j \in \Z}$. This is observed in the part c) of the proof of Theorem \ref{thm_ap_inverseimagegen}.   
\end{rem}
Having these statements in hand, we arrive to the following general definition. 

\begin{definice} \label{def_inverseimage}
Let $(\cS,\iota)$ be an embedded submanifold of $\M$. Suppose $\phi: \cN \rightarrow \M$ is a graded smooth map. We say that an embedded submanifold $(\phi^{-1}(\cS),\iota')$ of $\cN$ is an \textbf{inverse image submanifold of $\cS$ by $\phi$}, if it has the following properties:
\begin{enumerate}[(i)]
\item Its underlying submanifold is the set $\ul{\phi}^{-1}(S)$.
\item There exists a graded smooth map $\phi': \phi^{-1}(\cS) \rightarrow \cS'$ making (\ref{eq_inverseimagediagrampullback}) into a pullback diagram in $\gMan^{\infty}$: One has $\iota \circ \phi' = \phi \circ \iota'$ and it has the universal property described under (\ref{eq_inverseimagediagrampullback}). 
\end{enumerate}
Note that if an inverse image submanifold exists, it is unique up to an equivalence. 
\end{definice}
\begin{rem}
Theorem \ref{thm_inverseimagegen} shows that if $(\cS,\iota)$ is an embedded submanifold and $\phi \pitchfork \cS$, the inverse image submanifold $(\phi^{-1}(\cS),\iota')$ exists. However, it is well known fact that the transversality condition is not necessary. Also note that in general, only the inclusion $\subseteq$ in  (\ref{eq_inversetangent}) is true. 
\end{rem}

Recall that for each $m \in M$, we have a closed embedded submanifold $(\{m\}, \kappa_{m})$, see Example \ref{ex_singlepointsubmafold}. The corresponding sheaf of regular ideals is the sheaf $\J^{m}_{\M}$ of functions vanishing at $m$. 

\begin{tvrz} \label{tvrz_levelset}
Let $\phi: \cN \rightarrow \M$ be a graded smooth map. Let $m \in M$. Then $m$ is a regular value of $\phi$, if and only if $\phi$ is transversal to the submanifold $(\{m\}, \kappa_{m})$. 

For every regular value $m$ of $\phi$, there is thus a closed embedded submanifold $(\phi^{-1}(m), \iota')$ of $\cN$ over the submanifold $\ul{\phi}^{-1}(m) \subseteq N$ called the \textbf{regular level set submanifold}. Moreover, for every $n \in \ul{\phi}^{-1}(m)$, one has $T_{n}( \phi^{-1}(m)) = \ker( T_{n} \phi)$. In particular, if $\gdim(\cN) = (m_{j})_{j \in \Z}$ and $\gdim(\M) = (n_{j})_{j \in \Z}$, then $\gdim(\phi^{-1}(m)) = (m_{j} - n_{j})_{j \in \Z}$. 
\end{tvrz}
\begin{proof}
Recall that $m \in M$ is a regular value of $\phi$, if $\phi$ is a submersion at every $n \in \ul{\phi}^{-1}(m)$. The first claim is obvious, since $T_{m}( \{ m \})$ is a trivial subspace of $T_{m}\M$. The regular level set submanifold $\phi^{-1}(m)$ is obtained by Theorem \ref{thm_inverseimage}. The claim about tangent spaces follows immediately from (\ref{eq_inversetangent}) using the fact that $T_{\ul{\phi}(n)} \{ m\} = 0$ for all $n \in \ul{\phi}^{-1}(\{m\})$. 
\end{proof}
\subsection{Fiber products and intersections}
Having a good notion of an inverse image of an embedded submanifold is vital for several standard geometric constructions. Let us start with a very important example of a closed embedded submanifold.
\begin{example}
Let $\varphi: \M \rightarrow \cN$ be a graded smooth map. Define a graded smooth map $\iota_{\varphi}: \M \rightarrow \M \times \cN$ by declaring $\pi_{\M} \circ \iota_{\varphi} := \1_{\M}$ and $\pi_{\cN} \circ \iota_{\varphi} := \varphi$. 

We claim that $(\M,\iota_{\varphi})$ is a closed embedded submanifold of $\M \times \cN$. It is usually denoted simply as $\gr(\varphi) \subseteq \M \times \cN$ and called the \textbf{graph of $\varphi$}.

First, its underlying map is $\ul{\iota_{\varphi}}(m) = (m, \ul{\varphi}(m))$. It is a well-known fact that this makes $(M, \ul{\iota_{\varphi}})$ into a closed embedded submanifold of $M \times N$, see e.g. Proposition 5.4 in \cite{lee2012introduction}. It is usually identified with its image $\ul{\iota_{\varphi}}(M) = \gr(\ul{\varphi})$. 

It thus remains to show that $\iota_{\varphi}$ is an immersion. Hence suppose that $(T_{m} \iota_{\varphi})(v) = 0$ for some $v \in T_{m}\M$. But using the fact that $\pi_{\M} \circ \iota_{\varphi} = \1_{\M}$, we find 
\begin{equation} v = (T_{m}(\pi_{\M} \circ \iota_{\varphi}))(v) = (T_{(m,\ul{\varphi}(m))} \pi_{\M})(0) = 0. \end{equation}

For the special case $\varphi = \1_{\M}$, we use the notation $\Delta_{\M} := \gr(\1_{\M}) \subseteq \M \times \M$, $\Delta := \iota_{\1_{\M}}$, and call $\Delta$ the \textbf{diagonal embedding of $\M$ into $\M \times \M$}.
\end{example}

We will later make use of the following simple observation:
\begin{lemma} \label{lem_mapfactorizesthroughgraph}
Let $\varphi: \M \rightarrow \cN$ and $\chi: \Q \rightarrow \M \times \cN$ be arbitrary graded smooth maps. Suppose that they satisfy the condition
\begin{equation} \label{eq_mapfactorizesthroughgraph}
\varphi \circ \pi_{\M} \circ \chi = \pi_{\cN} \circ \chi.
\end{equation}
Then there exists a unique graded smooth map $\chi': \Q \rightarrow \M$, such that $\iota_{\varphi} \circ \chi' = \chi$. In fact, the condition (\ref{eq_mapfactorizesthroughgraph}) is necessary. 
\end{lemma}
\begin{proof}
Suppose there exists some $\chi'$ satisfying the equation $\iota_{\varphi} \circ \chi' = \chi$. Due to the universality of the product, it is equivalent to the pair of equations obtained by composing it with the projections $\pi_{\M}$ and $\pi_{\cN}$. But this gives a pair of conditions
\begin{equation}
\chi' = \pi_{\M} \circ \chi, \; \; \varphi \circ \chi' = \pi_{\cN} \circ \chi. 
\end{equation}
The first equation shows that we have to define $\chi' := \pi_{\M} \circ \chi$, hence it is unique and smooth. By plugging this into the second equation, we obtain precisely the assumption (\ref{eq_mapfactorizesthroughgraph}). This also shows that it is necessary. 
\end{proof}
The diagonal embedding $\Delta$ is useful to encode the transversality of graded smooth maps.
\begin{tvrz} \label{tvrz_fiberedproduct}
Let $\phi: \cN \rightarrow \M$ and $\psi: \cS \rightarrow \M$ be graded smooth maps. 

Then $\phi \pitchfork \psi$, if and only if $(\phi \times \psi) \pitchfork \Delta_{\M}$. 

Consequently, for $\phi \pitchfork \psi$, there exists a closed embedded submanifold $(\cN \times_{\M} \cS, \iota)$ of $\cN \times \cS$ and a graded smooth map $\phi \cap \psi: \cN \times_{\M} \cS \rightarrow \M$ fitting into the commutative diagram
\begin{equation} \label{eq_fiberproductfirstpullback}
\begin{tikzcd}
\cN \times_{\M} \cS \arrow{r}{\phi \cap \psi} \arrow{d}{\iota} & \M \arrow{d}{\Delta} \\
\cN \times \cS \arrow{r}{\phi \times \psi} & \M \times \M
\end{tikzcd}
\end{equation} 
The graded manifold $\cN \times_{\M} \cS$ is called the \textbf{fiber product of $\cN$ and $\cS$ over $\M$}. Its underlying manifold is the ordinary fiber product 
\begin{equation} N \times_{M} S = \{ (n,s) \in N \times S \; | \; \ul{\phi}(n) = \ul{\psi}(s) \},
\end{equation} 
and $\ul{\iota}$ is the inclusion of this subset into $N \times S$. Finally, for each $(n,s) \in N \times_{M} S$, one has 
\begin{equation} \label{eq_fiberproducttangent}
T_{(n,s)}(\cN \times_{\M} \cS) = \{ (v,w) \in T_{n}\cN \oplus T_{s}\cS \; | \; (T_{n}\phi)(v) = (T_{s}\psi)(w) \}. 
\end{equation}
\end{tvrz}
\begin{proof}
For each $(n,s) \in N \times S$, we have the canonical identification $T_{(n,s)}(\cN \times \cS) \cong T_{n}\cN \oplus T_{s}\cS$. With respect to this identification, one has $T_{(n,s)}(\phi \times \psi) \cong T_{n}\phi \times T_{s}\psi$. 

On the other hand, for each $m \in M$, one has $T_{(m,m)} \Delta_{\M} = \Delta(T_{m}\M)$, where
\begin{equation}
\Delta(T_{m}\M) = \{ (v,v) \; | \; v \in T_{m}\M \}.
\end{equation}
We thus have to show that to any $m \in \ul{\phi}(N) \cap \ul{\psi}(S)$ and any $(n,s) \in \ul{\phi}^{-1}(m) \times \ul{\psi}^{-1}(m)$, one has 
\begin{equation} \label{eq_transequivalence1}
T_{m}\M = (T_{n}\phi)(T_{n}\cN) + (T_{s} \psi)(T_{s}\cS), 
\end{equation}
if and only if there holds the condition
\begin{equation} \label{eq_transequivalence2}
T_{(m,m)}(\M \times \M) = (T_{n}\phi \times T_{s}\psi)(T_{n}\cN \oplus T_{s}\cS) + \Delta(T_{m}\M). 
\end{equation}
But this is an easy exercise a we leave it to the reader. Thus $\phi \pitchfork \psi$, if and only if $\phi \times \psi \pitchfork \Delta_{\M}$. Hence if $\phi$ and $\psi$ are transversal, we can employ Theorem \ref{thm_inverseimage} and define 
\begin{equation}
\cN \times_{\M} \cS := (\phi \times \psi)^{-1}(\Delta_{\M}).
\end{equation}
The rest of the claims follows immediately from the assertions of Theorem \ref{thm_inverseimage}. 
\end{proof}
It is useful to obtain some characterization of the fiber product independent of the transversality. The answer is given by the following proposition.
\begin{tvrz}
Let $\phi: \cN \rightarrow \M$ and $\psi: \cS \rightarrow \M$ be two graded smooth maps, such that $\phi \pitchfork \psi$. Let $(\cN \times_{\M} \cS, \iota)$ be the corresponding fiber product. Consider the graded smooth maps $\pi'_{\M} := \pi_{\M} \circ \iota$ and $\pi'_{\cN} := \pi_{\cN} \circ \iota$. Then the diagram 
\begin{equation} \label{eq_fiberproductpullback}
\begin{tikzcd}
\cN \times_{\M} \cS \arrow{r}{\pi'_{\cS}} \arrow{d}{\pi'_{\cN}} & \cS \arrow{d}{\psi} \\
\cN \arrow{r}{\phi}& \M 
\end{tikzcd}
\end{equation}
is a pullback square in the category $\gMan^{\infty}$.
\end{tvrz}
\begin{proof}
Recall that by construction, we have $(\phi \times \psi) \circ \iota = \Delta \circ (\phi \cap \psi)$. But then
\begin{equation}
\phi \circ \pi'_{\cN} = \pi_{\M} \circ ((\phi \times \psi) \circ \iota) = \pi_{\M} \circ (\Delta \circ (\phi \cap \psi)) = \1_{\M} \circ (\phi \cap \psi) = \phi \cap \psi.
\end{equation}
Similarly, one shows that $\psi \circ \pi'_{\cS} = \phi \cap \psi$. Hence (\ref{eq_fiberproductpullback}) commutes. It remains to prove the universal property. We have to show that for any graded manifold $\Q$ and any pair of graded smooth maps $\mu: \Q \rightarrow \cN$ and $\chi: \Q \rightarrow \cS$ satisfying $\phi \circ \mu = \psi \circ \chi$, there exists a unique graded smooth map $\varphi: \Q \rightarrow \cN \times_{\M} \cS$, making the both triangles in the diagram
\begin{equation} \label{eq_fiberproductuniversal}
\begin{tikzcd} 
\Q \arrow[bend right]{ddr}{\mu} \arrow[bend left=20]{rrd}{\chi} \arrow[dashed]{rd}{\varphi} & & \\
& \cN \times_{\M} \cS \arrow{d}{\pi'_{\cN}} \arrow{r}{\pi'_{\cS}} & \cS \arrow{d}{\psi}\\
& \cN \arrow{r}{\phi} & \M
\end{tikzcd}
\end{equation}
commute. The idea is of course to utilize the universal property of the pullback square (\ref{eq_fiberproductfirstpullback}) guaranteed by Theorem \ref{thm_inverseimage}. Let $\mu' := (\mu,\chi): \Q \rightarrow \cN \times \cS$. Consider the graded smooth map $(\phi \times \psi) \circ \mu': \Q \rightarrow \M \times \M$. Then it satisfies 
\begin{equation}
\1_{\M} \circ \pi^{(1)}_{\M} \circ \{ (\phi \times \psi) \circ \mu' \} = \phi \circ \mu = \psi \circ \chi = \pi^{(2)}_{\M} \circ \{ (\phi \times \psi) \circ \mu' \}, 
\end{equation}
where $\pi^{(1)}_{\M}: \M \times \M \rightarrow \M$ and $\pi^{(2)}_{\M}: \M \times \M \rightarrow \M$ are the projections onto the first and second factor, respectively. Since $\Delta_{\M} = \gr(\1_{\M})$, it follows from Lemma \ref{lem_mapfactorizesthroughgraph} that the graded smooth map $\chi': \Q \rightarrow \M$ defined by $\chi' := \pi^{(1)}_{\M} \circ \{ (\phi \times \psi) \circ \mu' \} = \phi \circ \mu$ satisfies 
\begin{equation}
(\phi \times \psi) \circ \mu' = \Delta \circ \chi'. 
\end{equation}
By the universal property of (\ref{eq_fiberproductfirstpullback}), there thus exists a unique graded smooth map $\varphi: \Q \rightarrow \cN \times_{\M} \cS$ making the both triangles of the diagram 
\begin{equation}
\begin{tikzcd}
\Q \arrow[bend right]{ddr}{\mu'} \arrow[bend left=20]{rrd}{\chi'} \arrow[dashed]{rd}{\varphi} & & \\
& \cN \times_{\M} \cS \arrow{d}{\iota} \arrow{r}{\phi \cap \psi} & \M \arrow{d}{\Delta}\\
& \cN \times \cS \arrow{r}{\phi \times \psi} & \M \times \M 
\end{tikzcd}
\end{equation}
commute. It follows that $\pi'_{\cN} \circ \varphi = \pi_{\cN} \circ (\iota \circ \varphi) = \pi_{\cN} \circ \mu' = \mu$, and similarly $\pi'_{\cS} \circ \varphi = \chi$. Hence $\varphi$ makes the both triangles in (\ref{eq_fiberproductuniversal}) commute. 

Finally, every graded smooth map $\varphi: \Q \rightarrow \cN \times_{\M} \cS$ making the both triangles in (\ref{eq_fiberproductuniversal}) commute satisfies $\iota \circ \varphi = \mu'$. But as $(\cN \times_{\M} \cS, \iota)$ is an embedded submanifold, $\varphi$ is unique by Proposition \ref{tvrz_weakembedd}. This finishes the proof. 
\end{proof}

We arrive to the following general definition.

\begin{definice} \label{def_fiberproduct}
Let $\phi: \cN \rightarrow \M$ and $\psi: \cS \rightarrow \M$ be graded smooth maps. We say that an embedded submanifold $(\cN \times_{\M} \cS, \iota)$ of $\M \times \cN$ is a \textbf{fiber product of $\cN$ and $\cS$ over $\M$}, if 
\begin{enumerate}[(i)]
\item its underlying manifold is a subset $N \times_{M} S \subseteq N \times S$; 
\item it makes the diagram (\ref{eq_fiberproductpullback}) into a pullback square in $\gMan^{\infty}$. 
\end{enumerate}
\end{definice}
\begin{rem} \label{rem_productconversestatement}
One can show that if a fiber product $(\cN \times_{\M} \cS,\iota)$ exists, it is equivalent to an inverse image submanifold $(\phi \times \psi)^{-1}(\Delta_{\M})$ in the sense of Definition \ref{def_inverseimage}. In particular, the condition $\phi \pitchfork \psi$ is sufficient but not necessary for its existence. Also note that in general, only the inclusion $\subseteq$ in the equation (\ref{eq_fiberproducttangent}) is true.
\end{rem}
The notion of fiber product also allows one to unambiguously define intersections of submanifolds, as shows the following proposition.
\begin{tvrz}
Let $(\cS,\iota)$ and $(\cS',\iota')$ be two immersed submanifolds of $\M$, such that $\iota \pitchfork \iota'$. Then $(\cS \times_{\M} \cS', \iota \cap \iota')$ defines an immersed submanifold of $\M$. One has 
\begin{equation} \label{eq_intersectionset} 
\ul{\iota \cap \iota'}(S \times_{M} S') = \ul{\iota}(S) \cap \ul{\iota}'(S'). 
\end{equation}
This submanifold is thus usually denoted simply as $\cS \cap \cS'$ and called the \textbf{intersection of $\cS$ and $\cS'$}. With respect to the usual identifications, for each $m \in \ul{\iota}(S) \cap \ul{\iota}'(S')$, one has 
\begin{equation} \label{eq_intersectiontangent}
T_{m}( \cS \cap \cS') = T_{m}\cS \cap T_{m}\cS'. 
\end{equation}
If both $(\cS,\iota)$ and $(\cS',\iota')$ are (closed) embedded submanifolds, then $(\cS \times_{\M} \cS', \iota \cap \iota')$ is a (closed) embedded submanifold. 
\end{tvrz}
\begin{proof}
By Proposition \ref{tvrz_fiberedproduct}, we obtain a commutative diagram 
\begin{equation}
\begin{tikzcd} \label{eq_intersectionpullback}
\cS \times_{\M} \cS' \arrow{r}{\iota \cap \iota'} \arrow{d}{\iota_{0}} & \M \arrow{d}{\Delta} \\
\cS \times \cS' \arrow{r}{\iota \times \iota'} & \M \times \M
\end{tikzcd},
\end{equation}
For any $(s,s') \in S \times_{M} S'$ and $m := \iota(s) = \iota(s')$, one has 
\begin{equation}
T_{m}\Delta \circ T_{(s,s')}(\iota \cap \iota') = (T_{s} \iota \times T_{s'} \iota') \circ T_{(s,s')}\iota_{0}. 
\end{equation} 
But since $\Delta$, $\iota$ and $\iota'$ are immersions, this immediately shows that so is $\iota \cap \iota'$. Similarly, one proves that $\ul{\iota \cap \iota'}$ is injective. Hence $(\cS \times_{\M} \cS', \iota \cap \iota')$ is an immersed submanifold of $\M$. (\ref{eq_intersectionset}) is easy to see and (\ref{eq_intersectiontangent}) follows from (\ref{eq_fiberproducttangent}). 

Suppose that $(\cS,\iota)$ and $(\cS',\iota')$ are embedded submanifolds. We have to prove that $\ul{\iota \cap \iota}: S \times_{M} S' \rightarrow M$ is an embedding. Let $V \in \Op(S \times_{M} S')$. Since $\ul{j} := (\ul{\iota} \times \ul{\iota}') \circ \ul{\iota_{0}}$ is an embedding, we have $\ul{j}(V) = U \cap \ul{j}(N \times_{M} N')$  for some $U \in \Op(M \times M)$. We claim that
\begin{equation} \label{eq_iotacapiotaprime}
\ul{\iota \cap \iota'}(V) = \ul{\Delta}^{-1}(U) \cap \ul{\iota \cap \iota'}(S \times_{M} S').
\end{equation}
The inclusion $\subseteq$ follows immediately from the commutativity of (\ref{eq_intersectionpullback}). Hence suppose that $m \in \ul{\Delta}^{-1}(U) \cap \ul{\iota \cap \iota'}(S \times_{M} S')$. Then $\ul{\Delta}(m) \in U$ and $m = \ul{\iota \cap \iota'}(v)$ for some $v \in S \times_{M} S'$. Whence $\ul{\Delta}(m) = \ul{j}(v)$ by the commutativity of (\ref{eq_intersectionpullback}). But this means that $\ul{\Delta}(m) \in U \cap \ul{j}(S \times_{M} S') = \ul{j}(V)$. As $\ul{j}$ is injective, we have $v \in V$ and thus $m \in \ul{\iota \cap \iota'}(V)$. This proves the inclusion $\supseteq$. Since we can find the expression (\ref{eq_iotacapiotaprime}) for any $V \in \Op(S \times_{M} S')$, this proves that $\ul{\iota \cap \iota'}$ is a homeomorphism onto its image, hence an embedding. Finally, if $\ul{\iota}(S)$ and $\ul{\iota}'(S')$ are closed subsets of $M$, then so is their intersection. 
\end{proof}
Again, we can find a more general definition.
\begin{definice}
Let $(\cS,\iota)$ and $(\cS', \iota')$ be two immersed submanifolds of $\M$. We say that $\cS$ and $\cS'$ \textbf{intersect well}, if there exists a fibered product $(\cS \times_{\M} \cS', \iota_{0})$ in the sense of Definition \ref{def_fiberproduct}. 

We say that $\cS$ and $\cS'$ \textbf{intersect cleanly}, if they intersect well and (\ref{eq_intersectiontangent}) is true. 
\end{definice}
\begin{rem}
It follows from Remark \ref{rem_productconversestatement} that if $\cS$ and $\cS'$ intersect well, we still obtain the commutative diagram (\ref{eq_intersectionpullback}), hence $(\cS \times_{\M} S', \iota \cap \iota')$ becomes an immersed submanifold of $\M$, which is (closed) embedded if $(\cS,\iota)$ and $(\cS',\iota)$ are. However, only the inclusion $\subseteq$ in (\ref{eq_intersectiontangent}) is true in general. We have proved that if $\cS \pitchfork \cS'$, then $\cS$ and $\cS'$ intersect cleanly. 
\end{rem}
\section*{Final thoughts, outlooks} \label{sec_final}
\addcontentsline{toc}{section}{\nameref{sec_final}}
Due to the overall length of the paper, there are some important topics we had to leave out. To name a few, one should investigate graded regular distributions and their integrability. In particular, one would like to have a graded version of Frobenius' theorem. Regular (involutive) distributions can be defined as (involutive) subbundles of the tangent bundle. However, one would need to have a notion of maximal flows of vector fields, which turns out to be a bit problematic in the graded case. In fact, to define flows, one needs to understand actions of graded Lie groups.

This brings us to the another white spot on the map worth of exploring, the theory of graded Lie groups. One can study their relation to graded Lie algebras, their actions on graded manifolds and the corresponding theory of graded principal bundles. A very important for applications in physics is the question of integration. This can be a tougher nut to crack since it is an intriguing topic already in supergeometry, see e.g. \cite{bernshtein1977integral, Catenacci:2018xsv}. Last but not least, there is a graded symplectic and Poisson geometry to consider. Our plan is to focus on these topic in the future sequence of shorter papers, building on the foundations constructed in this article. 

There is another important question to answer. What is the relation of graded manifolds to supermanifolds? First, recall that a (Berezein-Leites) supermanifold is a pair $\cS = (S, \C^{\infty}_{\cS})$, where $S$ is a second countable Hausdorff topological space and $\C^{\infty}_{\cS}$ is a sheaf of superalgebras, whose stalks are local commutative superrings. Moreover, it is assumed to be locally isomorphic to the sheaf assigning to each $U \in \Op(\R^{n})$ the superalgebra $\C^{\infty}_{n}(U) \otimes \Lambda(V)$, where $V$ is an ordinary finite-dimensional real vector space and $\Lambda(V)$ is the corresponding Grassmann algebra. 

There is a way how to view $\cS$ as a graded manifold. Indeed, one utilizes the famous Batchelor's theorem \cite{batchelor1979structure}. If $q := \dim(V)$, there is a vector bundle $\pi: A \rightarrow S$ of rank $q$, such that $\cS \cong \Pi A$. $\Pi A$ denotes the ``parity reversed'' vector bundle, where all fiber coordinates on $A$ are declared to be odd. The details of this construction are completely analogous to Example \ref{ex_degreeshiftedordvect}. In fact, one can view $\Pi A$ as a graded manifold $A[1]$. However, the Batchelor's isomorphism is very non-canonical and so is the correspondence $\cS \leftrightarrow A[1]$. 

On the other hand, let $\M = (M, \C^{\infty}_{\M})$ be a graded manifold. Is there a way how to view this as a supermanifold? Recall that in Example \ref{ex_evenpart}, we have constructed the even part submanifold $(\M_{0},\iota_{0})$. Recall that one has $\M_{0} = (M, \C^{\infty}_{\M} / \J^{\odd}_{\M})$. Let $\A$ be any graded smooth atlas for $\M$ and $\A_{0}$ be the induced graded smooth atlas for $\M_{0}$ By construction, it only has local coordinates of even degree. If we assume that $\M_{0}$ is non-negatively (or non-positively) graded, all transition maps are polynomial in the ``purely graded'' coordinates $\xi_{\mu}$, see Remark \ref{rem_simplerextsymmetric}. Under these circumstances, one can use the transition maps to construct an ordinary smooth manifold $S$ together with a smooth surjective submersion $\pi: S \rightarrow M$. 

The transition maps of the graded smooth atlas $\A$ and the smooth atlas on $S$ can be then used to construct a supermanifold $\cS = (S, \C^{\infty}_{\cS})$ in the way completely analogous to Proposition \ref{tvrz_gmgluing1}. If we choose a different graded smooth atlas $\A'$ to construct $\cS' = (S', \C^{\infty}_{\cS'})$, there is a superdiffeomorphism $\varphi: \cS \rightarrow \cS'$ over $\ul{\varphi}: S \rightarrow S'$ satisfying $\pi' \circ \ul{\varphi} = \pi$. 

However, we have no answer for the case where $\M_{0}$ contains even coordinates of arbitrary degrees and transition functions can be infinite formal power series. This leads us to the final question. Is there a Batchelor-like theorem for arbitrary graded manifolds? In \cite{bonavolonta2013category}, there is a version for non-negatively graded manifolds (N-manifolds). However, in a general case, the proof certainly does not carry over. It would be interesting to try to mimic Batchelor's approach using Čech cohomology to obtain some structure theorem for arbitrary graded manifolds.

There is another remark. We have defined graded domain $\C^{\infty}_{(n_{j})}$ by formula (\ref{eq_grdomain}), so that ``purely graded'' coordinates $(\xi_{\mu})_{\mu=1}^{n_{\ast}}$ always have a non-zero degree, $|\xi_{\mu}| \neq 0$ for all $\mu \in \{1,\dots,n_{\ast}\}$. In fact, this is not necessary. For any $m \in \N_{0}$ and a sequence $(n_{j})_{j \in \Z}$ satisfying $\sum_{j \in \Z} n_{j} < \infty$, one can consider a sheaf on $\R^{m}$ assigning to any $U \in \Op(\R^{m})$ a graded associative commutative algebra $\C^{\infty}_{m,(n_{j})}(U) := \bar{S}( \R^{(n_{j})}, \C^{\infty}_{m}(U))$. All of the statements in this paper remain valid for graded manifolds locally isomorphic to this sheaf. This modification allows one to work with functions which are formal power series in some degree zero coordinate (say it is denoted as $\hbar$). Some concepts (e.g. graded dimension) would have to be changed accordingly and there would be some non-canonical choices required (e.g. in Theorem \ref{tvrz_totalspace}). Nevertheless, there is no conceptual obstacle forcing us to use the graded domain $\C^{\infty}_{(n_{j})}$ and not $\C^{\infty}_{m,(n_{j})}$. 

Finally, we are aware that graded manifolds as presented here are a little bit ``topologically trivial''. This is manifested e.g. by the fact that the graded de Rham cohomology is not interesting at all, see Proposition \ref{tvrz_degreekdeRhamtrivial} and Theorem \ref{thm_degreezerogrdeRhamisdeRham}. In the $\Z_{2}$-graded setting, this was remedied by considering the model sheaf on a more complicated space 
\begin{equation}
\Lambda(\R^{k})^{m,n} := \underbrace{\Lambda(\R^{k})_{0} \oplus \dots \oplus \Lambda(\R^{k})_{0}}_{m \text{ times}} \oplus \underbrace{\Lambda(\R^{k})_{1} \oplus \dots \oplus \Lambda(\R^{k}))_{1}}_{n \text{ times}},
\end{equation}
where $\Lambda(\R^{k})$ is the Grassmann algebra of $\R^{k}$ viewed as a vector superspace with the obvious $\Z_{2}$-grading. It is highly non-trivial to choose the correct sheaf of ``supersmooth functions'' on this topological space. See \cite{bartocci2012geometry} for a detailed discussion and references. The obvious choice in the $\Z$-graded setting would be to choose a finite-dimensional graded vector space $V$, pick a sequence $(n_{j})_{j \in \N}$ with $\sum_{j \in \Z} n_{j} < \infty$, and consider the vector space 
\begin{equation}
S(V)^{(n_{j})} := \bigoplus_{j \in \Z} \underbrace{S(V)_{j} \oplus \dots \oplus S(V)_{j}}_{n_{j} \text{ times}}.
\end{equation}
If we ensure that $S(V)_{j}$ is finite-dimensional whenever $n_{j} \neq 0$, this is a finite-dimensional vector space. This is true e.g. if $V$ is non-negatively graded. However, at this moment, we have no idea whether there is some suitable sheaf of ``graded smooth functions'' on $S(V)^{(n_{j})}$. We shall address this question in the future. 

\section*{Acknowledgments}
First and foremost, I would like to express gratitude to my family for their patience and support during the somewhat complicated last year. 

Next, I would like to thank Mark Bugden, Pavel Hájek, Branislav Jurčo, Pavol Ševera, Rudolf Šmolka, Dmitry Roytenberg, Vít Tuček and Fridrich Valach for helpful discussions. 
The author is grateful for a financial support from MŠMT under grant no. RVO 14000. 

\bibliography{bib}	
\appendix
\section{Proofs of certain statements}
\subsection{Graded domains}
\begin{lemma}[\textbf{classical Hadamard lemma}] \label{lem_ap_Hadimrdclassic}
For the statement, see Lemma \ref{lem_Hadimrdclassic}. 
\end{lemma}
\begin{proof}
Let us start with the $q = 0$ case. We look for a smooth map $h_{a}^{(0)}: U \rightarrow 
Lin(\R^{n},\R)$, such that for all $x \in U$, one can write
\begin{equation} \label{eq_hadamardbasicformula}
f(x) = f(a) + [h_{a}^{(0)}(x)](x-a),
\end{equation}
Note that necessarily $h^{(0)}_{a}(a) = \fD{f}(a)$. Let $B_{r}(a) \subseteq U$ be some open ball around $a$ of radius $r$. For each $x \in B_{r}(a)$, we may define a smooth function $g(t) := f((1-t)a + tx): I \rightarrow \R$, where $I \subseteq \R$ is some open interval containing $[0,1]$. Then
\begin{equation}
g(t) = g(0) + \int_{0}^{t} \frac{d}{dt} g(t)\,dt = g(0) + \int_{0}^{t} [(\fD{f}((1-t)a + tx)](x-a)\,dt,
\end{equation}
for all $t \in I$. Evaluating this at $t = 1$ gives 
\begin{equation} \label{eq_hadamardintegralformula}
f(x) = f(a) + \int_{0}^{1} [(\fD{f}((1-t)a + tx)](x-a)\,dt.
\end{equation}
By declaring the second term to be $[h_{1}(x)](x-a)$, we obtain a smooth map $h_{1}: B_{r}(a) \rightarrow \Lin(\R^{n},\R)$, such that $f(x) = f(a) + [h_{1}(x)](x-a)$ for all $x \in B_{r}(a)$. Note that this is the reason why $U$ is often assumed to be star-shaped around $a$. 

Next, fix some $r' < r$ and consider $U' := U - \ol{B_{r'}(a)}$. Let us define $h_{2}: U' \rightarrow \Lin(\R^{n},\R)$ by 
\begin{equation}
[h_{2}(x)](y) := \fD{f}(a) + \frac{\<x-a,y\>}{\| x - a \|^{2}} \epsilon_{a}(x-a),
\end{equation}
where $\epsilon_{a}$ is defined by $\epsilon_{a}(y) := f(a+y) - f(a) - [\fD{f}(a)](y)$ for all $y \in \{ x - a \; | \; x \in U \}$. Clearly $h_{2}: U' \rightarrow \Lin(\R^{n},\R)$ is smooth and for every $x \in U'$, one finds 
\begin{equation}
f(x) = f(a) + \fD{f}(a)(x-a) + \epsilon_{a}(x-a) = f(a) + [h_{2}(x)](x-a).
\end{equation}
In fact, one can attempt to define $h_{2}$ on entire $U$ by setting $h_{2}(a) := \fD{f}(a)$. It fits into the above formula and it is even continuous at $a$. However, in general, it is not differentiable at $a$. 

Finally, observe that $B_{r}(a)$ and $U'$ form an open cover of $U$ and one has the corresponding partition of unity $\{ \rho_{1}, \rho_{2} \}$. We may thus define $h^{(0)}_{a} := \rho_{1} h_{1} + \rho_{2} h_{2}: U \rightarrow \Lin(\R^{n},\R)$. This map is smooth and fits into the formula (\ref{eq_hadamardbasicformula}). On $B_{r'}(a)$, it is given by the integral formula (\ref{eq_hadamardintegralformula}). 

One can now iterate this procedure. By applying (\ref{eq_hadamardbasicformula}) to $h^{(0)}_{a}$, there is a smooth function $h^{(1)}_{a}: U \rightarrow \Lin^{(2)}(\R^{n},\R)$ satisfying
\begin{equation}
[h_{a}^{(0)}(x)](x-a) = [\fD{f}(a)](x-a) + [h_{a}^{(1)}(x)](x-a)^{2},
\end{equation}
for all $x \in U$. To proceed further, one only has to know the value $h_{a}^{(1)}(a)$. To do so, one has to use the integral formula (\ref{eq_hadamardintegralformula}). For $x \in B_{r'}(a)$ and $y_{1},y_{2} \in \R^{n}$, one finds
\begin{equation}
[\fD{h_{a}^{(0)}}(x)](y_{1},y_{2}) = \int_{0}^{1} t \cdot [\fD^{2}{f}((1-t)a + tx)](y_{1},y_{2}) \,dt,
\end{equation}
where $h_{a}^{(0)}: U \rightarrow \Lin(\R^{n},\R)$, so that $\fD{h}_{a}^{(0)}: U \rightarrow \Lin^{(2)}(\R^{n},\R)$. By plugging this back into the integral formula (\ref{eq_hadamardintegralformula}), for every $x \in B_{r'}(a)$ one finds
\begin{equation}
[h_{a}^{(1)}(x)](y_{1},y_{2}) = \int_{0}^{1} ds \int_{0}^{1} dt \; t \cdot [\fD^{2}{f}((1-st)a + stx)](y_{1},y_{2}).
\end{equation}
In particular, for $x = a$, one obtains $h_{a}^{(1)}(a) = \frac{1}{2} \fD^{2}{f}(a)$. One can now repeat this indefinitely and produce a sequence $h^{(q)}_{a}: U \rightarrow \Lin^{(q+1)}(\R^{n},\R)$ defined iteratively by (\ref{eq_hadamardbasicformula}), so that
\begin{equation}
[h_{a}^{(q-1)}(x)](x-a)^{q} = [h_{a}^{(q-1)}(a)](x-a)^{q} + [h_{a}^{(q)}(x)](x-a)^{q+1},
\end{equation}
for all $x \in U$. For $x \in B_{r'}(a)$, it is given by the explicit integral formula
\begin{equation}
h_{a}^{(q)}(x) = \int_{0}^{1} dt_{1} \dots \int_{0}^{1} dt_{q+1} \; (t_{1})^{0} (t_{2})^{1} \dots (t_{q+1})^{q} \fD^{q+1}f( (1 - t_{1} \cdots t_{q+1})a + t_{1} \cdots t_{q+1} x).
\end{equation}
In particular, one has $h_{a}^{(q)}(a) = \frac{1}{(q+1)!} \fD^{q+1}{f}(a)$. Finally, it follows from the construction that 
\begin{equation}
f(x) = f(a) + \sum_{k=1}^{q} \frac{1}{k!} [\fD^{k}f(a)](x-a)^{k} + [h^{(q)}(x)](x-a)^{q+1}. 
\end{equation}
This finishes the proof of the classical Hadamard lemma.
\end{proof}

\begin{lemma} \label{lem_ap_olphumorphism}
This is the complicated bit of the proof of Lemma \ref{lem_olphumorphism}.
\end{lemma}
\begin{proof}
We are going to prove that $\ol{\varphi}^{\ast}_{W}: \C^{\infty}_{m_{0}}(W) \rightarrow \C^{\infty}_{(n_{j})}( \ul{\varphi}^{-1}(W))$ preserves the algebra multiplication. Let us prove it for $W = V$. 

First, for each $\fp \in \N^{n_{\ast}}_{0}$, define the following (finite) set:
\begin{equation}
\N^{(k)}_{\fp} := \{ \vec{\fq} := (\fq_{1},\dots,\fq_{k}) \in (\N_{0}^{n_{\ast}})^{k} \; | \; \fq_{1} + \dots + \fq_{k} = \fp \}.
\end{equation}
For each $\vec{\fq} \in \N^{(k)}_{\fp}$, there is a sign $\epsilon_{\fp}^{\vec{\fq}}$ uniquely determined by the formula $\xi^{\fp} = \epsilon_{\fp}^{\vec{\fq}} \xi^{\fq_{1}} \cdots \xi^{\fq_{k}}$. This allows one to write the following $k$-fold product:
\begin{equation}
(\bar{y}^{j_{1}} \cdots \bar{y}^{j_{k}})_{\fp} = \sum_{\vec{\fq} \in \N^{(k)}_{\fp}} \epsilon_{\fp}^{\vec{\fq}} (\bar{y}^{j_{1}})_{\fq_{1}} \dots (\bar{y}^{j_{k}})_{\fq_{k}}. 
\end{equation}
Now, for two functions $f,g \in \C^{\infty}_{m_{0}}(V)$, we obtain the expression
\begin{equation}
\begin{split}
(&\ol{\varphi}^{\ast}_{V}(fg))_{\fp} = \sum_{k=0}^{w(\fp)} \frac{1}{k!} (\frac{\partial^{k} (fg)}{\partial y^{j_{1}} \dots \partial y^{j_{k}}} \circ \ul{\varphi})  \sum_{\vec{\fq} \in \N^{(k)}_{\fp}} \epsilon_{\fp}^{\vec{\fq}} (\bar{y}^{j_{1}})_{\fq_{1}} \dots (\bar{y}^{j_{k}})_{\fq_{k}} \\
= & \ \sum_{k=0}^{w(\fp)} \frac{1}{k!} \sum_{r=0}^{k} \binom{k}{r} (\frac{\partial^{r} f}{\partial y^{j_{1}} \dots \partial y^{j_{r}}} \circ \ul{\varphi}) (\frac{\partial^{k-r} g}{\partial y^{j_{r+1}} \dots \partial y^{j_{k}}} \circ \ul{\varphi}) \sum_{\vec{\fq} \in \N^{(k)}_{\fp}} \epsilon_{\fp}^{\vec{\fq}} (\bar{y}^{j_{1}})_{\fq_{1}} \dots (\bar{y}^{j_{k}})_{\fq_{k}} \\
= & \ \sum_{k=0}^{w(\fp)} \sum_{r=0}^{k} \frac{1}{r!} (\frac{\partial^{r} f}{\partial y^{j_{1}} \dots \partial y^{j_{r}}} \circ \ul{\varphi}) \frac{1}{(k-r)!}(\frac{\partial^{k-r} g}{\partial y^{j_{r+1}} \dots \partial y^{j_{k}}} \circ \ul{\varphi}) \sum_{\vec{\fq} \in \N^{(k)}_{\fp}} \epsilon_{\fp}^{\vec{\fq}} (\bar{y}^{j_{1}})_{\fq_{1}} \dots (\bar{y}^{j_{k}})_{\fq_{k}}. \\
\end{split}
\end{equation}
On the other hand, one finds the expression 
\begin{equation}
\begin{split}
( \ol{\varphi}^{\ast}_{V}(f) \cdot \ol{\varphi}^{\ast}_{V}(g))_{\fp} = & \ \sum_{\vec{\fr} \in \N^{(2)}_{\fp}} \epsilon^{\vec{\fr}}_{\fp} (\ol{\varphi}^{\ast}_{V}(f))_{\fr_{1}} (\ol{\varphi}^{\ast}_{V}(g))_{\fr_{1}}  \\
= & \ \sum_{\vec{\fr} \in \N_{\fp}^{(2)}} \epsilon^{\vec{\fr}}_{\fp} \big( \sum_{k_{1}=0}^{w(\fr_{1})} \frac{1}{k_{1}!} (\frac{\partial^{k_{1}} f}{\partial y^{l_{1}} \dots \partial y^{l_{k_{1}}}} \circ \ul{\varphi}) \sum_{\vec{\mathbf{u}} \in \N_{\fr_{1}}^{(k_{1})}} \epsilon_{\fr_{1}}^{\vec{\mathbf{u}}} (\bar{y}^{l_{1}})_{\mathbf{u}_{1}} \dots (\bar{y}^{l_{k_{1}}})_{\mathbf{u}_{k_{1}}} \big) \\
& \big( \sum_{k_{2}=0}^{w(\fr_{2})} \frac{1}{k_{2}!} (\frac{\partial^{k_{2}} g}{\partial y^{t_{1}} \dots \partial y^{t_{k_{2}}}} \circ \ul{\varphi}) \sum_{\vec{\mathbf{v}} \in \N_{\fr_{2}}^{(k_{2})}} \epsilon_{\fr_{2}}^{\vec{\mathbf{v}}} (\bar{y}^{t_{1}})_{\mathbf{v}_{1}} \dots (\bar{y}^{t_{k_{2}}})_{\mathbf{v}_{k_{2}}} \big).
\end{split}
\end{equation}
We have to argue that these two rather complicated sums are the same. The idea is to choose an arbitrary combination $(k,r,(j_{1},\dots,j_{k}),\vec{\fq})$ of summation indices of the first sum and assign to it the corresponding summand in the second sum.

\begin{enumerate}[(i)]
\item Let $\fr_{1} := \fq_{1} + \dots + \fq_{r}$, $\fr_{2} := \fq_{r+1} + \dots + \fq_{k}$. As $\vec{\fq} \in \N^{(k)}_{\fp}$, we have $\vec{\fr} := (\fr_{1},\fr_{2}) \in \N^{(2)}_{\fp}$. 
\item Let $k_{1} := r$, $k_{2} := k - r$. As $w(\fq_{i}) \geq 1$, one finds $k_{1} \leq w(\fr_{1})$ and $k_{2} \leq w(\fr_{2})$. 
\item Let $(l_{1},\dots,l_{k_{1}}) := (j_{1},\dots,j_{r})$ and $(t_{1},\dots,t_{k_{2}}) := (j_{r+1},\dots,j_{k})$. 
\item Let $\vec{\mathbf{u}} := (\fq_{1},\dots,\fq_{r})$ and $\vec{\mathbf{v}} := (\fq_{r+1}, \dots, \fq_{k})$. Clearly $\vec{\mathbf{u}} \in \N_{\fr_{1}}^{(k_{1})}$ and $\vec{\mathbf{v}} \in \N_{\fr_{2}}^{(k_{2})}$. 
\end{enumerate}
With this choice of indices, the corresponding summand in the second sum takes the form 
\begin{equation}
\frac{1}{r!} (\frac{\partial^{r} f}{\partial y^{j_{1}} \dots \partial y^{j_{r}}} \circ \ul{\varphi}) \frac{1}{(k-r)!}(\frac{\partial^{k-r} g}{\partial y^{j_{r+1}} \dots \partial y^{j_{k}}} \circ \ul{\varphi}) \epsilon_{\fp}^{\vec{\fr}} \epsilon^{\vec{\mathbf{u}}}_{\fr_{1}} \epsilon^{\vec{\mathbf{v}}}_{\fr_{2}} (\bar{y}^{j_{1}})_{\fq_{1}} \dots (\bar{y}^{j_{k}})_{\fq_{k}}.
\end{equation}
We thus only have to show that the product of the three signs is actually $\epsilon^{\vec{\fq}}_{\fp}$. We have
\begin{equation}
\xi^{\fp} = \epsilon_{\fp}^{\vec{\fr}} \xi^{\fr_{1}} \xi^{\fr_{2}} = \epsilon^{\vec{\fr}}_{\fp} ( \epsilon_{\fr_{1}}^{\vec{\mathbf{u}}} \xi^{\mathbf{u}_{1}} \cdots \xi^{\mathbf{u}_{k_{1}}})( \epsilon^{\vec{\mathbf{v}}}_{\fr_{2}} \xi^{\mathbf{v}_{1}} \cdots \xi^{\mathbf{v}_{k_{2}}}) = \epsilon_{\fp}^{\vec{\fr}} \epsilon^{\vec{\mathbf{u}}}_{\fr_{1}} \epsilon^{\vec{\mathbf{v}}}_{\fr_{2}} \xi^{\fq_{1}} \cdots \xi^{\fq_{k}}.
\end{equation}
But this equation also uniquely determines the sign $\epsilon^{\vec{\fq}}_{\fp}$, whence they are the same. It is not difficult to find the unique converse assignment, thus finding a bijection of the two indexing sets. Hence the two finite sums coincide and the map $\ol{\varphi}^{\ast}_{V}: \C^{\infty}_{m_{0}}(V) \rightarrow \C^{\infty}_{(n_{j})}(U)$ preserves the product.
\end{proof}

\begin{lemma} \label{lem_ap_themorphism}
This is the complicated bit of the proof of Lemma \ref{lem_themorphism}.
\end{lemma}
\begin{proof}
We have to show that $\varphi^{\ast}_{W}: \C^{\infty}_{(m_{j})}(W) \rightarrow \C^{\infty}_{(n_{j})}(\ul{\varphi}^{-1}(W))$ preserves the algebra multiplication for any $W \in \Op(V)$. Let us do this for $W = V$. 

First, recall that for $f \in \C^{\infty}_{(n_{j})}(V)$, its components $f_{\fq} \in \C^{\infty}_{n_{0}}(V)$ are defined for each $\fq \in \N^{n_{\ast}}_{|f|}$. To simplify the notation in the rest of the proof, we define $f_{\fq} := 0$ for all $\fq \in \cup_{k \neq |f|} \N^{n_{\ast}}_{k}$. Write $\N^{n_{\ast}}_{\bullet} := \cup_{k \in \Z} \N^{n_{\ast}}_{k}$. For each $k \in \N$ and $\fp \in \N^{n_{\ast}}_{\bullet}$ define
\begin{equation}
\N^{(k)}_{\fp} := \{ \vec{\fq} = (\fq_{1},\dots,\fq_{k}) \in (\N^{n_{\ast}}_{\bullet})^{k} \; | \; \fq_{1} + \dots + \fq_{k} = \fp \},
\end{equation}
These sets are finite and for each $\vec{\fq} \in \N^{(k)}_{\fp}$, there is a unique sign $\epsilon_{\fp}^{\vec{\fq}}$ defined by $\xi^{\fp} = \epsilon_{\fp}^{\vec{\fq}} \xi^{\fq_{1}} \cdots \xi^{\fq_{k}}$. 

Now, for $f,g \in \C^{\infty}_{(m_{j})}(V)$ and each $\fr \in \N^{m_{\ast}}_{\bullet}$, one has $(f \cdot g)_{\fr} = \sum_{\vec{\mathbf{u}} \in \N^{(2)}_{\fr}} \epsilon^{\vec{\mathbf{u}}}_{\fr} f_{\mathbf{u}_{1}} g_{\mathbf{u}_{2}}$, whence 
\begin{equation}
\ol{\varphi}^{\ast}_{V}( (f \cdot g)_{\fr}) = \sum_{\vec{\mathbf{u}} \in \N^{(2)}_{\fr}} \epsilon_{\fr}^{\vec{\mathbf{u}}} \ol{\varphi}^{\ast}_{V}(f_{\mathbf{u}_{1}}) \cdot \ol{\varphi}^{\ast}_{V}(g_{\mathbf{u}_{2}}),
\end{equation}
since we have proved in the previous lemma that $\ol{\varphi}^{\ast}_{V}$ is a graded algebra morphism. Now, let $\fp \in \N^{m_{\ast}}_{|f| + |g|}$ be given. One can then write 
\begin{equation}
\begin{split}
(& \varphi^{\ast}_{V}(f \cdot g))_{\fp} = \sum_{\substack{\fr \in \N^{m_{\ast}}_{\bullet} \\ w(\fr) \leq w(\fp)}} \sum_{\vec{\fq} \in \N^{(2)}_{\fp}} \epsilon_{\fp}^{\vec{\fq}} (( \ol{\varphi}^{\ast}_{V}((f \cdot g)_{\fr}))_{\fq_{1}} ( \bar{\theta}^{\fr})_{\fq_{2}}\\
= & \ \sum_{\substack{\fr \in \N^{m_{\ast}}_{\bullet} \\ w(\fr) \leq w(\fp)}} \sum_{\vec{\fq} \in \N^{(2)}_{\fp}} \sum_{\vec{\mathbf{u}} \in \N^{(2)}_{\fr}} \sum_{\vec{\mathbf{s}} \in \N^{(2)}_{\fq_{1}}} \epsilon_{\fp}^{\vec{\fq}} \epsilon_{\fr}^{\vec{\mathbf{u}}}  \epsilon_{\fq_{1}}^{\vec{\mathbf{s}}} (\ol{\varphi}^{\ast}_{V}(f_{\mathbf{u}_{1}}))_{\mathbf{s}_{1}} (\ol{\varphi}^{\ast}_{V}(f_{\mathbf{u}_{2}}))_{\mathbf{s}_{2}} (\epsilon^{\vec{\mathbf{u}}}_{\fr} \bar{\theta}^{\mathbf{u}_{1}} \bar{\theta}^{\mathbf{u}_{2}})_{\fq_{2}} \\
= & \ \sum_{\substack{\fr \in \N^{m_{\ast}}_{\bullet} \\ w(\fr) \leq w(\fp)}} \sum_{\vec{\fq} \in \N^{(2)}_{\fp}} \sum_{\vec{\mathbf{u}} \in \N^{(2)}_{\fr}} \sum_{\vec{\mathbf{s}} \in \N^{(2)}_{\fq_{1}}} \sum_{\vec{\mathbf{x}} \in \N^{(2)}_{\fq_{2}}} \epsilon_{\fp}^{\vec{\fq}} \epsilon_{\fq_{1}}^{\vec{\mathbf{s}}}  \epsilon^{\vec{\mathbf{x}}}_{\fq_{2}} (\ol{\varphi}^{\ast}_{V}(f_{\mathbf{u}_{1}}))_{\mathbf{s}_{1}} (\ol{\varphi}^{\ast}_{V}(f_{\mathbf{u}_{2}}))_{\mathbf{s}_{2}} (\bar{\theta}^{\mathbf{u}_{1}})_{\mathbf{x}_{1}} (\bar{\theta}^{\mathbf{u}_{2}})_{\mathbf{x}_{2}}. 
\end{split}
\end{equation}
On the other hand, one finds the expression 
\begin{equation}
\begin{split}
(& \varphi^{\ast}_{V}(f) \cdot \varphi^{\ast}_{V}(g))_{\fp} = \sum_{\vec{\mathbf{t}} \in \N^{(2)}_{\fp}} \epsilon^{\vec{\mathbf{t}}}_{\fp} (\varphi^{\ast}_{V}(f))_{\mathbf{t}_{1}} (\varphi^{\ast}_{V}(g))_{\mathbf{t}_{2}} \\
= & \ \sum_{\vec{\mathbf{t}} \in \N^{(2)}_{\fp}} \epsilon^{\vec{\mathbf{t}}}_{\fp} \big( \hspace{-4mm} \sum_{\substack{\mathbf{v} \in \N^{m_{\ast}}_{\bullet} \\ w(\mathbf{v}) \leq w(\mathbf{t}_{1})}} \hspace{-4mm} (\ol{\varphi}^{\ast}_{V}(f_{\mathbf{v}}) \cdot \bar{\theta}^{\mathbf{v}})_{\mathbf{t}_{1}} \big) \cdot \big( \hspace{-4mm} \sum_{\substack{\mathbf{w} \in \N^{m_{\ast}}_{\bullet} \\ w(\mathbf{w}) \leq w(\mathbf{t}_{2})}} \hspace{-4mm} (\ol{\varphi}^{\ast}_{V}(g_{\mathbf{w}}) \cdot \bar{\theta}^{\mathbf{w}})_{\mathbf{t}_{2}} \big) \\
= & \ \sum_{\vec{\mathbf{t}} \in \N^{(2)}_{\fp}} \hspace{-2mm} \sum_{\substack{\mathbf{v} \in \N^{m_{\ast}}_{\bullet} \\ w(\mathbf{v}) \leq w(\mathbf{t}_{1})}}  \sum_{\substack{\mathbf{w} \in \N^{m_{\ast}}_{\bullet} \\ w(\mathbf{w}) \leq w(\mathbf{t}_{2})}} \sum_{\vec{\mathbf{a}} \in \N^{(2)}_{\mathbf{t}_{1}}}  \sum_{\vec{\mathbf{b}} \in \N^{(2)}_{\mathbf{t}_{2}}} \epsilon^{\vec{\mathbf{t}}}_{\fp} \epsilon^{\vec{\mathbf{a}}}_{\mathbf{t}_{1}} \epsilon^{\vec{\mathbf{b}}}_{\mathbf{t}_{2}} (\ol{\varphi}^{\ast}_{V}(f_{\mathbf{v}}))_{\mathbf{a}_{1}} (\ol{\varphi}^{\ast}_{V}(g_{\mathbf{w}}))_{\mathbf{b}_{1}} (\bar{\theta}^{\mathbf{v}})_{\mathbf{a}_{2}} (\bar{\theta}^{\mathbf{w}})_{\mathbf{b}_{2}}.
\end{split}
\end{equation}
Now, the approach is the same as in the proof of the previous lemma. To each combination of summation indices $(\fr, \vec{\mathbf{q}}, \vec{\mathbf{u}}, \vec{\mathbf{s}}, \vec{\mathbf{x}})$ in the first sum, we shall now assign the corresponding term in the second sum. One proceeds from the outermost to the innermost summation.
\begin{enumerate}[(i)]
\item Let $\vec{\mathbf{t}} := \vec{\mathbf{s}} + \vec{\mathbf{x}}$. As $\vec{\mathbf{s}} \in \N^{(2)}_{\fq_{1}}$, $\vec{\mathbf{x}} \in \N^{(2)}_{\fq_{2}}$ and $\vec{\fq} \in \N^{(2)}_{\fp}$, one has $\vec{\mathbf{t}} \in \N^{(2)}_{\fq_{1} + \fq_{2}} = \N^{(2)}_{\fp}$. 
\item Let $\mathbf{v} := \mathbf{u}_{1}$ and $\mathbf{w} := \mathbf{u}_{2}$. In the first sum, we get a non-trivial contribution only if $w(\mathbf{u}_{1}) \leq w(\mathbf{x}_{1})$ and $w(\mathbf{u}_{2}) \leq w(\mathbf{x}_{2})$, so it suffices to consider that this is the case. This ensures that the constraints $w(\mathbf{v}) \leq w(\mathbf{t}_{1})$ and $w(\mathbf{w}) \leq w(\mathbf{t}_{2})$ are satisfied. 
\item Set $\vec{\mathbf{a}} := (\mathbf{s}_{1}, \mathbf{x}_{1})$ and $\vec{\mathbf{b}} := (\mathbf{s}_{2}, \mathbf{x}_{2})$. Clearly $\vec{\mathbf{a}} \in \N_{\mathbf{t}_{1}}^{(2)}$ and $\vec{\mathbf{b}} \in \N_{\mathbf{t}_{2}}^{(2)}$. 
\end{enumerate}
One only has to argue that for this choice of indices, the corresponding summands coincide. This boils down to the comparison of signs, namely we have to prove that
\begin{equation} \label{eq_themorphismsigneequation}
\epsilon_{\fp}^{\vec{\fq}} \epsilon_{\fq_{1}}^{\vec{\mathbf{s}}} \epsilon^{\vec{\mathbf{x}}}_{\fq_{2}} = \epsilon^{\vec{\mathbf{t}}}_{\fp} \epsilon^{\vec{\mathbf{a}}}_{\mathbf{t}_{1}} \epsilon^{\vec{\mathbf{b}}}_{\mathbf{t}_{2}}.
\end{equation}
To this account, observe that one can write 
\begin{equation}
\xi^{\fp} = \epsilon_{\fp}^{\vec{\fq}} \xi^{\fq_{1}} \xi^{\fq_{2}} = \epsilon_{\fp}^{\vec{\fq}} \epsilon_{\fq_{1}}^{\vec{\mathbf{s}}} \epsilon_{\fq_{2}}^{\vec{\mathbf{x}}} \xi^{\mathbf{s}_{1}} \xi^{\mathbf{s}_{2}} \xi^{\mathbf{x}_{1}} \xi^{\mathbf{x}_{2}}, \; \; \xi^{\fp} = \epsilon_{\fp}^{\vec{\mathbf{t}}} \xi^{\mathbf{t}_{1}} \xi^{\mathbf{t}_{2}} = \epsilon_{\fp}^{\vec{\mathbf{t}}} \epsilon_{\mathbf{t}_{1}}^{\vec{\mathbf{a}}} \epsilon_{\mathbf{t}_{2}}^{\vec{\mathbf{b}}} \xi^{\mathbf{s}_{1}} \xi^{\mathbf{x}_{1}} \xi^{\mathbf{s}_{2}} \xi^{\mathbf{x}_{2}}. 
\end{equation}
Now, see that it suffices to consider the case where $\mathbf{s}_{2} \in \N^{m_{\ast}}_{0}$. But then $\xi^{\mathbf{x}_{1}} \xi^{\mathbf{s}_{2}} = \xi^{\mathbf{s}_{2}} \xi^{\mathbf{x}_{1}}$ and we see that (\ref{eq_themorphismsigneequation}) holds. Hence every non-trivial summand of the first sum appears in the second one.

Similarly, one can find the converse assignment and the two sums coincide. As $\fp \in \N^{n_{\ast}}_{|f|+|g|}$ was arbitrary, this proves that $\varphi^{\ast}_{V}(f \cdot g) = \varphi^{\ast}_{V}(f) \cdot \varphi^{\ast}_{V}(g)$. 
\end{proof}
\begin{theorem} \label{thm_ap_globaldomain}
This is a detailed proof of Theorem \ref{thm_globaldomain}.
\end{theorem}
\begin{proof}
Let $\phi: \M \rightarrow \hat{V}^{(m_{j})}$ be a graded smooth map. Define the functions $\hat{y}^{j}$ and $\hat{\theta}_{\nu}$ by (\ref{eq_hatvariablesglobaldomainthm}). The map $(\ul{\hat{y}}^{1},\dots,\ul{\hat{y}}^{m_{0}}): M \rightarrow \R^{m_{0}}$ is then just $\ul{\phi}$, which by definition takes values in $\hat{V}$. We thus indeed obtain the data $(i)$ and $(ii)$. 

Conversely, let us start with the data $(i)$ and $(ii)$. We are going to construct a unique graded smooth map $\phi: \M \rightarrow \hat{V}^{(m_{j})}$ related to these data by (\ref{eq_hatvariablesglobaldomainthm}). Let $\A = \{ (U_{\alpha}, \varphi_{\alpha}) \}_{\alpha \in I}$ be a graded smooth atlas for $\M$. We will construct a collection $\{ \phi_{\alpha} \}_{\alpha \in I}$, where $\phi_{\alpha}: \M|_{U_{\alpha}} \rightarrow \hat{V}^{(m_{j})}$ satisfy $\phi_{\alpha}|_{U_{\alpha \beta}} = \phi_{\beta}|_{U_{\alpha \beta}}$ for all $(\alpha,\beta) \in I^{2}$. By Proposition \ref{tvrz_gLRSgluing}, there is then a unique graded smooth map $\phi: \M \rightarrow \hat{V}^{(m_{j})}$, such that $\phi|_{U_{\alpha}} = \phi_{\alpha}$ for each $\alpha \in I$. 

Let $\ul{\phi} := (\ul{\hat{y}}^{1}, \dots, \ul{\hat{y}}^{m_{0}}): M \rightarrow \hat{V}$ be the smooth map obtained from the data in $(i)$. Fix $\alpha \in I$ and let $\ul{\phi_{\alpha}} := \ul{\phi}|_{U_{\alpha}}$. For each $j \in \{1,\dots,m_{0}\}$, define $\bar{y}^{j}_{(\alpha)} \in \C^{\infty}_{(n_{j})}( \hat{U}_{\alpha})_{0}$ as
\begin{equation}
\bar{y}^{j}_{(\alpha)} := (\varphi^{-1}_{\alpha})^{\ast}_{U_{\alpha}}( \hat{y}^{j}|_{U_{\alpha}}) - y^{j} \circ \ul{\hat{\phi}_{\alpha}},
\end{equation}
where $\ul{\hat{\phi}_{\alpha}}: \hat{U}_{\alpha} \rightarrow \hat{V}$ is the composition $\ul{\hat{\phi}_{\alpha}} = \ul{\phi_{\alpha}} \circ \ul{\varphi_{\alpha}}^{-1}$. Similarly, for each $\nu \in \{1,\dots,m_{\ast}\}$, let 
\begin{equation}
\bar{\theta}_{\nu}^{(\alpha)} := (\varphi_{\alpha}^{-1})^{\ast}_{U_{\alpha}}( \hat{\theta}_{\nu}|_{U_{\alpha}} ) \in \C^{\infty}_{(n_{j})}(\hat{U}_{\alpha})_{|\theta_{\mu}|}.
\end{equation}
It is obvious that $\bar{y}^{j}_{(\alpha)} \in \J^{\pg}_{(n_{j})}(\hat{U}_{\alpha})_{0}$. Consequently, the smooth map $\ul{\hat{\phi}_{\alpha}}: \hat{U}_{\alpha} \rightarrow
 \hat{V}$ together with the collections $\{ \bar{y}^{j}_{(\alpha)} \}_{j=1}^{m_{0}}$ and $\{ \bar{\theta}^{(\alpha)}_{\nu} \}_{\nu=1}^{m_{\ast}}$ form the data $(i)$ - $(iii)$ in Theorem \ref{thm_gradedomaintheorem}. There is thus a unique morphism of graded domains $\hat{\phi}_{\alpha}: \hat{U}_{\alpha}^{(n_{j})} \rightarrow \hat{V}^{(m_{j})}$ satisfying 
\begin{equation}
(\hat{\phi}_{\alpha})^{\ast}_{\hat{V}}(y^{j}) = (\varphi^{-1}_{\alpha})^{\ast}_{U_{\alpha}}( \hat{y}^{j}|_{U_{\alpha}}), \; \; (\hat{\phi}_{\alpha})^{\ast}_{V}(\theta_{\nu}) = (\varphi_{\alpha}^{-1})^{\ast}_{U_{\alpha}}( \hat{\theta}_{\nu}|_{U_{\alpha}}),
\end{equation}
for every $j \in \{1,\dots,m_{0}\}$ and $\nu \in \{1,\dots,m_{\ast}\}$. Consequently, the graded smooth map $\phi_{\alpha} := \hat{\phi}_{\alpha} \circ \varphi_{\alpha}: \M|_{U_{\alpha}} \rightarrow \hat{V}^{(m_{j})}$ for each $j \in \{1,\dots,m_{0}\}$ and $\nu \in \{1,\dots,m_{\ast}\}$ satisfies
\begin{equation}
(\phi_{\alpha})^{\ast}_{\hat{V}}(y^{j}) = \hat{y}^{j}|_{U_{\alpha}}, \; \; (\phi_{\alpha})^{\ast}_{\hat{V}}(\theta_{\nu}) = \hat{\theta}_{\nu}|_{U_{\alpha}}.
\end{equation}
Now, it remains to prove that $\phi_{\alpha}|_{U_{\alpha \beta}} = \phi_{\beta}|_{U_{\alpha \beta}}$ for all $(\alpha,\beta) \in I^{2}$. This can be rephrased as 
\begin{equation}
\hat{\phi}_{\alpha}|_{\ul{\varphi_{\alpha}}(U_{\alpha \beta})} \circ \varphi_{\alpha \beta} = \hat{\phi}_{\beta}|_{\ul{\varphi_{\beta}}(U_{\alpha \beta})}
\end{equation}
Both sides are morphisms of graded domains, hence by Remark \ref{rem_itsufficestocalculatepullbacks}, it suffices to compare the pullbacks of global coordinate functions. For each $j \in \{1,\dots,m_{0}\}$, one finds
\begin{equation}
( \hat{\phi}_{\alpha}|_{\ul{\varphi_{\alpha}}(U_{\alpha \beta})})^{\ast}_{\hat{V}}(y^{j}) = (\hat{\phi}_{\alpha})^{\ast}_{\hat{V}}(y^{j})|_{\ul{\varphi_{\alpha}}(U_{\alpha \beta})}  = (\varphi_{\alpha}^{-1})^{\ast}_{U_{\alpha \beta}}( \hat{y}^{j}|_{U_{\alpha \beta}}).
\end{equation}
Consequently, for each $j \in \{1,\dots,m_{0}\}$, we obtain 
\begin{equation}
\begin{split}
(\hat{\phi}_{\alpha}|_{\ul{\varphi_{\alpha}}(U_{\alpha \beta})} \circ \varphi_{\alpha \beta})^{\ast}_{\hat{V}}(y^{j}) = & \ (\varphi_{\alpha \beta})^{\ast}_{\ul{\varphi_{\alpha}}(U_{\alpha \beta})}( (\varphi_{\alpha}^{-1})^{\ast}_{U_{\alpha \beta}}( \hat{y}^{j}|_{U_{\alpha \beta}}) ) \\
= & \ (\varphi_{\beta}^{-1})^{\ast}_{U_{\alpha \beta}}( \hat{y}^{j}|_{U_{\alpha \beta}}) = ( \hat{\phi}_{\beta}|_{\ul{\varphi_{\beta}}(U_{\alpha \beta})})^{\ast}_{\hat{V}}(y^{j}).
\end{split}
\end{equation}
The proof for the coordinate functions $\theta_{\nu}$ is completely the same and the conclusion follows.  

We can thus construct $\phi$ by gluing the graded smooth maps $\{ \phi_{\alpha} \}_{\alpha \in I}$ as in Proposition \ref{tvrz_gLRSgluing}. The equations (\ref{eq_hatvariablesglobaldomainthm}) and the uniqueness of $\phi$ satisfying this equations can be now easily verified by restricting $\phi$ to the open cover $\{ U_{\alpha} \}_{\alpha \in I}$. This finishes the proof.
\end{proof}
\subsection{Partitions of unity}
\begin{tvrz} \label{tvrz_ap_partition}
This is the proof of Proposition \ref{tvrz_partition}.
\end{tvrz}
Before we can actually prove the statement, we need some terminology and auxiliary statements. 

Let $\M = (M,\C^{\infty}_{\M})$ be a graded manifold of a graded dimension $(n_{j})_{j \in \Z}$. $B \subseteq M$ is called a \textbf{regular coordinate ball} on $\M$, if there is $B' \subseteq M$ such that $\ol{B} \subseteq B'$ and there is a graded local chart $(B',\varphi)$ such that for some positive real numbers $r < r'$, one has 
\begin{equation} \ul{\varphi}(B) = B_{r}(0), \; \; \ul{\varphi}(\ol{B}) = \ol{B}_{r}(0) \text{ and } \ul{\varphi}(B') = B_{r'}(0). \end{equation}
 By $B_{r}(x) \subseteq \R^{n_{0}}$ we denote the open ball of radius $r$ around $x \in \R^{n_{0}}$.
 
\begin{lemma} \label{lem_balls}
For every graded manifold $\M = (M,\C^{\infty}_{\M})$, there exists a countable basis $\B$ for the topology on $M$ consisting of regular coordinate balls.
\end{lemma}
\begin{proof}
This is a generalization of the proof of Lemma 1.10 in \cite{lee2012introduction}. First, suppose that there is a single global chart $(M,\varphi)$ for $\M$. Hence $\ul{\varphi}: M \rightarrow \hat{U} \in \Op(\R^{n_{0}})$. 

Let $\hat{\B}$ be a collection of all open balls $B_{r}(x)$ with rational $r > 0$ and $x \in \mathbb{Q}^{n_{0}}$, such that there exists $r' > r$ satisfying $B_{r'}(x) \subseteq \hat{U}$. We claim that $\hat{\B}$ forms a basis for the topology of $\hat{U}$. We have to argue that for any $\hat{W} \in \Op(\hat{U})$ and any $y \in \hat{W}$, there is $B_{r}(x) \in \hat{\B}$, such that $y \in B_{r}(x) \subseteq \hat{W}$. Fix $q > 0$ so that $B_{q}(y) \subseteq \hat{W}$. It thus suffices to find $x \in \mathbb{Q}^{n_{0}}$ and $0 < r < r'$ with rational $r$, such that $B_{r'}(x) \subseteq B_{q}(y)$ and $y \in B_{r}(x)$. For any $z \in B_{r'}(x)$, one has $\| z - y \| < r' + \| x - y \|$. We thus need to satisfy $r' + \| x - y \| < q$. On the other hand, we also need $x$ to satisfy $\| x - y \| < r$. Altogether, we have to make $\| x - y \|$ small enough so that there exist $r \in \mathbb{Q}$ and $r'$ so that
\begin{equation}
\| x - y \| < r < r' < q - \| x - y \|.
\end{equation}
This is possible as soon as $\| x - y \| < \frac{q}{2}$. But such $x \in \mathbb{Q}^{n_{0}}$ can certainly be found as $\mathbb{Q}^{n_{0}}$ is dense in $\R^{n_{0}}$. We see that $\hat{\B}$ is a countable basis for the topology of $\hat{U}$. 

It follows that $\B := \{ \ul{\varphi}^{-1}(\hat{B}) \; | \; \hat{B} \in \hat{\B} \}$ forms a countable basis for the topology on $M$. 

For each $B \in \B$, one has $\ul{\varphi}(B) = B_{r}(x)$ for some $x \in \hat{U}$ and $r > 0$. By construction, there is $r' > r$ such that $B_{r'}(x) \subseteq \hat{U}$, hence one can define $B' := \ul{\varphi}^{-1}(B_{r'}(x))$. Let $\varphi_{B}: \M|_{B'} \rightarrow B_{r'}(0)^{(n_{j})}$ be the restriction of $\varphi$ composed with the graded diffeomorphism $B_{r'}(x)^{(n_{j})} \rightarrow B_{r'}(0)^{(n_{j})}$ induced by the translation $B_{r'}(x) \mapsto B_{r'}(0)$. It is now easy to see that $(B',\varphi_{B})$ is a graded chart making $B$ into a regular coordinate ball.

Finally, let $\M$ be a general graded manifold with a graded smooth atlas $\A = \{ (U_{\alpha}, \varphi_{\alpha}) \}_{\alpha \in I}$. As $M$ is a second countable topological space, one can assume that $I$ is countable. By the previous part, for each $\alpha \in I$, we can find a countable basis $\B_{\alpha}$ for the topology of $U_{\alpha}$ consisting of regular coordinate balls. Then $\B = \cup_{\alpha \in I} \B_{\alpha}$ is a desired countable basis for the topology of $M$.
\end{proof}
Let us recall the following standard statement. See e.g. Theorem 1.15 in \cite{lee2012introduction}. 
\begin{lemma}[\textbf{Smooth manifolds are paracompact}] \label{lem_paracompact}
Let $\{ U_{\mu} \}_{\mu \in J}$ be any open cover of $M$. Let $\B$ be any basis for the topology of $M$. Then there exists a countable and locally finite refinement of this cover, consisting of the sets of $\B$.  
\end{lemma}
Finally, one has the following general lemma regarding the locally finite collections of sets.
\begin{lemma} \label{lem_locallyfinitesets}
Let $\{ X_{\mu} \}_{\mu \in J}$ be a locally finite collection of subsets of any topological space $M$. Then the collection $\{ \ol{X}_{\mu} \}_{\mu \in J}$ is also locally finite and $\ol{\cup_{\mu \in J} X_{\mu}} = \cup_{\mu \in J} \ol{X}_{\mu}$. 
\end{lemma}
\begin{proof}
The key to the first part is to observe that for any $U \in \Op(M)$ and any subset $X \subseteq M$, one has $U \cap X \neq \emptyset$, iff $U \cap \ol{X} \neq \emptyset$. One direction is obvious, so suppose that $m \in U \cap (\ol{X} - X)$. But then $m$ is a limit point of $X$ and $U$ is its open neighborhood, hence $U \cap X \neq \emptyset$. 

For the second statement, the inclusion $\supseteq$ is true for any collection of subsets. The inclusion $\subseteq$ also holds for any collection of subsets whenever $J$ is finite. Indeed, $\cup_{\mu \in J} \ol{X}_{\mu}$ is then a closed set containing $\cup_{\mu \in J} X_{\mu}$, hence it must contain also its closure. 

Now, suppose that $\{ X_{\mu} \}_{\mu \in J}$ is locally finite and let $m \in \ol{\cup_{\mu \in J} X_{\mu}}$. There is thus a sequence $\{ m_{n} \}_{n = 1}^{\infty} \subseteq \cup_{\mu \in J} X_{\mu}$ converging to $m$. By definition, there is $U \in \Op_{m}(M)$ and a finite subset $J_{0} \subseteq J$ such that $U \cap X_{\mu} \neq \emptyset$ only if $\mu \in J_{0}$. But this implies the existence of some $n_{0} \in \N$, such that $\{ m_{n} \}_{n>n_{0}} \in \cup_{\mu \in J_{0}} X_{\mu}$. Whence $m \in \ol{ \cup_{\mu \in J_{0}} X_{\mu}}$. Since $J_{0}$ is finite, the previous paragraph states that $m \in \cup_{\mu \in J_{0}}  \ol{X}_{\mu} \subseteq \cup_{\mu \in J} \ol{X}_{\mu}$. This finishes the proof.
\end{proof}
These tools allow us to finally prove the proposition itself.
\begin{proof}[Proof of Proposition \ref{tvrz_ap_partition}] This is a slight modification of the proof or Theorem 2.23 in \cite{lee2012introduction}. 

Let $\{ U_{\mu} \}_{\mu \in J}$ be any open cover of $M$. For each $\mu \in J$, $\M|_{U_{\mu}}$ is itself a graded manifold, hence by Lemma \ref{lem_balls}, there exists a countable basis $\B_{\mu}$ for the topology of $U_{\mu}$ consisting of regular coordinate balls. Then $\B = \cup_{\mu \in J} \B_{\mu}$ is a basis for the topology of $M$. By Lemma \ref{lem_paracompact}, there exists a countable and locally finite refinement $\{ B_{\nu} \}_{\nu \in J'}$ of the cover $\{ U_{\mu} \}_{\mu \in J}$, such that $B_{\nu} \in \B$ for each $\nu \in J'$. In particular, we have a set map $\zeta: J' \rightarrow J$ such that $B_{\nu} \subseteq U_{\zeta(\nu)}$. The collection $\{ \ol{B}_{\nu} \}_{\nu \in J'}$ is also locally finite by Lemma \ref{lem_locallyfinitesets}. For every $\nu \in J'$, there exists a graded chart $(B'_{\nu}, \varphi_{\nu})$ for $\M$, such that $B_{\nu} \subseteq \ol{B}_{\nu} \subseteq B'_{\nu} \subseteq U_{\zeta(\nu)}$, and 
\begin{equation}
\ul{\varphi_{\nu}}(B_{\nu}) = B_{r_{\nu}}(0), \; \; \ul{\varphi_{\nu}}( \ol{B}_{\nu}) = \ol{B}_{r_{\nu}}(0), \; \; \ul{\varphi_{\nu}}(B'_{\nu}) = B_{r'_{\nu}}(0),
\end{equation}
for some real numbers $0 < r_{\nu} < r'_{\nu}$. There always exists a smooth function $\hat{h}_{\nu} \in \C^{\infty}_{n_{0}}(\R^{n_{0}})$ which is strictly positive on $B_{r_{\nu}}(0)$ and zero everywhere else. See Lemma 2.22 of \cite{lee2012introduction} for details. We can promote it to a function $\hat{h}_{\nu} \in \C^{\infty}_{(n_{j})}( \R^{n_{0}})_{0}$. For each $\nu \in J'$, define $h_{\nu} \in \C^{\infty}_{\M}(M)_{0}$ by
\begin{equation} \label{eq_hnufunkce}
h_{\nu}|_{B'_{\nu}} := (\varphi^{\ast}_{\nu})_{B_{r'_{\nu}}(0)}( \hat{h}_{\nu}|_{B_{r'_{\nu}}(0)}), \; \; h_{\nu}|_{(\ol{B}_{\nu})^{c}} := 0
\end{equation}
One has to check that the definitions agree on $B'_{\nu} \cap (\ol{B}_{\nu})^{c} = \ul{\varphi}^{-1}( B_{r'_{\nu}}(0) - \ol{B}_{r_{\nu}}(0))$. But
\begin{equation}
((\varphi^{\ast}_{\nu})_{B_{r'_{\nu}}(0)}( \hat{h}_{\nu}|_{B_{r'_{\nu}}(0)}))|_{B'_{\nu} \cap (\ol{B}_{\nu})^{c}} = \hat{h}_{\nu}|_{B_{r'_{\nu}}(0) - \ol{B}_{r_{\nu}}(0)} = 0.
\end{equation} 
One has $h_{\nu}(m) = \hat{h}_{\nu}( \ul{\varphi_{\nu}}(m)) > 0$ for every $m \in B_{\nu}$. Observe that $\supp(h_{\nu}) = \ol{B}_{\nu}$, whence $\{ \supp(h_{\nu}) \}_{\nu \in J'}$ forms a locally finite collection of compact sets. Let $h = \sum_{\nu \in J'} h_{\nu}$. As $\{ B_{\nu} \}_{\nu \in J'}$ forms an open cover of $M$, it follows that $h(m) > 0$ for all $m \in M$. By Proposition \ref{tvrz_invertibility}-$(i)$, it has a multiplicative inverse $h^{-1} \in \C^{\infty}_{\M}(M)$ and one can define $\lambda'_{\nu} := h^{-1} \cdot h_{\nu}$ for each $\nu \in J'$. Note that $\supp(\lambda'_{\nu}) = \supp(h_{\nu}) = \ol{B}_{\nu} \subseteq U_{\zeta(\nu)}$, and $\ul{\lambda'_{\nu}} \geq 0$. Finally, one has $\sum_{\nu \in J'} \lambda'_{\nu} = 1$. In other words, we have constructed a compactly supported partition of unity $\{ \lambda'_{\nu} \}_{\nu \in J'}$ subordinate to $\{ U_{\mu} \}_{\mu \in J}$. This proves the alternative part of the proposition, let us prove the main one.

Note that $\{ B'_{\nu} \}_{\nu \in J'}$ is also a refinement of $\{ U_{\mu} \}_{\mu \in J}$ as $B'_{\nu} \subseteq U_{\zeta(\nu)}$ for every $\nu \in J'$. Let 
\begin{equation}
\lambda_{\mu} := \hspace{-2mm} \sum_{\nu \in \zeta^{-1}(\mu)} \hspace{-2mm} \lambda'_{\nu},
\end{equation}
for each $\mu \in J$. This is well-defined as $\{ \supp(\lambda'_{\nu}) \}_{\nu \in \zeta^{-1}(\mu)}$ is locally finite. If $\zeta^{-1}(\mu) = \emptyset$, declare $\lambda_{\mu} := 0$. Let us examine supports of those functions. We claim that 
\begin{equation} \label{eq_lambdamusupports}
\supp(\lambda_{\mu}) = \cup_{\nu \in \zeta^{-1}(\mu)} \ol{B}_{\nu}.
\end{equation}
For any $m \in M$, we have $U \in \Op_{m}(M)$ and a finite subset $J'_{0} \subseteq J$, such that $U \cap \ol{B}_{\nu} \neq \emptyset$ only if $\nu \in J'_{0}$. Consequently, one has $[\lambda_{\mu}]_{m} = \sum_{\nu \in \zeta^{-1}(\mu) \cap J'_{0}} [\lambda'_{\nu}]_{m}$. This implies the inclusion $\supp(\lambda_{\mu}) \subseteq \cup_{\nu \in \zeta^{-1}(\mu)} \supp(\lambda'_{\nu})$. As $\supp(\lambda'_{\nu}) = \ol{B}_{\nu}$, this proves the inclusion $\subseteq$ of (\ref{eq_lambdamusupports}). Conversely, suppose that there is $\nu_{0} \in \zeta^{-1}(\mu)$, such that $m \in \ol{B}_{\nu_{0}} = \supp(\ul{\lambda'_{\nu_{0}}})$. Clearly $\nu_{0} \in J'_{0}$. Since $\ul{\lambda_{\mu}}|_{U} = \sum_{\nu \in \zeta^{-1}(\mu) \cap J'_{0}} \ul{\lambda'_{\nu}}|_{U}$ and $\ul{\lambda'_{\nu}} \geq 0$ for all $\nu \in J'$, this implies that $m \in \supp( \ul{\lambda_{\mu}})$. Since $\supp(\ul{\lambda_{\mu}}) \subseteq \supp(\lambda_{\mu})$, the inclusion $\supseteq$ of (\ref{eq_lambdamusupports}) follows. 

It now follows that $\supp(\lambda_{\mu}) \subseteq U_{\mu}$, $\ul{\lambda_{\mu}} \geq 0$. For each $m \in M$, let $U \in \Op_{m}(M)$ and $J'_{0} \subseteq J'$ be as in the previous paragraph. It follows from (\ref{eq_lambdamusupports}) that $U \cap \supp(\lambda_{\mu}) \neq \emptyset$ only if $\mu \in \zeta(J'_{0})$, where $\zeta(J'_{0}) \subseteq J$ is finite. Hence $\{ \supp(\lambda_{\mu}) \}_{\mu \in J}$ is locally finite. Finally, one has $\sum_{\mu \in J} \lambda_{\mu} = \sum_{\nu \in J'} \lambda'_{\nu} = 1$. This finishes the proof. 
\end{proof}
\subsection{Vector fields}
\begin{tvrz} \label{tvrz_ap_partialofxixibyxi}
The formula (\ref{eq_partialofxixibyxi}) is true.
\end{tvrz}
\begin{proof}
For the left-hand side, one has 
\begin{equation}
\frac{\partial(\xi^{\fq} \cdot \xi^{\fp-\fq})}{\partial \xi_{\mu}} = \epsilon^{\fq,\fp-\fq} \frac{\partial \xi^{\fp}}{\partial \xi^{\mu}} = \epsilon^{\fq,\fp-\fq} p_{\mu} (-1)^{|\xi_{\mu}|( p_{1}|\xi_{1}| + \dots + p_{\mu-1} |\xi_{\mu-1}|)} \xi^{\fp[\mu]},
\end{equation}
where $\xi^{\fp[\mu]} := (\xi_{1})^{p_{1}} \dots (\xi_{\mu})^{p_{\mu}-1} \dots (\xi_{n_{\ast}})^{p_{n_{\ast}}}$. Note that the case $p_{\mu} = 0$ is dealt with by the scalar factor. Altogether, the factor in front of $\xi^{\fp[\mu]}$ on the left-hand side of (\ref{eq_partialofxixibyxi}) is 
\begin{equation} \label{eq_xixixiprefactor1}
\epsilon^{\fq,\fp-\fq} (q_{\mu} + (p_{\mu} - q_{\mu})) (-1)^{|\xi_{\mu}|( p_{1}|\xi_{1}| + \dots + p_{\mu-1} |\xi_{\mu-1}|)}.
\end{equation}
For the first term on the right-hand side, one obtains
\begin{equation}
\frac{\partial \xi^{\fq}}{\partial \xi_{\mu}} \cdot \xi^{\fp - \fq} = q_{\mu} (-1)^{|\xi_{\mu}|( q_{1}|\xi_{1}| + \dots + q_{\mu - 1} |\xi_{\mu - 1}|)} (\xi_{1})^{q_{1}} \cdots (\xi_{\mu})^{q_{\mu} - 1} \cdots (\xi_{n_{\ast}})^{q_{n_{\ast}}} \cdot \xi^{\fp - \fq}. 
\end{equation}
Now, one only needs to find the sign acquired by reordering of the monomial on the right-hand side to get $\xi^{\fp[\mu]}$. It can be written using $\epsilon^{\fq,\fp-\fq}$ as in only differs by commuting $(\xi_{\nu})^{p_{\nu} - q_{\nu}}$ through $(\xi_{\mu})^{q_{\mu - 1}}$ instead of $(\xi_{\mu})^{q_{\mu}}$ for all $\nu \in \{1, \dots, \mu-1 \}$. We find
\begin{equation}
(\xi_{1})^{q_{1}} \cdots (\xi_{\mu})^{q_{\mu} - 1} \cdots (\xi_{n_{\ast}})^{q_{n_{\ast}}} \cdot \xi^{\fp - \fq} = (-1)^{|\xi_{\mu}|( (p_{1} - q_{1})|\xi_{1}| + \dots + (p_{\mu-1} - q_{\mu-1}) |\xi_{\mu-1}|)} \epsilon^{\fq,\fp-\fq} \xi^{\fp[\mu]}. 
\end{equation}
We conclude that the factor in front of $\xi^{\fp[\mu]}$ for the first right-hand side term of (\ref{eq_partialofxixibyxi}) reads 
\begin{equation} \label{eq_xixixiprefactor2}
\epsilon^{\fq,\fp-\fq} q_{\mu} (-1)^{|\xi_{\mu}|(p_{1}|\xi_{1}| + \dots + p_{\mu-1} |\xi_{\mu-1}|)} 
\end{equation}
For the second term on the right-hand side, one finds the expression 
\begin{equation}
\begin{split}
(-1)^{|\xi_{\mu}||f|} \xi^{\fq} \frac{\partial \xi^{\fp - \fq}}{\partial \xi_{\mu}} = & \ (p_{\mu} - q_{\mu}) (-1)^{|\xi_{\mu}|(|f| + (p_{1} - q_{1})|\xi_{1}| + \dots + (p_{\mu-1} - q_{\mu-1})|\xi_{\mu-1}|)}  \\
& \times \xi^{\fq} \cdot (\xi_{1})^{p_{1} - q_{1}} \cdots (\xi_{\mu})^{p_{\mu}-q_{\mu}-1} \cdots (\xi_{n_{\ast}})^{p_{n_{\ast}}}. 
\end{split}
\end{equation}
Again, we have to find a sign required to reorder the right-hand side monomial to $\xi^{\fp[\mu]}$. It differs from $\epsilon^{\fq,\fp-\fq}$ by a sign obtained by commuting $\xi_{\mu}$ through $(\xi_{\kappa})^{q_{\kappa}}$  for all $\kappa \in \{ \mu+1, \dots, n_{\ast} \}$, that is 
\begin{equation}
\xi^{\fq} \cdot (\xi_{1})^{p_{1} - q_{1}} \cdots (\xi_{\mu})^{p_{\mu}-q_{\mu}-1} \cdots (\xi_{n_{\ast}})^{p_{n_{\ast}}} = \epsilon^{\fq,\fp-\fq} (-1)^{|\xi_{\mu}|( q_{\mu+1} |\xi_{\mu+1} + \dots +q_{n_{\ast}} |\xi_{n_{\ast}}|)} \xi^{\fp[\mu]}. 
\end{equation}
One now has use the fact that $|f| = q_{1} |\xi_{1}| + \dots + q_{n_{\ast}} |\xi_{n_{\ast}}|$ to conclude that the overall factor in front of $\xi^{\fp[\mu]}$ for the second right-hand side term of (\ref{eq_partialofxixibyxi}) takes the form
\begin{equation} \label{eq_xixixiprefactor3}
\epsilon^{\fq,\fp-\fq} (p_{\mu} - q_{\mu}) (-1)^{|\xi_{\mu}|( p_{1}|\xi_{1}| + \dots + p_{\mu-1} |\xi_{\mu-1}|)} (-1)^{q_{\mu} |\xi_{\mu}|^{2}}.
\end{equation}
Finally, note that the sign $(-1)^{q_{\mu} |\xi_{\mu}|^{2}}$ can be negative only for odd $|\xi_{\mu}|$ and $q_{\mu} = 1$. Then necessarily $p_{\mu} \in \{0,1\}$ as $\fp \in \N^{n_{\ast}}_{|f|+|g|}$. Moreover, as $\fq \leq \fp$, we have $p_{\mu} = 1$. But then $p_{\mu} - q_{\mu} = 0$, so we may safely ignore the additional sign. We see that (\ref{eq_xixixiprefactor1}) is precisely the sum of (\ref{eq_xixixiprefactor2}) and (\ref{eq_xixixiprefactor3}). This proves the equation (\ref{eq_partialofxixibyxi}) and we are finished. 
\end{proof}
\subsection{Graded vector bundles}
\begin{tvrz} \label{tvrz_ap_pullbackVB}
This is the proof of Proposition \ref{tvrz_pullbackVB}.
\end{tvrz}
\begin{proof}
Let $\E$ be a graded vector bundle over $\M$ and let $\phi: \cN \rightarrow \M$ be a graded smooth map. Let us start by constructing the pullback sheaf $\phi^{\ast} \Gamma_{\E}$. This is done in several steps and we will leave some details to the interested reader. 

\begin{enumerate}[(a)]

\item \textbf{Construction of the pullback sheaf:} First, we have the underlying smooth map $\ul{\phi}: N \rightarrow M$. For each $V \in \Op(N)$, one defines
\begin{equation}
(\ul{\phi}^{-1}_{P}\C^{\infty}_{\M})(V) := \colim_{U \supseteq \ul{\phi}(V)} \C^{\infty}_{\M}(U).
\end{equation}
This defines the so called \textbf{inverse image presheaf} $\ul{\phi}_{P}^{-1} \C^{\infty}_{\M} \in \PSh(N,\gcAs)$. The colimit exists and it can be constructed in the same way as in the proof of Proposition \ref{tvrz_stalk}. Elements of $(\ul{\phi}_{P}^{-1} \C^{\infty}_{\M})(V)$ are classes $[f]_{V}$ represented by $f \in \C^{\infty}_{\M}(U)$ on some open subset $U \supseteq \ul{\phi}(V)$, where $[f]_{V} = [g]_{V}$, if $f|_{W} = g|_{W}$ for some open subset $W \supseteq \ul{\phi}(V)$. 

Note that $\ul{\phi}_{P}^{-1} \C^{\infty}_{\M}$ is in general not a sheaf. Indeed, if $\ul{\phi}$ is a constant mapping, then $\ul{\phi}_{P}^{-1} \C^{\infty}_{\M}$ is a constant presheaf, which in general is not a sheaf. The \textbf{inverse image sheaf} is thus defined as the corresponding sheafification, $\ul{\phi}^{-1} \C^{\infty}_{\M} := \Sff( \ul{\phi}_{P}^{-1} \C^{\infty}_{\M})$. 

Similarly, for any $\F \in \Sh^{\C^{\infty}_{\M}}(M,\gVect)$, one can define the inverse image presheaf $\ul{\phi}_{P}^{-1} \F \in \PSh(N,\gVect)$. It is not difficult to see that for each $V \in \Op(N)$, $(\ul{\phi}_{P}^{-1} \F)(V)$ has a structure of a graded $(\ul{\phi}_{P}^{-1} \C^{\infty}_{\M})(V)$-module. Moreover, there is a structure of a graded $(\ul{\phi}_{P}^{-1} \C^{\infty}_{\M})(V)$-module on $\C^{\infty}_{\cN}(V)$. Indeed, for any $[f]_{V} \in (\ul{\phi}_{P}^{-1} \C^{\infty}_{\M})(V)$ and $g \in \C^{\infty}_{\cN}(V)$, set $[f]_{V} \tr g := \phi^{\ast}_{U}(f)|_{V} \cdot g$, where $[f]_{V}$ is represented by $f \in \C^{\infty}_{\M}(U)$ on some open subset $U \supseteq \ul{\phi}(V)$.

For each $V \in \Op(N)$, we may thus consider the tensor product
\begin{equation}
(\phi^{\ast}_{P} \F)(V) := \C^{\infty}_{\cN}(V) \otimes_{(\ul{\phi}_{P}^{-1} \C^{\infty}_{\M})(V)} (\ul{\phi}_{P}^{-1} \F)(V).
\end{equation}
This graded vector space has a natural structure of a graded $\C^{\infty}_{\cN}(V)$-module. In fact, it defines a \textit{presheaf} of graded $\C^{\infty}_{\cN}$-modules, $\phi^{\ast}_{P} \F \in \PSh^{\C^{\infty}_{\cN}}(N, \gVect)$. Again, in general, this may fail to be a sheaf. The \textbf{pullback sheaf} is then defined as its sheaffication 
\begin{equation}
\phi^{\ast} \F := \Sff( \phi^{\ast}_{P} \F) \in \Sh^{\C^{\infty}_{\cN}}(N, \gVect). 
\end{equation}
Importantly, the assignment $\F \mapsto \phi^{\ast} \F$ defines a functor 
\begin{equation} \label{eq_pullbackfunctor}
\phi^{\ast}: \Sh^{\C^{\infty}_{\M}}(M,\gVect) \rightarrow \Sh^{\C^{\infty}_{\cN}}(N, \gVect).
\end{equation}
Now, observe that in Remark \ref{rem_pushforwardmodules}, we have in fact constructed a pushforward functor 
\begin{equation}
\phi_{\ast}: \Sh^{\C^{\infty}_{\cN}}(N,\gVect) \rightarrow \Sh^{\C^{\infty}_{\M}}(M,\gVect),
\end{equation}
These two functors are mutually adjoint, that is there is a canonical bijection 
\begin{equation} \label{eq_pushpulladjunction}
\mu_{\F,\G} : \Sh^{\C^{\infty}_{\cN}}_{0}( \phi^{\ast} \F, \G) \rightarrow \Sh^{\C^{\infty}_{\M}}_{0}(\F, \phi_{\ast} \G),
\end{equation}
natural in $\F$ and $\G$ (also see Definition \ref{def_ShAFG}). This concludes the construction of the pullback sheaf of graded modules. Note that pullback is compatible with restrictions. For any $U \in \Op(M)$, let $\phi' := \phi|_{\ul{\phi}^{-1}(U)}: \cN|_{\ul{\phi}^{-1}(U)} \rightarrow \M|_{U}$. Then there is a canonical identification
\begin{equation} \label{eq_pullrestriction}
\phi'^{\ast} \F|_{U} \cong (\phi^{\ast} \F)|_{\ul{\phi}^{-1}(U)}.
\end{equation}

\item \textbf{Producing the local trivialization}: Let $K \in \gVect$ be a finite-dimensional graded vector space. It is straightforward to prove that there is a canonical isomorphism $\phi^{\ast} (\C^{\infty}_{\M}[K]) \cong \C^{\infty}_{\cN}[K]$. Now, let $\Gamma_{\E}$ be a sheaf of sections of a graded vector bundle $\E$. Let $(U,\varphi)$ be a local trivialization chart for $\E$. We thus have the $\C^{\infty}_{\M}|_{U}$-linear sheaf isomorphism 
\begin{equation}
\varphi: \Gamma_{\E}|_{U} \rightarrow \C^{\infty}_{\M}|_{U}[K].
\end{equation}
As $\phi'^{\ast}$ is a functor, we obtain a $\C^{\infty}_{\cN}|_{\ul{\phi}^{-1}(U)}$-linear sheaf isomorphism 
\begin{equation}
\phi'^{\ast}(\varphi): (\phi^{\ast} \Gamma_{\E})_{\ul{\phi}^{-1}(U)} \rightarrow \C^{\infty}_{\cN}|_{\ul{\phi}^{-1}(U)}[K], 
\end{equation}
where we have used (\ref{eq_pullrestriction}) and the above observation. This proves that $\phi^{\ast} \Gamma_{\E}$ is locally freely and finitely generated of a constant graded rank. One can thus define $\Gamma_{\phi^{\ast} \E} := \phi^{\ast} \Gamma_{\E}$. We have shown that any local trivialization chart $(U,\varphi)$ for $\E$ induces a local trivialization chart $(\ul{\phi}^{-1}(U), \phi'^{\ast}(\varphi))$ for $\phi^{\ast} \E$. We have also proved the claim $(i)$. 

Next, let us prove the claim $(iv)$. Let $\sigma \in \Gamma_{\E}(U)$ and consider the section 
\begin{equation}
\sigma^{!}_{P} := 1 \otimes [\sigma]_{\ul{\phi}^{-1}(U)} \in (\phi_{P}^{\ast} \Gamma_{\E})(\ul{\phi}^{-1}(U)).
\end{equation}
Recall that one always has a $\C^{\infty}_{\cN}$-linear presheaf morphism $\eta: \phi^{\ast}_{P} \Gamma_{\E} \rightarrow \phi^{\ast} \Gamma_{\E} \equiv \Sff( \phi^{\ast}_{P} \Gamma_{\E})$. Define the pullback section as
\begin{equation}
\sigma^{!} := \eta_{\ul{\phi}^{-1}(U)}( \sigma^{!}_{P}) \in \Gamma_{\phi^{\ast} \E}(\ul{\phi}^{-1}(U)).
\end{equation}
Next, if $(\Phi_{\lambda})_{\lambda=1}^{m}$ is a local frame for $\E$ over $U$, there is a local trivialization chart $(U,\varphi)$ for $\E$ and a total basis $(\vartheta_{\lambda})_{\lambda=1}^{m}$ for $K$, such that $\Phi_{\lambda} = \varphi_{U}^{-1}( 1 \otimes \vartheta_{\lambda})$. As noted above, $(\ul{\phi}^{-1}(U), \phi'^{\ast}(\varphi))$ is then a local trivialization chart for $\varphi^{\ast} \E$. By unfolding the definitions, one can compute that 
\begin{equation}
\Phi_{\lambda}^{!} = (\phi'^{\ast}(\varphi))^{-1}_{\ul{\phi}^{-1}(U)}( 1 \otimes \vartheta_{\lambda}),
\end{equation} 
for all $\lambda \in \{1,\dots,m\}$. This proves that $(\Phi^{!}_{\lambda})_{\lambda=1}^{m}$ is a local frame for $\phi^{\ast} \E$ over $\ul{\phi}^{-1}(U)$. 

\item \textbf{Pullback of the dual}: We claim that there is a canonical $\C^{\infty}_{\cN}$-linear sheaf isomorphism 
\begin{equation} \label{eq_pullbackofdual}
\nu: \phi^{\ast} \Gamma^{\ast}_{\E} \rightarrow \Gamma^{\ast}_{\phi^{\ast} \E}.
\end{equation}
In other words, there is a canonical vector bundle isomorphism $\nu: \phi^{\ast} \E^{\ast} \rightarrow (\phi^{\ast} \E)^{\ast}$. Thanks to the universality of the sheafification, we only need to construct a $\C^{\infty}_{\cN}$-linear presheaf isomorphism $\nu_{P}: \phi^{\ast}_{P} \Gamma^{\ast}_{\E} \rightarrow \Gamma^{\ast}_{\phi^{\ast}\E}$. One can use the simplified interpretation of the sections of the dual sheaf (see Proposition \ref{tvrz_dualsectionseasy}) to our advantage. 

For each $V \in \Op(N)$, let $g \otimes [\alpha]_{V} \in (\phi^{\ast}_{P} \Gamma^{\ast}_{\E})(V)$, where $g \in \C^{\infty}_{\cN}(V)$ and $\alpha \in \Omega^{1}_{\E}(U)$ for some $U \supseteq \ul{\phi}(V)$. We are to define $(\nu_{P})_{V}(g \otimes [\alpha]_{V}) \in \Omega^{1}_{\phi^{\ast}\E}(V)$. It suffices to define its values on the sections of $\varphi^{\ast} \E$ in the form $\sigma^{!}|_{V}$, where $\sigma \in \Gamma_{\E}(U)$. Set
\begin{equation}
[(\nu_{P})_{V}(g \otimes [\alpha]_{V})]( \sigma^{!}|_{V}) := g \cdot \phi^{\ast}_{U}(\alpha(\sigma))|_{V}. 
\end{equation}
The left-hand side defines an element of $\Omega^{1}_{\phi^{\ast}\E}(V)$, if and only if 
\begin{equation}
[(\nu_{P})_{V}(g \otimes [\alpha]_{V})]( (f\sigma)^{!}|_{V}) = (-1)^{(|g| + |\alpha|)|f|)} \phi^{\ast}_{U}(f)|_{V} \cdot [(\nu_{P})_{V}(g \otimes [\alpha]_{V})]( \sigma^{!}|_{V}),
\end{equation}
for any $f \in \C^{\infty}_{\M}(U)$. As $\phi^{\ast}_{U}$ is a graded algebra morphism and $\alpha \in \Omega^{1}_{\E}(U)$, we obtain 
\begin{equation}
\begin{split}
g \cdot \phi^{\ast}_{U}(\alpha(f \sigma))|_{V} = & \ (-1)^{|\alpha||f|} g \cdot \phi^{\ast}_{U}(f \cdot \alpha(\sigma))|_{V} = (-1)^{|\alpha||f|} g \cdot \phi^{\ast}_{U}(f)|_{V} \cdot \phi^{\ast}_{U}(\alpha(\sigma))|_{V} \\
= & \ (-1)^{(|g|+|\alpha|)|f|} \phi^{\ast}_{U}(f)|_{V} \cdot ( g \cdot \phi^{\ast}_{U}(\alpha(\sigma))|_{V}). 
\end{split}
\end{equation}
Next, one has to prove that $(\nu_{P})_{V}$ is well-defined. As we define it on the generators of the tensor product over $(\ul{\phi}_{P}^{-1} \C^{\infty}_{\M})(V)$, for every $f \in \C^{\infty}_{\M}(U)$ we have to check that 
\begin{equation}
(\nu_{P})_{V}( ([f]_{V} \tr g) \otimes [\alpha]_{V}) = (-1)^{|f||g|} (\nu_{P})_{V}( g \otimes [f \alpha]_{V})
\end{equation}
But for every $\sigma \in \Gamma_{\E}(U)$, we get 
\begin{equation}
\begin{split}
([f]_{V} \tr g) \cdot \phi^{\ast}_{U}( \alpha(\sigma))|_{V} = & \ ( \phi^{\ast}_{U}(f)|_{V} \cdot g) \cdot \phi^{\ast}_{U}(\alpha(\sigma))|_{V} = (-1)^{|f||g|} g \cdot \phi^{\ast}_{U}( f \cdot \alpha(\sigma))|_{V} \\
= & \ (-1)^{|f||g|} g \cdot \phi^{\ast}_{U}( (f\alpha)(\sigma))|_{V}. 
\end{split}
\end{equation}
Moreover, it is easy to check that $(\nu_{P})_{V}$ is $\C^{\infty}_{\cN}(V)$-linear and natural in $V$, hence $\nu_{P}: \phi^{\ast}_{P} \Gamma_{\E}^{\ast} \rightarrow \Gamma^{\ast}_{\phi^{\ast}\E}$ is a $\C^{\infty}_{\cN}$-linear sheaf morphism. By the universal property of the sheafification, it determines a unique $\C^{\infty}_{\cN}$-linear sheaf morphism $\nu: \phi^{\ast} \Gamma^{\ast}_{\E} \rightarrow \Gamma^{\ast}_{\phi^{\ast}\E}$. 

Finally, let $(\Phi_{\lambda})_{\lambda=1}^{m}$ be the local frame for $\E$ over $U$. We have the induced dual local frame $(\Phi^{\lambda})_{\lambda=1}^{m}$ for $\E^{\ast}$ over $U$, and consequently the local frame $((\Phi^{\lambda})^{!})_{\lambda=1}^{m}$ for $\phi^{\ast} \E^{\ast}$ over $\ul{\phi}^{-1}(U)$. On the other hand, there is the local frame $(\Phi^{!}_{\lambda})_{\lambda=1}^{m}$ for $\phi^{\ast} \E$ over $\ul{\phi}^{-1}(U)$, so we may take consider its dual $((\Phi^{!})^{\lambda})_{\lambda=1}^{m}$ for $(\phi^{\ast} \E)^{\ast}$ over $\ul{\phi}^{-1}(U)$. To prove that $\nu$ is a sheaf isomorphism, it suffices to show that for each $\lambda \in \{1,\dots,m\}$, one has 
\begin{equation}
\nu_{\ul{\phi}^{-1}(U)}( (\Phi^{\lambda})^{!}) = (\Phi^{!})^{\lambda}. 
\end{equation}
To see this, observe that for each $\kappa \in \{1,\dots,m\}$, one has 
\begin{equation}
\begin{split}
[\nu_{\ul{\phi}^{-1}(U)}( (\Phi^{\lambda})^{!})]( \Phi_{\kappa}^{!}) = & \ [(\nu_{P})_{\ul{\phi}^{-1}(U)}( 1 \otimes [\Phi^{\lambda}]_{\ul{\phi}^{-1}(U)})](\Phi^{!}_{\kappa}) \\
= & \ \phi^{\ast}_{U}( \Phi^{\lambda}( \Phi_{\kappa})) = \phi^{\ast}_{U}( \delta^{\lambda}_{\kappa}) = \delta^{\lambda}_{\kappa}. 
\end{split}
\end{equation}
This proves the claim and we conclude that $\nu$ is a $\C^{\infty}_{\cN}$-linear sheaf isomorphism. 
\item \textbf{Proving the universal property}: Let us construct a canonical graded vector bundle morphism $\hat{\phi}: \phi^{\ast} \E \rightarrow \E$ over $\phi$. We thus look for a $\C^{\infty}_{\M}$-linear sheaf morphism $\hat{\phi}^{\dagger}: \Gamma^{\ast}_{\E} \rightarrow \phi_{\ast} \Gamma^{\ast}_{\phi^{\ast} \E}$. In (c), we have constructed the element $\nu \in \Sh^{\C^{\infty}_{\cN}}_{0}( \phi^{\ast} \Gamma^{\ast}_{\E}, \Gamma^{\ast}_{\phi^{\ast} \E})$. Let  
\begin{equation}
\hat{\phi}^{\dagger} := \mu_{ \Gamma^{\ast}_{\E}, \Gamma^{\ast}_{\phi^{\ast}\E}}( \nu) \in \Sh_{0}^{\C^{\infty}_{\M}}( \Gamma^{\ast}_{\E}, \phi_{\ast} \Gamma^{\ast}_{\phi^{\ast}\E}),
\end{equation}
that is we just use the adjunction (\ref{eq_pushpulladjunction}). This proves the claim $(ii)$. It remains to prove the universal property $(iii)$. Let $F: \E' \rightarrow \E$ be a given graded vector bundle morphism over $\phi$. We thus have $F = (\phi, F^{\dagger})$, where $F^{\dagger}: \Gamma_{\E}^{\ast} \rightarrow \phi_{\ast} \Gamma_{\E'}^{\ast}$ is a $\C^{\infty}_{\M}$-linear sheaf morphism. As $\nu$ is a sheaf isomorphism, there is a unique $\C^{\infty}_{\cN}$-linear sheaf morphism $\hat{F}^{\dagger}: \Gamma^{\ast}_{\phi^{\ast} \E} \rightarrow \Gamma^{\ast}_{\E'}$ such that 
\begin{equation} \label{eq_FdaggerFhatdaggernu}
\mu_{\Gamma^{\ast}_{\E}, \Gamma^{\ast}_{\E'}}^{-1}(F^{\dagger}) = \hat{F}^{\dagger} \circ \nu. 
\end{equation}
This morphism defines a graded vector bundle morphism $\hat{F} = (\1_{N}, \hat{F}^{\dagger}): \E' \rightarrow \phi^{\ast} \E$. The naturality of the adjunction now gives 
\begin{equation}
F^{\dagger} = \phi_{\ast} \hat{F}^{\dagger} \circ \mu_{ \Gamma^{\ast}_{\E}, \Gamma^{\ast}_{\phi^{\ast}\E}}( \nu) = \phi_{\ast} \hat{F}^{\dagger} \circ \hat{\phi}^{\dagger}. 
\end{equation}
But this proves that $F = \hat{\phi} \circ \hat{F}$, see (\ref{eq_gvbmorfcomp}). Every such $\hat{F}$ must fit into (\ref{eq_FdaggerFhatdaggernu}), hence it is unique. This proves the universal property $(iii)$. 
\item \textbf{Pullback by composition}: Let $\psi: \cS \rightarrow \cN$ be another graded smooth map. It is easy to see that $(\psi \circ \phi)_{\ast} \H = \psi_{\ast}(\phi_{\ast} \H)$ for any $\H \in \Sh^{\C^{\infty}_{\cS}}(S, \gVect)$. It thus follows from (\ref{eq_pushpulladjunction}) that $(\phi \circ \psi)^{\ast}$ and $\psi^{\ast} \circ \phi^{\ast}$ are adjoint to the same functor $(\psi \circ \phi)_{\ast} = \psi_{\ast} \circ \phi_{\ast}$. But this implies that they are canonically naturally isomorphic, see e.g. Chapter IV, Section 1 of \cite{mac2013categories}. In particular, for each $\F \in \Sh^{\C^{\infty}_{\M}}(M,\gVect)$, there is a canonical $\C^{\infty}_{\M}$-linear sheaf isomorphism
\begin{equation}
(\phi \circ \psi)^{\ast}\F \cong \psi^{\ast}(\phi^{\ast} \F).
\end{equation}
For $\F = \Gamma_{\E}$, this implies (\ref{eq_pullbackcomp}) and the claim $(v)$ is verified. 
\end{enumerate}
The proposition is thereby proved. 
\end{proof}
\begin{tvrz} \label{tvrz_ap_pullbacktensor} For any $\F,\G \in \Sh^{\C^{\infty}_{\M}}(M,\gVect)$ and any graded smooth map $\phi: \cN \rightarrow \M$, there is a canonical $\C^{\infty}_{\cN}$-linear isomorphism 
\begin{equation} \label{eq_pullbacktensormain}
\phi^{\ast}( \F \otimes_{\C^{\infty}_{\M}} \G) \cong \phi^{\ast}\F \otimes_{\C^{\infty}_{\cN}} \phi^{\ast}\G,
\end{equation}
natural in $\F$ and $\G$. In particular, if $\F = \Gamma_{\E}$ and $\G = \Gamma_{\E'}$ are sheaves of sections of two graded vector bundles $\E$ and $\E'$ over $\M$, there is a canonical graded vector bundle isomorphism 
\begin{equation}
\phi^{\ast}(\E \otimes \E') \cong \phi^{\ast}\E \otimes \phi^{\ast}\E'.
\end{equation}
\end{tvrz}
\begin{proof}
It is not difficult to prove the following generalization of the adjunction (\ref{eq_pushpulladjunction}): For every $\G \in \Sh^{\C^{\infty}_{\M}}(M,\gVect)$ and $\H \in \Sh^{\C^{\infty}_{\cN}}(N, \gVect)$, there is a canonical $\C^{\infty}_{\M}$-linear sheaf isomorphism 
\begin{equation} \label{eq_uShpulllpush}
\ul{\Sh}^{\C^{\infty}_{\M}}(\G, \phi_{\ast} \H) \cong \phi_{\ast} (\ul{\Sh}^{\C^{\infty}_{\cN}}( \phi^{\ast} \G, \H))
\end{equation}
natural in $\G$ and $\H$, see Proposition \ref{tvrz_uSh} for the definitions. Now, for any $\F,\G \in \Sh^{\C^{\infty}_{\M}}(M, \gVect)$ and any $\H \in \Sh^{\C^{\infty}_{\cN}}(N, \gVect)$, we obtain the following sequence of canonical bijections:
\begin{equation} \label{eq_pulltensorbijection}
\begin{split}
\Sh_{0}^{\C^{\infty}_{\cN}}( \phi^{\ast}(\F \otimes_{\C^{\infty}_{\M}} \G), \H) \cong & \ \Sh_{0}^{\C^{\infty}_{\M}}( \F \otimes_{\C^{\infty}_{\M}} \G, \phi_{\ast} \H) \cong \Sh_{0}^{\C^{\infty}_{\M}}( \F, \ul{\Sh}(\G,\phi_{\ast} \H)) \\
\cong & \ \Sh_{0}^{\C^{\infty}_{\M}}(\F, \phi_{\ast}( \ul{\Sh}^{\C^{\infty}_{\cN}}( \phi^{\ast} \G, \H))) \\
\cong & \ \Sh_{0}^{\C^{\infty}_{\cN}}( \phi^{\ast}\F, \ul{\Sh}^{\C^{\infty}_{\cN}}( \phi^{\ast}\G, \H)) \\
\cong & \ \Sh_{0}^{\C^{\infty}_{\cN}}( \phi^{\ast}\F \otimes_{\C^{\infty}_{\cN}} \phi^{\ast} \G, \H),
\end{split}
\end{equation}
where each one of them is natural in $\F$, $\G$ and $\H$. 

Now, the Yoneda lemma (see e.g. Chapter III, Section 2 of \cite{mac2013categories}) states that the assignment $\cK \mapsto \Sh_{0}^{\C^{\infty}_{\cN}}(\cK,-)$ defines a full and faithful functor from $\Sh^{\C^{\infty}_{\cN}}(N,\gVect)$ to the category of functors from $\Sh^{\C^{\infty}_{\cN}}(N,\gVect)$ to $\Set$. In particular, every bijection $\Sh^{\C^{\infty}_{\cN}}_{0}(\cK,\H) \cong \Sh^{\C^{\infty}_{\cN}}_{0}(\cK', \H)$ natural in $\H$ is induced by the unique $\C^{\infty}_{\cN}$-linear sheaf isomorphism $\cK \cong \cK'$. 

Applying this on the bijection (\ref{eq_pulltensorbijection}), $\cK = \phi^{\ast}(\F \otimes_{\C^{\infty}_{\M}} \G)$ and $\cK' = \phi^{\ast}\F \otimes_{\C^{\infty}_{\cN}} \phi^{\ast} \G$, we obtain the $\C^{\infty}_{\cN}$-linear sheaf isomorphism (\ref{eq_pullbacktensormain}). As (\ref{eq_pulltensorbijection}) is natural in $\F$ and $\G$, the faithfulness assertion of the Yoneda lemma implies that this isomorphism is also natural in $\F$ and $\G$. 
\end{proof}
\begin{tvrz}\label{tvrz_ap_totalspaceunique}
Using the assumptions and notation of Proposition \ref{tvrz_totalspace}, we have the following additional statements regarding the uniqueness of the construction:
\begin{enumerate}[(a)]
\item Suppose $\{ V_{\mu} \}_{\mu \in J}$ is a refinement of the cover $\{ U_{\alpha} \}_{\alpha \in I}$, that is $\{ V_{\mu} \}_{\mu \in J}$ is an at most countable open cover of $M$ and there exists a map $\zeta: J \rightarrow I$, such that $V_{\mu} \subseteq U_{\zeta(\mu)}$. 

Let $\{ (V_{\mu}, \varphi'_{\mu}) \}_{\mu \in J}$ and $\A' = \{ (V_{\mu}, \phi'_{\mu}) \}_{\mu \in J}$ be obtained by restrictions. Suppose $\E' = (E', \C^{\infty}_{\E'})$, $\pi': \E' \rightarrow \M$, and $\varphi'_{\mu}: \E'|_{\V'_{\mu}} \rightarrow \M|_{V_{\mu}} \times K$ are objects constructed using using these data from $\Gamma_{\E}$ as in Subsection \ref{subsec_totalspace}, where we write $\V'_{\mu} := \ul{\pi}'^{-1}(V_{\mu})$. 

Then there exists the unique graded diffeomorphism $\phi: \E \rightarrow \E'$, such that $\ul{\phi}: E \rightarrow E'$ is a vector bundle isomorphism over $\1_{M}$, and $\varphi'_{\mu} \circ \phi|_{\V_{\mu}} = \varphi_{\zeta(\mu)}|_{\V_{\mu}}$ for all $\mu \in J$, where $\V_{\mu} := \ul{\pi}^{-1}(V_{\mu})$. Consequently, one has $\pi' \circ \phi = \pi$. 
\item Let $\{ (U_{\alpha}, \varphi'_{\alpha}) \}_{\alpha \in I}$ and $\A' = \{ (U_{\alpha}, \phi'_{\alpha}) \}_{\alpha \in I}$ be another local trivialization of $\E$ and another graded smooth atlas for $\M$, respectively, defined on the same open cover $\{ U_{\alpha} \}_{\alpha \in I}$ of $M$. 

For each $\alpha \in I$, one can thus consider a $\C^{\infty}_{\M}|_{U_{\alpha}}$-linear sheaf isomorphism 
\begin{equation}
\chi_{\alpha} := \varphi'_{\alpha} \circ \varphi_{\alpha}^{-1}: \C^{\infty}_{\M}|_{U_{\alpha}}[K] \rightarrow \C^{\infty}_{\M}|_{U_{\alpha}}[K].
\end{equation}
It defines a collection of functions $(\chi_{\alpha}){}^{\kappa}_{\lambda} \in \C^{\infty}_{\M}(U_{\alpha})$ by $\chi_{\alpha}(1 \otimes \vartheta_{\lambda}) = (\chi_{\alpha})^{\kappa}{}_{\lambda} \otimes \vartheta_{\kappa}$, where $(\vartheta_{\lambda})_{\lambda=1}^{m}$ is a fixed total basis of $K$. One can use these functions to define a graded diffeomorphism $\hat{\chi}_{\alpha}: \M|_{U_{\alpha}} \times K \rightarrow \M|_{U_{\alpha}} \times K$. Indeed, if $\pi_{1}$ and $\pi_{2}$ are the projections, set $\pi_{1} \circ \hat{\chi}_{\alpha} := \pi_{1}$ and 
\begin{equation}
(\pi_{2} \circ \hat{\chi}_{\alpha})^{\ast}_{K_{0}}( \vartheta^{\lambda}) := (-1)^{|\vartheta_{\lambda}|(|\vartheta_{\kappa}| - |\vartheta_{\lambda}|)} (\pi_{1})^{\ast}_{U_{\alpha}}( (\chi_{\alpha})^{\lambda}{}_{\kappa}) \cdot \vartheta^{\kappa},
\end{equation}
where $(\vartheta^{\lambda})_{\lambda=1}^{m}$ are viewed as global coordinate functions on $K$ corresponding to the total basis of $K^{\ast}$ dual to $(\vartheta_{\lambda})_{\lambda=1}^{m}$. 

Suppose $\E' = (E', \C^{\infty}_{\E'})$, $\pi': \E' \rightarrow \M$, and $\varphi'_{\alpha}: \E'|_{\U'_{\alpha}} \rightarrow \M|_{U_{\alpha}} \times K$ are objects constructed using these data from $\Gamma_{\E}$ as in Subsection \ref{subsec_totalspace}, where we write $\U'_{\alpha} := \ul{\pi}'^{-1}(U_{\alpha})$. 

Then there exists the unique graded diffeomorphism $\phi: \E \rightarrow \E'$, such that $\ul{\phi}: E \rightarrow E'$ is a vector bundle isomorphism over $\1_{M}$, and $\varphi'_{\alpha} \circ \phi|_{\U_{\alpha}} = \hat{\chi}_{\alpha} \circ \varphi_{\alpha}$ for all $\alpha \in I$. Consequently, one has $\pi' \circ \phi = \pi$. 
\end{enumerate}
This shows that the construction of $\E$ does not depend on the particular choice of a local trivialization of $\E$ and a graded smooth atlas $\A$ for $\M$. 
\end{tvrz}
\begin{proof}
We will only briefly sketch the proof. First, recall that the statements $(a)$ and $(b)$ hold for ordinary vector bundles, thus providing the uniqueness claims for Proposition \ref{tvrz_totalspaceE}. This is a standard statement which follows immediately from the way the total space manifold $E$ is actually constructed, see e.g. Chapter II, Section 6 of \cite{kolar2013natural}. In this way, we obtain the unique vector bundle isomorphism $\ul{\phi}: E \rightarrow E'$ satisfying 
\begin{enumerate}[(a)]
\item $\ul{\varphi'_{\mu}} \circ \ul{\phi}|_{\V_{\mu}} = \ul{\varphi_{\zeta(\mu)}}|_{\V_{\mu}}$ for all $\mu \in J$, under the assumptions of the statement (a);
\item $\ul{\varphi'_{\alpha}} \circ \ul{\phi}|_{\U_{\alpha}} = \ul{\hat{\chi}_{\alpha}} \circ \ul{\varphi_{\alpha}}$ for all $\alpha \in I$, under the assumptions of the statement (b); 
\end{enumerate}
The idea is to prove the statements themselves using Proposition \ref{tvrz_gmgluing2}, except that $\E$ and $\E'$ obtained by the gluing procedure of Proposition \ref{tvrz_gmgluing1} do not share the underlying manifold. 

However, one can consider the ``pushforward'' graded manifold $\ul{\phi}_{\ast} \E := (E', \ul{\phi}_{\ast} \C^{\infty}_{\E})$. The uniqueness claim of Proposition \ref{tvrz_gmgluing1} shows that $\ul{\phi_{\ast}} \E$ can be viewed as a graded manifold obtained by the gluing procedure using the atlas $\ul{\phi}_{\ast} \ul{\B} := \{ (\U'_{\alpha}, \ul{\rho_{\alpha}} \circ \ul{\phi}^{-1}|_{\U'_{\alpha}}) \}_{\alpha \in I}$ and \textit{the same} transition data. 

The rest is a straightforward application of Proposition \ref{tvrz_gmgluing2}. Note that every graded diffeomorphism $\eta: \ul{\phi}_{\ast} \E \rightarrow \E'$ over $\1_{E'}$ can be equivalently viewed as a graded diffeomorphism $\phi = (\ul{\phi}, \eta^{\ast})$ from $\E$ to $\E'$, as $\eta^{\ast}: \C^{\infty}_{\E'}\rightarrow \C^{\infty}_{\ul{\phi}_{\ast} \E} \equiv \ul{\phi}_{\ast} \C^{\infty}_{\E}$. This finishes the proof. 
\end{proof}
\subsection{Submanifolds}
\begin{lemma} \label{lem_fdecompasgenerators}
Let $(s_{j})_{j \in \Z}$ and $(n_{j})_{j \in \Z}$ be any sequences of non-negative integers, such that $\sum_{j \in \Z} s_{j} < \infty$ and $\sum_{j \in \Z} n_{j} < \infty$. Let $\hat{U} \subseteq \R^{s_{0}}$ and $\hat{W} \subseteq \R^{n_{0} - s_{0}}$ be two open cubes containing the respective origins. Suppose $\iota_{0}: \hat{U}^{(s_{j})} \rightarrow (\hat{U} \times \hat{W})^{(n_{j})}$ is the ``zero section'', that is 
\begin{align}
(\iota_{0})^{\ast}_{\hat{U} \times \hat{V}}(\bbz^{A_{j}}_{(j)}) = & \ \bbu^{A_{j}}_{(j)}, \text{ for } A_{j} \in \{1, \dots, s_{j} \}, \\
(\iota_{0})^{\ast}_{\hat{U} \times \hat{V}}(\bbz^{A_{j}}_{(j)}) = & \ 0, \text{ for } A_{j} \in \{s_{j}+1, \dots, n_{j} \},
\end{align}
where for each $j \in \Z$, $( \bbz^{A_{j}}_{(j)} )_{A_{j}=1}^{n_{j}}$ and $( \bbu_{(j)}^{A_{j}})_{K_{j}=1}^{s_{j}}$ denote the standard coordinate functions of degree $j$ on $(\hat{U} \times \hat{W})^{(n_{j})}$ and $\hat{U}^{(s_{j})}$, respectively. 

Let $\hat{V} \in \Op(\hat{U} \times \hat{W})$ be arbitrary and consider a function $f \in \C^{\infty}_{(n_{j})}(\hat{V})$, such that $(\iota_{0})^{\ast}_{\hat{V}}(f) = 0$. Then $f$ can be written as a finite sum 
\begin{equation} \label{eq_fdecompasgenerators}
f = \sum_{j \in \Z} \sum_{A_{j}=s_{j}+1}^{n_{j}} f^{A_{j}}_{(j)} \cdot \bbz^{A_{j}}_{(j)},
\end{equation} 
for some functions $f^{A_{j}}_{(j)} \in \C^{\infty}_{(n_{j})}(\hat{V})$.
\end{lemma}
\begin{proof}
First, one can certainly write $f$ as a finite sum 
\begin{equation} \label{eq_fdecompasgenerators0}
f = f_{0} + \sum_{j \neq 0} \sum_{A_{j}=s_{j}+1}^{n_{j}} f^{A_{j}}_{(j)} \cdot \bbz^{A_{j}}_{(j)},
\end{equation}
for some functions $f^{A_{j}}_{(j)} \in \C^{\infty}_{(n_{j})}(\hat{V})$, where $f_{0} \in \C^{\infty}_{(m_{j}+s_{j})}(\hat{V})$ does \emph{not} contain any $\bbz^{A_{j}}_{(j)}$ for $j \neq 0$ and $s_{j} < A_{j} \leq n_{j}$. Consequently, for each $\fp \in \N^{m_{\ast} + s_{\ast}}_{|f|}$, the component $(f_{0})_{\fp} \in \C^{\infty}_{n_{0}}(\hat{V})$ must vanish when pulled back by $\ul{\iota}_{0}: \hat{U} \rightarrow \hat{U} \times \hat{W}$ to $\ul{\iota_{0}}^{-1}(\hat{V}) = \{ u \in \hat{U} \; | \; (u,0) \in \hat{V} \}$.

First, suppose that $v = (u,w) \in \hat{V}$ satisfies $w \neq 0$. We can then find an open cube $\hat{V}_{v} \subseteq \hat{V}$ around $v$, such that $\bbz^{A_{0}}_{(0)}(v') \neq 0$ for all $v' \in \hat{V}_{v}$ and some $A_{0} \in \{s_{0}+1,\dots,n_{0}\}$. Let 
\begin{equation}
(f^{A_{0}}_{(0),v})_{\fp} := \frac{1}{\bbz^{A_{0}}_{(0)}} \cdot (f_{0})_{\fp}|_{\hat{V}_{v}}, \; \; (f^{A'_{0}}_{(0),v})_{\fp} := 0 \text{ for all } A'_{0} \in \{s_{0}+1,\dots, \hat{A}_{0}, \dots, n_{0} \}.
\end{equation}
We can thus write $(f_{0})_{\fp}|_{\hat{V}_{v}} = \sum_{A_{0}=s_{0}+1}^{n_{0}} (f^{A_{0}}_{(0),v})_{\fp} \cdot \bbz^{A_{0}}_{(0)}$ for some smooth functions $(f^{A_{0}}_{(0),v})_{\fp} \in \C^{\infty}_{n_{0}}(\hat{V}_{v})$. 

Next, consider $v = (u,0) \in \hat{V}$. One can find open cubes $\hat{U}_{u} \subseteq \R^{s_{0}}$ around $u$ and $\hat{W}_{0} \subseteq \R^{n_{0}-s_{0}}$ around $0$, so that $\hat{V}_{v} := \hat{U}_{u} \times \hat{W}_{0} \subseteq \hat{V}$. By assumption, we have $(f_{0})_{\fp}(u',0) = 0$ for all $u' \in \hat{U}_{u}$. By using Proposition \ref{lem_Hadimrdclassic} for the functions $(f_{0})_{\fp}(u',\cdot): \hat{W}_{0} \rightarrow \R$, we thus obtain smooth functions $(f^{A_{0}}_{(0),v})_{\fp} \in \C^{\infty}_{m_{0}+s_{0}}( \hat{V}_{v})$, where $A_{0} \in \{s_{0}+1,\dots,n_{0}\}$, such that $(f_{0})_{\fp}|_{\hat{V}_{v}} = \sum_{A_{0}=s_{0}+1}^{n_{0}} (f^{A_{0}}_{(0),v})_{\fp} \cdot \bbz^{A_{0}}_{(0)}$. 

Now, let $\{ \lambda_{v} \}_{v \in \hat{V}}$ be the partition of unity subordinate to the open cover $\{ \hat{V}_{v} \}_{v \in \hat{V}}$ of $\hat{V}$. For each $A_{0} \in \{s_{0}+1,\dots,n_{0}\}$, define $(f^{A_{0}}_{(0)})_{\fp} \in \C^{\infty}_{n_{0}}(\hat{V})$ by $(f^{A_{0}}_{(0)})_{\fp} := \sum_{v \in \hat{V}} \lambda_{v} \cdot (f^{A_{0}}_{(0),v})_{\fp}$. By construction, it follows that $(f_{0})_{\fp}$ can be written as 
\begin{equation}
(f_{0})_{\fp} = \sum_{A_{0}=1}^{m_{0}} (f^{A_{0}}_{(0)})_{\fp} \cdot \bbz^{A_{0}}_{(0)}. 
\end{equation}
For each $A_{0} \in \{s_{0}+1,\dots,n_{0}\}$, declare $(f^{A_{0}}_{(0)})_{\fp}$ to be the component functions of $f^{A_{0}}_{(0)} \in \C^{\infty}_{(n_{j})}(\hat{V})$. Consequently, one finds the expression 
\begin{equation}
f_{0} = \sum_{A_{0}=s_{0}+1}^{n_{0}} f^{A_{0}}_{(0)} \cdot \bbz^{A_{0}}_{(0)}.
\end{equation}
Plugging this into (\ref{eq_fdecompasgenerators0}), we obtain the decomposition (\ref{eq_fdecompasgenerators}). 
\end{proof}
\begin{theorem}\label{thm_ap_idealissubmanifold}
This is the proof of some technical details of the proof of Theorem \ref{thm_idealissubmanifold}. 
\end{theorem}
\begin{proof} Let us divide the proof in several parts. 
\begin{enumerate}[(a)]
\item \textbf{The inverse image sheaf}: Let $\ul{\iota}: S \rightarrow M$ be a continuous map. 

For any $\F \in \Sh(M,\gcAs)$, the inverse image sheaf $\ul{\iota}^{-1} \F \in \Sh(S,\gcAs)$ is defined in the part (a) of the proof of Proposition \ref{tvrz_ap_pullbackVB}. The assignment $\F \mapsto \ul{\iota}^{-1} \F$ defines a functor from $\Sh(M,\gcAs)$ to $\Sh(S,\gcAs)$. For any sheaf morphism $\varphi: \F \rightarrow \G$, define the presheaf morphism $\ul{\iota}_{P}^{-1}(\varphi): \ul{\iota}_{P}^{-1}\F \rightarrow \ul{\iota}_{P}^{-1}\G$ as follows. For each $V \in \Op(S)$ and $[f]_{V} \in (\ul{\iota}_{P}^{-1} \F)(V)$ represented by $f \in \F(U)$ for some $U \supseteq \ul{\iota}(V)$, set 
\begin{equation}
(\ul{\iota}_{P}^{-1}(\varphi))_{V}([f]_{V}) := [ \varphi_{U}(f)]_{V}.
\end{equation}  
It is straightforward to see that this is well-defined graded algebra morphism natural in $V$. The sheaf morphism $\ul{\iota}^{-1}(\varphi): \ul{\iota}^{-1}\F \rightarrow \ul{\iota}^{-1}\G$ is then defined by the universality of the sheafification to be the unique sheaf morphism fitting into the commutative diagram
\begin{equation}
\begin{tikzcd}[column sep=large]
\ul{\iota}^{-1}_{P}\F \arrow{r}{\ul{\iota}_{P}^{-1}(\varphi)} \arrow{d}{\eta_{\ul{\iota}_{P}^{-1}\F}} & \ul{\iota}^{-1}_{P}\G \arrow{d}{\eta_{\ul{\iota}_{P}^{-1}\G}} \\
\ul{\iota}^{-1} \F \arrow[dashed]{r}{\ul{\iota}^{-1}(\varphi)} & \ul{\iota}^{-1}\G
\end{tikzcd},
\end{equation}
where for each $\H \in \PSh(S,\gcAs)$, $\eta_{\H}: \H \rightarrow \Sff(\H)$ is the presheaf morphism induced by the sheafification procedure. 

For each $\F \in \Sh(M,\gcAs)$ and $\G \in \Sh(S,\gcAs)$, the bijection
\begin{equation}
\mu_{\F,\G}: \Sh_{0}( \ul{\iota}^{-1} \F, \G) \rightarrow \Sh_{0}( \F, \ul{\iota}_{\ast} \G) 
\end{equation}
is defined as follows. First, construct a bijection $\mu_{\F,\G}^{P}: \PSh_{0}( \ul{\iota}_{P}^{-1} \F, \G) \rightarrow \Sh_{0}(\F, \ul{\iota}_{\ast} \G)$. Let $\psi: \ul{\iota}_{P}^{-1}\F \rightarrow \G$ be a presheaf morphism. For each $U \in \Op(M)$ and $f \in \F(U)$, set
\begin{equation}
(\mu^{P}_{\F,\G}(\psi))_{U}(f) := \psi_{\ul{\iota}^{-1}(U)}( [f]_{\ul{\iota}^{-1}(U)}) \in \G( \ul{\iota}^{-1}(U)) = (\ul{\iota}_{\ast} \G)(U). 
\end{equation}
Conversely, let $\psi': \F \rightarrow \ul{\iota}_{\ast} \G$ be a sheaf morphism. For each $V \in \Op(S)$ and $[f]_{V} \in (\ul{\iota}_{P}^{-1} \F)(V)$ represented by $f \in \F(U)$ for some $U \supseteq \ul{\iota}(V)$, define
\begin{equation}
((\mu^{P}_{\F,\G})^{-1}(\psi'))_{V}([f]_{V}) := \psi'_{U}(f)|_{V}. 
\end{equation}
It is an easy exercise to check that $\mu^{P}_{\F,\G}$ is natural in both $\F$ and $\G$ and $(\mu^{P}_{\F,\G})^{-1}$ is indeed its inverse. The bijection $\mu_{\F,\G}$ is obtained by composing $\mu^{P}_{\F,\G}$ with the canonical bijection 
\begin{equation}
\Sh_{0}( \ul{\iota}^{-1}\F, \G) \cong \PSh_{0}( \ul{\iota}^{-1}_{P} \F, \G)
\end{equation}
obtained by the universality of the sheafification. It is also natural in both $\F$ and $\G$. 

Finally, observe that for any $s \in S$, there is a canonical identification 
\begin{equation} \label{eq_pullbackstalk}
(\ul{\iota}^{-1} \F)_{s} \cong \F_{\ul{\iota}(s)}.
\end{equation}
It suffices to prove that $(\ul{\iota}_{P}^{-1}\F)_{s} \cong \F_{\ul{\iota}(s)}$. To do so, it suffices to construct a universal colimiting cone $\{ \pi_{V,s} \}_{V \in \Op_{s}(S)}$, where $\pi_{V,s}: (\ul{\iota}_{P}^{-1} \F)(V) \rightarrow \F_{\ul{\iota}(s)}$. Let $[f]_{V} \in (\ul{\iota}_{P}^{-1} \F)(V)$ be represented by $f \in \F(U)$ for some $U \supseteq \ul{\iota}(V)$. As $U \in \Op_{\ul{\iota}(s)}(M)$, we can define
\begin{equation}
\pi_{V,s}([f]_{V}) := [f]_{\ul{\iota}(s)} \in \F_{\ul{\iota}(s)}. 
\end{equation}
The fact that this is a universal colimiting cone is left to the reader as an excercise. 
\item \textbf{Some claims about graded locally ringed spaces}:
First, we will use the following general statement for local graded commutative rings.
\begin{lemma} \label{lem_quotientringislocal}
Let $R \in \gcRng$ be a local graded commutative ring. Let $I \subsetneq R$ be its ideal.

Then the quotient $R / I$ is also a local graded commutative ring, and $\frJ(R/I) = \frJ(R) / I$. 
\end{lemma}
\begin{proof}
This can be proved directly from the definition. Let $q: R \rightarrow R/I$ denote the canonical quotient map. Suppose $\hat{J} \subsetneq R/I$ is a maximal ideal in $R/I$. It is easy to prove that $J := q^{-1}(\hat{J}) \subseteq R$ is a maximal ideal in $R$. Consequently, we have $J = \frJ(R)$ as $R$ is local. Whence $\hat{J} = \frJ(R)/I$ is the unique maximal ideal in $R/I$, and $\frJ(R/I) = \frJ(R)/I$. 
\end{proof}

Now, let $\F \in \Sh(M,\gcAs)$ be any sheaf, and let $\I \subseteq \F$ be a sheaf of ideals if $\F$. We can then define a quotient presheaf $\F / \I$ as in Remark \ref{rem_quotientsheaf}. For each $m \in M$, we may view the stalk $\I_{m}$ as an ideal in $\F_{m}$. Observe that there is a canonical graded algebra isomorphism
\begin{equation} \label{eq_stalkofquotient}
(\F / \I)_{m} \cong \F_{m} / \I_{m}
\end{equation}
Let $\natural: \F \rightarrow \F / \I$ be the canonical presheaf morphism obtained from the quotient maps. For each $m \in M$, let $q_{m}: \F_{m} \rightarrow \F_{m} / \I_{m}$ be the quotient map. For each $U \in \Op(M)$ and $f \in \F(U)$, define $\pi^{\natural}_{U,m}: (\F/\I)(U) \rightarrow \F_{m} / \I_{m}$ by 
\begin{equation}
\pi^{\natural}_{U,m}( \natural_{U}(f)) := q_{m}([f]_{m}).
\end{equation}
It is easy to see that this defines a universal colimiting cone from $\F$ to $\F_{m} / \I_{m}$, which proves the existence of a canonical graded algebra isomorphism (\ref{eq_stalkofquotient}). Moreover, with respect to this identification, the induced stalk map $\natural_{m}: \F_{m} \rightarrow (\F/\I)_{m} \cong \F_{m} / \I_{m}$ corresponds precisely to the quotient map $q_{m}$. 

We can now prove the following observation:
\begin{lemma} Let $\ul{\iota}: S \rightarrow M$ be the inclusion of the subset $S = \{ m \in M \; | \; \I(M) \subseteq \J^{m}_{\M}(M) \}$, see the part (a) of the proof of Theorem \ref{thm_idealissubmanifold}. Let $\C^{\infty}_{\cS} := \ul{\iota}^{-1}( \C^{\infty}_{\M} / \I)$. 

Then $(S,\C^{\infty}_{\cS})$ is a graded locally ringed space. For any $s \in S$, one has $\frJ(\C^{\infty}_{S,s}) \cong \frJ(\C^{\infty}_{\M,s}) / \I_{s}$.
\end{lemma}
\begin{proof}
It follows from (\ref{eq_pullbackstalk}) and (\ref{eq_stalkofquotient}) that for each $s \in S$, there is a canonical isomorphism $\C^{\infty}_{\cS,s} \cong \C^{\infty}_{\M,s} / \I_{s}$, where we identify $s \equiv \ul{\iota}(s)$. It follows from Lemma \ref{lem_quotientringislocal} that $\C^{\infty}_{\cS,s}$ is a local graded ring, if we argue that $\I_{s} \neq \C^{\infty}_{\M,s}$. But every $[f]_{s} \in \I_{s}$ must satisfy $f(s) = 0$ as $\I(M) \subseteq \J^{s}_{\M}(M)$. In particular, $\I_{s}$ cannot contain $[1]_{s}$. Moreover, under the above isomorphism, one has $\frJ(\C^{\infty}_{\M,s}) \cong \frJ(\C^{\infty}_{\M,s}) / \I_{s}$. This finishes the proof.
\end{proof}

Now, we will use the following general observation:
\begin{lemma} \label{lem_adjunctarelocal}
Let $\ul{\iota}: S \rightarrow M$ be any continuous map. Suppose $(M,\F)$ and $(S,\G)$ are graded locally ringed spaces. Let $\psi: \ul{\iota}^{-1}\F \rightarrow \G$ be any sheaf morphism. 

Then the pair $\iota := (\ul{\iota}, \mu_{\F,\G}(\psi))$ defines a morphism of graded locally ringed spaces $(S,\G)$ and $(M,\F)$, iff for each $s \in S$, the induced stalk map $\psi_{s}: (\ul{\iota}^{-1}\F)_{s} \cong \F_{\ul{\iota}(s)} \rightarrow \G_{s}$ is a local graded ring morphism.
\end{lemma}
\begin{proof}
It suffices to prove the statement for the inverse image presheaf. Let $\psi: \ul{\iota}_{P}^{-1}\F \rightarrow \G$ be a presheaf morphism. For each $s \in S$, let us look on the induced stalk map $\psi_{s}: \F_{\ul{\iota}(s)} \rightarrow \G_{s}$. Let $[f]_{\ul{\iota}(s)} \in \F_{\ul{\iota}(s)}$ be represented by $f \in \F(U)$ for some $U \in \Op_{\ul{\iota}(s)}(M)$. On the other hand, $[f]_{\ul{\iota}(s)}$ is a germ of the section $[f]_{\ul{\iota}^{-1}(U)} \in (\ul{\iota}_{P}^{-1} \F)(\ul{\iota}^{-1}(U))$. Hence
\begin{equation}
\psi_{s}([f]_{\ul{\iota}(s)}) = [\psi_{\ul{\iota}^{-1}(U)}( [f]_{\ul{\iota}^{-1}(U)})]_{s}.
\end{equation}
Next, $\mu_{\F,\G}(\psi): \F \rightarrow \ul{\iota}_{\ast} \G$ is a sheaf morphism, so $\iota := (\ul{\iota}, \mu_{\F,\G}(\psi))$ is a morphism of graded locally ringed spaces $(S,\G)$ and $(M,\F)$, if it has the property (iii) of Definition \ref{def_gLRS}. In other words, $\iota_{(s)}: \F_{\ul{\iota}(s)} \rightarrow \G_{s}$ defined by (\ref{eq_vartphixmap2}) is a local graded ring morphism. Explicitly, one has 
\begin{equation}
\iota_{(s)}( [f]_{\ul{\iota}(s)}) = [ (\mu_{\F,\G}(\psi))_{U}(f)]_{s} = [\psi_{\ul{\iota}^{-1}(U)}( [f]_{\ul{\iota}^{-1}(U)})]_{s},
\end{equation}
for all $[f]_{\ul{\iota}(s)} \in \F_{\ul{\iota}(s)}$ represented $f \in \F(U)$ for some $U \in \Op_{\ul(\iota(s)}(M)$. But we have just proved that $\psi_{s} = \iota_{(s)}$ and the statement follows immediately. 
\end{proof}
We can now apply this lemma to $\iota = (\ul{\iota}, \iota^{\ast})$, where $\iota^{\ast}: \C^{\infty}_{\M} \rightarrow \ul{\iota}_{\ast} \C^{\infty}_{\cS}$ is defined by (\ref{eq_iotaast}). For each $s \in S$, with respect to the identification (\ref{eq_pullbackstalk}), the induced stalk map $(\ul{\iota}^{-1}(\natural))_{s}$ corresponds to the quotient map $\natural_{s}: \C^{\infty}_{\M,s} \rightarrow \C^{\infty}_{\M,s} / \I_{s}$. By Lemma \ref{lem_quotientringislocal}, this is a morphism of local graded rings. Consequently, $\iota = (\ul{\iota}, \iota^{\ast}) = (\ul{\iota}, \mu_{\C^{\infty}_{\M},\C^{\infty}_{\cS}}( \ul{\iota}^{-1}(\natural)))$ is a morphism of graded locally ringed spaces by Lemma \ref{lem_adjunctarelocal}. 

\item \textbf{The proof of Lemma \ref{lem_factorthroughI}}: Let $\cN = (N, \O_{N})$ be any graded locally ringed space. Suppose $\chi_{0}: \cN \rightarrow \M|_{U}$ is a morphism of graded locally ringed spaces, such that $\ul{\chi_{0}}(N) \subseteq U \cap S$ and $\chi_{0}^{\ast}(\I|_{U}) = 0$. We look for $\chi: \cN \rightarrow \cS|_{U \cap S}$, such that $\iota|_{U \cap S} \circ \chi = \chi_{0}$. 

Without the loss of generality assume that $U = M$, so we do not have to explicitly write the restrictions. As $S$ is an embedded submanifold of $M$, the assumption $\ul{\chi_{0}}(N) \subseteq S$ ensures that there is a continuous map $\ul{\chi}: N \rightarrow S$, such that $\ul{\iota} \circ \ul{\chi} = \ul{\chi_{0}}$. The second assumption ensures that there is a unique sheaf morphism $\hat{\chi}_{0}^{\ast}: \C^{\infty}_{\M} / \I \rightarrow \ul{\chi_{0}}_{\ast} \C^{\infty}_{\cN}$ making the diagram
\begin{equation}
\begin{tikzcd}
\C^{\infty}_{\M} \arrow{d}{\natural} \arrow{r}{\chi_{0}^{\ast}}  & \ul{\chi_{0}}_{\ast} \C^{\infty}_{\cN} \\
\C^{\infty}_{\M} / \I \arrow[dashed]{ur}[swap]{\hat{\chi}_{0}^{\ast}} & 
\end{tikzcd}
\end{equation}
commutative. Now, suppose that such $\chi$ exists. It thus has to satisfy 
\begin{equation}
\begin{split}
\chi_{0}^{\ast} = & \ \ul{\iota}_{\ast}( \chi^{\ast}) \circ \iota^{\ast} = \ul{\iota}_{\ast}(\chi^{\ast}) \circ \mu_{\C^{\infty}_{\M}, \C^{\infty}_{\cS}}( \ul{\iota}^{-1}(\natural)) = \mu_{\C^{\infty}_{\M}, \ul{\chi}_{\ast} \C^{\infty}_{\cN}}( \chi^{\ast} \circ \ul{\iota}^{-1}(\natural)) \\
= & \ \mu_{\C^{\infty}_{\M}/\I, \ul{\chi}_{\ast} \C^{\infty}_{\cN}}( \chi^{\ast}) \circ \natural,
\end{split}
\end{equation}
where we have used the naturality of the adjunction twice. This proves that necessarily 
\begin{equation}
\mu_{ \C^{\infty}_{\M}/\I, \ul{\chi}_{\ast} \C^{\infty}_{\cN}}( \chi^{\ast}) = \hat{\chi}_{0}^{\ast}. 
\end{equation}
This shows that if $\chi$ exists, it is unique. Moreover, we may use this equation to \textit{define} $\chi^{\ast}$. Note that this makes sense, as we already have the continuous map $\ul{\chi}: N \rightarrow S$ and $\hat{\chi}_{0}^{\ast}: \C^{\infty}_{\M} / \I \rightarrow \ul{\chi_{0}}_{\ast} \C^{\infty}_{\cN} = \ul{\iota}_{\ast}( \ul{\chi}_{\ast} \C^{\infty}_{\cN})$. To finish the proof, we only have to argue that $\chi$ is a morphism of graded locally ringed spaces. For each $n \in N$, we have the canonical identification \begin{equation}
\C^{\infty}_{\cS,\ul{\chi}(n)} \cong \C^{\infty}_{\M,\ul{\chi_{0}}(n)} / \I_{\ul{\chi_{0}}(n)}.
\end{equation}
After some effort, one can show that with respect to this identification, the induced map $\chi_{(n)}: \C^{\infty}_{\cS,\ul{\chi}(n)} \rightarrow \C^{\infty}_{\cN,n}$ looks as follows. For each $U \in \Op_{\ul{\chi_{0}}(n)}(M)$ and $f \in \C^{\infty}_{\M}(U)$, one has 
\begin{equation}
\chi_{(n)}( \natural_{\ul{\chi_{0}}(n)}([f]_{\ul{\chi_{0}}(n)})) = (\chi_{0})_{(n)}( [f]_{\ul{\chi_{0}}(n)}). 
\end{equation}
Since $\frJ(\C^{\infty}_{\cS, \ul{\chi}(n)}) \cong \frJ( \C^{\infty}_{\M,\ul{\chi_{0}}(n)}) / \I_{\ul{\chi_{0}}(n)}$, we see that $\chi_{(n)}$ is a morphism of local graded rings due to the assumption that $\chi_{0}$ is a morphism of graded locally ringed spaces. 
\end{enumerate}
The proof is now finished. 
\end{proof}
\begin{theorem} \label{thm_ap_inverseimage}
This is the proof of some technical details of the proof of Theorem \ref{thm_inverseimage}. 
\end{theorem}
\begin{proof}
We have to prove the equation (\ref{eq_stalkinverseideal}). To do so, let us first examine stalks of the pullback presheaf $\phi^{\ast}_{P} \J_{\cS}$. We claim that for each $n \in N$, one has 
\begin{equation} \label{eq_stalkpullbacksheafofJS}
(\phi^{\ast}_{P} \J_{\cS})_{n} \cong \C^{\infty}_{\cN,n} \otimes_{\C^{\infty}_{\M, \ul{\phi}(n)}} (\J_{\cS})_{\ul{\phi}(n)},
\end{equation}
where the action of $\C^{\infty}_{\M,\ul{\phi}(n)}$ on $\C^{\infty}_{\cN,n}$ is induced by the graded algebra morphism $\phi_{(n)}$. To prove this, one simply constructs a universal colimiting cone $\{ \pi^{\otimes}_{V,n} \}_{V \in \Op_{n}(N)}$, where 
\begin{equation}
\pi_{V,n}^{\otimes}: (\phi^{\ast}_{P} \J_{\cS})(V) \rightarrow \C^{\infty}_{\cN,n} \otimes_{\C^{\infty}_{\M, \ul{\phi}(n)}} (\J_{\cS})_{\ul{\phi}(n)},
\end{equation}
for each $V \in \Op_{n}(N)$. Let $g \in \C^{\infty}_{\cN}(V)$ and let $[f]_{V} \in (\ul{\phi}_{P}^{-1} \J_{\cS})(V)$ be represented by $f \in \J_{\cS}(U)$ for some $U \supseteq \ul{\phi}(V)$. Note that $U \in \Op_{\ul{\phi}(n)}(M)$ and one can define 
\begin{equation}
\pi^{\otimes}_{V,n}(g \otimes [f]_{V}) := [g]_{n} \otimes [f]_{\ul{\phi}(n)}.
\end{equation}
It is easy to see that this is well-defined and defines a universal colimiting cone, see the proof of Proposition \ref{tvrz_stalk}. With respect to the identification, we may examine the induced map $(\phi^{\ast}_{P}(j))_{n}$ of stalks. Using (\ref{eq_phiPjmap}), one finds
\begin{equation}
(\phi^{\ast}_{P}(j))_{n}( [g]_{n} \otimes [f]_{\ul{\phi}(n)}) = [ (\phi^{\ast}_{P}(j))_{V}( g \otimes [f]_{V}) ]_{n} = [ g \cdot \phi^{\ast}_{U}(f)|_{V}]_{n} = [g]_{n} \cdot \phi_{(n)}( [f]_{\ul{\phi}(n)}), 
\end{equation}
for all $[g]_{n} \in \C^{\infty}_{\cN,n}$ and $[f]_{\ul{\phi}(n)} \in (\J_{\cS})_{\ul{\phi}(n)}$, represented by $g \in \C^{\infty}_{\cN}(V)$ for some $V \in \Op_{n}(N)$ and $f \in C^{\infty}_{\M}(U)$ for some $U \supseteq \ul{\phi}(V)$. But this proves that 
\begin{equation}
\I'_{n} = \im( \phi^{\ast}(j))_{n} = \im( (\phi^{\ast}(j))_{n}) = \im( (\phi^{\ast}_{P}(j))_{n}) = \< \phi_{(n)}( (\J_{\cS})_{\ul{\phi}(n)}) \>. 
\end{equation}
This is precisely the equation (\ref{eq_stalkinverseideal}) we were to prove. 
\end{proof}

\begin{theorem} \label{thm_ap_inverseimagegen}
This is the proof of Theorem \ref{thm_inverseimagegen}. 
\end{theorem}
\begin{proof}
Let $(\cS,\iota)$ be a given embedded submanifold of $\M$. Let $\phi: \cN \rightarrow \M$ be a graded smooth map transversal to $\cS$. We divide the proof into several parts. To somewhat simplify the notation, we shall write $N' := \ul{\phi}^{-1}(S)$ and $\cN' := \phi^{-1}(\cS)$. 
\begin{enumerate}[a)]
\item \textbf{Embedded submanifolds as level sets}: We start by proving the following lemma.
\begin{lemma} \label{lem_embeddedislocallylevel}
Let $(\cS, \iota)$ be an embedded submanifold. Let $(s_{j})_{j \in \Z} := \gdim(\cS)$.

Then for any $s \in S$, there exists $U \in \Op_{s}(M)$ and a graded smooth submersion $\psi: \M|_{U} \rightarrow \hat{W}^{(n_{j} - s_{j})}$ over a surjective map $\ul{\psi}: U \rightarrow \hat{W} \subseteq \R^{n_{0} - s_{0}}$, such that $(\cS|_{U \cap S}, \iota|_{U \cap S})$ is a closed embedded submanifold of $\M|_{U}$ equivalent to the regular level set submanifold $(\psi^{-1}(0), \iota')$.
\end{lemma}
\begin{proof}
Let $s \in S$ be fixed. As in the proof of Proposition \ref{tvrz_JSprops}-$(ii)$, we may find a graded local chart $(U,\varphi)$ inducing for each $j \in \Z$ the coordinate functions $(\bbz^{A_{j}}_{(j)})_{A_{j}=1}^{n_{j}}$, such that 
\begin{enumerate}[(i)]
\item $\ul{\varphi}(U) = \hat{V} \times \hat{W}$ for open cubes $\hat{V} \subseteq \R^{s_{0}}$ and $\hat{W} \subseteq \R^{n_{0}-s_{0}}$, and $\ul{\varphi}(s) = (0,0)$;
\item $\ul{\varphi}$ is an $s_{0}$-slice chart for $S$, that is $\ul{\varphi}(U \cap S) = \ul{\varphi}(U) \cap (\R^{s_{0}} \times \{0\})$;
\item the graded set $\{ \{ \bbz^{A_{j}}_{(j)}|_{V} \}_{A_{j}=s_{j}+1}^{n_{j}} \}_{j \in \Z}$ generates $\J_{\cS}(V)$ for all $V \in \Op(U)$. 
\end{enumerate}
First, $U \cap S$ is homeomorphic to a $s_{0}$-slice, hence a closed subset of $U$. Consequently, the restriction  $(\cS|_{U \cap S}, \iota|_{U \cap S})$ is a closed embedded submanifold of $\M|_{U}$. 

Next, consider the graded domain $\hat{W}^{(n_{j}-s_{j})}$ with standard coordinate functions $(\bby^{K_{j}}_{(j)})_{K_{j}=1}^{n_{j} - s_{j}}$ and define $\psi: \M|_{U} \rightarrow \hat{W}^{(n_{j}-s_{j})}$ by declaring 
\begin{equation}
\psi^{\ast}_{\hat{W}}( \bby^{K_{j}}_{(j)}) := \bbz^{K_{j}+s_{j}}_{(j)},
\end{equation}
for all $j \in \Z$ and $K_{j} \in \{1,\dots,n_{j}-s_{j}\}$. By Proposition \ref{thm_globaldomain}, this uniquely determines a graded smooth map $\psi$, and for all $m \in U$, one has 
\begin{equation} \ul{\psi}(m) = ( \bbz^{s_{0}+1}_{(0)}(m), \dots, \bbz^{n_{0}}_{(0)}(m)). \end{equation}
Clearly $\psi$ is a submersion, see Theorem \ref{thm_submersion}. Moreover, $\ul{\psi}$ is surjective, as it is just a bijection $\ul{\varphi}: U \rightarrow \hat{V} \times \hat{W}$ composed by the projection $\hat{V} \times \hat{W} \rightarrow \hat{W}$. 

It remains to prove that $(\cS|_{U \cap S}, \iota|_{U \cap S})$ is equivalent to the regular level set submanifold $(\psi^{-1}(0), \iota')$. By Lemma \ref{lem_equivalentsubmafolds}, it suffices to prove that $\J_{\cS|_{U \cap S}} = \I'$, where $\I' \subseteq \C^{\infty}_{\M|_{U}}$ is the sheaf of regular ideals corresponding to $(\psi^{-1}( 0), \iota')$. Clearly $\J_{\cS|_{U \cap S}} = \J_{\cS}|_{U}$. 

Recall that $\I' = \im( \psi^{\ast}(j))$, where $j: \J^{0}_{(n_{j}-s_{j})} \rightarrow \C^{\infty}_{(n_{j}-s_{j})}$ is the inclusion of the sheaf of functions vanishing at $0 \in \hat{W}$. Observe that the finite graded set $\G = \{ \{ \bby^{K_{j}}_{(j)} \}_{K_{j}=1}^{n_{j} - s_{j}} \}_{j \in \Z}$ has the property that $\G|_{\hat{X}}$ generates $\J^{0}_{(n_{j} - s_{j})}( \hat{X})$ for all $\hat{X} \in \Op(\hat{W})$. 

As noted in the part (b) of the proof of Theorem \ref{thm_inverseimage}, this implies that the finite graded set 
\begin{equation} \G'|_{V} = \{ \{ \psi^{\ast}_{\hat{W}}( \bby^{K_{j}}_{(j)})|_{V} \}_{K_{j}=1}^{n_{j}-s_{j}} \}_{j \in \Z} = \{ \{ \bbz^{A_{j}}_{(j)}|_{V} \}_{A_{j}=s_{j}+1}^{n_{j}} \}_{j \in \Z}
\end{equation}
 generates the ideal $\I'(V)$ for any $V \in \Op(U)$. But by the property (iii) of the graded local chart $(U,\varphi)$ we have started with, this proves that $\J_{\cS}(V) = \I'(V)$ for all $V \in \Op(U)$. Hence $\J_{\cS}|_{U} = \I'$ and the proof of the lemma is finished. 
\end{proof}
\item \textbf{Constructing the pieces}: Now, let $\phi: \cN \rightarrow \M$ be a graded smooth map transversal to an embedded submanifold $(\cS,\iota)$. Let $(n_{j})_{j \in \Z} := \gdim(\M)$, $(m_{j})_{j \in \Z} := \gdim(\cN)$ and $(s_{j})_{j \in \Z} := \gdim(\cS)$. For each $n \in N'$, we get $U_{(n)} \in \Op_{\ul{\phi}(n)}(M)$ and a graded smooth submersion 
\begin{equation}
\psi_{(n)}: \M|_{U_{(n)}} \rightarrow \hat{W}_{(n)}^{(n_{j}-s_{j})},
\end{equation}
having the properties as in Lemma \ref{lem_embeddedislocallylevel}. Let $V_{(n)} := \ul{\phi}^{-1}(U_{(n)}) \in \Op_{n}(N)$ and consider a graded smooth map $\psi'_{(n)} := \psi_{(n)} \circ \phi|_{V_{(n)}}: \cN|_{V_{(n)}} \rightarrow \hat{W}^{(n_{j}-s_{j})}_{(n)}$. 

\begin{lemma}
$0 \in \hat{W}^{(n_{j}-s_{j})}_{(n)}$ is a regular value of $\psi'_{(n)}$.
\end{lemma}
\begin{proof}
First, observe that $\ul{ \psi'_{(n)}}^{-1}(0) = V_{(n)} \cap N'$. Let $v \in T_{0} \hat{W}^{(n_{j}-s_{j})}$. Let $n' \in V_{(n)} \cap N'$. As $\phi \pitchfork \cS$, we have $T_{\ul{\phi}(n')} \M = (T_{n'}\phi)( T_{n'}\cN) + T_{\ul{\phi}(n')} \cS$. Since $0$ is a regular value of $\psi_{(n)}$ and $\ul{\psi_{(n)}}^{-1}(0) = U_{(n)} \cap S$, for any $s \in U_{(n)} \cap S$, there exists $x \in T_{s} \M$, such that $v = (T_{s} \psi_{(n)})(x)$. In particular, this is true for $s = \ul{\phi}(n')$. We can decompose $x$ as 
\begin{equation}
x = (T_{n'}\phi)(y) + z, 
\end{equation}
where $y \in T_{n'}\cN$ and $z \in T_{\ul{\phi}(n')} \cS$. But $\cS|_{U_{(n)} \cap S}$ is equivalent to the regular level set submanifold $\psi_{(n)}^{-1}(0)$. In particular, it must have the same tangent space at each its point. This follows from Lemma \ref{lem_equivalentsubmafolds} and Proposition \ref{tvrz_tangenttosubmafold}. By Proposition \ref{tvrz_levelset}, we thus have $T_{\ul{\phi}(n')} \cS = \ker( T_{\ul{\phi}(n')} \psi_{(n)})$. It follows that we can write $v$ as 
\begin{equation}
v = (T_{\ul{\phi}(n')} \psi_{(n)})(x) = (T_{\ul{\phi}(n')} \psi_{(n)})((T_{n'} \phi)(y) + z) = (T_{n'} \psi'_{(n)})(y). 
\end{equation}
This shows that $\psi'_{(n)}$ indeed is a submersion for all $n' \in \ul{\psi'_{(n)}}^{-1}(0)$. 
\end{proof} 

Consequently, for each $n \in N'$, we may define a closed embedded submanifold $(\cN'_{(n)}, \iota'_{(n)})$ of $\cN|_{V_{(n)}}$ as $\cN'_{(n)} := \psi'^{-1}_{(n)}(0)$. Its underlying submanifold is $N'_{(n)} := \ul{\psi'_{(n)}}^{-1}(0)  = V_{(n)} \cap N'$. Let $\J'_{(n)} \subseteq \C^{\infty}_{\cN}|_{V_{(n)}}$ denote its corresponding sheaf of regular ideals. 
\item \textbf{Collating the pieces}: Observe that $N' \subseteq N$ is a second countable Hausdorff topological space, and $\{ N'_{(n)} \}_{n \in N'}$ is its open cover. We have constructed a collection $\{ \cN'_{(n)} \}_{n \in N'}$ of graded manifolds of the same graded dimension $\gdim(\cN'_{(n)}) = (s_{j} + m_{j} - n_{j})_{j \in \Z}$, see Proposition \ref{tvrz_levelset}. Write $N'_{(nn')} := N'_{(n)} \cap N'_{(n')}$, $V_{(nn')} := V_{(n)} \cap V_{(n')}$ and $U_{(nn')} := U_{(n)} \cap U_{(n')}$. 

We will now construct a collection of graded diffeomorphisms $\{ \phi'_{(nn')} \}_{(n,n') \in N'^{2}}$, where 
\begin{equation} \label{eq_phinnprime}
\phi'_{(nn')}: \cN'_{(n')}|_{N'_{(nn')}} \rightarrow \cN'_{(n)}|_{N'_{(nn')}},
\end{equation}
such that $\ul{\phi'_{(nn')}} = \1_{N'_{(nn')}}$, and for all $(n,n',n'') \in N'^{3}$, there holds the cocycle condition 
\begin{equation} \label{eq_phinnprimecocycle}
\phi'_{(nn'')} = \phi'_{(nn')} \circ \phi'_{(n' n'')},
\end{equation}
where all maps are assumed to be restricted to the open subset $N'_{(nn'n'')} := N'_{(n)} \cap N'_{(n')} \cap N'_{(n'')}$. The graded manifold $\cN'$ will be then constructed using Proposition \ref{tvrz_gmcollation1}. 

For any $(n,n') \in N'^{2}$, we have a pair of closed embedded submanifolds $(\cN'_{(n')}|_{N'_{(nn')}}, \iota'_{(n')}|_{N'_{(nn')}})$ and $(\cN'_{(n)}|_{N'_{(nn')}}, \iota'_{(n)}|_{N'_{(nn')}})$ of $\cN|_{V_{(nn')}}$. We claim that they are equivalent, that is there is a unique graded smooth diffeomorphism (\ref{eq_phinnprime}), such that 
\begin{equation} \label{eq_phinnprimeequation}
\iota'_{(n)}|_{N'_{(nn')}} \circ \phi'_{(nn')} = \iota'_{(n')}|_{N'_{(nn')}}. 
\end{equation}
Thanks to Lemma \ref{lem_equivalentsubmafolds}, it suffices to prove the following statement:

\begin{lemma} For each $(n,n') \in N'^{2}$, one has $\J'_{(n)}|_{V_{(nn')}} = \J'_{(n')}|_{V_{(nn')}}$. \end{lemma}
\begin{proof}
For each $j \in \Z$, let $\{ \bby^{K_{j}}_{(j)} \}_{K_{j}=1}^{n_{j} - s_{j}}$ be the standard coordinate functions in $\hat{W}^{n_{j} - s_{j}}_{(n)}$ of degree $j$. As noted in the part (a) of this proof, this implies that for any $W \in \Op(V_{(nn')})$, the ideal  $\J'_{(n)}(W)$ is generated by a finite graded set $\{ \{ (\psi'_{(n)})^{\ast}_{\hat{W}_{(n)}}( \bby^{K_{j}}_{(j)})|_{W} \}_{K_{j}=1}^{n_{j}-s_{j}} \}_{j \in \Z}$. 

On the other hand, for any function $f \in \J^{0}_{(n_{j}-s_{j})}(\hat{W}_{(n)})$, one has 
\begin{equation}
(\psi'_{(n)})^{\ast}_{\hat{W}_{(n)}}(f)|_{W} = \phi^{\ast}_{U_{(n)}}( (\psi_{(n)})^{\ast}_{\hat{W}_{(n)}}(f))|_{W} = \phi^{\ast}_{U_{(nn')}}( (\psi_{(n)})^{\ast}_{\hat{W}_{(n)}}(f)|_{U_{(nn')}})|_{W}.
\end{equation}
It follows from the construction that $(\psi_{(n)})^{\ast}_{\hat{W}_{(n)}}(f)|_{U_{(nn')}} \in \J_{\cS}(U_{(nn')})$. But this ideal is also generated by the finite graded set $\{ \{ (\psi_{(n')})^{\ast}_{\hat{W}_{(n')}}( \bby^{K_{j}}_{(j)})|_{U_{(nn')}} \}_{K_{j}=1}^{n_{j}-s_{j}} \}_{j \in \Z}$. Whence
\begin{equation}
(\psi_{(n)})^{\ast}_{\hat{W}_{(n)}}(f)|_{U_{(nn')}} = \sum_{j \in \Z} \sum_{L_{j}=1}^{n_{j}-s_{j}} h^{L_{j}}_{(j)} \cdot (\psi_{(n')})^{\ast}_{\hat{W}_{(n')}}( \bby^{L_{j}}_{(j)})|_{U_{(nn')}},
\end{equation}
for some functions $h^{L_{j}}_{(j)} \in \C^{\infty}_{\M}( U_{(nn')})$. Consequently, one has 
\begin{equation}
(\psi'_{(n)})^{\ast}_{\hat{W}_{(n)}}(f)|_{W} = \sum_{j \in \Z} \sum_{L_{j}=1}^{n_{j}-s_{j}} \phi^{\ast}_{U_{(nn')}}( h^{L_{j}}_{(j)})|_{W} \cdot (\psi'_{(n')})^{\ast}_{\hat{W}_{(n')}}( \bby^{L_{j}}_{(j)})|_{W} \in \J'_{(n')}(W).
\end{equation}
If we choose $f := \bby^{K_{j}}_{(j)}$ for each $j \in \Z$ and $K_{j} \in \{1,\dots,n_{j}-s_{j} \}$, this proves that any element of the finite generating set of $\J'_{(n)}(W)$ lies in $\J'_{(n')}(W)$. Hence $\J'_{(n)}(W) \subseteq \J'_{(n')}(W)$. By replacing the role of $n$ and $n'$, we can prove the converse inclusion. As $W \in \Op(V_{(nn')})$ was arbitrary, the statement of the lemma follows. 
\end{proof}
This proves that the both submanifolds are equivalent and we obtain the unique graded smooth diffeomorphism (\ref{eq_phinnprime}) satisfying (\ref{eq_phinnprimeequation}). Since both $\ul{\iota_{(n)}}: N'_{(n)} \rightarrow V_{(n)}$ and $\ul{\iota_{(n')}}: N'_{(n')} \rightarrow V_{(n')}$ are just the inclusions of subsets, it follows that $\ul{\phi_{(nn')}} = \1_{V_{(nn')}}$. Finally, the cocycle condition (\ref{eq_phinnprimecocycle}) follows immediately from the uniqueness assertion in the definition of equivalent embedded submanifolds. 

By Proposition \ref{tvrz_gmcollation1}, there is thus a graded smooth manifold $\cN' = (N',\C^{\infty}_{\cN'})$ together with a collection $\{ \lambda_{(n)} \}_{n \in N'}$, where each $\lambda_{(n)}: \cN'_{(n)} \rightarrow \cN'|_{N'_{(n)}}$ is a graded diffeomorphism such that $\ul{\lambda_{(n)}} = \1_{N'_{(n)}}$ and for all $(n,n') \in N'^{2}$, one has $\lambda_{(n)}|_{N'_{(nn')}} \circ \phi'_{(nn')} = \lambda_{(n')}|_{N'_{(nn')}}$. 

\item \textbf{Embedding and the pullback diagram}: 
Let us now construct the embedding $\iota': \cN' \rightarrow \cN$. Recall that for each $n \in N'$, we have constructed the embedding $\iota'_{(n)}: \cN'_{(n)} \rightarrow \cN|_{V_{(n)}}$. It can be viewed as a graded smooth map $\iota'_{(n)}: \cN'_{(n)} \rightarrow \cN$. Define a graded smooth map
\begin{equation}
\iota'_{[n]} := \iota'_{(n)} \circ \lambda_{(n)}^{-1}:  \cN'|_{N'_{(n)}} \rightarrow \cN
\end{equation}
For each $(n,n') \in N'^{2}$, one can use (\ref{eq_phinnprimeequation}) to show that $\iota'_{[n]}|_{N'_{(nn')}} = \iota'_{[n']}|_{N'_{(nn')}}$. By Proposition \ref{tvrz_gLRSgluing}, there thus exists a unique graded smooth map $\iota': \cN' \rightarrow \cN$, such that $\iota'|_{N'_{(n)}} = \iota'_{[n]}$. Observe that $\ul{\iota}': N' \rightarrow N$ is just the inclusion of a subset $N' = \ul{\phi}^{-1}(S)$. It is obvious that $\iota'$ is an immersion. Moreover, for each $n \in N'$, we have $V_{(n)} \in \Op_{n}(N)$, such that $V_{(n)} \cap N'$ is a regular level set of a submersion $\ul{\psi'_{(n)}}: V_{(n)} \rightarrow \hat{W}^{n_{0} - s_{0}}$. It is a well-known fact that in this case $(N', \ul{\iota}')$ is an embedded submanifold of $N$, see e.g. Proposition 5.16 in \cite{lee2012introduction}.  

We conclude that $(\cN', \iota')$ is an embedded submanifold of $\cN$. 

The next step is to construct a graded smooth map $\phi': \cN' \rightarrow \cS$, such that $\phi \circ \iota' = \iota \circ \phi'$. We will use Proposition \ref{tvrz_weakembedd} to do so. We consider a graded smooth map $\phi \circ \iota': \cN' \rightarrow \M$. Clearly $\ul{\phi \circ \iota'}(N') \subseteq S$. It suffices to prove that $(\phi \circ \iota')^{\ast}( \J_{\cS}) = 0$. By Remark \ref{rem_weakweakembedd}, it suffices to find $U \in \Op(M)$, such that $S \subseteq U$ and $(\phi \circ \iota')^{\ast}_{U}(\J_{\cS}(U)) = 0$. Let $U := \cup_{n \in N'} U_{(n)}$ and consider $f \in \J_{\cS}(U)$. Then for each $n \in N'$, one has 
\begin{equation}
(\phi \circ \iota')^{\ast}_{U}(f)|_{N'_{(n)}} = \iota'^{\ast}_{V_{(n)}}( \phi^{\ast}_{U_{(n)}}( f|_{U_{(n)}})).
\end{equation}
But $\J_{\cS}(U_{(n)})$ is generated by a finite graded set $\{ \{ (\psi_{(n)})^{\ast}_{\hat{W}_{(n)}}( \bby^{K_{j}}_{(j)}) \}_{K_{j}=1}^{n_{j} - s_{j}} \}_{j \in \Z}$.  Hence
\begin{equation}
\phi^{\ast}_{U_{(n)}}(f|_{U_{(n)}}) = \sum_{j \in \Z} \sum_{K_{j}=1}^{n_{j}-s_{j}} \phi^{\ast}_{U_{(n)}}(h^{K_{j}}_{(j)}) \cdot (\psi'_{(n)})^{\ast}_{\hat{W}_{(n)}}( \bby^{K_{j}}_{(j)}) \in \J'_{(n)}(V_{(n)}),
\end{equation}
where $h^{K_{j}}_{(j)} \in \C^{\infty}_{\M}(U_{(n)})$ are some graded smooth functions. One has 
\begin{equation}
\ker( \iota'^{\ast}_{V_{(n)}}) = \ker( (\iota'_{[n]})^{\ast}_{V_{(n)}}) = \ker( (\iota'_{(n)})^{\ast}_{V_{(n)}}) = \J'_{(n)}(V_{(n)}).
\end{equation}
But this proves that $(\phi \circ \iota')^{\ast}_{U}(f)|_{N'_{(n)}} = 0$ for any $n \in N'$. As $\{ N'_{(n)} \}_{n \in N'}$ covers $N'$, this proves that $(\phi \circ \iota')^{\ast}_{U}( \J_{\cS}(U)) = 0$. We thus obtain a graded smooth map $\phi': \cN' \rightarrow \cN$ fitting into the commutative diagram (\ref{eq_inverseimagediagrampullback}). 

\item \textbf{Universality of the pullback}: Let us now prove that the pullback diagram (\ref{eq_inverseimagediagrampullback}) has the universal property. To any graded manifold $\Q$ and a pair of graded smooth maps $\mu: \Q \rightarrow \cS$ and $\chi: \Q \rightarrow \cN$ satisfying $\phi \circ \mu = \iota \circ \chi$, we must find a unique graded smooth map $\varphi: \Q \rightarrow \cN'$, such that $\phi' \circ \varphi = \chi$ and $\iota' \circ \varphi = \mu$. 

Since $(\cN',\iota')$ is an embedded submanifold, we will employ Proposition \ref{tvrz_weakembedd} to construct $\varphi$. As $\phi \circ \mu = \iota \circ \chi$, we see that $\ul{\mu}(Q) \subseteq N'$. By Remark \ref{rem_weakweakembedd}, it suffices to find $V \in \Op(N)$, such that $N' \subseteq V$ and $\mu^{\ast}_{V}( \J_{\cN'}(V)) = 0$. Let $V := \cup_{n \in N'} V_{(n)}$. Let $f \in \J_{\cN'}(V)$. For any $n \in N'$, one has 
\begin{equation}
\mu^{\ast}_{V}(f)|_{\ul{\mu}^{-1}(V_{(n)})} = \mu^{\ast}_{V_{(n)}}( f|_{V_{(n)}}).
\end{equation}
But $\J_{\cN'}(V_{(n)}) = \J'_{(n)}(V_{(n)})$ is generated by a finite graded set $\{ \{ (\psi'_{(n)})^{\ast}_{\hat{W}_{(n)}}(\bby^{K_{j}}_{(j)}) \}_{K_{j}=1}^{n_{j}-s_{j}} \}_{j \in \Z}$. There thus exist graded smooth functions $h^{K_{j}}_{(j)} \in \C^{\infty}_{\cN}(V_{(n)})$, such that 
\begin{equation}
\begin{split}
\mu^{\ast}_{V_{(n)}}(f|_{V_{(n)}}) = & \ \sum_{j \in \Z} \sum_{K_{j}=1}^{n_{j} - s_{j}} \mu^{\ast}_{V_{(n)}}(h^{K_{j}}_{(j)}) \cdot (\mu^{\ast}_{V_{(n)}} \circ \phi^{\ast}_{U_{(n)}})( (\psi_{(n)})^{\ast}_{\hat{W}_{(n)}}( \bby^{K_{j}}_{(j)})) \\
= & \ \sum_{j \in \Z} \sum_{K_{j}=1}^{n_{j} - s_{j}} \mu^{\ast}_{V_{(n)}}(h^{K_{j}}_{(j)}) \cdot (\chi^{\ast}_{U_{(n)} \cap S} \circ \iota^{\ast}_{U_{(n)}})( (\psi_{(n)})^{\ast}_{\hat{W}_{(n)}}( \bby^{K_{j}}_{(j)})) = 0,
\end{split}
\end{equation}
where we have used the fact that $(\psi_{(n)})^{\ast}_{\hat{W}_{(n)}}( \bby^{K_{j}}_{(j)}) \in \J_{\cS}(U_{(n)}) = \ker(\iota^{\ast}_{U_{(n)}})$ for all $j \in \Z$ and $K_{j} \in \{1,\dots,n_{j}-s_{j}\}$. Since $\{ \ul{\mu}^{-1}(V_{(n)}) \}_{n \in N'}$ covers $\ul{\mu}^{-1}(V)$ and $f \in \J_{\cN'}(V)$ was arbitrary, this proves that that $\mu^{\ast}_{V}(\J_{\cN'}(V)) = 0$. There is thus a unique graded smooth map $\varphi: \Q \rightarrow \cN'$, such that $\iota' \circ \varphi = \mu$. The other condition $\phi' \circ \varphi = \chi$ now follows automatically from the fact that $(\cS,\iota)$ is an embedded submanifold of $\M$, using the same arguments as in the part (c) of the proof of Theorem \ref{thm_inverseimage}. Hence (\ref{eq_inverseimagediagrampullback}) is a pullback square in $\gMan^{\infty}$. 

\item \textbf{Tangent space}: It remains to prove the formula (\ref{eq_inversetangent}). 

But that is trivial, as for any $n \in N'$, one has 
\begin{equation}
\begin{split}
T_{n} \cN' = & \ T_{n} \cN'_{(n)} = (T_{n} \psi'_{(n)})^{-1}( T_{0} \hat{W}^{(n_{j}-s_{j})}) = (T_{n}\phi)^{-1}( (T_{\ul{\phi}(n)} \psi_{(n)})^{-1}(  T_{0} \hat{W}^{(n_{j}-s_{j})})) \\
= & \ (T_{n}\phi)^{-1}( T_{\ul{\phi}(n)} \cS).
\end{split}
\end{equation}
We have used the fact that $\cN'_{(n)} = \psi'^{-1}_{(n)}(0)$ and $\cS|_{U_{(n)} \cap S}$ is equivalent to $\psi_{(n)}^{-1}(0)$, so one can employ the formula (\ref{eq_inversetangent}) for inverse images of closed embedded submanifolds. 
\end{enumerate}
The proof is now finished.
\end{proof}

\end{document}